\documentclass{aomart}
\chardef\bslash=`\\

\usepackage[shortlabels]{enumitem}
\usepackage{amssymb,amsmath,amsfonts,amsthm,epsfig,amscd,stmaryrd}
\usepackage{stmaryrd}
\usepackage[all,cmtip,poly]{xy}
\usepackage{color}

\usepackage{tikz}
\usepackage{tikz-cd}
\usepackage{microtype}

\makeatletter
\newtheorem*{rep@theorem}{\rep@title}
\newcommand{\newreptheorem}[2]{%
\newenvironment{rep#1}[1]{%
 \def\rep@title{#2 \ref{##1}}%
 \begin{rep@theorem}}%
 {\end{rep@theorem}}}
\makeatother

\usepackage{silence}
\newtheorem{theorem}[subsubsection]{Theorem}
\newtheorem{thm}[subsubsection]{Theorem}
\newreptheorem{theorem}{Theorem}
\newtheorem{lemma}[subsubsection]{Lemma}
\newtheorem{lem}[subsubsection]{Lemma}

\newtheorem{cor}[subsubsection]{Corollary}

\newtheorem{prop}[subsubsection]{Proposition}

\newtheorem{defn}[subsubsection]{Definition}
\theoremstyle{remark}
\newtheorem{remark}[subsubsection]{Remark}
\newtheorem{rem}[subsubsection]{Remark}
\newtheorem{assumption}[subsubsection]{Assumption}
\newtheorem{hypothesis}[subsubsection]{Hypothesis}

\def\numequation{\addtocounter{subsubsection}{1}\begin{equation}}
\def\nummultline{\addtocounter{subsubsection}{1}\begin{multline}}
\def\anumequation{\addtocounter{subsection}{1}\begin{equation}}
\def\anummultline{\addtocounter{subsection}{1}\begin{multline}}

\renewcommand{\theequation}{\arabic{section}.\arabic{subsection}.\arabic{subsubsection}}

\newif\iffinalrun
\iffinalrun
\else
 \fi

\iffinalrun
  \newcommand{\need}[1]{}
  \newcommand{\mar}[1]{}
\else
  \newcommand{\need}[1]{{\tiny *** #1}}
\newcommand{\mar}[1]{\marginpar{\raggedright\tiny FIXME: #1 }}\fi

\renewcommand\mathbb{\mathbf}

\newcommand{\laux}{l^{\aux}}
\newcommand{\Tor}{\mathrm{Tor}}
\newcommand{\suffi}{{\operatorname{suffices}}}

\newcommand{\auxp}{l}
\newcommand{\Lie}{{\operatorname{Lie}\,}}
\newcommand{\semis}{{\operatorname{ss}}}
\newcommand{\gog}{{\mathfrak{g}}}

\newcommand{\mf}{\mathfrak}
\newcommand{\tor}{{\operatorname{tor}}}
\newcommand{\tf}{{\operatorname{tf}}}
\newcommand{\hatQQ}{{\widehat{\Q}}}
\newcommand{\hatZZ}{{\widehat{\Z}}}
\newcommand{\Spp}{{\operatorname{Sp}\,}}
\newcommand{\BC}{{\operatorname{BC}\,}}
\newcommand{\norm}{{\mbox{\bf N}}}
\newcommand{\wt}{{\operatorname{wt}}}

\newcommand{\ndiv}{{\mbox{$\not| $}}}
\newcommand{\Br}{{\operatorname{Br}}}
\newcommand{\rec}{{\operatorname{rec}}}
\newcommand{\avoid}{\mathrm{avoid}} 
\newcommand{\barepsilon}{{\overline{\epsilon}}}
\newcommand{\mat}[4]{\left( \begin{array}{cc} {#1} & {#2} \\ {#3} & {#4}\end{array}\right)}

\newcommand{\ra}{\rightarrow}
\newcommand{\lra}{\longrightarrow}
\newcommand{\liso}{\stackrel{\sim}{\lra}}

\newcommand{\barM}{\overline{{M}}}
\newcommand{\barK}{\overline{{K}}}
\newcommand{\barD}{\overline{{D}}}
\newcommand{\barF}{\overline{{F}}}
\newcommand{\barE}{\overline{{E}}}

\newcommand{\barQQ}{\overline{{\Q}}}
\newcommand{\barFF}{\overline{{\F}}}

\newcommand{\barr}{\overline{{r}}}

\newcommand{\rbar}{\overline{r}}

\newcommand{\chibar}{\overline{\chi}}
\newcommand{\barchi}{{\overline{\chi}}}

\newcommand{\barpsi}{{\overline{\psi}}}

\renewcommand{\ell}{l}
\newcommand{\HHH}{$\mathrm{DGI}$}

\def\Pbarrho{\mathrm{P}\rhobar}

\def\PSL{\mathrm{PSL}}
\def\PGL{\mathrm{PGL}}

\def\chara{\mathrm{char}}

\def\Iw{\mathrm{Iw}}

\def\tF{{\widetilde{F}}}
\def\tE{{\widetilde{E}}}

\def\tT{{\widetilde{T}}}

\newcommand{\loc}{\operatorname{loc}}
\newcommand{\ad}{\operatorname{ad}}
\newcommand{\diag}{\operatorname{diag}}
\newcommand{\tr}{\operatorname{tr}}
\newcommand{\CNL}{\operatorname{CNL}}
\newcommand{\gF}{{\mathfrak{F}}}

\newcommand{\wotimes}{\widehat{\otimes}}
\def\iso{\buildrel \sim \over \longrightarrow}

\newcommand{\A}{\mathbf{A}}
\newcommand{\bA}{\ensuremath{\mathbf{A}}}
\newcommand{\CC}{{\mathbb C}}
\newcommand{\C}{\CC}
\newcommand{\bC}{\ensuremath{\mathbf{C}}}
\newcommand{\bD}{\ensuremath{\mathbf{D}}}
\newcommand{\F}{\FF}
\newcommand{\FF}{{\mathbb F}}

\newcommand{\bG}{\ensuremath{\mathbf{G}}}

\newcommand{\bK}{\ensuremath{\mathbf{K}}}
\newcommand{\bL}{\ensuremath{\mathbf{L}}}
\newcommand{\LL}{{\mathbb L}}
\newcommand{\bQ}{\ensuremath{\mathbf{Q}}}
\newcommand{\Q}{\QQ}
\newcommand{\QQ}{{\mathbb Q}}
\newcommand{\NN}{{\mathbb N}}

\newcommand{\bR}{\ensuremath{\mathbf{R}}}
\newcommand{\R}{\RR}
\newcommand{\RR}{{\mathbb R}}
\newcommand{\bT}{\ensuremath{\mathbf{T}}}
\newcommand{\TT}{{\mathbb T}}
\newcommand{\T}{{\mathbb T}}
\newcommand{\Z}{\ZZ}
\newcommand{\ZZ}{{\mathbb Z}}
\newcommand{\bZ}{\ensuremath{\mathbf{Z}}}

\newcommand{\cC}{{\mathcal C}}
\newcommand{\cD}{{\mathcal D}}

\newcommand{\cF}{{\mathcal F}}
\newcommand{\cG}{{\mathcal G}}
\newcommand{\cH}{{\mathcal H}}
\newcommand{\cI}{{\mathcal I}}

\newcommand{\cL}{{\mathcal L}}
\newcommand{\CL}{{\mathcal{L}}}

\newcommand{\cM}{{\mathcal M}}
\newcommand{\cO}{{\mathcal O}}
\newcommand{\cR}{{\mathcal R}}
\newcommand{\cS}{{\mathcal S}}
\newcommand{\cT}{{\mathcal T}}
\newcommand{\cV}{{\mathcal V}}
\newcommand{\cW}{{\mathcal W}}

\newcommand{\CO}{{\mathcal{O}}}
\newcommand{\CR}{{\mathcal{R}}}
\newcommand{\CS}{{\mathcal{S}}}
\newcommand{\CX}{{\mathcal{X}}}
\newcommand{\cX}{{\mathcal{X}}}
\newcommand{\calD}{{\mathcal{D}}}

\newcommand{\ssc}{\mathrm{sc}}

\newcommand{\aux}{\mathrm{aux}}
\newcommand{\phibar}{\overline{\phi}}

\newcommand{\m}{\frakm}
\newcommand{\n}{\frakn}
\newcommand{\ffrm}{{\mathfrak m}}
\newcommand{\frakm}{\mathfrak{m}}
\newcommand{\frakn}{\mathfrak{n}}
\newcommand{\frakp}{\mathfrak{p}}
\newcommand{\p}{\frakp}
\newcommand{\frakq}{\mathfrak{q}}
\newcommand{\q}{\frakq}

\newcommand{\gM}{\mathfrak{M}}

\newcommand{\ga}{\mathfrak{a}}
\newcommand{\gp}{\mathfrak{p}}
\newcommand{\gm}{\mathfrak{m}}
\newcommand{\gq}{\mathfrak{q}}

\newcommand{\fra}{{\mathfrak a}}

\newcommand{\Fbar}{\overline{\F}}
\newcommand{\Qbar}{\overline{\Q}}

\newcommand{\barx}{\overline{x}}

\newcommand{\Zp}{\Z_p}

\newcommand{\Ql}{\Q_{\ell}}

\newcommand{\Qpbar}{\Qbar_p}

\DeclareMathOperator{\Aut}{Aut}

\DeclareMathOperator{\End}{End}

\DeclareMathOperator{\Fil}{Fil}
\DeclareMathOperator{\gr}{gr}
\DeclareMathOperator{\Gal}{Gal}
\newcommand{\GL}{\mathrm{GL}}
\newcommand{\GSp}{\mathrm{GSp}}
\DeclareMathOperator{\Hom}{Hom}
\DeclareMathOperator{\im}{im}
\DeclareMathOperator{\Ind}{Ind}
\DeclareMathOperator{\nInd}{n-Ind}
\DeclareMathOperator{\ord}{ord}
\DeclareMathOperator{\SL}{SL}
\DeclareMathOperator{\Spec}{Spec}
\DeclareMathOperator{\Supp}{Supp}
\DeclareMathOperator{\WD}{WD}
\DeclareMathOperator{\Sym}{Symm}
\DeclareMathOperator{\Symm}{Symm}

\newcommand{\ab}{\mathrm{ab}}
\newcommand{\Frob}{\mathrm{Frob}}
\newcommand{\HT}{\mathrm{HT}}
\newcommand{\nr}{\mathrm{nr}}

\newcommand{\barrho}{\overline{\rho}}
\newcommand{\rhobar}{\overline{\rho}}
\newcommand*{\invlim}{\varprojlim}                              
 \newcommand{\into}{\hookrightarrow}
\newcommand{\onto}{\twoheadrightarrow}
\newcommand{\toisom}{\buildrel\sim\over\to}
\newcommand{\Art}{{\operatorname{Art}}}
\newcommand{\Res}{\operatorname{Res}}
\newcommand{\RS}{\operatorname{Res}}

\newcommand{\detord}{\mathrm{det,ord}}
\newcommand{\univ}{\mathrm{univ}}

\newcommand{\doubleslash}{/\kern-0.2em{/}}
\newcommand{\Grm}{\mathrm{G}}

\newcommand{\Prm}{\mathrm{P}}

\title{Potential automorphy over CM fields}

\author[P.~Allen]{Patrick B. Allen} \email{patrick.allen@mcgill.ca}
 \address{Department of Mathematics and Statistics, 
McGill University,
Montreal, QC H3A 0B9, Canada}

\author[F.~Calegari]{Frank Calegari}  \email{fcale@math.uchicago.edu} \address{Department of
  Mathematics, University of Chicago,
5734 S University Ave,
Chicago, IL 60637, USA}

\author[A.~Caraiani]{Ana Caraiani} \email{caraiani.ana@gmail.com}
\address{Department of
  Mathematics, Imperial College London,
  London SW7 2AZ, UK}

\author[T.~Gee]{Toby Gee} \email{toby.gee@imperial.ac.uk} \address{Department of
  Mathematics, Imperial College London,
  London SW7 2AZ, UK}
 
\author[D.~Helm]{David Helm} \email{d.helm@imperial.ac.uk}
 \address{Department of
  Mathematics, Imperial College London,
  London SW7 2AZ, UK}

\author[B.~V.~Le Hung]{Bao V. Le Hung} \email{lhvietbao@googlemail.com}
 \address{Department of
  Mathematics, Northwestern University, 2033 Sheridan Road,
Evanston, IL 60208, USA}

\author[J.~Newton]{James Newton} \email{newton@maths.ox.ac.uk}
\address{Mathematical Institute, University of Oxford, Woodstock Road, Oxford OX2 6GG, UK}

\author[P.~Scholze]{Peter Scholze} \email{scholze@math.uni-bonn.de} \address{Universit\"{a}t Bonn,
Endenicher Allee 60,
53115 Bonn, Germany}

\author[R.~Taylor]{Richard Taylor}  \email{rltaylor@stanford.edu} \address{Department of Mathematics, Stanford University,  Stanford CA 94305, USA}

\author[J.~A.~Thorne]{Jack A. Thorne} \email{thorne@dpmms.cam.ac.uk} \address{Department of Pure Mathematics and Mathematical Statistics, Wilberforce Road, Cambridge CB3 0WB, UK}

\thanks{P.A. \ was supported in part by Simons Foundation Collaboration Grant for Mathematicians 527275, NSF Grant DMS-1902155, and by NSERC}

\thanks{F.C. \ was supported in part by NSF Grants  DMS-1701703 and DMS-2001097.}

\thanks{A.C. \ was supported in part by NSF Grant DMS-1501064, 
by a Royal Society University Research Fellowship, by ERC Starting Grant 804176
and by a Leverhulme Prize.}

\thanks{T.G.\ was
  supported in part by a Leverhulme Prize, EPSRC grant EP/L025485/1,
  ERC Starting Grant 306326, and a Royal Society Wolfson Research
  Merit Award.}

\thanks{B.L. \ was    supported in part by NSF Grant  DMS-1802037, NSF Grant DMS-1952678 and the Alfred P. Sloan Foundation.}

\thanks{J.N. \ was supported by a UKRI Future Leaders Fellowship, grant MR/V021931/1.}

\thanks{P.S. \ was supported in part by a DFG Leibniz Grant, and by the DFG under the Excellence Strategy  EXC-2047/1-390685813}

\thanks{R.T.\ was supported by NSF Grant DMS-1902265 during the revision of this paper.}

\thanks{J.T.\ was supported by a Clay Research Fellowship and ERC Starting Grant 714405.}

\begin{document}

\begin{abstract}
Let~$F$ be a CM number field.
     We prove modularity lifting theorems for regular 
        $n$-dimensional Galois representations over~$F$  without
        any self-duality condition. We deduce that all elliptic curves~$E$ over~$F$
        are potentially modular, and furthermore satisfy the Sato--Tate conjecture.
        As an application of a different sort, we  also prove the Ramanujan Conjecture for
        weight zero cuspidal automorphic representations for~$\GL_2(\A_F)$.
\end{abstract}

\maketitle
\setcounter{tocdepth}{2}
{\footnotesize
\tableofcontents
}

\section{Introduction}

In this paper, we prove the first unconditional modularity lifting theorems for $n$-dimensional regular Galois representations without any self-duality conditions. A version of these results were proved in~\cite{CG} \emph{conditional} on two conjectures. The first conjecture was that the Galois representations constructed by Scholze in~\cite{scholze-torsion} satisfy a strong form of local-global compatibility at all primes. The second was a vanishing conjecture for the mod-$p$ cohomology of arithmetic groups localized at non-Eisenstein primes which mirrored the corresponding (known) vanishing theorems for cohomology corresponding to tempered automorphic representations in characteristic zero. We prove many cases of the first of these conjectures
in this paper. %
Our arguments crucially exploit work of Caraiani and Scholze~\cite{caraiani-scholze-noncompact} on the cohomology
of non-compact Shimura varieties (see also~\cite{caraiani-scholze-compact} for the compact version of these results).
The details of this argument are carried out in~\S\ref{section:lep} and~\S\ref{sec:ordsection}.
(It turns out that, in the easier case when~$l \ne p$, one can argue more directly
using the original construction in~\cite{scholze-torsion}, and this is done in~\S\ref{section:lnep}.)
On the other hand, we do \emph{not} resolve the second conjecture   concerning
the vanishing of mod-$p$ cohomology in this paper. Rather, we sidestep this difficulty by a new technical innovation; a derived version of  ``Ihara avoidance'' which simultaneously generalizes
the main idea of~\cite{tay} as well as a localization in characteristic zero
idea first used in~\cite{KT}. This argument, together with
the proofs of the main automorphy lifting theorems, is given in~\S\ref{section:alt}.
The result is that we are able to prove quite
general modularity lifting theorems in both the ordinary and Fontaine--Laffaille case
for general~$n$-dimensional representations  over CM fields,
in particular Theorems~\ref{thm:main_automorphy_lifting_theorem} and~\ref{thm:main_ordinary_automorphy_lifting_theorem}.  Instead of
reproducing those theorems here (which require  a certain amount of notation), we
instead reproduce here  a few 
corollaries of our main theorems which are worked out in~\S\ref{section:tricks}.
The first theorem is a special case of Corollaries~\ref{maincor} and~\ref{cor:satotate}:

\begin{theorem} \label{thm:satotateintro} Let $E$ be an elliptic curve over a CM number field~$F$. Then~$E$ and all
the symmetric powers of~$E$ are potentially modular.  Consequently, the Sato--Tate conjecture holds for~$E$.
\end{theorem}

For an application of a different sort, we also have the following
special case of the Ramanujan conjecture (see Corollary~\ref{cor:ramanujan}):

\begin{theorem}  \label{thm:ramanujanintro} 
 Let~$F$ be a CM field, and let~$\pi$ be a regular
  algebraic cuspidal automorphic representation of $\GL_2(\A_F)$ of
weight~$0$.
Then, for all  primes $v$ of
  $F$, the representation $\pi_v$ is tempered. \end{theorem}

This is, to our knowledge, the first case of the Ramanujan conjecture
to be proved  for
which neither the underlying Galois representation~$V$ nor some closely related Galois representation (such as~$V^{\otimes 2}$ or~$\Sym^2 V$) is known to occur
as a summand of the \'{e}tale cohomology of some smooth proper algebraic
variety over a number field; in such cases temperedness (at unramified primes)  is ultimately a consequence of 
 Deligne's purity theorem. Our proof, in contrast, follows  more closely the original strategy proposed by Langlands.
 Langlands explained~\cite{LanglandsProblems} how one could deduce Ramanujan from
 functoriality; namely, functoriality implies the automorphy of~$\Symm^n(\pi)$ and~$\Symm^n(\pi^{\vee})$ as
 well as the product~$\Symm^n(\pi) \boxtimes \Symm^n(\pi^{\vee})$. Then, by considering 
 standard analytic properties of the standard~$L$-function associated to~$\Symm^n(\pi) \boxtimes \Symm^n(\pi^{\vee})$
 (and exploiting a positivity property of the coefficients of this~$L$-function) one deduces the required bounds.
 As an approximation to this, we show  that all the symmetric powers of~$\pi$ (and~$\pi^{\vee}$) are  \emph{potentially} automorphic, and then
 invoke analytic properties of the Rankin--Selberg $L$-function
 (in the guise of the Jacquet--Shalika bounds~\cite{jsajm103}) as a replacement for the (potential) automorphy
 of their product.

\subsection{A brief overview of the argument}
Let~$F/F^{+}$ be an imaginary CM field, let~$K \subset \GL_n(\A_F^\infty)$ be a
compact open subgroup, let~$X_K$
denote the corresponding (non-Hermitian) locally symmetric space,  let~$E/\Q_p$ denote a finite extension with ring of integers~$\CO$,
and let~$\cV = \cV_\lambda$ denote a local system on~$X_K$
which is a  lattice inside an algebraic representation of weight~$\lambda$ defined over~$E$. (For example, $\cV$ could be the trivial
local system~$\CO$.) After omitting a finite set of primes~$S$ containing the~$p$-adic places (and satisfying some further hypotheses),
one may define a Hecke algebra~$\T = \T^S$  as the image of a formal
ring of Hecke operators in $\End_{\mathbf{D}(\cO)}( R \Gamma(X_K, \cV) )$ where~$\mathbf{D}(\cO)$ is the derived
category of~$\cO$-modules. (This is  isomorphic  to the usual ring of Hecke operators
acting on~$H^*(X_K,\cV)$ up to a nilpotent ideal, but for technical reasons it is better to work in the derived setting, cf.~\cite{new-tho}.)
For a non-Eisenstein maximal ideal~$\m$, the main result of~\cite{scholze-torsion}
guarantees the existence of a Galois representation
\[ \rho_\ffrm : G_{F, S} \to \GL_n(\T_\ffrm / J) \]
characterized, up to conjugation, by the characteristic polynomials of Frobenius elements at places $v \not\in S$, where
$J$ is a nilpotent ideal whose exponent depends only on $n$ and $[F : \Q]$.
It is crucial for applications to modularity lifting theorems (following the strategy outlined in~\cite{CG}) to know that this Galois representation satisfies local--global compatibility at 
\emph{all} primes.
(As usual, in order to talk about local-global compatibility at a prime in~$S$, one has to work
with variants of~$\T$ including Hecke operators at these primes --- we ignore all such distinctions here).
Since~$ \T_\ffrm / J$ is (in general) not flat over~$\CO$, it is not exactly clear
what one should expect to mean  by local--global compatibility.
For example, for primes~$l \ne p$, a (torsion) representation which is Steinberg at~$l$ need not be ramified at~$l$. Instead,
we ask that
 the characteristic polynomials of~$\rho_\ffrm(\sigma)$ for~$\sigma \in I_{v}$ for~$v | l \in S$ and~$l \ne p$  have the expected shape.
 Such a condition is amenable to arguments using congruences, and we prove a version of this compatibility in~\S\ref{section:lnep} (see Theorem~\ref{thm:lgc_at_l_neq_p}). Note that our theorem only applies to a limited range of~$l$; in particular,
we assume that the level~$K_v$ (for~$v | l \in S$ and~$l \ne p$) satisfies the inclusions~$\Iw_{v,1} \subset K_v \subset \Iw_v$ (where~$\Iw_{v}$ and~$\Iw_{v,1}$ are the Iwahori and pro-$l$
Iwahori respectively)  and additionally~$l$  satisfies various
splitting conditions relative to the field~$F$. This suffices for
applications to modularity, where we make a soluble base change to
ensure that Theorem~\ref{thm:lgc_at_l_neq_p} applies to both
Taylor--Wiles primes and the ramified primes~$S$ away from~$p$. %
This part of
the argument requires only the construction of Galois representations in~\cite{scholze-torsion}.

Local--global compatibility for~$l = p$ is more subtle. Indeed, we are not confident enough to formulate a precise conjecture of what local--global compatibility means in general in the torsion setting.
 Instead, we restrict to two settings where  the conjectural
formulation of local--global compatibility is more transparent; the case when~$\rho_{\ffrm}$ should be Fontaine--Laffaille (assuming,
in particular, that $p$ is unramified in~$F$) and the ordinary case (with no restriction on~$F$); \S\ref{section:lep}
and~\S\ref{sec:ordsection} are devoted to proving such theorems.
In both of these cases, the underlying strategy is as follows.
Associated to our data is a quasi-split unitary group~$\widetilde{G}$ over~$F^{+}$ which is a form of~$\GL_{2n}$ that splits over~$F/F^{+}$.
There is a parabolic subgroup~$P$ of~$\widetilde{G}$ whose
Levi subgroup~$G$ over~$F^{+}$ may be identified with $\Res_{F /F^{+}} \GL_n$, and hence associated with the locally
symmetric spaces~$X_K$ as above.
The point of this construction is that~$\widetilde{G}$  may be associated to a Shimura variety~$\widetilde{X}_{\widetilde{K}}$ (and thus to Galois
representations of known provenance) whereas the cohomology of~$X_K$ appears inside (in some non-trivial way) a spectral sequence computing
the cohomology of the boundary $\partial \widetilde{X}_{\widetilde{K}}$ of the Borel--Serre compactification of $\widetilde{X}_{\widetilde{K}}$. 
One now faces several complications. The first is that the cohomology of the boundary involves different parabolic subgroups of~$\widetilde{G}$ besides~$P$.
This is resolved by the assumption that~$\m$ is non-Eisenstein. The second is separating inside the boundary cohomology (associated to~$P$) the contribution coming
from~$G$ and that coming from the unipotent subgroup~$U$ of~$P$. Fortunately, the unipotent subgroup~$U$ is abelian and well understood, and we show (for~$p > n^2$)
that the relevant cohomology we are interested in occurs as a direct summand of the cohomology of~$\partial \widetilde{X}_{\widetilde{K}}$ (see Theorem~\ref{direct summand at finite level}).  
Note that, for a general coefficient system~$\cV = \cV_\lambda$  on~$X_K$, there are a number of different coefficient systems~$\cV_{\widetilde{\lambda}}$ on~$\widetilde{X}_{\widetilde{K}}$
for which~$H^*(\partial \widetilde{X}_{\widetilde{K}},\cV_{\widetilde{\lambda}})$ can be related to~$H^*(X_K,\cV_{\lambda})$, and this freedom of choice
will be important in what follows.
By these arguments, we may exhibit~$R \Gamma(X_K, \cV_\lambda / \varpi^m)_\m$
up to shift
 as a direct summand of~$R \Gamma(\partial\widetilde{X}_{\widetilde{K}}, \cV_{\widetilde{\lambda}} / \varpi^m)_{\widetilde{\m}}$.
 (Here~$\widetilde{\m}$ is the corresponding ideal of the Hecke algebra~$\widetilde{\T}$ for~$\widetilde{G}$,
 and~$\overline{\rho}_{\widetilde{\m}}$ is the corresponding (reducible) $2n$-dimensional representation associated to~$\m$,
 from which~$\rhobar_{\m}$ was constructed.)
  Now suppose that~$d$ is the complex (middle) dimension of~$\widetilde{X}_{\widetilde{K}}$.
 We now make crucial use of the following theorem, which is the main theorem of~\cite{caraiani-scholze-noncompact}
 (see Theorem~\ref{what we get from CS} for a more general statement.)
 
 \begin{theorem}[Caraiani--Scholze~{\cite[Theorem~1.1]{caraiani-scholze-noncompact}}]  \label{CSintro} Assume that $F^+ \ne \Q$, 
 that~$\m$ is non-Eisenstein, and that $\overline{\rho}_{\widetilde{\m}}$ is decomposed generic in the sense of 	Definition~\ref{defn:decomposed_generic}.
 Assume that, for every prime~$l$ which is the residue characteristic of a prime dividing~$S$ or~$\Delta_F$, there 
 exists an imaginary quadratic field $F_0 \subset F$ in which $l$ splits.
 Then
	\[
	H^i(\widetilde{X}_{\widetilde{K}}, 
	\cV_{\widetilde{\lambda}}/\varpi)_{\widetilde{\m}} = 0\ 
	\mathrm{if}\ i<d,\ 
	\mathrm{and}\  
	H^i_c(\widetilde{X}_{\widetilde{K}}, 
	\cV_{\widetilde{\lambda}}/\varpi)_{\widetilde{\m}} = 0\ 
	\mathrm{if}\ i>d.
	\]
\end{theorem}
This immediately gives  a diagram as follows:
	\[
	H^d(\widetilde{X}_{\widetilde{K}},\cV_{\widetilde{\lambda}}[1/p])_{\widetilde{\m}}\hookleftarrow
	H^d(\widetilde{X}_{\widetilde{K}},\cV_{\widetilde{\lambda}})_{\widetilde{\m}}\twoheadrightarrow
	H^d(\partial\widetilde{X}_{\widetilde{K}},\cV_{\widetilde{\lambda}})_{\widetilde{\m}}.
	\]  
where the leftmost term can be understood in terms of automorphic forms on Shimura varieties,
and in particular (under appropriate assumptions) gives rise to Galois representations 
having the desired $p$-adic Hodge theoretic properties,
and the rightmost term (by construction) now 
sees the part of~$R \Gamma(\partial\widetilde{X}_{\widetilde{K}}, \cV_{\widetilde{\lambda}} / \varpi^m)_{\widetilde{\m}}$
which (after shifting) contributes \emph{in degree~$d$}, at least up to a fixed level of nilpotence.

The idea is then to choose the weight $\widetilde{\lambda}$ so that~$\cV_{\widetilde{\lambda}}$
on~$\widetilde{X}_{\widetilde{K}}$ is related to~$\cV_{\lambda}$
on~$X_K$ by the action of a Weyl group as in Kostant's formula
\cite[Theorem 5.14]{kostant} 
(to do this integrally, we need to assume that $p$ is sufficiently large), 
and that by varying~$\widetilde{\lambda}$ we may see \emph{all} of the cohomology 
of~$R \Gamma(X_K, \cV_\lambda / \varpi^m)_{\m}$
in the degree $d$ cohomology of $R \Gamma(\partial\widetilde{X}_{\widetilde{K}}, \cV_{\widetilde{\lambda}} / \varpi^m)_{\widetilde{\m}}$.
This idea only works for some weights and degrees, so to get around this, 
we first deepen the levels $K$ and $\widetilde{K}$ at some other place above $p$ which 
allows us to modify the weight $\lambda$ at the corresponding embeddings without changing the Hecke algebra. 
For the modified $\lambda$, we can then find $\widetilde{\lambda}$ and a Weyl group element giving us to access 
to $H^q(X_K, \cV_\lambda)_{\m}$ for $q \ge \lfloor \frac{d}{2} \rfloor$  (see Proposition~\ref{degree-shifting}), 
and we handle the remaining degrees by taking duals.
This part of the argument (including the invocation of Theorem~\ref{CSintro}) requires various local assumptions on~$F$ which can always be achieved after
a soluble base change but are not generally satisfied (in particular, they are not satisfied when~$F^{+} = \Q$). 
We then extract the relevant properties of~$\rho_{\m}$ from those of the determinant associated to~$\widetilde{\m}$. 
This summarizes the argument of~\S\ref{section:lep}.

In~\S\ref{sec:ordsection}, we prove a different local--global compatibility theorem in the ordinary case. Although not  strictly
necessary for our main theorems (for compatible families, by taking sufficiently large primes, one can aways reduce
to the Fontaine--Laffaille case),
this allows us to prove a modularity lifting theorem which may have wider applicability --- in particular, 
the main local--global compatibility result of this section (Theorem~\ref{mySecondAmazingTheorem})
applies to any prime $p$, provided $F$ contains an imaginary quadratic field in which $p$ splits.
The general approach in this section is similar to that of~\S\ref{section:lep}. However, 
instead of exhibiting~$R \Gamma(X_K, \cV_\lambda / \varpi^m)_\m$
up to shift
as a direct summand  (as a Hecke module) of~$R \Gamma(\partial\widetilde{X}_{\widetilde{K}}, \cV_{\widetilde{\lambda}} / \varpi^m)_{\widetilde{\m}}$,
(whose proof in~\S\ref{section:lep} required~$p > n^2$),  we make arguments on the level of completed cohomology, and exploit a version of Emerton's
ordinary parts functor. 
A key computation is that of the ordinary part of a parabolic induction from $P$ to $\widetilde{G}$ in \S\ref{computation of ordinary parts}
following arguments of Hauseux \cite{hauseux}.
Because only part of the cohomology of the unipotent radical $U$ is ordinary, only relative Weyl group elements appear in the degree shifts 
(see Theorem~\ref{thm:ord_deg_shifting}) and consequently we only obtain shifts by multiples of $[F^+:\Q]$ in this way. 
We get around this by a trick using the centre of $G$, 
showing that the Hecke algebra acting on $H^\ast(X_K, \cV_\lambda)$ 
can be understood in terms of the Hecke algebra acting only in degrees that are multiples of $[F^+:\Q]$ (Lemma~\ref{centraldegreeshifting}). 
As in the Fontaine--Laffaille case, we can then extract the relevant properties of~$\rho_{\m}$ from those of the determinant associated to~$\widetilde{\m}$.

 We now turn to the modularity lifting theorems of~\S\ref{section:alt}. A key hypothesis of~\cite{CG} was the truth of a
 vanishing conjecture for integral cohomology  localized at a non-Eisenstein maximal ideal~$\m$ outside a prescribed range (mirrored by the characteristic zero vanishing
 theorems of Borel and Wallach~\cite{MR1721403}). This conjecture remains unresolved.
 Instead, we exploit a localization in characteristic zero idea first employed in~\cite{KT}.
This requires a slightly stronger residual modularity hypothesis --- namely, that~$\rhobar_{\m}$ actually comes
from an automorphic representation rather than one merely associated to a torsion class ---
but this will be satisfied for our applications, and is at any rate required at  other points at the argument
(for example to know that the residual modularity hypothesis is preserved under soluble base change).  Two points remain. The first, which is mostly technical,
is to show that the approach of~\cite{CG} and~\cite{KT} is compatible with the fact that we only
have Galois representations to~$\T/J$ for some nilpotent ideal~$J$. The second, which is more serious,
is to show that the localization argument of~\cite{KT} is compatible with the ``Ihara avoidance argument''
of~\cite{tay} and the (essentially identical)~$l_0 > 0$ version of
this argument in~\cite{CG}. (Here~$l_0$ is the parameter of~\cite{MR1721403} which measures the failure of the underlying real group to admit discrete series and which plays plays a fundamental role in~\cite{CG}.)
 To explain the problem,
we briefly recall the main idea of~\cite{tay} in the~$l_0 = 0$ setting (the difficulties are already apparent in this case).
 One compares two  global deformation
problems which (for exposition) differ only  at an auxiliary prime~$v$ with~$l = N(v) \equiv 1 \mod p$,
and which at all other primes have smooth local deformation conditions.
The corresponding local deformation rings~$R^{(1)}_v$ and~$R^{(2)}_v$ 
at the prime~$v$ are taken to be tame local deformation rings which the image of tame inertia has minimal 
polynomial~$(X-1)^n$ or~$(X-\zeta_1)\ldots(X - \zeta_n)$ respectively for distinct roots of unity~$\zeta_i \equiv 1 \mod v$.
The corresponding patched modules~$H^{(1)}_{\infty}$ and~$H^{(2)}_{\infty}$ constructed via the Taylor--Wiles
method (\cite{ MR1333036,kis04}) have the expected
depth over~$S_{\infty}$. On the one hand, the generic fibre of~$R^{(2)}_v$ is geometrically irreducible, which forces~$H^{(2)}_{\infty}$
to have full support over~$R^{(2)}_v$. On the other hand, there is an
isomorphism~$R^{(1)}_v/\varpi = R^{(2)}_v/\varpi$, and this gives %
an identification~$H^{(1)}_{\infty}/\varpi \simeq H^{(2)}_{\infty}/\varpi$.  But now,  the ring~$R^{(1)}_v$  has the convenient
property that any irreducible component of its special fibre comes from a unique irreducible component of the generic
fibre, %
and from this a modularity result is deduced in~\cite{tay} using commutative algebra. 
Suppose we now drop the hypothesis that
the integral cohomology all contributes to cohomology in a single degree  (still in our~$l_0 =0$ setting), but we  continue to assume this holds after inverting~$p$. Now we can no longer
control the depth of the~$S_{\infty}$-modules~$H^{(1)}_{\infty}$ and~$H^{(2)}_{\infty}$, and so knowing~$H^{(2)}_{\infty}[1/p] \ne 0$
and~$H^{(1)}_{\infty}/\varpi = H^{(2)}_{\infty}/\varpi$
 does not  imply that~$H^{(1)}_{\infty}[1/p] \ne 0$.
For example, it could happen that~$H^{(1)}_{\infty} = H^{(1)}_{\infty}/\varpi = H^{(2)}_{\infty}/\varpi$.
The resolution of this difficulty is
not to simply compare the patched modules in fixed (final) degree, but the entire patched complex in the derived category.
The point is now that these complexes in characteristic~$p$ (which are derived reductions of perfect~${S_\infty}$-complexes
for the ring of diamond operators~$S_{\infty}$) remember information about 
characteristic zero. As a simple avatar of this idea, if~$M$ is a finitely generated~$\Z_p$-module, then~$M[1/p]$ is non-zero if and only
if~$M \otimes^{\bL} \FF_p$ has non-zero Euler characteristic over~$\FF_p$.
The main technical formulation of this principle which allows us to prove a version of Ihara avoidance
in our setting is  Lemma~\ref{comalg2}.

Finally, in~\S\ref{section:tricks}, we apply the results of previous sections to prove Theorems~\ref{thm:satotateintro} and~\ref{thm:ramanujanintro}. 
We begin with some preliminaries on compatible systems in order to show there are enough primes such that the corresponding
residual representations satisfy hypotheses of our modularity lifting theorems.
As expected, the arguments of this section make  use of the~$p$-$q$ switch (\cite{fermat}, but first exploited in
the particular context of potential automorphy in~\cite{tay-fm}) and a theorem of Moret-Bailly~\cite{mb}.

\subsection*{Acknowledgments}
It was realized by two of us (A.C. and P.S.) that the local-global compatibility for torsion Galois representations should be a consequence of a non-compact version~\cite{caraiani-scholze-noncompact} of their recent work on the cohomology of Shimura varieties~\cite{caraiani-scholze-compact}.  This led to 
an emerging topics workshop at the Institute for Advanced Study (organized by A.C and R.L.T. and attended by all the authors of this paper) whose goal was to
explore possible consequences for modularity. It was during this
workshop (in November 2016) that the new  Ihara avoidance argument was
found. The authors gratefully acknowledge the IAS for the opportunity
to run this workshop, and thank Matthew Emerton for
his participation. We are also very grateful to Lambert A'Campo and Konstantin Miagkov for helpful comments and questions on an earlier draft.

\subsection{Notation}\label{sec:misc_notation}
We write all matrix transposes on the left; so ${}^t\!A$ is the
transpose of $A$. We will write $\chara_A$ for the characteristic polynomial of a matrix $A$. We write $\GL_n$ for the usual general linear group (viewed as a reductive group scheme over $\Z$) and $T_n \subset B_n \subset \GL_n$ for its subgroups of diagonal and of upper triangular matrices, respectively.
We will write $O(n)$ (resp.\ $U(n)$) for the group of matrices $g \in \GL_n(\R)$ (resp.\ $\GL_n(\C)$) such that ${}^tg^cg=1_n$.

If $R$ is a local ring, we write $\mf{m}_{R}$ for the
maximal ideal of $R$. 

If $\Delta$ is an abelian group, we will let $\Delta^\tor$ denote its maximal torsion subgroup and $\Delta^\tf$ its maximal torsion free quotient. If $\Delta$ is profinite  and abelian, we will also write $\Delta(l)$ for its Sylow pro-$l$-subgroup, which is naturally isomorphic to its maximal pro-$l$ continuous quotient.
If $\Gamma$ is a profinite group, then
$\Gamma^\ab$ will denote its maximal abelian quotient by a closed
subgroup. If $\rho:\Gamma \to \GL_n(\barQQ_l)$ is a continuous
homomorphism, then we will let $\barrho:\Gamma \ra \GL_n(\barFF_l)$
denote the {\em semi-simplification} of its reduction, which is well defined
up to conjugacy (by the Brauer--Nesbitt theorem). 
If %
$M$ is a topological abelian group with a  continuous action of $\Gamma$, then by $H^i(\Gamma,M)$ we shall mean the continuous cohomology.

If $R$ is a (possibly non-commutative) ring, then we will write $\mathbf{D}(R)$ for the derived category of $R$-modules. By definition, an object of $\mathbf{D}(R)$ is a cochain complex of $R$-modules. An object of $\mathbf{D}(R)$ is said to be perfect if it is isomorphic in this category to a bounded complex of projective $R$-modules. 

If $R$ is a complete Noetherian local ring, $C \in \mathbf{D}(R)$ is a perfect complex, and $T \to \End_{\mathbf{D}(R)}(C)$ is a homomorphism of $R$-algebras, then the image $\overline{T}$ of $T$ in $\End_{\mathbf{D}(R)}(C)$ is a finite $R$-algebra, which can therefore be written as a product $\overline{T} = \prod_{\m} \overline{T}_\m$ of its localizations at maximal ideals. There is a corresponding decomposition $1 = \sum_{\m} e_\m$ of the unit in $\overline{T}$ as a sum of idempotents. Since $\mathbf{D}(R)$ is idempotent complete, this determines a decomposition $C = \oplus_\m C_\m$ in $\mathbf{D}(R)$. The direct summands $C_\m$ are well-defined up to unique isomorphism. We usually reserve the symbol~$C^\bullet$ to
refer to an element in the category of cochain complexes, although hopefully statements of the form~$C^\bullet = 0$ in $\mathbf{D}(R)$ will not cause any confusion.

If $G$ is a locally profinite group, and $U \subset G$ is an open compact subgroup, then we write $\cH(U, G)$ for the algebra of compactly supported, $U$-biinvariant functions $f : G \to \Z$, with multiplication given by convolution with respect to the Haar measure on $G$ which gives $U$ volume 1. If $X \subset G$ is a compact $U$-biinvariant subset, then we write $[X]$ for the characteristic function of $X$, an element of  $\cH(U, G)$. 

If $G$ is a reductive group over a field $k$ and $T \subset G$ is a split maximal torus, then we write $W(G, T)$ for the Weyl group (the set of $k$-points of the quotient $N_G(T) / T)$. For example, if $F / \Q$ is a number field, then we may identify  $W((\Res_{F / \Q} \GL_n)_\C, (\Res_{F / \Q} T_n)_\C)$ with $S_n^{\Hom(F, \C)}$. If $P \subset G$ is a parabolic subgroup which contains $T$, then there is a unique Levi subgroup $L \subset P$ which contains $T$. We write $W_P(G, T)$ for the absolute Weyl group of this Levi subgroup, which may be identified with a subgroup of $W(G, T)$. 

Suppose that $G$ comes equipped with a Borel subgroup $B$ containing
$T$. Then we can form $X^\ast(T)^+ \subset X^\ast(T)$, the subset of
$B$-dominant characters. If $P$ is a parabolic subgroup of $G$ which
contains $B$, with Levi~$L$ as above, then $B \cap L$ is a Borel subgroup of $L$ and we write $X^\ast(T)^{+, P}$ for the subset of $(B \cap L)$-dominant characters. The set 
\[ W^P(G, T) = \{ w \in W(G, T) \mid w(X^\ast(T)^+) \subset X^\ast(T)^{+, P} \} \]
is a set of representatives for the quotient $W_P(G, T) \backslash W(G, T)$. 

\subsubsection*{Galois representations}

If $F$ is a perfect field, we let $\barF$ denote an algebraic closure of $F$
and $G_F$ the absolute Galois group $\Gal(\barF/F)$.  We will use
$\zeta_n$ to denote a primitive $n^{th}$-root of $1$.  Let
$\epsilon_l$ denote the $l$-adic cyclotomic character and
$\barepsilon_l$ its reduction modulo $l$. We will also let
$\omega_l:G_F \ra \mu_{l-1} \subset \Z_l^\times$ denote the
Teichm\"{u}ller lift of $\barepsilon_l$.  If $E/F$ is a separable
quadratic extension, we will let $\delta_{E/F}$ denote the non-trivial
character of $\Gal(E/F)$. We will write $\Br_F$ for the Brauer group of $F$.

We will write $\Q_{l^r}$ for the unique unramified extension of $\Q_l$ of degree $r$ and $\Z_{l^r}$ for its ring of integers. We will write $\Q_l^\nr$ for the maximal unramified extension of $\Q_l$ and $\Z_l^\nr$ for its ring of integers. We will also write $\hatZZ_l^\nr$ for the $l$-adic completion of $\Z_l^\nr$ and $\hatQQ_l^\nr$ for its field of fractions. 

If $K$ is a finite
extension of $\Q_p$ for some $p$, we write $K^\nr$ for its maximal unramified extension; $I_K$ for the inertia
subgroup of $G_K$; $\Frob_K \in G_K/I_K$ for the geometric Frobenius; and
$W_K$ for the Weil group. %
If $K'/K$ is a Galois extension we will write $I_{K'/K}$ for the inertia subgroup of $\Gal(K'/K)$.
We will write $\Art_K:K^\times \iso W_K^\ab$ for the Artin map normalized to send uniformizers to geometric Frobenius elements. We will write $\omega_{l,r}$ for the character $G_{\Q_{l^r}} \ra \Z_{l^r}^\times$ such that $\omega_{l,r} \circ \Art_{\Q_{l^r}}$ sends $l$ to $1$ and sends $a \in \Z_{l^r}^\times$ to the Teichm\"{u}ller lift of $a \bmod l$. This is sometimes referred to as ``the fundamental character of niveau $r$.'' (Thus $\omega_{l,1} = \omega_l$.)

We will let $\rec_K$ be the local Langlands correspondence of
\cite{ht}, so that if $\pi$ is an irreducible complex
admissible representation of $\GL_n(K)$, then $\rec_K(\pi)$ is a
Frobenius semi-simple Weil--Deligne representation of the Weil group $W_K$. 
We will write $\rec$ for $\rec_K$
when the choice of $K$ is clear. We write $\rec^T_K$ for the arithmetic normalization of the local Langlands correspondence, as defined in e.g.\ ~\cite[\S 2.1]{Clo14}; it is defined on irreducible admissible representations of $\GL_n(K)$ defined over any field which is abstractly isomorphic to $\C$ (e.g.\ $\overline{\Q}_l$).

If $(r,N)$ is a Weil--Deligne representation of $W_K$, we will write $(r,N)^{F-\semis}$ for its Frobenius semisimplification.
If $\rho$ is a continuous representation of $G_K$ over $\barQQ_l$ with $l\neq p$ then we will write $\WD(\rho)$ for the corresponding Weil--Deligne representation of $W_K$. (See for instance section 1
of \cite{ty}.) By a {\em Steinberg} representation of $\GL_n(K)$ we will mean a representation $\Spp_n(\psi)$  (in the notation of section 1.3 of \cite{ht}) where $\psi$ is an unramified character of $K^\times$.
If $\pi_i$ is an irreducible smooth representation of $\GL_{n_i}(K)$ for $i=1,2$, we will write $\pi_1 \boxplus \pi_2$ for the irreducible smooth representation of $\GL_{n_1+n_2}(K)$ with $\rec(\pi_1 \boxplus \pi_2)=\rec(\pi_1) \oplus \rec(\pi_2)$.
If $K'/K$ is a finite extension and if $\pi$ is an irreducible smooth representation of $\GL_n(K)$ we will write $\BC_{K'/K}(\pi)$ for the base change of $\pi$ to $K'$ which is characterized by $\rec_{K'}(\BC_{K'/K}(\pi))=
\rec_K(\pi)|_{W_{K'}}$.

If $\rho$ is a  de Rham representation of $G_K$ over $\barQQ_p$, then we will write $\WD(\rho)$ for the corresponding Weil--Deligne representation of $W_K$, and if $\tau:K \into \barQQ_p$ is a
continuous embedding of fields, then we will write $\HT_\tau(\rho)$ for the multiset of Hodge--Tate numbers of $\rho$ with respect to $\tau$. Thus $\HT_\tau(\rho)$ is a multiset of $\dim \rho$ integers. 
In fact if $W$ is a de Rham representation of $G_K$ over $\barQQ_p$ and if $\tau:K \into \barQQ_p$, then the multiset $\HT_\tau(W)$ contains
$i$ with multiplicity $\dim_{\barQQ_l} (W \otimes_{\tau,K} \widehat{\barK}(i))^{G_K} $. Thus, for example,
$\HT_\tau(\epsilon_p)=\{ -1\}$. 

If $G$ is a reductive group over $K$ and $P$ is a parabolic subgroup
with unipotent radical  $N$ and Levi component $L$, and if $\pi$ is a
smooth representation of $L(K)$, then we define $\Ind_{P(K)}^{G(K)}
\pi$ to be the set of locally constant functions $f:G(K) \ra \pi$ such
that $f(hg)=\pi(hN(K)) f(g)$ for all $h \in P(K)$ and $g\in G(K)$. It
is a smooth representation of $G(K)$ where
$(g_1f)(g_2)=f(g_2g_1)$. This is sometimes referred to as `natural' or
`un-normalized' induction. We let $\delta_P$ denote the determinant of
the action of $L$ on $\Lie N$. Then we define the `normalized' or
`unitary' induction $\nInd_{P(K)}^{G(K)} \pi$ to be
$\Ind_{P(K)}^{G(K)} (\pi \otimes |\delta_P|_K^{1/2})$. %
If $P$ is any parabolic in $\GL_{n_1+n_2}$ with Levi component $\GL_{n_1} \times \GL_{n_2}$, then $\pi_1 \boxplus \pi_2$ is a sub-quotient of $\nInd_{P(K)}^{\GL_{n_1+n_2}(K)} \pi_1 \otimes \pi_2$. 

We will let $c$ denote complex conjugation on $\C$. We will write $\Art_\R$ (resp.\ $\Art_\C$) for the unique continuous surjection
\[ \R^\times \onto \Gal(\C/\R) \]
(resp.\ $\C^\times \onto \Gal(\C/\C)$).
We will write $\rec_\C$ (resp.\ $\rec_\R$), or simply $\rec$, for the local Langlands correspondence from irreducible admissible $(\Lie \GL_n(\R) \otimes_\R \C, O(n))$-modules (resp.\ $(\Lie \GL_n(\C) \otimes_\R \C, U(n))$-modules) to continuous, semi-simple $n$-dimensional representations of the Weil group $W_\R$ (resp.\ $W_\C$). (See \cite{langlandsrg}.) If $\pi_i$ is an irreducible admissible $(\Lie \GL_{n_i}(\R) \otimes_\R \C, O(n_i))$-module (resp.\ $(\Lie \GL_{n_i}(\C) \otimes_\R \C, U(n_i))$-module) for $i=1,\dots,r$ and if $n=n_1+\dots+n_r$, then we define an irreducible admissible $(\Lie \GL_n(\R) \otimes_\R \C, O(n))$-module (resp.\ $(\Lie \GL_n(\C) \otimes_\R \C, U(n))$-module) $\pi_1 \boxplus \cdots \boxplus \pi_r$ by
\[ \rec(\pi_1 \boxplus \cdots \boxplus \pi_r) = \rec(\pi_1) \oplus \cdots \oplus \rec(\pi_r). \]
If $\pi$ is an irreducible admissible $(\Lie \GL_{n}(\R) \otimes_\R \C, O(n))$-module, then we define $\BC_{\C/\R}(\pi)$ to be the irreducible admissible $(\Lie \GL_{n}(\C) \otimes_\R \C, U(n))$-module defined by
\[ \rec_\C(\BC_{\C/\R}(\pi))=\rec_\R(\pi)|_{W_\C}. \]%

If $\pi$ is an irreducible admissible representation of
$\GL_n(\A_F)$ and $\xi \in (\Z^n_+)^{\Hom(F, \C)}$, we say that $\pi$
is regular algebraic of weight $\xi$ if the infinitesimal character of $\pi_\infty$ is
the same as that of $V_\xi^\vee$, where~$V_\xi$ is the algebraic
representation of $\Res_{F/\Q}\GL_n$ of highest weight~$\xi$ (see
Section~\ref{sec:unitary_group_setup}). We say that it is regular
algebraic if it is regular algebraic of some weight.

We will write $||\,\,\,||$ for the continuous homomorphism
\[ ||\,\,\,||=\prod_v |\,\,\,|_v: \A^\times/\Q^\times \lra \R^\times_{>0}, \]
where each $|\,\,\,|_v$ has its usual normalization, i.e.\ $|p|_p=1/p$. 

Now suppose that $K/\Q$ is a finite extension. 
We will write $||\,\,\,||_K$ (or simply $||\,\,\,||$) for $||\,\,\, || \circ \norm_{K/\Q}$. We will also write
\[ \Art_K = \prod_v \Art_{K_v}:\A_K^\times /\overline{K^\times (K_\infty^\times)^0} \liso G_K^\ab. \]
If $v$ is a finite place of $K$, we will write $k(v)$ for its residue field, $q_v$ for $\# k(v)$, and $\Frob_v$ for $\Frob_{K_v}$. If $v$ is a real place of $K$, then we will let
$[c_v]$ denote the conjugacy class in $G_K$ consisting of complex conjugations associated to $v$. 
If $K'/K$ is a quadratic extension of number fields, we will denote by $\delta_{K'/K}$ the nontrivial
character of $\A_K^\times/K^\times \norm_{K'/K}\A_{K'}^\times$. (We hope that this will cause no confusion with the Galois character $\delta_{K'/K}$. One equals the composition of the other with the Artin map for $K$.) 
If $K'/K$ is a soluble, finite Galois extension and if $\pi$ is a
cuspidal automorphic representation of $\GL_n(\A_K)$ we will write
$\BC_{K'/K}(\pi)$ for its base change to $K'$, an (isobaric) automorphic representation
of $\GL_n(\A_{K'})$ satisfying 
\[ \BC_{K'/K}(\pi)_v=\BC_{K'_v/K_{v|_K}}(\pi_{v|_K}) \]
for all places $v$ of $K'$. If $\pi_i$ is an automorphic representation of $\GL_{n_i}(\A_K)$ for $i=1,2$, we will write $\pi_1 \boxplus \pi_2$ for the automorphic representation of $\GL_{n_1+n_2}(\A_K)$ satisfying
\[ (\pi_1 \boxplus \pi_2)_v=\pi_{1,v} \boxplus \pi_{2,v} \]
for all places $v$ of $K$.

We will call a number field $K$ a CM field if it has an automorphism
$c$ such that for all embeddings $i:K \into \C$ one has $c \circ i = i
\circ c$. In this case, either $K$ is totally real or a totally
imaginary quadratic extension of a totally real field. In either case,
we will let $K^+$ denote the maximal totally real subfield of $K$.

Suppose that $K$ is a number field and 
\[ \chi: \A_K^\times /K^\times \lra \C^\times \]
is a continuous character. If there exists $a \in \Z^{\Hom(K,\C)}$ such that
\[ \chi|_{(K_\infty^\times)^0}: x \longmapsto \prod_{\tau \in \Hom(K,\C)} (\tau x)^{a_\tau}, \]
we will call $\chi$ algebraic. In this case, we can attach to $\chi$ and a rational prime $l$ and an isomorphism $\imath:\barQQ_l \iso \C$, a unique continuous character
\[ r_{l,\imath}(\chi): G_K \lra \barQQ_l^\times \]
such that for all $v\ndiv l$ we have
\[ \imath \circ r_{l,\imath}(\chi)|_{W_{K_v}} \circ \Art_{K_v} = \chi_v. \] 
There is also an integer $\wt(\chi)$, the weight of $\chi$, such that
\[ |\chi|= ||\,\,\,||_K^{-\wt(\chi)/2}. \]
(See the discussion at the start of \cite[\S A.2]{BLGGT} for more details.)

If $K$ is a totally real field, we call a continuous character 
\[ \chi:\A_K^\times/K^\times \lra \C^\times \]
totally odd if $\chi_v(-1)=-1$ for all $v|\infty$. Similarly, we call a continuous character
\[ \mu:G_K \lra \barQQ_l^\times \]
totally odd if $\mu(c_v)=-1$ for all $v|\infty$.

\section{Preliminaries on the cohomology of locally symmetric spaces and Galois representations}
\label{section:notation}

Our main objects of study in this paper are $n$-dimensional Galois representations and their relation to the cohomology of congruence subgroups of $\GL_n$ (equivalently, the cohomology of the locally symmetric spaces attached to congruence subgroups of $\GL_n$). In this introductory section we establish some basic notation and definitions concerning these objects, and recall some of their fundamental known properties. In particular, we will define cohomology groups associated to an arbitrary weight and level and also define the Hecke algebras which act on these cohomology groups.

\subsection{Arithmetic locally symmetric spaces: generalities}

\subsubsection{Symmetric spaces}\label{sssec:symmetric}  Let $F$ be a number field and let $\mathrm{G}$ be a connected linear algebraic group over $F$. We 
consider a 
space of type $S-\Q$ for $\mathbb{G}:=\mathrm{Res}_{F/\Q}\mathrm{G}$,
in the sense of~\cite[\S 2]{MR0387495} (see also~\cite[\S 3.1]{new-tho}). This 
is a pair consisting of a homogeneous space $X^{\mathrm{G}}$ for 
$\mathbb{G}(\R)$ and a family of Levi subgroups of $\mathbb{G}_{\R}$ 
satisfying 
certain conditions. From~\cite[Lem. 2.1]{MR0387495}, %
the homogeneous space 
$X^{\mathrm{G}}$ is determined up to isomorphism. We will refer to 
$X^{\mathrm{G}}$ as \emph{the symmetric space} for $\mathrm{G}$. For 
example, 
if $\mathrm{G}=\GL_{n, F}$, we can take $X^{\mathrm{G}} = 
\GL_n(F_\infty)/K_\infty \R^\times$ for  $K_\infty\subset \GL_n(F_\infty)$ a 
maximal compact subgroup. 

An open compact subgroup $K_{\mathrm{G}} \subset
\mathrm{G}(\A_F^\infty)$ is said to be \emph{neat} if all of its
elements are neat. An element $g = (g_v)_v \in
\mathrm{G}(\A_F^\infty)$ is said to be neat if the intersection
$\cap_v \Gamma_v$ is trivial, where $\Gamma_v \subset
\overline{\bQ}^\times$ is the torsion subgroup of the subgroup of
$\overline{F}_v^\times$ generated by the eigenvalues of $g_v$ acting via
some faithful representation of $\mathrm{G}$.

We will call a `good subgroup' any neat open compact subgroup $K_{\mathrm{G}} \subset \mathrm{G}(\A_F^\infty)$ of the form $K_G = \prod_v K_{G, v}$, the product running over finite places $v$ of ~$F$. If $K_G$ is a good subgroup, then we define
 \[
 X^{\mathrm{G}}_{K_{\mathrm{G}}}:= \mathrm{G}(F)\backslash\left(
 X^{\mathrm{G}}\times \mathrm{G}(\A^\infty_{F})/K_{\mathrm{G}}\right) \text{ and }
 \mathfrak{X}_{\mathrm{G}}:=\mathrm{G}(F)\backslash\left(X^{\mathrm{G}}\times 
 \mathrm{G}(\A^\infty_{F})\right),
 \]
  the latter with the discrete topology on $\mathrm{G}(\A^\infty_{F})$.
 
These topological spaces may be given the structure of smooth manifolds, and $\mathrm{G}(\A^\infty_F)$ acts on $\mathfrak{X}_{\mathrm{G}}$ by right translation. We can identify $X^{\mathrm{G}}_{K_\mathrm{G}} = \mathfrak{X}_{\mathrm{G}} / K_{\mathrm{G}}$. Note that the space $X^\mathrm{G}$ is diffeomorphic to Euclidean space. The neatness condition on $K_{\mathrm{G}}$ implies that $X^{\mathrm{G}}_{K_{\mathrm{G}}}$ can be identified with a finite disjoint union of quotients of $X^\mathrm{G}$ by the action of torsion-free arithmetic subgroups of $\mathrm{G}(F)$. 

We let $\overline{X}^{\Grm}$ denote the partial Borel--Serre
compactification of $X^{\Grm}$ (see~\cite[\S 7.1]{MR0387495}). Define 
\[
\overline{X}^{\mathrm{G}}_{K_{\mathrm{G}}}:= \mathrm{G}(F)\backslash\left(
\overline{X}^{\mathrm{G}}\times 
\mathrm{G}(\A^\infty_{F})/K_{\mathrm{G}}\right)
\ \mathrm{and}\ 
\overline{\mathfrak{X}}_{\mathrm{G}}:=\mathrm{G}(F)\backslash\left(
\overline{X}^{\mathrm{G}}\times \mathrm{G}(\A^\infty_{F})\right). 
\]
For any good subgroup $K_{\Grm}\subset 
\mathrm{G}(\A_{F}^\infty)$, 
the space $\overline{X}^{\mathrm{G}}_{K_{\mathrm{G}}}$, 
which can be identified with $\overline{\mathfrak{X}}_{\Grm}/K_{\Grm}$, is 
compact (see~\cite[Theorem 9.3]{MR0387495}). More precisely, $\overline{X}^{\Grm}_{K_{\Grm}}$ is a compact
smooth manifold with corners with interior $X^{\Grm}_{K_{\Grm}}$; 
the inclusion $X^{\Grm}_{K_{\Grm}}
\hookrightarrow \overline{X}^{\Grm}_{K_{\Grm}}$ is a homotopy 
equivalence. We also define $\partial X^\Grm = \overline{X}^\Grm - X^\Grm$ and
\[
\partial X^{\mathrm{G}}_{K_{\mathrm{G}}}:= \mathrm{G}(F)\backslash\left(
\partial X^{\mathrm{G}}\times 
\mathrm{G}(\A^\infty_{F})/K_{\mathrm{G}}\right)
\ \mathrm{and}\ 
\partial\mathfrak{X}_{\mathrm{G}}:=\mathrm{G}(F)\backslash\left(
\partial X^{\mathrm{G}}\times \mathrm{G}(\A^\infty_{F})\right). 
\]

\subsubsection{Hecke operators and coefficient systems}\label{sec:general_definition_of_hecke_operators}
If $S$ is a finite set of finite places of $F$ 
we set $\Grm^S := \mathrm{G}(\A^{\infty,S}_{F})$ and $\Grm_S := 
\mathrm{G}(\A_{F, S})$, and similarly $K_\Grm^S = \prod_{v \not\in S} K_{\Grm, v}$ and $K_{\Grm, S} = \prod_{v \in S} K_{\Grm, v}$. We also sometimes write $\Grm^\infty = \Grm(\A_F^\infty)$.

Let $R$ be a ring and let $\cV$ be an $R[\Grm(F) \times K_{\Grm, S}]$-module, finite free as $R$-module. We now explain how to obtain a local system of finite free $R$-modules, also denoted $\cV$, on 
$X^{\mathrm{G}}_{K_{\mathrm{G}}}$, 
and how to equip the complex $R\Gamma(X^{\mathrm{G}}_{K_{\mathrm{G}}}, \cV) \in \mathbf{D}(R)$ with an action 
of the Hecke algebra $\cH(\Grm^S, K_G^S)$, following the formalism of 
\cite{new-tho} (in particular, viewing $X^\Grm \times \Grm(\A_F^\infty)$ as a right $\Grm(F) \times \Grm(\A_F^\infty)$-space). 

 The $R[\Grm(F) \times K_{\mathrm{G}, S}]$-module 
$\cV$ 
determines (by pullback from a point) a 
$\Grm(F) \times \mathrm{G}^S\times 
K_{\mathrm{G},S}$-equivariant sheaf, also denoted $\cV$, of 
finite free $R$-modules on 
$X^\Grm \times \Grm(\A_F^\infty)$, hence (by descent under a free action, as in ~\cite[Lem. 
2.17]{new-tho}) a  $\mathrm{G}^S\times 
K_{\mathrm{G},S}$-equivariant sheaf $\cV$ on $\mathfrak{X}_\Grm$.
 By taking derived global sections we obtain 
$R\Gamma(\mathfrak{X}_\Grm, \cV)$, which is an object
of the derived category of $R[\mathrm{G}^S\times 
K_{\mathrm{G},S}]$-modules. By taking derived invariants under $K_{\mathrm{G}}$ 
we obtain $R\Gamma(K_{\mathrm{G}}, R\Gamma(\mathfrak{X}_{\mathrm{G}}, 
\cV_{\mathfrak{X}_{\mathrm{G}}}))$, which is an object of the derived category 
of $\cH(\Grm^S, K_{\Grm}^S) \otimes_\Z R$-modules. 

On the other hand, if we only think of $\cV$ as a 
$K_{\mathrm{G}}$-equivariant sheaf on $\mathfrak{X}_{\mathrm{G}}$, it is 
equivalent to a sheaf $\cV$ on $X^{\mathrm{G}}_{K_{\Grm}}$ (applying once again \cite[Lem. 
2.17]{new-tho}). 
The complex $R\Gamma(X^{\mathrm{G}}_{K_{\Grm}}, \cV)$ 
is naturally isomorphic in $\mathbf{D}(R)$ to the image of the complex $R\Gamma(K_{\mathrm{G}}, 
R\Gamma(\mathfrak{X}_{\mathrm{G}}, \cV_{\mathfrak{X}_{\mathrm{G}}}))$ under the exact forgetful functor
\[ \mathbf{D}(\cH(\Grm^S, K_{\Grm}^S) \otimes_\Z R) \to \mathbf{D}(R), \]
cf.~\cite[Prop. 2.18]{new-tho}. In this way, we obtain a canonical homomorphism
\numequation\label{eqn:hecke_morphism}
\cH(\Grm^S, K_{\Grm}^S)\otimes_\Z R \to \End_{\mathbf{D}(R)}(R\Gamma(X^{\mathrm{G}}_{K_{\Grm}}, \cV)).
\end{equation}
The same formalism applies equally well to the Borel--Serre compactification (because $G(F) \times K_\mathrm{G}$ acts freely on $\overline{X}^\mathrm{G} \times \Grm(\A_F^\infty)$). Even more generally, if $Y$ is any right $\Grm^S \times K_{\Grm, S}$-space and $C$ is any bounded-below complex of $\Grm^S \times K_{\Grm, S}$-equivariant sheaves of $R$-modules on $Y$, there is a homomorphism 
\[ \cH(\Grm^S, K_{\Grm}^S)\otimes_\Z R \to \End_{\mathbf{D}(R)}(R\Gamma(K_{\Grm}, R \Gamma(Y, C)  )).\]
Taking $j : X^\Grm \times \Grm(\A_F^\infty) \to \overline{X}^\Grm \times \Grm(\A_F^\infty)$ to be the canonical open immersion and $\cV$ to be an $R[\Grm(F) \times K_{\Grm, S}]$-module, finite free as $R$-module, this determines an action of the Hecke algebra on the cohomology groups with compact support:
\begin{multline*}  \cH(\Grm^S, K_{\Grm}^S)\otimes_\Z R \to \End_{\mathbf{D}(R)}(R\Gamma(\Grm(F) \times  K_{\Grm}, R \Gamma(\overline{X}^{\Grm} \times \Grm(\A_F^\infty), j_! \cV)  )) \\
= \End_{\mathbf{D}(R)}(R \Gamma_c(X^{\Grm}_{K_{\Grm}}, \cV)). 
\end{multline*}
We have the following lemma, which is a consequence of the existence of the Borel--Serre compactification (see~\cite[\S 11]{MR0387495}):
\begin{lemma}\label{lem:finite_generation_of_cohomology}
	Let $K_\Grm$ be a good subgroup, let $R$ be a Noetherian ring, and let $\cV$ be an $R[\Grm(F) \times K_\Grm]$-module, finite free as $R$-module. Then $H^\ast(X^\Grm_{K_\Grm}, \cV)$ is a finitely generated $R$-module.
\end{lemma}
A variant of this construction arises when we are given a normal good subgroup $K_{\Grm}' \subset K_{\Grm}$ with the property that $K_{\Grm}^S = (K'_{\Grm})^S$. Then we write $R \Gamma_{K_{\Grm} / K'_{\Grm} }(X^\Grm_{K'_\Grm}, \cV) \in \mathbf{D}(R[K_{\Grm} / K_{\Grm}'])$ for the complex in this category computing the cohomology of $H^\ast(X^\Grm_{K'_\Grm}, \cV)$ with its natural $K_{\Grm} / K'_{\Grm} = K_{\Grm, S} / K'_{\Grm, S}$-action. The image of this complex under the forgetful functor $\mathbf{D}(R[K_{\Grm} / K_{\Grm}']) \to \mathbf{D}(R)$ is $R \Gamma(X^\Grm_{K'_\Grm}, \cV)$, and there is a homomorphism
\numequation\label{eqn:hecke_morphism_with_coefficients}
\cH(\Grm^S, K_{\Grm}^S)\otimes_\Z R \to \End_{\mathbf{D}(R[K_{\Grm} / K_{\Grm}'])}(R\Gamma_{K_{\Grm} / K_{\Grm}'}(X^{\mathrm{G}}_{K_{\Grm}'}, \cV))
\end{equation}
which recovers (\ref{eqn:hecke_morphism}) after composition with the map
\numequation
\End_{\mathbf{D}(R[K_{\Grm} / K_{\Grm}'])}(R\Gamma_{K_{\Grm} / K_{\Grm}'}(X^{\mathrm{G}}_{K_{\Grm}'}, \cV)) \to \End_{\mathbf{D}(R)}(R\Gamma(X^{\mathrm{G}}_{K_{\Grm}}, \cV))
\end{equation}
given by the functor $R \Gamma(K_{\Grm} / K_{\Grm}', ?)$. 

The following lemma is a strengthening of Lemma \ref{lem:finite_generation_of_cohomology}:
\begin{lemma}\label{lem:cohomology_is_perfect}
	Let $K_\Grm$ be a good subgroup, and let $K'_\Grm \subset K_\Grm$ be a normal subgroup which is also good. Let $R$ be a Noetherian ring, and let $\cV$ be an $R[\Grm(F) \times K_\Grm]$-module, finite free as $R$-module. Then $R \Gamma_{K_{\Grm} / K_{\Grm}'}( X^\Grm_{K_\Grm'}, \cV)$ is a perfect object of $\mathbf{D}(R[K_\Grm / K_\Grm'])$; in other words, it is isomorphic in this category to a bounded complex of projective $R[K_\Grm / K_\Grm']$-modules.
\end{lemma}
\begin{proof}
	Pullback induces an isomorphism $R \Gamma_{K_{\Grm} / K_{\Grm}'} (\overline{X}^\Grm_{K'_\Grm}, \cV) \to R \Gamma_{K_{\Grm} / K_{\Grm}'}( X^\Grm_{K_\Grm'}, \cV)$, so it suffices to show that $R \Gamma_{K_{\Grm} / K_{\Grm}'} (\overline{X}^\Grm_{K'_\Grm}, \cV)$ is a perfect complex. As in~\cite[\S 11]{MR0387495}, we see that $\overline{X}^\Grm_{K_\Grm}$ admits a finite triangulation; this pulls back to a $\Grm(F) \times K_\Grm$-invariant triangulation of $\overline{X}^\Grm \times \Grm(\A_F^\infty)$. Let $C_\bullet$ be the corresponding complex of simplicial chains. It is a bounded complex of finite free $\Z[\Grm(F) \times K_\Grm]$-modules. The lemma now follows on observing that $R \Gamma_{K_{\Grm} / K_{\Grm}'} (\overline{X}^\Grm_{K'_\Grm}, \cV)$ is isomorphic in $\mathbf{D}(R[K_\Grm / K_\Grm'])$ to the complex $\Hom_{ \Z[\Grm(F) \times K'_\Grm] }( C_\bullet, \cV )$. 
\end{proof}
Finally we introduce some notation relevant for relating the Hecke operators of $\Grm$ and of its parabolic subgroups. Let us therefore now assume that $\Grm$ is reductive, and let $\Prm = \mathrm{M} \mathrm{N}$ be a parabolic subgroup with Levi subgroup $\mathrm{M}$. Let $K_\Grm \subset \Grm(\A_F^\infty)$ be a good subgroup. In this situation, we define $K_\Prm = K_\Grm \cap \Prm(\A_F^\infty)$, $K_\mathrm{N} = K_\Grm \cap \mathrm{N}(\A_F^\infty)$, and define $K_\mathrm{M}$ to be the image of $K_\Prm$ in $\mathrm{M}(\A_F^\infty)$. We say that $K_\Grm$ is decomposed with respect to $\Prm = \mathrm{M} \mathrm{N}$ if we have $K_\Prm = K_\mathrm{M} \ltimes K_{\mathrm{N}}$; equivalently, if $K_\mathrm{M} = K_\Grm \cap \mathrm{M}(\A_F^\infty)$. 

Assume now that $K_\Grm$ is decomposed with respect to $\Prm = \mathrm{M} \mathrm{N}$, and let $S$ be a finite set of finite places of $F$ such that for all $v \not\in S$, $K_{\Grm, v}$ is a hyperspecial maximal compact subgroup of $\Grm(F_v)$. In this case, we can define homomorphisms 
\[ r_\Prm : \cH(\Grm^S, K_{\Grm}^S) \to \cH(\Prm^S, K_\Prm^S) \text{ and } r_{\mathrm{M}} : \cH(\Prm^S, K_\Prm^S) \to \cH(\mathrm{M}^S, K_\mathrm{M}^S), \]
given respectively by ``restriction to $\Prm$'' and ``integration along $\mathrm{N}$''; see \cite[\S 2.2.3, 2.2.4]{new-tho} for the definitions of these maps, along with the proofs that they are indeed algebra homomorphisms and that integration along $\mathrm{N}$ preserves integrality. We write
\numequation\label{eqn:satake}
\cS = r_\mathrm{M} \circ r_\Prm
\end{equation}
 for the composite map, or $\cS = \cS^\Grm_\mathrm{M}$ when we wish to emphasize the ambient groups. By abuse of notation, we also denote by $r_{\Prm}$, $r_{\mathrm{M}}$ and $\cS= \cS^\Grm_\mathrm{M}$ the same maps for the local Hecke algebras at $v\not \in S$. 

\subsubsection{The Hecke algebra of a monoid}\label{sec:hecke_algebra_of_a_monoid}

We in fact need a slight generalization of the discussion in the previous section, which we outline now in a similar way to~\cite[\S 2.2]{new-tho}. 

We first discuss the local situation. Let $F$ be a non-archimedean local field, and let $G$ be a reductive group over $F$. Let $q$ denote the cardinality of the residue field of $F$. If $U \subset G(F)$ is an open compact subgroup and $\Delta \subset G(F)$ is an open submonoid which is invariant under left and right multiplication by elements of $U$, then we can consider the subset $\cH(\Delta, U) \subset \cH(G(F), U)$ of functions $f : G(F) \to \mathbb{Z}$ which are supported in $\Delta$. It follows from the definition of the convolution product that this subset is in fact a subalgebra. If $R$ is a ring and $M$ is an $R[\Delta]$-module (or more generally, a complex of $R[\Delta]$-modules) then there is a corresponding homomorphism $\cH(\Delta, U) \to \End_{\mathbf{D}(R)}(R \Gamma(U, M))$. This extends the formalism for the full Hecke algebra described in~\cite[2.2.5]{new-tho} and recalled in the previous section.

Now let $P \subset G$ be a parabolic subgroup with Levi decomposition $P = MN$, and let $\overline{P} = M \overline{N}$ denote the opposite parabolic. Let $U \subset G$ be an open compact subgroup which admits an Iwahori decomposition with respect to $P$. By definition, this means that if we define $U_N = U \cap N(F)$, $U_M =  U \cap M(F)$, and $U_{\overline{N}} = U \cap \overline{N}(F)$, then the two product maps
\[ U_N \times U_M \times U_{\overline{N}} \to U 
\text{ and }
U_{\overline{N}} \times U_M \times U_N \to U \]
are bijective. In this case, we write $\Delta_M \subset M(F)$ for the set of $U$-positive elements, i.e.\ those $t \in M(F)$ which satisfy $t U_N t^{-1} \subset U_N$ and $U_{\overline{N}} \subset t U_{\overline{N}} t^{-1}$. We define $\Delta = U_N \Delta_M U_{\overline{N}}$.
\begin{lemma}
	$\Delta_M$ and $\Delta$ are monoids. Moreover, $\Delta_M$ is open in $M(F)$, $\Delta$ is open in $G(F)$, we have $U \Delta U = \Delta$, and $\Delta \cap M(F) = \Delta_M$.
\end{lemma}
\begin{proof}
	It is clear from the definition that $\Delta_M$ is closed under multiplication, and also that $\Delta_M$, $\Delta$ are open in $M(F)$ and $G(F)$, respectively. To show that $U \Delta U = \Delta$, we simply observe that if $m \in \Delta_M$, then the definition of positivity gives
	\[ \begin{split} U m U& = U m U_N U_M U_{\overline{N}} =  U m U_M U_{\overline{N}} = U_N U_M U_{\overline{N}} m U_M U_{\overline{N}} \\ &= U_N U_M m U_M U_{\overline{N}} \subset U_N \Delta_M U_{\overline{N}} = \Delta. \end{split} \]
	To show that $\Delta$ is closed under multiplication, we must show that $U m_1 U m_2 U \subset U \Delta_M U$. Using the definition of positivity, we see that
	\[ U m_1 U m_2 U = U m_1 U_N U_M U_{\overline{N}} m_2 U = U m_1 U_M m_2 U, \]
	so it is equivalent to show $m_1 U_M m_2 \subset \Delta_M$; and this is true, since $U_M \subset \Delta_M$. Finally, the identity $\Delta \cap M(F) = \Delta_M$ follows from the following observation: if $u_1t \bar{u}_2 = m\in M(F)$, for $u_1\in U_{N}$, $t\in \Delta_M$ and $\bar{u}_2 \in U_{\overline{N}}$, then $\bar{u}_2 = t^{-1}u_1^{-1}m \in P(F)\cap \overline{N}(F)$, so $\bar{u}_2$ must be the identity. Similarly, $u_1$ must be the identity, so $m=t \in \Delta_M$. 
\end{proof}
It follows that the Hecke algebras $\cH(\Delta, U)$ and $\cH(\Delta_M, U_M)$ are defined. Moreover, $\Delta_P = \Delta \cap P(F)$ is a monoid, and we can consider also the Hecke algebra $\cH(\Delta_P, U_P)$.
\begin{lemma}
	Consider the two maps $r_P : \cH(\Delta, U) \to \cH(\Delta_P, U_P)$ and $r_M : \cH(\Delta_P, U_P) \to \cH(\Delta_M, U_M)$ given by restriction to $P(F)$ and integration along $U_N$, respectively. Then both $r_P$ and $r_M$ are algebra homomorphisms.
\end{lemma}
\begin{proof}
	It follows from \cite[Lemma 2.7]{new-tho} that the map $\cH(P(F), U_P) \to \cH(M(F), U_M)$ is an algebra homomorphism whenever the condition $U_P = U_N \rtimes U_M$ is satisfied. It remains to show that $r_P$ is an algebra homomorphism. The proof is the same as the proof of \cite[Lemma 2.4, 1.]{new-tho} once we take into account the identity, valid for any function $f : G(F) \to \R$ with compact support contained in $U P(F)$ (and a fortiori, any function $f \in \cH(\Delta, U)$):
	\[ \int_{g \in G(F)} f(g) \, dg = \int_{u \in U} \int_{p \in P(F)} f(pu) \, dp \, du.\qedhere \]
\end{proof}
It will be helpful later to note that the maps $r_P$ and $r_M \circ r_P$ are quite simple, being given on basis elements by the formulae $r_P([U m U]) = [U_P m U_P]$ and $r_M \circ r_P ([U m U]) = \# (U_N / m U_N m^{-1}) [U_M m U_M] = | \delta_P(m) |_F^{-1} [U_M m U_M]$, respectively. As in the unramified case, we will write $\cS$ or $\cS^G_M$ for the composite $r_M \circ r_P$. 
\begin{lemma}
Consider the map $t: \cH(\Delta_M, U_M) \to \cH(\Delta, U)$ of $\bZ$-modules given on basis elements by $t([U_M m U_M]) = [U m U]$. Then $t$ is an algebra homomorphism.
\end{lemma}
\begin{proof}
This is \cite[Corollary 6.12]{MR1643417}. 
\end{proof}
Thus we have constructed injective algebra homomorphisms 
\[ t : \cH(\Delta_M, U_M)  \to \cH(\Delta, U) \]
\[ \cS  : \cH(\Delta, U) \to \cH(\Delta_M, U_M) \]
with the property that for any $m \in \Delta_M$, $t \circ \cS ([U m U]) =  | \delta_P(m) |_F^{-1} [UmU] $ and $\cS \circ t ([U_M m U_M]=  | \delta_P(m) |_F^{-1} [U_M m U_M]$. In certain circumstances, we can extend the domain of definition of these homomorphisms. Following \cite{MR1643417}, we say that an element $z \in \Delta_M$ which lies in the centre of $M$ is strongly positive if for any open compact subgroups $H_1, H_2$ of $U_N$ (resp. $\overline{H}_1, \overline{H}_2$ of $U_{\overline{N}}$), there exists $n \geq 0$ such that $z^n H_1 z^{-n} \subset H_2$ (resp. $z^{-n} \overline{H}_1 z^n \subset \overline{H}_2$). 
\begin{lemma}\label{lem:strongly_positive_elements} Let $z \in \Delta_M$ be strongly positive. Then:
\begin{enumerate}
\item $[U_M z U_M]$ lies in the centre of $\cH(\Delta_M, U_M)$, is invertible in $\cH(M(F), U_M)$, and $\cH(\Delta_M, U_M)[ [U_M z U_M]^{-1} ] = \cH(M(F), U_M)$. 
\item Let $R$ be a ring in which $q$ is a unit, and suppose that $[U z U]$ is invertible in $\cH(G(F), U) \otimes_\bZ R$. Then $t\otimes_\bZ R$ and $\cS \otimes_\bZ R$ extend uniquely to algebra isomorphisms
\[ t :  \cH(M(F), U_M) \otimes_\bZ R \to (\cH(\Delta, U) \otimes_\bZ R)[ [U z U]^{-1} ] \]
and 
\[ \cS : (\cH(\Delta, U) \otimes_\bZ R)[ [U z U]^{-1} ] \to \cH(M(F), U_M)\otimes_\bZ R.  \]
\end{enumerate}
\end{lemma}
\begin{proof}
The element $[U_M z U_M]$ lies in the centre of $\cH(\Delta_M, U_M)$ because $z$ lies in the centre of $M(F)$, by assumption. Its inverse is $[U_M z^{-1} U_M]$. The equality $\cH(\Delta_M, U_M)[ [U_M z U_M]^{-1} ] = \cH(M(F), U_M)$ holds because for any $m \in M(F)$, there exists $n \geq 0$ such that $z^n m \in \Delta_M$, hence $[U_M m U_M] = [ U_M z U_M]^{-n} [U_M z^n m U_M] \in \cH(\Delta_M, U_M)[ [U_M z U_M]^{-1} ]$. This shows the first part. The second part is elementary. 
\end{proof}
\begin{lemma}\label{lem:functorial_splitting_for_induced_module}
	Let $R$ be a ring, let $W$ be an $R[P(F)]$-module, and let $V = \Ind_{P(F)}^{G(F)} W$. Then there is a natural morphism $\phi : V^U \to r_P^\ast W^{U_P}$ of $\cH(\Delta, U) \otimes_\Z R$-modules. Moreover, writing $(?)^\sim$ for the forgetful functor from $\cH(\Delta, U) \otimes_\Z R$-modules to $R$-modules, the induced morphism $(V^U)^\sim \to (r_P^\ast W^{U_P})^\sim$ has a functorial splitting.
\end{lemma}
\begin{proof}
	Let $g_1, \dots, g_n \in G(F)$ be representatives for the double quotient $P(F) \backslash G(F) / U$; we assume that $g_1 = 1$. Then there is an isomorphism of $R$-modules $V^U \cong \oplus_{i=1}^n W^{g_i U g_i^{-1} \cap P(F)}$, which sends a function $f \in V^U$ to the tuple $(f(g_1), \dots, f(g_n))$. This is the desired functorial splitting. We claim that the map $V^U \to W^{U_P}$ corresponding given by projection to the first component is in fact Hecke equivariant (with respect to $r_P$). To see this, choose $f \in V^U$, and let $v = f(1)$, $m \in \Delta_M$. We calculate
	\[ \begin{split} ([U m U] \cdot f)(1) & = \int_{g \in U m U} f(g) \, dg = \int_{p \in P(F)} \int_{u \in U} \mathbf{1}_{pu \in U m U} f(pu) \, dp \, du \\ &= \int_{p \in P(F) \cap U m U} f(p) \, dp = [ U_P m U_P] \cdot f(1),\end{split} \]
	as required.
\end{proof}
We now describe how we will apply the above discussion in the global situation. Let $F$ now denote a number field, let $\mathrm{G}$ be a reductive group over $F$, and let $\mathrm{P} \subset \mathrm{G}$ be a parabolic subgroup with Levi decomposition $\mathrm{P} = \mathrm{M} \mathrm{N}$. Let $K_{\mathrm{G}} \subset \mathrm{G}(\mathbb{A}_F^\infty)$ be a good subgroup of the form $K = K_{\Grm, S} K_{\Grm, T} K_\Grm^{T \cup S}$, notation and assumptions being as follows:
\begin{enumerate}
	\item $T, S$ are finite disjoint sets of finite places of $F$.
	\item For each place $v \not\in S \cup T$ of $F$, $\Grm_{F_v}$ is unramified and $K_v$ is a hyperspecial maximal subgroup of $\Grm(F_v)$.
	\item For each place $v \in T$, $K_{\Grm, v}$ admits an Iwahori decomposition with respect to $\mathrm{P}$. We write $\Delta_{\Grm, v} \subset \Grm(F_v)$ for the corresponding open submonoid and $\Delta_{\Grm, T} = \prod_{v \in T} \Delta_{\Grm, v}$. We define $\Delta_{\mathrm{P}, T}$ and $\Delta_{\mathrm{M}, T}$ similarly.
\end{enumerate}
We thus have a map
\[ \cS : \cH( \mathrm{G}^{S \cup T} \times \Delta_{\Grm, T}, K^S_{\mathrm{G}} ) \to \cH( \mathrm{M}^{S \cup T} \times \Delta_{\mathrm{M}, T}, K^S_{\mathrm{M}} ) . \]
Let $R$ be a ring. Applying Lemma \ref{lem:functorial_splitting_for_induced_module} (cf. \cite[Corollary 2.6]{new-tho}), we see that there is a split morphism in $\mathbf{D}(R)$
\[ R \Gamma([\Ind_{\mathrm{P}^\infty}^{\mathrm{G}^\infty} \mathfrak{X}_{\mathrm{P}}] / K_{\mathrm{G}}, R ) \to R \Gamma(X^{\mathrm{P}}_{K_{\mathrm{P}}}, R), \]
which is equivariant for the action of  $\cH( \mathrm{G}^{S \cup T} \times \Delta_{\Grm, T}, K^S_{\mathrm{G}} ) \otimes_\mathbb{Z} R$ by endomorphisms on the source and target (the latter action being via the map $r_{\mathrm{P}}$, and induction being in the same sense as in \cite[\S 3.1]{new-tho}). The splitting need not be equivariant, but we see that in any case there is a surjective morphism
\[ H^\ast([\Ind_{\mathrm{P}^\infty}^{\mathrm{G}^\infty} \mathfrak{X}_{\mathrm{P}}] / K_{\mathrm{G}}, R) \to r_P^\ast H^\ast(X^{\mathrm{P}}_{K_{\mathrm{P}}}, R)  \]
of $\cH( \mathrm{G}^{S \cup T} \times \Delta_{\Grm, T}, K^S_{\mathrm{G}} ) \otimes_\mathbb{Z} R$-modules. Similarly  \cite[Proposition 3.4]{new-tho} shows that there is a split morphism in $\mathbf{D}(R)$
\[  R \Gamma(X^{\mathrm{M}}_{K_{\mathrm{M}}}, R) \to R \Gamma(X^{\mathrm{P}}_{K_{\mathrm{P}}}, R), \]
which is equivariant for the action of $\cH( \mathrm{P}^{S \cup T} \times \Delta_{\mathrm{P}, T}, K^S_{\mathrm{P}} ) \otimes_\mathbb{Z} R$ by endomorphisms on the source and target (the action on the source being via the map $r_{\mathrm{M}}$). Altogether there is no $\cS$-equivariant map between the complexes $R \Gamma([\Ind_{\mathrm{P}^\infty}^{\mathrm{G}^\infty} \mathfrak{X}_{\mathrm{P}}] / K_{\mathrm{G}}, R )$ and $R \Gamma(X^{\mathrm{M}}_{K_{\mathrm{M}}}, R)$, these morphisms considered above will together allow us, in the course of proving Theorem \ref{thm:Hecke_reduction_to_Siegel_with_R_ramification} below, to show that $\cS$ descends to a map between the Hecke algebras which act faithfully on these complexes. Moreover, in the presence of invertible strongly positive elements as in the statement of Lemma \ref{lem:strongly_positive_elements}, we will be able to show that this induced map on Hecke algebras is compatible with localisation.

\subsection{Arithmetic locally symmetric spaces: the quasi-split unitary group}

\subsubsection{The quasi-split unitary group, the Siegel parabolic, and its Levi subgroup}\label{sec:unitary_group_setup}

We now specialize the above discussion to our case of interest. We fix an integer $n \geq 1$. Let $F$ be an (imaginary) CM number field with maximal totally real subfield $F^+$. Let $\Psi_n$ be the matrix with 1's on the anti-diagonal and 0's elsewhere, and set
\[ J_n = \left( \begin{array}{cc} 0 & \Psi_n \\ -\Psi_n & 0 \end{array}\right). \]
We write $\widetilde{G}_n = \widetilde{G}$ for the group scheme over $\cO_{F^+}$ with functor of points
\[ \widetilde{G}(R) = \{ g \in \GL_{2n}( R \otimes_{\cO_{F^+}} \cO_F)
  \mid {}^t g J_n g^c  = J_n \}. \]
Then $\widetilde{G}_{F^+}$ is a quasi-split reductive group over $F^+$; it is a form of $\GL_{2n}$ which becomes split after the quadratic base change $F / F^+$. If $\overline{v}$ is a place of $F^+$ which splits in $F$, then a choice of place $v | \overline{v}$ of $F$ determines a canonical isomorphism $\iota_v : \widetilde{G}(F^+_{\overline{v}}) \cong \GL_{2n}(F_v)$. Indeed, there is an isomorphism $F^+_{\overline{v}} \otimes_{F^+} F \cong F_v \times F_{v^c}$ and $\iota_v$ is given by the natural inclusion $\widetilde{G}(F^+_{\overline{v}}) \subset \GL_{2n}(F_v) \times \GL_{2n}(F_{v^c})$ followed by projection to the first factor. 

We write $T \subset B \subset \widetilde{G}$ for the subgroups consisting, respectively, of the diagonal and upper-triangular matrices in $\widetilde{G}$. Similarly we write $G \subset P \subset \widetilde{G}$ for the subgroups consisting, respectively, of the block upper diagonal and block upper-triangular matrices with blocks of size $n \times n$. Then $P = U \rtimes G$, where $U$ is the unipotent radical of $P$, and we can identify $G$ with $\Res_{\cO_F / \cO_{F^+}} \GL_n$ via the map
\[ g = \left( \begin{array}{cc} A & 0 \\ 0 & D \end{array}\right) \mapsto D \in \GL_n(R \otimes_{\cO_{F^+}} \cO_F). \]
We observe that after extending scalars to $F^+$, $T$ and $B$ form a maximal torus and a Borel subgroup, respectively, of $\widetilde{G}$, and $G$ is the unique Levi subgroup of the parabolic subgroup $P$ of $\widetilde{G}$ containing $T$.

In order to simplify notation, we now write $\widetilde{X} = X^{\widetilde{G}}$ and $X = X^G$. Similarly, we will use the symbols $\widetilde{K}$ and $K$ to denote good subgroups of $\widetilde{G}(\A_{F^+}^\infty)$ and $G(\A_{F^+}^\infty) = \GL_n(\A_F^\infty)$, respectively.

We note that the dimensions of these symmetric spaces are
\[\dim_{\R} \widetilde{X} = 2n^2[F^+:\Q], \quad\quad \dim_{\R} X = n^2[F^+:\Q]-1.\]

We now want to describe some explicit (rational and integral) coefficient systems for these symmetric spaces. The integral coefficient systems we define will depend on a choice of a prime $p$ and a dominant weight for either $G$ or $\widetilde{G}$. We therefore fix a prime $p$ and a finite extension $E / \Q_p$ in $\overline{\Q}_p$ which contains the images of all embeddings $F \hookrightarrow \overline{\Q}_p$. We write $\cO$ for the ring of integers of~ $E$, and $\varpi \in \cO$ for a choice of uniformizer. 

We first treat the case of $G$. Let $\Omega$ be a field of characteristic 0 and large enough such that $\Hom(F, \Omega)$ has $[F : \Q]$ elements. We identify the character group $X^\ast( (\Res_{F^+ / \Q} T)_\Omega )$ with $(\Z^n)^{\Hom(F, \Omega)}$ in the usual way, by identifying 
\[ (\Res_{F / \Q} \GL_n)_\Omega = \prod_{\tau \in \Hom(F, \Omega)} \GL_n \]
 and by identifying $(\lambda_1, \dots,\lambda_n) \in \Z^n$ with the character
\[ \diag(t_1, \dots, t_n) \mapsto t_1^{\lambda_1} \dots t_n^{\lambda_n} \]
of the diagonal maximal torus in $\GL_n$. The $\Res_{F^+ / \Q} (B \cap G)_\Omega$-dominant weights are exactly those in the subset $(\Z^n_+)^{\Hom(F, \Omega)}$ given by those tuples $(\lambda_{\tau, i})$ satisfying the condition
\[ \lambda_{\tau, 1} \geq \lambda_{\tau, 2} \geq \dots \geq \lambda_{\tau, n} \]
for each $\tau \in \Hom(F, \Omega)$. 

Associated to $\lambda$ we have the algebraic representation $V_\lambda$ of $(\Res_{F/\mathbb Q} \GL_n)_\Omega$ of highest weight $\lambda$. We may identify $V_\lambda = \otimes_{\tau \in \Hom(F, \Omega)} V_{\lambda_\tau}$, where $V_{\lambda_\tau}$ is the irreducible representation of $\GL_{n, \Omega}$ of highest weight $\lambda_\tau$. If $\lambda \in (\Z^n_+)^{\Hom(F, \Omega)}$, we define $\lambda^\vee  \in (\Z^n_+)^{\Hom(F, \Omega)}$ by the formula $\lambda^\vee_{\tau, i} = -\lambda_{\tau, n+1-i}$. Then there is an isomorphism $V_\lambda^\vee \cong V_{\lambda^\vee}$, although this is not true for the integral lattices defined below without further hypotheses on $\lambda$.

Now take $\Omega = E$. For each $\tau\in \Hom(F,E)$, we let $\cV_{\lambda_\tau} \subset V_{\lambda_\tau}$ be the $\GL_n(\cO)$-invariant $\cO$-lattice defined in~\cite[\S 2.2]{ger} (and called $M_{\lambda_\tau}$ there). We note that this is the integral dual Weyl module of highest weight $\lambda_{\tau}$, obtained by evaluating an algebraic induction on $\cO$. (Geometrically, the dual Weyl module of highest weight $\lambda_\tau$ is obtained as in the Borel--Weil theorem, by taking global sections of a line bundle determined by 
$\lambda_{\tau}$ over the full flag variety associated to $\GL_n$.) We write $\cV_\lambda = \otimes_{\tau \in \Hom(F, E)} \cV_{\lambda_\tau}$ for the corresponding $\cO$-lattice in $V_\lambda$. Thus $\cV_\lambda$ is an $\cO[\prod_{v | p} \GL_n(\cO_{F_v})]$-module, finite free as $\cO$-module.

We next treat the case of $\widetilde{G}$. Let $\widetilde{I} \subset \Hom(F, \Omega)$ be a subset such that $\Hom(F, \Omega) = \widetilde{I} \sqcup \widetilde{I}^c$. Given $\tau \in \Hom(F^+, E)$, we will sometimes write $\widetilde{\tau}$ for the unique element of $\widetilde{I}$ extending $\tau$. The choice of $\widetilde{I}$ determines an isomorphism
\[ ( \Res_{F^+ / \Q} \widetilde{G})_\Omega \cong  \prod_{\tau \in \Hom(F^+, \Omega)} \GL_{2n, \Omega} \]
taking $( \Res_{F^+ / \Q} T)_\Omega$ to the product of the diagonal maximal tori in the $\GL_{2n}$'s, and hence 
an identification of the character group $X^\ast( (\Res_{F^+ / \Q} T)_\Omega )$ with $(\Z^{2n})^{\Hom(F^+, \Omega)}$. The $(\Res_{F^+ / \Q} B)_\Omega$-dominant weights are exactly those in the subset $(\Z^{2n}_+)^{\Hom(F^+, \Omega)}$. The isomorphism $(\Z^n)^{\Hom(F, \Omega)} \cong (\Z^{2n})^{\Hom(F^+, \Omega)}$ identifies a weight $\lambda$ with the weight $\widetilde{\lambda} = (\widetilde{\lambda}_{\tau, i})$ given by the formula
\numequation\label{eqn:wt_dictionary} \widetilde{\lambda}_{\tau} = (-\lambda_{\widetilde{\tau}c, n}, \dots, -\lambda_{\widetilde{\tau}c, 1}, \lambda_{\widetilde{\tau}, 1}, \lambda_{\tau, 2}, \dots, \lambda_{\widetilde{\tau}, n} ). 
\end{equation}
Now let $\Omega = E$. We define integral structures under the assumption that each place $\overline{v}$ of $F^+$ above $p$ splits in $F$. Let $S_p$ denote the set of $p$-adic places of $F$, and let $\overline{S}_p$ denote the set of $p$-adic places of $F^+$. Let $\widetilde{S}_p \subset S_p$ be a subset such that $S_p = \widetilde{S}_p \sqcup \widetilde{S}_p^c$. Let $\widetilde{I} = \widetilde{I}_p$ denote the set of embeddings $\tau : F \hookrightarrow E$ inducing an element of $\widetilde{S}_p$. Given $\overline{v} \in \overline{S}_p$, we will sometimes write $\widetilde{v}$ for the unique element of $\widetilde{S}_p$ lying above $\overline{v}$.

The choice of $\widetilde{S}_p$ determines isomorphisms
\[ \widetilde{G} \otimes_{\cO_{F^+}} \cO_{F^+, p} \cong \prod_{\overline{v} \in \overline{S}_p} \GL_{2n, \cO_{F^+_{\overline{v}}}} \]
The lattice $\cV_{\widetilde{\lambda}} \subset V_{\widetilde{\lambda}}$ corresponding to a dominant weight $\widetilde{\lambda} \in X^\ast((\Res_{F^+ / \Q} \widetilde{G})_E )$ is defined as in the previous paragraph. Thus $\cV_{\widetilde{\lambda}}$ is an $\cO[\prod_{\overline{v} | p} \widetilde{G}(\cO_{F^+_{\overline{v}}})]$-module, finite free as $\cO$-module. 

We can now define Hecke algebras. Again, we do this first for $G$. Let $S$ be a finite set of finite places of $F$ containing the $p$-adic ones, and let $K$ be a good subgroup of $\GL_n(\A_F^\infty)$ such that $K_v = \GL_n(\cO_{F_v})$ if $v \not\in S$ and $K_v \subset \GL_n(\cO_{F_v})$ if $v|p$. Then for any $\lambda \in (\Z^n_+)^{\Hom(F, E)}$ the complex $R \Gamma(X_K, \cV_\lambda)$ is defined (as an object of $\mathbf{D}(\cO)$, up to unique isomorphism), and comes equipped with an action of Hecke algebras by endomorphisms (see (\ref{eqn:hecke_morphism})). We define $\T_v = \cH(\GL_n(F_v), \GL_n(\cO_{F_v})) \otimes_\bZ \cO$ and $\T^S = \cH(\GL_n^S, K^S) \otimes_\Z \cO$ and, if $\cV$ is an $\cO[K_S]$-module, finite free as $\cO$-module, then we write $\T^S(K, \cV)$ for the image of the $\cO$-algebra homomorphism 
\[ \T^S \to \End_{\mathbf{D}(\cO)}( R \Gamma(X_K, \cV) ) \]
constructed in \S \ref{sec:general_definition_of_hecke_operators}. If $\cV = \cV_\lambda$ then we even write $\T^S(K, \lambda) = \T^S(K, \cV_\lambda)$.

We now treat the case of $\widetilde{G}$. Let $S$ be a finite set of finite places of $F$ containing the $p$-adic ones and such that $S = S^c$, and let $\overline{S}$ denote the set of places of $F^+$ below a place of $S$. Let $\widetilde{K}$ be a good subgroup of $\widetilde{G}(\A_{F^+}^\infty)$ such that $\widetilde{K}_{\overline{v}} = \widetilde{G}(\cO_{F^+_{\overline{v}}})$ for each place $\overline{v} \not\in \overline{S}$, and such that $\widetilde{K}_{\overline{v}} \subset \widetilde{G}(\cO_{F^+_{\overline{v}}})$ for each place $\overline{v}|p$. In order to simplify notation, we set $\widetilde{G}^S = \widetilde{G}^{\overline{S}}$, $\widetilde{G}_S = \widetilde{G}_{\overline{S}}$, and similarly $\widetilde{K}^S = \widetilde{K}^{\overline{S}}$ and $\widetilde{K}_S = \widetilde{K}_{\overline{S}}$.

For any $\widetilde{\lambda} \in (\Z^{2n}_+)^{\Hom(F^+, E)}$ the complex $R \Gamma(X_{\widetilde{K}}, \cV_{\widetilde{\lambda}})$ is defined, and comes equipped with an action as in (\ref{eqn:hecke_morphism}). We define $\widetilde{\T}^S = \cH(\widetilde{G}^S, \widetilde{K}^S) \otimes_\Z \cO$ and if $\widetilde{\cV}$ is an $\cO[\widetilde{K}_S]$-module, finite free as $\cO$-module, then we write $\widetilde{\T}^S(\widetilde{K}, \widetilde{\cV})$ for the image of the $\cO$-algebra homomorphism
\[ \widetilde{\T}^S \to \End_{\mathbf{D}(\cO)}(R \Gamma(X_{\widetilde{K}},\widetilde{\cV}) )\]
constructed in \S \ref{sec:general_definition_of_hecke_operators}. If $\widetilde{\cV} = \cV_{\widetilde{\lambda}}$, then we even write $\widetilde{\T}^S(\widetilde{K}, \widetilde{\lambda}) = \widetilde{\T}^S(\widetilde{K}, \cV_{\widetilde{\lambda}})$. We also denote by 
\begin{equation}\label{eqn:satake T}
\cS: \widetilde{\T}^S \to \T^S 
\end{equation}
the map induced by~\eqref{eqn:satake}. 

Note that  Lemma \ref{lem:cohomology_is_perfect} shows that both $\T^S(K, \lambda)$ and $\widetilde{\T}^S(\widetilde{K}, \widetilde{\lambda})$ are finite $\cO$-algebras. We emphasize that the Hecke algebra $\T^S(K, \lambda)$ is defined only under the assumptions that  $S$ contains the $p$-adic places of $F$, that $K$ is a good subgroup such that $K_v = \GL_n(\cO_{F_v})$ for all $v\not\in S$, and that $\lambda$ is a dominant weight for $G$. The use of this notation therefore implies that these assumptions are in effect. Similar remarks apply to the Hecke algebra $\widetilde{\T}^S(\widetilde{K}, \widetilde{\lambda})$ (in particular, the use of this notation implies that $S$ is stable under complex conjugation, a condition we do not impose for $\GL_n$).

If $K' \subset K$ is a normal good subgroup with $(K')^S = K^S$, $R$ is an $\cO$-algebra, and $\cV$ is an $R[K_S]$-module, finite free as $R$-module, then we write $\T^S(K / K', \cV) = \T^S(R \Gamma_{K / K'}(X_{K'}, \cV))$ for the image of the homomorphism (cf. \ref{eqn:hecke_morphism_with_coefficients}):
\[ \bT^S  \to \End_{\mathbf{D}(R[K / K'])}( R \Gamma_{K / K'}(X_{K'}, \cV) ).  \]
There are canonical surjective homomorphisms $\T^S(R \Gamma_{K / K'}(X_{K'}, \cV)) \to \T^S(K', \cV)$  and $\T^S(R \Gamma_{K / K'}(X_{K'}, \cV)) \to \T^S(K, \cV)$, induced respectively by the forgetful functor $\mathbf{D}(R[K / K']) \to \mathbf{D}(R)$ and the functor $R \Gamma(K / K', ?) : \mathbf{D}(R[K / K']) \to \mathbf{D}(R)$. If further $K / K'$ is abelian, then we define ${}_{K / K'} \T^S = \T^S \otimes_\cO \cO[K / K']$ and write ${}_{K / K'} \T^S(R \Gamma_{K / K'}(X_{K'}, \cV))$ for the image of the homomorphism
\[ {}_{K / K'} \T^S \to \End_{\mathbf{D}(R[K / K'])}( R \Gamma_{K / K'}(X_{K'}, \cV) ). \]
 The analogous construction is valid as well for $\widetilde{G}$ but since we will not use it, we do not write down the definition. 

We will also occasionally encounter other complexes endowed with actions of the rings $\T^S$ and $\widetilde{\T}^S$. (For example, the cohomology $R \Gamma(\partial \widetilde{X}_{\widetilde{K}}, \cV_{\widetilde{\lambda}})$ of the boundary $\partial \widetilde{X}_{\widetilde{K}}$ of the Borel--Serre compactification of $\widetilde{X}_{\widetilde{K}}$.) If $C^\bullet \in \mathbf{D}(R)$ and we are given an $\cO$-algebra homomorphism $\T^S \to \End_{\mathbf{D}(R)}
(C^\bullet)$, then we will write $\T^S(C^\bullet)$ for the image of this  homomorphism. More generally, if $K' \subset K$ is a normal good subgroup with $(K')^S = K^S$ and $C^\bullet \in \mathbf{D}(R[K / K'])$ is a complex endowed with an $\cO$-algebra homomorphism $\T^S \to \End_{\mathbf{D}(R[K / K'])}(C^\bullet)$, then we will write $\T^S(C^\bullet)$ for the image of $\T^S$ in $\End_{\mathbf{D}(R[K / K'])}(C^\bullet)$. If further $K / K'$ is abelian, then we will write ${}_{K / K'} \T^S(C^\bullet)$ for the image of ${}_{K / K'} \T^S$ in $\End_{\mathbf{D}(R[K / K'])}(C^\bullet)$.

 If the complex $C^\bullet$ has bounded cohomology, then the map $\bT^S(C^\bullet) \to \bT^S(H^\ast(C^\bullet))$ has nilpotent kernel; this is a consequence of the following lemma.
 
\begin{lemma}\label{lem_cohomology_nilpotent_annihilator}
Let $R$ be a (possibly non-commutative) $\bZ$-algebra, let $C^\bullet \in \mathbf{D}(R)$ be a complex, and let $T \subset \End_{\mathbf{D}(R)}(C^\bullet)$ be a commutative subring. Let $I = \ker(T \to \End_{R}(H^\ast(C^\bullet)))$, and suppose that there exists an integer $d \geq 0$ such that $H^i(C^\bullet) = 0$ if $i \not\in [0, d]$. Then $I^{d+1} = 0$.
\end{lemma}
\begin{proof}
We show by induction on $d \geq 0$ that if $\phi_0, \dots, \phi_d \in \End_{\mathbf{D}(R)}(C^\bullet)$ satisfy $H^\ast(\phi_i) = 0$ for $i = 0, \dots, d$, then $\phi_0 \circ \phi_1 \circ \dots \circ \phi_d = 0$ in $\End_{\mathbf{D}(R)}(C^\bullet)$. The case $d = 0$ follows because in this case, $C^\bullet$ is isomorphic to $H^0(C^\bullet)$ in $\mathbf{D}(R)$.

In general, we can assume that $\tau_{\leq d-1} (\phi_0 \circ \dots \circ \phi_{d-1}) = 0$. There is an exact triangle
\[ \xymatrix@1{ \tau_{\leq d-1} C^\bullet \ar[r]^-f & C^\bullet \ar[r]^-g & H^d(C^\bullet) \ar[r] & }\]
We obtain exact sequences
\[ \xymatrix@1{ \Hom_{\mathbf{D}(R)}(H^d(C^\bullet), C^\bullet) \ar[r] & \Hom_{\mathbf{D}(R)}(C^\bullet, C^\bullet) \ar[r] & \Hom_{\mathbf{D}(R)}(\tau_{\leq d-1} C^\bullet, C^\bullet) }\]
and
\[ \xymatrix@1{ \Hom_{\mathbf{D}(R)}(C^\bullet, \tau_{\leq d-1} C^\bullet) \ar[r] & \Hom_{\mathbf{D}(R)}(C^\bullet, C^\bullet) \ar[r] & \Hom_{\mathbf{D}(R)}(C^\bullet, H^d(C^\bullet)). } \]
We deduce the existence of elements $\alpha \in \Hom_{\mathbf{D}(R)}(H^d(C^\bullet), C^\bullet) $ and $\beta \in \Hom_{\mathbf{D}(R)}(C^\bullet, \tau_{\leq d-1} C^\bullet)$ such that $\alpha \circ g = \phi_0 \circ \dots \circ \phi_{d-1}$ and $f \circ \beta = \phi_d$. We thus conclude that $\phi_0 \circ \dots \circ \phi_{d} = \alpha \circ g \circ f \circ \beta = 0$.
\end{proof}
As an illustration of the use of this result, suppose that $K' \subset K$ is a normal good subgroup with $(K')^S = K^S$, so that the Hecke algebra $\T^S(K / K', \cV_\lambda)$ is defined. We then have a diagram of Hecke algebras
\[ \T^S(K', \cV_\lambda) \leftarrow \T^S(K / K', \cV_\lambda) \rightarrow \T^S(K, \cV_\lambda), \]
where the kernel $I$ of the left arrow satisfies $I^{\dim_\bR X} = 0$ (by Lemma \ref{lem_cohomology_nilpotent_annihilator} applied with $R = \cO[K / K']$). In particular, if $J \subset \T^S(K, \cV_\lambda)$ denotes the image of $I$, then there exists a canonical map $\T^S(K', \cV_\lambda) \to  \T^S(K, \cV_\lambda) / J$ of $\T^S$-algebras. Similar statements for the Hecke algebras which act faithfully on cohomology could be proved using the Hochschild--Serre spectral sequence (for the covering $X_{K'} \to X_K$). 

Nilpotent ideals of Hecke algebras will occur frequently throughout this paper and they often have their origins in applications of the above Lemma \ref{lem_cohomology_nilpotent_annihilator}. (Compare, for example, the statement and proof of Proposition \ref{prop:existence_of_hecke_representation_for_U(n,n)_no_R} below.) We note that the integer $\dim_\bR X$ depends only on $n$ and the degree $[F : \bQ]$; the exponents of the nilpotent ideals we consider will also usually have this property. 

\subsubsection{Some useful Hecke operators}\label{sec:some_useful_hecke_operators}

In this section we define most of the Hecke operators that we will need at various points later in the paper. We fix once and for all a choice $\varpi_v$ of uniformizer at each finite place $v$ of $F$. 

We first define notation for unramified Hecke operators.  If $v$ is a finite place of $F$ and $1 \leq i \leq n$ is an integer then we write $T_{v, i} \in \cH(\GL_n(F_v), \GL_n(\cO_{F_v}))$ for the double coset operator
\[ T_{v, i} = [ \GL_n(\cO_{F_v}) \diag(\varpi_v, \dots, \varpi_v, 1, \dots, 1) \GL_n(\cO_{F_v})], \]
where $\varpi_v$ appears $i$ times on the diagonal. This is the same as the operator denoted by $T_{M, v, i}$ in ~\cite[Prop.-Def. 5.3]{new-tho}. We define a polynomial 
\numequation\label{eqn:hecke_pol_for_GL_n} \begin{split} P_v(X) = X^n&-T_{v, 1}X^{n-1} + \dots + (-1)^iq_v^{i(i-1)/2}T_{v, i}X^{n-i}+\dots \\ & + q_v^{n(n-1)/2}T_{v,n} \in \cH(\GL_n(F_v), \GL_n(\cO_{F_v}))[X].
	\end{split}  
\end{equation}
It corresponds to the characteristic polynomial of a Frobenius element on $\rec^T_{F_v}(\pi_v)$, where $\pi_v$ is an unramified representation of $\GL_n(F_v)$. We also find it helpful to introduce, for any $\sigma \in W_{F_v}$, the polynomial $P_{v, \sigma}(X) \in \T_v[X] = (\cH(\GL_n(F_v), \GL_n(\cO_{F_v})) \otimes_\bZ \cO)[X]$, which equals the characteristic polynomial of $\sigma$ on $\rec^T_{F_v}(\pi_v)$. We write $P_{v, \sigma}(X) = \sum_{i=0}^n (-1)^i e_{v, i}(\sigma) X^{n-i}$.

If $\overline{v}$ is a place of $F^+$ unramified in $F$, and $v$ is a place of $F$ above $\overline{v}$, and $1 \leq i \leq 2n$ is an integer, then we write $\widetilde{T}_{v, i} \in \cH( \widetilde{G}(F^+_{\overline{v}}),\widetilde{G}(\cO_{F^+_{\overline{v}}})) \otimes_\Z \Z[q_{\overline{v}}^{-1}]$ for the operator denoted $T_{G, v, i}$ in~\cite[Prop.-Def. 5.2]{new-tho}. We define a polynomial
\numequation\label{eqn_hecke_pol_for_tilde_G} \begin{split} \widetilde{P}_v(X) = X^{2n} & - \widetilde{T}_{v, 1} X^{2n-1} +  \dots + (-1)^j q_v^{j(j-1)/2} \widetilde{T}_{v, j} +  \dots \\&+ q_v^{n(2n-1)} \widetilde{T}_{v, 2n} \in \cH( \widetilde{G}(F^+_{\overline{v}}),\widetilde{G}(\cO_{F^+_{\overline{v}}})) \otimes_\Z \Z[q_{\overline{v}}^{-1}][X]. \end{split} 
\end{equation}
It corresponds to the characteristic polynomial of a Frobenius element on $\rec^T_{F_v}(\pi_v)$, where $\pi_v$ is the base change of an unramified representation $\sigma_{\overline{v}}$ of the group $\widetilde{G}(F^+_{\overline{v}})$. Again if $\sigma \in W_{F_v}$ then we write 
\[ \widetilde{P}_{v, \sigma}(X) = \sum_{i=0}^{2n} (-1)^i \widetilde{e}_{v, i}(\sigma) X^{n-i} \in \widetilde{\T}_v[X] = (\cH( \widetilde{G}(F^+_{\overline{v}}),\widetilde{G}(\cO_{F^+_{\overline{v}}})) \otimes_\Z \cO)[X] \]
for the polynomial corresponding to the characteristic polynomial of $\sigma$ on $\rec^T_{F_v}(\pi_v)$.

We next define notation for some ramified Hecke operators. If $v$ is a
finite place of $F$, and $c \geq b \geq 0$ are integers, then we write
$\Iw_v(b, c)$ for the subgroup of~$\GL_n(\cO_{F_v})$ consisting of those matrices which reduce to an upper-triangular matrix modulo $\varpi_v^c$, and to a unipotent upper-triangular matrix modulo $\varpi_v^b$. We define $\Iw_v = \Iw(0, 1)$ and $\Iw_{v, 1} = \Iw_v(1, 1)$; thus $\Iw_v$ is the standard Iwahori subgroup of $\GL_n(\cO_{F_v})$. If $1 \leq i \leq n$ is an integer and $c \geq 1$, then we will write $U_{v, i} \in \cH(\GL_n(F_v), \Iw_v(b, c))$ for the double coset operator
\[ U_{v, i} = [ \Iw_v(b, c) \diag(\varpi_v, \dots, \varpi_v, 1, \dots, 1) \Iw_v(b, c) ], \]
where $\varpi_v$ appears $i$ times on the diagonal. Note that this depends both on the uniformizer $\varpi_v$ and on the chosen level. We hope that this abuse of notation will not cause confusion. We also define 
\[ U_v = [ \Iw_v(b, c) \diag(\varpi_v^{n-1}, \varpi_v^{n-2}, \dots, \varpi_v, 1) \Iw_v(b, c) ]. \]
If $u \in T_n(\cO_{F_v})$, then we define
\[ \langle u \rangle = [ \Iw_v(b, c) u \Iw_v(b, c) ]. \]
Note that the subgroups $\Iw_v(b, c)$ all admit Iwahori decompositions with respect to the standard upper-triangular Borel subgroup of $\GL_n$. We write $\Delta_v \subset \GL_n(F_v)$ for the submonoid defined by
\[ \Delta_v = \sqcup_{\mu \in \Z_+^n} \Iw_v \diag(\varpi_v^{\mu_1}, \dots, \varpi_v^{\mu_n}) \Iw_v. \]
We now assume that each $p$-adic place of $F^+$ splits in $F$. In this case we set $\Delta_p = \prod_{v \in S_p}\Delta_v$. If $\lambda \in (\Z^n)^{\Hom(F, E)}$, then we define a homomorphism (of monoids) $\alpha_\lambda : \Delta_p \to E^\times$ by the formula
\[ \alpha_\lambda( ( k_{v, 1} \diag(\varpi_{v}^{a_{v, 1}}, \dots, \varpi_v^{a_{v, n}})k_{v, 2})_{v \in S_p}) = \prod_{v \in S_p} \prod_{\tau \in \Hom_{\Q_p}(F_v, E)} \prod_{i=1}^n \tau(\varpi_v)^{a_{v, i} (w^G_0 \lambda)_{\tau, i}}, \]
where $w_0^G$ is the longest element in the Weyl group $W((\Res_{F^+ / \Q} G)_E, (\Res_{F^+ / \Q} T)_E)$. If $\lambda \in (\Z^n_+)^{\Hom(F, E)}$ is dominant, then $\GL_n(F_p)$ acts on $\cV_\lambda \otimes_\cO E = V_\lambda$; we write $(g, x) \mapsto g \cdot x$ for this action. We endow $\cV_\lambda$ with the structure of $\cO[\Delta_p]$-module via the formula
\[ g \cdot_p x = \alpha_\lambda(g)^{-1} g \cdot x. \]
This is well defined: the fact that the lattice $\cV_{\lambda}$ is preserved under this twisted action follows as in~\cite[Definition 2.8]{ger} from Lemma 2.2 of \emph{loc. cit.} - the point is that $w_0^G \lambda$ is the lowest weight vector in $\cV_{\lambda}$, so $g\cdot x$ is divisible by $\alpha_\lambda(g)$ when $g\in \Delta_p$. 
Using the construction in \S \ref{sec:general_definition_of_hecke_operators}, we see that if $K \subset \GL_n(\A_F^\infty)$ is a good subgroup, and for each place $v | p$ of $F$ we have $K_v = \Iw_v(b, c)$, then there is a canonical homomorphism
\[ \cH( G^S, K^S) \otimes_\Z \cH(\Delta_p, K_p) \to \End_{\mathbf{D}(\cO)}(R \Gamma(X_K, \cV_\lambda)), \]
and in particular all the Hecke operators $U_{v, i}$ and $U_v$ act as endomorphisms of $R \Gamma(X_K, \cV_\lambda)$. Note that the action of these operators depends on the choice of uniformizer $\varpi_v$, because the twisted action $\cdot_p$ does. 

Now suppose that $\overline{v}$ is a finite place of $F^+$ which splits in $F$, and let $v$ be a place of $F$ above it. Then $\iota_v^{-1}(\Iw_{v}(b, c)) = \iota_{v^c}^{-1}(\Iw_{v^c}(b, c))$ (where here the subgroup $\Iw_{v}(b, c)$ is inside $\GL_{2n}(F_v)$), and we write $\widetilde{\Iw}_{\overline{v}}(b, c)$ for this subgroup of $\widetilde{G}(F^+_{\overline{v}})$. We define a Hecke operator in $\cH(\widetilde{G}(F^+_{\overline{v}}), \widetilde{\Iw}_{\overline{v}}(b, c))$ for each $i = 1, \dots, 2n$ by the formula
\[ \widetilde{U}_{v, i} = \iota_v^{-1}[ \Iw_v(b, c) \diag(\varpi_v, \dots, \varpi_v, 1, \dots, 1) \Iw_v(b, c) ], \]
where $\varpi_v$ appears $i$ times on the diagonal. We also define
\[ \widetilde{U}_v = \iota_v^{-1}[ \Iw_v(b, c) \diag(\varpi_v^{2n-1}, \varpi_v^{2n-2}, \dots, \varpi_v, 1) \Iw_v(b, c) ]. \]
If $u \in T(\cO_{F^+, \overline{v}})$, then we define
\[ \langle u \rangle = [ \widetilde{\Iw}_{\overline{v}}(b, c) u \widetilde{\Iw}_{\overline{v}}(b, c)]. \]
If $\varpi_{v^c} = \varpi^c_v$, then $\widetilde{U}_{v^c, i} = \widetilde{U}_{v, 2n - i} \widetilde{U}_{v, 2n}^{-1}$ and $\widetilde{U}_{v^c} = \widetilde{U}_v \widetilde{U}_{v, 2n}^{1-2n}$.

We write $\widetilde{\Delta}_{\overline{v}} \subset \widetilde{G}(F^+_{\overline{v}})$ for the submonoid defined by
\[ \widetilde{\Delta}_{\overline{v}} =  \iota_v^{-1} \left( \sqcup_{\mu \in \Z_+^{2n}} \Iw_v \diag(\varpi_v^{\mu_1}, \dots, \varpi_v^{\mu_{2n}}) \Iw_v \right) \]
(which is independent of the choice of $v$). Now suppose that each $p$-adic place of $F^+$ splits in $F$. In this case we set $\widetilde{\Delta}_p = \prod_{v \in \overline{S}_p} \widetilde{\Delta}_{\overline{v}}$. If $\widetilde{\lambda} \in (\Z^{2n})^{\Hom(F^+, E)}$, then we define a homomorphism  $\widetilde{\alpha}_{\widetilde{\lambda}} : \widetilde{\Delta}_p \to E^\times$ by the formula
\[ \widetilde{\alpha}_{\widetilde{\lambda}}(( k_{\overline{v}, 1} \iota_{\widetilde{v}}^{-1}(\diag(\varpi_{\widetilde{v}}^{a_{{\overline{v}}, 1}}, \dots, \varpi_{\widetilde{v}}^{a_{\overline{v}, 2n}})) k_{\overline{v}, 2})_{\overline{v} \in \overline{S}_p}) = \prod_{\overline{v} \in \overline{S}_p} \prod_{\tau \in \Hom_{\Q_p}(F_{\widetilde{v}}, E)} \prod_{i=1}^{2n} \tau(\varpi_{\widetilde{v}})^{a_{\overline{v}, i} (w^{\widetilde{G}}_0 \widetilde{\lambda})_{\tau, i}}, \]
	where $w_0^{\widetilde{G}}$ is the longest element in the Weyl group $W((\Res_{F^+ / \Q} \widetilde{G})_E, (\Res_{F^+ / \Q} T)_E)$. Here we recall that $\widetilde{v} \in \widetilde{S}_p$ is a fixed choice of place of $F$ lying above $\overline{v}$, and that it appears also in the definition of $\cV_{\widetilde{\lambda}}$. If $\widetilde{\lambda} \in (\Z^{2n}_+)^{\Hom(F^+, E)}$ is dominant, then $\widetilde{G}_p$ acts on $\cV_{\widetilde{\lambda}} \otimes_\cO E = V_{\widetilde{\lambda}}$; we write $(g, x) \mapsto g \cdot x$ for this action. We endow $\cV_{\widetilde{\lambda}}$ with the structure of $\cO[\widetilde{\Delta}_p]$-module via the formula
	\[ g \cdot_p x = \widetilde{\alpha}_{\widetilde{\lambda}}(g)^{-1} g \cdot x. \] 
  Using the construction in \S \ref{sec:general_definition_of_hecke_operators}, we see that if $\widetilde{K} \subset \widetilde{G}(\A_{F^+}^\infty)$ is a good subgroup, and for each place $\overline{v} | p$ of $F^+$ we have $\widetilde{K}_{\overline{v}} = \widetilde{\Iw}_{\overline{v}}(b, c)$, then there is a canonical homomorphism
\[ \cH( \widetilde{G}^S, \widetilde{K}^S) \otimes_\Z \cH(\widetilde{\Delta}_p, \widetilde{K}_p) \to \End_{\mathbf{D}(\cO)}(R \Gamma(\widetilde{X}_{\widetilde{K}}, \cV_{\widetilde{\lambda}})), \]
and in particular all the Hecke operators $\widetilde{U}_{v, i}$ and $\widetilde{U}_v$ act as endomorphisms of $R \Gamma(\widetilde{X}_{\widetilde{K}}, \cV_{\widetilde{\lambda}})$.

If $v$ is a finite place of $F$, prime to $p$, and $I_v$ is an open
compact subgroup of $\GL_n(F_v)$ satisfying $\Iw_v(1, 1) \subset I_v
\subset \Iw_v(0, 1)$, then $\Iw_v(0, 1) / I_v$ can be identified with
a quotient of $(k(v)^\times)^n$, and we define
\[ \Xi_v = (F_v^\times)^n / (\ker (\cO_{F_v}^\times)^n \to (k(v)^\times)^n \to \Iw_v(0, 1) / I_v). \]
The group $I_v$ admits an Iwahori decomposition with respect to the parabolic subgroup $B_n = T_n N_n$ of $\GL_n$, so we may apply the theory of \S \ref{sec:hecke_algebra_of_a_monoid}. Moreover,   \cite[Corollary 3.4]{Fli11} shows that for any $g \in \Xi_v$, the element $[I_v g I_v] \in \cH(\GL_n(F_v), I_v) \otimes_\bZ \cO$ is invertible (this uses our assumption that $q_v$ is a unit in $\cO$). Lemma \ref{lem:strongly_positive_elements} thus implies that there is an injective $\cO$-algebra homomorphism
\[ t : \Xi_v \to  \cH(\GL_n(F_v), I_v) \otimes_\Z \cO, \]
which sends any positive $g \in \Xi_v$ to the double coset $[I_v g I_v]$. Given any $\alpha \in F_v^\times$, we write $t_{v, i}(\alpha) \in \cH(\GL_n(F_v), I_v) \otimes_\Z \cO$ for the image under $t$ of the element $q_v^{- \langle \lambda, \rho + (1-n)/2 \det \rangle}(1, \dots, 1, \alpha, 1, \dots, 1)$ of $\cO[\Xi_v]$, where $\lambda \in X_\ast(T_n)$ denotes the image of $\alpha$ under the natural projection $\Xi_v \to (F_v^\times/ \cO_{F_v}^\times)^n = X_\ast(T_n)$ and $\rho \in X^\ast(T_n)$ is the usual half-sum of positive roots, and where $\alpha$ sits in the $i^\text{th}$ position. We write $e_{v, i}(\alpha) \in \cH(\GL_n(F_v), I_v) \otimes_\Z \cO$ for the coefficient of $(-1)^i X^{n-i}$ in the polynomial $\prod_{i=1}^n(X - t_{v, i}(\alpha))$. If $\sigma \in W_{F_v}$, then we set $t_{v, i}(\sigma) = t_{v, i}(\alpha)$ and $e_{v, i}(\sigma) = e_{v, i}(\alpha)$, where $\alpha \in F_v^\times$ is such that the restriction of $\sigma$ to the maximal abelian extension of $F_v$ equals $\Art_{F_v}(\alpha)$. We define the polynomial
\numequation P_{v, \sigma}(X) = \prod_{i=1}^n(X - t_{v, i}(\sigma)) = \sum_{i=0}^n (-1)^i e_{v, i}(\sigma) X^{n-i}  \in \cH(\GL_n(F_v), I_v) \otimes_\Z \cO[X].
\end{equation}
\begin{prop}\label{prop_action_of_central_elements_of_pro_p_iwahori_hecke_algebra}
	Let $\pi_v$ be an irreducible admissible $\overline{\Q}_p[\GL_n(F_v)]$-module. 
	\begin{enumerate}
		\item We have $\pi_v^{I_v} \neq 0$ if and only if
                  $\pi_v$ is isomorphic to an irreducible
                  subquotient %
                  of a representation $\Ind_{B_n(F_v)}^{\GL_n(F_v)}
                  \chi$, where $\chi = \chi_1 \otimes \dots \otimes \chi_n : (F_v^\times)^n \to \overline{\Q}_p^\times$ is a smooth character which factors through the quotient $(F_v^\times)^n \to \Xi_v$.
		\item Suppose that $\pi_v^{I_v} \neq 0$. Then for any $\alpha \in F_v^\times$, $e_{v, i}(\alpha)$ acts on $\pi_v^{I_v}$ as a scalar $e_{v, i}(\alpha, \pi_v) \in \overline{\Q}_p^\times$.
		\item Suppose that $\pi_v^{I_v} \neq 0$, and let $(r_v, N_v) = \rec_{F_v}^T(\pi_v)$. Then for any $\sigma \in W_{F_v}$, the characteristic polynomial of $r_v(\sigma)$ is equal to $\sum_{i=0}^n (-1)^i e_{v, i}(\alpha, \pi_v) X^{n-i}$, where $\alpha = \Art_{F_v}^{-1}(\sigma|_{F_v^\text{ab}})$.
	\end{enumerate}
\end{prop}
\begin{proof}
	The first part follows from \cite[Theorem 2.1]{Fli11}. The second part is a consequence of the fact that the elements $e_{v, i}(\alpha)$ lie in the centre of $\cH(\GL_n(F_v), I_v) \otimes_\Z \cO$, which in turn follows from the explicit description of the centre in  \cite[Proposition 4.11]{Fli11}. The final part follows from the description in \cite[\S 4]{Fli11} of the action of this centre on the $I_v$-invariants in the induced representation $\nInd_{B_n(F_v)}^{\GL_n(F_v)} \chi$.
\end{proof}
Now suppose that $v$ is a finite place of $F$, prime to $p$ and split over $F^+$, and write $\p_v \subset \GL_{2n}(\cO_{F_v})$ for the parahoric subgroup consisting of matrices whose reduction modulo $\varpi_v$ is block upper-triangular, with blocks of sizes $n, 1, 1, \dots, 1$. Projection to the lower right-hand block determines a homomorphism $\p_v \to B_n(k(v))$. We write $\p_{v, 1} \subset \p_v$ for the kernel of the composite homomorphism $\p_v \to B_n(k(v)) \to T_n(k(v))$. 

Let $\q_v \subset \GL_{2n}(F_v)$ be an open compact subgroup such that $\p_{v, 1} \subset \q_v \subset \p_v$, and set $\widetilde{\p}_v = \iota_v^{-1}(\p_v)$, $\widetilde{\p}_{v, 1} = \iota_v^{-1}(\p_{v, 1})$, and $\widetilde{\q}_v = \iota_v^{-1}(\q_v)$. These are open compact subgroups of $\widetilde{G}(F^+_{\overline{v}})$ and we can identify $\widetilde{\p}_v \cap G(F^+_{\overline{v}}) = \GL_n(\cO_{F_{v^c}}) \times \Iw_{v}$ and $\widetilde{\p}_{v, 1} \cap G(F^+_{\overline{v}}) = \GL_n(\cO_{F_{v^c}}) \times \Iw_{v, 1}$. In particular, we may identify the quotient $\widetilde{\p}_v / \widetilde{\p}_{v, 1}$ with $(k(v)^\times)^n$. The group $\widetilde{\q}_v$ admits an Iwahori decomposition with respect to the parabolic subgroup $P = G U$, so we may use the theory developed in \S \ref{sec:hecke_algebra_of_a_monoid}.
\begin{lemma}
The element $g = ( \varpi^{-c}_{v} \cdot 1_n, 1_n) \in \GL_n(F_{v^c}) \times \GL_n(F_v) = G(F^+_{\overline{v}})$ is strongly positive and the element $[\widetilde{\q}_v g \widetilde{\q}_v] \in \cH(\widetilde{G}(F^+_{\overline{v}}), \widetilde{\q}_v) \otimes_\bZ \cO$ is invertible. 
\end{lemma}
\begin{proof}
After applying $\iota_v$, we see that to prove the lemma it is enough to explain why $[ \q_v \diag (\varpi_v, \dots, \varpi_v, 1, \dots, 1) \q_v ]$ is an invertible element of $\cH(\GL_{2n}(F_v), \q_v) \otimes_\bZ \cO$ (where $\varpi_v$, $1$ each occur $n$ times). It follows from \cite[Corollary 3.4]{Fli11} that $[ \Iw_v(1) \diag (\varpi_v, \dots, \varpi_v, 1, \dots, 1) \Iw_v(1) ]$ is invertible in $\cH(\GL_{2n}(F_v),  \Iw_v(1)) \otimes_\bZ \cO$, while it follows from \cite[Theorem 4.5]{Fli11} that $[ \Iw_v(1) \diag (\varpi_v, \dots, \varpi_v, 1, \dots, 1) \Iw_v(1) ]$ and $[ \q_v ]$ commute. We deduce that $[ \q_v \diag (\varpi_v, \dots, \varpi_v, 1, \dots, 1) \q_v ] = [ \q_v ] \cdot [ \Iw_v(1) \diag (\varpi_v, \dots, \varpi_v, 1, \dots, 1) \Iw_v(1) ]$ is invertible, as required. 
\end{proof}
Lemma \ref{lem:strongly_positive_elements} implies that there is an injective $\cO$-algebra homomorphism
\[ t : \cH(\GL_n(F_{v^c}) \times \GL_n(F_v), \GL_n(\cO_{F_{v^c}}) \times I_{v}) \otimes_\bZ \cO \to \cH(\widetilde{G}(F^+_{\overline{v}}), \widetilde{\q}_v) \otimes_\bZ \cO. \]
We write $\Xi_v$ for the quotient of $(F_v^\times)^n$ associated to the group $I_v$. If $\sigma \in W_{F_v}$, then we write, with apologies for the abuse of notation, $t_{v, i}(\sigma) \in \cH(\widetilde{G}(F^+_{\overline{v}}), \widetilde{\q}_v) \otimes_\bZ \cO$ for the image under $t$ of the element $\| \sigma \|_v^{-n} t_{v, i}(\sigma) \in \cH(\GL_n(F_v), I_v) \otimes_\Z \cO$ defined previously. We write $e_{v, i}(\sigma) \in \cH(\widetilde{G}(F^+_{\overline{v}}), \widetilde{\q}_v) \otimes_\bZ \cO$ for the coefficient of $(-1)^i X^{n-i}$ in the polynomial 
\numequation P_{v, \sigma}(X) = \prod_{i=1}^n(X - t_{v, i}(\sigma))  \in (\cH(\widetilde{G}(F^+_{\overline{v}}), \widetilde{\q}_v) \otimes_\bZ \cO)[X].
\end{equation}
If $\sigma \in W_{F_{v^c}}$, then we define $e_{v^c, i}(\sigma) \in  \cH(\widetilde{G}(F^+_{\overline{v}}), \widetilde{\q}_v) \otimes_\bZ \cO$ to be the image under $t$ of the element $\| \sigma \|_v^{i(n-1)} e_{v^c, i}(\sigma) \in \cH(\GL_n(F_{v^c}), \GL_n(\cO_{F_{v^c}})) \otimes_\bZ \cO$. We define the polynomial 
\numequation P_{{v^c}, \sigma}(X) = \sum_{i=0}^n (-1)^i e_{v, i}(\sigma) X^{n-i}  \in (\cH(\widetilde{G}(F^+_{\overline{v}}), \widetilde{\q}_v) \otimes_\bZ \cO)[X].
\end{equation}
We finally define for any $\sigma \in W_{F_v}$ the polynomial
\[ \widetilde{P}_{v, \sigma}(X) = P_{v^c, \sigma^{-c}}(X) P_{v, \sigma}(X) \in (\cH(\widetilde{G}(F^+_{\overline{v}}), \widetilde{\q}_v) \otimes_\bZ \cO)[X], \]
and define elements $\widetilde{e}_{v, i}(\sigma) \in \cH(\widetilde{G}(F^+_{\overline{v}}), \widetilde{\q}_v) \otimes_\bZ \cO$ by the formula $\widetilde{P}_{v, \sigma}(X) = \sum_{i=0}^{2n}(-1)^i \widetilde{e}_{v, i}(\sigma) X^{n-i}$.
\begin{lemma}\label{lem:ramified_hecke_operators_are_central}
Suppose given an irreducible admissible $\overline{\bQ}_p[\widetilde{G}(F^+_{\overline{v}})]$-module $\widetilde{\pi}_{\overline{v}}$ such that $\widetilde{\pi}_{\overline{v}}^{\widetilde{\q}_v} \neq 0$, and let $\sigma \in W_{F_v}$. Then:
\begin{enumerate}
\item Each operator $\widetilde{e}_{v, i}(\sigma)$ acts by a scalar $\widetilde{e}_{v, i}(\sigma, \widetilde{\pi}_{\overline{v}}) \in \overline{\bQ}_p$ on $\widetilde{\pi}_{\overline{v}}^{\widetilde{\q}_v}$.\item Let $(r_v, N_v) = \rec^T_{F_v}(\widetilde{\pi}_{\overline{v}} \circ \iota_v^{-1})$. Then the characteristic polynomial of $r_v(\sigma)$ is equals $\sum_{i=0}^{2n} (-1)^i \widetilde{e}_{v, i}(\sigma, \widetilde{\pi}_{\overline{v}}) X^{n-i}$. 
\end{enumerate}
\end{lemma}
\begin{proof}
We fix a choice of isomorphism $\iota : \overline{\bQ}_p \to \bC$, so that normalized induction and normalized restriction (i.e. Jacquet module) may be defined over $\overline{\bQ}_p$. The proof uses well-known principles (cf. \cite[Lemma 1.17]{MR771671}). First, there exists a tamely ramified character $\chi : T(F^+_{\overline{v}}) \to \overline{\bQ}_p^\times$ such that $\widetilde{\pi}_{\overline{v}}$ is a subquotient of an induced representation $\widetilde{\Pi} = \nInd_{B(F^+_{\overline{v}})}^{\widetilde{G}(F^+_{\overline{v}})} \chi$. Identifying $T(F^+_{\overline{v}}) = T_{2n}(F_v)$, we may identify $\chi$ with a tuple of tamely ramified characters $\psi_1, \dots, \psi_{2n} : F_v^\times \to \overline{\bQ}_p$. To prove the the lemma, it suffices to show that $\widetilde{e}_{v, i}(\sigma)$ acts as a scalar on the subspace of $\widetilde{\q}_v$-invariants of $\widetilde{\Pi}$, this scalar being equal to the degree $i$ symmetric polynomial in $\psi_1(\Art_{F_v}^{-1}(\sigma)), \dots, \psi_{2n}(\Art_{F_v}^{-1}(\sigma))$. 

Let $R =  \overline{\bQ}_p[T(F^+_{\overline{v}}) / T(\cO_{F^+_{\overline{v}}})]$ and let $\chi_u : T(F^+_{\overline{v}}) \to R^\times$ be the universal unramified character. We consider the induced representation $\widetilde{\Pi}_u = \nInd_{B(F^+_{\overline{v}})}^{\widetilde{G}(F^+_{\overline{v}})} (\chi \otimes \chi_u)$, a smooth $R[\widetilde{G}(F^+_{\overline{v}})]$-module. Then $\widetilde{\Pi}_u^{\widetilde{\q}_v}$ is a finite free $R$-module and for any homomorphism $x : R \to \overline{\bQ}_p$, corresponding to an unramified character $\chi_x : T(F^+_{\overline{v}}) \to \overline{\bQ}_p^\times$ with induced representation $\widetilde{\Pi}_x = \nInd_{B(\overline{F}^+_{\overline{v}})}^{\widetilde{G}(F^+_{\overline{v}})} (\chi \otimes \chi_x)$, the induced map
\[ ( \widetilde{\Pi}_u^{\widetilde{\q}_v} )\otimes_{R, x} \overline{\bQ}_p \to \widetilde{\Pi}_x^{\widetilde{\q}_v} \]
is an isomorphism. We may identify $\chi \otimes \chi_x$ with a tuple $\psi_{x, 1}, \dots, \psi_{x, 2n} : F_v^\times \to \overline{\bQ}_p$ of tamely ramified characters. To prove the lemma, it in fact suffices to show for a Zariski dense set of points $x \in \Spec R (\overline{\bQ}_p)$ that the Hecke operator $\widetilde{e}_{v, i}(\sigma)$ acts by a scalar on $ \widetilde{\Pi}_x^{\widetilde{\q}_v}$ which is equal to the degree $i$ symmetric polynomial in $\psi_{x,1 }(\Art_{F_v}^{-1}(\sigma)), \dots, \psi_{x, 2n}(\Art_{F_v}^{-1}(\sigma))$. 

Consider the Jacquet module $r_P(\widetilde{\Pi}_x) = (\widetilde{\Pi}_x)_{U(F^+_{\overline{v}})} \otimes \delta_P^{-1/2}$, an admissible $\overline{\bQ}_p[G(F^+_{\overline{v}})]$-module. Then \cite[Theorem 7.9]{MR1643417} asserts that the natural map
\[ q : (\widetilde{\Pi}_x)^{\widetilde{\q}_v} \to r_P(\widetilde{\Pi}_x)^{\GL_n(\cO_{F_{v^c}}) \times I_v} \]
is an isomorphism which satisfies the formula 
\[ h q(x) = \delta_P^{1/2}(g) q( t(h) x ) \]
for any $x \in (\widetilde{\Pi}_x)^{\widetilde{\q}_v}$ and Hecke operator $h = [ (\GL_n(\cO_{F_{v^c}}) \times I_v) g (\GL_n(\cO_{F_{v^c}}) \times I_v)] \in \cH( \GL_n(F_{v^c}) \times \GL_n(F_v), \GL_n(\cO_{F_{v^c}}) \times I_v) \otimes_\bZ \cO$. The geometrical lemma (\cite[1.2, Theorem]{Zel80}) asserts that $r_P(\widetilde{\Pi}_x)$ admits a filtration by 	induced representations $\sigma_{x, Y_{v^c}, Y_v}$, indexed by partitions $\{ 1, \dots, 2n \} = Y_{v^c} \sqcup Y_v$, that may be described as follows:
\[ \sigma_{x, Y_{v^c}, Y_v} = \left( \nInd_{B_n(F_{v^c})}^{\GL_n(F_{v^c})} \otimes_{i \in Y_{v^c}} \psi_{x, i}^{-c} \right) \otimes \left( \nInd_{B_n(F_{v})}^{\GL_n(F_{v})} \otimes_{i \in Y_{v}} \psi_{x, i}\right). \]
For a Zariski dense set of points $x$ (including those for which the central element $(\varpi_{v^c} \cdot 1_n, 1_n) \in \GL_n(F_{v^c}) \times \GL_n(F_v)$ acts by a distinct scalar on each induced representation $\sigma_{x, Y_{v^c}, Y_v}$), this filtration splits and $r_P(\widetilde{\Pi}_x)$ is isomorphic to a sum of induced representations. The Hecke operators $e_{v, i}(\sigma)$ and $e_{v^c, i}(\sigma^{-c})$ act as a scalar in the subspace of $\GL_n(\cO_{F_{v^c}}) \times I_v$-invariants in each summand and a calculation shows that the scalar by which $\sum_{i+j = k } e_{v, i}(\sigma) e_{v^c, j}(\sigma^{-c})$ acts in each summand is the degree $k$ elementary symmetric polynomial in $\psi_{x, 1}(\Art_{F_v}^{-1}(\sigma)), \dots, \psi_{x, 2n}(\Art_{F_v}^{-1}(\sigma))$ -- in particular, independent of the choice of summand. 

Transferring this information back along the map $q$ shows that at such points $x$, the element $\widetilde{e}_{v, k}(\sigma)= \sum_{i+j = k } e_{v, i}(\sigma) e_{v^c, j}(\sigma^{-c})$ acts by a scalar on $(\widetilde{\Pi}_x)^{\widetilde{\q}_v}$, which equals the degree $k$ elementary symmetric polynomial in $\psi_{x, 1}(\Art_{F_v}^{-1}(\sigma)), \dots, \psi_{x, 2n}(\Art_{F_v}^{-1}(\sigma))$. This completes the proof. 
\end{proof}
Fix a choice of Frobenius lift $\phi_{v} \in W_{F_{v}}$. We define $\Res_{v} \in (\cH(\widetilde{G}(F^+_{\overline{v}}), \widetilde{\q}_v) \otimes_\bZ \cO)$ to be the resultant of the polynomials $P_{v^c, \phi_v^{-c}}(X)$ and $P_{v, \phi_{v}}(X)$.  
\begin{prop}\label{prop:unramified_subrep_for_invertible_resultant} Let $\widetilde{\pi}_{\overline{v}}$ be an irreducible admissible $\overline{\bQ}_p[\widetilde{G}(F^+_{\overline{v}})]$-module and suppose that $\widetilde{\pi}_{\overline{v}}^{\widetilde{\q}_v} \neq 0$, and let $(r_v, N_v) = \rec^T_{F_v}( \widetilde{\pi}_{\overline{v}} \circ \iota_v^{-1})$. Let $T_{\overline{v}}$ denote the $\overline{\bQ}_p$-subalgebra of $\End_{\overline{\bQ}_p}( \widetilde{\pi}_{\overline{v}}^{\widetilde{\q}_v}  )$ generated by the images of the elements $e_{v^c, i}(\phi_v^{-c})$ and $e_{v, i}(\phi_v)$. Then for each maximal ideal $\m \subset T_{\overline{v}}$, either $\Res_v \in \m$ or $\Res_v \not\in \m$, $T_{\overline{v}, \m} = T_{\overline{v}} / \m = \overline{\bQ}_p$, and for all $\tau_v \in I_{F_v}$, $N_v P_{v, \phi_{v}}(r_v(\phi_v)) = 0$ and $(r_v(\tau_v) - 1) P_{v, \phi_{v}}(r_v(\phi_v)) = 0$ in $M_{2n}(T_{\overline{v}} / \m) = M_{2n}(\overline{\bQ}_p)$.
\end{prop}
\begin{proof}
We use again some of the ideas in the proof of Lemma \ref{lem:ramified_hecke_operators_are_central}. Choose an isomorphism $\iota : \overline{\bQ}_p \to \bC$. For $m \geq 1$, let $\operatorname{St}_m$ denote the Steinberg representation of $\GL_m(F_v)$ (i.e.\ the square-integrable quotient of $\Ind_{B_m(F_v)}^{\GL_m(F_v)} \overline{\bQ}_p)$. Then there is an isomorphism $\rec^T_{F_v}(\operatorname{St}_m) = \operatorname{Sp}_m$, where $\operatorname{Sp}_m$ is the Weil--Deligne representation on $\overline{\bQ}_p^m = \oplus_{i=1}^m \overline{\bQ}_p \cdot e_i$ where $W_{F_v}$ acts on $e_i$ by the character $| \cdot |^{1-i} \circ \Art_{F_v}^{-1}$ and $N_v$ acts by $N_v e_1 = 0$, $N_v e_i = e_{i-1}$ ($i = 2, \dots, m$).

Since $\widetilde{\pi}_{\overline{v}}^{\widetilde{\q}_v} \neq 0$, we can find an isomorphism 
\[ (r_v, N_v) \cong \oplus_{i=1}^s \operatorname{Sp}_{\alpha_i}(\psi_i | \cdot|^{(1-2n)/2})\]
 for some tamely ramified characters $\psi_i : F_v^\times \to \overline{\bQ}_p^\times$; then $\widetilde{\pi} \circ \iota_v^{-1}$ is isomorphic to a subquotient of the induced representation 
 \[ \Pi = \nInd_{P_\alpha(F_v)}^{\GL_{2n}(F_v)} \otimes_{i=1}^s \operatorname{St}_{\alpha_i}(\psi_i \circ \Art_{F_v}), \]
 where $P_\alpha$ is the standard parabolic subgroup of $\GL_{2n}$ corresponding to the partition $2n = \alpha_1 + \dots + \alpha_s$. Let $\widetilde{\Pi} = \Pi \circ \iota_v$. Let $T'_{\overline{v}}$ denote the $\overline{\bQ}_p$-subalgebra of $\End_{\overline{\bQ}_p}( \widetilde{\Pi}^{\widetilde{\q}_v}  )$ generated by the images of the elements $e_{v^c, i}(\phi_v^{-c})$ and $e_{v, i}(\phi_v)$. Then $T_{\overline{v}}$ is a quotient of $T'_{\overline{v}}$ and it suffices to show that the conclusion of the lemma holds with $T_{\overline{v}}$ replaced by $T'_{\overline{v}}$.

By the geometrical lemma, we can find a filtration of $r_P(\widetilde{\Pi})$ with graded pieces indexed by tuples $\mu = (\mu_{ij})_{i = 1, 2, j = 1, \dots, s}$, where the $\mu_{ij}$ are non-negative integers such that for each $j = 1, \dots, s$, $\mu_{1j} + \mu_{2j} = \alpha_j$ and for each $i = 1, 2$, $\mu_{i1} + \dots + \mu_{is} = n$. The representation of $G(F^+_{\overline{v}}) = \GL_n(F_{v^c}) \times \GL_n(F_v)$ indexed by the tuple $\mu$ is
\begin{multline*} \sigma_\mu = \left( \nInd_{P_1(F_{v^c})}^{\GL_n(F_{v^c})} \operatorname{St}_{\mu_{1s}}(\theta_{1s}^{-c}) \otimes \dots \otimes \operatorname{St}_{\mu_{11}}(\theta_{11}^{-c}) \right) \\ \otimes \left( \nInd_{P_2(F_{v})}^{\GL_n(F_{v})} \operatorname{St}_{\mu_{21}}(\theta_{21}) \otimes \dots \otimes \operatorname{St}_{\mu_{2s}}(\theta_{2s}) \right), 
\end{multline*}
where $\theta_{ij} : F_v^\times \to \overline{\bQ}_p^\times$ is the character given by the formulae
\[ \theta_{1j} = \psi_j | \cdot |^{\mu_{2j}/2}, \theta_{2j} = \psi_j | \cdot |^{-\mu_{1j}/2} \]
for $j =1, \dots, s$.

We recall that the natural projection $\widetilde{\Pi}^{\widetilde{\q}_v} \to r_P(\widetilde{\Pi})^{\GL_n(\cO_{F_{v^c}}) \times I_v}$ is an isomorphism. The maximal ideals $\m \subset T'_{\overline{v}}$ correspond to the different factorisations $\widetilde{P}_{v, \phi_v}(X) = P_{v^c, \phi_v^{-c}}(X) P_{v, \phi_v}(X)$ that occur in $\widetilde{\Pi}^{\widetilde{\q}_v}$. Each factorisation arises from (at least one) $\mu$ such that  $(\sigma_\mu)^{\GL_n(\cO_{F_{v^c}}) \times I_v} \neq 0$: the corresponding factorisation is given by
\begin{multline*} P_{v^c, \phi_v^{-c}}(X) = \prod_{j=1}^s \prod_{k=1}^{\mu_{1j}} (X - (\theta_{1j}^{-c} | \cdot |^{(1 - 2n - \alpha_j + 2k - 1) / 2})(\phi_v^{-c}))  \\ = \prod_{j=1}^s \prod_{k=1}^{\mu_{1j}} (X - (\theta^{1j}| \cdot |^{(1 - 2n + \alpha_j + 1 - 2 k) / 2})(\phi_v)), 
\end{multline*}
\[ P_{v, \phi_v}(X) = \prod_{j=1}^s \prod_{k=1}^{\mu_{2j}} (X - \theta_{2j} | \cdot |^{(1 - 2n +\mu_{2j} - \mu_{1j} + 1 - 2k) / 2})(\phi_v)). \]
If $(\sigma_\mu)^{\GL_n(\cO_{F_{v^c}}) \times I_v} \neq 0$, then we must have $\mu_{1j} \in \{ 0, 1 \}$ for all $j =1 , \dots, s$. Let us choose therefore a maximal ideal $\m$ such that $\Res_v \not\in \m$ and a tuple $\mu$ giving rise to $\m$. Combining \cite[Proposition 2.2]{Tho21} and Lemma \ref{lem:ramified_hecke_operators_are_central}, we find that $T'_{\overline{v}, \m} = T'_{\overline{v}} / \m = \overline{\bQ}_p$. Let $Q(X)$ denote the image of $P_{v, \phi_v}(X)$ modulo $\m$. Examining the action of $Q(r_v(\phi_v))$ in the summand $\operatorname{Sp}_{\alpha_j}(\psi_j | \cdot|^{(1-2n)/2})$ of $(r_v, N_v)$, we see that  $Q(r_v(\phi_v))$ either annihilates this summand (if $\mu_{1j} = 0$) or at least has image contained in the span of the vector $e_1$. In either case we find that $I_{F_v}$ acts trivially on the image of $Q(r_v(\phi_v))$ and $N_v$ annihilates this image. This is what we needed to show. 
\end{proof}
\begin{cor}\label{cor:existence_of_unramified_subrep}
Let $\widetilde{\pi}_{\overline{v}}$ be an irreducible admissible $\overline{\bQ}_p[\widetilde{G}(F^+_{\overline{v}})]$-module and suppose that $\widetilde{\pi}_{\overline{v}}^{\widetilde{\q}_v} \neq 0$. Let $\rho : G_{F_v} \to \GL_{2n}(\overline{\bQ}_p)$ be a continuous representation such that $\mathrm{WD}(\rho)^{F-ss} \cong  \rec^T_{F_v}( \widetilde{\pi}_{\overline{v}} \circ \iota_v^{-1})$. Let $T_{\overline{v}}$ be defined as in the proposition. Then for all $\tau_v \in I_{F_v}$, we have the equality
\[ \Res_v^{(2n)!} (\rho(\tau_v) - 1) P_{v, \phi_{v}}(\rho(\phi_v)) = 0 \]
in $M_{2n}(\overline{\bQ}_p) \otimes_{\overline{\bQ}_p} T_{\overline{v}} = M_{2n}(T_{\overline{v}})$.
\end{cor}
\begin{proof}
We can again take this statement `one maximal ideal of $T_{\overline{v}}$ at a time'. The number $(2n)!$ is a crude upper bound for the $\overline{\bQ}_p$-dimension of $T_{\overline{v}}$. In particular, if $\Res_v \in \m$ then $\Res_v^{(2n)!} T_{\overline{v}, \m} = 0$. We therefore need only show that for each maximal ideal $\m$ such that $\Res_v \not\in \m$, we have the equality
\[ (\rho(\tau_v) - 1) Q(\rho(\phi_v)) = 0 \]
in $M_{2n}(\overline{\bQ}_p)$ for every $\tau_v \in I_{F_v}$, where $Q(X)$ denotes the image of $P_{v, \phi_v}(X)$ modulo $\m$. Let $\rho(\phi_v) = su$ be the multiplicative Jordan decomposition (so that $s$ is semisimple, $u$ is unipotent, and $s, u$ commute). Then $r_v(\phi_v) = s$, by definition of Frobenius semi-simplification. Since $\Res_v \text{ mod } \m$ is non-zero, $Q(\rho(\phi_v))$ and $Q(r_v(\phi_v))$ have the same image, which is the the span of the eigenspaces of $r_v(\phi_v)$ corresponding to eigenvalues which are not roots of $Q(X)$. Since $N_v Q(r_v(\phi_v)) = 0$, we find that for each $\tau_v \in I_{F_v}$, $\rho(\tau_v)$ and $r(\tau_v)$ agree on the image of $Q(r_v(\phi_v))$. We finally conclude that
\[  (\rho(\tau_v) - 1) Q(\rho(\phi_v)) = (r_v(\tau_v) - 1) Q(r_v(\phi_v)) = 0, \]
as required. 
\end{proof}
We now describe the behavior of some of the above Hecke operators under parabolic restriction, with respect to the Siegel parabolic. We first give the statements in the unramified case. In order to ease notation, we use the following convention: if $f(X)$ is a polynomial of degree $d$, with unit constant term $a_0$, then $f^\vee(X) = a_0^{-1} X^d f(X^{-1})$. 
\begin{prop}\label{prop_satake_transform_unramified_case}
Let $v$ be a place of $F$, unramified over the place $\overline{v}$ of $F^+$. Let 
\[ \cS : \cH( \widetilde{G}(F^+_{\overline{v}}), \widetilde{G}(\cO_{F^+_{\overline{v}}})) \to \cH( G(F^+_{\overline{v}}), G(\cO_{F^+_{\overline{v}}})) \]
denote the homomorphism defined by~\eqref{eqn:satake}. %
Then we have
\[ \cS( \widetilde{P}_v(X)) = P_v(X) q_v^{n(2n-1)} P_{v^c}^\vee(q_v^{1-2n} X). \]
\end{prop}
\begin{proof}
	See \cite[Proposition-Definition 5.3]{new-tho}.
\end{proof}
We now discuss the ramified split case. Suppose first that $v$ is a place of $F$ which is split over the place $\overline{v}$ of $F^+$. Let $\widetilde{I}_{\overline{v}}$ be a subgroup of $\widetilde{G}(F^+_{\overline{v}})$ satisfying $\widetilde{\Iw}_{\overline{v}}(1, 1) \subset \widetilde{I}_{\overline{v}} \subset \widetilde{\Iw}_{\overline{v}}(0, 1)$. Then $\widetilde{I}_{\overline{v}} \cap G(F^+_{\overline{v}})$ may be identified with a product $I_{v^c} \times I_{v}$ of open compact subgroups of $\GL_n(F_{v^c})$ and $\GL_n(F_{v})$, respectively. 	If $\sigma \in W_{F_v}$, then we define
	\numequation \widetilde{P}_{v, \sigma}(X) = \prod_{i=1}^{2n}(X - \iota_v^{-1}(t_{v, i}(\sigma))) = \sum_{i=0}^{2n} (-1)^i \iota_v^{-1}(e_{v, i}(\sigma)) X^{2n-i} \in \cH(\widetilde{G}(F^+_{\overline{v}}), \widetilde{I}_{\overline{v}}) \otimes_\Z \cO[X].
\end{equation}
The group $\widetilde{I}_{\overline{v}}$ admits an Iwahori decomposition with respect to the parabolic subgroup $P$, and the element $(\varpi_{v^c}^{-1} \cdot 1_n, 1_n) \in G(F^+_{\overline{v}}) = \GL_n(F_{v^c}) \times \GL_n(F_v)$ is strongly positive and defines a Hecke operator $[\widetilde{I}_{\overline{v}} (\varpi_{v^c}^{-1} \cdot 1_n, 1_n) \widetilde{I}_{\overline{v}}]$ which is invertible in $\cH(\widetilde{G}(F^+_{\overline{v}}), \widetilde{I}_{\overline{v}}) \otimes_\Z \cO$. We can therefore apply Lemma \ref{lem:strongly_positive_elements}, allowing us to state the following result.
\begin{prop}\label{prop_satake_transform_ramified_iwahori_case}
For any $\sigma \in W_{F_v}$, we have 
	\[ \cS( \widetilde{P}_{v, \sigma}(X)) = P_{v, \sigma}(X) \| \sigma \|_v^{n(1-2n)} P_{v^c, \sigma^{-c}}( \| \sigma \|_v^{2n - 1} X). \]
\end{prop}
\begin{proof}
This results from a calculation using the definition of $\widetilde{P}_{v, \sigma}(X)$ and the formula for the composite $\cS \circ t$ given in \S \ref{sec:hecke_algebra_of_a_monoid}. (Let $\alpha\in F^\times_{v^c}$ be such that $\Art_{F_{v^c}}(\alpha)$ agrees with the restriction of $\sigma^c$ to the maximal abelian extension of $F_{v^c}$. For the element $((1,\dots\alpha^{-1}, \dots 1), 1_n)$ of $\GL_n(F_{v^c})\times \GL_n(F_v)$, the action of $\cS\circ t$ on the corresponding Hecke operator is by multiplication by $\|\sigma\|_v^{-n}$.) 
\end{proof}
Suppose next that $v$ is a place of $F$ which is split over the place $\overline{v}$ of $F^+$ and that $\widetilde{\q}_v \subset \widetilde{G}(F^+_{\overline{v}})$ is an open compact subgroup such that $\widetilde{\p}_{v, 1} \subset \widetilde{\q}_v \subset \widetilde{\p}_v$. Write $\widetilde{\q}_v \cap G(F^+_{\overline{v}}) = \GL_n(\cO_{F_{v^c}}) \times I_{v}$. We have already observed that  Lemma \ref{lem:strongly_positive_elements} applies in this situation, and we have the following analogue of Proposition \ref{prop_satake_transform_ramified_iwahori_case}, which is proved in the same way. 
\begin{prop}\label{prop_satake_transform_half_ramified_iwahori_case}
For any $\sigma \in W_{F_v}$, we have $\cS(P_{v, \sigma}(X)) = P_{v, \sigma}(X)$ and $\cS(P_{v^c, \sigma^{-c}}(X)) = \| \sigma \|_v^{n(1-2n)} P_{v^c, \sigma^{-c}}( \| \sigma \|_v^{2n - 1} X)$.
\end{prop}
\subsubsection{Duality and twisting}\label{sec:twisting_and_duality}

In this section we record some operations that relate different cohomology
groups and the actions of the corresponding Hecke operators. We deal with duality first. Let $S$ be a finite set of finite places of $F$ such that $S = S^c$. There are anti-involutions
\[ \iota : \cH(G^S, K^S) \to \cH(G^S, K^S) \]
and (if $S = S^c$)
\[ \widetilde{\iota} : \cH(\widetilde{G}^S, \widetilde{K}^S) \to \cH(\widetilde{G}^S, \widetilde{K}^S) \]
given on double cosets by $\widetilde{\iota}([\widetilde{K}^S g \widetilde{K}^S]) = [\widetilde{K}^S g^{-1} \widetilde{K}^S]$ (resp. $\iota([K^S g K^S]) = [K^S g^{-1} K^S]$). In particular we have anti-involutions $\widetilde{\iota} : \widetilde{\T}^S \to \widetilde{\T}^S$ and $\iota : \T^S \to \T^S$ (which are actually involutions, since these Hecke algebras are commutative). If $v \not\in S$ then we have the formulae
\[ \widetilde{\iota}(\widetilde{P}_v(X)) = q_v^{2n(2n-1)} \widetilde{P}_v^\vee(q_v^{1-2n}X) = \widetilde{P}_{v^c}(X), \]
\[ \iota(P_v(X)) = q_v^{n(n-1)} P_v^\vee(q_v^{1-n}X). \]
If $\widetilde{\m} \subset \widetilde{\T}^S$ (resp. $\m \subset \T^S$) is a maximal ideal with residue field a finite extension of $k$, then we define $\widetilde{\m}^\vee = \widetilde{\iota}(\m)$ (resp. $\m^\vee = \iota(\m)$). The following lemma describes the action of $\widetilde{\iota}$ at ramified split places.
The interaction between these involutions and Poincar\'e duality is described by the following proposition. We write $\widetilde{D} = \dim_\R \widetilde{X}$ (resp. $D = \dim_\R X$).
\begin{prop}\label{prop_poincare_duality}
	Let $R = \cO$ or $\cO / \varpi^m$ for some $m \geq 1$. Let $\widetilde{K} \subset \widetilde{G}(\A_{F^+}^\infty)$ (resp. $K  \subset \GL_n(\A_F^\infty)$) be a good subgroup, and let $\cV$
        be an $R[\widetilde{K}_S]$-module (resp. $R[K_S]$-module), which is finite free as an $R$-module. Let $\cV^\vee = \Hom(\cV, R)$. Then there is an isomorphism
	\[ R \Hom_R( R \Gamma_c(\widetilde{X}_{\widetilde{K}}, \cV), R) \cong R \Gamma(\widetilde{X}_{\widetilde{K}}, \cV^\vee)[\widetilde{D}] \]
	(resp.
		\[ R \Hom_R( R \Gamma_c(X_K, \cV), R) \cong R \Gamma(X_K, \cV^\vee)[D]) \]
	in $\mathbf{D}(R)$ that is equivariant for the action of $\cH(\widetilde{G}^S, \widetilde{K}^S)$ (resp. $\cH(G^S, K^S)$), when this Hecke algebra acts by $\widetilde{\iota}^t$ (resp. $\iota^t$) on the left-hand side and in the usual way on the right-hand side. 
\end{prop}
\begin{proof}
	See \cite[Prop. 3.7]{new-tho}.
\end{proof}
\begin{cor}\label{cor:dual_Hecke_alg}
	Let $R = \cO$ or $\cO / \varpi^m$ for some $m \geq 1$. Let $\widetilde{K} \subset \widetilde{G}(\A_{F^+}^\infty)$ (resp. $K  \subset \GL_n(\A_F^\infty)$) be a good subgroup, and let $\cV$
        be an $R[\widetilde{K}_S]$-module (resp. $R[K_S]$-module), which is finite free as an $R$-module. Let $\cV^\vee = \Hom(\cV, R)$. Then $\widetilde{\iota}$ (resp. $\iota$) descends to an isomorphism 
        \[ \widetilde{\T}^S( R \Gamma_c(\widetilde{X}_{\widetilde{K}}, \cV) ) \cong \widetilde{\T}^S( R \Gamma(\widetilde{X}_{\widetilde{K}}, \cV^\vee)) \]
        (resp. an isomorphism
	\[ \T^S( R \Gamma_c(X_K, \cV) ) \cong \T^S( R \Gamma(X_K, \cV^\vee)) \]
	of $R$-algebras. In particular, if $\widetilde{\m}$ (resp. $\m$) is a maximal ideal of $\widetilde{\T}^S$ (resp. $\T^S$) in the support of $H^\ast_c(\widetilde{X}_{\widetilde{K}}, \cV)$ (resp. $H^\ast_c(X_K, \cV)$), then $\widetilde{\m}^\vee$ (resp. $\m^\vee$) is in the support of $H^\ast(\widetilde{X}_{\widetilde{K}}, \cV^\vee)$ (resp. $H^\ast(X_K, \cV^\vee)$). 
\end{cor}
\begin{proof}
	We justify the statements for $\GL_n$. The proposition implies that there is a commutative diagram
\[ \xymatrix{ 	\T^S \ar[r] \ar[d]_\iota & \End_{\mathbf{D}(R)}( R \Gamma_c(X_K, \cV)) \ar[d] \\
	\T^S\ar[r] & \End_{\mathbf{D}(R)}( R \Gamma(X_K, \cV^\vee)), } \]
	where the horizontal arrows are the canonical ones and the right vertical arrow is the one induced by the Poincar\'e duality isomorphism. The statement of the corollary is equivalent to the assertion that image under $\iota$ of the kernel of top horizontal arrow is equal to the kernel of the lower horizontal arrow. This follows from the commutativity of the diagram.
\end{proof}
We next deal with twisting for the group $G$. Let $K \subset \GL_n(\A_F^\infty)$ be a good subgroup and let $\psi : G_F \to \cO^\times$ be a continuous character such that $\psi \circ \Art_{F_v}$ is trivial on $\det(K_v)$ for each place $v\not\in S$ of $F$. We define an isomorphism of $\cO$-algebras $f_\psi : \cH(G^S, K^S) \otimes_\bZ \cO \to \cH(G^S, K^S) \otimes_\bZ \cO$ by the formula $f_\psi(f)(g) = \psi(\Art_F(\det(g)))^{-1} f(g)$. (It is an isomorphism because it has an inverse, given by the formula $f_\psi^{-1} = f_{\psi^{-1}}.$) If $K_v = \GL_n(\cO_{F_v})$ for each $ v \not\in S$ (so that $\psi$ is unramified outside $S$ and $\cH(G^S, K^S) \otimes_\bZ \cO = \T^S$) then we have the formula $f_\psi(T_{v, i}) = \psi(\Frob_v)^{-i} T_{v, i}$ for each finite place $v \not\in S$ of $F$.  If $\m \subset \T^S$ is a maximal ideal with residue field a finite extension of $k$, then we define $\m(\psi) = f_\psi(\m)$. 
\begin{prop}\label{prop:twisting_by_character}
	Let $K \subset \GL_n(\A_F^\infty)$ be a good subgroup, and suppose that $S$ contains the $p$-adic places of $F$. Let $\psi : G_F \to \cO^\times$ be a continuous character satisfying the following conditions:
	\begin{enumerate}
		\item For each finite place $v \nmid p$ of $F$, $\psi \circ \Art_{F_v}$ is trivial on $\det(K_v)$.
		\item There is $m = (m_\tau)_\tau \in \Z^{\Hom(F, E)}$ such that for each place $v | p$ of $F$, and for each $k \in \det(K_v)$, we have
		\[  \psi(\Art_{F_v}(k)) = \prod_{\tau \in \Hom_{\Q_p}(F_v, E)} \tau( k)^{-m_\tau}. \]
	\end{enumerate}
	Let $\mu \in (\Z^n_+)^{\Hom(F, E)}$ be the dominant weight defined by the formula $\mu_\tau = (m_\tau, \dots, m_\tau)$ for each $\tau \in \Hom(F, E)$. Then for any $\lambda \in  (\Z^n_+)^{\Hom(F, E)}$ there is an isomorphism 
	\[ R \Gamma(X_K, \cV_\lambda) \cong R \Gamma(X_K, \cV_{\lambda + \mu}) \]
	 in $\mathbf{D}(\cO)$ which is equivariant for the action of $\cH(G^S, K^S) \otimes_\bZ \cO$ when $\cH(G^S, K^S) \otimes_\bZ \cO$ acts in the usual way on the left-hand side and and by $f_\psi$ on the right-hand side.
\end{prop}
\begin{proof}
	The character $\psi$ defines a class in $H^0(X_K, \cV_\mu) = \Hom_{\mathrm{Sh}(X_K)}(\underline{\cO}, \cV_\mu)$. By tensor product this determines a morphism $\cV_\lambda \to \cV_{\lambda} \otimes_\cO \cV_{\mu} \cong \cV_{\lambda + \mu}$ of sheaves on $X_K$, hence a morphism $R \Gamma(X_K, \cV_\lambda) \to R \Gamma(X_K, \cV_{\lambda + \mu})$ in $\mathbf{D}(\cO)$. In order to determine how this morphism behaves with respect to the action of Hecke operators, we will repeat this calculation in $\mathbf{D}(\cH(G^S, K^S) \otimes_\bZ \cO)$.
	
	Let $\mathcal{A} = \Ind_{\GL_n(F)}^{\GL_n(\A_F^\infty)} \cO = H^0(\mathfrak{X}_G, \cO)$. There is an isomorphism 
	\[ H^0(\mathfrak{X}_G, \cV_\lambda) \cong \mathcal{A} \otimes_\cO \cV_\lambda \]
	of $\cO[\GL_n(\A_F^{\infty, S}) \times K_S)]$-modules, and hence a canonical isomorphism in $\mathbf{D}(\cH(G^S, K^S) \otimes_\bZ \cO)$:
	\[ R \Gamma(X_K, \cV_\lambda) \cong R \Gamma(K, \mathcal{A} \otimes_\cO \cV_\lambda). \]
	The same applies when $\lambda$ is replaced by any other dominant weight in $(\Z^n_+)^{\Hom(F, E)}$. The class $\psi$ in $H^0(X_K, \cV_\mu)$ corresponds to the $K$-equivariant map $g_\psi: \mathcal{A} \to \mathcal{A} \otimes_\cO \cV_\mu$ which sends a function $f : \GL_n(F) \backslash \GL_n(\A_{F}^\infty) \to \cO$ to $g_\psi(f)(g) = \psi(\det(g))f(g)$. The map $g_\psi$ becomes $G^S \times K_S$-equivariant when we twist the action on the source, giving
	\[ g_\psi : \mathcal{A} \to \mathcal{A} \otimes_\cO \cV_\mu(\psi^{-1,S}). \]
	By definition, the twist $(\psi^{-1,S})$ means that the action of an element $g \in G^S$ is twisted by $\psi(\det(g))^{-1}$. Taking the tensor product by $\cV_\lambda$ and then taking derived $K$-invariants gives a morphism
	\[   R \Gamma(X_K, \cV_\lambda) \to R \Gamma(X_K, \cV_{\lambda + \mu}(\psi^{-1,S})) \]
	in $\mathbf{D}(\cH(G^S, K^S) \otimes_\bZ \cO)$, hence a $\cH(G^S, K^S) \otimes_\bZ \cO$-equivariant isomorphism
	\[ R \Gamma(X_K, \cV_\lambda) \to R \Gamma(X_K, \cV_{\lambda + \mu}(\psi^{-1,S})) \]
	in $\mathbf{D}(\cO)$. The proof of the proposition is complete on noting that there is a canonical isomorphism 
	\[R \Gamma(X_K, \cV_{\lambda + \mu}(\psi^{-1,S})) \cong R \Gamma(X_K, \cV_{\lambda + \mu})  \]
	 in $\mathbf{D}(\cO)$, which is equivariant for the action of $\cH(G^S, K^S) \otimes_\bZ \cO$ when $\cH(G^S, K^S) \otimes_\bZ \cO$ acts in the natural way on the source and by $f_\psi$ on the target.
\end{proof}
\begin{cor}\label{cor:twisted_Hecke_alg}
	Suppose that $S$ contains the $p$-adic places of $F$, and let $K \subset \GL_n(\A_F^\infty)$ be a good subgroup such that $K_v = \GL_n(\cO_{F_v})$ for each place $v\not\in S$ of $F$. Let $\psi : G_F \to \cO^\times$ be a continuous character satisfying the following conditions:
	\begin{enumerate}
		\item For each finite place $v \nmid p$ of $F$, $\psi \circ \Art_{F_v}$ is trivial on $\det(K_v)$.
		\item There is $m = (m_\tau)_\tau \in \Z^{\Hom(F, E)}$ such that for each place $v | p$ of $F$, and for each $k \in \det(K_v)$, we have
		\[  \psi(\Art_{F_v}(k)) = \prod_{\tau \in \Hom_{\Q_p}(F_v, E)} \tau( k)^{-m_\tau}. \]
	\end{enumerate}
	Let $\mu \in (\Z^n_+)^{\Hom(F, E)}$ be the dominant weight defined by the formula $\mu_\tau = (m_\tau, \dots, m_\tau)$ for each $\tau \in \Hom(F, E)$. Then for any $\lambda \in  (\Z^n_+)^{\Hom(F, E)}$, $f_\psi$ descends to an isomorphism 
	\[ \T^S(K, \lambda) \cong \T^S(K, \lambda + \mu). \]
	 In particular, if $\m$ is a maximal ideal of $\T^S$ which is in the support of $H^\ast(X_K, \cV_\lambda)$, then $\m(\psi)$ is in the support of $H^\ast(X_K, \cV_{\lambda + \mu})$.
\end{cor}
\begin{proof}
	This is an immediate consequence of Proposition \ref{prop:twisting_by_character}.
\end{proof}

\subsection{Some automorphic Galois representations}

In the next two sections of this chapter, we state some results asserting the existence of Galois representations associated to automorphic forms. Although the main results of this paper concern the relation between classical automorphic representations and Galois representations, we must also consider the Galois representations associated to torsion classes, and therefore valued in (possibly $p$-torsion) Hecke algebras. This goes some way towards explaining the need to state so many closely related results here. A large part of this paper will be taken up with the problem of studying the local properties of the Hecke--algebra valued Galois representations whose existence is asserted in the statement of Theorem \ref{thm:existence_of_Hecke_repn_for_GL_n}.

\subsubsection{Existence of Galois representations attached to automorphic forms}

If $\pi$ is an irreducible admissible representation of $\GL_n(\A_F)$ and $\lambda \in (\Z^n_+)^{\Hom(F, \C)}$, we say that $\pi$ is of weight $\lambda$ if the infinitesimal character of $\pi_\infty$ is the same as that of $V_\lambda^\vee$. 
\begin{thm}\label{thm:char0galrepexists}
	Let $\pi$ be a cuspidal automorphic representation of $\GL_n(\A_F)$ of weight $\lambda \in (\Z^n_+)^{\Hom(F, \C)}$. Then for any isomorphism $\iota : \overline{\Q}_p \to \C$, there exists a continuous semisimple representation $r_\iota(\pi) : G_F \to \GL_n(\overline{\Q}_p)$ satisfying the following condition: for each prime $l \neq p$ above which both $F$ and $\pi$ are unramified, and for each place $v | l$ of $F$, $r_\iota(\pi)|_{G_{F_v}}$ is unramified and the characteristic polynomial of $r_\iota(\pi)(\Frob_v)$ is equal to the image of $P_v(X)$ in $\overline{\Q}_p[X]$ under the homomorphism $\T_v \to \overline{\bQ}_p$ associated to $\iota^{-1} \pi_v$. 
\end{thm}
\begin{proof}
	This is the main theorem of \cite{hltt}.
\end{proof}

\begin{thm}\label{thm:base_change_and_existence_of_Galois_for_tilde_G}
	Suppose that $F$ contains an imaginary quadratic field. Let $\pi$ be a cuspidal automorphic representation of $\widetilde{G}(\A_{F^+})$, and let $\xi$ be an irreducible algebraic representation of $\widetilde{G}_\C$ such that $\pi$ is $\xi$-cohomological. Then there exists a partition $2n = n_1 + \dots + n_r$ and discrete, conjugate self-dual automorphic representations $\Pi_1, \dots, \Pi_r$ of $\GL_{n_1}(\A_F), \dots, \GL_{n_r}(\A_F)$, satisfying the following conditions:
	\begin{enumerate}
		\item Let $\Pi = \Pi_1 \boxplus \dots \boxplus \Pi_r$. If $l$ is a prime unramified in $F$ and above which $\pi$ is unramified, then $\Pi$ is unramified above $l$ and for each place $v | l$ of $F$ lying above a place $\overline{v}$ of $F^+$, $\Pi_v$ and $\pi_{\overline{v}}$ are related by unramified base change.
		\item If  $F_0\subset F$ is an imaginary quadratic field and
                  $l'$ is a prime which splits in $F_0$, then for each place $v | l'$ of $F$ lying above a place $\overline{v}$ of $F^+$, $\Pi_v$ and $\pi_{\overline{v}}$ are identified under the induced isomorphism $\iota_v : \widetilde{G}(F^+_{\overline{v}}) \cong \GL_{2n}(F_v)$.
		\item The infinitesimal character of $\Pi$ is the same as that of the representation $(\xi \otimes \xi)^\vee$ of $\GL_{2n}(F \otimes_\Q \C)$. 
	\end{enumerate}
	Consequently\footnote{The fact that the $\Pi_i$ are not mentioned in these consequences is not an oversight. We use the Galois representations associated to the $\Pi_i$ in order to construct $r_{\iota}(\pi)$ and verify that it has the expected properties.}, there exists for any isomorphism $\iota: \overline{\Q}_p \to \C$ a continuous semisimple representation $r_\iota(\pi) : G_{F} \to \GL_{2n}(\overline{\Q}_p)$ satisfying the following conditions:
	\begin{enumerate}[(a)]
		\item For each prime $l \neq p$ which is unramified in $F$ and above which $\pi$ is unramified, and for each place $v | l$ of $F$, $r_\iota(\pi)|_{G_{F_v}}$ is unramified and the characteristic polynomial of $r_\iota(\pi)(\Frob_v)$ is equal to the image of $\widetilde{P}_v(X)$ in $\overline{\Q}_p[X]$.
		\item For each place $v | p$ of $F$, $r_\iota(\pi)$ is de Rham and for each embedding $\tau : F \hookrightarrow \overline{\Q}_p$, we have
		\[ \mathrm{HT}_\tau(r_\iota(\pi)) = \{ \widetilde{\lambda}_{\tau, 1} + 2n - 1, \widetilde{\lambda}_{\tau, 2} + 2n - 2, \dots, \widetilde{\lambda}_{\tau, 2n} \}, \]
		where $\widetilde{\lambda} \in (\Z^{2n}_+)^{\Hom(F, \overline{\Q}_p)}$ is the highest weight of the representation $\iota^{-1}(\xi \otimes \xi)^\vee$ of $\GL_{2n}$ over $\overline{\Q}_p$. 
		\item If $F_0\subset F$ is an imaginary quadratic
                  field and
                  $l$ is a prime which splits in $F_0$, then for each place $v | l$ of $F$ lying above a place $\overline{v}$ of $F^+$, there is an isomorphism $\mathrm{WD}(r_\iota(\pi)|_{G_{F_v}})^{\text{F-ss}} \cong \rec^T_{F_v}(\pi_{\overline{v}} \circ \iota_v)$.
	\end{enumerate}
\end{thm}
\begin{proof}
	We will deduce this from \cite{shin-basechange}. The main
        wrinkle is that this reference gives a case of base change for
        unitary similitude groups (while our group~$\widetilde{G}$ is
        a unitary group, with trivial similitude factor).  Let $\laux$ be 
        an auxiliary prime at which both $F$ and $\pi$ are unramified. %
        In order to prove the proposition, it suffices to prove
        the existence of an automorphic representation $\Pi$ of
        $\GL_{2n}(\A_F)$ satisfying the second and third requirements,
        and satisfying the first requirement at almost all rational
        primes, including $\laux$. %
      We can then use strong
        multiplicity 1 and our freedom to vary $\laux$ in order
        to recover the proposition as stated. The existence and local properties of the Galois representation are then a consequence of the existence of $\Pi$ (a result due to many people, but see e.g.~\cite{MR3272276}). 
	
	Let $\widetilde{G}'$ denote the similitude group associated to $\widetilde{G}$; thus there is a short exact sequence
	\[ \xymatrix@1{ 1 \ar[r] & \Res_{F^+ / \Q} \widetilde{G} \ar[r] & \widetilde{G}' \ar[r] & \mathbb{G}_m \ar[r] & 1 } \]
	of reductive groups over $\Q$. By the main result of \cite{shin-basechange}, it suffices to find an irreducible algebraic representation $\xi'$ of $\widetilde{G}'_\C$ and a cuspidal automorphic representation $\pi'$ of $\widetilde{G}'(\A_\Q)$ satisfying the following conditions:
	\begin{itemize}
		\item The restriction $\pi'|_{\widetilde{G}(\A_{F^+})}$ contains $\pi$.
		\item $\pi'$ is $\xi'$-cohomological.
		\item $\pi'$ is unramified at $\laux$.
	\end{itemize}
	Arguing as in the proof of \cite[Thm. VI.2.9]{ht}, we see that it is enough to show the existence of a continuous character $\psi : \A_{F_0}^\times / {F_0}^\times \to \C^\times$ satisfying the following conditions:
	\begin{itemize}
		\item The restriction $\psi|_{(\A_{F_0}^\times)^{c = 1}}$ is equal to the restriction of the central character $\omega_\pi : (\A_F^\times)^{c = 1} \to \C^\times$ of $\pi$ to $(\A_{F_0}^\times)^{c = 1}$.
		\item $\psi$ is of type $A_0$, i.e.\ its restriction to ${F}_{0,\infty}^\times$ arises from a character of the torus $(\Res_{F_0 / \Q} \mathbb{G}_m)_\C$.
		\item $\psi|_{\cO_{F_0, \laux}^\times}$ is trivial. 
	\end{itemize}
	The existence of such a character follows from the algebraicity of $\omega_\pi|_{(\A_{F_0}^\times)^{c = 1}}$, itself a consequence of the fact that $\pi$ is $\xi$-cohomological.
\end{proof}

\subsubsection{Existence of Hecke algebra-valued Galois representations}

Let $S$ be a finite set of finite places of $F$, containing the $p$-adic places.
\begin{theorem}\label{thm:existence_of_residual_representation_for_GL_n}
	Let $\ffrm \subset \T^S(K, \lambda)$ be a maximal ideal. Suppose that $S$ satisfies the following conditions:
	\begin{itemize}
	\item $S = S^c$. 
		\item Let $v$ be a finite place of $F$ not contained in $S$, and let $l$ be its residue characteristic. Then either $S$ contains no $l$-adic places of $F$ and $l$ is unramified in $F$, or there exists an imaginary quadratic field $F_0 \subset F$ in which $l$ splits.
	\end{itemize}
	Then there exists a continuous, semi-simple representation 
	\[ \overline{\rho}_\ffrm : G_{F, S} \to \GL_n(\T^S(K, \lambda) / \ffrm) \]
	satisfying the following condition: for each finite place $v \not\in S$ of $F$, the characteristic polynomial of $\overline{\rho}_\ffrm(\Frob_v)$ is equal to the image of $P_v(X)$	in $(\T(K, \lambda) / \ffrm)[X]$.
\end{theorem}
We note that our condition on $S$ can always be achieved after possibly enlarging $S$. 
\begin{proof}
	Fix an embedding $\T^S(K, \lambda) / \ffrm \hookrightarrow \overline{\mathbb{F}}_p$. According to ~\cite[Cor. 5.4.3]{scholze-torsion}, there is an $n$-dimensional 
	continuous semisimple Galois representation $\overline{\rho}_\ffrm : G_{F, S} \to \GL_n(\overline{\mathbb{F}}_p)$ such that for each finite place $v \not\in S$ of $F$, the characteristic polynomial of $\overline{\rho}_\ffrm(\Frob_v)$ is equal to the image of 
	\[
	X^n-T_{v, 1}X^{n-1} +\dots + (-1)^iq_v^{i(i-1)/2}T_{v, i}X^{n-i}+\dots + (-1)^n q_v^{n(n-1)/2}T_{v, n}
	\]
	in $\overline{\mathbb{F}}_p[X]$. (Our condition on $S$ ensures
        that we can appeal to the results of \cite{scholze-torsion} in
        a case where they are unconditional, cf.\ Theorem
        \ref{thm:base_change_and_existence_of_Galois_for_tilde_G} and the discussion in~\cite[Rem. 5.4.6]{scholze-torsion}).  Combining the Chebotarev density theorem, the Brauer--Nesbitt 
	Theorem and the vanishing of the Brauer group of a finite field 
	\cite[Lem.~6.13]{deligne-serre}, we see that $\overline{\rho}_\ffrm$ can in fact be realized over $\T^S(K, \lambda) / \ffrm$.
\end{proof}
\begin{defn}\label{dfn:non_Eisenstein} We say that a maximal ideal $\m \subset \T^S$ is of Galois type if its residue field is a finite extension of $k$, and there exists a continuous, semi-simple representation $\overline{\rho}_\m : G_{F, S} \to \GL_n(\T^S / \m)$ such that for each finite place $v \not\in S$ of $F$, the characteristic polynomial of $\overline{\rho}_\ffrm(\Frob_v)$ is equal to the image of $P_v(X)$	in $(\T^S / \m)[X]$. 

We say that a maximal ideal $\m \subset \T^S$ is non-Eisenstein if it is of Galois type and $\overline{\rho}_\m$ is absolutely irreducible. 
\end{defn}
Note that Theorem \ref{thm:existence_of_residual_representation_for_GL_n} can be viewed as asserting that, under a suitable condition on $S$, any maximal ideal of $\T^S$ in the support of $H^\ast(X_K, \cV_{\lambda})$ is of Galois type. We observe that if $\m \subset \T^S$ is of Galois type, then so is $\m^\vee$, and in fact $\overline{\rho}_{\m^\vee} \cong \overline{\rho}_\m^\vee \otimes \epsilon^{1-n}$. In particular, if $\m$ is non-Eisenstein, then so is $\m^\vee$. Similarly, if $\psi : G_{F, S} \to \cO^\times$ is a continuous character, and $\m \subset \T^S$ is a maximal ideal of Galois type, then so is $\m(\psi)$, and in fact $\overline{\rho}_{\m(\psi)} \cong \overline{\rho}_\m \otimes \overline{\psi}$. In particular, if $\m$ is non-Eisenstein, then so is $\m(\psi)$.

\begin{theorem}\label{thm:existence_of_Hecke_repn_for_GL_n}
	Let $\ffrm \subset \T^S(K, \lambda)$ be a maximal ideal. Suppose that $S$ satisfies the following conditions:
	\begin{itemize}
	\item $S = S^c$. 
		\item Let $v$ be a finite place of $F$ not contained in $S$, and let $l$ be its residue characteristic. Then either $S$ contains no $l$-adic places of $F$ and $l$ is unramified in $F$, or there exists an imaginary quadratic field $F_0 \subset F$ in which $l$ splits.
	\end{itemize}
	Suppose moreover that $\overline{\rho}_\ffrm$ is absolutely irreducible. Then there exists an integer $N \geq 1$, which depends only on $n$ and $[F : \Q]$, an ideal $I \subset \T^S(K, \lambda)$ satisfying $I^N = 0$, and a continuous homomorphism
	\[ \rho_\ffrm : G_{F, S} \to \GL_n(\T^S(K, \lambda)_\ffrm / I)  \]
	satisfying the following condition: for each finite place $v \not\in S$ of $F$, the characteristic polynomial of $\rho_\ffrm(\Frob_v)$ is equal to the image of $P_v(X)$	in $(\T(K, \lambda)_\ffrm / I)[X]$.
\end{theorem}
\begin{proof}
	This follows from \cite[Cor.\ 5.4.4]{scholze-torsion}. 
\end{proof}

\begin{theorem}\label{thm:existence_of_residual_representation_for_U(n,n)}
	Let $\widetilde{\ffrm} \subset \widetilde{\T}^S(\widetilde{K}, \widetilde{\lambda})$ be a maximal ideal. Suppose that $S$ satisfies the following condition:
	\begin{itemize}
		\item Let $v$ be a finite place of $F$ not contained in $S$, and let $l$ be its residue characteristic. Then either $S$ contains no $l$-adic places of $F$ and $l$ is unramified in $F$, or there exists an imaginary quadratic field $F_0 \subset F$ in which $l$ splits.
	\end{itemize}	
	(Note that the condition $S = S^c$ is implicit in the use of the notation $\widetilde{\T}^S$ here.)
	Then there is a continuous, semi-simple representation 
	\[ \overline{\rho}_{\widetilde{\ffrm}} : G_{F, S} \to \GL_{2n}(\widetilde{\T}^S(\widetilde{K}, \widetilde{\lambda})/\widetilde{\ffrm}) \]
	satisfying the following condition: for each finite place $v \not\in S$ of $F$, the characteristic polynomial of $\overline{\rho}_{\widetilde{\ffrm}}(\Frob_v)$ is equal to the image of $\widetilde{P}_v(X)$
	in  $(\widetilde{\T}^S(\widetilde{K}, \widetilde{\lambda}) / \widetilde{\ffrm} )[X]$.
\end{theorem}
\begin{proof}
	The existence of a $2n$-dimensional group determinant $\widetilde{D}_{\widetilde{\ffrm}}$ valued in $\widetilde{\T}^S(\widetilde{K}, \widetilde{\lambda}) / \widetilde{\ffrm}$ and with the given characteristic polynomials on Frobenius elements at places $v\not\in S$ is implicit in \cite{scholze-torsion} and also follows from \cite[Theorem 5.7]{new-tho}, as we now explain. The result \cite[Theorem 5.7]{new-tho} shows that if the group $\widetilde{K}$ is small, in the sense that there is a rational prime $q \neq p$ such that $\widetilde{K}_q$ is contained in the principal congruence subgroup at level $q$ (if $q$ is odd) or $2q$ (if $q = 2$), then there is even a $2n$-dimensional group determinant valued in $\widetilde{\T}^S( R \Gamma(\widetilde{X}_{\widetilde{K}}, \cV_{\widetilde{\lambda}} / (\varpi)) )$ such that for each finite place $v \not\in S$ of $F$, the characteristic polynomial of $\Frob_v$ is equal to the image of $\widetilde{P}_v(X)$. The surjection 
	\[ \widetilde{\T}^S(\widetilde{K}, \widetilde{\lambda}) \to \widetilde{\T}^S( R \Gamma(\widetilde{X}_{\widetilde{K}}, \cV_{\widetilde{\lambda}} / (\varpi)) ) \]
	is bijective at the level of maximal ideals, so this implies the existence of the desired group determinant when $\widetilde{K}$ is small. When $\widetilde{K}$ is not small, we choose an odd rational prime $q_1$ which is prime to $S$, and let $\widetilde{K}_1$ denote the intersection of $\widetilde{K}$ with the principal congruence subgroup of $\widetilde{G}(\widehat{\cO}_{F^+})$ of level $q_1$. Let $S_1$ denote the union of $S$ with the set of $q_1$-adic places of $F$. Then there is a diagram of Hecke algebras
	\[  \widetilde{\T}^{S_1}(\widetilde{K}_1, \widetilde{\lambda}) \leftarrow \widetilde{\T}^{S_1}(\widetilde{K} / \widetilde{K}_1, \widetilde{\lambda}) \twoheadrightarrow \widetilde{\bT}^{S_1}(\widetilde{K}, \widetilde{\lambda}) \hookrightarrow \widetilde{\bT}^{S}(\widetilde{K}, \widetilde{\lambda}), \]
	where the left-facing arrow has nilpotent kernel and so induces a bijection at the level of maximal ideals. Let $\widetilde{\ffrm}_1 \subset \widetilde{\bT}^{S_1}(\widetilde{K}, \widetilde{\lambda})$ denote the pullback of $\widetilde{\ffrm}$ along the right-hand inclusion. Since $\widetilde{K}_1$ is small, there exists a group determinant $\widetilde{D}_{\widetilde{\ffrm}_1}$ valued in $\widetilde{\T}^{S_1}(\widetilde{K}, \widetilde{\lambda})  / {\widetilde{\ffrm}_1}$ and with the correct characteristic polynomials at places outside $S_1$. Let $\widetilde{D}_{\widetilde{\ffrm}, 1}$ denote the pushforward of $\widetilde{D}_{\widetilde{\ffrm}_1}$ to $\widetilde{\T}^S(\widetilde{K}, \widetilde{\lambda}) / \widetilde{\ffrm}$.  Thus $\widetilde{D}_{\widetilde{\ffrm}, 1}$ is a $2n$-dimensional group determinant of $G_{F, S_1}$ with the property that for any finite place $v \not\in S_1$ of $F$, $\widetilde{D}_{\widetilde{\ffrm}, 1}(X - \Frob_v)$ equals the image of $\widetilde{P}_v(X)$.
	
	Choose another odd rational prime $q_2 \neq q_1$ which is prime to $S$, and repeat this construction to obtain a group determinant $\widetilde{D}_{\widetilde{\ffrm}, 2}$ of $G_{F, S_2}$ valued in $\widetilde{\T}^S(\widetilde{K}, \widetilde{\lambda}) / \widetilde{\ffrm}$ with the property that for any finite place $v \not\in S_2$ of $F$, $\widetilde{D}_{\widetilde{\ffrm}, 2}(X - \Frob_v)$ equals the image of $\widetilde{P}_v(X)$. Since the Frobenius elements at places $v \not\in S_1 \cup S_2$ are dense in $G_{F, S_1 \cup S_2}$, the group determinants $\widetilde{D}_{\widetilde{\ffrm}, 1}$ and $\widetilde{D}_{\widetilde{\ffrm}, 2}$ have the same characteristic polynomials on all elements of $G_{F, S_1 \cup S_2}$. By \cite[Lemma 1.12]{chenevier_det}, these group determinants are equal and we can take $\widetilde{D}_{\widetilde{\ffrm}} = \widetilde{D}_{\widetilde{\ffrm}, 1} = \widetilde{D}_{\widetilde{\ffrm}, 2}$. 
	
To obtain a true representation from this group determinant, we first fix an embedding $\widetilde{\T}^S(\widetilde{K}, \widetilde{\lambda}) / \widetilde{\ffrm} \hookrightarrow \overline{\mathbb{F}}_p$. The group determinant determines a representation over $\overline{\mathbb{F}}_p$, by~\cite[Theorem A]{chenevier_det}. It follows by the same argument as in the proof of Theorem \ref{thm:existence_of_residual_representation_for_GL_n} that this representation can in fact be realized over $\widetilde{\T}^S(\widetilde{K}, \widetilde{\lambda}) / \widetilde{\ffrm}$.
\end{proof}
A similar argument shows that \cite[Theorem 5.7]{new-tho} implies the following result.
\begin{prop}\label{prop:existence_of_hecke_representation_for_U(n,n)_no_R}
 Suppose that $S$ satisfies the following conditions:
	\begin{itemize}
	\item $S = S^c$. 
		\item Let $v$ be a finite place of $F$ not contained in $S$, and let $l$ be its residue characteristic. Then either $S$ contains no $l$-adic places of $F$ and $l$ is unramified in $F$, or there exists an imaginary quadratic field $F_0 \subset F$ in which $l$ splits.
	\end{itemize}	
	Then there exists an ideal $I \subset \widetilde{\T}^S(\widetilde{K}, \widetilde{\lambda})$ satisfying $I^{2 \dim_{\bR} \widetilde{X}} = 0$ and a $2n$-dimensional group determinant $\widetilde{D}$ of $G_{F, S}$ valued in $\widetilde{\T}^S(\widetilde{K}, \widetilde{\lambda} )/ I$ such that for each finite place $v \not\in S$ of $F$, the characteristic polynomial $\widetilde{D}(X - \Frob_v)$ is equal to the image of $\widetilde{P}_v(X)$
	in  $(\widetilde{\T}^S(\widetilde{K}, \widetilde{\lambda}) / I )[X]$.
\end{prop}
\begin{proof}
When $\widetilde{K}$ is small, this is an immediate consequence of \cite[Theorem 5.7]{new-tho}, together with the observation that the natural map
\[ \widetilde{\T}^S(\widetilde{K}, \widetilde{\lambda}) \to \varprojlim_{m \geq 1} \widetilde{\T}^S( R \Gamma(\widetilde{X}_{\widetilde{K}}, \cV_{\widetilde{\lambda}} / (\varpi^m))) \]
is an isomorphism. Moreover, in this case we can take $I = 0$. In general, we introduce an odd rational prime $q_1$ as in the proof of Theorem \ref{thm:existence_of_residual_representation_for_U(n,n)} and consider again the diagram
	\[  \widetilde{\T}^{S_1}(\widetilde{K}_1, \widetilde{\lambda}) \leftarrow \widetilde{\T}^{S_1}(\widetilde{K} / \widetilde{K}_1, \widetilde{\lambda}) \twoheadrightarrow \widetilde{\bT}^{S_1}(\widetilde{K}, \widetilde{\lambda}) \hookrightarrow \widetilde{\bT}^{S}(\widetilde{K}, \widetilde{\lambda}). \]
	The map $\widetilde{\T}^{S_1}(\widetilde{K} / \widetilde{K}_1, \widetilde{\lambda}) \to \widetilde{\T}^{S_1}(\widetilde{K}_1, \widetilde{\lambda})$ has kernel $J_1$ satisfying $J_1^{\dim_\bR \widetilde{X}} = 0$ (because the cohomology of $R \Gamma(\widetilde{X}_{\widetilde{K}_1}, \cV_{\widetilde{\lambda}})$ is 0 for degrees not lying in $[0, \dim_{\bR} \widetilde{X} - 1]$). Taking $I_1$ to be the ideal of $\widetilde{\bT}^{S_1}(\widetilde{K}, \widetilde{\lambda})$ generated by the image of $J_1$, we obtain a $2n$-dimensional group determinant $\widetilde{D}_1$ of $G_{F, S_1}$ valued in $\widetilde{\T}^S(\widetilde{K}, \widetilde{\lambda}) / I_1$ such that for each finite place $v\not\in S_1$ of $F$, $\widetilde{D}_1(X - \Frob_v)$ equals the image of $\widetilde{P}_v(X)$, and moreover $I_1^{\dim_\bR \widetilde{X}} = 0$.
	
	Introducing an odd rational prime $q_2 \neq q_1$ which is prime to $S$, we obtain similarly an ideal $I_2 \subset \widetilde{\T}^S(\widetilde{K}, \widetilde{\lambda}) $ satisfying $I_2^{\dim_\bR \widetilde{X}} = 0$ and a group determinant $\widetilde{D}_2$ valued in  $\widetilde{\T}^S(\widetilde{K}, \widetilde{\lambda})  / I_2$ and having properties analogous to $\widetilde{D}_1$. We take $I = (I_1, I_2)$ and $\widetilde{D}$ to be the projection of $\widetilde{D}_1$ to $\widetilde{\T}^S(\widetilde{K}, \widetilde{\lambda}) / I$. Consideration of characteristic polynomials at places $v\not\in S_1 \cup S_2$, as in the proof of Theorem \ref{thm:existence_of_residual_representation_for_U(n,n)}, shows that $\widetilde{D}$ equals the projection of $\widetilde{D}_2$ to $\widetilde{\T}^S(\widetilde{K}, \widetilde{\lambda}) / I$. It follows that $\widetilde{D}$ has the property required by the proposition. 
\end{proof}

\subsection{Boundary cohomology}\label{subsec: boundary cohomology}

In the remaining section of this chapter we prove some results
about the boundary cohomology of the arithmetic locally symmetric
spaces of $G$ and $\widetilde{G}$. This is made possible by the
existence of Galois representations attached to Hecke eigenclasses in
the cohomology of these groups and of their Levi subgroups. The
important observation is usually that the cohomology of a certain
stratum in the boundary can be observed to vanish after localization
at a sufficiently nice (e.g.\ non-Eisenstein) maximal ideal of  a suitable Hecke algebra. 

\subsubsection{The Siegel parabolic}\label{sec:siegel}

Let $\widetilde{K}\subset \widetilde{G}(\A_{F^+}^{\infty})$ be a good subgroup which is decomposed with respect to the Levi decomposition $P = GU$ (cf. \S \ref{sec:general_definition_of_hecke_operators}). We set $K = \widetilde{K} \cap G(\A_{F^+}^\infty)$ and $K_U = \widetilde{K} \cap U(\A_{F^+}^\infty)$.

Let $\m \subset \T^S(K, \lambda)$ be a non-Eisenstein maximal ideal, and suppose that $S = S^c$. Let $\widetilde{\m} \subset \widetilde{\T}^S$ denote the pullback of $\m$ under the homomorphism $\cS  : \widetilde{\T}^S \to \T^S$ defined in~\eqref{eqn:satake T}. In order to state the first main result of this subsection, we recall that the boundary 
$\partial\widetilde{X}_{\widetilde{K}} = \overline{\widetilde{X}}_{\widetilde{K}} - \widetilde{X}_{\widetilde{K}}$ of the Borel--Serre compactification of $\widetilde{X}_{\widetilde{K}}$ has a 
$\widetilde{G}(\A_{F^+}^{\infty})$-equivariant stratification indexed by 
the parabolic subgroups of $\widetilde{G}$ which contain $B$. See~\cite[\S 
3.1.2]{new-tho}, especially \cite[Lem.~3.10]{new-tho} %
for more details. For such a standard parabolic subgroup $Q$, 
we denote by $\widetilde{X}^{Q}_{\widetilde{K}}$ the stratum
labelled by $Q$. 
The stratum $\widetilde{X}^{Q}_{\widetilde{K}}$ can be written as
\[
\widetilde{X}^{Q}_{\widetilde{K}} = Q(F^+) \backslash \left(X^Q \times 
\widetilde{G}(\A^\infty_{F^+}) /\widetilde{K}\right). 
\]
As discussed in \S \ref{sec:general_definition_of_hecke_operators}, there is, for any $\widetilde{\lambda} \in (\Z^{2n}_+)^{\Hom(F^+, E)}$, a homomorphism
\[ \widetilde{\mathbb{T}}^S \to \End_{\mathbf{D}(\cO)} ( R\Gamma(\widetilde{X}^{Q}_{\widetilde{K}},\cV_{\widetilde{\lambda}})). \]
Therefore, we can define the localization 
$R\Gamma(\widetilde{X}^{Q}_{\widetilde{K}},\cV_{\widetilde{\lambda}})_{\widetilde{\m}}$. (This complex will be non-zero in $\mathbf{D}(\cO)$ if and only if the maximal ideal $\widetilde{\m}$ of $\widetilde{\T}^S$ occurs in the support of  the cohomology groups $H^\ast( \widetilde{X}^{Q}_{\widetilde{K}},\cV_{\widetilde{\lambda}})$.)
\begin{thm}\label{thm:reduction_to_Siegel} Let $\m \subset \T^S(K, \lambda)$ be a non-Eisenstein maximal ideal and let $\widetilde{\m} = \cS^\ast(\m) \subset \widetilde{\T}^S$. Let $\widetilde{\lambda} \in (\Z^{2n}_+)^{\Hom(F^+, E)}$. Then there is a natural $\widetilde{\mathbb{T}}^S$-equivariant isomorphism 
	\[
	R\Gamma(\widetilde{X}^{P}_{\widetilde{K}},\cV_{\widetilde{\lambda}})_{\widetilde{\m}}\toisom
	R\Gamma (\partial 
	\widetilde{X}_{\widetilde{K}},\mathcal{V}_{\widetilde{\lambda}})_{\widetilde{\m}}
	\]
	in $\mathbf{D}(\cO)$. 
\end{thm}

\begin{proof} There is no  harm in enlarging $S$, so we first add finitely many places to $S$, ensuring that it satisfies the condition of Theorem \ref{thm:existence_of_residual_representation_for_GL_n}. The proof is similar to the proof of~\cite[Thm.\ 4.2]{new-tho}, 
	which applies to the case of $\mathrm{Res}_{F/\Q}\GL_n$ and which shows that 
	the cohomology of the stratum labelled by any proper parabolic subgroup of 
	$\mathrm{Res}_{F/\Q}\GL_n$ is Eisenstein. Since $P$ is a maximal parabolic of 
	$\widetilde{G}$, the inclusion $\widetilde{X}^{P}_{\widetilde{K}} \subset \partial 
	\widetilde{X}_{\widetilde{K}}$ is an open embedding, which induces a natural, $\widetilde{\mathbb{T}}^S$-equivariant map 
	\[
	R\Gamma_c(\widetilde{X}^{P}_{\widetilde{K}},\cV_{\widetilde{\lambda}})_{\widetilde{\m}}
	\to
	R\Gamma (\partial 
	\widetilde{X}_{\widetilde{K}},\mathcal{V}_{\widetilde{\lambda}})_{\widetilde{\m}},
	\]
	and which fits into an excision distinguished triangle
	\[
	R\Gamma_c(\widetilde{X}^{P}_{\widetilde{K}},\cV_{\widetilde{\lambda}})_{\widetilde{\m}}
	\to 
	R\Gamma (\partial 
	\widetilde{X}_{\widetilde{K}},\mathcal{V}_{\widetilde{\lambda}})_{\widetilde{\m}}
	\to R\Gamma (\partial \widetilde{X}_{\widetilde{K}} \setminus 
	\widetilde{X}^{P}_{\widetilde{K}},\mathcal{V}_{\widetilde{\lambda}})_{\widetilde{\m}}\stackrel{[1]}{\to}.
	\]
    We will show that $R\Gamma (\partial \widetilde{X}_{\widetilde{K}} \setminus 
    \widetilde{X}^{P}_{\widetilde{K}},\mathcal{V}_{\widetilde{\lambda}})_{\widetilde{\m}} = 0$ in $\mathbf{D}(\cO)$, by showing that for each standard proper parabolic subgroup $Q \subset \widetilde{G}$ with $Q \neq P$, we have $R\Gamma_c(\widetilde{X}^{Q}_{\widetilde{K}},\cV_{\widetilde{\lambda}})_{\widetilde{\m}} = 0$ in $\mathbf{D}(\cO)$. Applying the same argument to the excision triangle for the inclusion from $\widetilde{X}^{P}_{\widetilde{K}}$ to its closure, this will also show that the natural map
    \[ R\Gamma_c(\widetilde{X}^{P}_{\widetilde{K}},\cV_{\widetilde{\lambda}})_{\widetilde{\m}} \to R\Gamma(\widetilde{X}^{P}_{\widetilde{K}},\cV_{\widetilde{\lambda}})_{\widetilde{\m}} \]
    is an isomorphism.
    
    In order to show this vanishing, it suffices (after possibly
    shrinking $\widetilde{K}$ at the $p$-adic places of $F^+$) to show
    that if $Q \neq P$ is a standard proper parabolic subgroup of
    $\widetilde{G}$, then
    $R\Gamma(\widetilde{X}^{Q}_{\widetilde{K}},k)_{\widetilde{\m}} =
    0$. (We are using here that if $C^\bullet$ is a perfect complex in
    $\mathbf{D}(\cO)$, then $C^\bullet = 0$ in $\mathbf{D}(\cO)$ if
    and only if $C^\bullet \otimes^\mathbf{L}_\cO k = 0$ in
    $\mathbf{D}(k)$. We are also using Poincar\'e duality to exchange
    cohomology with compact support for usual cohomology, as in
    \cite[Prop.\ 3.7]{new-tho}.)
	
	We will in fact show that, for any maximal ideal $\widetilde{\m}'\subset 
	\widetilde{\mathbb{T}}^S$ in the support of 
	$R\Gamma(\widetilde{X}^{Q}_{\widetilde{K}},k)$, there exists a semisimple 
	residual Galois representation 
	\[
	\bar\rho_{\widetilde{\m}'}: G_{F, S} \to 
	\GL_{2n}(\widetilde{\T}^S / \widetilde{\m}')
	\]  
	such that for each place $v \not\in S$ of $F$, the characteristic polynomial of $\overline{\rho}_{\widetilde{\m}'}$ equals the image of $\widetilde{P}_v(X)$ in $(\widetilde{\T}^S / \widetilde{\m}')[X]$. Moreover, assume that the Levi component $M$ of 
	$Q$ is of the form 
	\[
	\mathrm{Res}_{F/F^+}\GL_{n_1}\times \dots \times 
	\mathrm{Res}_{F/F^+}\GL_{n_r}\times \widetilde{G}_{n-s}
	\]
	for integers $r\geq 1, n_i\geq1, s\in \{1,\dots,n\}$ satisfying 
	$\sum_{i=1}^r n_i =s$. (More precisely, that it is the block diagonal subgroup of $\widetilde{G}$ associated to the partition $2n = n_1 + \dots + n_r + 2(n-s) + n_r + \dots + n_1$. These describe 
	all the standard $F^+$-rational Levi subgroups of $\widetilde{G}$.) Then we 
	have 
	\numequation\label{eqn:reducibility_for_parabolic}
	\bar{\rho}_{\widetilde{\m}'} = \oplus_{i=1}^r\bar{\rho}'_i \oplus 
	\bar{\rho}'(n-s) \oplus_{i=1}^r (\bar{\rho}'_i)^{c, \vee},
	\end{equation}
	where $\bar{\rho}'_i$ is $n_i$-dimensional and $\bar{\rho}'(n-s)$ is 
	$(2n-2s)$-dimensional. The non-Eisenstein condition on $\m$ implies that 
	\[
	\bar{\rho}_{\widetilde{\m}} = \bar{\rho}_1 \oplus \bar{\rho}_2,
	\]
	where both $\bar{\rho}_1, \bar{\rho}_2$ are (absolutely) irreducible 
	$n$-dimensional representations. 
	This shows that, unless $r=1$ and $s=n$, 
	$R\Gamma(\widetilde{X}^{Q}_{\widetilde{K}},k)_{\widetilde{\m}} = 0$. The 
	case $r=1,s=n$ corresponds precisely to the Siegel parabolic $P$. 
	
	Let us define $\T^S_Q = \cH( Q^S, \widetilde{K}_Q^S)
        \otimes_\Z \cO$ and $\T^S_M = \cH( M^S,
        \widetilde{K}_M^S) \otimes_\bZ \cO$. We recall (cf. \S
        \ref{sec:general_definition_of_hecke_operators}) that there
        are homomorphisms $r_Q : \widetilde{\T}^S \to \T^S_Q$ and $r_M
        : \T^S_Q \to \T^S_M$, and that we set $\cS^{\widetilde{G}}_M =
        r_M \circ r_Q$.  %
        We first claim that, for any maximal ideal $\widetilde{\m}'$ of $\widetilde{\T}^S$ in the support of $H^\ast(\widetilde{X}_{\widetilde{K}}^Q, k)$, there exists a good subgroup $\widetilde{K}_M' \subset \widetilde{K}_M$ with $(\widetilde{K}_M')^S = \widetilde{K}_M^S$ and a maximal ideal $\m'$ of $\T^S_M$ in the support of $H^\ast(X^M_{\widetilde{K}'_M}, k)$ such that $\widetilde{\m}' = \cS^{\widetilde{G}, \ast}_M(\m')$. This follows the same steps as the proof of~\cite[Theorem 4.2]{new-tho}, which we outline here. 
        
        Firstly, one can describe the cohomology 
	$R\Gamma(\widetilde{X}^{Q}_{\widetilde{K}},k)$ together with its  
	$\widetilde{\mathbb{T}}^S$-action in terms of the pullback under $r_{Q}: 
	\widetilde{\mathbb{T}}^S\to \mathbb{T}^S_Q$ of the cohomology of finitely 
	many locally symmetric spaces for $Q$. More precisely, using the Iwasawa decomposition away from $S$, 
	we can write 
	$\widetilde{G}(\A^\infty_{F^+}) = \bigsqcup_{i=1}^{r}Q(\A^\infty_{F^+}) g_i 
	\widetilde{K}$ and obtain $r$ locally symmetric 
	spaces $X^Q_{K_{Q,i}}$, with $K_{Q,i}:=Q(\A^\infty_{F^+})\cap g_i\widetilde{K}(g_i)^{-1}$,
	together with an isomorphism 
	\[
	R\Gamma(\widetilde{X}^{Q}_{\widetilde{K}}, 
	k)\simeq\bigoplus_{i=1}^{r} 
	r_Q^*\left(R\Gamma(X^Q_{K_{Q,i}}, k)\right)
	\]
	in $\mathbf{D}(\widetilde{\mathbb{T}}^S)$. The proof in~\cite{new-tho}, 
	which identifies equations (4.2) and (4.3) of \emph{loc. cit.}, 
	applies verbatim to our situation, so we do not repeat it here. 
	
	Secondly, fix a neat compact open subgroup $K_{Q}\subset Q(\A_{F^+}^\infty)$, which 
	can be any of the $K_{Q,i}$ considered above. Let $Q = 
	M\ltimes N$ be a Levi decomposition of $Q$. Let $K_{M}$ be the image of $K_{Q}$ 
	in $M(\A_{F^+}^\infty)$ and $K_N:= K_Q\cap N(\A_{F^+}^\infty)$.
	 Let $\cW$ be the object in the derived category 
	of sheaves on $X^M_{K_M}$ corresponding to the object $R\Gamma(K_{N,S}, 
	k)$ in $D(k[K_{M,S}])$ under the 
	formalism in section~\ref{sec:general_definition_of_hecke_operators}. Then, using an
	argument that is formally identical to that on pages 56-58 of~\cite{new-tho}, 
	we obtain an isomorphism 
	\[
	R\Gamma(X^{Q}_{K_Q}, k) \simeq 
	r_M^*\left(R\Gamma(X^M_{K_M}, \cW)\right)
	\]
	in $D(\mathbb{T}^S_Q)$. The corresponding statement in \emph{loc. cit.} is obtained
	by combining the second and fourth displayed equations on page 58.  
	
	Finally, we consider the spectral sequence that computes the total cohomology of 
	$R\Gamma(X^M_{K_M}, \cW)$. Let $\cW^i$ be the local systems on $X^M_{K_M}$ 
	corresponding to the cohomology groups $H^i(K_{N,S}, k)$. The first two steps above 
	show that there exists a maximal ideal $\m'$ of $\T^S_M$ in the support of some
	$H^\ast(X^M_{\widetilde{K}'_M}, \cW^i)$ such that $\widetilde{\m}' = \cS^{\widetilde{G}, \ast}_M(\m')$. 
	It remains to shrink the level $K_{M,S}$ in order to trivialize all the
	$\cW^i$. This could, a priori, cause a problem because
	the map on cohomology groups need not be injective. 
	However, we are only interested in keeping track of a maximal ideal $\m'$ 
	of $\mathbb{T}^S_{M}$. The Hochschild--Serre spectral sequence
	shows that shrinking the level does not cause problems, as in the proof of~\cite[Lemma 4.3]{new-tho}. (See also the example 
	below Lemma~\ref{lem_cohomology_nilpotent_annihilator} for an illustration 
	of the same phenomenon in the derived category.) 
		
    In order to complete the proof of the theorem, it therefore suffices to show that for any good subgroup $K_M \subset M(\A_F^\infty)$ with $K_M^S = \widetilde{K}_M^S$ and for any maximal ideal $\m'$ of $\T^S_M$ in the support of $H^\ast(X^M_{K_M}, k)$, there exists a semisimple 
    residual Galois representation 
    \[
    \bar\rho_{\cS^\ast(\m')}: G_{F, S} \to 
    \GL_{2n}(\widetilde{\T}^S / \cS^\ast(\m'))
    \]  
    such that for each place $v \not\in S$ of $F$, the characteristic polynomial of $\overline{\rho}_{\cS^\ast(\m')}$ equals the image of $\widetilde{P}_v(X)$ in $(\widetilde{\T}^S / \widetilde{\m}')[X]$; and moreover, that this Galois representation admits a decomposition of the form (\ref{eqn:reducibility_for_parabolic}). 
    
    After possibly shrinking $K_M$ once more, we can assume that it admits a
    decomposition $K_M = K_1 \times \dots \times K_r \times
    \widetilde{K}_s$, where $K_i  \subset \GL_{n_i}(\A_F^\infty)$ and
    $\widetilde{K}_s \subset
    \widetilde{G}_{n-s}(\A_{F^+}^\infty)$. After possibly enlarging
    $k$, we can moreover assume, in the obvious notation, (by the
    K\"unneth formula) that there exist maximal ideals $\m_1, \dots,
    \m_r, \widetilde{\m}_s$ of the Hecke algebras $\T^S_{\GL_{n_1}},
    \dots, \T^S_{\GL_{n_r}}, \T^S_{\widetilde{G}_{n-s}}$,
    respectively, which are in the supports of the groups
    $H^\ast(X^{\GL_{n_1}}_{K_1}, k),\dots,H^\ast(X^{\GL_{n_r}}_{K_r},
    k), H^\ast(X^{\widetilde{G}_{n-s}}_{\widetilde{K}_s}, k)$, respectively, and such that $\m'$ is identified with $(\m_1, \dots, \m_r, \widetilde{\m}_s)$ under the isomorphism
    \[ \T^S_M \to \T^S_{\GL_{n_1}} \otimes_\cO \otimes \dots \otimes_\cO \T^S_{\GL_{n_r}} \otimes_\cO \T^S_{\widetilde{G}_{n-s}}. \]
    We can moreover assume that all of the the maximal ideals $\m_1, \dots, \m_r, \widetilde{\m}_s$ and $\m'$ have residue field $k$.
    
    Let us write $P_v^{i}(X) \in \T^S_{\GL_{n_i}}[X]$ and $\widetilde{P}_v^s(X) \in \T^S_{\widetilde{G}_{n-s}}[X]$ for the analogues for the groups $\GL_{n_i}$ and $\widetilde{G}_{n-s}$ of the Hecke polynomials defined in \S \ref{sec:some_useful_hecke_operators}. By Theorem \ref{thm:existence_of_residual_representation_for_GL_n} and Theorem \ref{thm:existence_of_residual_representation_for_U(n,n)}, there exist continuous, semi-simple representations
    \[ \overline{\rho}_{\m_i} : G_{F, S} \to \GL_{n_i}( k ) \,\, (i = 1, \dots, r) \]
    and
    \[ \overline{\rho}_{\widetilde{\m}_s} : G_{F, S} \to \GL_{2(n-s)}(k) \]
    such that for each finite place $v \not\in S$ of $F$ and for each $i = 1, \dots, r$, the characteristic polynomial of $\overline{\rho}_{\m_i}(\Frob_v)$ is equal to $P_v^i(X) \text{ mod }\m_i$; and the characteristic polynomial of $\overline{\rho}_{\widetilde{\m}_s}(\Frob_v)$ is equal to $\widetilde{P}_v^s(X) \text{ mod }\widetilde{\m}_s$.
    
    The proof of the theorem is complete on noting that we can take 
    \[ \overline{\rho}_{\cS^\ast(\m')} = \bigoplus_{i=1}^r \left( \overline{\rho}_{\m_i} \otimes \epsilon^{n_1 + \dots + n_i - 2n} \oplus  \overline{\rho}_{\m_i}^{c, \vee} \otimes \epsilon^{1 - (n_1 + \dots + n_i)} \right) \oplus \overline{\rho}_{\widetilde{\m}_s} \otimes \epsilon^{-s}.  \]
    That this choice is valid rests on the computation of the image of $\widetilde{P}_v(X)$ under the map $\cS^{\widetilde{G}}_M$. The details are very similar to the proof of \cite[Prop.-Def. 5.3]{new-tho}, and are omitted. \end{proof}

We can now state the second main result of this subsection, which takes Theorem \ref{thm:reduction_to_Siegel} as its starting point. 
\begin{theorem}\label{thm:Hecke_reduction_to_Siegel}
	Let $\widetilde{K} \subset \widetilde{G}(\A_{F^+}^\infty)$ be as at the start of \S \ref{sec:siegel}, and let $\lambda \in (\Z^n_+)^{\Hom(F, E)}$ be a dominant weight whose image in $(\Z^{2n})^{\Hom(F^+, E)}$ is $\widetilde{G}$-dominant. Let $\m \subset \T^S(K, \lambda)$ be a non-Eisenstein maximal ideal, and let $\widetilde{\m} \subset \widetilde{\T}^S$ denote its pullback under the homomorphism $\cS : \widetilde{\T}^S \to \T^S$.  Then the homomorphism $\cS: \widetilde{\T}^S \to \T^S$ descends to a homomorphism
	\[  \widetilde{\T}^S( R \Gamma(\partial \widetilde{X}_{\widetilde{K}}, \cV_{\widetilde{\lambda}})_{\widetilde{\m}}) \to \T^S(R \Gamma(X_K, \cV_\lambda)_\m).\]
\end{theorem}
\begin{proof}
	The first step in the proof is to note that it suffices to show that $\cS$ descends to a homomorphism
	\[  \widetilde{\T}^S( R \Gamma(\partial \widetilde{X}^P_{\widetilde{K}}, \cV_{\widetilde{\lambda}})) \to \T^S( R \Gamma(X_K, \cV_\lambda)), \]
	by Theorem \ref{thm:reduction_to_Siegel}. On the other hand, the discussion at the end of \S \ref{sec:hecke_algebra_of_a_monoid} shows that $\cS$ descends to a homomorphism
	\numequation \widetilde{\T}^S( R \Gamma(\partial \widetilde{X}^P_{\widetilde{K}}, \cV_{\widetilde{\lambda}})) \to \widetilde{\T}^S( R \Gamma(X^P_{\widetilde{K}_P}, \cV_{\widetilde{\lambda}} )), 
	\end{equation}
	where $\widetilde{\T}^S$ acts on the latter complex via $r_P$. It therefore suffices to show that $\cS$ descends to a homomorphism
	\numequation\label{eqn:reduction_to_Siegel_parabolic} \widetilde{\T}^S(  R \Gamma(X^P_{\widetilde{K}_P}, \cV_{\widetilde{\lambda}} )) \to \widetilde{\T}^S(  R \Gamma(X_K, \cV_{\lambda} )), 
	\end{equation}
	where $\widetilde{\T}^S$ acts on the latter cohomology groups via $\cS = r_G \circ r_P$. In fact, it even suffices to show that for each $m \geq 1$, $\cS$ descends to a homomorphism
	\numequation\label{eqn:reduction_to_Siegel_parabolic_modulo_m} \widetilde{\T}^S(  R \Gamma(X^P_{\widetilde{K}_P}, \cV_{\widetilde{\lambda}} / \varpi^m )) \to \widetilde{\T}^S(  R \Gamma(X_K, \cV_{\lambda} / \varpi^m )), 
\end{equation}
cf.~\cite[Lemma 3.12]{new-tho}.
	
	However, arguing in the same way as on \cite[p. 58]{new-tho}, we see that there is an isomorphism
	\[ R \Gamma(X^P_{\widetilde{K}_P}, \cV_{\widetilde{\lambda}} / \varpi^m) \cong R \Gamma(\widetilde{K}_P^S \times K_S, R \Gamma( \operatorname{Inf}_{G^S \times K_S}^{P^S \times K_S} \mathfrak{X}_G, R 1_\ast^{\widetilde{K}_{U, S}} \cV_{\widetilde{\lambda}}/ \varpi^m)), \]
	where the derived pushforward sends (complexes of) $P^S \times \widetilde{K}_{P, S}$-equivariant sheaves on $\mathfrak{X}_G$ to $P^S \times  K_S$-equivariant sheaves on $\mathfrak{X}_G$. Suppose we knew that $\cV_\lambda/ \varpi^m$ was a direct summand of $R 1_\ast^{\widetilde{K}_{U, S}} \cV_{\widetilde{\lambda}}/ \varpi^m$; then we could conclude, by arguing in the same way as at the top of \cite[p. 59]{new-tho}, that $r_G^\ast R \Gamma(X_K, \cV_\lambda/ \varpi^m)$ is isomorphic to a direct summand of $R \Gamma(X^P_{\widetilde{K}_P}, \cV_{\widetilde{\lambda}}/ \varpi^m)$ in the category 
	\[ \mathbf{D}(\cH(P^S  \times \widetilde{K}_{P, S}, \widetilde{K}_P) \otimes_\Z \cO / \varpi^m), \]
	 implying the existence of the homomorphism (\ref{eqn:reduction_to_Siegel_parabolic_modulo_m}).
	
	It remains to construct the desired splitting of $\cV_\lambda / \varpi^m$ as a direct summand of $R 1_\ast^{\widetilde{K}_{U, S}} \cV_{\widetilde{\lambda}} / \varpi^m$. To do this, we recall the following two facts:
	\begin{itemize}
		\item $\widetilde{K}_P$ is a semidirect product $\widetilde{K}_P = \widetilde{K}_U \rtimes K$ (by assumption: $\widetilde{K}$ is decomposed with respect to the Levi decomposition $P = GU$).
		\item There is a $\widetilde{K}_P$-equivariant embedding $\cV_\lambda \to \cV_{\widetilde{\lambda}}$, which splits after restriction to $K$. (This follows from~\cite[Corollary 2.11]{new-tho}.)
	\end{itemize}
	The morphism $\cV_\lambda / \varpi^m\to R 1_\ast^{\widetilde{K}_{U, S}} \cV_{\widetilde{\lambda}}/ \varpi^m$ is the composite of the reduction modulo $\varpi^m$ of the given map $\cV_\lambda \to \cV_{\widetilde{\lambda}}^{\widetilde{K}_{U, S}}$, together with the morphism $(\cV_{\widetilde{\lambda}} / \varpi^m)^{\widetilde{K}_{U, S}} \to R 1_\ast^{\widetilde{K}_{U, S}} \cV_{\widetilde{\lambda}}/ \varpi^m$ whose existence is assured by the universal property of the derived functor.
	
	The morphism $R 1_\ast^{\widetilde{K}_{U, S}} \cV_{\widetilde{\lambda}}/ \varpi^m \to \cV_\lambda/ \varpi^m$ is the composite of the morphism $R 1_\ast^{\widetilde{K}_{U, S}} \cV_{\widetilde{\lambda}} / \varpi^m\to \cV_{\widetilde{\lambda}}/ \varpi^m$ (given by restriction to the trivial subgroup) and the reduction modulo $\varpi^m$ of the $K$-equivariant splitting $\cV_{\widetilde{\lambda}} \to \cV_\lambda$. This completes the proof.
\end{proof}
Here is a variant of Theorem \ref{thm:Hecke_reduction_to_Siegel} where we now take trivial coefficients but consider additional Hecke operators at some ramified places. 
\begin{theorem}\label{thm:Hecke_reduction_to_Siegel_with_R_ramification} 
Let $\widetilde{K} \subset \widetilde{G}(\A_{F^+}^\infty)$ be as at the start of \S \ref{sec:siegel}, and let $\m \subset \T^S(K, 0)$ be a non-Eisenstein maximal ideal, and let $\widetilde{\m} \subset \widetilde{\T}^S$ denote its pullback under the homomorphism $\cS : \widetilde{\T}^S \to \T^S$. Suppose moreover that there is a subset $R \subset S$ satisfying the following conditions:
	\begin{itemize}
	\item Each place $v \in R$ is prime to $p$ and is split over $F^+$.
	\item For each place $v \in R - R^c$ lying over a place $\overline{v}$ of $F^+$, $\widetilde{K}_{\overline{v}} = \widetilde{\q}_v$, where $\widetilde{\q}_v$ contains $\widetilde{\p}_{v, 1}$ and is contained in $\widetilde{\p}_v$. For each place $v \in R \cap R^c$ lying over a place $\overline{v}$ of $F^+$, $\widetilde{K}_{\overline{v}} = \widetilde{I}_{\overline{v}}$, where $\widetilde{I}_{\overline{v}}$ contains $\widetilde{\Iw}_{\overline{v}, 1}$ and is contained in $\widetilde{\Iw}_{\overline{v}}$. 
	\end{itemize}
	Let $T = S - (R^c - R)$. Let $\widetilde{\T}^T_R \subset \cH( \widetilde{G}(\A_{F^+}^\infty), \widetilde{K}) \otimes_\bZ \cO$ denote the (commutative) $\cO$-subalgebra generated by $\widetilde{\T}^S$ and all the elements $t_{v, i}(\sigma)$ ($v \in R, \sigma \in W_{F_v}$) and $e_{v, i}(\sigma)$ ($v \in R^c - R, \sigma \in W_{F_v}$). Let $\T^T_R \subset \cH( \GL_n(\A_{F}^\infty), K) \otimes_\bZ \cO$ denote the (commutative) $\cO$-subalgebra generated by $\T^T$ and all the elements $t_{v, i}(\sigma)$ ($v \in R, \sigma \in W_{F_v}$). Then there is a map $\cS : \widetilde{\T}^T_R \to \T^T_R$, which descends to an $\cO$-algebra homomorphism
	\[ \widetilde{\T}^T_R( R \Gamma(\partial \widetilde{X}_{\widetilde{K}}, \cO)_{\widetilde{\m}}) \to \T^T_R( R \Gamma(X_K, \cO)_\m ). \]
\end{theorem}
\begin{proof}
Let $R_0 = R \cup R^c$, $S_0 = S - R_0$.  The map $\cS$ is the one described in \S \ref{sec:general_definition_of_hecke_operators} (at unramified places) and \S \ref{sec:hecke_algebra_of_a_monoid} (see in particular Lemma \ref{lem:strongly_positive_elements}, which applies at the ramified places we consider here, cf. the discussion at the end of \S \ref{sec:some_useful_hecke_operators}). Once again, by Theorem \ref{thm:reduction_to_Siegel}, it will be enough to us to show that $\cS$ descends to a homomorphism
\[  \widetilde{\T}^T_R( R \Gamma(\widetilde{X}^P_{\widetilde{K}}, \cO)) \to \T^T_R( R \Gamma(X_K, \cO) ). \]
In order to show the existence of this homomorphism we first recall, following the discussion at the end of \S \ref{sec:hecke_algebra_of_a_monoid}, that $\cS$ arises by localisation from the composite of homomorphisms
\[  r_P : \cH(\widetilde{G}^{S} \times \Delta_{\widetilde{G}, R_0}, \widetilde{K}^{S_0}) \to \cH(P^S \times \Delta_{P, R_0}, \widetilde{K}_P^{S_0})\]
and
\[ r_G :  \cH(P^S \times \Delta_{P, R_0}, \widetilde{K}_P^{S_0}) \to \cH(G^S \times \Delta_{G, R_0}, K^{S_0}). \]
Moreover, there are morphisms of complexes
\[ R \Gamma(X^P_{\widetilde{K}_P}, \cO)  \overset{\alpha}{\to} R \Gamma(\widetilde{X}^P_{\widetilde{K}}, \cO) \overset{\beta}{\to}R \Gamma(X^P_{\widetilde{K}_P}, \cO)  \]
and
\[ R \Gamma(X_K, \cO) \overset{\gamma}{\to} R \Gamma(X^P_{\widetilde{K}_P}, \cO) \overset{\delta}{\to}R \Gamma(X_K, \cO) \]
satisfying the following conditions:
\begin{itemize}
\item $\beta$ respects the action of $\cH(\widetilde{G}^{S} \times \Delta_{\widetilde{G}, R_0}, \widetilde{K}^{S_0})$ (when this algebra acts by $r_P$ on the target of $\beta$).
\item $\gamma$ respects the action of $\cH(P^S \times \Delta_{P, R_0}, \widetilde{K}_P^{S_0})$ (when this algebra acts by $r_G$ on the source of $\gamma$).
\item $\beta \alpha$ and $\delta \gamma$ both equal the identity. 
\end{itemize}
Let $\widetilde{\T}^T_{R, +}$ denote the intersection of $\widetilde{\T}^T_R$ with $\cH(\widetilde{G}^{S} \times \Delta_{\widetilde{G}, R_0}, \widetilde{K}^{S_0}) \otimes_\bZ \cO$ (intersection taken inside $\cH(\widetilde{G}^{S_0}, \widetilde{K}^{S_0}) \otimes_\bZ \cO$). Define $\T^T_{R, +}$ similarly. Then $\cS(\widetilde{\T}^T_{R, +}) \subset \T^T_{R, +}$ and the above listed properties immediately imply that  $\cS_+ = \cS|_{\widetilde{\T}^T_{R, +}}$ descends to a morphism as in the top horizontal arrow of the following diagram
\[ \xymatrix{ \widetilde{\T}^T_{R, +}( R \Gamma(\widetilde{X}^P_{\widetilde{K}}, \cO)) \ar[r] \ar[d]^\subset & \T^T_{R, +}( R \Gamma(X_K, \cO) ) \ar[d]^\subset \\
 \widetilde{\T}^T_R( R \Gamma(\widetilde{X}^P_{\widetilde{K}}, \cO)) & \T^T_R( R \Gamma(X_K, \cO) ). } \]
 By construction, we can find an element $z \in \widetilde{\T}^T_{R, +}( R \Gamma(\widetilde{X}^P_{\widetilde{K}}, \cO))$ with the following properties:
 \begin{itemize}
 \item $z$ is a unit in $\widetilde{\T}^T_R( R \Gamma(\widetilde{X}^P_{\widetilde{K}}, \cO))$ and $\widetilde{\T}^T_R( R \Gamma(\widetilde{X}^P_{\widetilde{K}}, \cO))  = \widetilde{\T}^T_{R, +}( R \Gamma(\widetilde{X}^P_{\widetilde{K}}, \cO))[z^{-1}]$.
 \item $\cS_+(z)$ is a unit in $\T^T_R( R \Gamma(X_K, \cO) )$ and  $\T^T_R( R \Gamma(X_K, \cO) ) = \T^T_{R, +}( R \Gamma(X_K, \cO) ) [ \cS_+(z)^{-1} ]$.
 \end{itemize}
 Indeed, we can take $z$ to be a product (over places of $\overline{R}_0$) of strongly positive Hecke operators, as in the statement of Lemma \ref{lem:strongly_positive_elements}. We deduce that in fact $\cS$ descends to a homomorphism 
 \[  \widetilde{\T}^T_R( R \Gamma(\widetilde{X}^P_{\widetilde{K}}, \cO)) \to \T^T_R( R \Gamma(X_K, \cO) ), \]
 as required. 
\end{proof}
\subsubsection{Some results on rational cohomology} 

\begin{theorem}\label{thm:application_of_matsushima} Fix a choice of isomorphism $\iota : \overline{\bQ}_p \to \bC$.
	\begin{enumerate}
		\item Let $\pi$ be a cuspidal, regular algebraic automorphic representation of 
		$\GL_n(\bA_F)$ of weight $\iota\lambda$. Suppose that there 
		exists a good subgroup $K \subset \GL_n(\A_F)$ such that $(\pi^\infty)^K \neq 0$. Then the map $\T^S \to \overline{\Q}_p$ associated to the Hecke eigenvalues of $(\iota^{-1} \pi^\infty)^K$ factors through the quotient $\T^S \to \T^S(K, \lambda)$.
		\item 
		Let $q_0 = [F^+ : \Q] n (n-1)/2$, $l_0 = [F^+ : \Q] n - 1$.  Let $K \subset \GL_n(\A_F)$ be a good subgroup, and let $\ffrm \subset \bT^S(K, \cV_\lambda)$ be a maximal ideal such that $\overline{\rho}_\ffrm$ is absolutely irreducible. Then for each $j \in \bZ$, the group 
		\[ H^j(X_K, \cV_\lambda)_\ffrm[1/p] \]
		is non-zero only if $j \in [q_0, q_0 + l_0]$; moreover
                if one of the groups in this range is non-zero, then
                they all are.

                If $f : \bT^S(K, \cV_\lambda)_\ffrm \to \overline{\bQ}_p$ is a homomorphism, then there exists a cuspidal, regular algebraic automorphic representation $\pi$ of $\GL_n(\bA_F)$ of weight $\iota\lambda$ such that $f$ is associated to the Hecke eigenvalues of $(\iota^{-1} \pi^\infty)^K$. In particular, there is an isomorphism $\overline{r_\iota(\pi)} \cong \overline{\rho}_\ffrm$.
	\end{enumerate}
\end{theorem}
\begin{proof}
	For the first part, it suffices to show that there is a non-zero eigenvector 
	for $\cH(\GL_n(\bA_F^{\infty,S}),K^S)$ in $H^*(X_K, 
	V_{\iota\lambda})$ with the eigenvalue of $T_{v, i}$ equal to 
	its eigenvalue of $T_{v, i}$ on $\pi_v^{K_v}$. 
	
	Likewise, for the second part it suffices to show that the group $H^j(X_K, 
	V_{\iota\lambda})_{\m}:=H^j(X_K, 
	\cV_{\lambda})_{\m}\otimes_{\cO,\iota}\CC$ is non-zero only if $j \in 
	[q_0,q_0+l_0]$, that if 
	one of the groups in this range is non-zero they all are, and that if 
	$f:\cH(\GL_n(\bA_F^{\infty,S}),K^S)\rightarrow \CC$ is a system of Hecke 
	eigenvalues appearing in $H^*(X_K, 
	V_{\iota\lambda})_{\m}$ then there is a cuspidal, regular algebraic 
	automorphic 
	representation of 
	$\GL_n(\bA_F)$ of weight $\iota\lambda$ giving rise to this system of Hecke 
	eigenvalues.
	
	As a consequence of Franke's theorem \cite[Thm.~18]{franke}, as in 
	\cite[\S2.2]{frankeschwermer}, we have a canonical decomposition 
	
	\[H^*(X_K,V_{\iota\lambda}) = \left(\bigoplus_{\{Q\}\in \cC} 
	H^*(\frakm_G,K_\infty; 
	A_{V_{\iota\lambda},\{Q\}}\otimes_{\bC} 
	V_{\iota\lambda})(\chi_\lambda)\right)^{K}.\]
	
	In this formula, $\cC$ is the set of associate classes of parabolic 
	$\bQ$-subgroups of $\Res_{F 
		/ \bQ} \GL_n$. The cohomology on the right hand side is relative Lie algebra 
	cohomology, $\frakm_G$ is the Lie algebra of the real points of the algebraic 
	group given by the kernel of the map $N_{F/\bQ}\circ\det:\Res_{F 
		/ \bQ} \GL_n \rightarrow \GL_1$, and $A_{V_{\iota\lambda},\{Q\}}$ 
	is a certain space of automorphic forms (in particular, it is a 
	$\GL_n(\bA_F^\infty)$-module). Finally, the $(\chi_\lambda)$ denotes a 
	twist 
	of  the $\GL_n(\bA_F^\infty)$-module structure, determined by the central 
	character of $V_{\iota\lambda}$, which appears because the automorphic 
	forms considered in \emph{loc.~cit.}~are by definition invariant under 
	translation by $\bR^{> 0} \subset (\Res_{F 
		/ \bQ} \GL_n)(\bR)$.  We set $E_{\{Q\}} = 
	H^*(\frakm_G,K_\infty; 
	A_{V_{\iota\lambda},\{Q\}}\otimes_{\bC} 
	V_{\iota\lambda})(\chi_\lambda)$. The summand $E_{\{G\}}^{K}$ is the cuspidal 
	cohomology group \[H^*_{cusp}(X_K,V_{\iota\lambda}) = 
	\bigoplus_{\pi}(\pi^\infty)^K\otimes_{\bC}H^*(\frakm_G,K_\infty;\pi_\infty\otimes_{\bC}
	V_{\iota\lambda})\] where the sum is over cuspidal automorphic 
	representations $\pi$ of $\GL_n(\bA_F)$ with central character $\xi$ satisfying $\xi|_{\R_{>0}} = \xi_{\iota\lambda}^{-1}|_{\R_{>0}}$, where $\xi_{\iota\lambda}$ is the central character of $V_{\iota\lambda}$.

	Let $\gM$ be a maximal ideal of $\cH(\GL_n(\bA_F^{\infty, S}),K^S) \otimes_\Z \C$ in the support of 
	$E_{\{Q\}}^K$. Suppose $Q \subset \Res_{F 
		/ \bQ} \GL_n$ is the standard (block upper triangular) parabolic subgroup corresponding to the 
	partition $n = n_1 + \cdots + n_r$. We denote its standard (block diagonal) Levi factor by $L_Q$. In order to simplify notation, we set 
	\[ W = W((\Res_{F / \Q} \GL_n)_\C, (\Res_{F / \Q} T_n)_\C),\]
	 \[ W_Q = W_Q((\Res_{F / \Q} \GL_n)_\C, (\Res_{F / \Q} T_n)_\C), \] 
	 and 
	 \[ W^Q = W^Q((\Res_{F / \Q} \GL_n)_\C, (\Res_{F / \Q} T_n)_\C) \] 
	 for the respective Weyl groups
	 (notation as in \S \ref{sec:misc_notation}). It 
	follows from \cite[Prop.~3.3]{frankeschwermer} (see also the proof of 
	\cite[Thm.~20]{franke}) that 
	$\gM$ 
	corresponds to the system of Hecke eigenvalues for the (unnormalized) parabolic 
	induction $\Ind^{\GL_n(\bA_F^\infty)}_{Q(\bA^\infty)}\sigma^\infty$, 
	where 
	$\sigma = 
	\bigotimes_{i=1}^r\pi_i$ is a cuspidal automorphic representation of 
	$L_Q(\bA_\bQ) = \prod_{i=1}^r \GL_{n_i}(\bA_{F})$ whose infinitesimal character 
	matches that of the dual of the $(L_Q)_{\bC}$-representation with highest 
	weight 
	$w(\iota\lambda+\rho)-\rho$,
	for some $w$ in the set $W^Q$. Here $\rho$ denotes half the 
	sum of 
	the $(\Res_{F / \Q} B_n)_{\bC}$-positive roots, and we note that each 
	$w(\iota\lambda+\rho)-\rho$ is a dominant weight for $(L_Q)_{\bC}$. In 
	particular, the 
	$\pi_i$ are 
	regular algebraic cuspidal automorphic representations of 
	$\GL_{n_i}(\bA_F)$ (whose weight depends on $w$). 
	
	We sketch how this statement can be deduced from the proof of 
	\cite[Prop.~3.3]{frankeschwermer}. The 
	space $A_{V_{\iota\lambda},\{Q\}}$ decomposes, as a 
	$\GL_n(\bA_F^\infty)$-module, 
	into 
	a direct 
	sum $\oplus_\varphi A_{V_{\iota\lambda},\{Q\},\varphi}$. Each space of 
	automorphic forms $A_{V_{\iota\lambda},\{Q\},\varphi}$ is the 
	quotient of a space denoted $W_{Q,\widetilde{\pi}}\otimes 
	S(\check{\mathfrak{a}}_Q^G)$ in \emph{loc.~cit.} It is also observed in the 
	proof of \cite[Prop.~3.3]{frankeschwermer} that this space, as a 
	$\GL_n(\bA_F^\infty)$-module, has a filtration whose quotients are isomorphic 
	as $\GL_n(\bA_F^\infty)$-modules to a normalized parabolic induction 
	$\Ind_{Q(\bA_\bQ)}^{\GL_n(\bA_F^\infty)}(\delta_Q\otimes\pi^\infty)$. Our 
	notation differs from \cite{frankeschwermer}, as we are writing 
	$\Ind$ for unnormalized parabolic induction. Here 
	$\pi$ 
	is a cuspidal automorphic representation of $L_Q(\bA_\bQ)$ whose infinitesimal 
	character corresponds under the normalized Harish-Chandra isomorphism to a 
	weight in the $W$-orbit of the infinitesimal character of 
	$V^\vee_{\iota\lambda}$ (by \cite[1.2 c)]{frankeschwermer}). The normalization 
	is given by the character 
	$\delta_Q$ of $L_Q(\bA_\bQ)$ defined by $\delta_Q(l) = 
	e^{\langle H_Q(l),\rho_Q\rangle}$, where $H_Q$ is the standard height function 
	defined in \cite[p.769]{frankeschwermer} and $\rho_Q$ is half the sum of the 
	roots in the unipotent radical of $Q$.  Although $\pi$ will not always be 
	regular algebraic, the twist $\sigma:= \delta_Q\otimes \pi$ will be. More precisely, we show that the 
	infinitesimal character of $\sigma$ equals that of the 
	dual of the $(L_Q)_{\bC}$-representation with highest weight 
	$\lambda_w := w(\iota\lambda+\rho)-\rho$, for some $w \in W^Q$. Indeed, we have 
	$v \in W^Q$ such that the infinitesimal character $\chi_\sigma = \chi_{\pi} + 
	\rho_Q = 
	v(\iota\lambda^\vee + \rho)+\rho_Q$, where $\iota\lambda^\vee$ is the highest 
	weight of $V_{\iota\lambda}^\vee$ (which has infinitesimal character 
	$\iota\lambda^\vee + \rho$). A short calculation shows that $\chi_\sigma = 
	\lambda_{w_{0,Q}vw_0}^\vee + \rho_{L_Q}$, where $\rho_{L_Q}$ is half the sum of 
	the positive roots for $L_Q$, $w_0$ is the longest element of $W$ and $w_{0,Q}$ 
	is the longest element of $W_Q$. Note that since $W^Q$ is characterized by 
	taking dominant weights for $\GL_n$ to dominant weights for $L_Q$ (equivalently, 
	taking anti-dominant weights to anti-dominant weights),  $w_{0,Q}vw_0$ is an 
	element of $W^Q$, so this gives the desired statement.
	
	Returning to the proof of the theorem, it now follows from 
	Thm.~\ref{thm:char0galrepexists} that there is a Galois 
	representation \[r_\iota(\gM): G_F \rightarrow  \GL_n(\Qpbar)\] such 
	that, for all but finitely many $v\notin S$, the characteristic polynomial of 
	$r_\iota(\gM)(\Frob_v)$ equals $P_v(X) \text{ mod }\gM$. Indeed, 
	(cf.~the proof of \cite[Thm.~4.2]{new-tho}) 
	we have
	\[r_\iota(\gM) = \bigoplus_{i=1}^r 
	r_\iota(\pi_i)\otimes \epsilon^{-(n_{i+1}+\cdots+n_r)}.\] 
	
	We can now deduce that if $Q$ is a proper parabolic, then $(E_{\{Q\}}^K)_\m = 
	E_{\{Q\}}^K \otimes_{\bT^S(K, \cV_\lambda)}\bT^S(K, \cV_\lambda)_\m$ vanishes. 
	Suppose $\gM$ is a maximal ideal of $\cH(\GL_n(\bA_F^{\infty, S}),K^S) \otimes_\Z \C$ in the 
	support of $(E_{\{Q\}}^K)_\m$. On the one hand, the representation 
	$r_\iota(\gM)$ is reducible in this case, but we have an 
	isomorphism $\overline{r_\iota(\pi)}\cong \rhobar_\m$. This contradicts the 
	assumption that $\rhobar_\m$ is absolutely irreducible, so we deduce that 
	$(E_{\{Q\}}^K)_\m = 0$.
	
	Finally, we show both parts of the theorem. It suffices to show that if $\pi$ 
	is 
	a cuspidal automorphic representation of $\GL_n(\bA_F)$ with central character matching $\xi_{\iota\lambda}^{-1}$ on $\R_{>0}$, then 
	\begin{enumerate}
		\item $H^*(\m_G,K_\infty;\pi_\infty\otimes V_{\iota\lambda})$ is 
		zero unless $\pi$ is regular algebraic of weight $\iota\lambda$.
		\item If $\pi$ 
		is regular algebraic of weight $\iota\lambda$ then 
		$H^j(\m_G,K_\infty;\pi_\infty\otimes V_{\iota\lambda})$ vanishes for 
		$j\notin 
		[q_0,q_0+l_0]$ and is non-zero for $j \in [q_0,q_0+l_0]$.\end{enumerate} The 
	first claim follows from \cite[Ch.~II, Prop.~3.1]{MR1721403}. The second 
	claim follows from \cite[Lem.~3.14]{MR1044819} and the K\"{u}nneth formula for relative Lie algebra cohomology (in the notation of \emph{loc.~cit.}, our Lie algebra $\m_G$ is a direct sum of $\widetilde{\gog}_v$ for each infinite place $v$ of $F$ and an abelian Lie algebra of dimension $[F^+:\Q]-1$; the range of non-zero cohomological degrees is $n-1$ for $(\widetilde{\gog}_v,K_v)$-cohomology, so we get range $(n-1)[F^+:\Q] + [F^+:\Q]-1 = l_0$ in total).
\end{proof}

\begin{thm}\label{thm:automorphic reps contributing to char 0} 
	Let $\rho \in X^\ast(\Res_{F^+ / \Q} T_n)$ denote half the sum of the positive roots of $\Res_{F^+ / \Q} \widetilde{G}$. Fix an isomorphism $\iota : \overline{\Q}_p \to \C$. Let $\widetilde{\lambda} \in (\Z^{2n}_+)^{\Hom(F^+, E))}$ be a highest weight with the following 
	property: for any $w\in W^P((\Res_{F^+ / \Q} \widetilde{G})_\C, (\Res_{F^+ / \Q} T)_\C)$,	there are no (characteristic $0$) cuspidal automorphic representations for 
	$G$ of 
	weight $\iota \lambda_w$, where $\lambda_w = w(\widetilde{\lambda} + \rho)-\rho$.
	
	Let $\widetilde{\m}\subset 
	\widetilde{\mathbb{T}}^S$ be a maximal ideal which is in the support of 
	$H^*(\widetilde{X}_{\widetilde{K}},\cV_{\widetilde{\lambda}})$ with the property that 
	$\bar{\rho}_{\widetilde{\m}}$ is a direct sum of $n$-dimensional absolutely 
	irreducible representations of $G_F$. Let $d = \frac{1}{2} \dim_\R 
	X^{\widetilde{G}} =  n^2 [ F^+ : \Q]$. 
	
	Then $H^d(\widetilde{X}_{\widetilde{K}},\cV_{\widetilde{\lambda}})_{\widetilde{\m}}[1/p]$ is a semisimple $\widetilde{\T}^S[1/p]$-module, and for every homomorphism $f : \widetilde{\T}^S(H^d(\widetilde{X}_{\widetilde{K}},\cV_{\widetilde{\lambda}})_{\widetilde{\m}}) \to \overline{\Q}_p$, there exists a cuspidal, regular algebraic automorphic representation $\widetilde{\pi}$ of $\widetilde{G}(\A_{F^+})$ of weight $\iota \widetilde{\lambda}$ such that $f$ is associated to the Hecke eigenvalues of $(\iota^{-1} \widetilde{\pi}^\infty)^{\widetilde{K}}$. 
\end{thm}

\begin{proof} 
	 The proof uses similar ingredients to the proof of 
	Theorem~\ref{thm:application_of_matsushima} above. We must understand the systems of $\widetilde{\mathbb{T}}^S$-eigenvalues 
	occurring in 
	\[
	H^d(\widetilde{X}_{\widetilde{K}}, V_{\iota 
		\widetilde{\lambda}})_{\widetilde{\m}}: = 
	H^d\left(\widetilde{X}_{\widetilde{K}}, 
	\cV_{\widetilde{\lambda}}\right)_{\widetilde{\m}} [1/p]
	\otimes_{E,\iota}\C
	\]
	As a consequence of~\cite[Thm.~18]{franke}, as 
	in~\cite[\S2.2]{frankeschwermer}, applied to the group 
	$\mathrm{Res}_{F^+/\Q}\widetilde{G}$, we have a canonical decomposition
	\[
	H^d(\widetilde{X}_{\widetilde{K}}, V_{\iota 
		\widetilde{\lambda}}) = \left(\bigoplus_{\{\widetilde{P}\}\in 
		\cC}H^d\left(\widetilde{\mathfrak{g}}, 
	\widetilde{K}_{\infty};A_{V_{\iota 
			\widetilde{\lambda}},\{\widetilde{P}\} }\otimes_{\C}V_{\iota 
		\widetilde{\lambda}}\right)\right)^{\widetilde{K}}.
	\]
	\noindent Here, $\cC$ is the set of associate classes of parabolic 
	$\Q$-subgroups of $\mathrm{Res}_{F^+/\Q}\widetilde{G}$. The cohomology on 
	the right hand side is relative Lie algebra cohomology, 
	 $\mathfrak{\widetilde{g}}$ is the Lie algebra of 
	$(\mathrm{Res}_{F^+/\Q}\widetilde{G})(\R)$, and $A_{V_{\iota 
			\widetilde{\lambda}},\{\widetilde{P}\}}$ is a certain space of automorphic 
	forms for $\mathrm{Res}_{F^+/\Q}\widetilde{G}$. (We note that in this case 
	there is no additional character twist of the 
	$\widetilde{G}(\A_{F}^\infty)$-module structure, because the maximal split 
	torus in the center of $\mathrm{Res}_{F^+/\Q}\widetilde{G}$ is trivial; equivalently, the $\m_{\widetilde{G}}$ of \cite{frankeschwermer} is equal to $\mathfrak{\widetilde{g}}$ because $\mathrm{Res}_{F^+/\Q}\widetilde{G}$ has no rational characters.) 
	Set $E_{\{\widetilde{P}\}}:=H^d\left(\mathfrak{m}_{\widetilde{G}}, 
	\widetilde{K}_{\infty};A_{V_{\iota \widetilde{\lambda}}, 
		\{\widetilde{P}\}}\otimes_{\C}V_{\iota 
		\widetilde{\lambda}}\right)$. The summand 
	$E^{\widetilde{K}}_{\{\widetilde{G}\}}$ is the cuspidal cohomology group 
	\[
	H^d_{cusp} \left(\widetilde{X}_{\widetilde{K}}, 
	V_{\iota\widetilde{\lambda}}\right)= 
	\bigoplus_{\widetilde{\pi}}(\widetilde{\pi}^\infty)^{\widetilde{K}}\otimes_{\C}H^d\left(\widetilde{\mathfrak{g}},
	\widetilde{K}_{\infty}; 
	\widetilde{\pi}_{\infty}\otimes_{\C}V_{\iota\widetilde{\lambda}}\right),
	\]   
	where the sum runs over cuspidal automorphic representations of 
	$\widetilde{G}(\A_{F^+})$. We see that the theorem will be proved if we can establish the following two claims:
	\begin{enumerate}
		\item If $\widetilde{P}$ is a proper standard parabolic subgroup of 
		$\mathrm{Res}_{F^+/\Q}\widetilde{G}$ different from the Siegel 
		parabolic, then 
		\[ \left(\iota^{-1} E^{\widetilde{K}}_{\{\widetilde{P}\}}\right)_{\widetilde{\m}} := \iota^{-1} E^{\widetilde{K}}_{\{\widetilde{P}\}} \otimes_{\widetilde{\T}^S(\widetilde{K}, \widetilde{\lambda})} \widetilde{\T}^S(\widetilde{K}, \widetilde{\lambda})_{\widetilde{\m}} =0. \]
		\item If $\widetilde{P} = P$ is the Siegel parabolic subgroup of 
		$\mathrm{Res}_{F^+/\Q}\widetilde{G}$, then we also have 
		$\left(\iota^{-1} E^{\widetilde{K}}_{\{\widetilde{P}\}}\right)_{\widetilde{\m}}=0$.
	\end{enumerate} 
	The same argument as in the proof of Theorem~\ref{thm:application_of_matsushima} shows that if $\widetilde{\mathfrak{M}}$ is a maximal ideal of $\widetilde{\T}^S[1/p]$ which occurs in the support of $\iota^{-1} E^{\widetilde{K}}_{\{\widetilde{P}\}}$, then $\widetilde{\mathfrak{M}}$ corresponds to the system of Hecke eigenvalues appearing in 
	\[ \left(\Ind_{\widetilde{P}(\A_{F^+}^\infty)}^{\widetilde{G}(\A_{F^+}^\infty)} \iota^{-1} \sigma^\infty\right)^{\widetilde{K}}, \]
	where $\sigma$ is a cuspidal automorphic representation of $L_{\widetilde{P}}(\A_{F^+})$ whose infinitesimal character equals the dual of the infinitesimal character of the irreducible algebraic representation of $L_{\widetilde{P}}$ of highest weight $w(\iota \widetilde{\lambda} + \rho) - \rho$, for some $w \in W^{\widetilde{P}}((\Res_{F^+ / \Q} \widetilde{G})_\C, (\Res_{F^+ / \Q} T)_\C)$. The second claim now follows immediately from our hypothesis that there are no such automorphic representations in the case $\widetilde{P} = P$. 
		
	As in the proof of Theorem \ref{thm:reduction_to_Siegel}, we note that the Levi subgroup $L_{\widetilde{P}}$ is isomorphic to a product $\Res_{F / F^+} \GL_{n_1} \times \dots \times \Res_{F / F^+} \GL_{n_r} \times U(n-s, n-s)$, for some decomposition $2n = n_1 + \dots + n_r + 2(n-s)$. We can now establish the first claim: using the existence of Galois representations attached to regular algebraic cuspidal automorphic representations of $\GL_m$ and $U(m, m)$ for $m \leq n$ (i.e.\ using Theorem \ref{thm:char0galrepexists} and Theorem \ref{thm:base_change_and_existence_of_Galois_for_tilde_G}), we see that there exists a Galois representation $r(\widetilde{\mathfrak{M}}) : G_F \to \GL_{2n}(\overline{\Q}_p)$ such that, for all but finitely many places $v$ of $F$, $r(\widetilde{\mathfrak{M}})$ is unramified at $v$ and $r(\widetilde{\mathfrak{M}})$ has characteristic polynomial equal to $\widetilde{P}_v(X) \text{ mod }\widetilde{\mathfrak{M}}$. Moreover, this representation has at least 3 Jordan--H\"older factors as soon as $(r, s) \not\in \{ (1, n), (0, 0) \}$ (by an argument identical to the one appearing at the end of the proof of Theorem \ref{thm:reduction_to_Siegel}). Since we are assuming that $\overline{\rho}_{\widetilde{\m}}$ has 2 irreducible constituents, each of dimension $n$, this would lead to a contradiction, showing that we must in fact have $\left(\iota^{-1} E^{\widetilde{K}}_{\{\widetilde{P}\}}\right)_{\widetilde{\m}}=0$. This completes the proof.
\end{proof}

\section{Local-global compatibility, \texorpdfstring{$l\ne p$}{l ne p}}
\label{section:lnep}

\subsection{Statements}\label{sec:lneqp_statements}

Let $F$ be a CM field containing an imaginary quadratic field, and fix an integer $n \geq 1$. Let $p$ be a prime, and let $E$ be a finite extension of $\Q_p$ inside $\overline{\Q}_p$ large enough to contain the images of all embeddings of $F$ in $\overline{\Q}_p$. We assume that each $p$-adic place $\overline{v}$ of $F^+$ splits in $F$.

Let $K \subset \GL_n(\A_F^\infty)$ be a good subgroup, and let $\lambda \in (\Z_+^n)^{\Hom(F, E)}$. Let $S$ be a finite set of finite places of $F$, containing the $p$-adic places, and satisfying the following conditions:
	\begin{itemize}
	\item $S = S^c$. 
		\item Let $v$ be a finite place of $F$ not contained in $S$, and let $l$ be its residue characteristic. Then either $S$ contains no $l$-adic places of $F$ and $l$ is unramified in $F$, or there exists an imaginary quadratic field $F_0 \subset F$ in which $l$ splits.
	\end{itemize}
 We recall (Theorem \ref{thm:existence_of_Hecke_repn_for_GL_n})   that under these hypotheses, that if $\ffrm \subset \T^S(K, \lambda)$ is a non-Eisenstein maximal ideal, then there is a continuous homomorphism
\[ \rho_\ffrm : G_{F, S} \to \GL_n(\T^S(K, \lambda)_\ffrm / I) \]
characterized, up to conjugation, by the characteristic polynomials of Frobenius elements at places $v \not\in S$; here $I$ is a nilpotent ideal whose exponent depends only on $n$ and $[F : \Q]$. Our goal in this chapter is to describe the restriction of $\rho_\ffrm$ to decomposition groups at some prime-to-$p$ places where ramification is allowed. 

To this end, we suppose given as well a subset $R \subset S$ satisfying the following conditions:
\begin{itemize}
	\item Each place $v \in R$ is prime to $p$.
	\item For each place $v \in R$, there exists an imaginary quadratic field $F_0 \subset F$ in which the residue characteristic of $v$ splits. In particular, $v$ is split over $F^+$.
	\item For each place $v \in R$, $K_v$  contains $\Iw_{v, 1}$ and is contained in $\Iw_{v}$. For each place $v \in R^c - R$, $K_v = \GL_n(\cO_{F_v})$. (Note that $R^c \subset S$ since $S$ is assumed stable under complex  conjugation.)
\end{itemize}
Let $T = S - (R^c - R)$. We define $\T^T_R \subset \cH( \GL_n(\A_{F}^\infty), K) \otimes_\bZ \cO$ to be the (commutative) $\cO$-subalgebra generated by $\T^T$ and all the elements $t_{v, i}(\sigma)$ ($v \in R, \sigma \in W_{F_v}$), as  in the statement of Theorem \ref{thm:Hecke_reduction_to_Siegel_with_R_ramification}.  We define 
\[ \T^T_R(K, \lambda) \subset \End_{\mathbf{D}(\cO)}(R \Gamma(X_K, \cV_\lambda)) \]
to be the image of $\T_R^T$. Thus there are inclusions 
\[ \T^S(K, \lambda) \subset \T^T(K, \lambda) \subset  \T^T_R(K, \lambda). \]
\begin{theorem}\label{thm:lgc_at_l_neq_p}
	Let notation and assumptions be as above. Then we can find an integer $N \geq 1$ (depending only on $n$ and $[F : \Q]$), an ideal $I_R \subset \T^T_R(K, \lambda)_{\ffrm}$ satisfying $I_R^N = 0$, and a continuous homomorphism
		\[ \rho_{\ffrm, R} : G_{F, T} \to \GL_n(\T^T_R(K, \lambda)_{\ffrm} / I_R)  \]
		satisfying the following conditions:
		\begin{enumerate}
			\item For each place $v \not\in T$ of $F$, the characteristic polynomial of $\rho_{\ffrm, R}(\Frob_v)$ is equal to the image of $P_v(X)$ in $(\T^T_R(K, \lambda)_{\ffrm}/I_R)[X]$.
			\item For each place $v \in R$, and for each element $\sigma \in W_{F_v}$, the characteristic polynomial of $\rho_{\ffrm, R}(\sigma)$ is equal to the image of $P_{v, \sigma}(X)$ in  $(\T^T_R(K, \lambda)_{\ffrm}/I_R)[X]$.
	\end{enumerate}
\end{theorem}
In the statement of this theorem, $\T^T_R(K, \lambda)_{\ffrm}$ is the
localization of $\T^T_R(K, \lambda)$ as a $\T^S(K, \lambda)$-algebra; it is an $\cO$-subalgebra of $\End_{\mathbf{D}(\cO)}(R \Gamma(X_K, \cV_\lambda)_\ffrm)$ which contains $\T^S(K, \lambda)_\ffrm$. Instead of proving this theorem directly, we will in fact prove the following statement:
\begin{prop}\label{prop:lgc_at_l_neq_p_det_version}
	Let notation and assumptions be as above. Then there exists an integer $N \geq 1$ (depending only on $n$ and $[F : \Q]$), an ideal $I_R \subset \T^T_R(K, \lambda)_{\ffrm}$ satisfying $I_R^N = 0$, and a $\T^T_R(K, \lambda)_{\ffrm} / I_R$-valued determinant $D_{\ffrm, R}$ on $G_{F, T}$ of dimension $n$ satisfying the following conditions:
	\begin{enumerate}
		\item For each place $v \not\in T$, the characteristic polynomial of $\Frob_v$ in $D_{\ffrm, R}$ is equal to the image of $P_v(X)$ in $(\T^T_R(K, \lambda)_{\ffrm}/I_R)[X]$.
		\item For each place $v \in R$, and for each element $\sigma \in W_{F_v}$, the characteristic polynomial of $\sigma$ is equal to the image of $P_{v, \sigma}(X)$ in  $(\T^T_R(K, \lambda)_{\ffrm}/I_R)[X]$.
	\end{enumerate}
\end{prop}
Proposition \ref{prop:lgc_at_l_neq_p_det_version} implies Theorem \ref{thm:lgc_at_l_neq_p} by \cite[Theorem 2.22]{chenevier_det}. The remainder of \S \ref{section:lnep} is devoted to the proof of Proposition \ref{prop:lgc_at_l_neq_p_det_version}. Although Proposition \ref{prop:lgc_at_l_neq_p_det_version} is an assertion about determinants, not true representations, we will still use the assumption that $\m$ is non-Eisenstein in the proof, in particular as it simplifies our analysis of the boundary cohomology (using the results proved in \S \ref{sec:siegel}).

\subsection{The proof of Proposition \ref{prop:lgc_at_l_neq_p_det_version}}

Let $R_1 = R \cap R^c$. Let $\overline{R}$ (resp. $\overline{R}_1$) denote the set of places of $F^+$ lying below a place of $R$ (resp. $R_1$). We begin with a preliminary reduction.
\begin{lemma}\label{lem:lgc_l_neq_p_reduction}
Fix for each $v \in R$ a choice of Frobenius lift $\phi_v \in W_{F_v}$. In order to prove Proposition \ref{prop:lgc_at_l_neq_p_det_version}, it is enough to prove it under the following additional assumptions:
\begin{enumerate}
\item $K_v = \Iw_{v, 1}$ for each place $v \in R$. There exists an odd prime $q$, prime to $R$ and $p$, such that $K_q = \ker(\GL_n(\cO_{F, q}) \to \GL_n(\cO_F / (q)))$. 
\item For each place $v \in R$, the characteristic polynomials of $\overline{\rho}_\m(\phi_v)$ and $(\overline{\rho}_\m^{c, \vee} \otimes \epsilon^{1-2n})(\phi_v)$ are coprime.
\item There exists a character $\psi : G_{F} \to \cO^\times$ of finite prime-to-$p$ order, unramified above $R \cup R^c \cup S_p$, such that the composite $\psi \circ \Art_{F} \circ \det : K \to \cO^\times$ is trivial and for each $v \in R$, the characteristic polynomials of $\overline{\psi}(\Frob_v) \overline{\rho}_\m(\phi_v)$ and $\overline{\psi}(\Frob_{v^c})^{-1} (\overline{\rho}_\m \oplus \overline{\rho}_\m^{c, \vee} \otimes \epsilon^{1-2n})(\phi_v)$ are coprime.
\item $\lambda = 0$.
\item There exists a good subgroup $\widetilde{K} \subset \widetilde{G}(\A_{F^+}^\infty)$ satisfying the following conditions:
\begin{enumerate}
\item $\widetilde{K}$ is decomposed with respect to $P$, and $K = \widetilde{K} \cap G(\A_{F^+}^\infty)$.
\item $\widetilde{K}_q = \ker( \widetilde{G}(\cO_{F^+, q}) \to \widetilde{G}(\cO_{F^+} / (q)))$.
\item If $\overline{v}$ is a finite place of $F^+$ which is prime to $S$, then $\widetilde{K}_{\overline{v}} = \widetilde{G}(\cO_{F^+, v})$. If $\overline{v} \in \overline{R}_1$, then $\widetilde{K}_v = \widetilde{\Iw}_{\overline{v}}(1, 1)$. If $\overline{v} \in \overline{R} - \overline{R}_1$ and $v$ is the unique place of $R$ lying above $\overline{v}$, then $\widetilde{K}_{\overline{v}} = \widetilde{\p}_{v, 1}$.
\end{enumerate}
\end{enumerate}
\end{lemma}
\begin{proof}
We first show that if Proposition \ref{prop:lgc_at_l_neq_p_det_version} holds under assumption (1), then it holds without this assumption. Let assumptions be as in Proposition \ref{prop:lgc_at_l_neq_p_det_version}, and let $q_1, q_2 \neq p$ be distinct odd primes not dividing any element of $S$. Let $K_i \subset K$ be the normal subgroup with $K_{i, v} = \Iw_v(1, 1)$ if $v \in R$, $K_{i, q_i} = \ker(\GL_n(\cO_{F, q_i}) \to \GL_n(\cO_F / (q_i)))$, and $K_i^{R, q_i} = K$. Let $S_i$ (resp. $T_i$) denote the union of $S$ (resp. $T$) with the set of $q_i$-adic places of $F$. Let $\ffrm_i \subset \T^{T_i}$ denote the pullback of $\ffrm$ under the inclusion $\T^{T_i} \to \T^T$. For each $i = 1, 2$ there is a diagram of $\T^{T_i}_R$-algebras
\[ \T^{T_i}_R(K_i, \lambda)_{\ffrm_i} \leftarrow \T^{T_i}_R(K / K_i, \lambda)_{\ffrm_i} \twoheadrightarrow \T^{T_i}_R(K, \lambda)_{\ffrm_i} \to \T_R^T(K, \lambda)_{\ffrm}. \]
The left-hand arrow has nilpotent kernel of exponent $d$ depending only on $n$ and $[F : \bQ]$, by Lemma \ref{lem_cohomology_nilpotent_annihilator}. By hypothesis, there exists an integer $N \geq 1$, depending only on $n$ and $[F : \bQ]$, ideals $J_i \subset \T^{T_i}_R(K_i, \lambda)_{\ffrm_i}$ satisfying $J_i^N = 0$, and $n$-dimensional group determinants $D_i$ of $G_{F, T_i}$ with coefficients in $\T^{T_i}_R(K_i, \lambda)_{\ffrm_i} / J_i$ satisfying conditions (1) and (2) of Proposition \ref{prop:lgc_at_l_neq_p_det_version}. 

		Let $I_i$ denote the image in $\T^{T_i}_R(K, \lambda)_{\ffrm_i}$ of the pre-image of $J_i$ in $\T^{T_i}_R(K / K_i, \lambda)_{\ffrm_i}$, and let $I \subset \T^T_R(K, \lambda)_\ffrm$ denote the ideal generated by the images of $I_1$ and $I_2$. Then $I^{2 N d} = 0$. Let $D_\ffrm$ denote the pushforward of the determinant $D_1$ to $\T^T_R(K, \lambda)_\ffrm / I$. Then by construction, $D_\ffrm$ is an $n$-dimensional determinant of $G_{F, T_1}$ satisfying condition (1)  of Proposition \ref{prop:lgc_at_l_neq_p_det_version} at prime-to-$T_1$ places and condition (2) at each place of $R$. However, the Chebotarev density theorem and \cite[Lemma 1.12]{chenevier_det} imply that $D_\ffrm$ is also equal to the pushforward of $D_2$ to $\T^T_R(K, \lambda)_\ffrm / I$. We therefore obtain the required local-global compatibility also at the $q_1$-adic places of $F$. The proof of this step is complete on noting that the exponent $2 N d$ of $I$ indeed still depends only on $n$ and the degree $[F : \bQ]$.
		
			We next show that if Proposition \ref{prop:lgc_at_l_neq_p_det_version} holds under assumptions (1) and (2) in the statement of the lemma, then it holds under assumption (1). After possibly enlarging $\cO$, we can find characters $\psi_1, \psi_2 : G_{F} \to \cO^\times$ of finite, prime-to-$p$ order satisfying the following conditions:
			\begin{itemize}
			\item Both $\psi_1, \psi_2$ are unramified at each place of $S$. 
			\item There is no rational prime $r$ such that $\psi_1, \psi_2$ are both ramified at $r$.
			\item For each $i = 1, 2$ and for each place $v \in R$, the characteristic polynomials of $(\overline{\rho}_\m \otimes \overline{\psi_i})(\phi_v)$ and $( (\overline{\rho}_\m \otimes \overline{\psi}_i) )^{c, \vee} \otimes  \epsilon^{1-2n})(\phi_v)$ are coprime.
			\end{itemize}
			Let $K_i = \prod_v \ker( \psi_i \circ \Art_{F_v} \circ \det : K_v \to \cO^\times)$ and let $T_i$ denote the union of $T$ with the set of places dividing a rational prime above which $\psi_i$ is ramified. Let $\m_i$ denote the pullback of $\m$ to $\T^{T_i}$. Proposition \ref{prop:twisting_by_character} shows that the truth of Proposition \ref{prop:lgc_at_l_neq_p_det_version} for $\T^{T_i}_R(K_i, \lambda)_{\ffrm_i}$ is equivalent to the truth of Proposition \ref{prop:lgc_at_l_neq_p_det_version} for $\T^{T_i}_R(K_i, \lambda)_{\ffrm_i(\psi_i)}$, which we are assuming. On the other hand, an argument very similar to the one given in the first part of the proof shows that the truth of Proposition \ref{prop:lgc_at_l_neq_p_det_version} for $\T^{T_i}_R(K_i, \lambda)_{\ffrm}$ ($i =1, 2$) implies the truth of this proposition for $\T^{T_i}_R(K, \lambda)_{\ffrm_i}$ ($i = 1, 2$) and then for $\T^T_R(K, \lambda)_\ffrm$. A very similar argument shows that if the Proposition holds under assumptions (1) -- (3) in the statement of the lemma, then it holds under (1) and (2).
		
		We next show that if Proposition \ref{prop:lgc_at_l_neq_p_det_version} holds under assumptions (1) -- (4) in the statement of the lemma, then it holds under assumptions (1) -- (3). Let $K$ be a good subgroup satisfying assumptions (1) -- (3). The natural map
		\[ \T^T_R(K, \lambda) \to \varprojlim_{m \geq 1} \T^T_R(R \Gamma(X_K, \cV_\lambda / (\varpi^m))) \]
		is an isomorphism. For each $m \geq 1$, let $K(p^m) = \ker(K \to \GL_n(\cO_{F, p} / (p^m)))$. Then $K(p^m)$ also satisfies assumption (1). The local system $\cV_\lambda / (\varpi^m)$ on $X_{K(p^m)}$ is constant, so there is a canonical isomorphism of Hecke algebras
		\[ \T^T_R(  K(p^m), \cV_\lambda / (\varpi^m) )\cong \T^T_R( K(p^m), \cO / (\varpi^m)). \]
		There is also a canonical surjection
		\[ \T^T_R(K(p^m), \cO) \to \T^T_R( K(p^m), \cO / (\varpi^m)). \]
We consider the diagram of Hecke algebras
\[ \T^T_R( K(p^m), \cO / (\varpi^m)) \leftarrow \T^T_R( K / K(p^m), \cV_\lambda / (\varpi^m) ) \rightarrow \T^T_R( K, \cV_\lambda / (\varpi^m)), \]
where by Lemma \ref{lem_cohomology_nilpotent_annihilator}, there is an integer $d \geq 1$ depending only on $n$ and $[F : \bQ]$ such that the kernel of the left-hand arrow is nilpotent of exponent $d$.
		By assumption, therefore, we can find ideals $I_m \subset  \T^T_R(K, \cV_\lambda / (\varpi^m))_\ffrm$ satisfying $I_m^{N d} = 0$ and $n$-dimensional group determinants $D_m$ of $G_{F, T}$ valued in  $\T^T_R(K, \cV_\lambda / (\varpi^m))_\ffrm / I_m$ 
		 and satisfying the conditions (1) and (2) in the statement of Proposition \ref{prop:lgc_at_l_neq_p_det_version}. Let
		\[ I = \ker\left(   \T^T_R(K, \lambda)_\ffrm \to \prod_{m \geq 1} \T^T_R(K, \cV_\lambda / (\varpi^m))_\ffrm / I_m \right). \]
		Then $I^{Nd} = 0$ and, by \cite[Example 2.32]{chenevier_det}, there is a unique $n$-dimensional group determinant $D_\ffrm$ valued in $\T^T_R(K, \lambda)_\ffrm  / I$ whose pushforward to each ring $\T^T_R(K, \cV_\lambda / (\varpi^m))_\ffrm / I_m$ equals $D_m$. This determinant $D_\ffrm$ necessarily has the required properties.
		
		We finally show that if Proposition \ref{prop:lgc_at_l_neq_p_det_version} holds under assumptions (1) -- (5) in the statement of the lemma, then it holds under assumptions (1) -- (4). Assume (1) -- (4). We define $\widetilde{K} = \prod_{\overline{v}} \widetilde{K}_{\overline{v}}$ as follows:
		\begin{itemize}
		\item If $\overline{v} \not\in \overline{S}$, then $\widetilde{K}_{\overline{v}} = \widetilde{G}(\cO_{F^+_{\overline{v}}})$.
		\item If $\overline{v} \in \overline{R}_1$, then $\widetilde{K}_{\overline{v}} = \widetilde{\Iw}_{\overline{v}}(1, 1)$. If $\overline{v} \in \overline{R} - \overline{R}_1$ and $v$ is the unique place of $R$ lying above $\overline{v}$, then $\widetilde{K}_{\overline{v}} = \widetilde{\p}_{v, 1}$.
		\item $\widetilde{K}_q = \ker(\widetilde{G}(\cO_{F^+, q}) \to \widetilde{G}(\cO_{F^+} / (q)))$.
		\item If $\overline{v}$ is any other finite place of $F^+$, then fix $m \geq 1$ such that $\ker( \widetilde{G}(\cO_{F^+_{\overline{v}}}) \to \widetilde{G}(\cO_{F^+} / (\varpi_{\overline{v}}^m))) \cap G(\cO_{F^+_{\overline{v}}}) \subset K \cap G(\cO_{F^+_{\overline{v}}})$, and set $\widetilde{K}_{\overline{v}} = \ker( \widetilde{G}(\cO_{F^+_{\overline{v}}}) \to \widetilde{G}(\cO_{F^+} / (\varpi_{\overline{v}}^m)))  \cdot (K \cap G(\cO_{F^+_{\overline{v}}}))$.
		\end{itemize}
		It is easy to check that $\widetilde{K}$ is a good open subgroup of $\widetilde{G}(\A_{F^+}^\infty)$ which is decomposed with respect to $P$ and which satisfies $\widetilde{K} \cap G(\A_{F^+}^\infty) = K$. The group $K$ therefore satisfies condition (5) of the lemma, and the proof of the lemma is complete.
\end{proof}
We henceforth fix a choice of Frobenius lift $\phi_v \in W_{F_v}$ for each place $v \in R$ and assume that $K$ and $\m$ satisfy assumptions (1) -- (5) of Lemma \ref{lem:lgc_l_neq_p_reduction}, and prove Proposition \ref{prop:lgc_at_l_neq_p_det_version} for the Hecke algebra $\T^T_R(K, 0)$. Let $\widetilde{K}$ be the good subgroup of $\widetilde{G}(\A_{F^+}^\infty)$ as in the statement of the lemma and let $\widetilde{\T}^T_R \subset \cH( \widetilde{G}(\A_{F^+}^\infty), \widetilde{K}) \otimes_\bZ \cO$ denote the (commutative) $\cO$-subalgebra generated by $\widetilde{\T}^S$ and all the elements $t_{v, i}(\sigma)$ ($v \in R, \sigma \in W_{F_v}$) and $e_{v, i}(\sigma)$ ($v \in R^c - R, \sigma \in W_{F_v}$), as  in the statement of Theorem \ref{thm:Hecke_reduction_to_Siegel_with_R_ramification}. Thus we have constructed an extension of the homomorphism $\cS : \widetilde{\T}^S \to \T^S$ to a homomorphism $\cS : \widetilde{\T}^T_R \to \T^T_R$. These homomorphisms, together with the analogue of Proposition \ref{prop:lgc_at_l_neq_p_det_version} for the group $\widetilde{G}$, will be the key to the proof. This analogue is as follows; it makes use of the resultant $\Res_v\in \cH(\widetilde{G}(F^+_{\overline{v}}), \widetilde{K}_{\overline{v}}) \otimes_\bZ \cO$ of the polynomials $P_{v^c, \phi_v^{-c}}(X)$ and $P_{v, \phi_{v}}(X)$ for a place $v \in R - R^c$, which was introduced before Proposition \ref{prop:unramified_subrep_for_invertible_resultant}.
\begin{prop}\label{prop:lgc_at_l_neq_p_det_version_for_tilde_G}
There exists an integer $N \geq 1$, depending only on $[F : \Q]$ and $n$, 
	an ideal $\widetilde{I}_{c, R} \subset \widetilde{\T}^T_R(R \Gamma_c(\widetilde{X}_{\widetilde{K}}, 
	\cO))$ satisfying $\widetilde{I}_{c, R}^N = 0$, and a 
	$\widetilde{\T}^T_R(R \Gamma_c(\widetilde{X}_{\widetilde{K}}, 
	\cO)) / \widetilde{I}_{c, R}$-valued determinant 
	$\widetilde{D}_{c, R}$ on $G_{F, S}$ of dimension $2n$ satisfying the 
	following conditions:
	\begin{enumerate}
		\item For each place $v \not\in S$ of $F$, the characteristic polynomial of $\Frob_v$ is equal to the image of $\widetilde{P}_{v}(X)$ 		in $(\widetilde{\T}^T_R(R \Gamma_c(\widetilde{X}_{\widetilde{K}}, 
	\cO))/\widetilde{I}_R)[X]$.
		\item For each place $v \in R$, and for each element $\sigma \in W_{F_v}$, the characteristic polynomial of $\sigma$ is equal to the image of $\widetilde{P}_{v, \sigma}(X)$ in  $(\widetilde{\T}^T_R(R \Gamma_c(\widetilde{X}_{\widetilde{K}}, 
	\cO))/\widetilde{I}_R)[X]$.
	\item Let $\widetilde{\tr}_{c, R} : \widetilde{\T}^T_R(R \Gamma_c(\widetilde{X}_{\widetilde{K}}, 
	\cO))[G_{F, S}] \to \widetilde{\T}^T_R(R \Gamma_c(\widetilde{X}_{\widetilde{K}}, 
	\cO)) / \widetilde{I}_{c, R}$ be the trace associated to $\widetilde{D}_{c, R}$ (cf. \cite[\S 1.10]{chenevier_det}). Then for each place $v \in R - R^c$, for each $\sigma \in G_{F, S}$, and for each $\tau_v \in I_{F_v}$, we have $\Res_v^{(2n)!}\widetilde{\tr}_{c, R}(\sigma(\tau_v-1)P_{v, \phi_{v}}(\phi_{v})) = 0$.
	\end{enumerate}
\end{prop}
 We note that this result, in the case where $R$ is empty, is Proposition \ref{prop:existence_of_hecke_representation_for_U(n,n)_no_R}. The result in this case is also contained implicitly in the proof of ~\cite[Cor. 5.2.6]{scholze-torsion}.
\begin{proof}
The proposition can be proved by re-doing the proof of ~\cite[Cor. 5.2.6]{scholze-torsion} to keep track of the action of the additional Hecke operators at $R$. For the reader's benefit, we single out the following
        essential statement (cf.~\cite[Thm. 4.3.1, Cor. 5.1.11]{scholze-torsion}): let $C = \widehat{\overline{\Q}}_p$, and let $m \geq 1$ be an integer, and let $\T_\text{cl}$ denote $\widetilde{\T}_R^T$, endowed with the weakest topology for which all of the maps
	\[ \widetilde{\T}_R^T \to \End_C(H^0(\mathcal{X}_{\widetilde{K}^p \widetilde{K}_p}, \omega^{mk}_{\widetilde{K}^p \widetilde{K}_p} \otimes \cI )  ) \]
	are continuous. (Here the right-hand side, defined as in the statement of \cite[Thm. 4.3.1]{scholze-torsion}, is endowed with its natural (finite dimensional $C$-vector space) topology and we are varying over all $k \geq 1$ and open compact subgroups $\widetilde{K}_p \subset \widetilde{G}(F^+_p)$ such that $\widetilde{K}^p \widetilde{K}_p$ is a good subgroup.) Then for any continuous quotient $\T_\text{cl} \to A$, where $A$ is a ring with the discrete topology, there is a unique $A$-valued determinant $D_A$ of $G_{F, S}$ of dimension $2n$ satisfying the following conditions:
	\begin{itemize}
		\item For each place $v\not\in S$ of $F$, the characteristic polynomial of $\Frob_v$ equals the image of $\widetilde{P}_v(X)$ in $A[X]$.
		\item For each place $v \in R$, and for each element $\sigma \in W_{F_v}$, the characteristic polynomial of $\sigma$ equals the image of $\widetilde{P}_{v, \sigma}(X)$ in $A[X]$.
		\item Let $\tr_A : A[G_{F, S}] \to A$ denote the trace associated to $D_A$. Then for each place $v \in R - R^c$, for each $\sigma \in G_{F, S}$, and for each $\tau_v \in I_{F_v}$, we have $\Res_v^{(2n)!} \tr_A(\sigma(\tau_v-1)P_{v, \phi_{v}}(\phi_{v})) = 0$.
	\end{itemize}
	This statement can be proved in exactly the same way as ~\cite[Cor. 
	5.1.11]{scholze-torsion}, by combining \cite[Example 2.32]{chenevier_det} with the following observation: take a cuspidal, cohomological automorphic representation $\widetilde{\pi}$ of $\widetilde{G}(\A_{F^+})$ such that $\widetilde{\pi}^{\infty, \widetilde{K}} \neq 0$ and an isomorphism $\iota : \overline{\bQ}_p \to \bC$, and let $\widetilde{\T}_R^T(\widetilde{\pi}) = \im( \widetilde{\T}_R^T \otimes_\cO \overline{\bQ}_p \to \End_{\overline{\bQ}_p}(\iota^{-1}\widetilde{\pi}^{\infty, \widetilde{K}}))$. Consider the associated Galois representation (whose existence and local properties are described by  Theorem 
	\ref{thm:base_change_and_existence_of_Galois_for_tilde_G}):
	\[ r_\iota(\widetilde{\pi}) : G_{F, S} \to \GL_{2n}(\overline{\bQ}_p), \]
	and let $\rho : G_{F, S} \to \GL_{2n}(\widetilde{\T}_R^T(\widetilde{\pi}))$ denote the composite of $r_\iota(\widetilde{\pi})$ with the inclusion $\GL_{2n}(\overline{\bQ}_p) \subset \GL_{2n}(\widetilde{\T}_R^T(\widetilde{\pi}))$. Then we have the following properties:
	\begin{itemize}
	\item For each place $v\not\in S$ of $F$, the characteristic polynomial of $\rho(\Frob_v)$ equals the image of $\widetilde{P}_v(X)$ in $\widetilde{\T}_R^T(\widetilde{\pi})[X]$.
	\item For each place $v \in R$, and for each element $\sigma \in W_{F_v}$, the characteristic polynomial of $\rho(\sigma)$ equals the image of $\widetilde{P}_{v, \sigma}(X)$ in $\widetilde{\T}_R^T(\widetilde{\pi})[X]$.
	\item For each place $v \in R^c - R$ and for each $\tau_v \in I_{F_v}$, we have $\Res_v^{(2n)!} \rho((\tau_v-1)P_{v, \phi_{v}}(\phi_{v})) = 0$ in $M_{2n}(\widetilde{\T}_R^T(\widetilde{\pi}))$.
	\end{itemize}
	The first two points follow from Theorem 
	\ref{thm:base_change_and_existence_of_Galois_for_tilde_G} and Proposition \ref{prop_action_of_central_elements_of_pro_p_iwahori_hecke_algebra} (note that the images of $\widetilde{P}_v(X)$ and $\widetilde{P}_{v, \sigma}(X)$ in  $\widetilde{\T}_R^T(\widetilde{\pi})[X]$ in fact lie in $\overline{\bQ}_p[X]$). The third point follows from the same Theorem and Corollary \ref{cor:existence_of_unramified_subrep}. (Our appeals to Theorem 
	\ref{thm:base_change_and_existence_of_Galois_for_tilde_G} here are the source of our assumption, at 
	the beginning of this section, that each place of $R$ has residue 
	characteristic which splits in an imaginary quadratic subfield of $F$.) 
\end{proof}
\begin{cor}\label{cor:lgc_at_l_neq_p_det_version_for_partial_tilde_G}
There exists an integer $N \geq 1$, depending only on $[F : \Q]$ and $n$, 
	an ideal $\widetilde{I}_{\partial, R} \subset \widetilde{\T}^T_R(R \Gamma(\partial \widetilde{X}_{\widetilde{K}}, 
	\cO))$ satisfying $\widetilde{I}_{\partial, R}^N = 0$, and a 
	$\widetilde{\T}^T_R(R \Gamma(\partial \widetilde{X}_{\widetilde{K}}, 
	\cO)) / \widetilde{I}_{\partial, R}$-valued determinant 
	$\widetilde{D}_{\partial, R}$ on $G_{F, S}$ of dimension $2n$ satisfying the 
	following conditions:
	\begin{enumerate}
		\item For each place $v \not\in S$ of $F$, the characteristic polynomial of $\Frob_v$ is equal to the image of $\widetilde{P}_{v}(X)$ 		in $(\widetilde{\T}^T_R(R \Gamma(\partial \widetilde{X}_{\widetilde{K}}, 
	\cO))/\widetilde{I}_{ \partial, R})[X]$.
		\item For each place $v \in R$, and for each element $\sigma \in W_{F_v}$, the characteristic polynomial of $\sigma$ is equal to the image of $\widetilde{P}_{v, \sigma}(X)$ in  $(\widetilde{\T}^T_R(R \Gamma(\partial \widetilde{X}_{\widetilde{K}}, 
	\cO))/\widetilde{I}_{\partial, R})[X]$.
	\item Let $\widetilde{\tr}_{\partial, R} : \widetilde{\T}^T_R(R \Gamma(\partial \widetilde{X}_{\widetilde{K}}, 
	\cO))[G_{F, S}] \to \widetilde{\T}^T_R(R \Gamma(\partial \widetilde{X}_{\widetilde{K}}, 
	\cO)) / \widetilde{I}_{\partial, R}$ be the trace associated to $\widetilde{D}_{\partial, R}$. Then for each place $v \in R - R^c$, for each $\sigma \in G_{F, S}$, and for each $\tau_v \in I_{F_v}$, we have $\Res_v^{(2n)!} \tr_{\partial, R}(\sigma(\tau_v-1)P_{v, \phi_{v}}(\phi_{v})) = 0$.
	\end{enumerate}
\end{cor}
\begin{proof}
There is a $\widetilde{\T}^T_R$-equivariant exact triangle in $\mathbf{D}(\cO)$:
\[ \xymatrix@1{ R \Gamma_c(\widetilde{X}_{\widetilde{K}}, \cO) \ar[r] &  R \Gamma(\widetilde{X}_{\widetilde{K}}, \cO) \ar[r] & R \Gamma(\partial \widetilde{X}_{\widetilde{K}}, \cO) \ar[r] &, } \]
and consequently a natural homomorphism
\[ \widetilde{\T}^T_R( R \Gamma_c(\widetilde{X}_{\widetilde{K}}, \cO) \oplus R \Gamma(\widetilde{X}_{\widetilde{K}}, \cO) ) \to \widetilde{\T}^T_R( R \Gamma(\partial \widetilde{X}_{\widetilde{K}}, 
	\cO) ) / \widetilde{J}, \]
	where $\widetilde{J}$ is an ideal of square 0. To prove the corollary, it is therefore enough to show that there is an integer $N \geq 1$, depending only on $[F : \Q]$ and $n$, an ideal $\widetilde{I} \subset \widetilde{\T}^T_R(R \Gamma(\widetilde{X}_{\widetilde{K}}, \cO))$ satisfying $\widetilde{I}^N = 0$, and a $\widetilde{\T}^T_R(R \Gamma(\widetilde{X}_{\widetilde{K}}, \cO))$-valued determinant $\widetilde{D}$ on $G_{F, S}$ of dimension $2n$ satisfying the following conditions:
		\begin{enumerate}
		\item For each place $v \not\in S$ of $F$, the characteristic polynomial $\widetilde{D}(X - \Frob_v)$ is equal to the image of $\widetilde{P}_{v}(X)$ in $\widetilde{\T}^T_R(R \Gamma(\widetilde{X}_{\widetilde{K}}, \cO)) / \widetilde{I} [X]$.
		\item For each place $v \in R$, and for each element $\sigma \in W_{F_v}$, the characteristic polynomial of $\sigma$ is equal to the image of $\widetilde{P}_{v, \sigma}(X)$ in $\widetilde{\T}^T_R(R \Gamma(\widetilde{X}_{\widetilde{K}}, \cO)) / \widetilde{I} [X]$.
			\item Let $\widetilde{\tr} : \widetilde{\T}^T_R(R \Gamma(\widetilde{X}_{\widetilde{K}}, 
	\cO))[G_{F, S}] \to \widetilde{\T}^T_R(R \Gamma(\widetilde{X}_{\widetilde{K}}, 
	\cO)) / \widetilde{I}$ be the trace associated to $\widetilde{D}$. Then for each place $v \in R - R^c$, for each $\sigma \in G_{F, S}$, and for each $\tau_v \in I_{F_v}$, we have $\Res_v^{(2n)!} \widetilde{\tr}(\sigma(\tau_v-1)P_{v, \phi_{v}}(\phi_{v})) = 0$.
	\end{enumerate}
	By Proposition \ref{prop_poincare_duality}, there is a commutative diagram (determined by Verdier duality)
	\[ \xymatrix{ \cH(\widetilde{G}^\infty, \widetilde{K}) \ar[r] \ar[d]_{\widetilde{\iota}} & \End_{\mathbf{D}(\cO)}(R \Gamma_c(\widetilde{X}_{\widetilde{K}}, \cO)) \ar[d] \\
	\cH(\widetilde{G}^\infty, \widetilde{K}) \ar[r] & \End_{\mathbf{D}(\cO)}(R \Gamma(\widetilde{X}_{\widetilde{K}}, \cO)). } \]
	Let $\widetilde{\iota}(\widetilde{\T}^T_R)(R \Gamma_c(\widetilde{X}_{\widetilde{K}}, \cO))$ denote the image of the composite map
	\[ \widetilde{\T}^T_R \to \cH(\widetilde{G}^\infty, \widetilde{K}) \otimes_\bZ \cO \overset{\widetilde{\iota}}{\to} \cH(\widetilde{G}^\infty, \widetilde{K}) \otimes_\bZ \cO \to \End_{\mathbf{D}(\cO)}(R \Gamma_c(\widetilde{X}_{\widetilde{K}}, \cO)), \]
	where the first and last maps are the canonical ones. The existence of the above commutative diagram shows that $\widetilde{\iota}$ descends to an isomorphism
	\[ \widetilde{\iota}(\widetilde{\T}^T_R)(R \Gamma_c(\widetilde{X}_{\widetilde{K}}, \cO)) \to \widetilde{\T}^T_R(R \Gamma(\widetilde{X}_{\widetilde{K}}, \cO)). \]
To complete the proof of the corollary, it is therefore enough to show that there is a determinant $\widetilde{D}_{c, R, \vee}$ of $G_{F, S}$ of dimension $2n$ with coefficients in a quotient $\widetilde{\iota}(\widetilde{\T}^T_R)(R \Gamma_c(\widetilde{X}_{\widetilde{K}}, \cO)) / \widetilde{I}_{c, R, \vee}$ by some nilpotent ideal $\widetilde{I}_{c, R, \vee}$ of exponent bounded solely in terms of $[F : \bQ]$ and $n$, and satisfying conditions analogous to those required of $\widetilde{D}$. Using the same argument as in the statement of Proposition \ref{prop:lgc_at_l_neq_p_det_version_for_tilde_G}, it is enough to show the following: let $\widetilde{\pi}$ be a cuspidal, cohomological automorphic representation of $\widetilde{G}(\A_{F^+})$ such that $\widetilde{\pi}^{\infty, \widetilde{K}} \neq 0$, and let $\widetilde{\iota}(\widetilde{\T}^T_R)(\widetilde{\pi})$ denote the image of the composite
\[ \widetilde{\T}^T_R \otimes_\cO \overline{\bQ}_p \to \cH(\widetilde{G}^\infty, \widetilde{K}^\infty) \otimes_\bZ \overline{\bQ}_p \overset{\widetilde{\iota}}{\to}  \cH(\widetilde{G}^\infty, \widetilde{K}^\infty) \otimes_\bZ \overline{\bQ}_p  \to \End_{\overline{\bQ}_p}(\iota^{-1} \widetilde{\pi}^{\infty, \widetilde{K}}). \]
Consider the associated Galois representation $r_\iota(\widetilde{\pi}) : G_{F, S} \to \GL_{2n}(\overline{\bQ}_p)$, and let $\rho : G_{F, S} \to \GL_{2n}(\widetilde{\iota}(\widetilde{\T}^T_R)(\widetilde{\pi}))$ denote the composite of $r_\iota(\widetilde{\pi})^\vee \otimes \epsilon^{1-2n}$ with the inclusion $\GL_{2n}(\overline{\bQ}_p) \subset \GL_{2n}( \widetilde{\iota}(\widetilde{\T}^T_R)(\widetilde{\pi}))$. Then we have the following properties:
\begin{itemize}
\item For each place $v\not\in S$ of $F$, the characteristic polynomial of $\rho(\Frob_v)$ equals the image of $\widetilde{P}_v(X)$ in $\widetilde{\iota}(\widetilde{\T}^T_R)(\widetilde{\pi})[X]$.
\item For each place $v \in R$, and for each element $\sigma \in W_{F_v}$, the characteristic polynomial of $\rho(\sigma)$ equals the image of $\widetilde{P}_{v, \sigma}(X)$ in $\widetilde{\iota}(\widetilde{\T}^T_R)(\widetilde{\pi})[X]$.
\item For each place $v \in R - R^c$ and for each $\tau_v \in I_{F_v}$, we have $\Res_v^{(2n)!} \rho(\sigma(\tau_v-1) P_{v, \phi_{v}}(\phi_{v})) = 0$ in $M_{2n}(\widetilde{\iota}(\widetilde{\T}^T_R)(\widetilde{\pi}))$.
\end{itemize}
To see why these properties hold, we note that there is a commutative diagram
\[ \xymatrix{ \cH(\widetilde{G}^\infty, \widetilde{K}) \ar[d]_{\widetilde{\iota}} \ar[r] & \End_{\overline{\bQ}_p}( \iota^{-1} \widetilde{\pi}^{\vee, \infty, \widetilde{K}}) \ar[d] \\
\cH(\widetilde{G}^\infty, \widetilde{K}) \ar[r] & \End_{\overline{\bQ}_p}( \iota^{-1} \widetilde{\pi}^{ \infty, \widetilde{K}}) } \]
where the horizontal arrows are the canonical ones and the right vertical arrow is tranpose with respect to the natural duality between $\iota^{-1} \widetilde{\pi}^{\infty, \widetilde{K}}$ and $\iota^{-1} \widetilde{\pi}^{\vee, \infty, \widetilde{K}}$. In particular, $\widetilde{\iota}$ determines an isomorphism $\widetilde{\T}^T_R(\widetilde{\pi}^\vee) \to \widetilde{\iota}(\widetilde{\T}^T_R)(\widetilde{\pi})$. The above points therefore follow from the analogous points for the cuspidal, cohomological automorphic representation $\widetilde{\pi}^\vee$ of $\widetilde{G}(\A_{F^+})$, already established in the proof of Proposition \ref{prop:lgc_at_l_neq_p_det_version_for_tilde_G}, together with the observation that there is an isomorphism $r_\iota(\widetilde{\pi}^\vee) \cong r_\iota(\widetilde{\pi})^\vee \otimes \epsilon^{1-2n}$. 
	\end{proof}
	We need one more lemma, which is an analogue of Hensel's lemma for group determinants. 
	\begin{lemma}\label{lem:separating_determinants}
	Let $A$ be a complete Noetherian local $\cO$-algebra with residue field $k$ and let $\Gamma$ be a  group. Fix natural numbers $n_1, n_2$ and set $n = n_1 + n_2$. Suppose given group determinants $D_1, D_2$ of $\Gamma$ of dimensions $n_1, n_2$ with coefficients in $A$, and let $D = D_1 D_2$. Suppose moreover that, if $\overline{D}_i = D_i \text{ mod }\ffrm_A$, then the semisimple representations $\overline{\rho}_i : \Gamma \to \GL_{n_i}(k)$ with $\det \overline{\rho}_i = \overline{D}_i$ for $i = 1, 2$ have no common Jordan--H\"older factors. 
	
	Then:
	\begin{enumerate}
	\item For any other group determinants $E_1, E_2$ of $\Gamma$ of dimensions $n_1, n_2$ with $E_i \text{ mod }\ffrm_A = \overline{D}_i$ for $i = 1, 2$ and $E_1 E_2 = D$, we have $E_1 = D_1$ and $E_2 = D_2$.
	\item We have $\ker D = \ker(D_1) \cap \ker(D_2)$.
	\end{enumerate}
	\end{lemma}
	\begin{proof}
	We will give an expression for $D_1$ which depends only on $D$, $\overline{D}_1$, and $\overline{D}_2$. This will establish the first part of the lemma. Let $R = A[\Gamma]$ and let $S = R / \operatorname{CH}(D)$, where the Cayley--Hamilton ideal $\operatorname{CH}(D)$ is defined in \cite[\S 1.17]{chenevier_det}. By \cite[Lemma 1.21]{chenevier_det}, the homomorphism $R \to M_{n_1}(k) \times M_{n_2}(k)$ determined by $\overline{D}_1$, $\overline{D}_2$ factors through $S$. Let $\overline{e}_1, \overline{e}_2 \in M_{n_1}(k) \times M_{n_2}(k)$ be the central idempotents which are the identity in one factor and zero in the other. Following \cite[p. 32, footnote]{bellaiche_chenevier_pseudobook}, we may lift $\overline{e}_1, \overline{e}_2$ to idempotents $e_1, e_2 \in S$ such that $e_1 + e_2 = 1$ and $e_1e_2 = 0$.
	
	We now consider the polynomial law $D_{1, e_1}$ on $e_1 S e_1$ given by the formula $D_{1, e_1}(x) = D_1(x + e_2)$. According to \cite[Lemma 2.4]{chenevier_det}, $D_{1, e_1}$ is a determinant $e_1 S e_1 \to A$ of some dimension $d_1 \leq n_1$. Reducing modulo $\ffrm_A$, we see that $d_1 = n_1$. It follows that the polynomial law $D_{1, e_2}$ on $e_2 S e_2$ given by the formula $D_{1, e_2}(x) = D_1(x + e_1)$ is of dimension 0, therefore constant and equal to 1. Working over $A[X]$, and invoking \cite[Lemma 2.4(2)]{chenevier_det}, we have
	\[ D_1(X - e_2) = D_{1, e_1}(X) =  X^{n_1}, \]
	hence $e_2^{n_1} = e_2 \in \operatorname{CH}(D_1) \subset \ker(D_1)$. Similarly we deduce that $e_1 \in \ker(D_2)$. We find that for any $A$-algebra $B$ and any $x \in S \otimes_A B$, we have
	\[ D_1(x) = D_1( e_1 x + e_2 x) = D_1(e_1 x), \]
	and so
	\[ D(e_1 x + e_2) = D_1(x e_1) = D_1(x). \]
	Since the expression $D_1(x) = D(e_1 x + e_2)$ only depends on $D$, $\overline{D}_1$, and $\overline{D}_2$, this proves the first part of the lemma. For the second, we note that the inclusion $\ker(D_1) \cap \ker(D_2) \subset \ker(D)$ follows immediately from the definition. For the other inclusion, take $x \in \ker(D)$, an $A$-algebra $B$, and $y \in R \otimes_A B$. By symmetry, it is enough to show that $D_1 ( 1 + xy) = 1$. We have
	\[ D_1(1 + xy) = D_1(e_1 (1+xy)) = D(1 + e_1 x y) = 1, \]
	since $e_1 x \in \ker(D)$. This concludes the proof. 
	\end{proof}
We can now complete the proof of Proposition \ref{prop:lgc_at_l_neq_p_det_version}. 
\begin{proof}[Proof of Proposition \ref{prop:lgc_at_l_neq_p_det_version}]
Let $\widetilde{\m} = \cS^\ast(\m) \subset \widetilde{\T}^S$. By Theorem \ref{thm:Hecke_reduction_to_Siegel_with_R_ramification}, the map $\cS$ descends to a homomorphism
\[ \widetilde{\T}^T_R(R \Gamma(\partial \widetilde{X}_{\widetilde{K}}, \cO)_{\widetilde{\m}}) \to \T^T_R(R \Gamma(X_K, \cO)_\m).\]
By Propositions \ref{prop_satake_transform_unramified_case}, \ref{prop_satake_transform_ramified_iwahori_case} and \ref{prop_satake_transform_half_ramified_iwahori_case} and Corollary  \ref{cor:lgc_at_l_neq_p_det_version_for_partial_tilde_G}, we see that we can find an integer $N \geq 1$, depending only on $[F : \bQ]$ and $n$, an ideal $I_R \subset \T^T_R(R \Gamma(X_K, \cO)_\m)$ satisfying $I_R^N = 0$, and a  $\T^T_R(K, 0)_\ffrm / I_R$-valued determinant $D'$ on $G_{F, S}$ of dimension $2n$ satisfying the following conditions:
	\begin{enumerate}
		\item For each place $v\not\in S$ of $F$, the characteristic polynomial of $\Frob_v$ under $D'$ is equal to the image of $P_v(X) q_v^{n(2n-1)} P_{v^c}^\vee(q_v^{1-2n} X)$ in $(\T^T_R(K, 0)_\ffrm / I_R)[X]$. 
		\item For each $v \in R$ and for each $\sigma \in W_{F_v}$, the characteristic polynomial of $\sigma$ under $D'$ is equal to the image of $P_{v, \sigma}(X) \| \sigma \|_v^{n(1-2n)} P^\vee_{v^c, \sigma^c}( \| \sigma \|_v^{2n - 1} X)$ in $(\T^T_R(K, 0)_\ffrm / I_R)[X]$. 
	\item Let $\tr' : \T^T_R(R \Gamma(X_K, \cO)_\m)[G_{F, S}] \to \T^T_R(R \Gamma(X_K, \cO)_\m) / I_R$ be the trace associated to $D'$. Then for each place $v \in R - R^c$, for each $\sigma \in G_{F, S}$, and for each $\tau_v \in I_{F_v}$, we have $\Res_v^{(2n)!} \tr'(\sigma(\tau_v-1)P_{v, \phi_{v}}(\phi_{v})) = 0$.
	\end{enumerate}
By Theorem \ref{thm:existence_of_Hecke_repn_for_GL_n}, we can assume (after possibly enlarging $I_R$ and increasing $N$ in a way still depending only on $[F : \bQ]$ and $n$) that there exists a continuous representation $\rho_\ffrm : G_{F, S} \to \GL_n( \T^T_R(K, 0)_\ffrm / I_R )$ such that for each finite place $v\not\in S$ of $F$, $\det(X - \rho_\ffrm(\Frob_v))$ equals the image of $P_v(X)$ in $\T^T_R(K, 0)_\ffrm / I_R$. Let $D = \det \rho_\ffrm$. Looking at characteristic polynomials of Frobenius elements for places $v \not\in S$, we conclude that $D' = \det (\rho_\m \oplus \rho_\ffrm^{c, \vee} \otimes \epsilon^{1-2n}) = D (D^{c, \vee} \otimes \epsilon^{1-2n})$. (Note that our notation for twisted determinants is chosen so that it matches the twisted representation; the polynomial law underlying $D^{c, \vee} \otimes \epsilon^{1-2n}$ is given by twisting with $\det(\epsilon^{1-2n})= \epsilon^{n(1-2n)}$.) To complete the proof of Proposition \ref{prop:lgc_at_l_neq_p_det_version}, we need to show that $D$ satisfies the following conditions:
\begin{itemize}
\item For each place $v \in R$ and for each $\sigma \in W_{F_v}$, the characteristic polynomial of $\sigma$ under $D$ is the image of $P_{v, \sigma}(X)$ in $\T^T_R(K, 0)_\ffrm / I_R$.
\item $D$ factors through $G_{F, T}$ and for each $v \in S - T$, the characteristic polynomial of $\Frob_v$ under $D$ is the image of $P_v(X)$ in $\T^T_R(K, 0)_\ffrm / I_R$.
\end{itemize}
We will then (in the notation of Proposition \ref{prop:lgc_at_l_neq_p_det_version}) be able to take $D_{\m, R} = D$. We take points these in turn. If $v \in R$, then there is a unique $n$-dimensional group determinant $E_v$ of $W_{F_v}$ with coefficients in $\T^T_R(K, 0)_\ffrm$ such that for each $\sigma \in W_{F_v}$, the characteristic polynomial of $\sigma$ under $E_v$ equals the image of $P_{v, \sigma}(X)$. Similarly if $v \in R^c - R$ there is a unique $n$-dimensional group determinant $E_v$ of $W_{F_v}$ with coefficients in $\T^T_R(K, 0)_\ffrm$ which is unramified and such that the characteristic polynomial of $\Frob_v$ equals $P_v(X)$. Our assumptions imply that for each $v \in R$, we have $D|_{W_{F_v}} (D^{c, \vee} \otimes \epsilon^{1-2n})|_{W_{F_v}} = E_v (E_{v^c}^{c, \vee} \otimes \epsilon^{1-2n})$. We would like to deduce that $D|_{W_{F_v}} = E_v$.

We first show that this holds in any quotient of $\T^T_R(K, 0)_\ffrm$ by a maximal ideal. (Recall that $\m$ is, by assumption, a maximal ideal of $\T^S(K, 0)$, so that the ring $\T^T_R(K, 0)_\ffrm$ is not necessarily local.) By assumption (i.e. by Lemma \ref{lem:lgc_l_neq_p_reduction}(3)), there is a character $\psi : G_{F, S} \to \cO^\times$, unramified above $R \cup R^c \cup S_p$, such that for each place $v \in R$, the characteristic polynomials of $\overline{\psi}(\Frob_v) \overline{\rho}_\ffrm(\phi_v)$ and $\overline{\psi}(\Frob_{v^c})^{-1} ( \overline{\rho}_\m \oplus \overline{\rho}_\m^{c, \vee} \otimes \epsilon^{1-2n})(\phi_v)$ are coprime. On the other hand, Proposition \ref{prop:twisting_by_character} implies that we have for any $v \in R$ equalities 
\[ \det (\rho_\m \otimes \psi \oplus (\rho_\m \otimes \psi)^{c, \vee} \otimes \epsilon^{1-2n})|_{W_{F_v}} = (E_v \otimes \psi|_{W_{F_v}}  ) ((E_{v^c} \otimes \psi|_{W_{F_{v^c}}})^{c, \vee} \otimes \epsilon^{1-2n}). \]
Looking at the roots of the characteristic polynomial of $\phi_v$ in each determinant and using the bijection between group determinants over a finite field and isomorphism classes of semisimple representations \cite[Theorem 2.12]{chenevier_det}, we conclude that we must have $\det \overline{\rho}_\m|_{W_{F_v}} = E_v \text{ mod }\n$ for every maximal ideal $\n \subset \T^T_R(K, 0)_\ffrm$. Lemma \ref{lem:separating_determinants} then immediately implies that we have $D|_{W_{F_v}} = \det \rho_\m|_{W_{F_v}} = E_v$. 

It remains to check that for each place $v \in R - R^c$ (hence $v^c \in S - T$), $\rho_\m|_{W_{F_{v^c}}}$ is unramified and $\det(X - \rho_\m(\Frob_{v^c}))$ equals the image of $P_{v^c}(X)$ in $\T^T_R(K, 0)_\ffrm / I_R[X]$. Equivalently, we must check that for each place $v \in R - R^c$, $\rho_\m^{c, \vee} \otimes \epsilon^{1 - 2 n}|_{W_{F_v}}$ is unramified and $\det(X - (\rho_\m^{c, \vee} \otimes \epsilon^{1 - 2 n})(\Frob_v))$ equals the image of $q_v^{n(2n-1)}P_{v^c}^\vee(q_v^{1-2n}X)$ in $\T^T_R(K, 0)_\ffrm / I_R[X]$. The computation of $\det(X - (\rho_\m^{c, \vee} \otimes \epsilon^{1 - 2 n})(\Frob_v))$ follows from what we have done already, so we just need to show that $\rho_\m^{c, \vee} \otimes \epsilon^{1 - 2 n}|_{W_{F_v}}$ is unramified. (Note that this is stronger, in general, than the assertion that the associated group determinant of $\rho_\m^{c, \vee} \otimes \epsilon^{1 - 2 n}|_{W_{F_v}}$  is unramified.) To show this, we use the following set of relations, which follow on applying $\cS$ to the corresponding set of relations for the determinant $D'$:
	\begin{itemize}
	\item  For each place $v \in R - R^c$, for each $\sigma \in G_{F, S}$, and for each $\tau_v \in I_{F_v}$, we have 
\begin{multline*} \cS(\Res_v)^{(2n)!} \tr( \rho_\m(\sigma(\tau_v-1)P_{v, \phi_{v}}(\phi_{v}))) \\ + \cS(\Res_v)^{(2n)!}\tr ( (\rho_\m^{c, \vee} \otimes \epsilon^{1-2n})(\sigma(\tau_v-1)P_{v, \phi_{v}}(\phi_{v}))) = 0. 
\end{multline*}
	\end{itemize}
	We have already seen that if $v \in R - R^c$ then $P_{v, \phi_{v}}(\rho_\m(\phi_v)) = 0$, so we deduce that for each $v \in R - R^c$ and for each $\sigma \in G_{F, S}$ and for each $\tau_v \in I_{F_v}$, we have 
	\[ \cS(\Res_v)^{(2n)!} \tr(  (\rho_\m^{c, \vee} \otimes \epsilon^{1-2n})(\sigma(\tau_v-1)P_{v, \phi_{v}}(\phi_{v}))) = 0. \]
By definition, $\Res_v \in \widetilde{\T}^T_R$ is the resultant of the polynomials $P_{v, \phi_v}(X)$ and $P_{v^c, \phi_{v^c}^{-1}}(X)$ in  $\widetilde{\T}^T_R[X]$. The images of these polynomials in $\T^T_R[X]$ under the map $\cS$ are computed by Proposition \ref{prop_satake_transform_half_ramified_iwahori_case}; they are (respectively) $P_{v, \phi_v}(X)$ and $q_v^{n(2n-1)} P_{v^c, \phi_{v^c}^{-1}}(q_v^{1-2n} X) = q_v^{n(2n-1)} P_{v^c}^\vee(q_v^{1-2n}X)$. Thus $\cS(\Res_v) \in \T^T_R$ is the resultant of these two polynomials, and the image of $\cS(\Res_v)$ modulo any maximal ideal of $\T^T_R(K, 0)_\ffrm$ coincides with the resultant of $\det(X - \overline{\rho}_\m(\phi_v))$ and $\det(X - \overline{\rho}_\m^{c, \vee} \otimes \epsilon^{1 - 2n}(\phi_v))$. These polynomials in $k[X]$ are coprime by assumption (cf. Lemma \ref{lem:lgc_l_neq_p_reduction}), so we find that $\cS(\Res_v)$ is a unit in $\T^T_R(K, 0)_\m$ and therefore that we have the stronger identity
\[ \tr(  (\rho_\m^{c, \vee} \otimes \epsilon^{1-2n})(\sigma(\tau_v-1)P_{v, \phi_{v}}(\phi_{v}))) = 0. \]
The matrix $(\rho^{c, \vee}_\m \otimes \epsilon^{1-2n})( P_{v, \phi_{v}}(\phi_{v})) = P_{v, \phi_v}((\rho^{c, \vee}_\m \otimes \epsilon^{1-2n})(\phi_v)) $ has unit determinant. Since $\overline{\rho}_\m$ is absolutely irreducible and $\sigma \in G_{F, S}$ is arbitrary, we conclude that we must have $(\rho^{c, \vee}_\m \otimes \epsilon^{1-2n})(\tau_v - 1) = 0$ for all $\tau_v \in I_{F_v}$ or, equivalently, that $\rho_\m|_{W_{F_{v^c}}}$ is unramified. This is what we needed to show. 
\end{proof}

\section{Local-global compatibility, \texorpdfstring{$l=p$}{l eq p} (Fontaine--Laffaille case)}
\label{section:lep}

	\subsection{Statements}\label{sec:FL_statements}
	
	Let $F$ be a CM field containing an imaginary quadratic field, and fix an integer $n \geq 1$. Let $p$ be a prime, and let $E$ be a finite extension of $\Q_p$ inside $\overline{\Q}_p$ large enough to contain the images of all embeddings of $F$ in $\overline{\Q}_p$. We assume throughout this chapter that $F$ satisfies the following standing hypothesis:
	\begin{itemize}
		\item The prime $p$ is unramified in $F$. Moreover, $F$ contains an imaginary quadratic field in which $p$ splits.
	\end{itemize}
	Let $K \subset \GL_n(\A_F^\infty)$ be a good subgroup, and let $\lambda \in (\Z_+^n)^{\Hom(F, E)}$. Let $S$ be a finite set of finite places of $F$, containing the $p$-adic places, stable under complex conjugation, and satisfying the following condition:
	\begin{itemize}
		\item Let $v$ be a finite place of $F$ not contained in $S$, and let $l$ be its residue characteristic. Then either $S$ contains no $l$-adic places of $F$ and $l$ is unramified in $F$, or there exists an imaginary quadratic subfield of $F$ in which $l$ splits.
	\end{itemize}
	We recall (Theorem \ref{thm:existence_of_Hecke_repn_for_GL_n})
        that under these hypotheses, that if $\ffrm \subset \T^S(K,
        \lambda)$ is a non-Eisenstein maximal ideal, then there is a
        continuous homomorphism %
	\[ \rho_\ffrm : G_{F, S} \to \GL_n(\T^S(K, \lambda)_\ffrm / J) \]
	characterized, up to conjugation, by the characteristic polynomials of Frobenius elements at places $v \not\in S$; here $J$ is a nilpotent ideal whose exponent depends only on $n$ and $[F : \Q]$. Our goal in this chapter is to show that under certain conditions, we can show that the restrictions of $\rho_\ffrm$ to decomposition groups at the $p$-adic places of $F$ satisfy conditions coming from $p$-adic Hodge theory. More precisely, we can show, after perhaps enlarging the nilpotent ideal $J$, that they are Fontaine--Laffaille with the expected Hodge--Tate weights. 
	
	Before stating the main theorem of this chapter we first briefly recall some of the properties of the Fontaine--Laffaille functor \cite{fl}, with normalizations as in 
	\cite[Section~2.4.1]{cht}. 

	Let $v$ be a $p$-adic place of $F$. We are assuming that $F_v / \Q_p$ is unramified. Let $\mathcal{MF}_{\cO}$ be the category of finite $\cO_{F_v} \otimes_{\Z_p} \cO$-modules $M$ equipped with the following data.
	\begin{itemize}
		\item A decreasing filtration $\Fil^i M$ of $\cO_{F_v} \otimes_{\Z_p} \cO$-submodules that are direct summands as $\cO_{F_v}$-modules.
		For an embedding $\tau \colon F_v \hookrightarrow E$, define the filtered $\cO$-module $M_\tau = M \otimes_{\cO_{F_v} \otimes_{\Z_p} \cO} \cO$
		where we view $\cO$ as an $\cO_{F_v}\otimes_{\Z_p} \cO$-algebra via $\tau \otimes 1$. 
		We assume that for each $\tau$, there is an integer $a_\tau$ such that $\Fil^{a_\tau} M_\tau = M_\tau$ and $\Fil^{a_\tau + p - 1} M_\tau = 0$.
		\item $\Frob_p^{-1} \otimes 1$-linear maps $\Phi^i \colon \Fil^i M \rightarrow M$ such that $\Phi^i|_{\Fil^{i+1} M} = p\Phi^{i+1}$ and 
		$M = \sum_i \Phi^i\Fil^i M$.
	\end{itemize}
	Note that for $M \in \mathcal{MF}_{\cO}$, 
	\[
	\Fil^i M = \prod_\tau \Fil^i M_\tau \quad \text{and} \quad 
	\Phi^i = \prod_\tau \Phi_\tau^i \quad \text{with} \quad \Phi_\tau^i \colon \Fil^i M_\tau \to M_{\tau\circ \Frob_p^{-1}}.
	\]
	Given a tuple of integers $a = (a_\tau) \in \Z^{\Hom(F_v,E)}$, we let $\mathcal{MF}_{\cO}^a$ be 
	the full subcategory of $\mathcal{MF}_{\cO}$ consisting of objects $M$ such that for each $\tau$, $\Fil^{a_\tau} M_\tau = M_\tau$ and $\Fil^{a_\tau + p-1} M_\tau = 0$.
	We write $\mathcal{MF}_{\cO}^0$ for $\mathcal{MF}_{\cO}^{(0,\ldots,0)}$.
	We let $\mathcal{MF}_{k}$ and $\mathcal{MF}_{k}^a$ be the full subcategories of $\mathcal{MF}_{\cO}$ and $\mathcal{MF}_{\cO}^a$, 
	respectively, of objects annihilated by $\varpi$.
	
	When $p > 2$, there is an exact, fully faithful, covariant functor $\bG^0$ from $\mathcal{MF}_{\cO}^0$ to the category of finite $\cO$-modules with continuous 
	$\cO$-linear $G_{F_v}$-action (see \cite[Section~2.4.1]{cht}, where $\bG^0$ is denoted $\bG$).
	The essential image of $\bG^0$ is closed under subquotients, and the restriction of $\bG^0$ to $\mathcal{MF}_{k}^0$ takes values in the category of continuous $G_{F_v}$-representations on finite dimensional 
	$k$-vector spaces.
	Moreover, if $M_1$ and $M_2$ are objects of $\mathcal{MF}_{\cO}^0$ such that 
	$M_1 \otimes_{\cO_{F_v}\otimes_{\Zp} \cO} M_2$ also lies in $\mathcal{MF}_{\cO}^0$, then 
	\begin{equation}\label{eq:FL_tensor}
	\bG^0(M_1\otimes_{\cO_{F_v}\otimes_{\Z_p}\cO} M_2) = \bG^0(M_1)\otimes_{\cO} \bG^0(M_2).
	\end{equation}
	We extend $\bG^0$ to a functor $\bG$ on $\mathcal{MF}_{\cO}$ by twisting as follows. 
	Fix $M \in \mathcal{MF}_{\cO}$ and $a = (a_\tau) \in \Z^{\Hom(F_v,E)}$ such that $M \in \mathcal{MF}_{\cO}^a$.  
	Define the crystalline character $\psi_a \colon G_{F_v} \to \cO^\times$ by
	\[
	\psi_a\circ \Art_{F_v}(x) = \prod_{\tau} \tau(x)^{-a_\tau} \quad \text{for} \quad x\in \cO_{F_v}^\times 
	\quad \text{and} \quad \psi_a\circ \Art_{F_v}(p) = 1,
	\]
	and the object $M(a) \in \mathcal{MF}_{\cO}^0$ by $\Fil^i M(a)_\tau = \Fil^{i+a_\tau} M_\tau$ 
	and $\Phi_{M(a),\tau}^i = \Phi_{M,\tau}^{i+a_\tau}$.
	We then set
	\[
	\bG(M) = \bG^0(M(a))\otimes_{\cO} \psi_a.
	\]
	Using \eqref{eq:FL_tensor}, one checks that this is independent of $a$ such that $M \in \mathcal{MF}_{\cO}^a$. 
	We will denote by $\bG^a$ the restriction of $\bG$ to $\mathcal{MF}_{\cO}^a$. 
	Any $\bG^a$ is fully faithful and its essential image is stable under subquotients, 
	but $\bG$ is not full on all of $\mathcal{MF}_{\cO}$.
	We note also that the essential image of $\bG$ is stable under twists by crystalline characters.
	
	Let $\overline{M}$ be an object of $\mathcal{MF}_{k}$.
	For each embedding $\tau \colon F_v \hookrightarrow E$, we let $\mathrm{FL}_\tau(\overline{M})$ be the multiset of integers $i$ such that
	\[
	\gr^i \overline{M} \otimes_{\cO_{F_v} \otimes_{\Z_p} k} k \ne 0,
	\]
	counted with multiplicity equal to the $k$-dimension of this space, where we view 
	$k$ as a $\cO_{F_v} \otimes_{\Z_p} k$ algebra via $\tau \otimes 1$.
	If $p > 2$ and $M$ is a $p$-torsion free object of $\mathcal{MF}_{\cO}$, the representation $\bG(M) \otimes_{\cO} E$ is crystalline 
	and for every embedding $\tau \colon F_v \hookrightarrow E$ we have
	\[
	\HT_\tau(\bG(M) \otimes_{\cO} E) = \mathrm{FL}_\tau(M \otimes_{\cO} k).
	\]
	Moreover, if $W$ is an $\cO$-lattice in a crystalline representation of $G_{F_v}$ such that every $\tau$-Hodge--Tate weight lies in 
	$[a_\tau, a_\tau + p - 2]$ for some integer $a_\tau$, then $W$ is in the essential image of $\bG^{a}$.
	
	We can now state the main theorem of this chapter (with the same numbering as it occurs again immediately before the proof).
	\begin{reptheorem}{myAmazingTheorem}
	Let $\m \subset \T^S(K, \lambda)$ be a non-Eisenstein maximal ideal. Suppose that~$\T^S(K, \lambda) / \m = k$ has residue
	characteristic~$p$. Let $\overline{v}$ be a $p$-adic place of $F^{+}$, and suppose that the following additional conditions are satisfied:
\begin{enumerate}
\item The prime $p$ is unramified in $F$, and $F$ contains an imaginary quadratic field in which $p$ splits.
		\item Let $w$ be a finite place of $F$ not contained in $S$, and let $l$ be its residue characteristic. Then either $S$ contains no $l$-adic places of $F$ and $l$ is unramified in $F$, or there exists an imaginary quadratic field $F_0 \subset F$ in which $l$ splits.
		\item For each place $v | \overline{v}$ of $F$, $K_v = \GL_n(\cO_{F_v})$.
		\item For every embedding $\tau : F 
		\hookrightarrow E$ inducing the place $\overline{v}$ of $F^+$,
		\[
		\lambda_{\tau,1} + \lambda_{\tau c, 1} - \lambda_{\tau, n} - \lambda_{\tau c, n} \leq p - 2n - 1.
		\]
		\item $p > n^2$.
		\item There exists a $p$-adic place $\overline{v}' \neq \overline{v}$ of $F^+$ such that 
		\[
		\sum_{\substack{ \overline{v}'' \neq \overline{v}, \overline{v}'}} [F^+_{\bar v''}: \Q_p] >\frac{1}{2}[F^+:\Q].
		\]
		\item $\overline{\rho}_\m$ is decomposed generic (Definition~\ref{defn:decomposed_generic}).
		\item Assume that one of the following holds:
		\begin{enumerate}
			\item $H^\ast(X_K,\cV_\lambda)_{\frakm}[1/p] \ne 0$, or
			\item for every embedding $\tau \colon F \hookrightarrow E$ inducing the place $\overline{v}$ of $F^+$,
			\[
			-\lambda_{\tau c,n} - \lambda_{\tau , n} \leq p - 2n - 2 \quad \text{and} \quad -\lambda_{\tau c,1} -\lambda_{\tau, 1} \geq 0.
			\]
		\end{enumerate}
	\end{enumerate}
	Then there exists an integer $N \geq 1$ depending only on $[F^+ : \Q]$ and $n$, an ideal $J \subset \T^S(K, \lambda)$ satisfying $J^N = 0$, and a continuous representation
	\[ \rho_\m : G_{F, S} \to \GL_n(\T^S(K, \lambda)_{\m} / J) \]
	satisfying the following conditions:
	\begin{enumerate}
		\item[(a)] For each finite place $v \not\in S$ of $F$, the characteristic polynomial of $\rho_\m(\Frob_v)$ equals the image of $P_v(X)$ in $(\T^S(K, \lambda)_{\m} / J)[X]$.
		\item[(b)] For each place $v | \overline{v}$ of $F$, $\rho_\m|_{G_{F_v}}$ is in the essential image of  $\bG^a$ (with $a = (\lambda_{\tau, n}) \in \Z^{\Hom(F_v,E)}$). 
		\item[(c)] There is $\overline{M} \in \mathcal{MF}_{k}$ such that 
		$\overline{\rho}_\m|_{G_{F_v}} \cong \bG(\overline{M})$ and for any embedding $\tau \colon F_v \hookrightarrow E$,
		\[
		\mathrm{FL}_\tau(\overline{M}) = \{\lambda_{\tau,1} + n - 1, \lambda_{\tau,2} + n - 2, \ldots,\lambda_{\tau,n}\}.
		\]
	\end{enumerate}
	\end{reptheorem}
The rest of this chapter is devoted to the proof of Theorem \ref{thm:lgcfl}. The proof will be by reduction to known results for automorphic forms on $\widetilde{G}$ (in particular, Theorem \ref{thm:base_change_and_existence_of_Galois_for_tilde_G}).

\subsection{A direct summand of the boundary cohomology}\label{sec:direct summand at 
finite level} 

In this section, we show how to realize the cohomology of $X_K$ as a
direct summand of the cohomology of the boundary $\partial
\widetilde{X}_{\widetilde{K}}$ of the Borel--Serre compactification of
$\widetilde{X}_{\widetilde{K}}$. This is the first step in relating
the cohomology of $X_K$ to automorphic forms on $\widetilde{G}$. We
must first introduce some new notation, in addition to the notation
introduced in Section~\ref{subsec: boundary cohomology}. 

We recall (cf.\ \S \ref{sec:unitary_group_setup}) that we write $\overline{S}_p$ for the set of $p$-adic places of $F^+$, $S_p$ for the set of $p$-adic places of $F$, and that we have fixed a subset $\widetilde{S}_p = \{ \widetilde{v} \mid \overline{v} \in \overline{S}_p \}$ with the property that $S_p = \widetilde{S}_p \sqcup \widetilde{S}_p^c$. Moreover, we write $\widetilde{I}_p$ for the set of embeddings $\tau : F \hookrightarrow E$ inducing a place of $\widetilde{S}_p$. For any $\overline{v} \in \overline{S}_p$, we write $\widetilde{I}_{\overline{v}}$ for the set of embeddings $\tau : F \hookrightarrow E$ inducing $\widetilde{v}$. Similarly, we write $I_{\overline{v}}$ for the set of embeddings $\tau : F^+ \hookrightarrow E$ inducing $\overline{v}$.

These choices determine an isomorphism $(\Res_{F^+ / \Q} \widetilde{G})_E \cong \prod_{\tau \in \widetilde{I}_p} \GL_{2n}$. For any embedding $\tau : F^+ \hookrightarrow E$, we set 
\[ W_\tau = W(\widetilde{G} \otimes_{F^+, \tau} E, T \otimes_{F^+, \tau} E) \]
and
\[ W_{P, \tau} = W(G \otimes_{F^+, \tau} E, T \otimes_{F^+, \tau} E); \]
these may be identified with the Weyl groups of $\GL_{2n}$ and $\GL_n \times \GL_n$, respectively. Since $\widetilde{G}$ is equipped with the Borel subgroup $B$, we may also define the subset $W^P_\tau \subset W_{\tau}$ of representatives for the quotient $W_{P, \tau} \backslash W_\tau$ (cf.\ \S \ref{sec:misc_notation}). We write $\rho_\tau \in X^\ast(T \otimes_{F^+, \tau} E)$ for the half-sum of the $B \otimes_{F^+, \tau} E$-positive roots.

If $\overline{v} \in \overline{S}_p$, then we set $W_{\overline{v}} = \prod_{\tau \in I_{\overline{v}}} W_\tau$, $W_{P, \overline{v}} = \prod_{\tau \in I_{\overline{v}}} W_{P, \tau}$, and $W^P_{\overline{v}} = \prod_{\tau \in I_{\overline{v}}} W^P_{\tau}$. We define $\rho_{\overline{v}} \in X^\ast((\Res_{F^+_{\overline{v}} / \Q_p} T)_E)$ to be the half-sum of the $(\Res_{F^+_{\overline{v}} / \Q_p} B)_E$-positive roots; thus we can identify $\rho_{\overline{v}} = \sum_{\tau \in \Hom(F^+_{\overline{v}}, E)} \rho_\tau$. Given a subset $\overline{T} \subset \overline{S}_p$, we set $W_{\overline{T}} = \prod_{\overline{v} \in \overline{T}} W_{\overline{v}}$, and define $W_{P, \overline{T}}$ and $W^P_{\overline{T}}$ similarly. If $\overline{T} = \overline{S}_p$, then we drop $\overline{T}$ from the notation; thus $W$ may be identified with the Weyl group $W((\Res_{F^+ / \Q}\widetilde{G})_E, (\Res_{F^+ / \Q} T)_E)$ of $(\Res_{F^+ / \Q}\widetilde{G})_E$. We write $l : W \to \Z_{\geq 0}$ for the length function with respect to the Borel subgroup $B$, and $\rho \in X^\ast((\Res_{F^+ / \Q} T)_E)$ for the half-sum of the $(\Res_{F^+ / \Q} B)_E$-positive roots; thus we can identify $\rho = \sum_{\overline{v} \in \overline{S}_p} \rho_{\overline{v}}$.

If $\widetilde{\lambda} \in (\Z^{2n}_+)^{\Hom(F^+, E)}$, and $\overline{v} \in \overline{S}_p$, then we set 
\[ \widetilde{\lambda}_{\overline{v}} = (\widetilde{\lambda}_{\tau})_{\tau \in \Hom_{\Q_p}(F^+_{\overline{v}}, E)} \in (\Z^{2n}_+)^{\Hom(F^+_{\overline{v}}, E)}. \]
If $\lambda \in (\Z^n_+)^{\Hom(F, E)}$, and $\overline{v} \in \overline{S}_p$, then we set
\[ \lambda_{\overline{v}} = (\lambda_\tau)_{\tau \in \Hom_{\Q_p}(F_{\widetilde{v}}, E) \sqcup \Hom_{\Q_p}(F_{\widetilde{v}^c}, E)} \in (\Z^n_+)^{\Hom_{\Q_p}(F \otimes_{F^+} F^+_{\overline{v}}, E)}. \]
\begin{thm}\label{direct summand at finite level} 
	Let $\widetilde{K} \subset \widetilde{G}(\A_{F^+}^\infty)$ be a good subgroup which is decomposed with respect to $P$, and with the property that for each $\overline{v} \in \overline{S}_p$, $\widetilde{K}_{U, \overline{v}} = U(\cO_{F^+_{\overline{v}}})$. Let $\m \subset \T^S$ be a non-Eisenstein maximal ideal, and let $\widetilde{\m} = \cS^\ast(\m) \subset \widetilde{\T}^S$.
	
	Choose a partition 
	$\overline{S}_p = \overline{S}_1\sqcup \overline{S}_2$. Let $\widetilde{\lambda} \in (\Z^{2n}_+)^{\Hom(F^+, E)}$ and $\lambda \in (\Z^{n}_+)^{\Hom(F, E)}$ be dominant weights for $\widetilde{G}$ and $G$, respectively. We assume that the following conditions are satisfied: 
	\begin{enumerate}
		\item For each $\overline{v} \in \overline{S}_1$, $\widetilde{\lambda}_{\overline{v}} = \lambda_{\overline{v}}$ (identification as in (\ref{eqn:wt_dictionary})).
		\item For each $\overline{v} \in \overline{S}_2$, $\widetilde{\lambda}_{\overline{v}} = 0$.
		\item For each $\overline{v} \in \overline{S}_2$, there exists $w_{\overline{v}} \in W^P_{\overline{v}}$ such that $\lambda_{\overline{v}} = w_{\overline{v}}(\rho_{\overline{v}}) - \rho_{\overline{v}}$. 
		\item $p > n^2$. (We recall our blanket assumption throughout \S \ref{section:lep} that $p$ is unramified in $F$.)
	\end{enumerate}
	If $\overline{v} \in \overline{S}_1$, we let $w_{\overline{v}}$ denote the identity element of $W_{\overline{v}}$. We let $w = (w_{\overline{v}})_{\overline{v} \in \overline{S}_p}$.
	Then for any $m \geq 1$, $R \Gamma(X_K, \cV_\lambda / \varpi^m)_\m[ -l(w) ]$ is a $\widetilde{\T}^S$-equivariant direct summand of $R \Gamma(\partial\widetilde{X}_{\widetilde{K}}, \cV_{\widetilde{\lambda}} / \varpi^m)_{\widetilde{\m}}$.
\end{thm}
(If $S$ is a ring and $A, B \in \mathbf{D}(S)$ are complexes equipped with homomorphisms of $S$-algebras
\[ f_A : R \to \End_{\mathbf{D}(S)}(A), \, \, f_B : R \to \End_{\mathbf{D}(S)}(B), \]
then we say that $A$ is an $R$-equivariant direct summand of $B$ if there is a complex $C \in \mathbf{D}(S)$ equipped with a homomorphism of $S$-algebras
\[ f_C : R \to \End_{\mathbf{D}(S)}(C) \]
and an isomorphism $\phi : B \cong A \oplus C$ in $\mathbf{D}(S)$ such
that for each~$r\in R$, we have $f_B(r) = \phi^{-1} \circ (f_A(r) \oplus f_C(r) ) \circ \phi$.)
\begin{proof}  By Theorem~\ref{thm:reduction_to_Siegel}, it is enough to show that $R\Gamma (X_K, \cV_{\lambda} / \varpi^m
)[-l(w)]$ is a $\widetilde{\T}^S$-equivariant direct summand of 
	$R\Gamma(\widetilde{X}^P_{\widetilde{K}}, 
	\cV_{\widetilde{\lambda}} / \varpi^m)$. We will argue in a similar way to the proof of Theorem \ref{thm:Hecke_reduction_to_Siegel}.
	
	Looking at the proof of Theorem \ref{thm:Hecke_reduction_to_Siegel}, we see that there is a $\widetilde{\T}^S$-equivariant isomorphism 
	\[ R \Gamma(\widetilde{X}^P_{\widetilde{K}_P}, \cV_{\widetilde{\lambda}} / \varpi^m) \cong R \Gamma(\widetilde{K}_P^S \times K_S, R \Gamma( \operatorname{Inf}_{G^S \times K_S}^{P^S \times K_S} \mathfrak{X}_G, R 1_\ast^{\widetilde{K}_{U, S}} \cV_{\widetilde{\lambda}} / \varpi^m)) \]
	in $\mathbf{D}(\cO / \varpi^m)$, where $\widetilde{\T}^S$ acts on both sides via the map $r_P$, and that the current theorem will be proved if we can establish the following claim:
	\begin{itemize}
		\item $R 1_\ast^{\widetilde{K}_{U, S}} \cV_{\widetilde{\lambda}} / \varpi^m$ admits $\cV_\lambda / \varpi^m[-l(w)]$ as a direct summand in $\mathbf{D}(\text{Sh}_{P^S \times K_S }(\mathfrak{X}_G))$, the derived category of $P^S \times K_S$-equivariant sheaves  of $\cO / \varpi^m$-modules on $\mathfrak{X}_G$.
	\end{itemize}
	In fact, $R 1_\ast^{\widetilde{K}_{U, S}} \cV_{\widetilde{\lambda}} / \varpi^m$ is pulled back from $R \Gamma(\widetilde{K}_{U, S}, \cV_{\widetilde{\lambda}} / \varpi^m) \in \mathbf{D}(\text{Sh}_{K_S}(pt))$, so it suffices to show that $\cV_\lambda / \varpi^m[-l(w)]$ is a direct summand of $R \Gamma(\widetilde{K}_{U, S}, \cV_{\widetilde{\lambda}} / \varpi^m)$ in this category. 
	
	We observe that $\widetilde{K}_{U, S} = \prod_{\overline{v} \in \overline{S}} \widetilde{K}_{U, \overline{v}}$, and that $\cV_{\widetilde{\lambda}}$ admits a corresponding decomposition $\cV_{\widetilde{\lambda}} = \otimes_{\overline{v} \in \overline{S}_p} \cV_{\widetilde{\lambda}_v}$. By the K\"unneth formula, it is therefore enough to show the following two claims: 
	\begin{enumerate}
		\item If $\overline{v} \in \overline{S}_1$, then $\cV_{\lambda_{\overline{v}}} / \varpi^m$ is a direct summand of $R \Gamma( \widetilde{K}_{U, \overline{v}}, \cV_{\widetilde{\lambda}_{\overline{v}}} / \varpi^m)$ in $\mathbf{D}(\cO / \varpi^m[K_{\widetilde{v}} \times K_{\widetilde{v}^c}])$. 
		\item If $\overline{v} \in \overline{S}_2$, then $\cV_{\lambda_{\overline{v}}} / \varpi^m[-l(w_{\overline{v}})]$ is a direct summand of $R \Gamma( \widetilde{K}_{U, \overline{v}}, \cO / \varpi^m)$ in $\mathbf{D}(\cO / \varpi^m[K_{\widetilde{v}} \times K_{\widetilde{v}^c}])$.
	\end{enumerate}
	The first claim can be proved using the same argument as in the end of the proof of Theorem \ref{thm:Hecke_reduction_to_Siegel}. The second claim follows from Lemma \ref{lem:integral_kostant} and Lemma \ref{lem:direct sum decomposition for trivial coefficients} below (this is where we use our hypothesis $p > n^2$). This completes the proof of the theorem. 
\end{proof}

\begin{lemma}\label{lem:integral_kostant}
	Let $\overline{v} \in \overline{S}_p$, let $K = F^+_{\overline{v}}$, and fix an integer $m \geq 1$.
	\begin{enumerate}
		\item For each $i \in \Z_{\ge 0}$ there is a $G(\cO_K)$-equivariant isomorphism
		\[ H^i(U(\cO_K),\cO/\varpi^m) \cong 
		\Hom_{\Zp}(\wedge^i_{\Zp} U(\cO_K),\cO/\varpi^m) = 
		\Hom_{\cO}(\wedge^i_{\cO}(U(\cO_K)\otimes_{\Zp} \cO),\cO/\varpi^m)\] with 
		$G(\cO_K)$-action on the right hand side induced by its conjugation 
		action on 
		$U(\cO_K)$.
		\item Suppose $p \ge 2n - 1$. Given $w \in W_{\overline{v}}^P$, let $\lambda_w = w(\rho_{\overline{v}})-\rho_{\overline{v}} \in (\Z^n_+)^{\Hom_{\Q_p}(F \otimes_{F^+} F^+_{\overline{v}}, E)}$ (using the identification (\ref{eqn:wt_dictionary})).
		For each $i \in \Z_{\ge 0}$ there is a 
		$G(\cO_K)$-equivariant isomorphism \[ 
		\Hom_{\cO}(\wedge^i_{\cO}(U(\cO_K)\otimes_{\Zp} \cO),\cO) \cong 
		\bigoplus_{\substack{w \in W_{\bar v}^P\\
		l(w)=i}}\cV_{\lambda_w}\]
	\end{enumerate}
\end{lemma}
\begin{proof}
	Note that $U(\cO_K)$ is isomorphic (as an abstract group) to $\Z_p^{n^2[K : \Q_p]}$. The usual isomorphism $H^1(U(\cO_K), \cO / \varpi^m) \cong \Hom_{\Zp}(U(\cO_K), \cO / \varpi^m)$ extends, by cup product, to a morphism $\wedge^\ast \Hom_{\Zp}(U(\cO_K), \cO / \varpi^m) \to H^\ast(U(\cO_K), \cO / \varpi^m)$. This can be seen to be an isomorphism using the K\"unneth formula. This proves the first part of the lemma.  
	
	For the second part,  given $\tau \in \Hom_{\Q_p}(K, E)$ and $w \in W_\tau^P$, let $\lambda_w = w(\rho_\tau) - \rho_\tau \in (\Z^n_+)^2$. It is enough for us to show that for each $i \in \Z_{\ge 0}$ there is a 
	$G(\cO_K)$-equivariant isomorphism 
	\[\Hom_{\cO}(\wedge^i_{\cO}(U(\cO_K)\otimes_{\cO_K,\tau} \cO),\cO) \cong 
	\bigoplus_{\substack{w \in W_{\tau}^P\\
			l(w)=i}}\cV_{\lambda_w}.\]
	After tensoring up to $E$ we do have such an isomorphism, by \cite{kostant}:
	\[\Hom_{\cO}(\wedge^i_{\cO}(U(\cO_K)\otimes_{\cO_K,\tau} \cO),E) \cong 
	\bigoplus_{\substack{w \in W_{\tau}^P\\
			l(w)=i}}V_{\lambda_w}.\]
	Since $p \ge 2n-1 $, it follows from \cite[Cor.~II.5.6]{MR2015057} that 
	$\cV_{\lambda_w}\otimes_{\cO}k$ is a simple 
	$G \otimes_{\cO_K, \tau} k$-module for all $w \in W_{\tau}^P$. It follows that 
	intersecting 
	the lattice $\Hom_{\cO}(\wedge^i_{\cO}(U(\cO_K)\otimes_{\cO_K,\tau} \cO),\cO)$ 
	with 
	a copy of $V_{\lambda_w}$ arising from the above 
	decomposition gives a sublattice isomorphic to 
	$\cV_{\lambda_w}$.   By the remark 
	following \cite[Cor.~II.5.6]{MR2015057}, we know that there are no 
	non-trivial extensions between the simple modules 
	$\cV_{\lambda_w}\otimes_{\cO}k$ with varying $w$.  
	Combining this with the universal coefficient theorem 
	\cite[Prop.~I.4.18a]{MR2015057} we deduce that there are also no 
	non-trivial 
	extensions between the $G \otimes_{\cO_K, \tau} \cO$-modules 
	$\cV_{\lambda_w}$. This implies the existence of the desired isomorphism.
\end{proof}

\begin{lemma}\label{lem:direct sum decomposition for trivial coefficients} 
	Let $\overline{v} \in \overline{S}_p$, let $K = F^+_{\overline{v}}$, and fix an integer $m \geq 1$. Suppose that $p > n^2$.
	Then we have a 
	natural isomorphism (inducing the identity on 
	cohomology) 
	\[
	R\Gamma(U(\cO_K), \cO/\varpi^m)\toisom \bigoplus_{i=0}^{n^2 [K : \Q_p]} 
	H^i(U(\cO_K), \cO/\varpi^m)[-i]
	\] 
	in $\mathbf{D}(\cO/\varpi^m[G(\cO_K)])$. 
\end{lemma}

\begin{proof} 
	We have already observed that there is an isomorphism
	\[  R\Gamma(U(\cO_K), \cO/\varpi^m)\toisom H^0(U(\cO_K), \cO/\varpi^m) \oplus \tau_{\geq 1} R\Gamma(U(\cO_K), \cO/\varpi^m) \]
	(see claim (1) in the proof of Theorem \ref{direct summand at finite
          level}). %
        Under the assumption that $p > n^2$ we can distinguish the remaining
	degrees of 
	cohomology appearing in the above direct sum using the action of central 
	elements of $G(\cO_K)$. Let $f = [K : \Q_p]$. The centre of $G(\cO_K)$ is 
	$(\cO_F\otimes_{\cO_{F^+}} \cO_K)^\times$ and an element $z \in 
	(\cO_F\otimes_{\cO_{F^+}}\cO_K)^\times$ acts on $U(\cO_K)$ as multiplication by 
	$(N_{F/F^+}\otimes \mathrm{id})(z) \in \cO_K^\times$. We denote by $\zeta$ 
	a 
	primitive 
	$p^f-1$ root of unity in $\cO_K^\times$. We 
	can choose an element $z$ of the centre of $G(\cO_K)$ of order $p^f - 1$ which acts as 
	multiplication by $\zeta$ on $U(\cO_K)$. It follows from Lemma \ref{lem:integral_kostant} and the decomposition \[U(\cO_K)\otimes_{\Zp}\cO = 
	\bigoplus_{\sigma: 
		\cO_K \hookrightarrow \cO} U(\cO_K)\otimes_{\cO_K, \sigma}\cO\] that, for each 
	degree $i$, 
	we have a decomposition of $H^i(U(\cO_K), 
	\cO/\varpi^m)$ into a direct sum  of $G(\cO_K)$-modules \[M_{(i_\sigma)} = 
	\Hom_{\cO}\left(\bigotimes_{\sigma} 
	\wedge^{i_\sigma}_{\cO}(U(\cO_K)\otimes_{\cO_K, \sigma}
	\cO),\cO/\varpi^m\right)\]
	indexed 
	by $f$-tuples 
	of integers \[\{(i_\sigma)_{\sigma: \cO_K \hookrightarrow\cO}: 0\le 
	i_\sigma 
	\le n^2, \sum_{\sigma}i_\sigma = i\}.\]
	The action of $z$ on $M_{(i_\sigma)}$ is multiplication by 
	$\prod_{\sigma}\sigma(\zeta)^{-i_\sigma}$, so if we fix an embedding 
	$\sigma_0$ 
	and write $i_j$ for the $j$th Frobenius twist of $\sigma_0$ then $z$ acts 
	as 
	multiplication by $\sigma_0(\zeta)^{-\sum_{j=0}^{f-1} i_jp^j}$. Since 
	we 
	are 
	assuming $p > n^2$, the value of $\sum_{j=0}^{f-1} i_jp^j$ mod $p^f - 1$ 
	determines the integers $i_j$ uniquely, with the exception (only occurring 
	if $p 
	= n^2 + 1$) of when this value is $0$ mod $p^f-1$, in which case there are 
	two 
	possibilities: $i_\sigma = 0$ for all $\sigma$ and $i_\sigma = p-1$ for all 
	$\sigma$. As a consequence, for each degree $1 \le i \le n^2f$ we can write 
	down an idempotent $e_i \in \cO[z]$ which induces the identity on 
	$H^i(U(\cO_K), 
	\cO/\varpi^m)$ and the zero map on other degrees $i' \ne i$. 
	There is a homomorphism $\cO[z] \to \End_{\mathbf{D}(\cO/\varpi^m[G(\cO_K)])}(\tau_{\geq 1}R\Gamma(U(\cO_K), \cO / \varpi^m))$, so the idempotent-completeness of the derived 
	category 
	implies the existence of a natural decomposition \[
	\tau_{\ge 1}R\Gamma(U(\cO_K), \cO/\varpi^m) =  
	\bigoplus_{i=1}^{n^2f} e_i R\Gamma(U(\cO_K), 
	\cO/\varpi^m).
	\]
	This completes the proof.
\end{proof}

\subsection{Cohomology in the middle degree}

In this section we state the fundamental result that we need to study cohomology in the middle degree using automorphic representations of $\widetilde{G}$. We first need to recall a definition (\cite[Defn.\ 1.9]{caraiani-scholze-compact} --- although note that since our representations are in characteristic~$p$, 
the roles of~$p$ and~$\auxp$ are reversed).
\begin{defn}\label{defn:decomposed_generic}  Let~$k$ be a finite field of characteristic~$p$. \leavevmode
	\begin{enumerate}
		\item Let $\auxp \neq p$ be a prime, and let $L / \bQ_{\auxp}$ be a finite extension. We say that a continuous representation $\overline{r} : G_L \to \GL_n(k)$ is generic if it is unramified and the eigenvalues (with multiplicity) $\alpha_1, \dots, \alpha_n \in \overline{k}$ of $\overline{r}(\Frob_L)$ satisfy $\alpha_i / \alpha_j \neq  | \cO_L / \ffrm_L |$
for all $i \neq j$.
		\item Let $L$ be a number field, and let $\overline{r} : G_L \to \GL_n(k)$ be a continuous representation. We say that a prime $\auxp \neq p$ is decomposed generic for $\overline{r}$ if $\auxp$ splits completely in $L$ and for all places $v | \auxp$ of $L$, $\overline{r}|_{G_{L_v}}$ is generic.
		\item Let $L$ be a number field, and let $\overline{r} : G_L \to \GL_n(k)$ be a continuous representation. We say that $\overline{r}$ is decomposed generic if there exists a prime~$\auxp \neq p$ which is decomposed generic for $\overline{r}$.
	\end{enumerate}
\end{defn}
Note that if $\barr$ and $\barr'$ give rise to the same projective
representation then one is (decomposed) generic if and only if the other is.
\begin{lem}\label{lem:weak_generic}
Let $L$ be a number field, and let $\overline{r} : G_L \to \GL_n(k)$ be a continuous representation.
	Suppose that $\overline{r}$ is decomposed generic. Then there exist infinitely many primes $\auxp \neq p$ 
	which are decomposed generic for $\overline{r}$.
\end{lem}
\begin{proof}
	Let $K'/\bQ$ denote the Galois closure of the extension of
        $L(\zeta_p)$ cut out by $\overline{r}$. Let~$\auxp_0$ be a
        prime which is decomposed generic for~$\overline{r}$; then any
        other prime $\auxp$ which is unramified in $K'$ and such that
        $\Frob_{\auxp}$, $\Frob_{\auxp_0}$ lie in the same conjugacy
        class of $\Gal(K' / \bQ)$ is also decomposed generic for $\overline{r}$. There are infinitely many such primes, by the Chebotarev density theorem. 
\end{proof}
Let $d = n^2[F^+ : \Q] = \frac{1}{2} \dim_{\R} \widetilde{X} = \dim_{\R} X + 1$.
\begin{thm}\label{what we get from CS} Suppose that $[F^+ : \Q] > 1$. 
Let $\widetilde{\m} \subset \widetilde{\T}^S(\widetilde{K}, \widetilde{\lambda})$ be a maximal ideal,
and suppose that~$\overline{\rho}_{\widetilde{\m}}$ has length at most~$2$.
 Suppose that $S$ satisfies the following condition:
\begin{itemize}
	\item Let $v$ be a finite place of $F$ not contained in $S$, and let $l$ be its residue characteristic. Then either $S$ contains no $l$-adic places of $F$ and $l$ is unramified in $F$, or there exists an imaginary quadratic field $F_0 \subset F$ in which $l$ splits.
\end{itemize}	
 Suppose that $\overline{\rho}_{\widetilde{\m}}$ is decomposed generic, in the sense of 	Definition~\ref{defn:decomposed_generic}. 
 Then we have 
	\[
	H^d(\widetilde{X}_{\widetilde{K}},\cV_{\widetilde{\lambda}}[1/p])_{\widetilde{\m}}\hookleftarrow
	H^d(\widetilde{X}_{\widetilde{K}},\cV_{\widetilde{\lambda}})_{\widetilde{\m}}\twoheadrightarrow
	H^d(\partial\widetilde{X}_{\widetilde{K}},\cV_{\widetilde{\lambda}})_{\widetilde{\m}}.
	\]  
\end{thm}

\begin{proof} This is an immediate consequence of the main result 
in~\cite{caraiani-scholze-noncompact}. This states that 
	\[
	H^i(\widetilde{X}_{\widetilde{K}}, 
	\cV_{\widetilde{\lambda}}/\varpi)_{\widetilde{\m}} = 0\ 
	\mathrm{if}\ i<d,\ 
	\mathrm{and}\  
	H^i_c(\widetilde{X}_{\widetilde{K}}, 
	\cV_{\widetilde{\lambda}}/\varpi)_{\widetilde{\m}} = 0\ 
	\mathrm{if}\ i>d,
	\]
	under the assumptions on $\widetilde{\m}$ in the statement of the theorem. 
	By considering the short exact sequence of sheaves of $\cO$-modules on 
	$\widetilde{X}_{\widetilde{K}}$
	\[
	0\to \cV_{\widetilde{\lambda}} \to 
	\cV_{\widetilde{\lambda}}\to 
	\cV_{\widetilde{\lambda}}/\varpi \to 0
	\]
	and taking cohomology, we see that 
	$H^d(\widetilde{X}_{\widetilde{K}},\cV_{\widetilde{\lambda}})_{\widetilde{\m}}[\varpi]=0$, 
	since 
	$H^{d-1}(\widetilde{X}_{\widetilde{K}},\cV_{\widetilde{\lambda}}/\varpi)=0$.
	 By considering the excision sequence for
	\[
	\widetilde{X}_{\widetilde{K}}\hookrightarrow 
	\overline{\widetilde{X}}_{\widetilde{K}},
	\]
	we see that the cokernel of the map 
	$H^d(\widetilde{X}_{\widetilde{K}},\cV_{\widetilde{\lambda}})_{\widetilde{\m}}\to
	H^d(\partial\widetilde{X}_{\widetilde{K}},\cV_{\widetilde{\lambda}})_{\widetilde{\m}}$
	 injects into 
	$H_{c}^{d+1}(\widetilde{X}_{\widetilde{K}},\cV_{\widetilde{\lambda}})=0$.
\end{proof}

\begin{prop}\label{prop:coh_of_boundary} Suppose that $[F^+ : \Q] > 1$. 
	Let $\widetilde{K} \subset \widetilde{G}(\A_{F^+}^\infty)$ be a good subgroup which is decomposed with respect to $P$. Let $\widetilde{\lambda} \in (\Z^{2n}_+)^{\Hom(F^+, E)}$. Fix a decomposition $\overline{S}_p = \overline{S}_1 \sqcup \overline{S}_2$. Suppose that the following conditions are satisfied: 
	\begin{enumerate}
		\item For each $\overline{v} \in \overline{S}_2$, $\widetilde{\lambda}_{\overline{v}} = 0$.
		\item Let $v$ be a finite place of $F$ not contained in $S$, and let $l$ be its residue characteristic. Then either $S$ contains no $l$-adic places of $F$ and $l$ is unramified in $F$, or there exists an imaginary quadratic field $F_0 \subset F$ in which $l$ splits.
		\item $p > n^2$. (We remind the reader of our blanket assumption in \S \ref{section:lep} that $p$ is unramified in $F$.)
	\end{enumerate}
	Let $w \in W^P_{\overline{S}_2}$, and let $\lambda_w = w(\widetilde{\lambda} + \rho) - \rho \in (\Z^n_+)^{\Hom(F, E)}$. Let $\m \subset \T^S$ be a non-Eisenstein maximal ideal in the support of $H^\ast(X_K, \cV_{\lambda_w})$, and let $\widetilde{\m} = \cS^\ast(\m) \subset \widetilde{\T}^S$, and suppose that $\overline{\rho}_{\widetilde{\m}}$ is decomposed generic. 
	Then the map $\cS: \widetilde{\T}^S \to \T^S$ descends to a homomorphism
	\[ \widetilde{\T}^S(H^d(\widetilde{X}_{\widetilde{K}}, \cV_{\widetilde{\lambda}}))_{\widetilde{\m}} \to \T^S(H^{d - l(w)}(X_K, \cV_{\lambda_w}))_\m. \]
	Moreover, the map 
	\[ \widetilde{\T}^S(H^d(\widetilde{X}_{\widetilde{K}}, \cV_{\widetilde{\lambda}}))_{\widetilde{\m}} \to \widetilde{\T}^S(H^d(\widetilde{X}_{\widetilde{K}}, \cV_{\widetilde{\lambda}}))_{\widetilde{\m}}[1/p] \]
	is injective.
\end{prop}
\begin{proof}
	This results on combining Theorem \ref{what we get from CS} and Theorem \ref{direct summand at finite level}.
\end{proof}
We introduce some useful language.
\begin{defn}\label{defn:coh_trivial} A weight $\widetilde{\lambda} \in (\Z^{2n}_+)^{\Hom(F, E)}$ will be said to be CTG (``cohomologically trivial for $G$'') if it satisfies the following condition:
	\begin{itemize}
		\item Given $w \in W^P$, define $\lambda_w = w(\widetilde{\lambda} +\rho) - \rho \in (\Z^{n}_+)^{\Hom(F, E)}$. Then for all $w \in W^P$ and for all $i_0 \in \Z$, there exists $\tau \in \Hom(F, E)$ such that $\lambda_{w, \tau} - \lambda^\vee_{w, \tau c } \neq (i_0, i_0, \dots, i_0)$. 
	\end{itemize}
\end{defn}
This definition will be useful to us because Proposition \ref{prop:coh_of_boundary} shows how to relate a Hecke algebra for $G$ acting on cohomology with integral coefficients to a Hecke algebra for $\widetilde{G}$ acting on cohomology with rational coefficients of weight $\widetilde{\lambda}$ (say). If the weight $\widetilde{\lambda}$ is moreover CTG, then Theorem \ref{thm:automorphic reps contributing to char 0} (together with the purity lemma \cite[Lemma 4.9]{MR1044819}) shows that this rational cohomology can moreover be computed in terms of cuspidal automorphic forms for $\widetilde{G}$, which have associated Galois representations with well-understood local properties. 

Exploiting this is not straightforward since the weight for $G$ depends both on the chosen weight $\widetilde{\lambda}$ and the chosen Weyl group element $w$ (which must be  of a suitable length $l(w)$ in order to target a particular cohomological degree for $X_K$). This problem will be dealt with in the next section with a `degree shifting' argument. 

We first state a lemma which shows that there are ``many'' dominant weights for $\widetilde{G}$ which are CTG:
\begin{lemma}\label{central character} 
	Suppose that $[F^+ : \Q] > 1$. Let $\widetilde{\lambda} \in (\Z^{2n}_+)^{\Hom(F^+, E)}$, and fix a choice of embedding $\tau_0 : F^+ \hookrightarrow E$. Then there exists $\widetilde{\lambda}' \in (\Z^{2n}_+)^{\Hom(F^+, E)}$ satisfying the following conditions:
	\begin{enumerate}
		\item $\widetilde{\lambda}_\tau = \widetilde{\lambda}'_\tau$ for all $\tau \neq \tau_0$.
		\item $\widetilde{\lambda}'$ is CTG.
	\end{enumerate}
\end{lemma}

\begin{proof}
	Let $\tau \neq \tau_0$ be another embedding $\tau : F^+ \hookrightarrow E$. Note that a dominant weight $\widetilde{\mu} \in (\Z^{2n}_+)^{\Hom(F^+, E)}$ is CTG if it satisfies the following condition: for all $w \in W^P$, we have
    \numequation\label{eqn:criterion_for_CTG} \sum_{i=1}^n (\mu_{w, \widetilde\tau, i} - \mu_{w,  \widetilde\tau c, i}) \neq \sum_{i=1}^n (\mu_{w,  \widetilde\tau_0 , i} - \mu_{w,  \widetilde\tau_0 c, i}). 
	\end{equation}
	Let $a \in \Z_{\geq 0}$, and define $\widetilde{\lambda}' \in (\Z^{2n}_+)^{\Hom(F^+, E)}$ by the formula $\widetilde{\lambda}'_\tau = \widetilde{\lambda}_\tau$ if $\tau \neq \tau_0$, $\widetilde{\lambda}'_{\tau_0, 1} = \widetilde{\lambda}_{\tau_0, 1} + a$, $\widetilde{\lambda}'_{\tau, i} = \widetilde{\lambda}_{\tau_0, i}$ if $i > 1$. Then $\widetilde{\lambda}'$ will satisfy condition (\ref{eqn:criterion_for_CTG}) as soon as $a$ is sufficiently large (in a way depending on $\widetilde{\lambda}$). 
\end{proof}

\subsection{The degree shifting argument}\label{sec: FL 
degree-shifting} 

We are now going to show how to use Proposition \ref{prop:coh_of_boundary} to control the Hecke algebra of $G$ acting on the cohomology groups $H^q(X_K, \cV_\lambda)$. We will do this ``one place of $F^+$ above $p$ at a time''. The argument will involve induction on the cohomological degree $q$. Since the cohomology groups of locally symmetric spaces for $G$ may
contain torsion, one needs an inductive argument to pass from the cohomology  groups with $\cO$-coefficients (which appear in Proposition \ref{prop:coh_of_boundary}) to cohomology groups with $\cO/\varpi^m$-coefficients (where one can use congruences to 
modify the weight). 

The first step is the following proposition. Given a non-Eisenstein maximal ideal $\m \subset \T^S$, we will set $\widetilde{\m} = \cS^\ast(\m) \subset \widetilde{\T}^S$. We will use the notation
\[ A(K, \lambda, q) = \T^S(H^q(X_K, \cV_\lambda)_\m), \]
\[ A(K, \lambda, q, m) = \T^S(H^q(X_K, \cV_\lambda / \varpi^m)_\m), \]
and
\[ \widetilde{A}(\widetilde{K}, \widetilde{\lambda}) = \widetilde{\T}^S(H^d(\widetilde{X}_{\widetilde{K}}, \cV_{\widetilde{\lambda}})_{\widetilde{\m}}). \]
Note that there is no natural morphism $A(K, \lambda, q) \to A(K, \lambda, q, m)$.
\begin{prop}\label{degree-shifting} 
	Let $\overline{v}, \overline{v}'$ be distinct places of $\overline{S}_p$, and let $\lambda \in (\Z^n_+)^{\Hom(F, E)}$. 
	(The condition that~$\overline{S}_p$ has at least two
                distinct places implies, in particular, that~$F^+ \ne
                \Q$.)
	Fix an integer $m \geq 1$. Let $\widetilde{K} \subset \widetilde{G}(\A_{F^+}^\infty)$ be a good subgroup. Suppose that the following conditions are satisfied:
	\begin{enumerate}
		\item For each embedding $\tau : F \hookrightarrow E$ inducing the place $\overline{v}$ of $F^+$, we have $- \lambda_{\tau c, 1} - \lambda_{\tau, 1} \geq 0$. 
		\item We have 
		\[ \sum_{\substack{\overline{v}'' \in \overline{S}_p \\ \overline{v}'' \neq \overline{v}, \overline{v}'}} [F^+_{\overline{v}''} : \Q_p] > \frac{1}{2}[F^+ : \Q]. \]
		\item For each $p$-adic place $\overline{v}''$ of $F^+$ not equal to $\overline{v}$, we have 
		\[ U(\cO_{F^+_{\overline{v}''}}) \subset \widetilde{K}_{\overline{v}''} \subset \left\{ \left( \begin{array}{cc} 1_n & * \\ 0 & 1_n  \end{array}\right) \text{ mod }\varpi_{\overline{v}''}^{m}  \right\}. \]
		We have $\widetilde{K}_{\overline{v}} = \widetilde{G}(\cO_{F^+_{\overline{v}}})$.
		\item $p > n^2$. (We recall our blanket assumption in \S \ref{section:lep} that $p$ is unramified in $F$.)
		\item Let $v$ be a finite place of $F$ not contained in $S$, and let $l$ be its residue characteristic. Then either $S$ contains no $l$-adic places of $F$ and $l$ is unramified in $F$, or there exists an imaginary quadratic field $F_0 \subset F$ in which $l$ splits.
		\item $\m \subset \T^S$ is a non-Eisenstein maximal ideal such that $\overline{\rho}_{\widetilde{\m}}$ is decomposed generic. 
	\end{enumerate}
	Define a weight $\widetilde{\lambda} \in (\Z^{2n}_+)^{\Hom(F^+, E)}$ as follows: if $\tau \in \Hom(F^+, E)$ does not induce either $\overline{v}$ or $\overline{v}'$, then $\widetilde{\lambda}_\tau = 0$. If $\tau$ induces $\overline{v}$, then we set
	\[ \widetilde{\lambda}_{\tau} = (-\lambda_{\widetilde\tau c, n}, \dots, - \lambda_{\widetilde \tau c, 1}, \lambda_{\widetilde \tau, 1}, \dots, \lambda_{\widetilde \tau, n}). \]
	(Note that this is dominant because of our assumption on $\lambda$.) If $\tau$ induces $\overline{v}'$, then we choose $\widetilde{\lambda}_\tau$ to be an arbitrary element of $\Z^{2n}_+$.
	
	Let $q\in \left[\left \lfloor{\frac{d}{2}}\right \rfloor, 
	d-1\right]$. Then there exists an integer $m' \geq m$, an integer $N \geq 
	1$, a nilpotent ideal $J \subset A(K,\lambda,q, m)$ satisfying $J^N = 0$, 
	and a commutative diagram
	\[\xymatrix{
		\widetilde{\mathbb T}^S\ar[r]\ar[d]_-\cS & 
		\widetilde{A}(\widetilde{K}(m'), \widetilde{\lambda}) \ar[d]\\
		\mathbb T^S\ar[r] & A(K,\lambda,q, m)/J
	}\]
	where $\widetilde{K}(m') \subset \widetilde{K}$ is the good subgroup defined by setting
	\[ \widetilde{K}(m')_{\overline{v}''} = \widetilde{K}_{\overline{v}''} \cap \left\{ \left( \begin{array}{cc} 1_n & * \\ 0 & 1_n  \end{array}\right) \text{ mod }\varpi_{\overline{v}''}^{m'}  \right\} \subset \widetilde{G}(\cO_{F^+_{\overline{v}''}}) \]
	if $\overline{v}''$ is a $p$-adic place of $F^+$ not equal to $\overline{v}$, and $\widetilde{K}(m')_{\overline{v}''} = \widetilde{K}_{\overline{v}''}$ otherwise. (Thus $\widetilde{K} = \widetilde{K}(m)$, by hypothesis.) Moreover, the integer $N$ can be chosen to depend only on $n$ and $[F^+ : \Q]$.
\end{prop}

\begin{proof} The idea of the proof is to choose a 
	Weyl 
	group element $w=w(q)\in W^P$ such that $l(w)= d-q$ and a weight 
	$\widetilde{\lambda}$ such that $\lambda = 
	w(\widetilde{\lambda}+\rho)-\rho$, and then apply 
	Proposition \ref{prop:coh_of_boundary}. The actual 
	argument is more subtle, because we need to work with $\cO$-coefficients in 
	order to access the Hecke algebras
	$\widetilde{A}(\widetilde{K},\widetilde{\lambda})$, whilst the Hecke algebras 
	$A(K, \lambda, q, m)$ act on cohomology with torsion coefficients. We argue by descending induction on $q$, the induction hypothesis being as follows:
	\begin{hypothesis} Let $q\in \left[\left \lfloor{\frac{d}{2}}\right \rfloor, 
		d-1\right]$. Then the Proposition holds for every 
		cohomological degree 
		$i\in [q+1, d-1]$ and every $m\in \Z_{\geq 1}$. Moreover, the integer 
		$N$ can be chosen to depend only on $n$, $[F^+ : \Q]$, and $q$.
	\end{hypothesis}
	The induction hypothesis is always satisfied when $q=d-1$. Assume the induction 
	hypothesis holds for some $q\in \left[\left \lfloor{\frac{d}{2}}\right \rfloor 
	+1 
	, d-1\right]$. We will prove that the induction hypothesis holds for $q-1$. Let us fix $m$, $\widetilde{K}$, and $\lambda$ as in the statement of the proposition. Note that the $\T^S$-algebra $A(K, \lambda, q, m)$ is independent of $\lambda_{\overline{v}''}$ for $\overline{v}'' \in \overline{S}_p$, $\overline{v}'' \neq \overline{v}$, because $K_S$ acts on $\cV_\lambda / \varpi^m$ via the projection to $K_{\overline{v}}$. Modifying $\lambda$, we can therefore assume that in fact $\lambda_{\overline{v}'} = \widetilde{\lambda}_{\overline{v}'}$. 
	
	Let $\overline{S}_1 = \{ \overline{v}, \overline{v}' \}$, and $\overline{S}_2 = \overline{S}_p - \overline{S}_1$.
	Let $w = w(q) \in W^P_{\overline{S}_2}$ be any element of length $l(w) = d - q$. Such an element exists because for any $\tau \in \Hom(F^+, E)$, $l(w_\tau)$ takes all integer values in $[0, n^2]$ as $w_\tau$ ranges over elements of $W^P_\tau$. We have chosen our totally real field $F^+$ to satisfy
	\[
	\sum_{\bar{v}'' \in \overline{S}_2} [F^+_{\bar v''}: \Q_p] > \frac{1}{2}[F^+:\Q].
	\] 
	This means that the desired sum can take any value in $[0, \frac{d}{2}+ 
	\frac{n^2}{2}]$. On the other hand, $q \in \left[\left 
	\lfloor{\frac{d}{2}}\right \rfloor, d\right]$, so $d-q\leq d- \left 
	\lfloor{\frac{d}{2}}\right \rfloor$. Since $n\geq 1$, we can indeed make an 
	appropriate choice of $w$. 
	
	Now we let 
	$\lambda'(q) = w(q)(\widetilde{\lambda} +\rho)-\rho$.
	This can be different from $\lambda$ precisely at those embeddings inducing a place of $\overline{S}_2$. In particular, the Hecke algebras $A( K,\lambda'(q), q,m)$ and $A(K,\lambda,q, m)$ are 
	canonically isomorphic as $\T^S$-algebras, once again because $K_S$ acts on both $\cV_{\lambda'(q)}/\varpi^m$ and 
	$\cV_{\lambda}/\varpi^m$ via projection to $K_{\overline{v}}$.
	
	There is a short exact sequence of $\T^S$-modules
	\numequation\label{torsion short exact sequence} \begin{split}
	0\to H^q(X_{K},\cV_{\lambda'(q)})_{\m}/\varpi^m \to 
	H^q(X_{K},\cV_{\lambda'(q)}/\varpi^m)_{\m}  \\ 
	\to H^{q+1}(X_{K},\cV_{\lambda'(q)})_{\m}[\varpi^m]\to 0.
	\end{split}
\end{equation}
Note that the $\varpi^m$-torsion 
$H^{q+1}(X_{K},\cV_{\lambda'(q)})_{\m}[\varpi^m]$ does not, in general, inject 
into $H^{q+1}(X_{K},\cV_{\lambda'(q)}/\varpi^m)_{\m}$, so we cannot reduce to 
understanding the Hecke algebra $A(q+1, K,\lambda'(q),m)$. However, the 
cohomology group $H^{q+1}(X_{K},\cV_{\lambda'(q)})_{\m}$ is a finitely 
generated $\cO$-module, so $H^{q+1}(X_{K},\cV_{\lambda'(q)})_{\m}[\varpi^m]$ 
does inject into $H^{q+1}(X_{K},\cV_{\lambda'(q)})_{\m}/\varpi^{m'}$ provided that $m' \geq m$ is chosen large enough for $\varpi^{m'}$ to annihilate the torsion submodule of $H^{q+1}(X_{K},\cV_{\lambda'(q)})_{\m}$. This, in turn, injects into 
$H^{q+1}(X_{K},\cV_{\lambda'(q)}/\varpi^{m'})_{\m}$. It follows that we have an inclusion
\numequation\label{eqn:first_annihilator_comparison} \begin{split} \operatorname{Ann}_{\T^S} H^q(X_{K},\cV_{\lambda'(q)})_{\m} \cdot & \operatorname{Ann}_{\T^S}H^{q+1}(X_K, \cV_{\lambda'(q)} / \varpi^{m'})_{\m} \\ & \subset \operatorname{Ann}_{\T^S}H^q(X_K, \cV_{\lambda} / \varpi^m)_{\m}. 
\end{split}
\end{equation}
Let $K(m') = \widetilde{K}(m') \cap G(\A_{F^+}^\infty)$. Let $\m^\vee = \iota(\m) \subset \T^S$ (notation as in \S \ref{sec:twisting_and_duality}). Then $\m^\vee$ is a non-Eisenstein maximal ideal. Poincar\'e duality implies (cf.\ Corollary \ref{cor:dual_Hecke_alg} and \cite[Thm. 4.2]{new-tho}, and noting that $\cO / \varpi^m$ is an injective $\cO / \varpi^m$-module) that there is an equality
\[ \operatorname{Ann}_{\T^S}H^i(X_K, \cV_{\lambda'(q)} / \varpi^m)_\m = \iota(\operatorname{Ann}_{\T^S}H^{d-1-i}(X_K, \cV^\vee_{\lambda'(q)} / \varpi^m)_{\m^\vee}) \]
of ideals of $\T^S$. The existence of the Hochschild--Serre spectral sequence 
\[
H^{i}\left(K / K(m'), 
H^{j}(X_{K(m')},\cV_{\lambda'(q)}^\vee/\varpi^{m'})_{\m^\vee} 
\right) \Rightarrow 
H^{i+j}(X_{K},\cV_{\lambda'(q)}^\vee/\varpi^{m'})_{\m^\vee}
\]
implies that there is an inclusion
\[ \prod_{i=0}^{d-q-2} \operatorname{Ann}_{\T^S} H^i(X_{K(m')}, \cV_{\lambda'(q)}^\vee /\varpi^{m'})_{\m^\vee} \subset \operatorname{Ann}_{\T^S} H^{d - 2 - q}(X_K, \cV_{\lambda'(q)}^\vee / \varpi^{m'})_{\m^\vee}. \]
Applying Corollary \ref{cor:dual_Hecke_alg} once more, we see that there is an inclusion
\[ \prod_{i=q+1}^{d-1} \operatorname{Ann}_{\T^S} H^i(X_{K(m')}, \cV_{\lambda'(q)}/\varpi^{m'})_{\m} \subset \operatorname{Ann}_{\T^S} H^{q+1}(X_K, \cV_{\lambda'(q)} / \varpi^{m'})_{\m},  \]
or equivalently
\[ \prod_{i=q+1}^{d-1} \operatorname{Ann}_{\T^S} H^i(X_{K(m')}, \cV_{\lambda}/\varpi^{m'})_{\m} \subset \operatorname{Ann}_{\T^S} H^{q+1}(X_K, \cV_{\lambda'(q)} / \varpi^{m'})_{\m},  \]
Combining this with (\ref{eqn:first_annihilator_comparison}), we deduce that there is an inclusion
\numequation\label{eqn:inclusion_of_annihilators} \begin{split} \operatorname{Ann}_{\T^S}H^{q}(X_K, \cV_{\lambda'(q)})_{\m} & \cdot \prod_{i=q+1}^{d-1} \operatorname{Ann}_{\T^S} H^i(X_{K(m')}, \cV_{\lambda}/\varpi^{m'})_{\m} \\ & \subset \operatorname{Ann}_{\T^S}H^q(X_K, \cV_{\lambda} / \varpi^m)_{\m}.\end{split} 
\end{equation}
By induction, we can find an integer $N \geq 1$ and for each $i = q+1, \dots, d - 1$ an integer $m'_i \geq m'$ such that
\[ \cS\left( \operatorname{Ann}_{\widetilde{\T}^S} H^d(\widetilde{X}_{\widetilde{K}(m'_i)}, \cV_{\widetilde{\lambda}})_{\widetilde{\m}} \right)^N \subset \operatorname{Ann}_{\T^S} H^i(X_{K(m')}, \cV_{\lambda}/\varpi^{m'})_{\m}. \]
Moreover, Proposition \ref{prop:coh_of_boundary} implies that there is an inclusion
\[ \cS\left( \operatorname{Ann}_{\widetilde{\T}^S} H^d(\widetilde{X}_{\widetilde{K}}, \cV_{\widetilde{\lambda}})_{\widetilde{\m}}\right) \subset \operatorname{Ann}_{\T^S} H^q(X_K, \cV_{\lambda'(q)})_{\m}. \]
Let $m'' = \sup_i m'_i$, and note that for each $i$ we have
\[ \operatorname{Ann}_{\widetilde{\T}^S} H^d(\widetilde{X}_{\widetilde{K}(m'')}, \cV_{\widetilde{\lambda}})_{\widetilde{\m}} \subset \operatorname{Ann}_{\widetilde{\T}^S} H^d(\widetilde{X}_{\widetilde{K}(m'_i)}, \cV_{\widetilde{\lambda}})_{\widetilde{\m}} \]
(because this is true rationally, and the cohomology groups are torsion-free, by Theorem \ref{what we get from CS}). 
Finally, let $N' = 1 + (d-q-1)N$, and let $J$ denote the image of the ideal 
\[ \cS\left( \operatorname{Ann}_{\widetilde{\T}^S} H^d(\widetilde{X}_{\widetilde{K}(m'')}, \cV_{\widetilde{\lambda}})_{\widetilde{\m}} \right) \]
in $A(K, q, \lambda, m)$. The existence of the inclusion (\ref{eqn:inclusion_of_annihilators}) implies that $\cS$ descends to a morphism
\[ \widetilde{A}(\widetilde{K}(m''), \widetilde{\lambda}) \to A(K, \lambda, q, m) / J, \]
and that the ideal $J$ satisfies $J^{N'} = 0$. This completes the proof. 
\end{proof}
This proposition has the following consequence for Galois representations.
\begin{prop}\label{prop:first_consequence_for_FL_property}
	Let $\overline{v}, \overline{v}'$ be distinct places of $\overline{S}_p$, and let $\lambda \in (\Z^n_+)^{\Hom(F, E)}$. Fix an integer $m \geq 1$. Let $\widetilde{K} \subset \widetilde{G}(\A_{F^+}^\infty)$ be a good subgroup. Suppose that the following conditions are satisfied:
	\begin{enumerate}
		\item For each embedding $\tau : F \hookrightarrow E$ inducing the place $\overline{v}$, we have $-\lambda_{\tau c, 1} - \lambda_{\tau , 1} \geq 0$ and $- \lambda_{\tau c, n} - \lambda_{\tau, n} \leq p - 2n - 1$.
		\item We have 
		\[ \sum_{\substack{\overline{v}'' \in \overline{S}_p \\ \overline{v}'' \neq \overline{v}, \overline{v}'}} [F^+_{\overline{v}''} : \Q_p] > \frac{1}{2}[F^+ : \Q]. \]
		\item For each $p$-adic place $\overline{v}''$ of $F^+$ not equal to $\overline{v}$, we have 
		\[ U(\cO_{F^+_{\overline{v}''}}) \subset \widetilde{K}_{\overline{v}''} \subset \left\{ \left( \begin{array}{cc} 1_n & * \\ 0 & 1_n  \end{array}\right) \text{ mod }\varpi_{\overline{v}''}^{m}  \right\}. \]
		We have $\widetilde{K}_{\overline{v}} = \widetilde{G}(\cO_{F^+_{\overline{v}}})$.
		\item $p > n^2$. (We recall our blanket assumption in \S \ref{section:lep} that $p$ is unramified in $F$.)
	\item Let $v$ be a finite place of $F$ not contained in $S$, and let $l$ be its residue characteristic. Then either $S$ contains no $l$-adic places of $F$ and $l$ is unramified in $F$, or there exists an imaginary quadratic field $F_0 \subset F$ in which $l$ splits.
	\item $\m \subset \T^S$ is a non-Eisenstein maximal ideal such that $\overline{\rho}_{\widetilde{\m}}$ is decomposed generic. 
	\end{enumerate}
	Let $q\in \left[\left \lfloor{\frac{d}{2}}\right \rfloor, 
	d-1\right]$. Then there exists an integer $N \geq 1$ depending only on $[F : \Q]$ and $n$, an ideal $J \subset A(K, \lambda, q, m)$ satisfying $J^N = 0$, and a continuous representation
	\[ \rho_\m : G_{F, S} \to \GL_{n}(A(K, \lambda, q, m) / J) \]
	satisfying the following conditions:
	\begin{enumerate}
		\item[(a)] For each place $v \not\in S$ of $F$, the characteristic polynomial of $\rho_\m(\Frob_v)$ is equal to the image of $P_v(X)$ in $(A(K, \lambda, q, m) / J)[X]$.
		\item[(b)] For each place $v | \overline{v}$ of $F$, $\rho_\m|_{G_{F_{v}}}$ is in the essential image of the functor $\mathbf{G}^a$, for $a = ( \lambda_{\tau, n}) \in \Z^{\Hom_{\Q_p}(F_{v}, E)}$.
		\item[(c)] For each place $v | \overline{v}$ of $F$, there exists $\overline{N} \in \cM \cF_k$ with $\overline{\rho}_{\widetilde{\m}}|_{G_{F_{v}}} \cong \mathbf{G}(\overline{N})$ and
		\[ \mathrm{FL}_\tau(\overline{N}) = \{  -\lambda_{\tau c, n}  + 2 n - 1, \dots, -\lambda_{\tau c, 1} + n, \lambda_{\tau, 1} +(n-1), \dots, \lambda_{\tau, n} \}. \]
		for each embedding $\tau \in \Hom_{\Q_p}(F_{v}, E)$.
	\end{enumerate}	
\end{prop}
\begin{proof}
	Our hypotheses include those of Proposition \ref{degree-shifting}. We choose the weight $\widetilde{\lambda}$ of Proposition \ref{degree-shifting} to be CTG (as we may, using Lemma \ref{central character} and our freedom to specify $\widetilde{\lambda}_{\overline{v}'}$). Let $N_0$ be the integer denoted by $N$ in the statement of that proposition. Thus we can find an integer $m' \geq m$, a nilpotent ideal $J_0 \subset A(q, K, \lambda, m)$ satisfying $J_0^{N_0} = 0$, and a commutative diagram
	\[\xymatrix{
		\widetilde{\mathbb T}^S\ar[r]\ar[d]_-\cS & 
		\widetilde{A}(\widetilde{K}(m'), \widetilde{\lambda}) \ar[d]\\
		\mathbb T^S\ar[r] & A(K,\lambda,q,m)/J_0.
	}\]
	Let us abbreviate $\widetilde{A} = \widetilde{A}(\widetilde{K}(m'), \widetilde{\lambda})$ and $A = A(K,\lambda,q,m)$. By Theorem \ref{what we get from CS}, $\widetilde{A}$ is $\cO$-flat, and by Theorem \ref{thm:automorphic reps contributing to char 0}, $\widetilde{A} \otimes_\cO \overline{\Q}_p$ is semisimple and can be computed in terms of cuspidal automorphic representations of $\widetilde{G}$. By Theorem \ref{thm:base_change_and_existence_of_Galois_for_tilde_G}, there exists a continuous homomorphism
	\[ \widetilde{\rho} : G_{F, S} \to \GL_{2n}( \widetilde{A} \otimes_\cO \overline{\Q}_p) \]
	such that for any homomorphism $f :  \widetilde{A} \otimes_\cO \overline{\Q}_p \to \overline{\Q}_p$, and for any finite place $v\not\in S$ of $F$, $f \circ \widetilde{\rho}(\Frob_v)$ has characteristic polynomial equal to the image of $\widetilde{P}_v(X)$ in $\overline{\Q}_p[X]$; and for any place $v | \overline{v}$ of $F$, $(f \circ \widetilde{\rho})|_{G_{F_{v}}}$ is crystalline of Hodge--Tate weights 
	\[ \mathrm{HT}_\tau(f \circ \widetilde{\rho}|_{G_{F_{v}}}) = \{  -\lambda_{\tau c, n}  + (2 n - 1), \dots, -\lambda_{\tau c, 1} + n, \lambda_{\tau, 1} +(n-1), \dots, \lambda_{\tau, n} \}. \]
	In particular, any $G_{F_{v}}$-invariant $\cO$-lattice in $\widetilde{A}^{2n}$ is crystalline with all $\tau$-Hodge--Tate weights in the interval $[ \lambda_{\tau, n}, (2n -1) - \lambda_{\tau c, n}]$. Using our hypothesis that $- \lambda_{\tau c, n} + (2n - 1) - \lambda_{\tau, n} \leq p - 2$, we see that any $G_{F_{v}}$-invariant $\cO$-lattice in $\widetilde{A}^{2n}$ is in the image of the functor $\mathbf{G}^a$ with $a = ( \lambda_{\tau, n} ) \in \Z^{\Hom_{\Q_p}(F_{v}, E)}$ (cf.\ the discussion of the functor $\mathbf{G}^a$ at the beginning of \S \ref{section:lep}).
	
	This establishes part (c) of the proposition. Since for each $\tau \in \Hom_{\Q_p}(F_{v}, E)$ the integers 
	\[
 -\lambda_{\tau c, n}  + (2 n - 1), \dots, -\lambda_{\tau c, 1} + n, \lambda_{\tau, 1} +(n-1), \dots, \lambda_{\tau, n} 
	\]
	are all distinct, and $\overline{\rho}_{\widetilde{\m}} \cong \overline{\rho}_\m \oplus (\overline{\rho}_\m^{c, \vee} \otimes \epsilon^{1-2n})$, it follows as well that $\overline{\rho}_\m|_{G_{F_v}} \not\cong (\overline{\rho}_\m^{c, \vee} \otimes \epsilon^{1-2n})|_{G_{F_{v}}}$.
	
	Let $\widetilde{D} = \det \widetilde{\rho}$, a continuous determinant of $G_{F, S}$ of dimension $2n$ valued in $\widetilde{A}$ (by \cite[Ex. 2.32]{chenevier_det}). Its kernel is a 2-sided ideal of $\widetilde{A}[G_{F, S}]$ (see \cite[\S 1.17]{chenevier_det} for the definition of the kernel of a determinant). The formation of kernels commutes with flat base change over $\widetilde{A}$, so there is an algebra embedding
	\[
	(\widetilde{A}[G_{F, S}]/\ker(\widetilde{D}))\otimes_\cO \Qbar_p = 
	(\widetilde{A}\otimes_\cO \Qbar_p)[G_{F, S}]/\ker(\widetilde{D}\otimes_\cO \Qbar_p)
	\subset M_{2n}(\widetilde{A}\otimes_\cO \Qbar_p), 
	\]
	by \cite[Thm. 2.12]{chenevier_det}. This is in particular an embedding of left $\widetilde{A}[G_{F, S}]$-modules. 
	It follows that $(\widetilde{A}[G_{F, S}]/\ker(\widetilde{D})) \otimes_\cO \Qbar_p$ is a subrepresentation of $\widetilde{\rho}^{2n}$, 
	hence that for each $v | \overline{v}$, the $G_{F_{v}}$-representation $\widetilde{A}[G_{F, S}]/\ker(\widetilde{D})$ is in the essential image of $\bG^a$.
	
	Theorem \ref{thm:existence_of_Hecke_repn_for_GL_n} implies that there there is an integer $N_1$ depending only on $[F : \Q]$ and $n$, a nilpotent ideal $J_1 \subset A(K, \lambda,q,  m)$ satisfying $J_1^{N_1} = 0$, and a continuous representation
	\[ \rho_\m : G_{F, S} \to \GL_n( A(K, \lambda, q, m) / J_1) \]
	such that for each finite place $v \not\in S$ of $F$, $\rho_\m(\Frob_v)$ has characteristic polynomial equal to the image of $P_v(X)$ in $( A(K, \lambda, q, m) / J_1) [X]$. Let $J = (J_0, J_1) \subset A(K, \lambda, q, m)$; then $J^N = 0$, where $N = N_0 + N_1$. We will show that the proposition holds with this choice of $J$ and this value of $N$. Let us now write $\rho_\m$ for the projection of $\rho_\m$ to a representation with coefficients in $A(K, \lambda, q, m) / J = A/J$.
	
	Set $\widetilde{D}_{A/J} = \widetilde{D} \otimes_{\widetilde{A}} A/J$. Then $\widetilde{D}_{A / J} = \det(\rho_\m \oplus \rho_\m^{c, \vee} \otimes{\epsilon^{1-2n}})$, hence 
	\[ (\ker \det \rho_\m) \cap (\ker \det \rho_\m^{c, \vee} \otimes \epsilon^{1-2n}) \subset \ker \widetilde{D}_{A/J}. \]
	The representation $\rho_\m \oplus (\rho_\m^{c, \vee} \otimes {\epsilon^{1-2n}})$ induces an~$A$-algebra homomorphism  
	\[ (A / J)[G_{F, S}] \to M_n(A/J) \oplus M_n(A/J) \]
	 which, by \cite[Thm.~2.22(i)]{chenevier_det}, is surjective with kernel equal to $(\ker \det \rho_\m) \cap (\ker \det \rho_\m^{c, \vee} \otimes \epsilon^{1-2n})$. We deduce that $(A/J)[G_{F, S}]/\ker(\widetilde{D}_{A/J})$ is a quotient $A/J$-algebra of $M_n(A/J) \times M_n(A/J)$. By \cite[Thm.~2.22(ii)]{chenevier_det}, this forces $(A/J)[G_{F, S}]/\ker(\widetilde{D}_{A/J}) = M_n(A/J) \times M_n(A/J)$. 
	
	The surjection $\widetilde{A}[G_{F, S}] \to (A/J)[G_{F, S}]/\ker(\widetilde{D}_{A/J})$ 
	factors through the quotient $\widetilde{A}[G_{F, S}]/\ker(\widetilde{D})$ (see \cite[Lem.~1.18]{chenevier_det}). It follows that for each place $v | \overline{v}$ of $F$ that $M_n(A/J) \times M_n(A/J)$, viewed as a left $(A/J)[G_{F_{v}}]$-module, is in the essential image of the functor $\bG^a$ (the essential image is stable under passage to subquotients). Since $M_n(A/J) \times M_n(A/J)$ contains $\rho_\m$ as a subobject, it follows that $\rho_\m|_{G_{F_{v}}}$ is in the essential image of $\bG^a$, as desired.
\end{proof}
	
	\begin{remark}\label{rmk:WWE}
		Ideas similar to, and more general than, those used in the proof above were developed by Wake--Wang-Erickson \cite{WWE_pseudo_def_cond}.
	\end{remark}
We now extend the range of cohomological degrees and allowable level subgroups to which Proposition \ref{prop:first_consequence_for_FL_property} applies.
\begin{cor}\label{cor:first_consequence_for_FL_property}
	Let $\overline{v} \in S_p$, and let $K \subset \GL_n(\A_F^\infty)$ be a good subgroup. Let $\lambda \in (\Z^n_+)^{\Hom(F, E)}$, and let $\m \subset \T^S(K, \lambda)$ be a non-Eisenstein maximal ideal. Suppose that the following conditions are satisfied:
	\begin{enumerate}
		\item For each place $v | \overline{v}$ of $F$, we have $K_v = \GL_n(\cO_{F_v})$.
		\item There exists a place $\overline{v}' \in \overline{S}_p$ such that $\overline{v}' \neq \overline{v}$ and
		\[ \sum_{\substack{\overline{v}'' \in \overline{S}_p \\ \overline{v}'' \neq \overline{v}, \overline{v}'}} [F^+_{\overline{v}''} : \Q_p] > \frac{1}{2}[F^+ : \Q]. \]
		\item\label{p_bounded} For each embedding $\tau : F \hookrightarrow E$ inducing the place $\overline{v}$ of $F^+$, we have $-\lambda_{\tau c, 1} - \lambda_{\tau, 1} \geq 0$ and $ - \lambda_{\tau c, n} - \lambda_{\tau, n} \leq p - 1 - 2n$.
		\item $p > n^2$. (We recall our blanket assumption in \S \ref{section:lep} that $p$ is unramified in $F$.)
		\item \label{part:runningassumptioncor} Let $v$ be a finite place of $F$ not contained in $S$, and let $l$ be its residue characteristic. Then either $S$ contains no $l$-adic places of $F$ and $l$ is unramified in $F$, or there exists an imaginary quadratic field $F_0 \subset F$ in which $l$ splits.
		\item   \label{part:weakdecomposedgeneric} $\overline{\rho}_{\m}$ is decomposed generic. 
	\end{enumerate}
	Let $q \in [0, d-1]$ and $m \geq 1$ be integers.  Then there exists an integer $N \geq 1$ depending only on $[F : \Q]$ and $n$, an ideal $J \subset A(K, \lambda, q, m)$ satisfying $J^N = 0$, and a continuous representation
	\[ \rho_\m : G_{F, S} \to \GL_{n}(A(K, \lambda, q, m) / J) \]
	satisfying the following conditions:
	\begin{enumerate}
		\item[(a)] For each place $v \not\in S$ of $F$, the characteristic polynomial of $\rho_\m(\Frob_v)$ is equal to the image of $P_v(X)$ in $(A(K, \lambda, q, m) / J)[X]$.
	\item[(b)] For each place $v | \overline{v}$ of $F$, $\rho_\m|_{G_{F_{v}}}$ is in the essential image of the functor $\mathbf{G}^a$, for $a = ( \lambda_{\tau, n}) \in \Z^{\Hom_{\Q_p}(F_{v}, E)}$.
	\item[(c)] For each place $v | \overline{v}$ of $F$, there exists $\overline{N} \in \cM \cF_k$ with $\overline{\rho}_{\widetilde{\m}}|_{G_{F_{v}}} \cong \mathbf{G}(\overline{N})$ and
	\[ \mathrm{FL}_\tau(\overline{N}) = \{  -\lambda_{\tau c, n}  + (2 n - 1), \dots, -\lambda_{\tau c, 1} + n, \lambda_{\tau, 1} +(n-1), \dots, \lambda_{\tau, n} \}. \]
	for each embedding $\tau \in \Hom_{\Q_p}(F_{v}, E)$.
	\end{enumerate}	
\end{cor}
\begin{proof}
	Note that the existence of a $\rho_\m$ satisfying only condition (a) (local-global compatibility at unramified places) is already known (Theorem \ref{thm:existence_of_residual_representation_for_GL_n}). We are therefore free to enlarge $S$ if necessary. We first prove the corollary with hypothesis (\ref{part:weakdecomposedgeneric}) replaced by the stronger assumption that $\overline{\rho}_{\widetilde{\m}}$ is decomposed generic. Let $K' \subset K$ be the good normal subgroup defined by the formula $K'_v = K_v$ if $v \nmid p$ or $v | \overline{v}$, and $K'_v = K_v \cap \ker( \GL_n(\cO_{F_v}) \to \GL_n(\cO_{F_v} / \varpi_v^m) )$ otherwise. Let $\widetilde{K} \subset \widetilde{G}(\A_{F^+}^\infty)$ be a good subgroup satisfying the following conditions:
	\begin{itemize}
		\item $\widetilde{K} \cap G(\A_{F^+}^\infty) = K$.
		\item $\widetilde{K}^S = \widetilde{G}(\widehat{\cO}_{F^+}^S)$.
		\item For each place $\overline{v}'' | p$ of $F^+$, $U(\cO_K) \subset 
		\widetilde{K}_{\overline{v}''}$.
		\item $\widetilde{K}_{\overline{v}} = G(\cO_{F^+_{\overline{v}}})$.
	\end{itemize}
	Let $\widetilde{K}' = \widetilde{K}(m)$ be the good subgroup defined as follows: if $\overline{v}''$ is a finite place of $F^+$ which is prime to $p$ or equal to $\overline{v}$, then $\widetilde{K}'_{\overline{v}''} = \widetilde{K}_{\overline{v}''}$. Otherwise, we set
	\[ \widetilde{K}'_{\overline{v}''} = \widetilde{K}_{\overline{v}''}\cap\left\{ \left( \begin{array}{cc} 1_n & * \\ 0 & 1_n  \end{array}\right) \text{ mod }\varpi_{\overline{v}''}^{m}  \right\}. \]
	Note that the triple $(\widetilde{K}', \lambda, m)$ satisfies the hypotheses of Proposition \ref{prop:first_consequence_for_FL_property}. We let $K' = \widetilde{K}' \cap G(\A_{F^+}^\infty) $. There is a Hochschild--Serre spectral sequence
	\[ H^i(K / K', H^j( X_{K'}, \cV_\lambda / \varpi^m)_\m ) \Rightarrow H^{i+j}(X_K, \cV_\lambda / \varpi^m)_\m. \]
	It follows that we have an inclusion
	\[ \prod_{i=0}^q \operatorname{Ann}_{\T^S} H^{q-i}(X_{K'}, \cV_\lambda / \varpi^m)_\m \subset \operatorname{Ann}_{\T^S} H^q(X_K, \cV_\lambda / \varpi^m)_\m. \]
	Suppose we could show that there is an integer $N_0$ depending only on $[F^+ : \Q]$ and $n$ and for each $i = 0, \dots, q$ an ideal $J_i \subset A(K', \lambda, q-i, m)$ satisfying $J_i^{N_0} = 0$ and a continuous representation $\rho_{\m, i} : G_{F, S} \to \GL_n(A(K', \lambda, q-i, m) / J_i)$ satisfying the conditions the same conditions as $\rho_\m$. Then the corollary would follow, with $J$ equal to the image in $A(K, \lambda, q, m)$ of the intersection of the pre-images of $J_0, \dots, J_q$ in $\T^S$, and $N = q N_0$. A theorem of Carayol, \cite[Th\'eor\`eme~2]{MR1279611}, implies that the product representation
	\[ \prod_{i=0}^q \rho_{\m, i} : G_{F, S} \to \GL_n\left( \prod_{i=0}^q A(K', \lambda, q-i, m) / J_i \right) \]
	can be conjugated to take values in $\GL_n(\im(\T^S \to \prod_{i=0}^q A(K', 
	\lambda, q-i, m) / J_i ) )$, and the ring $\im(\T^S \to \prod_{i=0}^q A(K', 
	\lambda, q-i, m) / J_i )$ has $A(K, \lambda, q, m) / J$ as a quotient. 
	
	We are therefore free to assume that $K = K'$ and $\widetilde{K} = 
	\widetilde{K}'$, which we now do. In this case, we can moreover assume that 
	$\lambda_{\overline{v}''} = 0$ if $\overline{v}'' \in \overline{S}_p$ and 
	$\overline{v}'' \neq \overline{v}$. Note that $\widetilde{K}$ satisfies the 
	conditions of Proposition \ref{prop:first_consequence_for_FL_property}, so 
	if $q - i \geq \lfloor d / 2 \rfloor$, there's nothing to do. Suppose 
	instead that $q - i < \lfloor d /2 \rfloor$. Then $d - 1 - q + i \ge 
	\lfloor d / 2 \rfloor$.
	
	Our condition on $\lambda_{\overline{v}}$ then implies, together with \cite[Cor.~II.5.6]{MR2015057}, that there is an isomorphism $\cV_{\lambda^\vee} \cong \cV^\vee_{\lambda}$. Let $n_0 = (2n + 1 - p) / 2$, and let $\mu_0 \in (\Z^n_+)^{\Hom(F, E)}$ be defined by $\mu_{0, \tau} = (n_0,\dots, n_0)$ for each $\tau$. Then the maximal ideal $\m^\vee(\epsilon^{-n_0})$ of $\T^S$  (cf.\ \S \ref{sec:twisting_and_duality}) is in the support of $H^\ast(X_K, \cV_{\lambda^\vee + \mu_0})$, and the weight $\lambda^\vee + \mu_0$ also satisfies the hypothesis (\ref{p_bounded}) of the corollary. 
	
	Proposition \ref{prop:first_consequence_for_FL_property} implies the existence of an ideal 
	\[ J_i' \subset \T^S( H^{d-1-q+i}(X_K, \cV_{\lambda^\vee + \mu_0} / \varpi^m))_{\m^\vee(\epsilon^{-n_0})} \]
	 and a continuous representation 
	 \[\rho'_{\m, i} : G_{F, S} \to \GL_n(\T^S( H^{d-1-q+i}(X_K, \cV_{\lambda^\vee + \mu_0} / \varpi^m))_{\m^\vee(\epsilon^{-n_0})} / J_i') \]
	 satisfying the same conditions as $\rho_\m$. Proposition \ref{prop_poincare_duality} and Proposition \ref{prop:twisting_by_character} together imply that the isomorphism 
	 \[ \T^S \to \T^S,\,\,[K^S g K^S] \mapsto \epsilon(\Art_K(\det(g)))^{-n_0} [K^S g^{-1} K^S]\]
	  descends to an isomorphism
	\[ f :   \T^S( H^{d-1-q+i}(X_K, \cV_{\lambda^\vee + \mu_0} / \varpi^m))_{\m^\vee(\epsilon^{-n_0})} \to A(K, \lambda, q - i, m). \]
	The proof in this case is completed by taking $J_i = f(J_i')$ and $\rho_{\m, i} = (f \circ \rho_{\m, i}')^\vee \otimes \epsilon^{1 - 2n + (p-1)/2}$. 
	
	We now remove the assumption that $\overline{\rho}_{\widetilde{\m}}$ is decomposed generic, assuming instead only that $\overline{\rho}_\m$ is decomposed generic. After possibly enlarging $k$, we can find a character $\overline{\psi} : G_F \to k^\times$ such that 
	\[ (\overline{\rho}_{\m} \otimes \overline{\psi}) \oplus ((\overline{\rho}_{\m} \otimes \overline{\psi})^{c, \vee} \otimes \epsilon^{1-2n}) \]
	 is decomposed generic, and $\overline{\psi}|_{G_{F_{v}}}$ is trivial for each place $v \in S$ of $F$. Let $\psi : G_{F} \to \cO^\times$ denote the Teichm\"uller lift of $\overline{\psi}$. 
	
	Choose a finite set $S'$ containing $S$ and the set of places where $\psi$ is ramified and a good normal subgroup $K' \subset K$, all satisfying the following conditions:
	\begin{itemize}
		\item $(K')^{S' - S} = K^{S' - S}$.
		\item The quotient $K' / K$ is abelian of order prime to $p$. 
		\item For each place $v$ of $F$, the restriction of $\psi|_{G_{F_v}} \circ \Art_{F_v}$ to $\det(K'_v)$ is trivial.  
		\item $S'$ satisfies the analogue of hypothesis~(\ref{part:runningassumptioncor}) of the corollary.
	\end{itemize}
	Then there is a surjection $A(K', \lambda, q, m) \to A(K, \lambda, q, m)$ of $\T^{S'}$-algebras, so it suffices to establish the corollary for $A(K', \lambda, q, m)$. We write $\m(\psi) \subset \T^{S'}$ for the non-Eisenstein maximal ideal with $\overline{\rho}_{\m(\psi)} \cong \overline{\rho}_\m \otimes \overline{\psi}$.
	
	Let $\widetilde{\m}(\psi) = \cS^\ast(\m(\psi))$. Then $\overline{\rho}_{\widetilde{\m}(\psi)}$ is decomposed generic, so the already established case of the corollary implies that we can find an integer $N \geq 1$ depending only on $[F^+ : \Q]$ and $n$, and an ideal $J' \subset \T^{S'}(H^q(X_{K'}, \cV_\lambda / \varpi^m)_{\m(\psi)})$ satisfying ${J'}^n = 0$, and a continuous representation
	\[ \rho_{\m(\psi)} : G_{F, S'} \to \GL_n( \T^{S'}(H^q(X_{K'}, \cV_\lambda / \varpi^m)_{\m(\psi)}) / J' ) \]
	satisfying the conditions (a) -- (c) of the corollary. Proposition \ref{prop:twisting_by_character} implies that the
	isomorphism 
	\[ \T^{S'} \to \T^{S'},\,\,[{K'}^S g {K'}^S] \mapsto \psi(\Art_F(\det(g))) [{K'}^S g {K'}^S]\]
	descends to an isomorphism
	\[ f :   \T^{S'}( H^{q}(X_{K'}, \cV_{\lambda} / \varpi^m))_{\m(\psi)} \to A(K', \lambda, q, m). \]
	The proof is completed on taking $J = f(J')$ and $\rho_\m = (f \circ \rho_{\m(\psi)}) \otimes \psi^{-1}$.
\end{proof}

\subsection{The end of the proof}
We can now prove the main theorem of this chapter.
(For the reader's convenience, we repeat the statement here.)
To avoid confusion, we also restate the standing hypotheses for this chapter in the statement of the theorem.
\begin{theorem}\label{myAmazingTheorem} \label{thm:lgcfl}
Let $\m \subset \T^S(K, \lambda)$ be a non-Eisenstein maximal ideal. Suppose that~$\T^S(K, \lambda) / \m = k$ has residue
	characteristic~$p$. Let $\overline{v}$ be a $p$-adic place of $F^{+}$, and suppose that the following additional conditions are satisfied:
\begin{enumerate}
\item \label{part:unramifiedandsplit} The prime $p$ is unramified in $F$, and $F$ contains an imaginary quadratic field in which $p$ splits.
		\item \label{part:runningassumption}  Let $w$ be a finite place of $F$ not contained in $S$, and let $l$ be its residue characteristic. Then either $S$ contains no $l$-adic places of $F$ and $l$ is unramified in $F$, or there exists an imaginary quadratic field $F_0 \subset F$ in which $l$ splits.
	\item \label{part:part3} For each place $v | \overline{v}$ of $F$, $K_v = \GL_n(\cO_{F_v})$.
	\item  \label{part:part4}  For every embedding $\tau : F 
	\hookrightarrow E$ inducing the place $\overline{v}$ of $F^+$,
	\[
	\lambda_{\tau,1} + \lambda_{\tau c, 1} - \lambda_{\tau, n} - \lambda_{\tau c, n} \leq p - 2n - 1.
	\]
	\item  \label{part:pbig} $p > n^2$.
	\item \label{part:splittingcondition} There exists a $p$-adic place $\overline{v}' \neq \overline{v}$ of $F^+$ such that 
	\[
	\sum_{\substack{ \overline{v}'' \in \overline{S}_p \\ \overline{v}'' \neq \overline{v}, \overline{v}'}} [F^+_{\bar v''}: \Q_p] >\frac{1}{2}[F^+:\Q].
	\]
				\item  \label{part:decomposedgeneric} $\overline{\rho}_\m$ is decomposed generic (Definition~\ref{defn:decomposed_generic}).
		\item Assume that one of the following holds:
		\begin{enumerate}
			\item  \label{part:therearecharzero} $H^\ast(X_K,\cV_\lambda)_{\frakm}[1/p] \ne 0$, or
			\item for every embedding $\tau \colon F \hookrightarrow E$ inducing the place $\overline{v}$ of $F^+$,
			\[
			-\lambda_{\tau c,n} - \lambda_{\tau , n} \leq p - 2n - 2 \quad \text{and} \quad -\lambda_{\tau c,1} -\lambda_{\tau, 1} \geq 0.
			\]
		\end{enumerate}
		\end{enumerate}
	Then there exists an integer $N \geq 1$ depending only on $[F^+ : \Q]$ and $n$, an ideal $J \subset \T^S(K, \lambda)$ satisfying $J^N = 0$, and a continuous representation
	\[ \rho_\m : G_{F, S} \to \GL_n(\T^S(K, \lambda)_{\m} / J) \]
	satisfying the following conditions:
	\begin{enumerate}
		\item[(a)] For each finite place $v \not\in S$ of $F$, the characteristic polynomial of $\rho_\m(\Frob_v)$ equals the image of $P_v(X)$ in $(\T^S(K, \lambda)_{\m} / J)[X]$.
		\item[(b)] For each place $v | \overline{v}$ of $F$, $\rho_\m|_{G_{F_v}}$ is in the essential image of  $\bG^a$ (with $a = (\lambda_{\tau, n}) \in \Z^{\Hom(F_v,E)}$). 
		\item[(c)] There is $\overline{M} \in \mathcal{MF}_{k}$ such that 
		$\overline{\rho}_\m|_{G_{F_v}} \cong \bG(\overline{M})$ and for any embedding $\tau \colon F_v \hookrightarrow E$,
		\[
		\mathrm{FL}_\tau(\overline{M}) = \{\lambda_{\tau,1} + n - 1, \lambda_{\tau,2} + n - 2, \ldots,\lambda_{\tau,n}\}.
		\]
	\end{enumerate}
\end{theorem}

\begin{proof}
	Note that the existence of a $\rho_\m$ satisfying only condition (a) is already known (Theorem \ref{thm:existence_of_residual_representation_for_GL_n}). We are therefore free to enlarge $S$ if necessary. We first prove the theorem under the assumption that $H^\ast(X_K,\cV_\lambda)_{\frakm}[1/p] \ne 0$. 
	By Theorem~\ref{thm:application_of_matsushima},	there exists an isomorphism $\iota : \overline{\Q}_p \to \C$ and a cuspidal automorphic representation $\pi$ of $\GL_n(\A_F)$ of weight $\iota\lambda$ such that $(\pi^\infty)^K \neq 0$ and such that $\overline{r_\iota(\pi)} \cong \overline{\rho}_\m$.
	By \cite[Lemma 4.9]{MR1044819}, there is an integer $w \in \Z$ such that for each embedding $\tau \colon F \hookrightarrow E$ and for each $i = 1, \dots, n$, 
	we have $\lambda_{\tau, i} + \lambda_{\tau c, n + 1 - i} = w$.
	Fix an embedding $\tau_0 \in \Hom_{\Q_p}(F_{\widetilde{v}}, E)$ such that $\lambda_{\tau_0,1} + \lambda_{\tau_0 c, 1}$ is maximal. (Recall that $\widetilde{v}$ is a fixed choice of place of $F$ lying above $\overline{v}$.)

	After possibly enlarging $E$, we can (cf.\  \cite[Lemma~2.2]{HSBT}) find a continuous character $\psi \colon G_F \to \cO^\times$ satisfying
	\begin{itemize}
		\item $\psi$ is crystalline at each $v\mid p$, 
		\item $\psi$ is unramified at $\widetilde{v}$ and at each $v \in S - S_p$,
		\item $\psi\circ \Art_{F_{\widetilde{v}^c}}|_{\cO_{F_{\widetilde{v}^c}}^\times} = \prod_{\tau \colon F_{\widetilde{v}} \hookrightarrow E} (\tau c)^{\lambda_{\tau_0,1}+\lambda_{\tau_0 c, 1}}$.
	\end{itemize}
	Define a weight $\mu = (\mu_{\tau,1},\ldots, \mu_{\tau,n}) \in (\Z^n_+)^{\Hom(F, E)}$ by letting $\mu_{\tau,i}$ be the unique $\tau$-Hodge--Tate weight for $\psi$ for each $1\le i \le n$. 
	Note that for $\tau$ inducing $\widetilde{v}$, $\mu_{\tau, i} = 0$ and $\mu_{\tau c, i} = -\lambda_{\tau_0, 1} -\lambda_{\tau_0 c, 1}$ for all $1\le i \le n$. 
	
	Choose a finite set $S'$ containing $S$ and the set of places where $\psi$ is ramified and a good normal subgroup $K' \subset K$, all satisfying the following conditions:
	\begin{itemize}
		\item $(K')^{S' - S} = K^{S' - S}$.
		\item The quotient $K' / K$ is abelian.
		\item For each finite place $v \nmid p$ of $F$, the restriction of $\psi|_{G_{F_v}} \circ \Art_{F_v}$ to $\det(K'_v)$ is trivial.  
		\item $S'$ satisfies the analogue of hypothesis~(\ref{part:runningassumption}) of the theorem.
	\end{itemize}
	By an argument with the Hochschild--Serre spectral sequence, just as in the proof of Corollary \ref{cor:first_consequence_for_FL_property}, we are free to assume that $K = K'$ and $S = S'$, and we now do this. Let $\lambda' = \lambda + \mu$. By Proposition \ref{prop:twisting_by_character}, the map 
	\[ \T^S \to \T^S, \,\, [K^S g K^S] \mapsto \psi(\Art_F(\det(g))) [K^S g K^S] \]
	descends to an isomorphism $f : \T^S(K, \lambda')_{\m(\psi)} \to \T^S(K, \lambda)_\m$. We observe that for any $\tau \in \Hom_{\Q_p}(F_{\widetilde{v}}, E)$, we have
	\[
	-\lambda_{\tau c, 1}' - \lambda_{\tau, 1}' = -\lambda_{\tau c, 1} - \lambda_{\tau, 1} + \lambda_{\tau_0 c, 1} + \lambda_{\tau_0, 1} \geq 0
	\]
	and (using that $\lambda_{\tau,i}+\lambda_{\tau c,n+1-i} = w$ is independent of $\tau$ and $i$)
	\[\begin{split}
	- \lambda_{\tau c,n}' - \lambda_{\tau, n}' = - \lambda_{\tau c,n}- \lambda_{\tau , n} + \lambda_{{\tau_0}c,1} +\lambda_{\tau_0 ,1} 
	= \lambda_{{\tau}, 1} + \lambda_{\tau c, 1} - \lambda_{{\tau_0} , n} - \lambda_{\tau_0 c, n} \\
	\leq \lambda_{{\tau_0}, 1} +\lambda_{{\tau_0} c, 1} - \lambda_{{\tau_0} c, n} - \lambda_{{\tau_0} , n} \leq p - 1 - 2n.
	\end{split}\]
	In particular, $\lambda'$ satisfies the assumptions of Corollary \ref{cor:first_consequence_for_FL_property}. 
	
	We recall (Lemma \ref{lem:cohomology_is_perfect}) that $R \Gamma(X_K, \cV_{\lambda'})$ is a perfect complex, with cohomology concentrated in the range $[0, d-1]$. It follows (cf.~\cite[Lemma 3.11]{new-tho}) that the map 
	\[ \T^S(K, {\lambda'}) \to \invlim_{m \geq 1} \T^S(R \Gamma(X_K, \cV_{\lambda'} / \varpi^m)) \]
	is an isomorphism. On the other hand, \cite[Lem. 2.5]{KT} shows that for any $m \geq 1$, the kernel of the map
	\[ \T^S(R \Gamma(X_K, \cV_{\lambda'} / \varpi^m)) \to \prod_q \T^S(H^q(X_K, \cV_{\lambda'} / \varpi^m)) \]
	is a nilpotent ideal $I$ satisfying $I^d = 0$. Applying Corollary \ref{cor:first_consequence_for_FL_property}, we see that we can find an integer $N \geq 1$ depending only on $[F^+ : \Q]$ and $n$, an ideal $J' \subset \T^S(K, \lambda')_{\m(\psi)}$ satisfying $(J')^N = 0$, and a continuous representation
	\[ \rho_{\m(\psi)} : G_{F, S} \to \GL_n( \T^S(K, \lambda')_{\m(\psi)} / J'  ) \]
	satisfying the following conditions:
		\begin{itemize}
			\item[(a')] For each place $v \not\in S$ of $F$, the characteristic polynomial of $\rho_{\m(\psi)}(\Frob_v)$ is equal to the image of $P_v(X)$ in $(\T^S(K, \lambda')_{\m(\psi)} / J')[X]$.
			\item[(b')] For each place $v | \overline{v}$ of $F$, $\rho_{\m(\psi)}|_{G_{F_{v}}}$ is in the essential image of the functor $\mathbf{G}^{a'}$, for $a' = (\lambda'_{\tau, n}) \in \Hom_{\Q_p}(F_{v}, E)$.
			\item[(c')]  For each place $v | \overline{v}$ of $F$, there exists $\overline{N} \in \cM \cF_k$ with 
			\[ \left(\overline{\rho}_{\m(\psi)} \oplus (\overline{\rho}_{\m(\psi)}^{c, \vee} \otimes \epsilon^{1-2 n})\right)|_{G_{F_v}} \cong \mathbf{G}(\overline{N}) \] 
			and
			\[ \mathrm{FL}_\tau(\overline{N}) = \{ - \lambda'_{\tau c, n}  + (2n -1 ), \dots, -\lambda'_{\tau c, 1} + n, \lambda'_{\tau, 1} +(n-1), \dots, \lambda'_{\tau, n} \}. \]
			for each embedding $\tau \in \Hom_{\Q_p}(F_{v}, E)$.
		\end{itemize}	
	Let us define $J = f(J')$ and $\rho_\m = (f \circ \rho_{\m(\psi)}) \otimes \psi^{-1}$. We see immediately that $\rho_\m$ satisfies the requirements (a) and (b) of the theorem; it remains to establish requirement (c), in other words to recover the  Fontaine--Laffaille weights of $\rhobar_{\frakm}$. 
	
	By the above, there is $\overline{M} \in \mathcal{MF}_{k}^a$ such that $\rhobar_{\frakm} \cong \bG^a(\overline{M})$. Let $x \colon \bT(K,\lambda)_{\frakm} \to \Qbar_p$ denote the homomorphism which gives the action of Hecke operators on $\iota^{-1}(\pi^\infty)^K$. The pushforward $\rho_x = x \circ \rho_\m$ via $x$ is a continuous representation of $G_{F, S}$ which is crystalline at $\widetilde{v}$ and $\widetilde{v}^c$, satisfying $\HT_\tau(\rho_x) = \mathrm{FL}_\tau(\overline{M})$ for each $\tau \in \Hom(F, E)$ inducing the place $\overline{v}$ of $F^+$.  It therefore suffices to show that
	\[ \HT_\tau(\rho_x) = \{\lambda_{\tau, 1} + n -1,\lambda_{\tau,2}+n-2,\ldots, \lambda_{\tau, n}\} \]
	for each $\tau \in \Hom(F, E)$ inducing the place $\overline{v}$ of $F^+$, or equivalently that
	\[ \HT_\tau(\rho_x\otimes\psi) = \{\lambda'_{\tau, 1} + n -1,\lambda'_{\tau,2}+n-2,\ldots, \lambda'_{\tau, n}\}. \]
	
	Let $\omega_\pi : \A_F^\times \to \C^\times$ denote the central character of $\pi$. Then $\omega_\pi$ is a character of type $A_0$ and for each embedding $\tau : F \hookrightarrow E$ inducing the place $\overline{v}$ of $F^+$, we have
	\[ \mathrm{HT}_\tau( r_\iota(\omega_\pi)) = \left\{ \sum_{i=1}^n \lambda_{\tau , i}  \right\}. \]
	Moreover, we have $\det \rho_x = r_\iota(\omega_\pi) \epsilon^{n(1-n)/2}$, hence $\det (\rho_x \otimes \psi) = r_\iota(\omega_\pi) \epsilon^{n(1-n)/2} \psi^n$, as this can be checked on Frobenius elements at unramified places. We are now done: $\mathrm{HT}_\tau(\rho_x \otimes \psi)$ is an $n$-element subset of 
	\[  \mathrm{FL}_\tau(\overline{N}) = \{ - \lambda'_{\tau c, n}  + (2n -1 ), \dots, -\lambda'_{\tau c, 1} + n, \lambda'_{\tau, 1} +(n-1), \dots, \lambda'_{\tau, n} \}. \]
	with sum equal to $\sum_{i=1}^n (\lambda'_{\tau, i} + n - i)$.
	By construction, we have 
		\[
		 - \lambda'_{\tau c, n}  + (2n -1 )> \dots> -\lambda'_{\tau c, 1} + n> \lambda'_{\tau, 1} +(n-1)> \dots> \lambda'_{\tau, n}.
		\]
	The only possibility is that $\HT_\tau(\rho_x\otimes\psi)$ has the required form. This completes the proof of the theorem in the case $H^\ast(X_K,\cV_\lambda)_{\frakm}[1/p] \ne 0$. 
	
	We now treat the second case, assuming that for every embedding $\tau \in \Hom(F, E)$ inducing the place $\overline{v}$ of $F^+$, we have
	\[
	- \lambda_{\tau c,n} - \lambda_{\tau, n} \leq p - 2n - 2 \quad \text{and} \quad -\lambda_{\tau c,1} - \lambda_{\tau, 1} \geq 0.
	\]
	In this case Corollary
        \ref{cor:first_consequence_for_FL_property} applies directly,
        and it only remains to identify the Fontaine--Laffaille weights of $\rhobar_{\frakm}$ for each place $v | \overline{v}$ of $F$.
	There are $\overline{M},\overline{M}' \in \mathcal{MF}^a_{k}$ such that 
	$\rhobar_{\frakm}|_{G_{F_v}} \cong \bG(\overline{M})$ and $(\rhobar_{\frakm}^{c,\vee} \otimes \epsilon^{1-2n})|_{G_{F_v}} \cong \bG(\overline{M}')$. 
	We choose a continuous character $\psi \colon G_F \rightarrow \cO^\times$ 
	satisfying
	\begin{itemize}
		\item $\psi$ is crystalline at each $v'\mid p$, 
		\item $\psi$ is unramified at $\widetilde{v}$,
		\item $\psi\circ \Art_{F_{\widetilde{v}^c}} = \prod_{\tau \colon F_{\widetilde{v}} \hookrightarrow E} (\tau c)$ on $\cO_{F_{\widetilde{v}^c}}^\times$.
	\end{itemize}
	After enlarging $S$, as in the first part of the proof, we can assume that $\psi$ is unramified outside $S$, in which case the maximal ideal $\m(\psi)$ of $\T^S$ is defined and occurs in the support of $H^\ast(X_K, \cV_{\lambda'})$, where the weight $\lambda' \in (\Z^n_+)^{\Hom(F, E)}$ is defined by the formula $\lambda'_\tau = \lambda_\tau$ if $\tau$ does not induce the place $\widetilde{v}^c$ of $F$, and $\lambda'_\tau = \lambda_\tau - (1, \dots, 1)$ if $\tau$ does induce the place $\widetilde{v}^c$ of $F$. We observe that the weight $\lambda'$ also satisfies the assumptions of Corollary \ref{cor:first_consequence_for_FL_property}.
	
	We can now conclude. Let $\tau \in \Hom(F, E)$ be an embedding inducing the place $\widetilde{v}$ of $F$. The sets $\mathrm{FL}_\tau(\overline{M})$ and $\mathrm{FL}_\tau(\overline{M}')$ partition the $2n$ distinct integers
	\[ - \lambda_{\tau c, n}  + (2n -1 )> \dots> -\lambda_{\tau c, 1} + n> \lambda_{\tau, 1} +(n-1)> \dots> \lambda_{\tau, n} \]
	and $\mathrm{FL}_\tau(\overline{M})$ and $\mathrm{FL}_\tau(\overline{M}')+1$ partition the $2n$ distinct integers
	\[ - \lambda_{\tau c, n}  + 2n > \dots> -\lambda_{\tau c, 1} + (n + 1) > \lambda_{\tau, 1} +(n-1)> \dots> \lambda_{\tau, n}. \]
	Using Lemma~\ref{lem:integer_partition}, this forces 
	\[ \mathrm{FL}_\tau(\overline{M}) = \{ \lambda_{\tau,1}+(n-1),\lambda_{\tau,2} + (n-2),\ldots,\lambda_{\tau,n} \} \]
	 and 
	 \[ \mathrm{FL}_{\tau}(\overline{M}') = \{ - \lambda_{\tau c, n}  + (2n -1 ), \dots,  -\lambda_{\tau c, 1} + n \}. \]
	 Since $\mathbf{G}(\overline{M}') = (\overline{\rho}_\m^{c, \vee} \otimes \epsilon^{1 - 2n})|_{G_{F_v}}$, this implies that for each place $v | \overline{v}$ of $F$, $\overline{\rho}_\m|_{G_{F_v}}$ has the correct Fontaine--Laffaille weights.
\end{proof}

\begin{lemma}\label{lem:integer_partition}
	Let $m \ge 1$ be an integer and let $A,B,C,D$ be sets of integers each of size $m$. 
	Assume that for any $c \in C$ and $d\in D$, we have $c > d$. 
	If $A \cup B = C \cup D$ and $(A+1) \cup B = (C+1) \cup D$ and both of these sets have $2m$ elements, then $A = C$ and $B = D$.
\end{lemma}

\begin{proof}
	We induct on $m$.
	Let $c$ be the largest element of $C$ and let $d$ be the smallest element of $D$. 
	Since $A \cup B = C \cup D$ and $(A+1) \cup B = (C+1) \cup D$, we must have $c \in A$ and $d \in B$. 
	We can then apply the inductive hypothesis to $A' = A \setminus \{c\}$, $B' = B \setminus \{d\}$, $C' = C\setminus\{c\}$, and $D' = D \setminus\{d\}$.
\end{proof}

\section{Local-global compatibility, \texorpdfstring{$l=p$}{l eq p} (ordinary case)}
\label{section:lep_ord}

\label{sec:ordsection} 
\subsection{Statements}\label{sec:ord_statements}

	Let $F$ be a CM field, and fix an integer $n \geq 1$. Let $p$ be a prime, and let $E$ be a finite extension of $\Q_p$ inside $\overline{\Q}_p$ large enough to contain the images of all embeddings of $F$ in $\overline{\Q}_p$. We assume throughout this chapter that $F$ satisfies the following standing hypothesis:
	\begin{itemize}
		\item $F$ contains an imaginary quadratic field in which $p$ splits.
	\end{itemize}
	In contrast to \S \ref{section:lep}, we do not assume that $p$ is unramified in $F$. As in \S \ref{section:lep}, our goal in this chapter is to establish local-global compatibility for some Hecke algebra-valued Galois representations at the $p$-adic places of $F$. More precisely, we will show that after projection to the ordinary Hecke algebra, these Galois representations satisfy an ordinariness condition (see (b) and (c) in the statement of Theorem \ref{thm:lgcord_intro} below -- the consequences of this condition will be explored in \S \ref{sec:orddef}). Before formulating the main theorem of this chapter, we must define these ordinary Hecke algebras.
	
	Let $K \subset \GL_n(\A_F^\infty)$ be a good subgroup, and let $\lambda \in (\Z_+^n)^{\Hom(F, E)}$. Let $S$ be a finite set of finite places of $F$, containing the $p$-adic places, stable under complex conjugation. We assume that the following conditions are satisfied:
	\begin{itemize}
		\item Let $v$ be a finite place of $F$ not contained in $S$, and let $l$ be its residue characteristic. Then either $S$ contains no $l$-adic places of $F$ and $l$ is unramified in $F$, or there exists an imaginary quadratic subfield of $F$ in which $l$ splits.
		\item For each place $v | p$ of $F$, $K_v = \Iw_v$. For each finite place $v \not\in S$ of $F$, $K_v = \GL_n(\cO_{F_v})$.
	\end{itemize}
	If $c \geq b \geq 0$ are integers with $c \geq 1$, then we define a good subgroup $K(b, c) \subset K$ by the formula $K(b, c)_v = K_v$ if $v \nmid p$ and $K_v = \Iw_v(b, c)$ if $v | p$. Thus $K(0, 1) = K$. Then there is an isomorphism $K(0, c) / K(b, c) \cong \prod_{v | p} T_n(\cO_{F_v} / \varpi_v^b)$. (We are using here notation for open compact subgroups and Hecke operators that has been defined in \S \ref{sec:some_useful_hecke_operators}.)
	
	We define a Hecke algebra 
	\[ \T^{S, \ord} = \T^S \otimes_\cO \cO \llbracket T_n( \cO_{F, p} ) \rrbracket[ \{ U_{v, 1}, \dots, U_{v, n}, U_{v, n}^{-1} \}_{v | p}] \] 
	(where the $U_{v, i}$ are viewed as formal variables). We write $U_v = U_{v, 1} U_{v, 2} \cdots U_{v, n-1} \in \T^{S, \ord}$ and $U_p = \prod_{v | p} U_v$. We observe that there is a canonical surjective $\cO$-algebra homomorphism $\cO \llbracket T_n( \cO_{F, p} ) \rrbracket \to \cO[K(0, c) / K(b, c)]$. This extends to a homomorphism
	\[ \T^{S, \ord} \to \End_{\mathbf{D}( \cO[K(0, c) / K(b, c)] )}(R \Gamma_{K(0, c) / K(b, c)}( X_{K(b, c)}, \cV_\lambda)), \]
	where each element $U_{v, i}$ of $\T^{S, \ord}$ acts on the complex $R \Gamma_{K(0, c) / K(b, c)}( X_{K(b, c)}, \cV_\lambda)$ by the Hecke operator of the same name. By the theory of ordinary parts (cf. \cite[\S 2.4]{KT}), there is a well-defined direct summand $R \Gamma_{K(0, c) / K(b, c)}( X_{K(b, c)}, \cV_\lambda)^{\ord}$ of $R \Gamma_{K(0, c) / K(b, c)}( X_{K(b, c)}, \cV_\lambda)$ in $\mathbf{D}( \cO[K(0, c) / K(b, c)] )$ on which $U_p$ acts invertibly, and we define $\T^S(K(b, c), \lambda)^{\ord}$ to be the image of the associated homomorphism
	\[ \T^{S, \ord} \to \End_{\mathbf{D}( \cO[K(0, c) / K(b, c)])}(R \Gamma_{K(0, c) / K(b, c)}( X_{K(b, c)}, \cV_\lambda)^{\ord})  \]
	or equivalently, extending our usage for the Hecke algebra $\T^S$,
	\[ \T^S(K(b, c), \lambda)^{\ord} = \T^{S, \ord}( R \Gamma_{K(0, c) / K(b, c)}( X_{K(b, c)}, \cV_\lambda)^{\ord} ).  \]
	We observe that there is a canonical homomorphism $\T^S(K(0, c) / K(b, c), \cV_\lambda) \to \T^S(K(b, c), \lambda)^\text{ord}$. The Hecke algebra in the source is defined in \S\ref{sec:unitary_group_setup}. In general this homomorphism is neither injective nor surjective. However, we do see from the existence of this homomorphism that for any maximal ideal $\m$ of $\T^S(K(b, c), \lambda)^{\ord}$, there exists an associated Galois representation $\overline{\rho}_\m : G_{F, S} \to \GL_n( \T^S(K(b, c), \lambda)^\text{ord} / \m)$. We call a maximal ideal $\m$ of $\T^{S, \ord}$ with residue field a finite extension of $k$ \textit{of Galois type} (resp.\ \textit{non-Eisenstein}) if its pullback to $\T^S$ is of Galois type (resp.\ non-Eisenstein) in the sense of Definition \ref{dfn:non_Eisenstein}.
	
    The Hecke operators $U_{v, i} \in \T^S(K(b, c), \lambda)^{\ord}$ are invertible (because $U_p$ is). For each place $v | p$ and for each $i = 1, \dots, n$, we define a character $\chi_{\lambda, v, i} : G_{F_v} \to \T^S(K(b, c), \lambda)^{\ord, \times}$ as the unique continuous character satisfying the identities
	\[\chi_{\lambda, v,i}\circ \Art_{F_v}(u)=\epsilon^{1-i}(\Art_{F_v}(u))\left(\prod_{\tau}\tau(u)^{-(w^G_0 \lambda)_{\tau, i}}\right)\langle \diag(1,\ldots, u,\ldots,1)\rangle \text{ }(u \in \cO_{F_v}^\times) \]
    (the product being over $\tau \in \Hom_{\Q_p}(F_v, E)$) and
	\[\chi_{\lambda, v,i}\circ \Art_{F_v}(\varpi_v)=\epsilon^{1-i}(\Art_{F_v}(\varpi_v)) \frac{U_{v, i}}{U_{v, i-1}}.\]
	We can now state the main theorem of this chapter.
	(As with Theorem~\ref{thm:lgcfl} in~\S\ref{sec:FL_statements},
	we will repeat the statement immediately before its proof with the same numbering.)
	\begin{reptheorem}{mySecondAmazingTheorem} Suppose that $[F^+ : \Q] > 1$. 
	Let $K \subset \GL_n(\A_F^\infty)$ be a good subgroup such that for each place $v \in S_p$ of $F$, $K_v = \Iw_v$. Let $c \geq b \geq 0$ be integers with $c \geq 1$, let $\lambda \in (\Z^n)^{\Hom(F, E)}$, and let $\m \subset \T^S(K(b, c), \lambda)^{\ord}$ be a non-Eisenstein maximal ideal. Suppose that the following conditions are satisfied:
	\begin{enumerate}
		\item Let $v$ be a finite place of $F$ not contained in $S$, and let $l$ be its residue characteristic. Then either $S$ contains no $l$-adic places of $F$ and $l$ is unramified in $F$, or there exists an imaginary quadratic field $F_0 \subset F$ in which $l$ splits.
		\item $\overline{\rho}_\m$ is decomposed generic.
	\end{enumerate}
	Then we can find an integer $N \geq 1$, which depends only on $[F^+ : \Q]$ and $n$, an ideal $J \subset \T^{S}(K(b, c), \lambda)^{\ord}_\m$ such that $J^N = 0$, and a continuous representation
	\[ \rho_{\m} : G_{F, S} \to \GL_n( \T^{S}(K(b, c), \lambda)^{\ord}_\m / J ) \]
	satisfying the following conditions:
	\begin{enumerate}
	\item[(a)] For each finite place $v \not\in S$ of $F$, the characteristic polynomial of $\rho_{\m}(\Frob_v)$ equals the image of $P_v(X)$ in $(\T^{S}(K(b, c), \lambda)^{\ord}_\m / J)[X]$.
	\item[(b)] For each $v \in S_p$, and for each $g \in G_{F_v}$, the characteristic polynomial of $\rho_{\m}(g)$ equals $\prod_{i=1}^n (X - \chi_{\lambda, v, i}(g))$.
	\item[(c)] For each $v \in S_p$, and for each $g_1, \dots, g_n \in G_{F_v}$, we have
	\[ (\rho_{\m}(g_1) - \chi_{\lambda, v, 1}(g_1))(\rho_{\m}(g_2) - \chi_{\lambda, v, 2}(g_2)) \dots (\rho_\m(g_n) - \chi_{\lambda, v, n}(g_n)) = 0.  \]
		\end{enumerate}
		\end{reptheorem}
	We refer the reader to Lemma \ref{lem:detord} for the comparison between the condition (c) and the usual notion of an ordinary Galois representation. In short, they coincide for representations with coefficients in a field and distinct diagonal characters.
	
	The rest of \S \ref{section:lep_ord} is devoted to the proof of Theorem \ref{thm:lgcord_intro} (after proving the theorem, we record a local-global compatibility result for a single ordinary automorphic representation as a corollary). In the rest of the chapter, we make the following additional standing hypothesis:
	\begin{itemize}
		\item For each place $v | p$ of $F$, our fixed choices of uniformizer satisfy $\varpi_{v^c} = \varpi_v^c$. 
	\end{itemize}
	This simplifies notation once we introduce the group $\widetilde{G}$. It is important to note that while the definition of the operators $U_{v, i}$ above depends on the choice of uniformizer $\varpi_v$, neither the complex $R \Gamma_{K(0, c) / K(b, c)}( X_{K(b, c)}, \cV_\lambda)^{\ord}$, nor the Hecke algebra $\T^S(K(b, c), \lambda)^{\ord}$, nor the truth of Theorem \ref{thm:lgcord_intro} depend on this choice. 
	
	\subsection{Hida theory}\label{sec:hida_theory}
	
	In the previous section we introduced the ordinary Hecke algebras $\T^S(K(b, c), \lambda)^{\ord}$. In \S \ref{sec:hida_theory}, we recall the basic results about these Hecke algebras and the complexes on which they act: this material goes under the name ``Hida theory''. 	We also describe how this theory is related to the corresponding theory for the group $\widetilde{G}$.
\subsubsection{The ordinary part of a smooth representation}\label{sec_ordinary_part_of_smooth_rep}
 Our first goal is to show, following Emerton~\cite{emordone,emordtwo}, how to define ordinary parts in a more representation-theoretic way. We will work throughout with $\cO / \varpi^m$ coefficients (for some fixed $m \geq 1$) in order to avoid topological issues. We first need to introduce some more notation. If $\Grm$ is a locally profinite group, then we write $\operatorname{Mod}(\cO / \varpi^m[\Grm])$ for the category of $\cO / \varpi^m[\Grm]$-modules, and 
\numequation
\operatorname{Mod}_{\text{sm}}(\cO / \varpi^m[\Grm]) \subset \operatorname{Mod}(\cO / \varpi^m[\Grm])
\end{equation} for the full subcategory of smooth modules. More generally, if $\Delta \subset \Grm$ is an open submonoid which contains an open compact subgroup of $\Grm$, then we write 
\numequation\label{eqn:inclusion_of_categories}
\operatorname{Mod}_{\text{sm}}(\cO / \varpi^m[\Delta]) \subset \operatorname{Mod}(\cO / \varpi^m[\Delta])
\end{equation}
for the full subcategory of smooth modules (by definition, those for which every vector is fixed by an open subgroup of $\Delta$). We write 
\[ M \mapsto M^\text{sm} : \operatorname{Mod}(\cO / \varpi^m[\Delta]) \to \operatorname{Mod}_{\text{sm}}(\cO / \varpi^m[\Delta]) \] 
for the functor of smooth vectors; it is right adjoint to the inclusion (\ref{eqn:inclusion_of_categories}). 
\begin{lemma}\label{lem:injectives_for_smooth_monoid_actions} \leavevmode
	\begin{enumerate} 
		\item The category $\operatorname{Mod}_{\text{sm}}(\cO / \varpi^m[\Delta])$ is abelian and has enough injectives.
		\item Let $\Delta' \subset \Delta$ be a subgroup which is either compact or open ($\Delta'$ is therefore a locally profinite group). Then the forgetful functor
		\[ \operatorname{Mod}_{\text{sm}}(\cO / \varpi^m[\Delta]) \to \operatorname{Mod}_{\text{sm}}(\cO / \varpi^m[\Delta']) \]
		preserves injectives. 
		\end{enumerate} 
\end{lemma}
\begin{proof}
	The functor $M \mapsto M^\text{sm}$ has an exact left adjoint, so preserves injectives. Since the category $ \operatorname{Mod}(\cO / \varpi^m[\Delta])$ has enough injectives, so does $\operatorname{Mod}_{\text{sm}}(\cO / \varpi^m[\Delta])$. 
	
	For the second part of the lemma, we split into cases. Suppose first that $\Delta' \subset \Delta$ is an open subgroup. Then compact induction $\operatorname{c-Ind}_{\Delta'}^\Delta$ is an exact left adjoint to the forgetful functor. Suppose instead that $\Delta' \subset \Delta$ is a compact subgroup. In this case, we can find a compact open subgroup of $\Delta$ which contains $\Delta'$. Using what we have already proved, we can assume that $\Delta = \Grm$, in which case the result follows from \cite[Prop. 2.1.11]{emordtwo}.
\end{proof}
We write $\mathbf{D}_\text{sm}(\cO / \varpi^m[\Delta])$ for the derived category of $\operatorname{Mod}_{\text{sm}}(\cO / \varpi^m[\Delta])$. 

We introduce some monoids, with the aim of studying the theory for $\Grm = \GL_n(F_p)$. We write $T_n(F_p)^+ \subset T_n(F_p)$ for the open submonoid consisting of those elements $t \in T_n(F_p)$ with $t N_n(\cO_{F, p}) t^{-1} \subset N_n(\cO_{F, p})$, and $T_n(F_v)^+ = T_n(F_v) \cap T_n(F_p)^+$. We recall (\S \ref{sec:some_useful_hecke_operators}) that $\Delta_p \subset \GL_n(F_p)$ denotes the monoid $\prod_{v | p} \Iw_v T_n(F_v)^+ \Iw_v$. If $b \geq 0$ is an integer, we define 
\[ T_n(\cO_{F, p})(b) = \prod_{v \in S_p} \ker( T_n(\cO_{F, v}) \to T_n(\cO_{F, v} / \varpi_v^b)),\]
\[ T_n(\cO_{F,p})_b = T_n(\cO_{F,p}) / T_n(\cO_{F, p})(b), \]
\[ T_n(F_p)_b^+ = T_n(F_p)^+ / T_n(\cO_{F, p})(b) \]
and
\[ T_n(F_p)_b = T_n(F_p) / T_n(\cO_{F, p})(b). \]
We write $u_p \in T_n(\Q_p) \subset T_n(F_p)$ for the element $(p^{n-1}, p^{n-2}, \dots, 1)$. It lies in $T_n(F_p)^+$. We define $B_n(F_p)^+ = N_n(\cO_{F, p}) \cdot T_n(F_p)^+ \subset B_n(F_p)$. Note that $B_n(F_p)^+ \subset \Delta_p$. We write $B_n(\cO_{F, p})(b)$ for the pre-image in $B_n(\cO_{F, p})$ of $T_n(\cO_{F, p})(b)$. It will be important for us to note that a complex $C \in \mathbf{D}_{\text{sm}}(\cO / \varpi^m[ T_n(F_p)^+_b ])$ comes equipped with a functorial homomorphism 
\[ \cO \llbracket T_n(\cO_{F, p}) \rrbracket [ \{ U_{v, 1}, \dots, U_{v, n}, U_{v, n}^{-1} \}_{v \in S_p} ] \to \End_{ \mathbf{D}_{\text{sm}}(\cO / \varpi^m[ T_n(F_p)^+_b ])}(C) \]
 via the map which is the canonical homomorphism
\[ \cO \llbracket T_n(\cO_{F, p}) \rrbracket \to \cO / \varpi^m[ T_n(\cO_{F, p}) / T_n(\cO_{F, p})(b) ] \]
on this subalgebra and which sends $U_{v, i}$ to the matrix 
\[ \diag(\varpi_v, \dots, \varpi_v, 1, \dots, 1) \in T_n(F_v) \subset T_n(F_p) \]
 (with $i$ occurrences of $\varpi_v$). Consequently, if $\T^S$ acts on a complex $C$, then we can extend this to an action of the algebra $\T^{S, \ord}$.

If $\lambda \in X^\ast((\Res_{F/\Q} T_n)_E) = (\Z^n)^{\Hom(F, E)}$, then we write $\cO(\lambda)$ for the $\cO[T_n(F_p)]$-module defined as follows: it is a free rank 1 $\cO$-module on which an element $u \in T_n(\cO_{F, p})$ acts as multiplication by the scalar $\prod_{\tau \in \Hom(F, E)} \prod_{i=1}^n \tau(u_i)^{\lambda_{{\tau}, i}}$ and on which any element $\diag(\varpi^{a_1}_v, \dots, \varpi^{a_n}_v)$ ($a_i \in \Z$) acts trivially.

We recall that in \S \ref{sec:some_useful_hecke_operators} we have defined, for any $\lambda \in (\Z^n_+)^{\Hom(F, E)}$, a twisted action $(\delta, v) \mapsto \delta \cdot_p v$ of $\Delta_p$ on $\cV_\lambda$. Projection to the lowest weight space determines an $\cO$-module homomorphism $\cV_\lambda \to \cO(w_0^G \lambda)$ which is equivariant for the action of $B_n(F_p)^+$ (where $B_n(F_p)^+$ acts through the $\cdot_p$-action on the source and through its projection to $T_n(F_p)$ on the target). We write $\mathcal{K}_\lambda$ for the kernel of the projection $\cV_\lambda \to \cO(w_0^G \lambda)$; it is again an $\cO[B_n(F_p)^+]$-module, finite free as $\cO$-module.

We now define various functors that together will allow us to study ordinary parts using completed cohomology. We write 
\[ \Gamma(N_n(\cO_{F, p}), -) : \operatorname{Mod}_{\text{sm}}(\cO / \varpi^m[\Delta_p]) \to \operatorname{Mod}_{\text{sm}}(\cO / \varpi^m[T_n(F_p)^+]) \]
for the functor of $N_n(\cO_{F, p})$-invariants. If $V \in \operatorname{Mod}_{\text{sm}}(\cO / \varpi^m[\Delta_p])$, then the action of an element $t \in T_n(F_p)^+$ on $v \in \Gamma(N_n(\cO_{F, p}), V)$ is given by the formula 
\numequation\label{eqn:torus_action_on_N_invariants} t \cdot v = \sum_{n \in 
N_n(\cO_{F, p}) / t N_n(\cO_{F, p}) t^{-1}} ntv  
\end{equation}
(cf. \cite[\S 3]{emordone}, and note that the action of $t$ is by the `double coset operator' $[N_n(\cO_{F,p})tN_n(\cO_{F,p})]$). We write
\[ \Gamma(B_n(\cO_{F, p})(b), -) : \operatorname{Mod}_{\text{sm}}(\cO / \varpi^m[\Delta_p]) \to \operatorname{Mod}(\cO / \varpi^m[T_n(F_p)_b^+]) \]
for the functor of $B_n(\cO_{F, p})(b)$-invariants. The action of an element $t \in T_n(F_p)_b^+$ is given by the same formula (\ref{eqn:torus_action_on_N_invariants}). 

If $c \geq b \geq 0$ are integers with $c \geq 1$, then we define $\Iw_p(b, c) = \prod_{v \in S_p} \Iw_v(b, c) \subset \GL_n(F_p)$. We write
\[ \Gamma(\Iw_p(b, c), -) : \operatorname{Mod}_{\text{sm}}(\cO / \varpi^m[\Delta_p]) \to \operatorname{Mod}(\cO / \varpi^m[T_n(F_p)_b^+]) \]
for the functor of $\Iw_p(b, c)$-invariants. If $V \in \operatorname{Mod}_{\text{sm}}(\cO / \varpi^m[\Delta_p])$, then the action of an element $t \in T_n(F_p)^+$ on $v \in \Gamma(\Iw_p(b,c), V)$ is given by the action of the Hecke operator $[ \Iw_p(b, c) t \Iw_p(b, c) ]$ (cf. \S \ref{sec:hecke_algebra_of_a_monoid}).

For any $b \geq 0$, we consider the functors
\[ \Gamma(T_n(\cO_{F, p})(b), -) : \operatorname{Mod}_{\text{sm}}(\cO / \varpi^m[T_n(F_p)^+]) \to \operatorname{Mod}(\cO / \varpi^m[T_n(F_p)_b^+]) \]
and
\[  \Gamma(T_n(\cO_{F, p})(b), -) : \operatorname{Mod}_{\text{sm}}(\cO / \varpi^m[T_n(F_p)]) \to \operatorname{Mod}(\cO / \varpi^m[T_n(F_p)_b]) \]
of $T_n(\cO_{F, p})(b)$-invariants. Finally, we write 
\[ \ord : \operatorname{Mod}_{\text{sm}}(\cO / \varpi^m[T_n(F_p)^+]) \to \operatorname{Mod}_{\text{sm}}(\cO / \varpi^m[T_n(F_p)]) \]
and
\[ \ord_b : \operatorname{Mod}(\cO / \varpi^m[T_n(F_p)_b^+]) \to \operatorname{Mod}(\cO / \varpi^m[T_n(F_p)_b]) \]
for the localization functors $- \otimes_{\cO / \varpi^m[T_n(F_p)^+]}\cO / \varpi^m[T_n(F_p)]$ and  $- \otimes_{\cO / \varpi^m[T_n(F_p)_b^+]}\cO / \varpi^m[T_n(F_p)_b]$, respectively. (As the notation suggests, we will use localization to define ``ordinary parts''. The reader may object that the ordinary part usually denotes a direct summand, rather than a localization. At least in the context of $\cO / \varpi^m[T_n(F_p)_b^+]$-modules which are finitely generated as $\cO / \varpi^m$-modules, the two notions agree (cf. \cite[Lemma 3.2.1]{emordtwo} and also Proposition \ref{prop:comparison_of_completed_and_classical_ord_coh} below). We use localization here since it is easier to define without finiteness conditions.)
\begin{lemma}\label{lem_ord_functors_commute}
		The following diagram is commutative up to natural isomorphism:
		\[ \xymatrix{ \operatorname{Mod}_{\text{sm}}(\cO / \varpi^m[T_n(F_p)^+])   \ar[rrrr]^*+<1em>{\Gamma( T_n(\cO_{F,p})(b), - )} \ar[d]_{\ord}&&& & \operatorname{Mod}(\cO / \varpi^m[T_n(F_p)_b^+]) \ar[d]^{\ord_b} \\
			\operatorname{Mod}_{\text{sm}}(\cO / \varpi^m[T_n(F_p)])  \ar[rrrr]_*+<1em>{\Gamma( T_n(\cO_{F,p})(b), - )} &&&& \operatorname{Mod}(\cO / \varpi^m[T_n(F_p)_b]).
		}  \]
\end{lemma}
\begin{proof}
	Let $M \in  \operatorname{Mod}_{\text{sm}}(\cO / \varpi^m[T_n(F_p)^+])$. There is a natural morphism \[ \ord_b \Gamma(T_n(\cO_{F, p})(b), M ) \to \Gamma(T_n(\cO_{F, p})(b), \ord M), \]
	or equivalently
	\begin{multline*} M^{T_n(\cO_{F, p})(b)} \otimes_{\cO / \varpi^m[ T_n(F_p)^+_b ]} \cO / \varpi^m[T_n(F_p)_b] \\ \to (M \otimes_{\cO / \varpi^m[ T_n(F_p)^+ ]} \cO / \varpi^m[T_n(F_p)])^{T_n(\cO_{F, p})(b)}. \end{multline*}
	We must show that it is an isomorphism. It is injective because $M^{T_n(\cO_{F, p})(b)} \to M$ is injective and localization is exact. To show it is surjective, let $x \in M$, and suppose that $x \otimes 1 \in (M \otimes_{\cO / \varpi^m[ T_n(F_p)^+ ]} \cO / \varpi^m[T_n(F_p)])^{T_n(\cO_{F, p})(b)}$. We must show that there exists $n \geq 0$ such that $u_p^{n} x \in M^{T_n(\cO_{F, p})(b)}$. Since $M$ is smooth, there exists $c \geq b$ such that $x \in M^{T_n(\cO_{F, p})(c)}$. On the other hand, our assumption on $x \otimes 1$ means that for any $t \in T_n(\cO_{F, p})(b)$, there exists $n(t)$ such that $u_p^{n(t)} (t - 1) x = 0$ in $M$. Choosing $n(t)$ to be as small as possible, we see that $n(t)$ depends only on the image of $t$ in the (finite) quotient $T_n(\cO_{F, p})(b) / T_n(\cO_{F, p})(c)$. We can therefore take $n = \sup_t n(t)$.
\end{proof}
\begin{lemma}\label{lem:N_n-invariants_are_T_n_acyclic} \leavevmode
	\begin{enumerate}
		\item Each functor $\Gamma(N_n(\cO_{F, p}), -)$, $\Gamma(B_n(\cO_{F, p})(b), -)$, and $\Gamma(\Iw_p(b, c), -)$ is left exact. For any $b \geq 0$, the functor $\Gamma(N_n(\cO_{F, p}), -)$ sends injectives to $\Gamma(T_n(\cO_{F, p})(b), -)$-acyclics.
		\item The functors $\ord$ and $\ord_b$ are exact and preserve injectives.
	\end{enumerate}
\end{lemma}
\begin{proof}
	It is immediate from the definitions that the three functors in the first part are left exact. We now show that the functor $\Gamma(N_n(\cO_{F, p}), -)$ sends injectives to $\Gamma(T_n(\cO_{F, p})(b), -)$-acyclics. 
	
	We have a commutative diagram
	\[ \xymatrix{ \operatorname{Mod}_{\text{sm}}(\cO / \varpi^m[\Delta_p]) \ar[d]_\alpha \ar[r] & \operatorname{Mod}_{\text{sm}}(\cO / \varpi^m[T_n(F_p)^+]) \ar[d]_\beta\ar[r] & \operatorname{Mod}(\cO / \varpi^m[T_n(F_p)^+_b])\ar[d]_\gamma\\ 
		 \operatorname{Mod}_{\text{sm}}(\cO / \varpi^m[B_n(\cO_{F, p})(b)])  \ar[r] & \operatorname{Mod}_{\text{sm}}(\cO / \varpi^m[T_n(\cO_{F, p})(b)]) \ar[r] & \operatorname{Mod}(\cO / \varpi^m)	} \]
	where the horizontal arrows are taking invariants and the
        vertical arrows are restriction to compact or open
        subgroups. By Lemma
        \ref{lem:injectives_for_smooth_monoid_actions}, the vertical
        arrows are exact and preserve injectives. We must show that if
        $\cI \in \operatorname{Mod}_{\text{sm}}(\cO /
        \varpi^m[\Delta_p])$ is injective, then for each $i > 0$, $R^i
        \Gamma(T_n(\cO_{F, p})(b), \Gamma(N_n(\cO_{F, p}), \cI)) =
        0$. Equivalently (using the formula for a composition of
        derived functors, \cite[Corollary 10.8.3]{weibel}), we must show that
	\[ \gamma R^i \Gamma(T_n(\cO_{F, p})(b), \Gamma(N_n(\cO_{F, p}), \cI)) =  R^i \Gamma(T_n(\cO_{F, p})(b), \Gamma(N_n(\cO_{F, p}), \alpha \cI)) = 0. \]
	However, $\alpha \cI$ is injective, so this follows from the fact that the functor
	\[ \Gamma(N_n(\cO_{F, p}), -) :  \operatorname{Mod}_{\text{sm}}(\cO / \varpi^m[B_n(\cO_{F, p})(b)]) \to \operatorname{Mod}_{\text{sm}}(\cO / \varpi^m[T_n(\cO_{F, p})(b)]) \]
	preserves injectives (because it has an exact left adjoint, given by inflation). This proves the first part of the lemma. 
	
	We now prove the second part of the lemma. Both $\ord$ and $\ord_b$ are exact because localization is an exact functor. Since localization preserves injectives in the case of a Noetherian base ring, $\ord_b$ preserves injectives. To show that $\ord$ preserves injectives, we go back to the definitions. Let $\cI$ be an injective object of $\operatorname{Mod}_{\text{sm}}(\cO / \varpi^m[T_n(F_p)^+])$, let $M \hookrightarrow N$ be an inclusion in $\operatorname{Mod}_{\text{sm}}(\cO / \varpi^m[T_n(F_p)])$, and let $\alpha : M \to \ord(\cI)$ be a morphism. We must show that $\alpha$ extends to $N$.
	
	For any $b \geq 0$, passing to $T_n(\cO_{F, p})(b)$-fixed vectors gives a morphism (cf. Lemma \ref{lem_ord_functors_commute})
	\[ \alpha(b) : M^{T_n(\cO_{F, p})(b)} \to \ord(\cI)^{T_n(\cO_{F, p})(b)} \cong \ord_b(\cI^{T_n(\cO_{F, p})(b)}). \]
	 The object $\cI^{T_n(\cO_{F, p})(b)} \in \operatorname{Mod}(\cO / \varpi^m[T_n(F_p)_b^+])$ is injective, showing that we can extend $\alpha(b)$ to a morphism $\alpha(b)' : N^{T_n(\cO_{F, p})(b)} \to  \ord(\cI)^{T_n(\cO_{F, p})(b)}$. Zorn's lemma implies that there exists a maximal extension $\alpha': L_{\max} \to \ord(\cI)$ of $\alpha$. The preceding argument shows that we can extend the map induced by $\alpha'$ on $T_n(\cO_{F, p})(b)$-invariants from $L_{\max}^{T_n(\cO_{F, p})(b)}$ to $N^{T_n(\cO_{F, p})(b)}$. It follows that we can extend $\alpha'$ to $L_{\max}+N^{T_n(\cO_{F, p})(b)}$. By maximality, and since $N = \cup_{b \geq 0}  N^{T_n(\cO_{F, p})(b)}$, we have $L_{\max}=N$, as desired.
\end{proof}
\begin{lemma}\label{lem:hida_lemma}
	For any $c \geq b \geq 0$ with $c \geq 1$, there is a natural isomorphism 
	\[ \ord_b \circ \Gamma(\Iw_p(b, c), - ) \cong \ord_b \circ \Gamma(B_n(\cO_{F, p})(b), -) \]
	of functors 
	\[ \operatorname{Mod}_{\text{sm}}(\cO / \varpi^m[\Delta_p]) \to \operatorname{Mod}(\cO / \varpi^m[ T_n(F_p)_b]). \]
\end{lemma}
\begin{proof}
	We first show that for any $V \in \operatorname{Mod}_{\text{sm}}(\cO / \varpi^m[\Delta_p])$, the natural inclusion $\Gamma(\Iw_p(b, c), V) \subset \Gamma(B_n(\cO_{F, p})(b), V)$ is a morphism of $\cO / \varpi^m[T_n(F_p)^+_b]$-modules. A given element $t \in T_n(F_p)^+_b$ acts on the source via the Hecke operator $[\Iw_p(b, c) t \Iw_p(b, c)]$ and on the target by the formula (\ref{eqn:torus_action_on_N_invariants}). We see that we must show that the map 
	\[ N(\cO_{F, p}) / t N(\cO_{F, p}) t^{-1} \to \Iw_p(b, c) /  ( \Iw_p(b, c) \cap t \Iw_p(b, c) t^{-1} )\]
	is bijective. This is true, because $\Iw_p(b, c)$ admits an Iwahori decomposition with respect to $B_n$ (cf. \S \ref{sec:hecke_algebra_of_a_monoid}).
	
	The exactness of $\ord_b$ implies that for any $V \in \operatorname{Mod}_{\text{sm}}(\cO / \varpi^m[\Delta_p])$, there is an inclusion $\ord_b \Gamma(\Iw_p(b, c), V) \subset \ord_b  \Gamma(B_n(\cO_{F, p})(b), V)$. We must show that this is an equality. 
	
	We have $\cO / \varpi^m[ T_n(F_p)_b^+][u_p]^{-1} = \cO / \varpi^m[ T_n(F_p)_b]$. Consequently, the lemma will follow if we can show that for any $v \in \Gamma(B_n(\cO_{F, p})(b), V)$, there exists $n \geq 0$ such that $u_p^n \cdot v \in \Gamma(\Iw_p(b, c), V) = V^{\Iw_p(b, c)}$.
	
	Since $V$ is smooth, there exists $c' > c$ such that $v \in V^{\Iw_p(b, c')}$. By induction, it is enough to show that $U_p \cdot v \in V^{\Iw_p(b, c'-1)}$. The definition of the Hecke operator $U_p$ shows that this will follow if the double coset $\Iw_p(b, c') u_p \Iw_p(b, c')$ is invariant under left multiplication by the group $\Iw_p(b, c'-1)$. This is true, as proved in e.g.\ \cite[Lemma 2.19]{ger}. 
\end{proof}
\begin{lemma}\label{lem:compatibility_of_ordinary_parts}
	Let $\pi \in \mathbf{D}_{\text{sm}}(\cO / \varpi^m[\Delta_p])$ be a bounded below complex. Then for any $c \geq b \geq 0$, $c \ge 1$, there is a natural isomorphism
	\[ R \Gamma( T_n(\cO_{F, p})(b), \ord R \Gamma( N_n(\cO_{F, p}), \pi )) \cong \ord_b R \Gamma(\Iw_p(b, c), \pi) \]
	in $\mathbf{D}(\cO / \varpi^m[ T_n(F_p)_b])$.
\end{lemma}
\begin{proof}
	We will use \cite[Corollary 10.8.3]{weibel} (composition formula for derived functors) repeatedly. Since $\ord$ preserves injectives, this implies the existence of a natural isomorphism
	\[  \begin{split} R \Gamma( T_n(\cO_{F, p})(b), - ) \circ \ord & \cong R(\Gamma(  T_n(\cO_{F, p})(b), \ord(-))\\
	& \cong R( \ord_b \circ \Gamma(  T_n(\cO_{F, p})(b), -)) \\ & \cong \ord_b R \Gamma(T_n(\cO_{F, p})(b), -). \end{split} \]
	It follows that for $\pi$ as in the statement of the lemma, there is a natural isomorphism
	\[ R \Gamma( T_n(\cO_{F, p})(b), \ord R \Gamma( N_n(\cO_{F, p}), \pi )) \cong \ord_b R\Gamma(T_n(\cO_{F, p})(b), R \Gamma( N_n(\cO_{F, p}), \pi )). \]
	Using the first part of Lemma \ref{lem:N_n-invariants_are_T_n_acyclic}, we see that there is a natural isomorphism
	\[ 	R\Gamma(T_n(\cO_{F, p})(b), R \Gamma( N_n(\cO_{F, p}), \pi )) \cong R \Gamma(B_n(\cO_{F, p}(b), \pi)). \]
	Lemma \ref{lem:hida_lemma} implies the existence of a natural isomorphism
	\[ \begin{split} \ord_b R \Gamma(B_n(\cO_{F, p})(b), \pi)) & \cong R(\ord_b \Gamma(B_n(\cO_{F, p})(b), -))(\pi) \\ & \cong R(\ord_b \Gamma(\Iw_p(b, c), - ))(\pi) \\ & \cong \ord_b R \Gamma(\Iw_p(b, c), \pi). \end{split} \]
	This concludes the proof.
\end{proof}
\subsubsection{The ordinary part of completed cohomology}\label{sec:ordinary_part_of_completed_coh}
We now apply the formalism developed in the previous section to the cohomology groups of the spaces $X_K$. If $K \subset \GL_n(\A_F^\infty)$ is a good subgroup, then there are functors
\[ \Gamma_{K^p, \text{sm}} : \operatorname{Mod}(\cO / \varpi^m[G^{\infty}])  \to \operatorname{Mod}_{\text{sm}}(\cO / \varpi^m[G(F_p^+)]) \]
and
\[ \Gamma_{K^p, \text{sm}} : \operatorname{Mod}(\cO / \varpi^m[G^{p, \infty} \times \Delta_p])  \to \operatorname{Mod}_{\text{sm}}(\cO / \varpi^m[\Delta_p]) \]
which send a module $M$ to $\Gamma(K^p, M)^\text{sm}$. If $\lambda \in (\Z^n_+)^{\Hom(F, E)}$, then we define the weight $\lambda$ completed cohomology 
\[ \pi(K^p, \lambda, m) = R \Gamma_{K^p, \text{sm}} R \Gamma(\mathfrak{X}_G, \cV_\lambda / \varpi^m) \in \mathbf{D}_\text{sm}(\cO / \varpi^m[\Delta_p]). \]
If $K^S = \prod_{v \not\in S} \GL_n(\cO_{F, v})$, then $\pi(K^p, \lambda, m)$ comes equipped with a homomorphism 
\numequation\label{eqn:weight_lambda_completed_hecke_action} \T^S \to \End_{\mathbf{D}_\text{sm}(\cO / \varpi^m[\Delta_p])}(\pi(K^p, \lambda, m))
\end{equation}
and, if $K_p \subset \Delta_p$, a canonical $\T^S$-equivariant isomorphism
\numequation R \Gamma(K_p, \pi(K^p, \lambda, m)) \cong R \Gamma(X_K, \cV_\lambda / \varpi^m) 
\end{equation}
in $\mathbf{D}(\cO / \varpi^m)$. We define $\pi(K^p, m) = R \Gamma_{K^p, \text{sm}} R \Gamma(\mathfrak{X}_G, \cO / \varpi^m) \in \mathbf{D}_\text{sm}(\cO / \varpi^m[G(F^+_p)])$; this complex comes equipped with a homomorphism 
\numequation\label{eqn:weight_0_completed_hecke_action} \T^S \to \End_{\mathbf{D}_\text{sm}(\cO / \varpi^m[G(F^+_p)])}(\pi(K^p, m)),
\end{equation}
which recovers (\ref{eqn:weight_lambda_completed_hecke_action}) in the case $\lambda = 0$ after applying the forgetful functor to $\mathbf{D}_\text{sm}(\cO / \varpi^m[\Delta_p])$. We write $\T^S(K^p, m)$ for the image of (\ref{eqn:weight_0_completed_hecke_action}).
\begin{lemma}\label{lem:hecke_algebra_of_completed_cohomology}
	Let $K \subset \GL_n(\A_F^\infty)$ be a good subgroup. Then $\T^S(K^p, m)$ is a semi-local ring, complete with respect to the $J$-adic topology defined by its Jacobson radical~$J$. For each maximal ideal $\m \subset \T^S(K^p, m)$, there is a unique idempotent $e_\m \in \T^S(K^p, m)$ with the property $e_\m H^\ast( \pi(K^p, m)) = H^\ast( \pi(K^p, m))_\m$.
\end{lemma}
\begin{proof}
See \cite[Lemma 2.1.14]{geenew}.
\end{proof}
One important consequence of Lemma \ref{lem:hecke_algebra_of_completed_cohomology} is that the localization 
\[ \pi(K^p, m)_\m \in \mathbf{D}_\text{sm}(\cO / \varpi^m[G(F^+_p)]) \]
is defined.

We define the ordinary part of completed cohomology
\[ \pi^{\ord}(K^p, \lambda, m) = \ord R \Gamma(N_n(\cO_{F, p}), \pi(K^p, \lambda, m)) \in \mathbf{D}_\text{sm}(\cO / \varpi^m[T_n(F_p)]). \]
(If $\lambda = 0$, then we write simply $ \pi^{\ord}(K^p, m)$.) Its relation to the complex $R \Gamma_{K(0, c) / K(b, c)}( X_{K(b, c)}, \cV_\lambda)^{\ord}$ defined in \S \ref{sec:ord_statements} is the expected one:
\begin{prop}\label{prop:comparison_of_completed_and_classical_ord_coh}
	Let $K \subset G^\infty$ be a good subgroup with $K_v = \Iw_v$ for each $v | p$ and $K^S = \prod_{v \not\in S} \GL_n(\cO_{F_v})$. Let $c \geq b \geq 0$ be integers with $c \geq 1$. Then for any $\lambda \in (\Z^n_+)^{\Hom(F, E)}$, there is a $\T^{S, \ord}$-equivariant isomorphism
	\[ R \Gamma(T_n(\cO_{F, p})(b), \pi^{\ord}(K^p, \lambda, m)) \cong R \Gamma_{K(0, c) / K(b, c)}( X_{K(b, c)}, \cV_\lambda/ \varpi^m)^{\ord}   \]
	in $\mathbf{D}(\cO / \varpi^m[ K(0, c) / K(b, c) ])$. (Recall that we may identify $K(0,c)/K(b,c)$ with $T_n(\cO_{F,p})_b$.)
\end{prop}
\begin{proof}
	We compute. We have a $\T^S$-equivariant isomorphism
	\[ R \Gamma(T_n(\cO_{F, p})(b), \ord R \Gamma(N_n(\cO_{F, p}), \pi(K^p, \lambda, m)))    \cong \ord_b R \Gamma(\Iw_p(b, c), \pi(K^p, \lambda, m))  \]
	in $\mathbf{D}(\cO / \varpi^m[T_n(F_p)_b])$.
	We have a morphism
	\begin{align*} R\Gamma_{K(0,c)/K(b,c)}(X_{K(b, c)}, \cV_\lambda / \varpi^m)^\text{ord} \to & R\Gamma_{K(0,c)/K(b,c)}(X_{K(b, c)}, \cV_\lambda / \varpi^m) \\ \to & \ord_b R \Gamma(\Iw_p(b, c), \pi(K^p, \lambda, m))  \end{align*}
	in $\mathbf{D}(\cO/ \varpi^m[ T_n(\cO_{F, p})_b ])$. Note that we identify $R\Gamma_{K(0,c)/K(b,c)}(X_{K(b, c)}, \cV_\lambda / \varpi^m)$ and $R \Gamma(\Iw_p(b, c), \pi(K^p, \lambda, m))$ in $\mathbf{D}(\cO/ \varpi^m[ T_n(\cO_{F, p})_b ])$. To complete the proof, we must show that our morphism induces an isomorphism on cohomology groups. This, in turn, reduces us to the problem of showing that if $M$ is an $\cO / \varpi^m[U]$-module, finite as $\cO / \varpi^m$-module, and $M^{\ord}$ is the maximal direct summand of $M$ on which $U$ acts invertibly, then the natural map $M^{\ord} \to M \to M \otimes_{\cO / \varpi^m[U]} \cO / \varpi^m[U, U^{-1}]$ is an isomorphism of $\cO / \varpi^m$-modules. This is true (cf. \cite[Lemma 3.2.1]{emordtwo}).
\end{proof}
	
	\begin{cor}[Independence of level]\label{cor:independence_of_level}
			Let $K \subset \GL_n(\A_F^\infty)$ be a good subgroup with $K_v = \Iw_v$ for each $v | p$ and $K^S = \prod_{v \not\in S} \GL_n(\cO_{F_v})$. Let $c \geq b \geq 0$ be integers with $c \geq 1$. Then for any $\lambda \in (\Z^n_+)^{\Hom(F, E)}$, the natural morphism
			\[ R \Gamma_{K(0,\max(1,b))/K(b,\max(1,b))}(X_{K(b, \max(1, b))}, \cV_\lambda / \varpi^m)^{\ord} \to R \Gamma_{K(0,c)/K(b,c)}(X_{K(b, c)}, \cV_\lambda / \varpi^m)^{\ord} \]
			in $\mathbf{D}(\cO / \varpi^m[ T_n(\cO_{F, p})_b])$ is an isomorphism. 
	\end{cor}
\begin{prop}\label{prop:independence_of_weight}
	Let $K \subset \GL_n(\A_{F}^\infty)$ be a good subgroup with $K^S = \prod_{v\not\in S} \GL_n(\cO_{F, v})$. Then there are $\T^S$-equivariant isomorphisms in $\mathbf{D}( \cO / \varpi^m[T_n(F_p)] )$:
	\[ \begin{split} \pi^{\ord}(K^p, \lambda, m) & \cong \ord R \Gamma(N_n(\cO_{F, p}), R \Gamma_{K^p, \text{sm}} R \Gamma(\mathfrak{X}_G, \cO(w_0^G \lambda) / \varpi^m)) \\
	& \cong \pi^{\ord}(K^p, m) \otimes_{\cO}  \cO(w_0^G \lambda). \end{split} \]
\end{prop}
\begin{proof}
	By definition, we have
	\[ \pi^{\ord}(K^p, \lambda, m) = \ord R \Gamma(N_n(\cO_{F, p}), R \Gamma_{K^p, \text{sm}} R \Gamma(\mathfrak{X}_G, \cV_\lambda / \varpi^m)).  \]
	This depends only on the image of $R \Gamma_{K^p, \text{sm}} R \Gamma(\mathfrak{X}_G, \cV_\lambda / \varpi^m)$ in the category $\mathbf{D}(\cO / \varpi^m[ B_n(F_p)^+ ]$. In this category, the $B_n(F_p)^+$-equivariant morphism $\cV_\lambda \to  \cO(w_0^G \lambda)$ induces a morphism
	\[ \pi^{\ord}(K^p, \lambda, m) \to \ord R \Gamma(N_n(\cO_{F, p}), R \Gamma_{K^p, \text{sm}} R \Gamma(\mathfrak{X}_G,  \cO(w_0^G \lambda) / \varpi^m)). \]
	To show that this is an isomorphism, we just need to check that
	\[ \ord R \Gamma(N_n(\cO_{F, p}), R \Gamma_{K^p, \text{sm}} R \Gamma(\mathfrak{X}_G, \mathcal{K}_\lambda / \varpi^m)) = 0, \]
	where we recall that $\mathcal{K}_\lambda = \ker( \cV_\lambda \to  \cO(w_0^G \lambda))$. This follows from the observation that for sufficiently large $N \geq 1$, we have $u_p^N \mathcal{K}_\lambda / \varpi^m = 0$ (cf.~the proof of \cite[Proposition 2.22]{ger}). The existence of the second isomorphism follows from the fact that $N_n(\cO_{F, p})$ acts trivially on $ \cO(w_0^G \lambda)$.
\end{proof}
\begin{cor}[Independence of weight]\label{cor_independence_of_weight}
	Let $K \subset \GL_n(\A_F^\infty)$ be a good subgroup with $K_v = \Iw_v$ for each $v | p$ and $K^S = \prod_{v \not\in S} \GL_n(\cO_{F_v})$. Let $c \geq b \geq 0$ be integers with $c \geq 1$. Then for any $\lambda, \lambda' \in (\Z^n_+)^{\Hom(F, E)}$ such that $ \cO(w_0^G \lambda) / \varpi^m \cong  \cO(w_0^G \lambda') / \varpi^m$ as $\cO / \varpi^m[T_n(\cO_{F_p})(b)]$-modules, there is an $\T^{S, \ord}$-equivariant isomorphism 
	\begin{multline*} R \Gamma_{K(0,c)/K(b,c)}(X_{K(b, c)}, \cV_{\lambda} / \varpi^m)^{\ord} \\\cong  R \Gamma_{K(0,c)/K(b,c)}(X_{K(b, c)}, \cV_{\lambda'} / \varpi^m)^{\ord}\otimes_{\cO} \cO(w_0^G \lambda)\otimes_{\cO}\cO((w_0^G \lambda')^{-1})\end{multline*}
	in $\mathbf{D}(\cO / \varpi^m[ T_n(F_p)_b])$. 
\end{cor}
\begin{proof}
	Combine Propositions \ref{prop:comparison_of_completed_and_classical_ord_coh} and \ref{prop:independence_of_weight}.
\end{proof}

\subsubsection{Results for the group $\widetilde{G}$}

We recall that by assumption each $p$-adic place of $F^+$ splits in $F$, and that we have fixed for each place $\overline{v} \in \overline{S}_p$ of $F^+$ a lift $\widetilde{v} \in \widetilde{S}_p$ to a place of $F$. These choices determine an isomorphism 
\[ \prod_{\overline{v} \in \overline{S}_p} \iota_{\widetilde{v}} : \widetilde{G}(F^+_p) \cong \prod_{\overline{v} \in \overline{S}_p} \GL_{2n}(F_{\widetilde{v}}). \]
 We have also fixed a maximal torus and Borel subgroup $T \subset B \subset \widetilde{G}$ which correspond under this isomorphism to $T_{2n} \subset B_{2n} \subset \GL_{2n}$. The theory of \S \ref{sec_ordinary_part_of_smooth_rep} can thus be easily generalized to study the completed cohomology of $\widetilde{G}$. Since we will need to do this only in passing on our way to analyzing the complexes $\pi^{\ord}(K^p, \lambda, m)$, we just give some brief indications. We will use some of the Hecke operators and open compact subgroups defined in \S \ref{sec:some_useful_hecke_operators}. We define
 \[ \widetilde{\T}^{S, \ord} = \frac{\widetilde{\T}^S \otimes_{\cO} \cO \llbracket T(\cO_{F^+, p}) \rrbracket [ \{ \widetilde{U}_{v, 1}, \dots, \widetilde{U}_{v, 2n}, \widetilde{U}_{v, 2n}^{-1} \}_{v \in S_p} ] }{ ( \{ \widetilde{U}_{v^c, i} - \widetilde{U}_{v, 2n - i} \widetilde{U}_{v, 2n}^{-1} \}_{\substack{v \in S_p \\ i = 1, \dots, 2n}} )}. \]
 We define $\widetilde{U}_v = \widetilde{U}_{v, 1} \widetilde{U}_{v, 2} \cdots \widetilde{U}_{v, n-1}$ and $\widetilde{U}_p = \prod_{v \in S_p} \widetilde{U}_v \in \widetilde{\T}^{S, \ord}$. If $\widetilde{K} \subset \widetilde{G}(\A_{F^+}^\infty)$ is a good subgroup with $\widetilde{K}_{\overline{v}} = \widetilde{\Iw}_{\overline{v}}$ for each $\overline{v} \in \overline{S}_p$, and $c \geq b \geq 0$ are integers with $c \geq 1$, then we define $\widetilde{K}(b, c)$ to be the good subgroup with $\widetilde{K}(b, c)_{\overline{v}} = \widetilde{K}_{\overline{v}}$ if $\overline{v} \not\in \overline{S}_p$ and $\widetilde{K}(b, c)_{\overline{v}} = \widetilde{\Iw}_{\overline{v}}(b, c)$ otherwise. If $\widetilde{\lambda} \in (\Z^{2n}_+)^{\Hom(F^+, E)}$, then there is a well-defined direct summand $R \Gamma_{\widetilde{K}(0, c)/\widetilde{K}(b, c)}(\widetilde{X}_{\widetilde{K}(b, c)}, \cV_{\widetilde{\lambda}})^{\ord}$ of $R \Gamma_{\widetilde{K}(0, c)/\widetilde{K}(b, c)}(\widetilde{X}_{\widetilde{K}(b, c)}, \cV_{\widetilde{\lambda}})$ on which $\widetilde{U}_p$ acts invertibly, and we define
 \[ \widetilde{\T}(\widetilde{K}(b, c), \widetilde{\lambda})^{\ord} = \widetilde{\T}^{S, \ord}( R \Gamma_{\widetilde{K}(0, c)/\widetilde{K}(b, c)}(\widetilde{X}_{\widetilde{K}(b, c)}, \cV_{\widetilde{\lambda}})^{\ord} ) \]
 (i.e.\ the image of the $\widetilde{\T}^{S, \ord}$ in the endomorphism algebra in $\mathbf{D}(\cO[\widetilde{K}(0, c) / \widetilde{K}(b, c) ])$ of this direct summand). 
 
To compare Hida theory for $\widetilde{G}$ and for $\GL_n$, we recall that the Levi subgroup $G$ of $\widetilde{G}$ is identified with $\Res_{\cO_{F}/\cO_{F^+}}\GL_n$, which in particular identifies $T$ with $\Res_{\cO_{F}/\cO_{F^+}}T_n$. We extend the  homomorphism $\cS : \widetilde{\T}^S \to \T^S$ (defined by equation (\ref{eqn:satake})) to a homomorphism $\widetilde{\T}^{S, \ord} \to \T^{S, \ord}$, also denoted $\cS$, using the identification
 \[  \cO\llbracket T(\cO_{F^+, p}) \rrbracket \cong \cO \llbracket T_n(\cO_{F, p}) \rrbracket , \]
 and by sending each operator $\widetilde{U}_{v, i}$ to the operator $U_{v^c, n-i} U_{v^c, n}^{-1}$ (if $1 \leq i \leq n$) and $U_{v^c, n}^{-1} U_{v, i-n}$ (if $n+1 \leq i \leq 2n$). Note that these respective Hecke operators are double coset operators for elements of $T(F^+_p)$ and $T_n(F_p)$ which match under our identification $T(F^+_p) = T_n(F_p)$.
 
 We write $T(F^+_p)^+ \subset T(F^+_p)$ for the submonoid of elements which are contracting on $N(\cO_{F^+, p})$. Under our identification $T(F^+_p) = T_n(F_p)$, we have $T(F^+_p)^+ \subset T_n(F_p)^+$ (and the inclusion is strict provided $n \geq 2$). Let $\widetilde{\Iw}_p(b, c) = \prod_{\overline{v} \in \overline{S}_p} \widetilde{\Iw}_{\overline{v}}(b, c)$. We recall (\S \ref{sec:some_useful_hecke_operators}) that we have defined $\widetilde{\Delta}_p = \widetilde{\Iw}_p(b, c) T(F^+_p)^+ \widetilde{\Iw}_p(b, c)$, an open submonoid of $\widetilde{G}(F^+_p)$, and that we have defined an action $\cdot_p$ of this monoid on $\cV_{\widetilde{\lambda}}$. If $b \geq 0$ is an integer then we define $T(\cO_{F^+, p})(b) = T_n(\cO_{F, p})(b)$ and write $B(\cO_{F^+, p})(b)$ for the pre-image in $B(\cO_{F^+, p})$ of $T(\cO_{F^+, p})(b)$ under the natural projection to $T$. We define $B(F_p^+)^+ = N(\cO_{F^+, p}) \cdot T(F^+_p)^+$.
 
Fix $m \geq 1$. If $\widetilde{K} \subset \widetilde{G}(\A_{F^+}^\infty)$ is a good subgroup with $\widetilde{K}_{\overline{v}} = \widetilde{\Iw}_{\overline{v}}$ for each $\overline{v} \in \overline{S}_p$ and $\widetilde{\lambda} \in (\Z^{2n}_+)^{\Hom(F^+, E)}$, then we define
\numequation \widetilde{\pi}(\widetilde{K}^p, \widetilde{\lambda}, m) = R \Gamma_{\widetilde{K}^p, \text{sm}} R \Gamma(\mathfrak{X}_{\widetilde{G}}, \cV_{\widetilde{\lambda}} / \varpi^m) \in \mathbf{D}_{\text{sm}}(\cO / \varpi^m[\widetilde{\Delta}_p]). 
\end{equation}
If $\widetilde{K}^S = \widetilde{G}(\widehat{\cO}^S_{F^+})$, then this complex comes equipped with a homomorphism
\numequation\label{eqn:tildeG_completed_hecke_action} \widetilde{\T}^S \to \End_{\mathbf{D}_{\text{sm}}(\cO / \varpi^m[\widetilde{\Delta}_p])}(\widetilde{\pi}(\widetilde{K}^p, \lambda, m)). 
\end{equation}
We define $\widetilde{\pi}(\widetilde{K}^p, m) = R \Gamma_{\widetilde{K}^p, \text{sm}} R \Gamma(\mathfrak{X}_{\widetilde{G}}, \cO / \varpi^m) \in \mathbf{D}_{\text{sm}}(\cO / \varpi^m[\widetilde{G}(F^+_p)])$; this complex comes equipped with a homomorphism
\numequation  \widetilde{\T}^S \to \End_{\mathbf{D}_{\text{sm}}(\cO / \varpi^m[\widetilde{G}(F^+_p)])}(\widetilde{\pi}(\widetilde{K}^p, m))
\end{equation}
that recovers (\ref{eqn:tildeG_completed_hecke_action}) after applying the forgetful functor induced by the inclusion $\widetilde{\Delta}_p \subset \widetilde{G}(F^+_p)$. We also need the completed boundary cohomology. We thus define
\numequation \widetilde{\pi}_\partial(\widetilde{K}^p, \widetilde{\lambda}, m) = R \Gamma_{\widetilde{K}^p, \text{sm}} R \Gamma(\partial \mathfrak{X}_{\widetilde{G}}, \cV_{\widetilde{\lambda}} / \varpi^m) \in \mathbf{D}_{\text{sm}}(\cO / \varpi^m[\widetilde{\Delta}_p]). 
\end{equation}
This complex comes equipped with a homomorphism
\numequation\label{eqn:ptildeG_completed_hecke_action} \widetilde{\T}^S \to \End_{\mathbf{D}_{\text{sm}}(\cO / \varpi^m[\widetilde{\Delta}_p])}(\widetilde{\pi}_\partial(\widetilde{K}^p, \widetilde{\lambda}, m)). 
\end{equation}
We define $\widetilde{\pi}_{\partial}(\widetilde{K}^p, m) = R \Gamma_{\widetilde{K}^p, \text{sm}} R \Gamma(\partial \mathfrak{X}_{\widetilde{G}}, \cO / \varpi^m) \in \mathbf{D}_{\text{sm}}(\cO / \varpi^m[\widetilde{G}(F^+_p)])$; this complex comes equipped with a homomorphism
\numequation  \widetilde{\T}^S \to \End_{\mathbf{D}_{\text{sm}}(\cO / \varpi^m[\widetilde{G}(F^+_p)])}(\widetilde{\pi}_\partial(\widetilde{K}^p, m)).
\end{equation}
If $c \geq b \geq 0$ are integers with $c \geq 1$, then there are canonical $\widetilde{\T}^{S, \ord}$-equivariant isomorphisms
\numequation R \Gamma(\widetilde{\Iw}_p(b, c), \widetilde{\pi}(\widetilde{K}^p, \widetilde{\lambda}, m)) \cong R \Gamma(\widetilde{X}_{\widetilde{K}(b, c)}, \cV_{\widetilde{\lambda}} / \varpi^m) 
\end{equation}
and
\numequation R \Gamma(\widetilde{\Iw}_p(b, c), \widetilde{\pi}_\partial(\widetilde{K}^p, \widetilde{\lambda}, m)) \cong R \Gamma(\partial\widetilde{X}_{\widetilde{K}(b, c)}, \cV_{\widetilde{\lambda}} / \varpi^m) 
\end{equation}
in $\mathbf{D}(\cO / \varpi^m)$. We define the ordinary part of completed and completed boundary cohomology:
\[ \widetilde{\pi}^{\ord}(\widetilde{K}^p, \widetilde{\lambda}, m) = \ord R \Gamma(N(\cO_{F^+, p}), \widetilde{\pi}(\widetilde{K}^p, \widetilde{\lambda}, m)) \in \mathbf{D}_{\text{sm}}(\cO / \varpi^m[T(F^+_p)])  \]
and
\[  \widetilde{\pi}_\partial^{\ord}(\widetilde{K}^p, \widetilde{\lambda}, m) = \ord R \Gamma(N(\cO_{F^+, p}), \widetilde{\pi}_\partial(\widetilde{K}^p, \widetilde{\lambda}, m)) \in \mathbf{D}_{\text{sm}}(\cO / \varpi^m[T(F^+_p)]).  \]
If $\widetilde{\lambda} = 0$, then we omit it from the notation. We have the following result, which contains the analogues of some of the results in \S \ref{sec:ordinary_part_of_completed_coh} for the group $\widetilde{G}$. The proofs are entirely similar, so are omitted.
\begin{prop}\label{prop:independence_of_weight_tildeG}
	Let $\widetilde{K} \subset \widetilde{G}(\A_{F^+}^\infty)$ be a good subgroup with $\widetilde{K}_{\overline{v}} = \widetilde{\Iw}_{\overline{v}}$ for each $\overline{S}_p$ and $\widetilde{K}^S = \widetilde{G}(\widehat{\cO}^S_{F^+})$. Let $c \geq b \geq 0$ be integers with $c \geq 1$. Then for any $\widetilde{\lambda} \in (\Z^{2n}_+)^{\Hom(F^+, E)}$, there are $\widetilde{\T}^{S, \ord}$-equivariant isomorphisms
	\[ \begin{split} R \Gamma(T(\cO_{F^+, p})(b), \widetilde{\pi}^{\ord}(\widetilde{K}^p, \widetilde{\lambda}, m)) & \cong R \Gamma(T(\cO_{F^+, p})(b),  \cO(w_0^{\widetilde{G}} \widetilde{\lambda}) \otimes_{\cO} \widetilde{\pi}^{\ord}(\widetilde{K}^p, m)) \\ & \cong R \Gamma_{\widetilde{K}(0, c)/\widetilde{K}(b, c)}(\widetilde{X}_{\widetilde{K}(b, c)}, \cV_{\widetilde{\lambda}} / \varpi^m)^{\ord} \end{split} \]
	and
	\[ \begin{split} R \Gamma(T(\cO_{F^+, p})(b), \widetilde{\pi}_\partial^{\ord}(\widetilde{K}^p, \widetilde{\lambda}, m)) & \cong R \Gamma(T(\cO_{F^+, p})(b), \cO(w_0^{\widetilde{G}} \widetilde{\lambda}) \otimes_{\cO} \widetilde{\pi}_\partial^{\ord}(\widetilde{K}^p, m)) \\ & \cong R \Gamma_{\widetilde{K}(0, c)/\widetilde{K}(b, c)}(\partial \widetilde{X}_{\widetilde{K}(b, c)}, \cV_{\widetilde{\lambda}} / \varpi^m)^{\ord} \end{split} \]
	in $\mathbf{D}_{\text{sm}}(\cO / \varpi^m[\widetilde{K}(0, c)/\widetilde{K}(b, c)])$.
\end{prop}

\subsection{The ordinary part of a parabolic induction}\label{computation of ordinary parts} 

In this subsection, we compute the ordinary part (in the sense defined above) of a parabolic induction from $G$ to $\widetilde{G}$, with the aim of understanding the ordinary part of the cohomology of the boundary of the Borel--Serre compactification of $\widetilde{X}_{\widetilde{K}}$ in terms of the ordinary part of the cohomology of $X_K$. Our calculations here are purely local; the global application will be carried out in \S \ref{sec:ord_direct_summand} below. 

Let $\overline{v}$ be a $p$-adic place of $F^+$. In this section, we write ${}^r W_{\overline{v}} = W(\widetilde{G}_{F^+_{\overline{v}}}, T_{F^+_{\overline{v}}})$, ${}^rW_{P, \overline{v}} = W(G_{F^+_{\overline{v}}}, T_{F^+_{\overline{v}}})$, and ${}^r W^P_{\overline{v}} \subset {}^r W_{\overline{v}}$ for the set of representatives for the quotient ${}^r W_{P, \overline{v}} \backslash {}^r W_{\overline{v}}$ which is associated to the choice of Borel subgroup $B_{F^+_{\overline{v}}}$. We define ${}^r W = \prod_{\overline{v} \in \overline{S}_p} {}^rW_{\overline{v}}$, ${}^rW_P = \prod_{\overline{v} \in \overline{S}_p} {}^rW_{P,\overline{v}}$, and ${}^rW^P = \prod_{\overline{v} \in \overline{S}_p} {}^rW^P_{\overline{v}}$. Thus ${}^r W$ is the relative Weyl group of the group $(\Res_{F^+ / \Q} G)_{\Q_p}$.  Note that in \S \ref{section:lep} we made use of the absolute Weyl group $W$; there is a natural inclusion ${}^r W \subset W$, by which ${}^r W$ acts on e.g.\ the group $X^\ast((\Res_{F^+ / \Q} T)_E)$. We write $l_r(w)$ for the length of an element $w \in {}^r W$ as an element of the relative Weyl group, and $l(w)$ for its length as an element of the absolute Weyl group. Thus $w_0^P$, the longest element of $W^P$ (equivalently, of ${}^r W^P$) has $l_r(w_0^P) = | S_p | n^2$ and $l(w_0^P) = [F^+ : \Q] n^2$. As in \S \ref{section:lep}, we write $\rho \in X^\ast((\Res_{F^+ / \Q} T)_E)$ for the half-sum of the $(\Res_{F^+ / \Q} B)_E$-positive roots. 

We recall (cf. \S \ref{sec:unitary_group_setup}) that $P$ denotes the Siegel parabolic of $\widetilde{G}$, which has unipotent radical $U$, while the Borel subgroup $B$ has unipotent radical $N$. 
We identify $G$ with~$\Res_{F / F^+} \GL_n$; this group has standard Borel~$\Res_{F / F^+}B_n$ with unipotent radical~$\Res_{F / F^+} N_n$. The parabolic induction functor 
\[ \Ind_{P(F_p^+)}^{\widetilde{G}(F_p^+)} : \operatorname{Mod}_{\text{sm}}(\cO / \varpi^m[P(F_p^+)]) \to \operatorname{Mod}_{\text{sm}}(\cO / \varpi^m[\widetilde{G}(F_p^+)]) \]
is exact and preserves injectives (it is right adjoint to the exact restriction functor $\Res^{\widetilde{G}(F_p^+)}_{P(F_p^+)}$). We now define several more functors which are related to parabolic induction. 

We identify ${}^r W$ with the subgroup of permutation matrices of $\widetilde{G}(F^+_p) = \prod_{\widetilde{v} \in \widetilde{S}_p} \GL_{2n}(F_{\widetilde{v}})$. We recall (cf.~\cite[Cor. 5.20]{borel-tits}) that there is a (set-theoretic) decomposition
\[ \widetilde{G}(F_p^+) = \bigsqcup_{w \in  {}^r W^P} P(F_p^+) w B(F_p^+).  \]
 If $w \in {}^r W^P$, then we define $S_w = P(F_p^+) w N(F_p^+)$ and $S_w^\circ = P(F_p^+) w N(\cO_{F^+, p}) \subset S_w$. The closure $\overline{S}_w$ of $S_w$ in $\widetilde{G}(F_p^+)$ can be described in terms of the Bruhat ordering of ${}^rW^P$:
\[ \overline{S}_w=\bigsqcup_{w'\leq w} S_{w'}. \]
Note that if $w' < w$, then $l_r(w') < l_r(w)$. For an integer $i \geq 0$, we define
\[ \widetilde{G}_{\geq i} = \bigsqcup_{\substack{w \in {}^r W^P \\ l_r(w) \geq i}} S_w. \]
It is an open subset of $\widetilde{G}(F_p^+)$ which is invariant under left multiplication by $P(F_p^+)$ and right multiplication by $B(F_p^+)$.

If $i \geq 0$, then we define a functor
\[ I_{\geq i} : \operatorname{Mod}_{\text{sm}}(\cO / \varpi^m[P(F_p^+)]) \to \operatorname{Mod}_{\text{sm}}(\cO/ \varpi^m[B(F_p^+)]) \]
by sending $\pi$ to 
\[ \begin{split} I_{\geq i}(\pi) = \{ f : \widetilde{G}_{\geq i} \to \pi \mid f  \text{ locally constant,}  &\text{ of compact support modulo }P(F_p^+),\\ & \forall p \in P(F_p^+), g \in \widetilde{G}_{\geq i}, f(pg) = p f(g)  \}, \end{split} \]
where $B(F_p^+)$ acts by right translation. If $w \in {}^r W^P$, then we define a functor
\[ I_w : \operatorname{Mod}_{\text{sm}}(\cO / \varpi^m[P(F_p^+)]) \to \operatorname{Mod}_{\text{sm}}(\cO / \varpi^m[B(F_p^+)]) \]
by sending $\pi$ to 
\[ \begin{split} I_{w}(\pi) = \{ f : S_w \to \pi \mid f \text{ locally constant,} & \text{ of compact support modulo }P(F_p^+),\\ & \forall p \in P(F_p^+), g \in S_w, f(pg) = p f(g)   \}, \end{split} \]
where again $B(F_p^+)$ acts by right translation. We define a functor
\[ I_w^\circ : \operatorname{Mod}_{\text{sm}}(\cO / \varpi^m[P(F_p^+)]) \to \operatorname{Mod}_{\text{sm}}(\cO / \varpi^m[B(F^+_p)^+]) \]
by defining $I_w^\circ(\pi) \subset I_w(\pi)$ to be the set of functions with support in $S_w^\circ$.
\begin{prop}
	\begin{enumerate}
		\item $I_{\geq 0} = \Res_{B(F_p^+)}^{\widetilde{G}(F_p^+)} \circ \Ind_{P(F_p^+)}^{\widetilde{G}(F_p^+)}$.
		\item Each functor $I_{\geq i}$, $I_w$ and $I_w^\circ$ is exact. 
		\item For each integer $i \geq 0$ and each $\pi \in \operatorname{Mod}_{\text{sm}}(\cO / \varpi^m[P(F_p^+)])$, there is a functorial exact sequence
		\[0 \to I_{\geq i+1}(\pi) \to I_{\geq i}(\pi) \to \oplus_{\substack{w \in {}^r W^P \\ l_r(w) = i}} I_w(\pi) \to 0. 
		\]
	\end{enumerate}
\end{prop}
\begin{proof}
	The first part is the definition of induction. For the second part, denote by $I$ any of the functors appearing in the statement. To see the exactness  of $I$, choose a continuous section to the map $\widetilde{G}(F_p^+) \to P(F_p^+)\backslash \widetilde{G}(F_p^+)$ (the existence of such a section is explained in \cite[\S 2.1]{hauseux}). This allows us to functorially identify $I(\pi)$ with the space of locally constant and compactly supported functions from a subset $C \subset P(F_p^+)\backslash \widetilde{G}(F_p^+)$ to $\pi$. The formation of locally constant and compactly supported functions is exact. The third part is proved in the same way as \cite[Proposition 2.1.3]{hauseux}.
\end{proof}
It follows that for any $\pi \in \mathbf{D}_{\text{sm}}(\cO / \varpi^m[P(F_p^+)])$, there is a functorial distinguished triangle
	\numequation\label{eqn:schubert_cells_derived}  I_{\geq i+1}(\pi) \to I_{\geq i}(\pi) \to \oplus_{\substack{w \in {}^rW^P \\ l_r(w) = i}} I_w(\pi) \to  I_{\geq i+1}(\pi)[1]
\end{equation}
in $\mathbf{D}_{\text{sm}}(\cO / \varpi^m[B(F_p^+)])$. 
\begin{lemma}\label{lem:exactness_of_degree_j-ordinary_part}
	Let $\pi \in \mathbf{D}_{\text{sm}}(\cO / \varpi^m[P(F_p^+)])$ be a bounded below complex, and fix an integer $b \geq 0$. Let $\widetilde{\lambda} \in (\Z_+^{2n})^{\Hom(F^+, E)}$. Then for any $i \geq 0$ and any $j \in \Z$, the sequence
	\[ \begin{split} 0 & \to R^j \Gamma(B(\cO_{F, p})(b), \cO(w_0^{\widetilde{G}} \widetilde{\lambda}) \otimes_\cO I_{\geq i+1}(\pi)) \\ & \to R^j \Gamma(B(\cO_{F, p})(b),\cO(w_0^{\widetilde{G}} \widetilde{\lambda}) \otimes_\cO I_{\geq i}(\pi)) \\ & \to R^j \Gamma(B(\cO_{F, p})(b), \oplus_{\substack{w \in {}^rW^P\\ l_r(w) = i}}\cO(w_0^{\widetilde{G}} \widetilde{\lambda}) \otimes_\cO I_w(\pi)) \to 0 \end{split} \]
	in $\operatorname{Mod}( \cO / \varpi^m[T(F^+_p)_b^+])$ associated to (\ref{eqn:schubert_cells_derived}) is exact. 
\end{lemma}
\begin{proof}
	It suffices to show exactness after applying the exact forgetful functor to $\operatorname{Mod}( \cO / \varpi^m)$. We consider decompositions $\widetilde{G}_{\geq i} = U_1 \sqcup U_2$ where $U_1, U_2$ are open sets which are invariant under left multiplication by $P(F^+_p)$ and right multiplication by $B(\cO_{F^+, p})$, and such that $U_1 \subset \widetilde{G}_{\geq i+1}$. Any such decomposition determines a functorial decomposition $I_{\geq i}(\pi) = I_{U_1}(\pi) \oplus I_{U_2}(\pi)$, where $I_{U_1}$ denotes functions with support in $U_1$, and similarly for $U_2$. This decomposition exists in the category $\operatorname{Mod}_{\text{sm}}(\cO / \varpi^m[B(\cO_{F^+, p})])$. We see in particular that for any bounded below complex $\pi \in \mathbf{D}_{\text{sm}}(\cO / \varpi^m[P(F_p^+)])$, the associated morphism
	\[ R^j \Gamma(B(\cO_{F, p})(b),\cO(w_0^{\widetilde{G}} \widetilde{\lambda}) \otimes_\cO I_{U_1}(\pi)) \to R^j \Gamma(B(\cO_{F, p})(b),\cO(w_0^{\widetilde{G}} \widetilde{\lambda}) \otimes_\cO I_{\geq i}(\pi)) \]
	in $\operatorname{Mod}( \cO / \varpi^m)$ is injective. Since $I_{\geq i+1}$ is the filtered direct limit of the $I_{U_1}$ (which can be proven by following the same technique as in the proof of \cite[Prop. 2.2.3]{hauseux}), it follows that the morphism
	\[ R^j \Gamma(B(\cO_{F, p})(b), \cO(w_0^{\widetilde{G}} \widetilde{\lambda}) \otimes_\cO I_{\geq i+1}(\pi)) \to R^j \Gamma(B(\cO_{F, p})(b),\cO(w_0^{\widetilde{G}} \widetilde{\lambda}) \otimes_\cO I_{\geq i}(\pi)) \]
	is injective. Since this applies for any $j \in \Z$, the exactness of the long exact sequence in cohomology attached to the distinguished triangle (\ref{eqn:schubert_cells_derived}) implies that the sequence in the statement of the lemma is indeed a short exact sequence. 
\end{proof}
\begin{lemma}\label{lem:compact_schubert_cell_preserves_acyclics}
	Let $w \in {}^r W^P$. Then:
	\begin{enumerate}
		\item 	 $I_w^\circ$ takes injectives to $\Gamma(N(\cO_{F^+, p}), -)$-acyclics. 
		\item Let $\pi \in \mathbf{D}_{\text{sm}}(\cO / \varpi^m[P(F_p^+)])$ be a bounded below complex. Then there is a natural isomorphism 
		\[ \ord R \Gamma(N(\cO_{F^+, p}), I_w^\circ(\pi)) \cong  \ord R \Gamma(N(\cO_{F^+, p}), I_w(\pi)). \]
	\end{enumerate}
\end{lemma}
\begin{proof}
	For the first part, let $\pi \in \operatorname{Mod}_{\text{sm}}(\cO / \varpi^m[P(F^+_p)])$, and fix an $\cO / \varpi^m$-embedding $\pi \hookrightarrow I$, where $I$ is an injective $\cO / \varpi^m$-module. Then there is an embedding $\pi \hookrightarrow \Ind_1^{P(F_p^+)} I$ of $\cO / \varpi^m[P(F^+_p)]$-modules. We will show that $I_w^\circ(\Ind_1^{P(F_p^+)} I)$ is an injective smooth $\cO / \varpi^m[N(\cO_{F^+, p})]$-module. By \cite[Lem. 2.1.10]{emordtwo}, this will show the first part of the lemma.
	
	Let $\cC^\infty(P(F^+_p) w N(\cO_{F^+, p}), I)$ denote the set of locally constant functions $F : P(F^+_p) w N(\cO_{F^+, p}) \to I$. It is an injective smooth $\cO / \varpi^m[N(\cO_{F^+, p})]$-module when $N(\cO_{F^+, p})$ acts by right translation. There is a natural isomorphism
	\[ I_w^\circ(\Ind_1^{P(F_p^+)} I) \cong \cC^\infty(P(F^+_p) w N(\cO_{F^+, p}), I), \]
	which sends a function $f : P(F^+_p) w N(\cO_{F^+_p}) \to \Ind_1^{P(F_p^+)} I$ to the function $F : P(F^+_p) w N(\cO_{F^+_p}) \to I$ given by the formula $F(x) = f(x)(1)$. This proves the first part of the lemma.
	
	For the second part, we note that we may define an exact functor 
	\[ J_w : \operatorname{Mod}_{\text{sm}}(\cO / \varpi^m [P(F_p^+)]) \to \operatorname{Mod}_{\text{sm}}(\cO / \varpi^m[B(F_p^+)^+]) \]
	by the formula $J_w(\pi) = I_w(\pi) / I_w^\circ(\pi)$. Then for a bounded below complex $\pi \in \mathbf{D}_{\text{sm}}(\cO / \varpi^m[P(F_p^+)])$ there is a natural distinguished triangle
	\[ \begin{split} \ord R \Gamma(N(\cO_{F^+, p}), I_w^\circ(\pi)) \to \ord R \Gamma(N(\cO_{F^+, p}), I_w(\pi)) & \to \ord R \Gamma(N(\cO_{F^+, p}), J_w(\pi)) \\ & \to \ord R \Gamma(N(\cO_{F^+, p}), I_w^\circ(\pi))[1]. \end{split} \]
	To prove the desired result, it is therefore enough to show that 
	\[ \ord R \Gamma(N(\cO_{F^+, p}), J_w(\pi)) = 0. \] 
	It is even enough to show that for any $\pi \in \operatorname{Mod}_{\text{sm}}(\cO / \varpi^m[P(F_p^+)])$ and for any $j \in \Z$, we have $\ord H^j(N(\cO_{F^+, p}), J_w(\pi)) = 0$, and this can be proved in the same way as \cite[Lemme 3.3.1]{hauseux}. Indeed, it suffices to choose an element $t \in T(F_p^+)^+$, as in \cite[Lemme 3.1.3]{hauseux}, such that $S_w = \cup_{k \ge 0} t^{-k}S_w^\circ t^k$.  It follows that $t$ acts locally nilpotently on $J_w(\pi)$, and consequently that each element of $H^i(N(\cO_{F^+,p}),J_w(\pi))$ is annihilated by the Hecke action of a sufficiently high power of $t$.
\end{proof}
If $w \in {}^r W^P$, we define $N_w = P(F^+_p) \cap w N(\cO_{F^+, p}) w^{-1}$. It is a compact subgroup of $P(F^+_p)$ which contains $N_n(\cO_{F, p})$. We define a functor
\[ \Gamma(N_w, -  ) :  \operatorname{Mod}_{\text{sm}}(\cO / \varpi^m[P(F_{p}^+)]) \to \operatorname{Mod}_{\text{sm}}(\cO / \varpi^m[T(F_p^+)^+]), \]
where an element $t \in T(F^+_p)^+$ acts by the formula  $t \cdot v = \tr_{t^w N_w (t^w)^{-1} / N_w}(t^w v)$ ($t \in T(F^+_p)^+$). Note that this makes sense because $t^w N_w (t^w)^{-1} = P(F^+_p) \cap w t N(\cO_{F^+, p}) t^{-1} w^{-1} \subset N_w$. Note as well that $w T(F_p^+)^+ w^{-1} \subset T_n(F_p)^+$ (by definition of ${}^r W^P$).
\begin{lemma}\label{lem:w_induced_ordinary_parts}
	Let $w \in {}^r W^P$ and let $\pi \in \mathbf{D}_{\text{sm}}(\cO / \varpi^m[P(F_p^+)])$ be a bounded below complex. Then there is a natural isomorphism 
		\[ R \Gamma(N(\cO_{F^+, p}), I_w^\circ(\pi)) \cong  R \Gamma(N_w, \pi). \]
\end{lemma} 
\begin{proof}
	By the first part of Lemma \ref{lem:compact_schubert_cell_preserves_acyclics}, it's enough to show that there is a natural isomorphism of underived functors
	\[ \Gamma(N(\cO_{F^+, p}), I_w^\circ(-)) \cong \Gamma(N_w, -  ). \]
	The map sends an $N(\cO_{F^+, p})$-invariant function $f : P(F^+_p) w N(\cO_{F^+, p}) \to \pi$ to the value $f(w) \in \pi^{N_w}$. It is easy to see that this is an isomorphism of $\cO / \varpi^m$-modules; what we need to check is that it is equivariant for the action of $T(F^+_p)^+$. In other words, we need to check that for any $f \in \Gamma(N(\cO_{F^+, p}), I_w^\circ(\pi))$, we have
	\numequation\label{eqn:w_restriction_is_hecke_equivariant} \sum_{n \in N(\cO_{F^+, p}) / t N(\cO_{F^+, p}) t^{-1}} f(wnt) = \sum_{m \in N_w / t^w N_w (t^w)^{-1}} m w t w^{-1} f(w). 
	\end{equation}
	Conjugation by $w^{-1}$ determines a map $N_w / t^w N_w
        (t^w)^{-1} \to N(\cO_{F^+, p}) / t N(\cO_{F^+, p}) t^{-1}$,
        which is easily seen to be injective. On the other hand, if $n
        \in N(\cO_{F^+, p})$ and $f(w n t) \neq 0$, then the class of
        $n$ is in the image of this map; indeed, $f(w n t)$ can be
        non-zero only if $w n t \in P(F^+_p) w N(\cO_{F^+, p})$, in
        which case we  write $w n t = q w m$, with $q\in P(F^+_p)$
        and $m\in N(\cO_{F^+, p})$, hence $n = w^{-1} q w t^{-1} t m t^{-1}$. As $ w^{-1} q w t^{-1} \in w^{-1} P(F^+_p) w \cap N(\cO_{F^+, p})$ this shows that $n$ is in the image of this map. It follows that we can rewrite the left-hand side of (\ref{eqn:w_restriction_is_hecke_equivariant}) as
	\[ \sum_{m \in N_w / t^w N_w (t^w)^{-1}} f(m w t) =  \sum_{m \in N_w / t^w N_w (t^w)^{-1}} m w t w^{-1} f(w), \]
	which equals the right-hand side of (\ref{eqn:w_restriction_is_hecke_equivariant}).
\end{proof}
For the statement of the next lemma, we define, for any $w \in {}^r W^P$, a character $\chi_w : T(F^+_p) \to \cO^\times$ by the formula
\[ \chi_w(t) = \frac{\NN_{F^+_p / \Q_p} \det_{F^+_p} (\operatorname{Ad}(t^w)|_{\Lie U(F^+_p) \cap w N(F^+_p) w^{-1}})^{-1}}{ | \NN_{F^+_p / \Q_p} \det_{F^+_p} (\operatorname{Ad}(t^w)|_{\Lie U(F^+_p) \cap w N(F^+_p) w^{-1}}) |_p}. \]
Note that there is an isomorphism $\cO(\chi_w) \cong \cO(-\rho + w^{-1}w_0^P(\rho)) \otimes_{\cO} \cO(\alpha_w)$ of $\cO[T(F^+_p)]$-modules, where $w_0^P = w_0^G w_0^{\widetilde{G}}$ is the longest element of ${}^r W^P$, and where $\alpha_w : T(F^+_p) \to \cO^\times$ is the character which is trivial on $T(\cO_{F^+, p})$ and which satisfies the identity $\alpha_w(t) = \chi_w(t)$ for any element of the form $t = \iota_v^{-1}(\diag(\varpi_v^{a_1}, \dots, \varpi_v^{a_{2n}}))$ ($a_i \in \Z$). We also write 
\[ \tau_w : \operatorname{Mod}_{\text{sm}}(\cO / \varpi^m[T_n(F_p)]) \to \operatorname{Mod}_{\text{sm}}(\cO / \varpi^m[T_n(F_p)]) \]
for the functor which sends a module $\pi$ to $\tau_w(\pi) = \pi$, with action $\tau_w(\pi)(t)(v) = \pi(t^{w^{-1}})(v)$. 
\begin{lemma}\label{lem:shifted_ordinary_parts}
		Let $w \in {}^r W^P$ and let $\pi \in \mathbf{D}_{\text{sm}}(\cO / \varpi^m[G(F_p^+)])$ be a bounded below complex. Then there is a natural isomorphism between the following two complexes in $\mathbf{D}_{\text{sm}}(\cO / \varpi^m[T_n(F_p)])$:
		\[ \ord R \Gamma( N_w, \operatorname{Inf}_{G(F_p^+)}^{P(F_p^+)} \pi) \]
		and
		\[ \cO / \varpi^m(\chi_w) \otimes_{\cO / \varpi^m}\tau_w^{-1} \ord R \Gamma(N_n(\cO_{F, p}), \pi)[-[F^+ : \Q]n^2 + l(w)]. \]
\end{lemma}
\begin{proof}
	Let $N_w \rtimes_w T(F_p^+)^+$ denote the monoid $N_w \times T(F_p^+)^+$, equipped with multiplication $(t^w n (t^w)^{-1}, 1)(1, t) = (1, t) ( n, 1)$ (where the product $t^w n (t^w)^{-1}$ is formed using the usual multiplication of the group $\widetilde{G}(F^+_p)$). Let $N_{w, U} = N_w \cap U(F^+_p)$. Then there is a short exact sequence
	\[ 0 \to N_{w, U} \to N_w \to N_n(\cO_{F, p}) \to 0 \]
	which is equivariant for the conjugation action of $T(F_p^+)^+$ via the map $T(F_p^+)^+ \to T_n(F_p)^+$, $t \mapsto t^w$.  
	We consider the diagram, commutative up to natural isomorphism:
	\[ \xymatrix{ \operatorname{Mod}_{\text{sm}}(\cO / \varpi^m [P(F^+_p)]) \ar[d]_{\Res^w} \\
	\operatorname{Mod}_{\text{sm}}(\cO / \varpi^m[ N_w \rtimes_w T(F_p^+)^+ ]) \ar[d]_{\Gamma_{N_{w, U}}} \\
	\operatorname{Mod}_{\text{sm}}(\cO / \varpi^m[N_n(\cO_{F, p}) \rtimes_w T(F_p^+)^+]) \ar[r]^\alpha\ar[d]_{\Gamma_{N_n(\cO_{F, p})}} & \operatorname{Mod}_{\text{sm}}(\cO / \varpi^m[N_n(\cO_{F, p}) \rtimes T_n(F_p)^+]) \ar[d]^{\Gamma_{N_n(\cO_{F, p})}} \\
	\operatorname{Mod}_{\text{sm}}(\cO / \varpi^m[T(F_p^+)^+]) \ar[r]^\beta \ar[rd]^{\tau_w \circ \ord} & \operatorname{Mod}_{\text{sm}}(\cO / \varpi^m[T_n(F_p)^+]) \ar[d]^\ord \\
	& \operatorname{Mod}_{\text{sm}}(\cO / \varpi^m[T_n(F_p)]). } \]
	In this diagram we have abbreviated e.g.\ $\Gamma(N_{w, U}, -) = \Gamma_{N_{w, U}}$. We also abbreviate $\operatorname{Inf}_G^P = \operatorname{Inf}_{G(F_p^+)}^{P(F_p^+)}$. The torus action on e.g.\ $\Gamma_{N_{w, U}}$ is defined in the usual way (cf.\ \cite[\S 3.2]{hauseux}). The exact functor $\Res^w$ is defined by taking $\Res^w(\pi) = \pi$ as an $\cO / \varpi^m$-module, with $\Res^w(\pi)(nt)(v) = \pi(nt^w)(v)$. We also use $\Res^w$ to denote the functor $\Res^w \circ \operatorname{Inf}_G^P$. The $\alpha$ is the composite of the equivalence
	\[ \operatorname{Mod}_{\text{sm}}(\cO / \varpi^m[N_n(\cO_{F, p}) \rtimes_w T(F_p^+)^+]) \to  \operatorname{Mod}_{\text{sm}}(\cO / \varpi^m[N_n(\cO_{F, p}) \rtimes wT(F^+_p)^+ w^{-1}]) \]
	induced by the map $ n t \in N_n(\cO_{F, p}) \rtimes wT(F^+_p)^+w^{-1} \mapsto (n, t^{w^{-1}}) \in N_n(\cO_{F, p}) \rtimes_w T(F_p^+)^+$ with the localization
	\[ \operatorname{Mod}_{\text{sm}}(\cO / \varpi^m[N_n(\cO_{F, p}) \rtimes wT(F^+_p)^+ w^{-1}]) \to \operatorname{Mod}_{\text{sm}}(\cO / \varpi^m[N_n(\cO_{F, p}) \rtimes T_n(F_p)^+]) \]
	induced by the inclusion $wT(F^+_p)^+ w^{-1} \subset T_n(F_p)^+$.
	Similarly, the functor $\beta$ is the composite of the equivalence
	\[ \operatorname{Mod}_{\text{sm}}(\cO / \varpi^m[T(F_p^+)^+]) \to \operatorname{Mod}_{\text{sm}}(\cO / \varpi^m[wT(F_p^+)^+w^{-1}])  \]
	with the localization
	\[  \operatorname{Mod}_{\text{sm}}(\cO / \varpi^m[wT(F_p^+)^+w^{-1}]) \to\operatorname{Mod}_{\text{sm}}(\cO / \varpi^m[T_n(F_p)^+]).   \]
	Note that $\alpha$ takes injectives to $\Gamma_{N_n(\cO_{F, p})}$-acyclics; this can be deduced from \cite[Prop. 2.1.3]{emordtwo}, using the compactness of $N_n(\cO_{F, p})$ and the observation that this localization can be thought of as a direct limit. Note that the composite of all left vertical arrows is the functor $\Gamma_{N_w}$. 
	
	Let $\pi$ now be as in the statement of the lemma. We compute
	\[ \begin{split} \ord R \Gamma ( N_w, \operatorname{Inf}_G^P \pi) & = \ord \beta R \Gamma_{N_n(\cO_{F, p})} R \Gamma_{N_{w, U}} \Res^w \operatorname{Inf}_G^P \pi \\
	& = \ord R \Gamma_{N_n(\cO_{F, p})} \alpha R \Gamma_{N_{w, U}} \Res^w \operatorname{Inf}_G^P \pi. \end{split} \]
	Since $U(F^+_p)$ acts trivially on $\pi$, there is an isomorphism 
	\[ R \Gamma_{N_{w, U}} \Res^w \operatorname{Inf}_G^P \pi \cong \Res^w(\pi) \otimes_{\cO / \varpi^m} R \Gamma(N_{w, U}, \cO / \varpi^m) \]
	in $\mathbf{D}_{\text{sm}}(\cO / \varpi^m[N_n(\cO_{F, p}) \rtimes_w T(F_p^+)^+])$. To go further, we need to compute the complex $\alpha R \Gamma(N_{w, U}, \cO / \varpi^m)$. To this end, we consider the action of the element
	\[ z_p = \diag(p, \dots, p, 1, \dots, 1) \in T(F_p^+)^+ \]
	(where there are $n$ entries equal to $p$ and $n$ entries equal to $1$; note that this element depends on our choice of set $\widetilde{S}_p$, which determines the identification of $\widetilde{G}(F^+_p)$ with $\prod_{\overline{v} \in \overline{S}_p} \GL_{2n}(F_{\overline{v}}^+)$). It is in the centre of $G(F^+_p)$, and is therefore invertible in $T_n(F_p)^+$. Its action on the cohomology groups $H^i(N_{w, U}, \cO / \varpi^m)$ is the one induced by its natural conjugation action on $N_{w, U}$; in other words, multiplication by $p$ on this abelian group. The group $N_{w, U}$ has rank $n^2 [ F^+ : \Q] - l(w)$ as $\Z_p$-module, from which it follows that the Hecke action of $z_p$ on $H^i(N_{w, U}, \cO / \varpi^m)$ factors through multiplication by $p^{n^2 [F^+ : \Q] - l(w) - i}$ ($0 \leq i \leq n^2 [ F^+ : \Q] - l(w)$). The cohomology groups below the top degree $i = n^2 [F^+ : \Q] - l(w)$ therefore vanish after applying the functor $\alpha$, and it follows from \cite[Prop. 3.1.8]{hauseux} that 
	\[  \begin{split} \alpha R \Gamma(N_{w, U}, \cO / \varpi^m) & \cong \cO / \varpi^m((\chi_w)^w) [ -[F^+ : \Q]n^2 + l(w) ] \\ & \cong \alpha \cO / \varpi^m(\chi_w) [ -[F^+ : \Q]n^2 + l(w) ], \end{split} \]
	hence that
	\[ \beta R \Gamma_{N_{w}} \Res^w \operatorname{Inf}_G^P \pi  \cong \beta R \Gamma_{N_n(\cO_{F, p})} \Res^w \pi \otimes_{\cO / \varpi^m}  \cO / \varpi^m(\chi_w) [ -[F^+ : \Q]n^2 + l(w) ]. \]
		We finally see that $\ord R \Gamma (N_w, \operatorname{Inf}_G^P \pi)$ is isomorphic to 
		\[ \begin{split}   & \tau_w^{-1} \ord \beta R \Gamma_{N_n(\cO_{F, p})} \Res^w \pi \otimes_{\cO / \varpi^m}  \cO / \varpi^m(\chi_w) [ -[F^+ : \Q]n^2 + l(w) ] \\
			& \cong  \cO / \varpi^m(\chi_w)\otimes_{\cO / \varpi^m} \tau_w^{-1} \ord R \Gamma_{N_n(\cO_{F, p})} \pi  [ -[F^+ : \Q]n^2 + l(w) ].\qedhere
			\end{split} \]
\end{proof}
\begin{prop}\label{prop:ordinary_part_of_w_induction}
	Let $w \in {}^r W^P$ and let $\pi \in \mathbf{D}_{\text{sm}}(\cO / \varpi^m[G(F^+_{p})])$ be a bounded below complex. Then there is a natural isomorphism between the following two complexes in $\mathbf{D}_{\text{sm}}(\cO/ \varpi^m[ T_n(F_p)])$:
	\[ \ord R \Gamma(N(\cO_{F^+, p}), I_w(\operatorname{Inf}_{G(F_p^+)}^{P(F_p^+)} \pi)) \]
	and
	\[ \cO / \varpi^m(\chi_w) \otimes_{\cO / \varpi^m} \tau_{w^{-1}} \ord R \Gamma(N_n(\cO_{F, p}), \pi)[-[F^+ : \Q]n^2 + l(w)]. \]
\end{prop}
\begin{proof}
	This follows on combining Lemma \ref{lem:compact_schubert_cell_preserves_acyclics}, Lemma \ref{lem:w_induced_ordinary_parts}, and Lemma \ref{lem:shifted_ordinary_parts}. 
\end{proof}

\subsection{The degree shifting argument}\label{sec:ord_direct_summand}

In this section, we give the analogue for completed cohomology of the results of \S \ref{sec:direct summand at finite level}, by relating the completed cohomology of $X$ to the completed cohomology of the boundary $\partial \widetilde{X}$. The statement is simpler for completed cohomology than for cohomology at finite level because the contribution of the unipotent radical of the Siegel parabolic vanishes in the limit. 
\begin{thm}\label{thm:direct_summand_of_completed_boundary_cohomology}
	Let $\widetilde{K} \subset \widetilde{G}(\A_{F^+}^\infty)$ be a good subgroup which is decomposed with respect to $P$. Let $\m \subset \T^S$ be a non-Eisenstein maximal ideal, and let $\widetilde{\m} = \cS^\ast(\m) \subset \widetilde{\T}^S$. Then the complex $\Ind_{P(F_p^+)}^{\widetilde{G}(F^+_p)} \pi(K^p, m)_{\m}$ is a $\widetilde{\T}^S$-equivariant direct summand of the complex $\widetilde{\pi}_\partial(\widetilde{K}^p, m)_{\widetilde{\m}}$ 
	 in $\mathbf{D}_{\text{sm}}(\cO / \varpi^m[\widetilde{G}(F^+_p)])$.
\end{thm}
Here and below we have written $\pi(K^p, m)_{\m}$ for the complex previously denoted $\operatorname{Inf}_{G(F_p^+)}^{P(F_p^+)}\pi(K^p, m)_{\m}$ in order to lighten the notation.
\begin{proof}
	We first show that there is a $\widetilde{\T}^S$-equivariant isomorphism
	\[ ( R \Gamma_{\widetilde{K}^p, \text{sm}} R \Gamma( \Ind_{P^\infty}^{\widetilde{G}^\infty} \mathfrak{X}_P, \cO / \varpi^m) )_{\widetilde{\m}} \cong ( R \Gamma_{\widetilde{K}^p, \text{sm}} R \Gamma( \partial \mathfrak{X}_{\widetilde{G}}, \cO / \varpi^m) )_{\widetilde{\m}} = \widetilde{\pi}_\partial(\widetilde{K}^p, m)_{\widetilde{\m}} . \]
	As in the proof of Theorem \ref{thm:reduction_to_Siegel}, it suffices to show that for each standard proper parabolic subgroup $Q \subset \widetilde{G}$ with $Q \neq P$, we have
	\[ H^\ast(  R \Gamma_{\widetilde{K}^p, \text{sm}} R \Gamma(\Ind_{Q^\infty}^{\widetilde{G}^\infty} \mathfrak{X}_Q,k) )_{\widetilde{\m}} ) = \varinjlim_{ \widetilde{K}_p'} H^\ast(\widetilde{X}^Q_{\widetilde{K}^p \widetilde{K}_p'},k)_{\widetilde{\m}} = 0. \]
	This follows from the corresponding finite level statement, which has already been proved in the course of the proof of Theorem \ref{thm:reduction_to_Siegel}. 
	
	We therefore need to compute $R \Gamma_{\widetilde{K}^p, \text{sm}} R \Gamma( \Ind_{P^\infty}^{\widetilde{G}^\infty} \mathfrak{X}_P, \cO / \varpi^m)$. We will in fact show that this complex admits $\Ind_{P(F_p^+)}^{\widetilde{G}(F_p^+)} R \Gamma_{K^p, \text{sm}} R \Gamma(\mathfrak{X}_G, \cO / \varpi^m)$ as a $\widetilde{\T}^S$-equivariant direct summand in $\mathbf{D}_{\text{sm}}(\cO / \varpi^m[\widetilde{G}(F_p^+)])$, where $\widetilde{\T}^S$ acts on the latter complex via the map $\cS$.
	
	To see this, we compute
	\begin{multline*}  R \Gamma_{\widetilde{K}^p, \text{sm}} R \Gamma( \Ind_{P^\infty}^{\widetilde{G}^\infty} \mathfrak{X}_P, \cO / \varpi^m) ) \cong  R \Gamma_{\widetilde{K}^p, \text{sm}} \Ind_{P^\infty}^{\widetilde{G}^\infty} R \Gamma(\mathfrak{X}_P, \cO / \varpi^m) \\
	 \cong \oplus_{g \in P(F^+) \backslash \widetilde{G}_{S - S_p} / \widetilde{K}_{S - S_p}} \Ind_{P(F_p^+)}^{\widetilde{G}(F^+_p)} R \Gamma_{\widetilde{K}^p_P, \text{sm}} \Res^{P^\infty}_{P^{S - S_p} \times g \widetilde{K}_{P, S - S_p} g^{-1}} R \Gamma(\mathfrak{X}_P, \cO / \varpi^m). \end{multline*}
	Taking the summand corresponding to $g = 1$, we see that it will be enough to exhibit an isomorphism
	\begin{multline*}  R \Gamma_{\widetilde{K}^p_P, \text{sm}} \Res^{P^\infty}_{P^{S - S_p} \times \widetilde{K}_{P, S - S_p} } R \Gamma(\mathfrak{X}_P, \cO / \varpi^m) \\\cong \operatorname{Inf}_{G(F^+_p)}^{P(F^+_p)}  R \Gamma_{K^p, \text{sm}} \Res^{G^\infty}_{G^{S - S_p} \times K_{S - S_p}} R \Gamma(\mathfrak{X}_G, \cO / \varpi^m). \end{multline*}
	Let us write 
	\[ \Gamma_{P\text{-sm}} : \operatorname{Mod}(\cO / \varpi^m[ P(F_p^+) ]) \to \operatorname{Mod}_{\text{sm}}(\cO / \varpi^m[ P(F_p^+) ]), \]
	 \[  \Gamma_{U\text{-sm}} : \operatorname{Mod}(\cO / \varpi^m[ P(F_p^+) ]) \to \operatorname{Mod}_{U\text{-sm}}(\cO / \varpi^m[ P(F_p^+) ]), \]
	 and
	 \[  \Gamma_{G\text{-sm}} : \operatorname{Mod}(\cO / \varpi^m[ G(F_p^+) ]) \to \operatorname{Mod}_{\text{sm}}(\cO / \varpi^m[ G(F_p^+) ]) \]
	 for the functors of $P$, $U$ and $G$-smooth vectors, respectively. The target category for the second functor is $\cO / \varpi^m[ P(F_p^+) ]$-modules with a smooth action of $U(F_p^+)$. These functors are all right adjoint to forgetful functors, and therefore preserve injectives. The restriction of $\Gamma_{P\text{-sm}}$ to $\operatorname{Mod}_{U\text{-sm}}(\cO / \varpi^m[ P(F_p^+) ])$ is the same as taking $G$-smooth vectors. 
	 
	 Unpacking the above, we see that it is enough to construct a Hecke-equivariant isomorphism
	 \numequation\label{eqn:morphism_of_complexes} \begin{split} \operatorname{Inf}_{G(F_p^+)}^{P(F^+_p)} R \Gamma_{G\text{-sm}} H^0(G(F^+) \backslash & G(\A_{F^+}^\infty) / K^p, \cO / \varpi^m) \\ & \to  R \Gamma_{P\text{-sm}} H^0(P(F^+) \backslash P(\A_{F^+}^\infty) / \widetilde{K}_P^p, \cO / \varpi^m). \end{split} \end{equation}
	 The morphism (\ref{eqn:morphism_of_complexes}) is constructed using the canonical natural transformation $\operatorname{Inf}_{G(F_p^+)}^{P(F^+_p)} \circ R \Gamma_{G\text{-sm}} \to R \Gamma_{P\text{-sm}} \circ \operatorname{Inf}_{G(F_p^+)}^{P(F^+_p)}$ (\cite[Lemma 2.1]{new-tho}), and the morphism $\operatorname{Inf}_{G(F_p^+)}^{P(F^+_p)} H^0(G(F^+) \backslash G(\A_{F^+}^\infty) / K^p, \cO / \varpi^m) \to H^0(P(F^+) \backslash P(\A_{F^+}^\infty) / \widetilde{K}_P^p, \cO / \varpi^m)$ given by inflation of functions. The Hecke-equivariance follows from \cite[Corollary 2.8]{new-tho}. 
	 
	 To show that (\ref{eqn:morphism_of_complexes}) is an isomorphism, it will be enough to show that
	 \[ R \Gamma_{U\text{-sm}} H^0(P(F^+) \backslash P(\A_{F^+}^\infty) / \widetilde{K}_P^p, \cO / \varpi^m) \cong \operatorname{Inf}_{G(F_{p}^+)}^{P(F_p^+)}H^0(G(F^+) \backslash G(\A_{F^+}^\infty) / K^p, \cO / \varpi^m). \] Indeed, we can then take the derived functor of $G$-smooth vectors on both sides to obtain (\ref{eqn:morphism_of_complexes}) --- this operation commutes with inflation from $G$ to $P$, since (the inflation of) a $G$-injective is acyclic for the functor of $G$-smooth vectors.
	 
	 However, the cohomology groups of the left-hand side here can be computed as 
	 \[ \varinjlim_{V_p \subset U(F^+_p)} \prod_{g \in G(F^+) \backslash G(\A_{F^+}^\infty) / K^p} H^i(V_p, H^0(U(F^+) \backslash U(\A_{F^+}^\infty) / \widetilde{K}_U^p, \cO / \varpi^m), \] the limit running over all open compact subgroups $V_p \subset U(F^+_p)$.
	 
	 Using strong approximation, we compute
	 \[ H^i(V_p, H^0(U(F^+) \backslash U(\A_{F^+}^\infty) / \widetilde{K}_U^p, \cO / \varpi^m)) \\ = H^i( U(F^+) \cap (\widetilde{K}_U^p V_p), \cO / \varpi^m).\]
	Taking the limit, we get a product of copies of $\cO / \varpi^m$ in degree 0, and $0$ in all higher degrees. This completes the proof.
\end{proof}
Combining this theorem with the results of the previous section, we obtain the following.
\begin{thm}\label{thm:ord_deg_shifting}
	Let $\widetilde{K} \subset \widetilde{G}(\A_{F^+}^\infty)$ be a good subgroup which is decomposed with respect to $P$, and such that $\widetilde{K}_{\overline{v}} = \Iw_{\overline{v}}$ for each place $\overline{v} \in \overline{S}_p$. Let $\widetilde{\lambda} \in (\Z^{2n}_+)^{\Hom(F^+, E)}$, let $w \in {}^r W^P$, and let $\lambda_w = w(\widetilde{\lambda} + \rho) - \rho \in (\Z^{n}_+)^{\Hom(F, E)}$. Let $\m \subset \T^S$ be a non-Eisenstein maximal ideal, and let $\widetilde{\m} = \cS^\ast(\m)$. Fix integers $c \geq b \geq 0$ with $c \geq 1$. Then for any $j \in \Z$, $\cS$ descends to a surjective homomorphism
	\[ \widetilde{\T}^{S, \ord}( H^j(\partial \widetilde{X}_{\widetilde{K}(b,c)}, \cV_{\widetilde{\lambda}})^{\ord}_{\widetilde{\m}} ) \to \T^{S, \ord}(\cO(\alpha_{w_0^G w w_0^{\widetilde{G}}}) \otimes_{\cO} \tau^{-1}_{w_0^G w w_0^{\widetilde{G}}} H^{j - l(w)}( X_{K(b, c)}, \cV_{\lambda_w})_\m^{\ord}). \]
\end{thm}
\begin{proof}
Let $m \geq 1$. To save space, we abbreviate functors $\Gamma(H, -)$ of $H$-invariants as $\Gamma_H$. By Theorem \ref{thm:direct_summand_of_completed_boundary_cohomology}, Lemma \ref{lem_ord_functors_commute}, and Proposition \ref{prop:independence_of_weight_tildeG}, the complex
\[ R\Gamma_{\widetilde{K}(0, c)/\widetilde{K}(b, c)}(\partial \widetilde{X}_{\widetilde{K}(b,c)}, \cV_{\widetilde{\lambda}} / \varpi^m)_{\widetilde{\m}}^{\ord} \]
admits the complex
\[ \ord_b R \Gamma_{T(\cO_{F^+, p})(b)} \cO(w_0^{\widetilde{G}} \widetilde{\lambda})\otimes_\cO R \Gamma_{N(\cO_{F^+, p})} \Ind_{P(F^+_p)}^{\widetilde{G}(F^+_p)} \pi(K^p, m)_{\widetilde{\m}} \]
as a $\widetilde{\T}^{S, \ord}$-equivariant direct summand. These direct sum decompositions are compatible as $m$ varies, so after passing to the inverse limit we get a surjection of $\widetilde{\T}^{S, \ord}$-algebras:
\numequation\label{eqn:hecke_alg_surjection_0} \begin{split} \widetilde{\T}^{S, \ord}&( H^j(\partial \widetilde{X}_{\widetilde{K}(b,c)}, \cV_{\widetilde{\lambda}})_{\widetilde{\m}}^{\ord}) \\&\to \widetilde{\T}^{S, \ord}(\varprojlim_m \ord_b R^j \Gamma_{T(\cO_{F^+, p})(b)} \cO(w_0^{\widetilde{G}} \widetilde{\lambda}) \otimes_\cO R \Gamma_{N(\cO_{F^+, p})} \Ind_{P(F^+_p)}^{\widetilde{G}(F^+_p)} \pi(K^p, m)_{\widetilde{\m}}). \end{split}
\end{equation}
 On the other hand, it follows from Lemma \ref{lem:exactness_of_degree_j-ordinary_part} that for any $i \geq 0$, we have a short exact sequence of $\widetilde{\T}^{S, \ord}$-modules:
\[ \begin{split} 0 & \to \ord_b R^j \Gamma_{T(\cO_{F^+, p})(b)}\cO(w_0^{\widetilde{G}} \widetilde{\lambda}) \otimes_\cO R \Gamma_{N(\cO_{F^+, p})}  I_{\geq i+1} \pi(K^p, m)_{\widetilde{\m}}\\&\to \ord_b R^j \Gamma_{T(\cO_{F^+, p})(b)}\cO(w_0^{\widetilde{G}} \widetilde{\lambda}) \otimes_\cO R \Gamma_{N(\cO_{F^+, p})} I_{\geq i} \pi(K^p, m)_{\widetilde{\m}} \\ & \to \oplus_{\substack{ w \in {}^r W^P \\ l_r(w) = i}} \ord_b R^j \Gamma_{T(\cO_{F^+, p})(b)}\cO(w_0^{\widetilde{G}} \widetilde{\lambda}) \otimes_\cO R \Gamma_{N(\cO_{F^+, p})} I_{w} \pi(K^p, m)_{\widetilde{\m}} \to 0. \end{split} \]
These are compatible as $m$ varies, and the cohomology groups are finitely generated $\cO$-modules, so we can pass to the limit to obtain short exact sequences of $\cO$-modules. It follows that for any $i \geq 0$ and any element $w \in {}^rW^P$ of length $l_r(w) = i$, there are surjective homomorphisms of $\widetilde{\T}^{S, \ord}$-algebras
\numequation\label{eqn:hecke_alg_surjection_1} \begin{split} \widetilde{\T}^{S, \ord}&( \varprojlim_m \ord_b R^j \Gamma_{T(\cO_{F^+, p})(b)}\cO(w_0^{\widetilde{G}} \widetilde{\lambda}) \otimes_\cO R \Gamma_{N(\cO_{F^+, p})} I_{\geq i} \pi(K^p, m)_{\widetilde{\m}} ) \\ & \to \widetilde{\T}^{S, \ord}( \varprojlim_m \ord_b R^j \Gamma_{T(\cO_{F^+, p})(b)}\cO(w_0^{\widetilde{G}} \widetilde{\lambda})\otimes_\cO R \Gamma_{N(\cO_{F^+, p})} I_{\geq i+1} \pi(K^p, m)_{\widetilde{\m}}) \end{split}
\end{equation}
and
\numequation\label{eqn:hecke_alg_surjection_2} \begin{split} \widetilde{\T}^{S, \ord}&( \varprojlim_m \ord_b R^j \Gamma_{T(\cO_{F^+, p})(b)}\cO(w_0^{\widetilde{G}} \widetilde{\lambda}) \otimes_\cO R \Gamma_{N(\cO_{F^+, p})} I_{\geq i} \pi(K^p, m)_{\widetilde{\m}} ) \\ & \to \widetilde{\T}^{S, \ord}( \varprojlim_m \ord_b R^j \Gamma_{T(\cO_{F^+, p})(b)}\cO(w_0^{\widetilde{G}} \widetilde{\lambda}) \otimes_\cO R \Gamma_{N(\cO_{F^+, p})} I_{w} \pi(K^p, m)_{\widetilde{\m}} ). \end{split}
\end{equation}
By definition, $I_{\geq 0} \pi(K^p, m)$ is (the restriction to $B(F^+_p)$ of)  $\Ind_{P(F^+_p)}^{\widetilde{G}(F^+_p)} \pi(K^p, m)$. On the other hand, Proposition \ref{prop:ordinary_part_of_w_induction} shows that there is a $\widetilde{\T}^{S, \ord}$-equivariant isomorphism
\numequation\label{eqn:hecke_alg_surjection_3} \begin{split}   \ord_b R^j & \Gamma_{T(\cO_{F^+, p})(b)}\cO(w_0^{\widetilde{G}} \widetilde{\lambda}) \otimes_\cO R \Gamma_{N(\cO_{F^+, p})} I_{w} \pi(K^p, m)_{\widetilde{\m}} \\ & \cong R^{j - [F^+ : \Q]n^2 + l(w)} \Gamma_{T_n(\cO_{F, p})(b)} \cO / \varpi^m(\chi_w) \otimes_\cO \cO(w_0^{\widetilde{G}} \widetilde{\lambda})\otimes_{\cO} \tau^{-1}_w \pi^{\ord}(K^p, m).\end{split} 
\end{equation}
We recall that there is an isomorphism $\cO(\chi_w) \cong \cO(-\rho + w^{-1} w_0^P \rho) \otimes_{\cO} \cO(\alpha_w)$. We have $(-\rho+ w^{-1} w_0^P \rho + w_0^{\widetilde{G}} \widetilde{\lambda})^w = w_0^G \lambda_x$, where $x = w_0^G w w_0^{\widetilde{G}}$. Here we write $w_0^G$ for the longest element of $W_P$, $w_0^{\widetilde{G}}$ for the longest element of $W$, and note that the map $w \mapsto w_0^G w w_0^{\widetilde{G}}$ is an involution of ${}^r W^P$ which satisfies $l(w_0^G w w_0^{\widetilde{G}}) = [F^+ : \Q] n^2 - l(w)$. Applying  Propositions \ref{prop:comparison_of_completed_and_classical_ord_coh} and \ref{prop:independence_of_weight}, it follows that the cohomology group in (\ref{eqn:hecke_alg_surjection_3}) may be identified with
\[ \cO(\alpha_{w_0^G x w_0^{\widetilde{G}}}) \otimes_{\cO}  \tau^{-1}_{w_0^G x w_0^{\widetilde{G}}} H^{j - l(x)}( X_{K(b, c)}, \cV_{\lambda_x} / \varpi^m)^\text{ord}_{\widetilde{\m}}. \]
Putting all of this together, we see that we can chain together the surjections (\ref{eqn:hecke_alg_surjection_0}), (\ref{eqn:hecke_alg_surjection_1}) and (\ref{eqn:hecke_alg_surjection_2}) to obtain a surjection homomorphism of $\widetilde{\T}^{S, \ord}$-algebras
\[ \widetilde{\T}^{S, \ord}( H^j(\partial \widetilde{X}_{\widetilde{K}(b,c)}, \cV_{\widetilde{\lambda}})^{\ord}_{\widetilde{\m}} ) \to \T^{S, \ord}(\cO(\alpha_{w_0^G x w_0^{\widetilde{G}}}) \otimes_{\cO} \tau^{-1}_{w_0^G x w_0^{\widetilde{G}}} H^{j - l(x)}( X_{K(b, c)}, \cV_{\lambda_x})_{\widetilde{\m}}^{\ord}). \]
The proof is complete on noting that $H^\ast(X_{K(b, c)}, \cV_{\lambda_x})_\m^{\ord}$ is a $\T^{S, \ord}$-invariant direct summand of $H^\ast(X_{K(b, c)}, \cV_{\lambda_x})_{\widetilde{\m}}^{\ord}$.
\end{proof}
In order to apply Theorem \ref{thm:ord_deg_shifting}, we will make use of the following combinatorial lemma. We use the following notation: if $\lambda \in (\Z^n_+)^{\Hom(F, E)}$ and $a \in \Z$, then $\lambda(a) \in (\Z^n_+)^{\Hom(F, E)}$ is the highest weight defined by the formula $\lambda(a)_{\tau, i} = \lambda_{\tau, i} + a$ for all $\tau \in \Hom(F, E)$, $i = 1, \dots, n$. We recall as well that we have previously fixed the notation $\widetilde{S}_p$ for a set of $p$-adic places of $F$ lifting $\overline{S}_p$, and $\widetilde{I}_p$ for the set of embeddings $\tau : F \hookrightarrow E$ inducing a place of $\widetilde{S}_p$ (cf. \S \ref{sec:unitary_group_setup}).
\begin{lemma}\label{lem_ord_deg_shift_divisible_degrees}
	Fix $m \geq 1$. Then we can find $\lambda \in (\Z^n_+)^{\Hom(F, E)}$ with the following properties:
	\begin{enumerate}
		\item There is an isomorphism $\cO(\lambda) / \varpi^m \cong \cO / \varpi^m$ of $T_n(F_p)$-modules.
		\item The sum $\sum_{i=1}^n (\lambda_{\tau, i} + \lambda_{\tau c, i})$ is independent of the choice of $\tau \in \Hom(F, E)$.
		\item For each $i = 0, \dots, n^2$, there exists an element $w_i = (w_{i, \overline{v}})_{\overline{v} \in \overline{S}_p} \in {}^r W^P$, an integer $a_i \in (p-1)\Z$, and a dominant weight $\widetilde{\lambda}_i \in (\Z^{2n}_+)^{\Hom(F^+, E)}$, all satisfying the following conditions:
		\begin{enumerate}
			\item $\widetilde{\lambda}_i$ is CTG (cf. Definition \ref{defn:coh_trivial}).
			\item For each $\overline{v} \in \overline{S}_p$, $l_r(w_{i, \overline{v}}) = n^2-i$. Consequently, $l(w_i) = [F^+ : \Q] (n^2 -i)$.
			\item We have $w_i(\widetilde{\lambda}_i + \rho) - \rho = \lambda(a_i)$.
		\end{enumerate}
	\end{enumerate}
\end{lemma}
\begin{proof}
	Let $M > 16 n$ be a non-negative integer which is divisible by $8 (p-1) \# (\cO / \varpi^m)^\times$. We will show that we can take $\lambda$ to be the dominant weight defined by the formulae
	\[ \lambda_{\tau} = \left\{ \begin{array}{cc} (-n M, -2nM, \dots, -n^2 M) & \text{ if }\tau \in \widetilde{I}_p;\\ (0, -M, \dots, (1-n)M) & \text{ if }\tau c \in \widetilde{I}_p. \end{array}\right. \]
	If $\widetilde{\lambda}(a)$ denotes the element of $(\Z^{2n})^{\Hom(F^+, E)}$ that corresponds to $\lambda(a)$ under our identifications, then we have
	\[ \widetilde{\lambda}(a) = ( (n-1)M -a, \dots, -a, -n M + a, \dots, -n^2 M + a). \]
	In order to construct the elements $w_i$ and $a_i$, we make everything explicit. Our choice of the set $\widetilde{S}_p$ determines an isomorphism of the group $(\Res_{F^+ / \Q} \widetilde{G})_{\Q_p}$ with the group $\prod_{\overline{v} \in \overline{S}_p} \Res_{F_{\widetilde{v}} / \Q_p} \GL_{2n}$, hence an identification of ${}^r W_{\overline{v}}$ with $S_{2n}$ and of ${}^r W_{P, \overline{v}}$ with the subgroup $S_n \times S_n$. We can identify the set ${}^r W^P_{\overline{v}}$ of representatives for the quotient ${}^r W_{P, \overline{v}} \backslash {}^r W_{\overline{v}}$ with the set of $n$-element subsets of $\{ 1, \dots, 2n \}$. Given such a subset $X$, there is a unique permutation $\tau$ of $\{ 1, \dots, 2n \}$ with $\tau( \{ 1, \dots, n \}) = X$ and with the property that $\tau$ is increasing on both $\{ 1, \dots, n \}$ and $\{ n+1, \dots, 2n \}$. The corresponding element of ${}^r W^P_{\overline{v}}$ is $\sigma_X = \tau^{-1}$. The length of a permutation $w \in S_{2n}$ is given by the formula $l(w) = \#\{ 1 \leq i < j \leq 2n \mid w(i) > w(j) \}$. 
	
	Given $i$, we choose integers $r, x \geq 0$ with $n x + n - r = n^2 - i$ and $1 \leq r \leq n$ (the choice is unique). We define $w_i$ by setting $w_{i,\overline{v}} = \sigma_{X_i}$ for each $\overline{v} \in \overline{S}_p$, where $X_i = \{ x+1, x+2, \dots, x+r, x + r + 2, x + r + 3, \dots, x + n + 1 \}$. We have
	\[ \begin{split} (w_{i,\overline{v}}(1), \dots,  &w_{i,\overline{v}}(2n)) = \\  (n+1, n+2, \dots, n+x, &1, 2, \dots, r, n+ x + 1, r+1, \\  & r+2, \dots, n, n+ x + 2, n+x+3, \dots, 2n). \end{split} \]
	We observe that indeed $l_r(w_{i, \overline{v}}) = n^2 - i$. We need to choose $a_i$ so that the weight $\widetilde{\lambda}_i = w_i^{-1}(\widetilde{\lambda}(a_i) + \rho) - \rho$ is dominant. We first calculate $w_i^{-1}(\widetilde{\lambda}(a_i))$. For any $\tau \in \Hom(F^+, E)$, we have $w_i^{-1}(\widetilde{\lambda}(a_i))_{\tau, j} = \widetilde{\lambda}(a_i)_{\tau, w_i(j)}$, hence the $\tau$ component of $w_i^{-1}(\widetilde{\lambda}(a_i))$ is equal to
	\[ \begin{split} (\widetilde{\lambda}(a_i)_{\tau, n+1},  \dots, \widetilde{\lambda}(a_i)_{\tau, n+x},& \widetilde{\lambda}(a_i)_{\tau, 1},  \dots, \widetilde{\lambda}(a_i)_{\tau, r}, \widetilde{\lambda}(a_i)_{\tau, n+x+1}, \\ & \widetilde{\lambda}_{\tau, r+1}(a_i),  \dots, \widetilde{\lambda}_{\tau, n}(a_i), \widetilde{\lambda}_{\tau, n+x+2}(a_i), \dots, \widetilde{\lambda}_{\tau, 2n}(a_i)).  \end{split} \]
	\[ \begin{split} = ( -n M + a_i, \dots, - n x M + a_i,& (n-1)M - a_i, \dots, (n-r)M - a_i, -n (x+1) M + a_i,\\& (n-r-1) M - a_i, \dots, - a_i, -n(x+2)M + a_i, \dots, -n^2 M + a_i). \end{split}\]
	We see that $w_i^{-1}(\widetilde{\lambda}(a_i))$ is dominant if and only if the following 4 inequalities are satisfied:
	\numequation\label{eqn:dominant_weight_1}
	 - n x M + a_i \geq (n-1) M - a_i,
	\end{equation}
	\numequation\label{eqn:dominant_weight_2}
	(n-r)M - a_i \geq -n(x+1) M + a_i, 
	\end{equation}
	\numequation\label{eqn:dominant_weight_3}
	-n(x+1) M + a_i \geq (n-r-1) M - a_i,
	\end{equation}
	\numequation\label{eqn:dominant_weight_4}
	-a_i \geq -n(x+2)M + a_i.
	\end{equation}
 	These~$4$ inequalities are together equivalent to requiring that $a_i \in [(nx+2n -r - 1)M / 2, (nx+2n -r)M/2]$, a closed interval of length $M / 2$. Requiring instead that $w_i^{-1}(\widetilde{\lambda}(a_i) + \rho) - \rho$ is dominant leads to~$4$ similar inequalities, where the left-hand side and right-hand side differ to those in (\ref{eqn:dominant_weight_1})--(\ref{eqn:dominant_weight_4}) by an integer of absolute value at most $2n-1$. If we choose $a_i$ to be the unique integer in $[(nx+2n -r - 1)M / 2, (nx+2n -r)M/2]$ which is congruent to $M / 8 \text{ mod } M/2$, then $w_i^{-1}(\widetilde{\lambda}(a_i) + \rho) - \rho$ is dominant.
 	
 	To complete the proof of the lemma, we just need to explain why $\widetilde{\lambda}_i = w_i^{-1}(\widetilde{\lambda}(a_i) + \rho) - \rho$ is CTG. It suffices to show that for any $\tau \in \Hom(F^+, E)$, and for any $w \in W^P_{\overline{v}}$ (where $\overline{v}$ is the place of $F^+$ induced by $\tau$), the number 
  \begin{multline*} [w(\widetilde{\lambda}_{i} + \rho) - \rho]_{\tau, j} + [w(\widetilde{\lambda}_{i} + \rho) - \rho]_{\tau, 2n+1 - j}  \\
  = w(\widetilde{\lambda}_{i} + \rho)_{\tau, j} + w(\widetilde{\lambda}_{i} + \rho)_{\tau, 2n+1-j}
  	\end{multline*} 
  is not independent of $j$ as $j$ varies over integers $1 \leq j \leq n$. To show this, it suffices to show that the multiset
  \[ I = \{(\widetilde{\lambda}_{i} + \rho)_{\tau, j} + (\widetilde{\lambda}_{i} + \rho)_{\tau, k} \mid 1 \leq j < k \leq 2n \} \]
  does not contain any element with multiplicity at least $n$. We first consider the multiset
  \[ I' =  \{\widetilde{\lambda}_{i,\tau, j} + \widetilde{\lambda}_{i,\tau, k} \mid 1 \leq j < k \leq 2n \}. \]
  It is a union of the three multisets
  \[ I_1' = \{ ( - n \alpha +  \beta )M \mid 1 \leq \alpha \leq n, 0 \leq \beta \leq n-1 \}, \]
  \[ I_2' = \{ -n(\alpha + \beta)M + 2 a_i \mid 1 \leq \alpha < \beta \leq n \}, \]
  and
  \[ I_3' = \{ (\alpha + \beta) M - 2 a_i \mid 0 \leq \alpha < \beta \leq n-1 \}. \]
  Note that each element of $I_1'$ has multiplicity 1. Each element of $I'_2$ and $I'_3$ has multiplicity at most $n / 2$. Moreover, $I'_1$, $I'_2$, and $I'_3$ are mutually disjoint (look modulo $M$). It follows that no element of $I'$ has multiplicity at least $n$. To show that $I$ has no element of multiplicity at least $n$, we use the analogous decomposition $I = I_1 \cup I_2 \cup I_3$. The sets $I_1$, $I_2$ and $I_3$ are disjoint (look modulo $M$, and use the fact that each entry of $\rho$ has absolute value at most $(2n-1)/2$). Each element of $I_1$ appears with multiplicity 1, while each entry of $I_2$ and $I_3$ has multiplicity at most $n / 2$. This completes the proof.
\end{proof}
Lemma \ref{lem_ord_deg_shift_divisible_degrees} allows us to express certain cohomology groups of the spaces $X_K$ in degrees divisible by $[F^+ : \Q]$ in terms of middle degree cohomology of the spaces $\partial \widetilde{X}_{\widetilde{K}}$ (and hence, using Theorem \ref{what we get from CS}, of the spaces $\widetilde{X}_{\widetilde{K}}$). Indeed, combining the results so far of this section, we obtain the following result.
\begin{prop}\label{prop_ord_degree_shifting_in_all_degrees} Suppose that $[F^+ : \Q] > 1$. 
	Let $m \geq 1$ be an integer. Then there exists a dominant weight $\lambda 
	\in (\Z^n_+)^{\Hom(F, E)}$ such that a finite index subgroup of $\cO_F^\times$ acts trivially on $\cV_\lambda$ and for each  $i = 0, \dots, n^2-1$,  a dominant 
	weight $\widetilde{\lambda}_i \in (\Z^{2n}_+)^{\Hom(F^+, E)}$ which is CTG, 
	an integer $a_i$ divisible by $(p-1)$, and a Weyl element $w_i \in {}^r 
	W^P$ such that the following conditions are satisfied:
	Let $\widetilde{K} \subset \widetilde{G}(\A_{F^+}^\infty)$ be a good 
	subgroup which is decomposed with respect to $P$ and such that for each 
	$\overline{v} \in \overline{S}_p$, $\widetilde{K}_{\overline{v}} = 
	\widetilde{\Iw}_{\overline{v}}$. Fix integers $c \geq b \geq 0$ with $c 
	\geq 1$, and also an integer $m \geq 1$. Let $\m \subset \T^S$ be a 
	non-Eisenstein maximal ideal. Let $\widetilde{\m} = \cS^\ast(\m) \subset 
	\widetilde{\T}^S$, and suppose that $\overline{\rho}_{\widetilde{\m}}$ is 
	decomposed generic. Then:
	\begin{enumerate}
		\item There is an isomorphism $\cO(\lambda) / \varpi^m \cong \cO / \varpi^m$ of $\cO[T(F^+_p)]$-modules.
		\item For each  $i = 0, \dots, n^2-1$, the map $\cS$ descends to an 
		algebra homomorphism
		\[ \widetilde{\T}^{S, \ord}(H^d(\widetilde{X}_{\widetilde{K}(b, c)}, 
		\cV_{\widetilde{\lambda}_i})^{\ord}_{\widetilde{\m}}) \to \T^{S, 
			\ord}(\cO(\alpha_{w_i}) \otimes_{\cO} 
		\tau_{w_i}^{-1} H^{i[F^+:\Q]}(X_{K(b, c)}, 
		\cV_{\lambda(a_i)})_\m^{\ord}).   
		\]	
	\end{enumerate}
\end{prop}
\begin{proof}
	This follows on combining Theorem \ref{what we get from CS},  Theorem 
	\ref{thm:ord_deg_shifting}, and Lemma \ref{lem_ord_deg_shift_divisible_degrees}.
\end{proof}

In order to access all degrees of cohomology, we use a trick based on the fact that the group $G$ has a non-trivial centre. This is the motivation behind the next few results.

If $K \subset \GL_n(\A_F^{\infty})$ is a good subgroup, then we define
 \[A_K := 
	F^\times\backslash
	\A_F^\times /\det(K)\det(K_\infty)\R_{> 0}.\]
The quotient map \[ A_K \to 
F^\times\backslash\A_F^\times 
/\det(K)F_\infty^\times \] 
identifies 
$A_K$ with an extension of a ray class group by a real torus 
of dimension 
$[F^+:\Q]-1$ (with cocharacter lattice $F^\times \cap \det(K)$, a 
torsion-free congruence subgroup of $\cO_F^\times$). 
We denote the identity component of $A_K$ by $A_K^{\circ}$. If $g \in \GL_n(\A_{F}^\infty)$, then we set $\Gamma_{g, K} = \GL_n(F) \cap g K g^{-1}$, or $\Gamma_g = \Gamma_{g, K}$ if the choice of $K$ is fixed.
\begin{lemma}\label{splitcomponents} \leavevmode
	\begin{enumerate}
		\item The maps $x \mapsto (x,g)$ induce a homeomorphism 
		\[\coprod_{[g]\in\GL_n(F)\backslash 
			\GL_n(\A_{F}^\infty)/K} \Gamma_g \backslash X \cong X_K. \] 
		\item The determinant gives a continuous map \[X_K \overset{\det}{\to} A_K \] which 
		induces a 
		bijection on sets of connected components.  
		\item Suppose $g \in \GL_n(\A_{F}^\infty)$ and the two subgroups 
		$\det(\Gamma_g)$ and $\det(F^\times\cap K)$ of $F^\times$ are equal. Let $\Gamma_g^{1} = \SL_n(F) \cap 
		\Gamma_g$. Then the product map 
		\[ \Gamma_g^{1} \times (F^\times\cap K) \to \Gamma_g \]
		is a group isomorphism. Decomposing $X$ similarly as \[ X^{1} \times (\prod_{v |\infty} \R_{> 
			0})/\R_{> 
			0}  = X, \] where $X^{1} = \SL_n(F_\infty)/\prod_{v\mid \infty}\mathrm{SU}(n),$ we obtain a 
		decomposition \[\Gamma_g \backslash X = \left(\Gamma_g^1 \backslash X^1\right) 
		\times (F^\times \cap K) \backslash (\prod_{v |\infty} \R_{> 
			0})/\R_{> 
			0}.\]  
		\item Still assuming that $\det(\Gamma_g) = \det(F^\times\cap K)$, the map 
		$\det: F^\times \cap K \to F^\times \cap\det(K)$ is an isomorphism. The 
		composition of maps \[\left(\Gamma_g^1 \backslash X^1\right) \times (F^\times 
		\cap K) \backslash (\prod_{v |\infty} \R_{> 
			0})/\R_{> 
			0} = \Gamma_g\backslash X \hookrightarrow X_K \to A_K\] is given by 
		$(x,z) \mapsto \det(g)z^n$, and the map $z \mapsto \det(g)z^n$ is an 
		isomorphism from $(F^\times 
		\cap K) \backslash (\prod_{v |\infty} \R_{> 
			0})/\R_{> 
			0}$ to the connected component $A_K^{[\det(g)]}$ of $A_K$ containing 
		$[\det(g)]$.
	\end{enumerate}
\end{lemma}
\begin{proof}
	The first part can be checked directly. The second part is equivalent to 
	the statement that $\det$ induces a bijection \[G(F^+)\backslash 
	G(\A_{F^+}^\infty)/K \to F^\times \backslash (\A_F^\infty)^\times / 
	\det(K).\] This follows from strong approximation for the derived subgroup 
	of $G$, which is isomorphic to $\Res_{F / F^+} \SL_n$. For the third part, injectivity of 
	the natural map $\Gamma_g^{1} \times (F^\times\cap K) \to \Gamma_g$ follows 
	from neatness of $K$ (since $F^\times\cap K$ contains no roots of unity, 
	and hence no non-trivial elements of determinant $1$). Surjectivity follows from the 
	assumption that $\det(\Gamma_g) = \det(F^\times\cap K)$. The remainder of 
	the third part (on the decomposition of $\Gamma_g \backslash X$) is an 
	immediate consequence. Finally, for the fourth part, everything follows 
	from the claim that $\det: F^\times \cap K \to F^\times \cap\det(K)$ is an 
	isomorphism. Injectivity follows from neatness of $K$. Surjectivity follows 
	from strong approximation for $\SL_n$ and the assumption that 
	$\det(F^\times\cap K) = \det(\Gamma_g)$. Indeed, suppose we 
	have $k \in K$ with $\det(k) \in F^\times$. We can find $\gamma \in 
	\GL_n(F)$ such that $\det(\gamma) = \det(k)$, and strong approximation 
	implies that we can find $\gamma' \in \SL_n(F)$ and $k' \in gKg^{-1} \cap 
	\SL_n(\A_F^\infty)$ such that $\gamma (gkg^{-1})^{-1} = \gamma'k'$. We 
	deduce that 
	$(\gamma')^{-1}\gamma = k'gkg^{-1} \in gKg^{-1} \cap\GL_n(F)$ has the same 
	determinant as $k$, which shows surjectivity.
\end{proof}
The following lemma shows how to choose $K$ so that the conditions of Lemma \ref{splitcomponents} are satisfied.
\begin{lemma}\label{shrink}
	Let $K$ be a good subgroup of $G(\A_{F^+}^\infty)$. Fix a finite set $T$ of 
	finite places  of $F$. There exists a good normal subgroup $K' \subset K$ with $K'_T = K_T$ such that  
	$\det(\Gamma_{g, K'}) = \det(F^\times\cap K')$ 
	for all $g \in 
	\GL_n(\A_{F}^\infty)$.
\end{lemma}
\begin{proof}
	We begin by choosing an ideal $\ga$ of $\cO_F$, prime to $T$, such that 
	$\ker(\cO_F^\times \to (\cO_F / \ga)^\times)$ is torsion-free and is 
	contained in $F^\times\cap K$. 
	This is possible by 
	Chevalley's theorem \cite[Thm.~1]{chevalley1951}. Similarly, we can choose 
	another 
	ideal $\mathfrak{b}$ of $\cO_F$, prime to $\ga$ and $T$, such that 
	$\ker(\cO_F^\times \to (\cO_F / \ga\mathfrak{b})^\times)$ is contained in 
	$(\ker(\cO_F^\times \to (\cO_F / \ga)^\times))^n$. We claim that \[K' := 
	\ker(\cO_F^\times \to 
	(\cO_F / 
	\ga)^\times)\cdot K(\ga\mathfrak{b})\] has the desired properties, where 
	$K(\ga\mathfrak{b})$ is the 
	intersection of $K$ with the principal congruence subgroup of level 
	$\ga\mathfrak{b}$. Indeed, by construction we have $\det(\GL_n(F)\cap 
	gK'g^{-1}) = (\ker(\cO_F^\times \to (\cO_F / \ga)^\times))^n$ for all $g 
	\in \GL_n(\A_F^\infty)$, whilst $F^\times \cap K' = \ker(\cO_F^\times \to 
	(\cO_F / \ga)^\times)$.
\end{proof}
The next lemma shows how to use Lemma \ref{splitcomponents} to understand all cohomology groups of a space $X_K$ solely in terms of those in degrees divisible by $[F^+ : \Q]$. Note that $\dim(X_K) = d-1 = [F^+:\Q]n^2 - 1$ and $\dim(A_K) = [F^+:\Q]-1$.
\begin{lemma}\label{centraldegreeshifting}
	Let $K \subset \GL_n(\A_F^\infty)$ be a good subgroup, and let $\lambda \in (\Z^n_+)^{\Hom(F, E)}$. Suppose that the following conditions are satisfied:
	\begin{enumerate}
		\item  
		$\det(\Gamma_g)=\det(F^\times\cap K)$ for all $g \in 
		\GL_n(\A_{F}^\infty)$.
		\item $F^\times \cap K$ acts trivially on $\cV_\lambda$.
	\end{enumerate}
	Recall that we have defined a map $\det: X_K \to A_K$. Then 
	$R\det_*(\cV_{\lambda})$ (a complex of sheaves of $\cO$-modules) is 
	constant on each 
	connected component of $A_K$, and we have $R\det_*(\cV_{\lambda}) = 
	\bigoplus_{i=0}^{\dim(X^1)} R^i\det_*(\cV_{\lambda})[-i]$. 
	We obtain a $\TT^{S, \ord}$-equivariant isomorphism of graded $\cO$-modules
	\numequation\label{eqn:hecke_equivariant_isomorphism} \bigoplus_{i=0}^{\dim(X_K)}H^i(X_K,\cV_{\lambda}) \cong  
	\left(\bigoplus_{j=0}^{\dim(A_K^\circ)}H^j(A_K^\circ,\cO)\right)\otimes_{\cO}\left(\bigoplus_{k=0}^{\dim(X^1)}H^0(A_K,R^k\mathrm{det}_*(\cV_{\lambda}))\right)
	\end{equation}
	where the Hecke action on the first factor $\bigoplus_{j=0}^{\dim(A_K^\circ)}H^j(A_K^\circ,\cO)$ is 
	trivial.
	
	As a consequence, the image of $\TT^{S, \ord}$ in 
	$\End_{\cO}(\bigoplus_{i=0}^{\dim(X_K)}H^i(X_K,\cV_\lambda))$ is equal to its image in 
	$\End_{\cO}(\bigoplus_{i=0}^{n^2-1}H^{i[F^+:\Q]}(X_K,\cV_\lambda))$.
\end{lemma}
\begin{proof}
	It follows from our first assumption and Lemma \ref{splitcomponents} that 
	every connected component of $X_K$ decomposes as a product 
	$(\Gamma_g^1\backslash X^1) \times A_K^{[\det(g)]}$, with the map $\det$ 
	given by the projection to the second factor. Our second assumption 
	implies that the local system $\cV_{\lambda}$ on this component is 
	pulled back from a local system on $\Gamma_g^1\backslash X^1$. We deduce 
	that $R\det_*(\cV_\lambda)$ is constant on $A_K^{[\det(g)]}$
	(corresponding to $R\Gamma(\Gamma_g^1\backslash X^1 
	,\cV_{\lambda})$) and it decomposes as the direct sum of its shifted 
	cohomology 
	sheaves (since the same is true for any object in $\mathbf{D}(\cO)$, such as 
	$R\Gamma(\Gamma_g^1\backslash X^1 
	,\cV_{\lambda})$). To save space, we now write $H^*(\cdots)$ for the graded cohomology module $\bigoplus_i H^i(\cdots)$.
	
	Passing to global sections on $A_K$ we get an isomorphism
	\[ \begin{split}  H^\ast(X_K,\cV_{\lambda}) & \cong \oplus_{i=0}^{\dim(X^1)} H^{\ast}(A_K, R^i  \mathrm{det}_\ast(\cV_{\lambda})) \\ & \cong \oplus_{i=0}^{\dim(X^1)} \oplus_{[g] \in \GL_n(F) \backslash \GL_n(\A_F^\infty) / K} H^\ast(A_K^{[\det g]}, R^i  \mathrm{det}_\ast(\cV_{\lambda})) \end{split} \]
	\[ 
	\cong \oplus_{i=0}^{\dim(X^1)} \oplus_{[g] \in \GL_n(F) \backslash \GL_n(\A_F^\infty) / K} H^\ast(A_K^{[\det g]}, \cO) \otimes_\cO H^0(A_K^{[\det g]},  R^i  \mathrm{det}_\ast(\cV_{\lambda})).  \]
	Note that the cohomology groups $H^i(A_K^{[\det g]},\cO)$ are torsion-free. We now use that the groups $H^\ast(A_K^{[\det g]}, \cO)$ are canonically independent of $g$, so can all be identified with $H^\ast(A_K^\circ, \cO)$. We thus obtain an isomorphism of graded $\cO$-modules
	\[ \begin{split} H^\ast(X_K,\cV_{\lambda}) &  \cong H^\ast(A_K^\circ, \cO) \otimes_{\cO}  \oplus_{i=0}^{\dim(X^1)}  \oplus_{[g] \in \GL_n(F) \backslash \GL_n(\A_F^\infty) / K} H^0(A_K^{[\det g]},  R^i  \mathrm{det}_\ast(\cV_{\lambda})) \\  & \cong H^\ast(A_K^\circ, \cO) \otimes_{\cO} \bigoplus_{i=0}^{\dim(X^1)}H^0(A_K,R^i\mathrm{det}_*(\cV_{\lambda})). \end{split} \]
	We next need to understand the action of Hecke operators. If $g \in G^S$, then the action of the Hecke operator $[K^S g K^S]$ can be described with the aid of the diagram
	\[ \xymatrix{ X_K \ar[d]^\det & X_{K \cap g K g^{-1}} \ar[l]_-{p_1} \ar[r]^-{p_2} \ar[d]^\det & X_K \ar[d]^\det \\
	A_K & A_{K \cap g K g^{-1}} \ar[l]^-{q_1} \ar[r]_-{q_2} & A_K.  }\]
	Here $p_1$ and $q_1$ are induced by the action of $g$, while $p_2$ and $q_2$ are the natural projections; the action of $[K^S g K^S]$ on $R \Gamma(X_K, \cV_{\lambda})$ is given by the formula $p_{2, \ast} \circ p_{1}^\ast$. Pushing forward by $\det$, we have a morphism $q_1^\ast R \det_{\ast} \cV_\lambda \to q_2^\ast R \det_{\ast} \cV_\lambda$, and the induced endomorphism of the complex $R \Gamma(A_K, R \det_{\ast} \cV_{\lambda})$ in $\mathbf{D}(\cO)$ agrees with $[K^S g K^S]$ under the natural identification $R \Gamma(A_K, R \det_{\ast} \cV_\lambda) \cong R \Gamma(X_K, \cV_{\lambda})$. We see that the isomorphism (\ref{eqn:hecke_equivariant_isomorphism}) respects the action of $[K^S g K^S]$ if $[K^S g K^S]$ acts in the usual way on the left-hand side, as multiplication by $[ F^\times \cap \det(K) : F^\times \cap \det(K \cap g K g^{-1})]^i$ on $H^i(A_K^\circ, \cO)$, and in the natural way on $H^0(A_K, R^i \det_\ast \cV_{\lambda})$. Our assumption $\det(\Gamma_g)=\det(F^\times\cap K)$ implies that $F^\times \cap \det(K) = F^\times \cap \det(K \cap g K g^{-1})$, giving the statement in the lemma. 
	
	It remains to check the final statement of the lemma. There is an isomorphism $H^\ast(A_K^\circ, \cO) \cong \wedge_\cO^\ast \Hom( F^\times \cap \det(K), \cO)$ of graded $\cO$-modules. It follows that each cohomology group $H^0(A_K,R^i\mathrm{det}_*(\cV_{\lambda}))$ appears as a direct summand of $H^\ast(X_K, \cV_{\lambda})$ in $[F^+ : \Q]$ consecutive degrees. In particular, it appears as a direct summand of a cohomology group in a degree divisible by $[F^+ : \Q]$. This completes the proof.
\end{proof}

For the statement of the next proposition, we remind the reader that in \S \ref{sec:ord_statements} we have defined for each $\lambda \in (\Z^n_+)^{\Hom(F, E)}$, $v \in S_p$ and $i = 1, \dots, n$, a character \[
 \chi_{\lambda, v, i} : G_{F_v} \to \T^{S, \ord}(H^\ast(X_{K(b, c)}, \cV_{\lambda})^{\ord})^{\times}. \]
\begin{prop}\label{prop:first_consequence_for_ord_Galois_reps} Suppose that $[F^+ : \Q] > 1$. 
	Let $K \subset \GL_n(\A_F^\infty)$ be a good subgroup such that for each $v \in S_p$, $K_v = \Iw_v$. Fix integers $c \geq b \geq 0$ with $c \geq 1$, and also an integer $m \geq 1$. Let $\m \subset \T^S$ be a non-Eisenstein maximal ideal, and let $\widetilde{\m} = \cS^\ast(\m)$. Suppose that the following conditions are satisfied:
	\begin{enumerate}
		\item $\overline{\rho}_{\m}$ is decomposed generic.
		\item Let $v$ be a finite place of $F$ not contained in $S$, and let $l$ be its residue characteristic. Then either $S$ contains no $l$-adic places of $F$ and $l$ is unramified in $F$, or there exists an imaginary quadratic field $F_0 \subset F$ in which $l$ splits.
	\end{enumerate}
	 Then we can find $\lambda \in (\Z^n_+)^{\Hom(F, E)}$ and an integer $N \geq 1$ depending only on $[F^+ : \Q]$ and $n$ such that the following conditions are satisfied:
	\begin{enumerate}
		\item There is an isomorphism $\cO(\lambda) / \varpi^m \cong \cO / \varpi^m$ of $\cO[T(F^+_p)]$-modules.
		\item For each $i = 0, \dots, d-1$, there exists a nilpotent ideal $J_i$ of $\T^{S, \ord}(H^i(X_{K(b, c)}, \cV_{\lambda})_\m^{\ord})$ satisfying $J_i^{N} = 0$ and a  continuous $n$-dimensional representation
		\[ \rho_{\m} : G_{F, S} \to \GL_n(\T^{S, \ord}(H^i(X_{K(b, c)}, \cV_{\lambda})_\m^{\ord}) / J_i) \]
		such that the following conditions are satisfied:
		\begin{enumerate}
			\item[(a)] For each place $v\not\in S$ of $F$, the characteristic polynomial of $\rho_{\m}(\Frob_v)$ is equal to the image of $P_v(X)$ in $(\T^{S, \ord}(H^i(X_{K(b, c)}, \cV_{\lambda})_\m^{\ord}) / J_i)[X]$.
			\item[(b)] For each place $v | p$ of $F$ and for each $g \in G_{F_v}$, the characteristic polynomial of $\rho_{\m}(g)$ equals $\prod_{j=1}^n (X - \chi_{\lambda, v, j}(g))$.
			\item[(c)] For each place $v | p$ of $F$, and for each sequence $g_1, \dots, g_n \in G_{F_v}$, the image of the element 
			\[ (g_1 - \chi_{\lambda, v, 1}(g_1)) (g_2 - \chi_{\lambda, v, 2}(g_2)) \dots (g_n - \chi_{\lambda, v, n}(g_n)) \]
			of $\T^S(H^i(X_{K(b, c)}, \cV_{\lambda})_\m^{\ord})[G_{F_v}]$ in $M_n(\T^S(H^i(X_{K(b, c)}, \cV_{\lambda})_\m^{\ord}) / J_i)$ under $\rho_{\m}$ is zero. 
		\end{enumerate}	
	\end{enumerate}
\end{prop}
\begin{proof}
	We choose $\lambda$ using Proposition 
	\ref{prop_ord_degree_shifting_in_all_degrees}. Note that, for each 
	cohomological degree $i$, by Theorem 
	\ref{thm:existence_of_Hecke_repn_for_GL_n}, we can find $N$, $J_i$ and 
	\[ \rho_{\m} : G_{F, S} \to \GL_n(\T^{S, \ord}(H^i(X_{K(b, c)}, 
	\cV_{\lambda})_\m^{\ord}) / J_i) \]
	satisfying condition (a) of the proposition. Indeed, this theorem and the discussion after Lemma \ref{lem_cohomology_nilpotent_annihilator} gives a representation with values in a quotient of $\TT^S(K(0,c)/K(b,c),\cV_\lambda)_{\m}$ by a nilpotent ideal, which we compose with the canonical homomorphism to $\TT^S(K(b,c),\lambda)^{\ord}_{\m}$. Arguing with the 
	Hochschild--Serre spectral sequence and twisting with characters as in the 
	proof of Corollary \ref{cor:first_consequence_for_FL_property}, we are free 
	to enlarge $S$ and to shrink $K$ at the prime-to-$p$ places of $S$. We can 
	therefore assume that the following conditions are satisfied:
	\begin{enumerate}
		\item For each place $v \in S_p$, the two representations 
		$\overline{\rho}_\m|_{G_{F_v}}$, $(\overline{\rho}_\m^{c, \vee} \otimes 
		\epsilon^{1-2n})|_{G_{F_v}}$ have no Jordan--H\"older factors in common.
		\item $\overline{\rho}_{\widetilde{\m}}$ is decomposed generic. 
		\item $K$ satisfies the conditions of Lemma \ref{centraldegreeshifting}.
	\end{enumerate}
	After enlarging $\cO$, we can assume that there exists a character $\chi : 
	G_{F, S} \to k^\times$ satisfying the following conditions: 
	\begin{enumerate}
		\item For each place $v \in S_p$, $\chi|_{G_{F_v}}$ is unramified.
		\item For each place $v \in S_p$, the two representations 
		$(\overline{\rho}_\m \oplus \overline{\rho}_\m^{c, \vee} \otimes 
		\epsilon^{1-2n}) \otimes \chi|_{G_{F_v}}$ and $(\overline{\rho}_\m 
		\oplus \overline{\rho}_\m^{c, \vee} \otimes \epsilon^{1-2n}) \otimes 
		\chi^{c, \vee}|_{G_{F_v}}$ have no Jordan--H\"older factors in common.
		\item The representation $\overline{\rho}_\m \otimes \chi \oplus 
		\overline{\rho}_\m^{c, \vee} \otimes \chi^{c, \vee}\epsilon^{1-2n}$ is 
		decomposed generic.
	\end{enumerate}	
	
	It follows from Lemma \ref{centraldegreeshifting} and Carayol's lemma 
	(applied as in the proof of Corollary 
	\ref{cor:first_consequence_for_FL_property}) that it suffices to establish 
	conditions (b) and (c) for cohomological degrees $0, [F^+:\Q], \ldots, 
	(n^2-1)[F^+:\Q]$ (Carayol's lemma then gives us a Galois representation 
	with coefficients in $\T^{S, \ord}(\oplus_{i=0}^{n^2-1} 
	H^{i[F^+:\Q]}(X_{K(b, c)}, 
	\cV_{\lambda})_\m^{\ord}) = \T^{S, \ord}(\oplus_{i=0}^{d-1} H^{i}(X_{K(b, 
	c)}, 
	\cV_{\lambda})_\m^{\ord})$ modulo a nilpotent ideal with the desired 
	properties).

		We choose a good subgroup $\widetilde{K} \subset 
	\widetilde{G}(\A_{F^+}^\infty)$ which satisfies the conditions of 
	Proposition \ref{prop_ord_degree_shifting_in_all_degrees} and such that 
	$\widetilde{K} \cap G(\A_{F^+}^\infty) = K$. For each $i = 0, 1, \ldots, 
	n^2-1$, we let $\widetilde{\lambda}_i$, $a_i$ and $w_i$ be as in the 
	statement of Proposition \ref{prop_ord_degree_shifting_in_all_degrees}. 
	Generalizing Proposition \ref{prop:twisting_by_character} slightly, we note 
	that there is an isomorphism (cf. the proof of Theorem \ref{thm:lgcfl})
	\[ f : \T^{S, \ord}(H^{i[F^+:\Q]}(X_{K(b, c)}, 
	\cV_{\lambda(a_i)})_\m^{\ord}) \cong \T^{S, \ord}(H^{i[F^+:\Q]}(X_{K(b, 
	c)}, \cV_{\lambda})_\m^{\ord}), \]
	which carries $[K^S g K^S]$ to $\epsilon^{-a_i}(\Art_F(\det(g))) [K^S g 
	K^S]$ and satisfies the identity $f\circ \chi_{\lambda(a_i), v, j} = 
	\chi_{\lambda, v, j} \otimes \epsilon^{-a_i}$ ($v \in S_p$). (Note that 
	$a_i$ is divisible by $p-1$, by construction, so we have 
	$\m(\epsilon^{-a_i}) = \m$ in the notation of \S 
	\ref{sec:twisting_and_duality}.) To prove the proposition, it will 
	therefore suffice to prove the analogue of properties (b), (c) for the 
	representation $(f^{-1}\circ\rho_{\m})\otimes \epsilon^{-a_i}$ with 
	coefficients in $\T^{S, \ord}(H^{i[F^+:\Q]}(X_{K(b, c)}, 
	\cV_{\lambda(a_i)})_\m^{\ord})/f^{-1}J_{i[F^+ : \Q]}$, which we already know 
	satisfies the analogue of property (a). In order to simplify notation, we 
	now write $\rho_{\m}$ for this representation, $J_i$ for the ideal 
	$f^{-1}J_{i[F^+ : \Q]}$, and $\chi_{v, j}$ for the 
	character $\chi_{\lambda(a_i), v, j}$ valued in $\T^{S, 
	\ord}(H^{i[F^+:\Q]}(X_{K(b, c)}, 
	\cV_{\lambda(a_i)})_\m^{\ord})$.
	
We obtain from Proposition \ref{prop_ord_degree_shifting_in_all_degrees} a 
surjective algebra homomorphism 
	\[ \widetilde{\T}^{S, \ord}(H^d(\widetilde{X}_{\widetilde{K}(b, c)}, 
	\cV_{\widetilde{\lambda}_i})^{\ord}_{\widetilde{\m}}) \to \T^{S, 
	\ord}(\cO(\alpha_{w_i}) \otimes_{\cO} \tau^{-1}_{w_i} H^{i[F^+:\Q]}(X_{K(b, 
	c)}, \cV_{\lambda(a_i)})_\m^{\ord}).   \]
	Theorem \ref{what we get from CS} says that $H^d(\widetilde{X}_{\widetilde{K}(b, c)}, \cV_{\widetilde{\lambda}_i})^{\ord}_{\widetilde{\m}}$ is a torsion-free $\cO$-module, and Theorem \ref{thm:automorphic reps contributing to char 0} (or rather its proof) shows how to compute $H^d(\widetilde{X}_{\widetilde{K}(b, c)}, \cV_{\widetilde{\lambda}_i})^{\ord}_{\widetilde{\m}} \otimes_{\cO} \overline{\Q}_p$ in terms of cuspidal automorphic representations of $\widetilde{G}(\A_{F^+}^\infty)$. Then \cite[Lemma 5.4]{ger} (which is stated for automorphic representations of $\GL_n$, but which applies here, since $\widetilde{G}$ is split at the $p$-adic places of $F^+$) shows that $\widetilde{\T}^{S, \ord}(H^d(\widetilde{X}_{\widetilde{K}(b, c)}, \cV_{\widetilde{\lambda}_i})^{\ord}_{\widetilde{\m}})[1/p]$ is a semisimple $E$-algebra. By Theorem \ref{thm:base_change_and_existence_of_Galois_for_tilde_G} and \cite[Theorem 2.4]{jackreducible}, we can find a continuous representation 
	\[ \widetilde{\rho} : G_{F, S} \to \GL_{2n}(\widetilde{\T}^{S, \ord}(H^d(\widetilde{X}_{\widetilde{K}(b, c)}, \cV_{\widetilde{\lambda}_i})^{\ord}_{\widetilde{\m}}) \otimes_\cO \overline{\Q}_p) \]
	satisfying the following conditions:
	\begin{enumerate}
		\item For each finite place $v \not\in S$ of $F$, the characteristic polynomial of $\widetilde{\rho}(\Frob_v)$ is equal to the image of $\widetilde{P}_v(X)$.
		\item For each place $v | p$ of $F$, there is an isomorphism
		\numequation\label{eqn:ord_lgc_for_tildeG} \widetilde{\rho}|_{G_{F_v}} \sim \left( \begin{array}{cccc} \psi_{v, 1} & * & * & * \\ 
		0 & \psi_{v, 2} & * & * \\
		\vdots & \ddots & \ddots & * \\
		0 & \cdots & 0 & \psi_{v, 2n} \end{array}\right), 
		\end{equation}
		where for each $i = 1, \dots, 2n$, $\psi_{v, i} : G_{F, v} \to \overline{\Q}_p^\times$ is the continuous character defined as follows. First, if $v \in \widetilde{S}_p$, then $\psi_{v, j}$ is the unique continuous character satisfying the following identities:
		\[ \psi_{v, j} \circ\Art_{F_v}(u) = \epsilon^{1-j}(\Art_{F_v}(u)) \left( \prod_{\tau}\tau(u)^{-(w^{\widetilde{G}}_0 \widetilde{\lambda}_i)_{\tau|_{F^+}, j}} \right) \langle \diag( 1, \dots, 1, u, 1, \dots, 1 ) \rangle \text{ }(u \in \cO_{F_v}^\times) \]
		(the product being over $\tau \in \Hom_{\Q_p}(F_v, E)$) and
		\[ \psi_{v, j} \circ \Art_{F_v}(\varpi_v) = \epsilon^{1-j}(\Art_{F_v}(\varpi_v)) \widetilde{U}_{v, j} / \widetilde{U}_{v, j-1}. \]
		Second, if $v^c \in \widetilde{S}_p$, then $\psi_{v, j} = \psi_{v^c, 2n+1-j}^{c, \vee} \epsilon^{1-2n}$.
	\end{enumerate}
	We write $\widetilde{D}$ for the $2n$-dimensional determinant of $G_{F, S}$ associated to $\widetilde{\rho}$. By \cite[Example 2.32]{chenevier_det}, $\widetilde{D}$ is valued in $\widetilde{\T}^{S, \ord}(H^d(\widetilde{X}_{\widetilde{K}(b, c)}, \cV_{\widetilde{\lambda}_i})^{\ord}_{\widetilde{\m}})$. To conserve notation, we now write 
	\[ \widetilde{A}_0 =  \widetilde{\T}^{S, \ord}(H^d(\widetilde{X}_{\widetilde{K}(b, c)}, \cV_{\widetilde{\lambda}_i})^{\ord}_{\widetilde{\m}}) \]
	and
	\[ A_0 = \T^{S, \ord}(H^{i[F^+:\Q]}(X_{K(b, c)}, 
	\cV_{\lambda(a_i)})_\m^{\ord}), \]
	and $J = J_i$. By construction, we are given a homomorphism $\widetilde{A}_0 \to A_0$ which agrees with $\cS$ on Hecke operators away from $p$, and such that for each $v \in S_p$, the image of the sequence 
	\[ (\psi_{v, 1}, \dots, \psi_{v, 2n}) \]
	 of characters is the image of the sequence 
	 \[ (\chi^{c, \vee}_{v^c, n} \epsilon^{1-2n}, \dots, \chi^{c, \vee}_{v^c, 1} \epsilon^{1-2n}, \chi_{v, 1}, \dots, \chi_{v, n}) \]
	 under the permutation $w_i^{-1}$. 
	 
	The rings $\widetilde{A}_0$ and $A_0$ are semi-local finite $\cO$-algebras. Let $A$ be a local direct factor of $A_0$, and let $\widetilde{A}$ be the corresponding local direct factor of $\widetilde{A}_0$. Thus there is a map $\widetilde{A}\to A$ such that $\widetilde{A} \to A / J$ is surjective. We will show that the properties (b), (c) in the statement of the proposition hold in the ring $A / J$; since $A_0 / J$ is a direct product, this will give the desired result.
	
	We first verify that for each place $v \in S_p$, we have $(\overline{\rho}_{\m}|_{G_{F_v}})^\text{ss} \cong \oplus_{j=1}^n \overline{\chi}_{v, j}$, where the overline denotes reduction modulo the maximal ideal of $A$. By construction, we have
	\[ ((\overline{\rho}_{\m} \oplus \overline{\rho}_\m^{c, \vee} \otimes \epsilon^{1 - 2n})|_{G_{F_v}})^\text{ss} \cong (\overline{\rho}_{\widetilde{\m}}|_{G_{F_v}})^\text{ss} \cong \oplus_{j=1}^n ( \overline{\chi}_{v, j} \oplus \overline{\chi}_{v^c, j}^{c, \vee}  \epsilon^{1-2n}). \]
	Using the existence of the character $\chi$ and a character twisting argument as in the proof of Corollary \ref{cor:first_consequence_for_FL_property}, we see that we also have an isomorphism (over the residue field of $A$)
	\[ ((\overline{\rho}_{\m} \otimes \chi \oplus \overline{\rho}_\m^{c, \vee} \otimes \chi^{c, \vee}\epsilon^{1 - 2n})|_{G_{F_v}})^\text{ss} \cong  \oplus_{j=1}^n ( \overline{\chi}_{v, j} \chi \oplus \overline{\chi}_{v^c, j}^{c, \vee} \chi^{c, \vee}\epsilon^{1-2n}). \]
	Our conditions on the character $\chi$ now force $(\overline{\rho}_\m|_{G_{F_v}})^\text{ss} \cong \oplus_{j=1}^n \overline{\chi}_{v, j}$.
	
	We can now argue in a similar way to the proof of Proposition \ref{prop:first_consequence_for_FL_property}. Let $\widetilde{D}_{A / J} = \widetilde{D} \otimes_A A / J$. Then $\widetilde{D}_{A / J} = \det (\rho_{\m} \oplus \rho_{\m}^{c, \vee} \otimes \epsilon^{1-2n})$. Just as in the proof of Proposition \ref{prop:first_consequence_for_FL_property}, we can identify $(A / J)[G_{F, S}] / \ker(\widetilde{D}_{A / J})$ with $M_n(A/J) \times M_n(A/J)$ (where the first projection gives $\rho_{\m}^{c, \vee} \otimes \epsilon^{1-2n}$, and the second projection gives $\rho_{\m}$).
	
	On the other hand, the map $\widetilde{A}[G_{F, S}] \to (A / J)[G_{F, S}] / (\ker \widetilde{D}_{A / J})$ factors through the quotient $\widetilde{A}[G_{F, S}] / (\ker \widetilde{D})$. There is an algebra embedding
	\[ \widetilde{A}[G_{F, S}] / (\ker \widetilde{D}) \subset \widetilde{A}[G_{F, S}] / (\ker \widetilde{D}) \otimes_{\cO} \overline{\Q}_p \subset M_{2n}( \widetilde{A} \otimes_\cO \overline{\Q}_p ).  \]
	The explicit form of $\widetilde{\rho}|_{G_{F_v}}$ shows that for each $v \in S_p$ and for each sequence of elements $Y, Y_1, \dots, Y_{2n}$ of elements of $\widetilde{A}[G_{F_v}]$, we have
	\numequation \det(X - \widetilde{\rho}(Y)) = \prod_{j=1}^{2n}(X - \psi_{v, j}(Y)) 
	\end{equation}
	in $\widetilde{A}[X]$ and
	\numequation (\widetilde{\rho}(Y_1) - \psi_{v, 1}(Y_1))(\widetilde{\rho}(Y_2) - \psi_{v, 2}(Y_2)) \dots (\widetilde{\rho}(Y_{2n}) - \psi_{v,2n}(Y_{2n})) = 0
	\end{equation}
	in $M_{2n}(\widetilde{A} \otimes_\cO \overline{\Q}_p)$. It follows that the same identities hold in  $\widetilde{A}[G_{F, S}] / (\ker \widetilde{D})$, hence in 
	\[ (A / J)[G_{F, S}] / (\ker \widetilde{D}_{A / J}) = M_n(A/J) \times M_n(A/J). \]
	More precisely,	 for any sequence of elements $Y, Y_1, \dots, Y_{2n}$ of elements of $(A / J)[G_{F_v}]$, we have
		\numequation\label{eqn:ord_char_pol} \det(X - \rho_\m(Y)) \det(X - \rho_\m^{c, \vee} \epsilon^{1-2n}(Y))   = \prod_{j=1}^{2n}(X - \psi_{v, j}(Y)) 
	\end{equation}
	in $(A/J)[X]$ and
	\numequation\label{eqn:ord_det_identity} \left(\prod_{j=1}^{2n} (\rho_\m^{c, \vee} \otimes \epsilon^{1-2n}(Y_j) - \psi_{v, j}(Y_j)),  \prod_{j=1}^{2n} (\rho_\m(Y_j) - \psi_{v, j}(Y_j)) \right) = (0, 0)
	\end{equation}
	in $M_n(A/J) \times M_n(A/J)$ (note that order matters in these products). We need to show how to deduce our desired identities (b), (c) from these ones. We now fix a choice of place $v \in S_p$ for the rest of the proof.
	
	We can find an element $e \in (A / J)[G_{F_v}]$ which acts as 0 on $\overline{\rho}_\m^{c, \vee} \otimes \epsilon^{1-2n}|_{G_{F_{v}}}$ and as the identity in $\overline{\rho}_\m|_{G_{F_v}}$ (because these two representations have no Jordan--H\"older factors in common). By \cite[Ch. III, \S 4, Exercise 5(b)]{bourbaki34} (lifting idempotents), we can assume that $\rho_{\m}(e) = 1$ and $\rho_\m^{c, \vee} \otimes \epsilon^{1-2n}(e) = 0$, and moreover that $\psi_{v, j}(e) = 1$ if $\overline{\psi}_{v, j}$ appears in $\overline{\rho}_\m|_{G_{F_v}}$ (in other words, if $\overline{\psi}_{v, j} = \overline{\chi}_{v, j'}$ for some $1 \leq j' \leq n$, or equivalently if $j = w_i^{-1}(n+k)$ for some $1 \leq k \leq n$), and $\psi_{v, j}(e) = 0$ otherwise. Then applying the identity (\ref{eqn:ord_char_pol}) to $g e \in (A/J)[G_{F_v}]$ gives
	\[ X^n \det(X - \rho_\m(g))  =\prod_{j=1}^{2n}(X - \psi_{v, j}(ge)) = X^n \prod_{j=1}^n (X - \chi_{v, j}(g)), \]
	which is the sought-after property (b) of the proposition. To get property (c), let $g_1, \dots, g_n \in G_{F_v}$, and let $Y_1, \dots, Y_{2n} \in (A/J)[G_{F_v}]$ be defined by $Y_j = e$ if $j \in w_i^{-1}(\{ 1, \dots, n \})$, and $Y_j = g_k e$ if $j = w_i^{-1}(n+k)$, $k \in \{ 1, \dots, n \}$. The identity (\ref{eqn:ord_det_identity}) then becomes
	 \begin{multline*} (0, (\rho_\m(g_1) - \chi_{v, 1}(g_{1}))(\rho_\m(g_{2}) - \chi_{v, 2}(g_{2})) \dots (\rho_\m(g_{n}) - \chi_{v, n}(g_{n}))) = (0, 0) \end{multline*}
	in $M_n(A/J) \times M_n(A/J)$. This completes the proof.
\end{proof}

\subsection{The end of the proof}

We can now complete the proof of the main result of this chapter (Theorem \ref{thm:lgcord_intro}). For the convenience of the reader, we repeat the statement here. We recall our standing hypothesis in this chapter that $F$ contains an imaginary quadratic field in which $p$ splits.
\begin{thm} \label{mySecondAmazingTheorem} \label{thm:lgcord} \label{thm:lgcord_intro}  Suppose that $[F^+ : \Q] > 1$. 
	Let $K \subset \GL_n(\A_F^\infty)$ be a good subgroup such that for each place $v \in S_p$ of $F$, $K_v = \Iw_v$. Let $c \geq b \geq 0$ be integers with $c \geq 1$, let $\lambda \in (\Z^n)^{\Hom(F, E)}$, and let $\m \subset \T^S(K(b, c), \lambda)^{\ord}$ be a non-Eisenstein maximal ideal. Suppose that the following conditions are satisfied:
	\begin{enumerate}
		\item Let $v$ be a finite place of $F$ not contained in $S$, and let $l$ be its residue characteristic. Then either $S$ contains no $l$-adic places of $F$ and $l$ is unramified in $F$, or there exists an imaginary quadratic field $F_0 \subset F$ in which $l$ splits.
		\item $\overline{\rho}_\m$ is decomposed generic.
	\end{enumerate}
	Then we can find an integer $N \geq 1$, which depends only on $[F^+ : \Q]$ and $n$, an ideal $J \subset \T^{S}(K(b, c), \lambda)^{\ord}_\m$ such that $J^N = 0$, and a continuous representation
	\[ \rho_{\m} : G_{F, S} \to \GL_n( \T^{S}(K(b, c), \lambda)^{\ord}_\m / J ) \]
	satisfying the following conditions:
	\begin{enumerate}
		\item[(a)] For each finite place $v \not\in S$ of $F$, the characteristic polynomial of $\rho_{\m}(\Frob_v)$ equals the image of $P_v(X)$ in $(\T^{S}(K(b, c), \lambda)^{\ord}_\m / J)[X]$.
		\item[(b)] For each $v \in S_p$, and for each $g \in G_{F_v}$, the characteristic polynomial of $\rho_{\m}(g)$ equals $\prod_{i=1}^n (X - \chi_{\lambda, v, i}(g))$.
		\item[(c)] For each $v \in S_p$, and for each $g_1, \dots, g_n \in G_{F_v}$, we have
		\[ (\rho_\m(g_1) - \chi_{\lambda, v, 1}(g_1))(\rho_{\m}(g_2) - \chi_{\lambda, v, 2}(g_2)) \dots (\rho_\m(g_n) - \chi_{\lambda, v, n}(g_n)) = 0.  \]
	\end{enumerate}
\end{thm}
\begin{proof}
	Let $0 \leq q \leq d-1$, $m \geq 1$ be integers, and define
	\[ A(K, \lambda, q) = \T^{S, \ord}(H^q(X_{K(b, c)}, \cV_{\lambda})^{\ord}_\m). \]
	and
	\[ A(K, \lambda, q, m) = \T^{S, \ord}(H^q(X_{K(b, c)}, \cV_{\lambda} / \varpi^m)^{\ord}_\m). \] 
	By the same sequence of reductions as in the proof of Theorem \ref{thm:lgcfl}, it is enough to show the existence of an ideal $J \subset A(K, \lambda, q, m)$ satisfying $J^N = 0$ and a continuous representation $\rho_{\m} : G_{F, S} \to \GL_n(A(K, \lambda, q, m) / J)$ satisfying conditions (a), (b) and (c) of the theorem. After an application of the Hochschild--Serre spectral sequence and Corollary \ref{cor:independence_of_level}, we can assume that $c = b \geq m$. Corollary \ref{cor_independence_of_weight} allows us to assume that $\lambda$ is the weight whose existence is asserted by Proposition \ref{prop:first_consequence_for_ord_Galois_reps}. The existence of a Galois representation valued in (quotients by nilpotent ideals of) the Hecke algebras $A(K, \lambda, q)$ and $A(K, \lambda, q+1)$ is then a consequence of Proposition \ref{prop:first_consequence_for_ord_Galois_reps}. The existence of the short exact sequence of $\T^{S, \ord}$-modules
	\begin{multline*} 0 \to H^q(X_{K(b, c)}, \cV_{\lambda})^{\ord}_\m / \varpi^m \to H^q(X_{K(b, c)}, \cV_{\lambda} / \varpi^m)^{\ord}_\m\\ \to H^{q+1}(X_{K(b, c)}, \cV_{\lambda})^{\ord}_\m[\varpi^m] \to 0 
	\end{multline*}
	then implies the existence of a Galois representation $\rho_{\m}$ over a quotient of $A(K, \lambda, q, m)$ by a nilpotent ideal with the required properties. 
\end{proof}

As suggested by a referee, we finish this section by recording a local-global compatibility result for a single automorphic representation. This is a partial generalisation of \cite[Proposition 5.10]{ger} and \cite[Theorem 2.4]{jackreducible}, although we must impose an assumption on the residual Galois representation. In this result, we drop the standing hypothesis that $F$ contains an imaginary quadratic field in which $p$ splits. 
\begin{cor}\label{cor:ordlgc-single-aut-rep}
	Let $F$ be an imaginary CM field, let $\iota: \Qpbar \to \C$ be an isomorphism and let $\pi$ be a cuspidal automorphic representation of $\GL_n(\A_F)$, regular algebraic of weight $\iota\lambda$ for $\lambda \in (\Z^n_+)^{\Hom(F,\Qpbar)}$. Suppose that: 
	\begin{enumerate}
		\item For every $v \in S_p$, $\pi$ is $\iota$-ordinary at $v$ (in the sense of \cite[Def.~5.3]{ger}).
		\item The residual representation $\overline{r_\iota(\pi)}$ is decomposed generic and irreducible.
	\end{enumerate}

Then $r_\iota(\pi)|_{G_{F_v}}$ is ordinary of weight $\lambda$, in the sense of \cite[\S5.2]{ger}, for every $v \in S_p$. More precisely, for each place $v \in S_p$, there is an isomorphism \begin{equation*} r_\iota(\pi)|_{G_{F_v}} \sim \left( \begin{array}{cccc} \psi_{v, 1} & * & * & * \\ 
	0 & \psi_{v, 2} & * & * \\
	\vdots & \ddots & \ddots & * \\
	0 & \cdots & 0 & \psi_{v, n} \end{array}\right), 
\end{equation*}
where for each $i = 1, \dots, n$, $\psi_{v, i} : G_{F, v} \to \overline{\Q}_p^\times$ is the unique continuous character satisfying the identities (cf.~the definition of $\chi_{\lambda,v,i}$ in \S\ref{sec:ord_statements}): 	
\[\psi_{\lambda, v,i}\circ \Art_{F_v}(u)=\epsilon^{1-i}(\Art_{F_v}(u))\left(\prod_{\tau}\tau(u)^{-(w^G_0 \lambda)_{\tau, i}}\right)\langle u\rangle_{\iota,i} \text{ }(u \in \cO_{F_v}^\times) \]
(the product being over $\tau \in \Hom_{\Q_p}(F_v, \Qpbar)$) and, with fixed choices of uniformizers $\varpi_v$ for $v \in S_p$, 
\[\psi_{v,i}\circ \Art_{F_v}(\varpi_v)=\epsilon^{1-i}(\Art_{F_v}(\varpi_v)) \frac{u^{(i)}_{\lambda,\varpi_v}}{u^{(i-1)}_{\lambda,\varpi_v}},\] with $\langle u\rangle_{\iota,i}$ and $u^{(i)}_{\lambda,\varpi_v}$ denoting Hecke eigenvalues on $(\iota^{-1}\pi_v)^{\ord}$ defined in \cite[Definition 5.5]{ger}.
\end{cor}
\begin{proof}
	We make a solvable Galois base change to a CM field extension $F'/F$ which is disjoint over $F$ from the fixed field $\overline{F}^{\ker \overline{r_\iota(\pi)}}$, contains an imaginary quadratic field in which $p$ splits, and in which all the places in $S_p$ split completely. We will also assume that $[F':\Q] > 2$. Using \cite[Lemma 5.7]{ger}, we see that $\pi_{F'}$ is $\iota$-ordinary at $w$ for every place $w|p$ of $F'$ and it suffices to prove the corollary under the additional assumptions that $[F^+:\Q] > 1$ and $F$ contains an imaginary quadratic field in which $p$ splits. Now the result follows from Theorem \ref{thm:lgcord} and Lemma \ref{lem:detord}.
\end{proof}

\section{Automorphy lifting theorems}
\label{section:alt}

\subsection{Statements}

In this chapter, we will prove two automorphy lifting theorems (Theorem \ref{thm:main_automorphy_lifting_theorem} and Theorem \ref{thm:main_ordinary_automorphy_lifting_theorem}) for $n$-dimensional Galois representations of CM fields without imposing a self-duality condition. The first is for Galois representations which satisfy a Fontaine--Laffaille condition.
\begin{theorem}\label{thm:main_automorphy_lifting_theorem}
Let $F$ be an imaginary CM or totally real field, let $c \in \Aut(F)$ be complex conjugation, and let $p$ be a prime. Suppose given a continuous representation $\rho : G_F \to \GL_n(\overline{\bQ}_p)$ satisfying the following conditions:
\begin{enumerate}
\item $\rho$ is unramified almost everywhere.
\item\label{pa2} For each place $v | p$ of $F$, the representation $\rho|_{G_{F_v}}$ is crystalline. The prime $p$ is unramified in $F$.
\item $\overline{\rho}$ is absolutely irreducible and decomposed generic (Definition \ref{defn:decomposed_generic}). The image of $\overline{\rho}|_{G_{F(\zeta_p)}}$ is enormous (Definition \ref{defn:enormous image}). 
\item\label{pa4fl}  \label{part:scalartoremark} There exists $\sigma \in G_F - G_{F(\zeta_p)}$ such that $\overline{\rho}(\sigma)$ is a scalar. 
We have $p > n^2$.
\item There exists a cuspidal automorphic representation $\pi$ of $\GL_n(\bA_F)$ satisfying the following conditions:
\begin{enumerate}
\item $\pi$ is regular algebraic of weight $\lambda$, this weight satisfying 
\[ \lambda_{\tau, 1} + \lambda_{\tau c, 1} - \lambda_{\tau, n} -
  \lambda_{\tau c, n} < p - 2n \]for all~$\tau$.
\item There exists an isomorphism $\iota : \overline{\bQ}_p \to \bC$ such that $\overline{\rho} \cong \overline{r_\iota(\pi)}$ and the Hodge--Tate weights of $\rho$ satisfy the formula for each $\tau : F \hookrightarrow \overline{\bQ}_p$:
\[ \text{HT}_\tau(\rho) = \{ \lambda_{\iota \tau, 1} + n-1, \lambda_{\iota \tau, 2} + n - 2, \dots, \lambda_{\iota \tau, n} \}. \]
\item If $v | p$ is a place of $F$, then $\pi_v$ is unramified.
\end{enumerate}
\end{enumerate}
Then $\rho$ is automorphic: there exists a cuspidal automorphic representation $\Pi$ of $\GL_n(\bA_F)$ of weight $\lambda$ such that $\rho \cong r_\iota(\Pi)$. Moreover, if $v$ is a finite place of $F$ and either $v | p$ or both $\rho$ and $\pi$ are unramified at $v$, then $\Pi_v$ is unramified.
\end{theorem}

The second main theorem is for Galois representations which satisfy an ordinariness condition.
\begin{theorem}\label{thm:main_ordinary_automorphy_lifting_theorem} 
Let $F$ be an imaginary CM or totally real field, let $c \in \Aut(F)$ be complex conjugation, and let $p$ be a prime.  Suppose given a continuous representation $\rho : G_F \to \GL_n(\overline{\bQ}_p)$ satisfying the following conditions:
\begin{enumerate}
\item $\rho$ is unramified almost everywhere.
\item For each place $v | p$ of $F$, the representation $\rho|_{G_{F_v}}$ is potentially semi-stable, ordinary with regular Hodge--Tate weights. In other words, there exists a weight $\lambda \in (\bZ_+^n)^{\Hom(F, \overline{\bQ}_p)}$ such that for each place $v | p$, there is an isomorphism
\[ \rho|_{G_{F_v}} \sim \left( \begin{array}{cccc} \psi_{v, 1} & \ast & \ast & \ast \\
0 & \psi_{v, 2} & \ast & \ast \\
\vdots &\ddots & \ddots & \ast \\
0 & \cdots & 0 & \psi_{v, n}
\end{array}\right), \]
where for each $i = 1, \dots, n$ the character $\psi_{v, i} : G_{F_v} \to \overline{\bQ}_p^\times$ agrees with the character
\[ \sigma \in I_{F_v} \mapsto \prod_{\tau \in \Hom(F_v, \overline{\bQ}_p)} \tau (\Art_{F_v}^{-1}(\sigma))^{-(\lambda_{\tau, n - i + 1} + i - 1)} \]
on an open subgroup of the inertia group $I_{F_v}$.
\item $\overline{\rho}$ is absolutely irreducible and decomposed generic (Definition \ref{defn:decomposed_generic}). The image of $\overline{\rho}|_{G_{F(\zeta_p)}}$ is enormous (Definition \ref{defn:enormous image}). 
\item\label{par4or} There exists $\sigma \in G_F - G_{F(\zeta_p)}$ such that $\overline{\rho}(\sigma)$ is a scalar. 
We have $p > n$.
\item There exists a regular algebraic cuspidal automorphic representation 
$\pi$ of $\GL_n(\bA_F)$ and an isomorphism $\iota : \overline{\bQ}_p \to \bC$ 
such that $\pi$ is $\iota$-ordinary and $\overline{r_\iota(\pi)} \cong \rhobar$.
\end{enumerate}
Then $\rho$ is ordinarily automorphic of weight $\iota \lambda$: there exists an $\iota$-ordinary cuspidal automorphic representation $\Pi$ of $\GL_n(\bA_F)$ of weight $\iota \lambda$ such that $\rho \cong r_\iota(\Pi)$. Moreover, if $v \nmid p$ is a finite place of $F$ and both $\rho$ and $\pi$ are unramified at $v$, then $\Pi_v$ is unramified.
\end{theorem}

\begin{remark}
It follows from the existence of $\Pi$ that the weight $\lambda$ is conjugate self-dual up to twist: there is an integer $w \in \mathbb{Z}$ such that for all $\tau : F \hookrightarrow \mathbb{C}$ and for each $i = 1, \dots, n$, we have $\lambda_{\tau, i} + \lambda_{\tau c, n + 1 - i} = w$. (This in turn is a consequence of the purity lemma of \cite[Lemma 4.9]{MR1044819}.) However, we do not need to assume this at the outset. What we in fact prove is that $\rho$ contributes to the ordinary part of the completed cohomology; we then deduce the existence of $\Pi$ by an argument of ``independence of weight''.
\end{remark}

\begin{remark} 
The image of the projective representation~$\Pbarrho$ coincides with the image of the adjoint representation~$\ad \barrho$.
Hence the first part of conditions~Theorem~\ref{thm:main_automorphy_lifting_theorem}~(\ref{pa4fl}) and 
Theorem~\ref{thm:main_ordinary_automorphy_lifting_theorem}~(\ref{par4or}) are
equivalent to $\zeta_p \not\in \barF^{\ker \ad \barrho}$. If~$p$ is unramified in~$F$ (as in condition (\ref{pa2}) of Theorem~\ref{thm:main_automorphy_lifting_theorem}), it is implied by the non-existence of a surjection $(\ad \barrho)(G_F) \onto (\Z/p\Z)^\times$.
It may be possible to remove the requirement of such a~$\sigma$
by using arguments similar to those of~\cite{jack}, in particular by adding Iwahori level structure at a prime which is not~$1 \mod p$ 
and then using~\cite[Prop.~3.17]{jack}.  
However, this would (at least) necessitate some modifications to the Ihara avoidance arguments of~\S\ref{sec:AvoidIhara}, and so we have not  attempted to do this,
especially because condition~(\ref{part:scalartoremark})  is usually easy to verify in practice.
\end{remark}

The proof of these two theorems will occupy the rest of this chapter. Since this chapter is quite long, we now discuss the structure of the proof. We recall that the authors of \cite{CG} implemented a generalization of the Taylor--Wiles method in situations where the `numerical coincidence' fails to hold, assuming the existence of Galois representations associated to torsion classes in the cohomology of arithmetic locally symmetric spaces, and an appropriate form of local-global compatibility for these Galois representations. They also had to assume that the cohomology groups vanish in degrees outside a given range, after localization at a non-Eisenstein maximal ideal. (This range is the same range in which cohomological cuspidal automorphic representations of $\GL_n$ contribute non-trivially.) Under these assumptions, they proved rather general automorphy lifting theorems; in particular, they were able to implement the `Ihara avoidance' trick of \cite{tay} to obtain lifting results at non-minimal level. 

 There are a few innovations that allow us to obtain unconditional
 results here, building on the techniques of \cite{CG}. The first is
 the proof (in the preceding sections) of a sufficiently strong version of local-global compatibility for the torsion Galois representations constructed in \cite{scholze-torsion}. The second is the observation that one can carry out a version of the `Ihara avoidance' trick under somewhat weaker assumptions than those used in \cite{CG}. Indeed, in \cite{KT}, it was shown that one can prove some kind of automorphy lifting results using only that the rational cohomology is concentrated in the expected range -- and this is known unconditionally, by Matsushima's formula and its generalizations (in particular, Theorem \ref{thm:application_of_matsushima}). Here we show that the `Ihara avoidance' technique is robust enough to give a general automorphy lifting result using only the assumption that the rational cohomology is concentrated in the expected range. 
 
 We now describe the organization of this chapter. As the above discussion may suggest, our arguments are rather intricate, and we have broken them into several parts in the hope that this will make the individual steps easier to digest. We begin in~\S\ref{sec:Galdefthy} by giving a set-up for Galois deformation theory. This is mostly standard, although there are some differences to other works: we do not fix the determinant of our $n$-dimensional Galois representations, and we must prove slightly stronger versions of our auxiliary results (e.g. existence of Taylor--Wiles primes) because of the hypotheses required elsewhere to be able to prove local-global compatibility.
 
 In~\S\ref{sec:AvoidIhara} and~\S\ref{sec:patching}, we carry out the main technical steps. First, in~\S\ref{sec:AvoidIhara}, we give an axiomatic approach to the `Ihara avoidance' technique that applies in our particular set-up. Second, in~\S\ref{sec:patching}, we describe an abstract patching argument that gives as output the objects required in~\S\ref{sec:AvoidIhara}. We find it convenient to use the language of ultrafilters here, following \cite{scholze} and \cite{geenew}. Finally, in~\S\ref{sec:AppToALT}, we combine these arguments to prove Theorem \ref{thm:main_automorphy_lifting_theorem} and Theorem \ref{thm:main_ordinary_automorphy_lifting_theorem}

\subsection{Galois deformation theory}\label{sec:Galdefthy}

Let $E \subset \overline{\Q}_p$ be a finite extension of $\Q_p$, with valuation ring $\cO$, uniformizer $\varpi$, and residue field $k$.
Given a complete Noetherian local $\cO$-algebra $\Lambda$ with residue field $k$, we let $\CNL_{\Lambda}$ denote the category of 
complete Noetherian local $\Lambda$-algebras with residue field $k$. 
We refer to an object in $\CNL_{\Lambda}$ as a $\CNL_{\Lambda}$-algebra.

We fix a number field $F$, and let $S_p$ be the set of places of $F$ above $p$.
We assume that $E$ contains the images of all embeddings of $F$ in $\overline{\Q}_p$.
We also fix a continuous absolutely irreducible homomorphism $\rhobar \colon G_F \rightarrow \GL_n(k)$. 
We assume throughout that $p \nmid 2n$.

\subsubsection{Deformation problems}\label{sec:defprob}

Let $S$ be a finite set of finite places of $F$ containing $S_p$ and all places at which $\rhobar$ is ramified. We write $F_S$ for the maximal subextension of $\overline{F} / F$ which is unramified outside $S$.
For each $v\in S$, we fix $\Lambda_v \in \CNL_{\cO}$, and set $\Lambda = \widehat{\otimes}_{v \in S} \Lambda_v$, where the completed 
tensor product is taken over $\cO$. 
There is a forgetful functor $\CNL_\Lambda \to \CNL_{\Lambda_v}$ for each $v\in S$ via the canonical map $\Lambda_v \rightarrow \Lambda$.
A \emph{lift} (also called a \emph{lifting}) of $\rhobar|_{G_{F_v}}$ is a continuous homomorphism $\rho \colon G_{F_v} \rightarrow \GL_n(A)$ to a $\CNL_{\Lambda_v}$-algebra $A$ 
such that $\rho \bmod \frakm_A = \rhobar|_{G_{F_v}}$.

We let $\cD_v^\square$ denote the set valued functor on $\CNL_{\Lambda_v}$ that sends $A$ to the set of all lifts of $\rhobar|_{G_{F_v}}$ to $A$.
This functor is representable, and we denote the representing object by $R_v^\square$.

 A \emph{local deformation problem} for $\rhobar|_{G_{F_v}}$ is a subfunctor $\cD_v$ of $\cD_v^\square$ satisfying the following:
  \begin{itemize}
   \item $\cD_v$ is represented by a quotient $R_v$ of $R_v^\square$.
   \item For all $A \in \CNL_{\Lambda_v}$, $\rho \in \cD_v(A)$, and $a \in \ker(\GL_n(A) \rightarrow \GL_n(k))$, 
   we have $a\rho a^{-1} \in \cD_v(A)$. 
  \end{itemize}

The notion of global deformation problem that we use in this paper is the following:

\begin{defn}\label{def:globdefprob}
 A \emph{global deformation problem} is a tuple
  \[
   \cS = (\rhobar, S, \{\Lambda_v\}_{v\in S}, \{\cD_v\}_{v\in S}),
  \]
 where:
 \begin{itemize}
  \item $\rhobar$, $S$, and $\{\Lambda_v\}_{v\in S}$ are as above.
  \item For each $v\in S$, $\cD_v$ is a local deformation problem for $\rhobar|_{G_{F_v}}$.
 \end{itemize}
\end{defn}
This differs from that of \cite[\S8.5.2]{CG} and \cite[Definition~4.2]{KT} in that we don't fix the determinant.

As in the local case, 
a \emph{lift} (or \emph{lifting}) of $\rhobar$ is a continuous homomorphism $\rho \colon G_F \rightarrow \GL_n(A)$ to a $\CNL_\Lambda$-algebra $A$, 
such that $\rho \bmod \frakm_A = \rhobar$. 
We say that two lifts $\rho_1,\rho_2 \colon G_F \rightarrow \GL_n(A)$ are \emph{strictly equivalent} if there is 
$a\in \ker(\GL_n(A) \rightarrow \GL_n(k))$ such that $\rho_2 = a\rho_1 a^{-1}$. 
A \emph{deformation of $\rhobar$} is a strict equivalence class of lifts of $\rhobar$.

For a global deformation problem 
  \[
   \cS = (\rhobar, S, \{\Lambda_v\}_{v\in S}, \{\cD_v\}_{v\in S}),
  \]
we say that a lift $\rho \colon G_F \rightarrow \GL_n(A)$ is of \emph{type $\cS$} if $\rho|_{G_{F_v}} \in \cD_v(A)$ for each $v\in S$. 
Note that if $\rho_1$ and $\rho_2$ are strictly equivalent lifts of $\rhobar$, and $\rho_1$ is of type $\cS$, then so is $\rho_2$. 
A \emph{deformation of type $\cS$} is then a strict equivalence class
of lifts of type $\cS$, and we denote by $\cD_{\cS}$ the set-valued functor that 
takes a $\CNL_\Lambda$-algebra $A$ to the set of deformations $\rho \colon G_F \rightarrow \GL_n(A)$ of type $\cS$.

Given a subset $T\subseteq S$, a \emph{$T$-framed lift of type $\cS$} is a 
tuple $(\rho,\{\alpha_v\}_{v\in T})$, where $\rho \colon G_F \rightarrow 
\GL_n(A)$ is a lift of $\rhobar$ of type $\cS$ and $\alpha_v \in \ker(\GL_n(A) 
\rightarrow \GL_n(k))$ for each $v\in T$. 
We say that two $T$-framed lifts $(\rho_1,\{\alpha_v\}_{v\in T})$ and $(\rho_2,\{\beta_v\}_{v\in T})$ to a $\CNL_\Lambda$-algebra $A$ are strictly 
equivalent if there is $a\in \ker(\GL_n(A) \rightarrow \GL_n(k))$ such that $\rho_2 = a \rho_1 a^{-1}$, and $\beta_v = a\alpha_v$ for each $v\in T$.
A strict equivalence class of $T$-framed lifts of type $\cS$ is called a \emph{$T$-framed deformation of type $\cS$}. 
We denote by $\cD_{\cS}^T$ the set valued functor that sends a $\CNL_\Lambda$-algebra $A$ to the set of $T$-framed deformations to $A$ of type $\cS$. 

\begin{thm}\label{thm:representable}
 Let $\cS = (\rhobar, S, \{\Lambda_v\}_{v\in S}, \{\cD_v\}_{v\in S})$ be a global deformation problem, and let $T$ be a subset of $S$.
 The functors $\cD_{\cS}$ and $\cD_{\cS}^T$ are representable; we denote their representing objects by $R_{\cS}$ and $R_{\cS}^T$, respectively.
\end{thm}

\begin{proof}
 This is well known.
 See \cite[Appendix~1]{gouveaparkcity} for a proof of the representability of $\cD_{\cS}$. 
 The representability of $\cD_{\cS}^T$ can be deduced from this.
\end{proof}

If $T = \emptyset$, then tautologically $R_{\cS} = R_{\cS}^T$. 
Otherwise, the relation between these two deformation rings is given by the following lemma.

\begin{lem}\label{thm:framepresentation}
 Let $\cS = (\rhobar, S, \{\Lambda_v\}_{v\in S}, \{\cD_v\}_{v\in S})$ be a global deformation problem, and let $T$ be a nonempty subset of $S$. 
 Fix some $v_0 \in T$, and define $\cT = \cO\llbracket \{X_{v,i,j}\}_{v\in T,1\le i,j\le n}\rrbracket /(X_{v_0,1,1})$. 
 The choice of a representative $\rho_{\cS} \colon G_F \rightarrow \GL_n(R_{\cS})$ for the universal type $\cS$ 
 deformation determines a canonical isomorphism $R_{\cS}^T \cong R_{\cS} \widehat{\otimes}_\cO \cT$.
\end{lem}
\begin{proof}
	This can be proved in the same way as the second part of \cite{cht}, using Schur's lemma. A representative for the universal $T$-framed deformation over $R_{\cS} \widehat{\otimes}_\cO \cT$ is $(\rho_{\cS}, \{ 1 + ( X_{v, i, j}) \}_{v \in T} )$.
\end{proof}

\subsubsection{Some local deformation problems}\label{sec:locdef}

We now fix some finite place $v$ of $F$ and introduce the local deformation rings that we will use in the proofs of our automorphy lifting theorems.

\subsubsection{Ordinary deformations}\label{sec:orddef} \label{sssec:ord}

Assume that $v | p$, and that there is an increasing filtration
    \[
     0 = \overline{\Fil}_v^0 \subset \overline{\Fil}_v^1 \subset \cdots \subset \overline{\Fil}_v^n = k^n
    \]
that is $G_{F_v}$-stable under $\rhobar|_{G_{F_v}}$ with one dimensional graded pieces. 
We will construct and study a local deformation ring $R_{v}^{\detord}$ 
whose corresponding local deformation problem $\cD_v^{\detord}$ will be used in the proof of our ordinary automorphy lifting theorem.

Consider the completed group algebra $\cO\llbracket \cO_{F_v}^\times(p)^n \rrbracket$
where $\cO_{F_v}^\times(p)$ denotes the pro-$p$ completion of 
$\cO_{F_v}^\times$.
There is an isomorphism  $\Art_{F_v} : \cO_{F_v}^\times(p) \to I_{F_v^{ab}/F_v}(p)$. Fix a non-empty set of minimal prime ideals of $\cO\llbracket \cO_{F_v}^\times(p)^n \rrbracket$, and let $\fra$ be their intersection. 
We then set $\Lambda_v = \cO\llbracket \cO_{F_v}^\times(p)^n \rrbracket/\fra$. 

For each $1\le i \le n$, let $\overline{\widetilde{\chi}}_i \colon G_{F_v} \rightarrow k^\times$ denote the character given by $\rhobar|_{G_{F_v}}$ on 
$\overline{\Fil}_v^i/\overline{\Fil}_v^{i-1}$, and let $\overline{\chi}_i = \overline{\widetilde{\chi}}_i|_{I_{F_v}}$.
For each $1 \le i \le n$, we have a canonical character $\chi_i^{\univ} \colon I_{F_v} \rightarrow \Lambda_v^\times$
that is the product of the Teichm\"{u}ller lift of $\overline{\chi}_i$ with the map that sends $I_{F_v}$ 
to the $i$th copy of $\cO_{F_v}^\times(p)$ in $\cO_{F_v}^\times(p)^n$ via~$\Art_{F_v}^{-1}$.
The ideal $\fra$ corresponds to a fixed collection of ordered tuples of characters of the torsion subgroup of $I_{F_v^{ab} / F_v}(p)$. 

We recall some constructions from \cite[\S 3.1]{ger}. We recall that  
$R^\square_v \in \CNL_{\Lambda_v}$ denotes the universal lifting ring of 
$\rhobar|_{G_{F_v}}$.  Let $\cF$ denote the flag variety over $\cO$ classifying 
complete flags $0=\Fil^0 \subset \cdots \subset \Fil^n =\cO^n$, and let  
$\cG_{v}\subset \cF\times_{\Spec \cO}\Spec R^{\square}_{v}$ denote the closed 
subscheme whose $A$-points for an $R^{\square}_{v}$-algebra $A$ consist of 
those filtrations $\Fil \in \cF(A)$ such that for each $i = 1, \dots, n$, 
$\Fil^i$ is preserved by the specialization of the universal lifting to $A$ and 
such that the induced action of $I_{F_v} \subset G_{F_v}$ on $\Fil^i / 
\Fil^{i-1}$ is by the pushforward of the character $\chi^{\univ}_i$.

We now define two ordinary deformation rings:
\begin{itemize}
\item
We define $R^{\triangle}_{v}$ to be the image of the homomorphism
\[ R^{\square}_{v}\to H^0(\cG_v, \cO_{\cG_{v}}).
\]
\item Let $\widetilde{\Lambda}_v = \cO \llbracket F_v^\times(p)^n \rrbracket \otimes_{\cO\llbracket \cO_{F_v}^\times(p)^n \rrbracket} \Lambda_v$, and let $\widetilde{R}_v^\square = R_v^\square \otimes_{\Lambda_v} {\widetilde{\Lambda}_v}$. The characters $\chi^{\univ}_i$ naturally extend to characters $\widetilde{\chi}^{\univ}_i : G_{F_v} \to \widetilde{\Lambda}_v^\times$ lifting $\overline{\widetilde{\chi}}_i$. Let $\widetilde{R}_v^{\detord}$ denote the maximal quotient of $\widetilde{R}_v^\square$ over which the relations
\numequation\label{eqn:det_ord_one} \det(X - \rho^\square(g)) = \prod_{i=1}^n(X - \widetilde{\chi}^{\univ}_i(g)) \end{equation}
and
\numequation\label{eqn:det_ord_two} (\rho^\square(g_1) - \widetilde{\chi}_1^{\univ}(g_1)) (\rho^\square(g_2) - \widetilde{\chi}_2^{\univ}(g_2)) \dots (\rho^\square(g_n) - \widetilde{\chi}_n^{\univ}(g_n)) = 0 \end{equation}
hold for all $g, g_1, \dots, g_n \in G_{F_v}$. We define $R_v^{\detord}$ to be 
the image of the homomorphism
\[ R_v^\square \to \widetilde{R}_v^{\detord}. \]
\end{itemize}
(A ring
similar to~$R_{v}^{\detord}$ was also defined in~\cite{Specter}.) 
\begin{lemma}\label{lem:det_ord_finite}
	$\widetilde{R}_v^{\detord}$ is a finite  $R_v^{\detord}$-algebra.
\end{lemma}
\begin{proof}
	It is enough to show that $\widetilde{R}_v^{\detord}$ is a finite $R_v^\square$-algebra or, by the completed version of Nakayama's lemma, that $\widetilde{R}_v^{\detord} / \ffrm_{R_v^\square}$ is an Artinian $k$-algebra. This follows from the relation (\ref{eqn:det_ord_one}) applied with $g = \Art_{F_v}(\varpi_v)$. 
\end{proof}
For a domain $R\in \CNL_{\Lambda_v}$ and $K$ an algebraic closure of the fraction field of $R$, an $R$-point of $\Spec R^{\square}_{v}$ factors through $\Spec R^{\triangle}_{v}$ if and only if the following condition is satisfied:
\begin{itemize}
\item Let $\rho\colon G_{F_v}\to \GL_n(R)$ be the pushforward of the universal lifting to $R$. Then there is a filtration $0=\Fil^0 \subset \ldots\subset \Fil^n=K^n$ on $\rho\otimes_R K$ which is preserved by $G_{F_v}$, and such that the action of $I_{F_v}$ on $\Fil^i / \Fil^{i-1}$ ($i = 1, \dots, n$)  is given by the push-forward of the universal character $\chi^{\univ}_j$ to $R$.  
\end{itemize}
On the other hand, suppose that $R \to S$ is an injective morphism of $R^\square_v$-algebras, and suppose that there exist characters $\psi_1, \dots, \psi_n : G_{F_v} \to S^\times$ such that for each $i = 1, \dots, n$, $\psi_i|_{I_{F_v}}$ equals the pushforward of $\chi^{\univ}_i$ to $S$, and that for each $g, g_1, \dots, g_n \in G_{F_v}$, the analogues of the relations (\ref{eqn:det_ord_one}) and (\ref{eqn:det_ord_two}) for the characters $\psi_i$ and the pushforward of the universal lifting hold in $S$. Then $R^\square_v \to R$ factors through $R_v^{\detord}$. We see in particular that there is an inclusion of topological spaces $\Spec R^{\triangle}_{v} \subset \Spec R^{\detord}_{v}$: Indeed, applying the above for $R=R^{\triangle}_v/\frakp$, where $\frakp$ is a mininal prime of $R^{\triangle}_v$, and $S$ its integral closure in a sufficiently large finite extension of its fraction field; we see that $R^\square_v\to R^{\triangle}_v/\frakp$ factors through $R^{\detord}_v$. Since the maximal reduced quotient $(R^{\triangle}_v)_{\mathrm{red}}$ of $R^{\triangle}_v$ is the image of the map $R^{\square}_v\to \prod_{\frakp} R^{\triangle}_v/\frakp$, we deduce that there is a surjection of $R^\square_v$-algebras $R^{\detord}_v\onto (R^{\triangle}_v)_{\mathrm{red}}$.

The ring $R_v^\triangle$ was introduced in \cite{ger}. Its properties in an important special case are summarized in the following proposition.
\begin{prop}\label{prop:ordefringproperties}
If $[F_v : \Q_p] > \frac{n(n-1)}{2} + 1$ and $\rhobar|_{G_{F_v}}$ is trivial, then $R^{\triangle}_{v}$ is $\cO$-flat, reduced and equidimensional of dimension $1+n^2 + \frac{n(n+1)}{2}[F_v:\Q_p]$. Moreover, the map $\Spec R^{\triangle}_{v} \rightarrow \Spec \Lambda_v$ is bijective on the level of generic points, hence on the level of irreducible components.
\end{prop}
\begin{proof}
	This is essentially contained in \cite[Proposition~3.14]{jackreducible}. More precisely, that reference proves the proposition under the assumption that $\Lambda_v = \cO\llbracket \cO_{F_v}^\times(p)^n \rrbracket$, but also shows that minimal prime ideals of 
$\Lambda_v$ generate minimal prime ideals of $R^{\triangle}_{v}$. 
The more general case where $\Lambda_v$ is allowed to be a quotient of $\cO\llbracket \cO_{F_v}^\times(p)^n \rrbracket$ by the intersection of 
an arbitrary collection of minimal prime ideals follows from this.
\end{proof}
Our analysis of the ring $R_v^{\detord}$ will be coarser. It begins with the following lemma. 
\begin{lem} \label{lem:detord} 
Let $K$ be a field, $G$ a group and $\rho:G\to \GL_n(K)$ a representation. Suppose that there exist pairwise distinct characters $\chi_1, \dots,\chi_n : G \to K^\times$ satisfying the following conditions:
\begin{enumerate}
\item For all $g \in G$,
\[
\det(X - \rho(g)) = \prod_{i=1}^n (X - \chi_i(g)). %
\]
\item For all $g_1, \dots, g_n \in G$,
\[
(\rho(g_1)-\chi_1(g_1))(\rho(g_2)-\chi_2(g_2))\cdots(\rho(g_n)-\chi_n(g_n))=0. %
\]
\end{enumerate}
Then there is a filtration $0=\Fil^0\subset \dots \subset \Fil^n=K^n$ by $G$-stable subspaces such that for each $i = 1, \dots, n$, $\Fil^i / \Fil^{i-1} \cong K(\chi_i)$.
\end{lem}
\begin{proof} 
We define subspaces $0 = V_0 \subset V_1 \subset V_2 \subset \dots \subset V_n \subset V = K^n$ be declaring that for each $i = 1, \dots, n$, $V_i / V_{i-1}$ is the maximal subspace of $V / V_{i-1}$ where $G$ acts by the character $\chi_i$. Each $V_i$ is $G$-stable and the second condition of the lemma implies that $V_n = V$. On the other hand, each $V_i / V_{i-1}$ is isomorphic to $K(\chi_i)^{\dim_K V_i / V_{i-1}}$. The first condition of the lemma implies that we must therefore have $\dim_K V_i / V_{i-1} = 1$ for each $i = 1, \dots, n$. The proof is complete on taking $\Fil^i = V_i$.
\end{proof}
Let $U\subset \Spec \Lambda_v$ be the open subscheme where the characters $\chi^{\univ}_1, \dots, \chi^{\univ}_n$ are pairwise distinct, and let $Z$ denote its complement.
\begin{prop}\label{prop:detord_components} Let $f : \Spec R_v^\triangle \to \Spec \Lambda_v$, $g : \Spec R_v^{\detord} \to \Spec \Lambda_v$ be the structural maps. Suppose that $\overline{\rho}|_{G_{F_v}}$ is trivial and that $[F_v : \Q_p] > \frac{n(n+1)}{2} + 1$. 
\begin{enumerate}
	\item We have $f^{-1}(U) = g^{-1}(U)$ as subspaces of $\Spec R_v^\square$. Consequently, for each irreducible component $C$ of $\Spec \Lambda_v$, there is a unique irreducible  component~$C'$ of $\Spec R_v^{\detord}$ which dominates $C$. It has dimension $n^2 + 1 + \frac{n(n+1)}{2}[F_v:\Q_p]$.
	\item Let $C'$ be an irreducible component of $R_v^{\detord}$ which does not dominate an irreducible component of $\Spec \Lambda_v$. Then $C' \subset g^{-1}(Z)$ and $C'$ has dimension at most $n^2 - 1+ \frac{n(n+1)}{2}[F_v:\Q_p]$.
\end{enumerate}
\end{prop}
\begin{proof}
	We have already observed that there is an inclusion $\Spec R_v^\triangle \subset \Spec R_v^{\detord}$. We must first show that if $s : \Spec K \to g^{-1}(U) \subset \Spec R_v^{\detord}$ is a geometric point, then $s$ factors through $\Spec R_v^\triangle$. By Lemma \ref{lem:det_ord_finite}, $s$ lifts to a point $s' : \Spec K \to \Spec \widetilde{R}_v^{\detord}$. Then Lemma \ref{lem:detord} shows that $s$ factors through $R_v^\triangle$. The first part of the proposition now follows from Proposition \ref{prop:ordefringproperties}, which says that $f|_U$ induces a bijection on generic points, hence on irreducible components.

           For the second part, let $C'$ be an irreducible component of 
	$R_v^{\detord}$ which does not dominate an irreducible component of $\Spec 
	\Lambda_v$. It follows from the first part that we must have $g(C') \subset 
	Z$. To bound the dimension of $C'$, we claim that there is a permutation 
	$\sigma \in S_n$ such that $C'$ is contained in the closed subspace 
	$h^{-1}(Z)$ of $\Spec R_v^{\triangle, \sigma}$, where $h : \Spec 
	R_v^{\triangle, \sigma} \to \Spec \Lambda_v$ is the quotient of 
	$R_v^\square$ which is defined in the same way as $R_v^\triangle$, except 
	that we require the action of $I_{F_v}$ on the $i^{\text{th}}$ graded piece 
	of the filtration to be by the character $\chi^{\univ}_{\sigma(i)}$. There 
	is a 
	corresponding surjective morphism $\cG^\sigma_v \to \Spec R_v^{\triangle, 
	\sigma}$. To show the claim, it suffices to check that there is a $\sigma$ 
	such that a geometric generic point of $C'$ is contained in $\Spec 
	R_v^{\triangle, \sigma}$. To see this, we observe that the Galois 
	representation corresponding to a geometric generic point of $C'$ has 
	semisimplification a direct sum of characters whose restriction to $I_v$ 
	is the push-forward of $\oplus_{i=1}^n\chi^{\univ}_i$. It 
	follows that this representation has a filtration with the Galois action on 
	its graded pieces given by the universal characters in some order.
	
	We thus have
	\[  \dim C' \leq \dim h^{-1}(Z) \leq \dim \cG_v^\sigma\times_{\Spec \Lambda_v} Z. \]
	We can bound $\dim \cG_v^\sigma\times_{\Spec \Lambda_v} Z$ by bounding the dimension of the completed local rings at its closed points, using essentially the same tangent space calculation as in \cite[Lemma 3.7]{ger} (although over a finite field). This yields
	\[  \dim \cG_v^\sigma\times_{\Spec \Lambda_v}Z \leq 1 + n^2 + n(n+1)/2 + n(n+1)[F_v : \Q_p] / 2 - [F_v : \Q_p] \leq n^2 - 1 + n(n+1) [ F_v : \Q_p]/2, \]
	using our assumption $[F_v : \Q_p] > \frac{n(n+1)}{2} + 1$.  This completes the proof.
\end{proof}

\subsubsection{Fontaine--Laffaille deformations}\label{sec:FLdef}
We again suppose $v | p$, but take $\Lambda_v = \cO$.
We assume that $F_v/\Q_p$ is unramified. 
Recall that in~\S\ref{sec:FL_statements} we defined a category $\mathcal{MF}_{\cO}$ and a functor $\bG$ on $\mathcal{MF}_{\cO}$ 
that take values in the category of finite $\cO$-modules with continuous $\cO$-linear $G_{F_v}$-action.

For each embedding $\tau \colon F_v \hookrightarrow E$, let $\lambda_\tau = (\lambda_{\tau,1},\ldots, \lambda_{\tau,n})$ 
be a tuple of integers satisfying 
  \[ \lambda_{\tau,1} \ge \lambda_{\tau,2} \ge \cdots \ge \lambda_{\tau,n}. \] 
  and
  \[ \lambda_{\tau, 1} - \lambda_{\tau, n} < p - n. \]
We say a representation of $G_{F_v}$ on a finite $\cO$-module $W$ is \emph{Fontaine--Laffaille of type} $(\lambda_\tau)_{\tau \in \Hom(F_v,E)}$ 
if there is $M \in \mathcal{MF}_{\cO}$ with $W \cong \bG(M)$, and
  \[
   \mathrm{FL}_\tau(M \otimes_{\cO} k) = \{\lambda_{\tau,1}+n-1, \lambda_{\tau,2} + n-2, \ldots, \lambda_{\tau,n}\}
  \]
for each $\tau \colon F_v \hookrightarrow E$. 
The following proposition follows from \cite[\S2.4.1]{cht} and a twisting argument (see~\S\ref{sec:FL_statements}).

\begin{prop}\label{thm:FLring}
 Assume that $\rhobar|_{G_{F_v}}$ is Fontaine--Laffaille of type $(\lambda_\tau)_{\tau \in \Hom(F_v,E)}$. 
 Then there is a quotient $R_v^{\mathrm{FL}}$ of $R_v^\square$ satisfying the following.
 \begin{enumerate}
  \item $R_v^{\mathrm{FL}}$ represents a local deformation problem $\cD_v^{\mathrm{FL}}$.
  \item For a $\CNL_{\cO}$-algebra $A$ that is finite over $\cO$, a lift $\rho \in \cD_v^\square(A)$ lies in $\cD_v^{\mathrm{FL}}$ if and only if 
  $\rho$ is Fontaine--Laffaille of type $(\lambda_\tau)_{\tau \in \Hom(F_v,E)}$. 
  \item $R_v^{\mathrm{FL}}$ is a formally smooth over $\cO$ of dimension $1 + n^2 + \frac{n(n-1)}{2}[F_v: \Q_p]$.
 \end{enumerate}
\end{prop}

\subsubsection{Level raising deformations}\label{sec:Iharadef}
Assume that $q_v \equiv 1 \bmod p$, that $\rhobar|_{G_{F_v}}$ is trivial, and that $p > n$.
We take $\Lambda_v = \cO$.

Let $\chi = (\chi_1,\ldots,\chi_n)$ be a tuple of continuous characters $\chi_i \colon \cO_{F_v}^\times \rightarrow \cO^\times$ that are trivial modulo $\varpi$. 
We let $\cD_v^\chi$ be the functor of lifts $\rho \colon G_{F_v} \rightarrow \GL_n(A)$ such that
  \[
   \mathrm{char}_{\rho(\sigma)}(X) = \prod_{i=1}^n (X - 
   \chi_i(\Art_{F_v}^{-1}(\sigma)))
  \]
for all $\sigma \in I_{F_v}$. 
Then $\cD_v^\chi$ is a local deformation problem, and we denote its representing object by $R_v^\chi$. 
The following two propositions are contained in \cite[Proposition~3.1]{tay}.

\begin{prop}\label{thm:Ihara1ring}
 Assume that $\chi_i = 1$ for all $1\le i \le n$. 
 Then $R_v^1$ satisfies the following properties: 
 \begin{enumerate}
  \item $\Spec R_v^1$ is equidimensional of dimension $1+n^2$ and every generic point has characteristic zero.
  \item Every generic point of $\Spec R_v^1/\varpi$ is the specialization of a unique generic point of $\Spec R_v^1$.
 \end{enumerate}
\end{prop}

\begin{prop}\label{thm:Iharachiring}
 Assume that the $\chi_i$ are pairwise distinct. 
 Then $\Spec R_v^\chi$ is irreducible of dimension $1+n^2$, and its generic point has characteristic zero. 
\end{prop}

\subsubsection{Taylor--Wiles deformations}\label{sec:TWdef}
Assume that $q_v \equiv 1 \bmod p$, and that $\rhobar|_{G_{F_v}}$ is unramified. 
We take $\Lambda_v = \cO$.
We assume that $\rhobar|_{G_{F_v}}$ has $n$-distinct eigenvalues $\alpha_1,\ldots,\alpha_n \in k$. 
For each $1\le i \le n$, let $\overline{\gamma}_i \colon G_{F_v} \rightarrow k^\times$ be the unramified character that sends $\Frob_v$ to $\alpha_i$. 

\begin{lem}\label{thm:anyliftisTW}
 Let $\rho \colon G_{F_v} \rightarrow \GL_n(A)$ be any lift of $\rhobar$. 
 There are unique continuous characters $\gamma_i \colon G_{F_v} \rightarrow A^\times$, for $1\le i \le n$, such that
 $\rho$ is $\GL_n(A)$-conjugate to a lift of the form $\gamma_1 \oplus \cdots \oplus \gamma_n$, where $\gamma_i \bmod \frakm_A = \overline{\gamma}_i$ for each $1 \le i \le n$.
\end{lem}

\begin{proof}
 This is similar to \cite[Lemma~2.44]{MR1605752}. 
 The details are left to the reader.
\end{proof}

Let $\Delta_v = k(v)^\times(p)^n$, where $k(v)^\times(p)$ is the maximal $p$-power quotient of $k(v)^\times$.
Let $\rho \colon G_{F_v} \rightarrow \GL_n(R_v^\square)$ denote the universal lift. Then $\rho$ is $\GL_n(R_v^\square)$-conjugate to a lift of the form $\gamma_1 \oplus \cdots \oplus \gamma_n$, with $\gamma_i \bmod \frakm_{R^\square} = \overline{\gamma}_i$. 
For each $1\le i \le n$, the character $\gamma_i \circ \Art_{F_v}|_{\cO_{F_v}^\times}$ factors through $k(v)^\times(p)$, so we obtain a canonical 
local $\cO$-algebra morphism $\cO[\Delta_v] \rightarrow R_v^\square$. 
Note that this depends on the choice of ordering $\alpha_1,\ldots,\alpha_n$.
It is straightforward to check that this morphism is formally smooth of relative dimension $n^2$. 

\subsubsection{Formally smooth deformations}\label{sec:formsmoothdef}
Assume that $v \nmid p$.
The following is a standard argument in obstruction theory, and the proof is left to the reader.

\begin{prop}\label{thm:smoothlifts}
 If $H^2(F_v,\ad\rhobar) = 0$, then $R_v^\square$ is isomorphic to a power series ring over $\cO$ in $n^2$ variables.
\end{prop}

\subsubsection{Presentations}\label{sec:present}
Fix a global deformation problem
  \[
   \cS = (\rhobar, S, \{\Lambda_v\}_{v\in S}, \{\cD_v\}_{v\in S}),
  \]
and for each $v\in S$, let $R_v$ denote the object representing $\cD_v$.
Let $T$ be a (possibly empty) subset of $S$ such that $\Lambda_v = \cO$ for all $v\in S \smallsetminus T$, 
and define $R_{\cS}^{T,\loc} = \widehat{\otimes}_{v\in T} R_v$, with the completed tensor product being taken over $\cO$. 
It is canonically a $\Lambda$-algebra, via the canonical isomorphism $\widehat{\otimes}_{v\in T} \Lambda_v \cong \widehat{\otimes}_{v\in S} \Lambda_v$. 
For each $v\in T$, the morphism $\cD_{\cS}^T \rightarrow \cD_v$ given by $(\rho,\{\alpha_v\}_{v\in T}) \mapsto \alpha_v^{-1}\rho|_{G_{F_v}} \alpha_v$ 
induces a local $\Lambda_v$-algebra morphism $R_v \rightarrow R_{\cS}^T$. 
We thus have a local $\Lambda$-algebra morphism $R_{\cS}^{T,\loc} \rightarrow R_{\cS}^T$. 
To understand the relative tangent space of this map, we use a Galois cohomology complex following \cite[\S2]{cht} (cf. \cite[\S4.2]{KT}).

We let $\ad\rhobar$ denote the space of $n\times n$ matrices $\mathrm{M}_{n\times n}(k)$ over $k$ with adjoint $G_F$-action via $\rhobar$. 
For each $v\in S$, we let $Z^1(F_v,\ad\rhobar)$ denote the $k$-vector space of continuous $1$-cocycles of $G_{F_v}$ with coefficients in $\ad\rhobar$. 
The map $c \mapsto (1+\varepsilon c)\rhobar$ gives an isomorphism
  \[
   Z^1(F_v,\ad\rhobar) \xrightarrow{\sim} \Hom_{\CNL_{\Lambda_v}}(R_v^\square,k[\varepsilon]/(\varepsilon^2)). 
  \]
We denote by $\cL_v^1$ the pre-image of 
  \[
   \Hom_{\CNL_{\Lambda_v}}(R_v,k[\varepsilon]/(\varepsilon^2)) \subseteq \Hom_{\CNL_{\Lambda_v}}(R_v^\square,k[\varepsilon]/(\varepsilon^2))
  \]
under this isomorphism. 
Note that $\cL_v^1$ contains the subspace of coboundaries. 
We then let $\cL_v$ be the image of $\cL_v^1$ in $H^1(F_v,\ad\rhobar)$. 

We define a complex $C_{\cS,T}^\bullet(\ad\rhobar)$ by
  \[
  C_{\cS,T}^i(\ad\rhobar) = 
  \begin{cases}
   C^0(F_S/F,\ad\rhobar) & \text{if } i = 0, \\
   C^1(F_S/F,\ad\rhobar) \oplus \bigoplus_{v\in T} C^0(F_v,\ad\rhobar) & \text{if } i =1, \\
   C^2(F_S/F,\ad\rhobar) \oplus \bigoplus_{v\in T} C^1(F_v,\ad\rhobar) \oplus_{v\in S\smallsetminus T} C^1(F_v,\ad\rhobar)/\cL_v^1 & \text{if } i = 2, \\
   C^i(F_S/F,\ad\rhobar) \oplus \bigoplus_{v\in S} C^{i-1}(F_v,\ad\rhobar) & \text{otherwise,}
  \end{cases}
  \] 
with boundary map $C_{\cS,T}^i(\ad\rhobar)  \rightarrow C_{\cS,T}^{i+1}(\ad\rhobar)$ given by
  \[
   (\phi,(\psi_v)_v) \mapsto (\partial \phi, (\phi|_{G_{F_v}} - \partial \psi_v)_v).
  \]
We denote the cohomology groups of this complex by $H_{\cS,T}^i(\ad\rhobar)$, 
and denote their $k$-dimension by $h_{\cS,T}^i(\ad\rhobar)$ 
(we use similar notation for the $k$-dimension of local and global Galois cohomology groups).

There is a long exact sequence in cohomology
  \begin{align}
   0 & \rightarrow H_{\cS,T}^0(\ad\rhobar) \rightarrow H^0(F_S/F,\ad\rhobar) \rightarrow \oplus_{v\in T} H^0(F_v,\ad\rhobar)  \label{eq:defthyexact} \\
   & \rightarrow H_{\cS,T}^1(\ad\rhobar) \rightarrow H^1(F_S/F,\ad\rhobar) \rightarrow \oplus_{v\in T} H^1(F_v,\ad\rhobar) 
   \oplus_{v\in S \smallsetminus T} H^1(F_v,\ad\rhobar)/\cL_v \nonumber \\ 
   & \rightarrow H^2_{\cS,T}(\ad\rhobar) \rightarrow H^2(F_S/F,\ad\rhobar) \rightarrow \oplus_{v\in S} H^2(F_v,\ad\rhobar) \nonumber
    \rightarrow \cdots.
  \end{align}
Since we are assuming that $p > 2$, the groups $H^i(F_S/F,\ad\rhobar)$ vanish for $i \ge 3$, as do the groups $H^i(F_v,\ad\rhobar)$. 
So $H_{\cS,T}^i(\ad\rhobar) = 0$ for $i > 3$, and we have a relation among Euler characteristics
  \numequation \label{eq:Euler}
   \chi_{\cS,T}(\ad\rhobar) = \chi(F_S/F,\ad\rhobar)  - \sum_{v\in S} \chi(F_v,\ad\rhobar) 
   - \sum_{v\in S\smallsetminus T} (\dim_k \cL_v - h^0(F_v,\ad\rhobar)).
  \end{equation}
The trace pairing $(X,Y) \mapsto \tr(XY)$ on $\ad\rhobar$ is perfect and $G_F$-equivariant, so $\ad\rhobar(1)$ is isomorphic to the 
Tate dual of $\ad\rhobar$. 
For each $v\in S$, we let $\cL_v^\perp \subseteq H^1(F_v,\ad\rhobar(1))$ be the exact annihilator of $\cL_v$ under local Tate duality. 
We then define 
  \[
   H_{\cS^\perp,T}^1(\ad\rhobar(1)) = \ker \left( H^1(F_S/F,\ad\rhobar(1)) \rightarrow 
   \prod_{v\in S \smallsetminus T} H^1(F_v,\ad\rhobar(1))/\cL_v^\perp \right).
  \]
The following is proved in the same way as \cite[Proposition~4.7]{KT}, based on ideas of Kisin~\cite[Prop.\ 4.1.5, Rem.\ 4.1.7]{MR2459302}.
  
\begin{prop}\label{thm:genoverloc}
 Let the notation and assumptions be as in the beginning of~\S\ref{sec:present}, assume further that $T$ is nonempty. 
 Then there is a local $\Lambda$-algebra surjection $R_{\cS}^{T,\loc}\llbracket X_1,\ldots,X_g \rrbracket \rightarrow R_{\cS}^T$, with 
    \begin{multline*}
    g =  h_{\cS,T}^1(\ad\rhobar)  = h_{\cS^\perp,T}^1(\ad\rhobar(1)) - h^0(F_S/F,\ad\rhobar(1)) \\
      - \sum_{v\mid \infty} h^0(F_v,\ad\rhobar) + \sum_{v \in S \smallsetminus T} (\dim_k \cL_v - h^0(F_v,\ad\rhobar)).
    \end{multline*}
\end{prop}

\begin{proof}
 The first claim with $g = h_{\cS,T}^1(\ad\rhobar)$ follows from showing 
 \begin{align*}
   H_{\cS,T}^1(\ad\rhobar) & \cong \Hom_{\CNL_\Lambda}(R_{\cS}^T/(\frakm_{R_{\cS}^{T,\loc}}), k[\varepsilon]/(\varepsilon^2))\\
  & \cong \Hom_k(\frakm_{R_{\cS}^T}/(\frakm_{R_{\cS}^T}^2,\frakm_{R_{\cS}^{T,\loc}}), k). 
  \end{align*}
 To see this, note that any $T$-framed lifting of $\rhobar$ to $k[\varepsilon]/(\varepsilon^2)$ can be written as $((1+\varepsilon \kappa)\rhobar, (1+\varepsilon\alpha_v)_{v\in T})$, 
 with $\kappa \in Z^1(F_S/F,\ad\rhobar)$, and $\alpha_v \in \ad\rhobar$. 
 It is the trivial lift at $v\in T$ if and only if
  \[
   (1-\varepsilon\alpha_v)(1+\varepsilon \kappa|_{G_{F_v}})\rhobar|_{G_{F_v}} (1+\varepsilon\alpha_v) = \rhobar|_{G_{F_v}},
  \]
 equivalently,
  \[
   \kappa|_{G_{F_v}} = (\ad\rhobar|_{G_{F_v}} - 1)\alpha_v.
  \]
 Such a lift is further of type $\cS$ if and only if $\kappa|_{G_{F_v}} \in \cL_v^1$ for all $v\in S\smallsetminus T$. 
 This sets up a bijection between the set of $1$-cocycles of the complex $C_{\cS,T}^\bullet(\ad\rhobar)$ and the set of $T$-framed lifts 
 of type $\cS$ that are trivial at $v\in T$. 
 Two cocycles $(\kappa,\{\alpha_v\}_{v \in T})$ and
 $(\kappa',\{\alpha'\}_{v\in T})$ define strictly equivalent
 $T$-framed lifts if and only 
 if there is $\beta\in \ad\rhobar$ such that
  \[
   \kappa' = \kappa + (\ad\rhobar - 1)\beta \quad \text{and} \quad \alpha_v' = \alpha_v + \beta,
  \]
 for all $v\in T$, i.e. if and only if they differ by a coboundary.
 This induces the desired isomorphism 
  \[
   H_{\cS,T}^1(\ad\rhobar) \cong \Hom_{\CNL_\Lambda}(R_{\cS}^T/(\frakm_{R_{\cS}^{T,\loc}}), k[\varepsilon]/(\varepsilon^2)).
  \]
 
 Since $T$ is nonempty, $h_{\cS,T}^0(\ad\rhobar) = 0$. 
 Then \eqref{eq:Euler} together with the local and global Euler characteristic formulas imply 
  \[
   h_{\cS,T}^1(\ad\rhobar) = h_{\cS,T}^2(\ad\rhobar) - h_{\cS,T}^3(\ad\rhobar) - \sum_{v\mid \infty} h^0(F_v,\ad\rhobar) + \sum_{v\in S\smallsetminus T} (\dim_k \cL_v - h^0(F_v,\ad\rhobar)).
  \]
 To finish the proof, we deduce equalities $h_{\cS,T}^2(\ad\rhobar) = h_{\cS^\perp,T}^1(\ad\rhobar(1))$ and $h^3_{\cS,T}(\ad\rhobar) = h^0(F_S/F,\ad\rhobar(1))$ 
 by comparing the exact sequence
  \begin{align*}
   & \rightarrow H^1(F_S/F,\ad\rhobar)  \rightarrow \oplus_{v\in T} H^1(F_v,\ad\rhobar) 
   \oplus_{v\in S \smallsetminus T} H^1(F_v,\ad\rhobar)/\cL_v \\
    \rightarrow H_{\cS,T}^2(\ad\rhobar) & \rightarrow H^2(F_S/F,\ad\rhobar)  \rightarrow \oplus_{v\in S} H^2(F_v,\ad\rhobar)  \\
    \rightarrow H_{\cS,T}^3(\ad\rhobar) & \rightarrow 0,
  \end{align*}
 which is part of \eqref{eq:defthyexact}, with the exact sequence
   \begin{align*}
   & \rightarrow H^1(F_S/F,\ad\rhobar)  \rightarrow \oplus_{v\in T} H^1(F_v,\ad\rhobar) 
   \oplus_{v\in S \smallsetminus T} H^1(F_v,\ad\rhobar)/\cL_v \\
    \rightarrow H^1_{\cS^\perp,T}(\ad\rhobar(1))^\vee & \rightarrow H^2(F_S/F,\ad\rhobar)  \rightarrow \oplus_{v\in S} H^2(F_v,\ad\rhobar)  \\
    \rightarrow H^0(F_S/F,\ad\rhobar(1))^\vee & \rightarrow 0,
  \end{align*}
 which is part of the Poitou--Tate long exact sequence. 
\end{proof}

We will apply this with our choices of local deformation rings as in~\S\ref{sec:locdef}.
By applying Propositions~\ref{thm:FLring}, \ref{thm:Ihara1ring}, \ref{thm:Iharachiring}, \ref{thm:smoothlifts}, and \cite[Lemma 3.3]{blght}, we obtain the following:

\begin{lemma}\label{lem:localFLcase}
We assume that our deformation problem $\cS$ and $T\subseteq S$ satisfy the following. 
\begin{itemize}
 \item $T$ is a disjoint union $S_p \sqcup R \sqcup S_a$.
 \item For each $v\in S_p$, we assume that $F_v/\Q_p$ is unramified and that $\rhobar|_{G_{F_v}}$ is as in Proposition~\ref{thm:FLring}.
 We take $\cD_v = \cD_v^{\mathrm{FL}}$.
 \item For each $v\in R$, we assume that $q_v \equiv 1 \bmod p$ and that $\rhobar|_{G_{F_v}}$ is trivial. 
 We take $\calD_v = \calD_v^{\chi_v}$ for some tuple $\chi_v = (\chi_{v,1},\ldots,\chi_{v,n})$ of characters $\chi_{v,i} \colon \cO_{F_v}^\times \rightarrow \cO^\times$ that are trivial modulo $\varpi$. 
 \item For each $v\in S_a$, we assume that $H^2(F_v,\ad\rhobar) = 0$ and we take $\calD_v = \calD_v^\square$.
\end{itemize}
Then $R_{\cS}^{T,\loc}$ satisfies the following properties.
\begin{enumerate}
 \item Assume that $\chi_{v,i} = 1$ for each $v\in R$ and $1\le i \le n$. 
 Then $\Spec R_{\cS}^{T,\loc}$ is equidimensional of dimension $1+n^2|T| + \frac{n(n-1)}{2}[F:\Q]$, and every generic point has characteristic $0$.
 Further, every generic point of $\Spec R_{\cS}^{T,\loc}/\varpi$ is the specialization of a unique generic point of $\Spec R_{\cS}^{T,\loc}$.
 \item Assume that $\chi_{v,1},\ldots,\chi_{v,n}$ are pairwise distinct for each $v\in R$. 
 Then $\Spec R_{\cS}^{T,\loc}$ is irreducible of dimension $1+n^2|T| + \frac{n(n-1)}{2}[F:\Q]$ and its generic point has characteristic $0$.
\end{enumerate}
\end{lemma}

In the ordinary case, we will use the following.

\begin{lemma}\label{lem:localordcase}
We assume that our deformation problem $\cS$ and $T\subseteq S$ satisfy the following. 
\begin{itemize}
 \item $T$ is a disjoint union $S_p \sqcup R \sqcup S_a$. 
 \item For each $v\in S_p$, we assume that $[F_v:\Q_p]> \frac{n(n+1)}{2} + 1$ and that $\rhobar|_{G_{F_v}}$ is trivial.
 We take $\Lambda_v$ to be the quotient of $\cO\llbracket \cO_{F_v}^\times(p)^n \rrbracket$ by a minimal prime ideal $\wp_v$ 
 and take $\cD_v = \cD_v^{\detord}$ to be the local deformation problem classified by $R_{v}^{\detord}$. 
 \item For each $v\in R$, we assume that $q_v \equiv 1 \bmod p$ and that $\rhobar|_{G_{F_v}}$ is trivial. 
 We take $\calD_v = \calD_v^{\chi_v}$ for some tuple $\chi_v = (\chi_{v,1},\ldots,\chi_{v,n})$ of characters $\chi_{v,i} \colon \cO_{F_v}^\times \rightarrow \cO^\times$ that are trivial modulo $\varpi$. 
 \item For each $v\in S_a$, we assume that $H^2(F_v,\ad\rhobar) = 0$ and we take $\calD_v = \calD_v^\square$.
\end{itemize}
Then $R_{\cS}^{T,\loc}$ satisfies the following properties.
\begin{enumerate}
 \item Assume that $\chi_{v,i} = 1$ for each $v\in R$ and $1\le i \le n$.
 Then $\Spec R_{\cS}^{T,\loc}$ has dimension $1+n^2|T| + \frac{n(n+1)}{2}[F:\Q]$, any irreducible component of maximum dimension has a characteristic $0$ generic point, 
 and any irreducible component that does not have maximum dimension has dimension $\le n^2|T| - 1 + \frac{n(n+1)}{2}[F:\Q]$. 
 Further, any irreducible component of $\Spec R_{\cS}^{T,\loc}/(\lambda)$ of maximum dimension is the specialization of a unique generic point of $\Spec R_{\cS}^{T,\loc}$.
 \item Assume that $\chi_{v,1},\ldots,\chi_{v,n}$ are pairwise distinct for each $v\in R$. 
 Then $\Spec R_{\cS}^{T,\loc}$ has dimension $1+n^2|T| + \frac{n(n+1)}{2}[F:\Q]$, it has a unique irreducible component of maximum dimension 
 and the generic point of this irreducible component has characteristic $0$. 
 Any other irreducible component has dimension $\le n^2|T| - 1 + \frac{n(n+1)}{2}[F:\Q]$. 
 \item If $x$ is a point of $\Spec R_{\cS}^{T,\loc}$ lying in an irreducible component of non-maximum dimension, 
 then there is some $v\in S_p$ such that the image of $x$ in $\Spec \Lambda_v$ lies in the closed locus defined by $\chi_i^{\univ} = \chi_j^{\univ}$ 
 for some $i\ne j$.
\end{enumerate}
\end{lemma}

\begin{proof}
For each $v\in S_p$, Proposition \ref{prop:detord_components} implies that $\Spec R_v$ has a unique irreducible component of dimension $\dim R_v = 1 + n^2 + \frac{n(n+1)}{2}[F_v : \Q_p]$, 
and this irreducible component has characteristic $0$. 
Let $\mathfrak{q}_v$ be the minimal prime of $R_v$ corresponding to this irreducible component.
Then we can apply \cite[Lemma 3.3]{blght} to
  \[
   R' = \widehat{\otimes}_{v\in S_p} R_v/\mathfrak{q}_v \widehat{\otimes}_{v\in R \cup S_a} R_v
  \]
together with Propositions~\ref{thm:Ihara1ring}, \ref{thm:Iharachiring}, and \ref{thm:smoothlifts} to obtain the following:
\begin{enumerate}
 \item If $\chi_{v,i} = 1$ for each $v\in R$ and $1\le i \le n$, 
 then $\Spec R'$ is equidimensional of dimension $1+n^2|T| + \frac{n(n+1)}{2}[F:\Q]$, and every generic point has characteristic $0$.
 Further, every generic point of $\Spec R'/\varpi$ is the specialization of a unique generic point of $\Spec R'$.
 \item If $\chi_{v,1},\ldots,\chi_{v,n}$ are pairwise distinct for each $v\in R$, 
 then $\Spec R'$ is irreducible of dimension $1+n^2|T| + \frac{n(n+1)}{2}[F:\Q]$ and its generic point has characteristic $0$.
\end{enumerate}
Since any minimal prime $\frakp$ of $R_{\cS}^{T,\loc}$ pulls back to minimal prime ideals $\mathfrak{p}_v$ of $R_v$ for each $v\in T$, and induces a surjection
  \[
   \widehat{\otimes}_{v\in T} R_v/\mathfrak{p}_v \rightarrow R_{\cS}^{T,\loc}/\mathfrak{p},
  \]
we see that $\Spec R'$ is a union of irreducible components of $\Spec R_{\cS}^{T,\loc}$. 
To finish the proof of the lemma, it suffice to note that if $\mathfrak{p}_v \ne \mathfrak{q}_v$ for some $v\in S_p$, 
then by Proposition \ref{prop:detord_components}, $\dim R_v/\mathfrak{p}_v \le n^2 - 1 + \frac{n(n+1)}{2}[F_v:\Q_p]$ 
and the image of $\mathfrak{p}_v$ in $\Lambda_v$ lies in the closed locus defined by $\chi_i^{\univ} = \chi_j^{\univ}$ for some $i\ne j$.
In this case, $\dim R_{\cS}^{T,\loc}/\mathfrak{p} \le n^2|T| - 1 + \frac{n(n+1)}{2}[F:\Q]$.
\end{proof}

\subsubsection{Taylor--Wiles primes}\label{sec:TWprimes}

In this section we show how to generate Taylor--Wiles data. We first need to introduce a definition, essentially equivalent to that of~\cite[Defn.\ 4.10]{KT} and~\cite[\S 9.2]{CG} (see Remark \ref{rmk:enormous} below). For the moment, let $k$ be any algebraic extension of $\F_p$.
\begin{defn}\label{defn:enormous} \label{defn:enormous image}
	Let $\ad^0$ denote the space of trace zero matrices in $\mathrm{M}_{n\times n}(k)$ with the adjoint $\GL_n(k)$-action. 
	An absolutely irreducible subgroup $H \subseteq \GL_n(k)$ is called \emph{enormous} over $k$ if it satisfies the following:
	\begin{enumerate}
		\item $H$ has no nontrivial $p$-power order quotient.
		\item $H^0(H,\ad^0) = H^1(H,\ad^0) = 0$.
		\item \label{thirdcondition} For any simple $k[H]$-submodule $W \subseteq \ad^0$, there is a regular semisimple $h\in H$
		such that $W^h \ne 0$.
	\end{enumerate}
\end{defn}
Note that this only depends on the image of $H$ in $\PGL_n(k)$. If $p$ divides $n$, then no subgroup of $\GL_n(k)$ is enormous (because $\ad^0$ contains the scalar matrices). 
\begin{lem}
	Let $k' / k$ be an algebraic extension, and let $H \subset \GL_n(k)$ be a subgroup. Then $H$ is enormous over $k$ if and only if it is enormous over $k'$.
\end{lem}
\begin{proof}
	It suffices to address  condition~(\ref{thirdcondition}), which is equivalent to the following statement: for all non-zero $k[H]$-submodules $W \subseteq \ad^0$, there is a regular semisimple element $h \in H$ such that $W^h \neq 0$. This makes it clear that if $H$ is enormous over $k'$, then it is enormous over $k$.
	
	Suppose therefore that $H$ is enormous over $k$. The property that a~$k'[H]$-module~$V$ satisfies~$V^h = 0$ is closed under taking
	direct sums and taking quotients (the latter is true because~$V^h \ne 0$ if and only if~$V_{h} \ne 0$).
	If~$V \subset \ad^0 \otimes_{k} k'$, 
	then $\sigma h v = h \sigma v$ for all~$\sigma \in \Gal(k'/k)$ (since~$H \subset \GL_n(k)$)  and so~$V^h \ne 0$ if and only 
	if~$(\sigma V)^{h} \ne 0$.
	In particular, if~$W'$ is a simple~$k'[H]$-submodule of~$\ad^0 \otimes_k k'$ with no invariants by~$h \in H$,
	the same is true for~$\sigma W'$ for all~$\sigma \in \Gal(k'/k)$, as well as the submodule of~$\ad^0 \otimes_k k'$ generated by the sum of all such~$\sigma W'$. But the latter is stable under both~$H$ and~$\Gal(k'/k)$, and thus (by descent) has the form~$W \otimes_k k'$ for some~$k[H]$-submodule of~$\ad^0$. But now applying
	condition~(\ref{thirdcondition}) to
	any~$k[H]$-simple submodule of~$W$, we deduce that~$W^h \ne 0$ for some regular semisimple~$h$, from which it
	follows that the same holds for~$W'$.
\end{proof}
Henceforth we drop the `over $k$' and refer simply to enormous subgroups of $\GL_n(k)$.
\begin{remark}\label{rmk:enormous}
	Assuming that $k$ is sufficiently large to contain all eigenvalues of the elements of $H$, 
	it can be checked that Definition \ref{defn:enormous} is equivalent to \cite[Definition~4.10]{KT}. 
\end{remark}
We now return to the assumptions described at the beginning of \S \ref{sec:Galdefthy}, assuming further that $k$ contains all eigenvalues of the elements of $\rhobar(G_F)$.
We again fix a global deformation problem
  \[
   \cS = (\rhobar, S, \{\Lambda_v\}_{v\in S}, \{\cD_v\}_{v\in S}).
  \]
We define a \emph{Taylor--Wiles datum} to be a tuple $(Q,(\alpha_{v,1},\ldots,\alpha_{v,n})_{v\in Q})$ consisting of:
\begin{itemize}
 \item A finite set of finite places $Q$ of $F$, disjoint from $S$, such that $q_v \equiv 1 \bmod p$ for each $v\in Q$.
 \item For each $v\in Q$, an ordering $\alpha_{v,1},\ldots,\alpha_{v,n}$ of the eigenvalues of $\rhobar(\Frob_v)$, 
 which are assumed to be $k$-rational and distinct.
\end{itemize}
Given a Taylor--Wiles datum $(Q,(\alpha_{v,1},\ldots,\alpha_{v,n})_{v\in Q})$, 
we define the augmented global deformation problem
  \[
   \cS_Q = (\rhobar, S\cup Q, \{\Lambda_v\}_{v\in S} \cup \{\cO\}_{v\in Q}, \{\cD_v\}_{v\in S} \cup \{\cD_v^\square\}_{v\in Q}).
  \]
Set $\Delta_Q = \prod_{v\in Q} k(v)^\times(p)^n$.
By~\S\ref{sec:TWdef}, the fixed ordering $\alpha_{v,1},\ldots,\alpha_{v,n}$, for each $v\in Q$, 
determines a $\Lambda[\Delta_Q]$-algebra structure on $R^T_{\cS_Q}$ for any subset $T$ of $S$.
Letting $\fra_Q = \ker(\Lambda[\Delta_Q] \rightarrow \Lambda)$ be the augmentation ideal, 
the natural surjection $R^T_{\cS_Q} \rightarrow R^T_{\cS}$ has kernel $\fra_Q R^T_{\cS_Q}$.

\begin{lem}\label{thm:TWprimes}
 Let $T \subseteq S$. 
 Assume that $F$ is CM with maximal totally real subfield $F^+$, that $\zeta_p \notin F$, and that $\rhobar(G_{F(\zeta_p)})$ is enormous. 
 Let $q \ge h_{\cS^\perp,T}^1(\ad\rhobar(1))$. 
 Then for every $N\ge 1$, there is a choice of Taylor--Wiles datum $(Q_N,(\alpha_{v,1},\ldots,\alpha_{v,n})_{v\in Q_N})$ satisfying the following:
 \begin{enumerate}
  \item $\# Q_N = q$.
  \item For each $v\in Q_N$, $q_v \equiv 1 \bmod p^N$, and $v$ has degree one over $\Q$.
  \item $h_{\cS_{Q_N}^\perp,T}^1(\ad\rhobar(1)) = 0$.
 \end{enumerate}
\end{lem}

\begin{proof}
 Since the augmented deformation datum $\cS_{Q_N}$ has $\cD_v = \cD_v^\square$ for $v\in Q_N$, we have $\cL_v = H^1(G_v,\ad\rhobar)$ and
  \[
   H_{\cS_{Q_N}^\perp,T}^1(\ad\rhobar(1)) = \ker\left(H_{S^\perp,T}^1(\ad\rhobar(1)) \rightarrow \prod_{v\in Q_N} H^1(F_v,\ad\rhobar(1))\right).
  \]
 So by induction, it suffices to show that given any cocycle $\kappa$ representing a nonzero element of $H_{\cS^\perp,T}^1(\ad\rhobar(1))$, 
 there are infinitely many finite places $v$ of $F$ such that
  \begin{itemize}
  \item $v$ has degree one over $\Q$ and splits in $F(\zeta_{p^N})$;
  \item $\rhobar(\Frob_v)$ has $n$-distinct eigenvalues $\alpha_{v,1},\ldots,\alpha_{v,n}$ in $k$;
  \item the image of $\kappa$ in $H^1(F_v,\ad\rhobar(1))$ is nonzero.
 \end{itemize}
 The set of places of $F$ that have degree one over $\Q$ has density one, so it suffices to show that the remaining properties are satisfied 
 by a positive density set of places of $F$.
 Then by Chebotarev density, we are reduced to showing that given any cocycle $\kappa$ representing a nonzero element of $H_{\cS^\perp,T}^1(\ad\rhobar(1))$, 
 there is $\sigma \in G_{F(\zeta_{p^N})}$ such that
  \begin{itemize}
   \item $\rhobar(\sigma)$ has distinct $k$-rational eigenvalues;
   \item $p_\sigma \kappa(\sigma) \ne 0$, where $p_\sigma \colon \ad\rhobar \rightarrow (\ad\rhobar)^\sigma$ is the $\sigma$-equivariant projection.
  \end{itemize} (The second condition guarantees that the restriction
  of~$\kappa$ will not be a coboundary.)
 Since $p \nmid n$, we have a $G_F$-equivariant decomposition $\ad\rhobar = k \oplus \ad^0\rhobar$, and we treat separately the cases where $\kappa$ represents 
 a cohomology class in $H^1(F_S/F,\ad^0\rhobar(1))$ and in $H^1(F_S/F,k(1))$. 

 First assume that $\kappa$ represents a cohomology class in $H^1(F_S/F,\ad^0\rhobar(1))$. 
 Let $L/F$ be the splitting field of $\rhobar$. 
 The definition of enormous implies that the restriction map 
  \[
   H^1(F_S/F,\ad^0\rhobar(1)) \rightarrow H^1(F_S/L(\zeta_{p^N}),\ad^0\rhobar(1))^{G_F}
  \]
 is injective. 
 Indeed, letting $H =\rhobar(G_{F(\zeta_p)})$, since $H$ has no $p$-power order quotients, $H = \rhobar(G_{F(\zeta_{p^N})})$ and $H^0(H,\ad^0\rhobar) = 0$ implies that
 the restriction to $H^1(F_S/F(\zeta_{p^N}),\ad^0\rhobar)$ is injective. 
 Then the condition $H^1(H,\ad^0\rhobar) = 0$ implies that the further restriction to $H^1(F_S/L(\zeta_{p^N}),\ad^0\rhobar(1))$ is injective.
 So the restriction of $\kappa$ defines a nonzero $G_{F(\zeta_{p^N})}$-equivariant homomorphism $\Gal(F_S/L(\zeta_{p^N})) \rightarrow \ad^0\rhobar$. 

 Let $W$ be a nonzero irreducible subrepresentation in the $k$-span of
 $\kappa(\Gal(F_S/L(\zeta_{p^N}))$. The enormous assumption implies that there is $\sigma_0 \in G_{F(\zeta_{p^N})}$ such that $\rhobar(\sigma_0)$ has distinct $k$-rational eigenvalues
 and such that %
 $W^{\sigma_0}\ne 0$. 
 This implies that $\kappa(\Gal(F_S/L(\zeta_{p^N}))$ is not contained in the kernel of the $\sigma_0$-equivariant projection
 $p_{\sigma_0} \colon \ad^0\rhobar \rightarrow (\ad^0\rhobar)^{\sigma_0}$. 
 If $p_{\sigma_0} \kappa(\sigma_0) \ne 0$, then we take $\sigma = \sigma_0$. 
 Otherwise, we choose $\tau \in G_{L(\zeta_{p^N})}$ such that $p_{\sigma_0}\kappa(\tau) \ne 0$, and we take $\sigma = \tau\sigma_0$. 
 This does the job since $\rhobar(\sigma) = \rhobar(\sigma_0)$ and $\kappa(\sigma) = \kappa(\sigma_0) + \kappa(\tau)$. 

 Now assume that $\kappa$ represents a cohomology class in $H^1(F_S/F,k(1))$. 
 The cohomology class of $\kappa$ corresponds to a Kummer extension $F(\zeta_p,y)$ with $y^p \in F(\zeta_p)$. 
 Since $\kappa$ is nontrivial and $\zeta_p \notin F$, this extension $F(\zeta_p,y)$ is not abelian over $F$. 
 It follows that $y^p \notin F(\zeta_{p^N})$ for any $N\ge 1$, and the restriction of $\kappa$ to $G_{F(\zeta_{p^N})}$ is nontrivial. 
 Since the extension $F(\zeta_{p^N},y)/F(\zeta_{p^N})$ has degree $p$, 
 it is disjoint from the extension cut out by the restriction of $\rhobar$ to $G_{F(\zeta_{p^N})}$
 by the enormous assumption. 
 It follows that we can find $\sigma \in G_{F(\zeta_{p^N})}$ such that $\rhobar(\sigma)$ has distinct eigenvalues and 
 such that $\kappa(\sigma) \ne 0 \in k$. 
 This completes the proof.
\end{proof}

\begin{prop}\label{thm:TWgen}
 Take $T=S$, and let $q \ge h_{\cS^\perp,T}^1(\ad\rhobar(1))$. 
 Assume that $F = F^+ F_0$ with $F^+$ totally real and $F_0$ an imaginary quadratic field, 
 that $\zeta_p \notin F$, and that $\rhobar(G_{F(\zeta_p)})$ is enormous.
 Then for every $N\ge 1$, there is a choice of Taylor--Wiles datum $(Q_N,(\alpha_{v,1},\ldots,\alpha_{v,n})_{v\in Q_N})$ satisfying the following:
 \begin{enumerate}
  \item $\# Q_N = q$.
  \item For each $v\in Q_N$, $q_v \equiv 1 \bmod p^N$ and the rational prime below $v$ splits in $F_0$.
  \item There is a local $\Lambda$-algebra surjection $R_{\cS}^{T,\loc}\llbracket X_1,\ldots,X_g \rrbracket \rightarrow R_{\cS_{Q_N}}^T$, with
  \[
   g = qn - n^2[F^+:\Q].
  \]
 \end{enumerate}
\end{prop}

\begin{proof}
 If $v$ is a finite place of $F$ that is degree one over $\Q$, then the rational prime below it must split in $F_0$.
 So Proposition \ref{thm:genoverloc} and Lemma \ref{thm:TWprimes} imply that the proposition holds with
  \[
   g = - h^0(F_S/F,\ad\rhobar(1)) - n^2[F^+:\Q] + \sum_{v\in Q} (\dim \cL_v - \dim h^0(F_v,\ad\rhobar)).
  \]
 The assumptions that $\rhobar(G_{F(\zeta_p)})$ is enormous and that $\zeta_p \not\in F$ imply that $H^0(F_S/F,\ad\rhobar(1))$ is trivial. 
 For each $v\in Q$, we have $\cL_v = H^1(F_v,\ad\rhobar)$, so 
  \[
   \dim \cL_v - \dim h^0(F_v,\ad\rhobar) = h^0(F_v,\ad\rhobar(1)) = n,
  \]
 where the first equality follows from local Tate duality and the local Euler characteristic, 
 and the second from the fact that $q_v \equiv 1 \bmod p$ and $\rhobar(\Frob_v)$ has distinct eigenvalues.
\end{proof}

\subsection{Avoiding Ihara's lemma}\label{sec:AvoidIhara}

In this section we will axiomatically explain how to deduce a patched automorphy theorem from the result of the patching process. See section \ref{sec:AvoidIharasetup} and particularly Proposition \ref{prop:full_support_for_unipotent_deformation_rinG}. We begin however with a little commutative algebra.

\subsubsection{Some commutative algebra}

\begin{lem}\label{comalg1} Suppose that $T$ is an excellent local ring with $\Spec T$ irreducible, that $f \in \gm_T$ and that $T/(f)$ has Krull dimension $0$. 

If $T$ has dimension $0$ then for every finitely generated $T$ module $M$ we have
\[ \lg_{T}(M/fM) - \lg_{T}(M[f]) = 0 \]
(and these lengths are both finite).

Otherwise $T$ has dimension $1$ and a unique prime ideal $\gp$ other than $\gm_T$. In this case there is a constant $a\in \Z_{>0}$ such that for any finitely generated $T$-module $M$ we have
\[ \lg_{T}(M/fM) - \lg_{T}(M[f]) = a \lg_{T_{\gp}}(M_{\gp}) \]
(and all these lengths are finite).
\end{lem}

\begin{proof} If $T$ has dimension $0$ then it is Artinian and every finitely generated $T$-module has finite length. If the desired length equality holds for two modules in a short exact sequence it also holds for the third (by the snake lemma). Thus we are reduced to checking the lemma in the case $M=T/\gm_T$, in which case it is obvious.

Now suppose that $T$ has dimension $1$. Note that $T/(f)$ is Artinian and so any finitely generated $T/(f)$-module has finite length over $T$. Let $\tT$ denote the normalization of $T/\gp$. As $T$ is excellent, $\tT$ is a finitely generated $T$-module. We will take $a = \lg_T(\tT/(f))$. (This is positive because $f$ is not a unit in $\tT$. In fact it lies in every maximal ideal.)

Note that the conclusion of the lemma is true for $M=T/\gm$ and for $M=\tT$. Also note that, if the conclusion of the lemma holds for two modules in a short exact sequence, then it holds for the third (again by the snake lemma). In particular the lemma holds for all finite length $T$-modules. 

Filtering $M$ by the submodules $\gp^iM$ we reduce to checking the lemma for $M$ a $T/\gp$-module.
Write $Q$ for the quotient $\tT/(T/\gp)$. It has support $\{\gm_T\} \subset \Spec T$. If $M$ is any finitely generated $T/\gp$-module we have an exact sequence
\[  \Tor_1^T(M,Q) \lra M \lra M \otimes_T \tT \lra M \otimes_T Q \lra (0). \]
Both $M \otimes_T Q$ and $\Tor_1^T(M,Q)$ are finitely generated $T$-modules with support contained in $\{ \gm_T\}$, and hence of finite length. Thus we are reduced to proving the lemma for $M$ a finitely generated $\tT$-module.

Note that $\tT$ is a Dedekind domain with only finitely many maximal ideals, and hence a PID. By the structure theorem for finitely generated modules over a PID it suffices to check the conclusion of the lemma in the two cases: $M$ is a finite length $\tT$-module and $M=\tT$. However we have already treated both these cases.
\end{proof}

We will really make use of a derived version of this lemma. Suppose that $S$ is a ring, that $T$ is a noetherian $S$-algebra and that $C \in D^b(S)$ is equipped with a map $T \ra \End_{D^b(S)}(C)$ over $S$ such that the cohomology of $C$ has finite length over $T$. Then we define
\[ \lg_T(C) = \sum_i (-1)^i \lg_T(H^i(C)). \]
Note that if 
\[ C_1 \lra C_2 \lra C_3 \lra \]
is an exact triangle in $D^b(S)$ with compatible actions of $T$, and if two of the $C_i$ have cohomology of finite length over $T$, then so does the third and we have
\[ \lg_T(C_2) = \lg_T(C_1)+\lg_T(C_3). \]
Note also that if $f \in S$ and the cohomology of $C$ is finitely generated over $T$, then the cohomology of $C \otimes^{\bL}_S S/(f)$ is also finitely generated over $T$. (Look at the long exact sequence in cohomology coming from the exact triangle
\[ (C \stackrel{f}{\lra} C \lra C \otimes^{\bL}_S S/(f) \lra )\,\,=\,\, C \otimes^{\bL}_S (S \stackrel{f}{\lra} S \lra S/(f) \lra )\,\,\, .)\]

Before stating the derived version we need one other remark:

\begin{lem} Suppose that $A$ is a noetherian ring and that $\gm$ is a maximal ideal of $A$ that is simultaneously a minimal prime ideal. Then $A \iso A_\gm \times B$ for some ring $B$. \end{lem}

\begin{proof} Let $\gp_1,...,\gp_r$ denote the other minimal prime ideals of $A$ and set $I=\gp_1 \cap ... \cap \gp_r$. If $\gm \supset I$ then $\gm \supset \gp_i$ for some $i$, a contradiction. Thus $\gm +I=A$ and $A/(\gm \cap I) \iso A/\gm \times A/I$. However $\gm \cap I$ is nilpotent so we can lift the idempotent $(1,0) \in A/\gm \times A/I$ to an idempotent $e \in A$. Then $1-e \in \gm$ and $e$ will lie in every prime ideal of $A$ other than $\gm$. Thus $\gm=e\gm \times (1-e)A \subset eA \times (1-e)A$, and every other prime ideal of $A$ contains $e$ and so has the form $eA \times \gq$. In particular $eA$ is Artinian local and $A_\gm=(eA)_{e\gm}=eA$. \end{proof}

Suppose that $S$ is a noetherian ring, that $T$ is a finite $S$-algebra, that $\gp$ is a minimal prime ideal of $T$ and that $C \in D^b(S)$ is equipped with a map $T \ra \End_{D^b(S)}(C)$ over $S$. Let $\gq$ denote the contraction of $\gp$ to $S$. As $T$ is finite over $S$, we see that $\gp$ is also maximal ideal in $T_\gq$, and so by the above lemma we can write $T_\gq \cong T_\gp \times B$ for some $S$-algebra $B$. Let $e_\gp \in T_\gq$ be the idempotent corresponding $(1,0) \in T_\gp \times B$. Then, perhaps by an abuse of notation, we will write
\[ C_\gp = e_\gp (C \otimes_S S_\gq). \]
It is an object of $D^b(S_\gq)$ with an action of $T_\gp$. It is not literally a localization over $T$, but if $C$ (with its action of $T$) happens to be represented by a complex of $T$-modules $C^i$, then $C_\gp$ is represented by the complex $C_\gp^i$. Moreover if the cohomology of $C$ is finitely generated over $S$, then the cohomology of $C_\gp$ has finite length over $T_\gp$ (being finitely generated over the Artinian ring $T_\gp$).

\begin{lem}\label{comalg2} Suppose that $S$ is an excellent local ring and that $f \in \gm_S$ is a non-zero divisor. Suppose also that $T$ is a finite $S$-algebra with a maximal ideal $\gm$ such that $\Spec T_\gm$ is irreducible and $T_\gm/(f)$ has Krull dimension $0$. Note that $T_\gm$ has dimension at most $1$. 

If $T_\gm$ has dimension $0$, then for every $C \in D^b(S)$ such that $T \ra \End_{D^b(S)}(C)$ over $S$ and $C$ has finitely generated cohomology we have  
\[ \lg_{T_\gm}((C \otimes^{\bL}_S S/(f))_\gm)=0. \]

If not, then $T_\gm$ has a unique prime ideal $\gp$ other than $\gm$. In this case there is $a \in \Z_{>0}$ with the following property:

Suppose that $C \in D^b(S)$ such that $C$ has finitely generated cohomology, and that $T \ra \End_{D^b(S)}(C)$ over $S$. Then 
\[ \lg_{T_\gm}((C \otimes^{\bL}_S S/(f))_\gm)= a \lg_{T_{\gp}}(C_\gp). \]
\end{lem} 

\begin{proof} 
We will take the $a$ as in lemma \ref{comalg1} for the ring $T_\gm$. If the lemma holds for two terms in an exact triangle it holds for the third term too. Thus one may inductively reduce to the case that $C$ is quasi-isomorphic to $M[i]$ for a finitely generated $S$-module $M$ with a compatible action of $T$. In this case $C \otimes^{\bL}_S S/(f)$ is quasi-isomorphic to $(M \stackrel{f}{\ra} M)[-i]$. Moreover 
\[ \lg_{T_\gm}((C \otimes_S^{\bL} S/(f))_\gm) = (-1)^i (\lg_{T_\gm}((M)/fM)_\gm) - \lg_{T_\gm}(M[f]_\gm)) \]
and
\[ \lg_{T_{\gp}} (C_\gp) = (-1)^i \lg_{T_{\gp}}(M_\gp). \]
Thus the present lemma follows from lemma \ref{comalg1}.
\end{proof}

We remark that if $C$ is perfect then the cohomology of $C$ and $C \otimes_S S/(f)$ will automatically be finitely generated (over $S$ and hence over $T$). In this case, if $T$ is an $S$ subalgebra of $\End_{D^b(S)}(C)$, then it will automatically be finite over $S$. \vspace{.5cm}

\subsubsection{Application}\label{sec:AvoidIharasetup}

Let $\Lambda$ be a ring which is isomorphic to a power series ring over $\cO$. We assume given the following objects:
\begin{enumerate}
\item A power series ring $S_\infty=\Lambda[[X_1,\cdots,X_r]]$ with augmentation ideal $\mathfrak{a}_\infty = (X_1, \dots, X_r)$. 
\item Perfect complexes $C_\infty, 
C'_\infty$ of $S_\infty$-modules, and a fixed isomorphism 
\[ C_\infty \otimes^\bL_{S_\infty} S_{\infty} / \varpi \cong C'_\infty \otimes^\bL_{S_\infty} S_{\infty} / \varpi \]
in $\mathbf{D}(S_\infty / \varpi)$.
\item Two $S_\infty$-subalgebras
\[ T_\infty\subset 
\End_{\bD(S_\infty)}(C_\infty) \]
 and
 \[ T'_\infty \subset 
\End_{\bD(S_\infty)}(C'_\infty), \]
which have the same image in
\[ \End_{\bD(S_\infty/\varpi)}(C_\infty \otimes^\bL_{S_\infty} S_{\infty} / \varpi ) = \End_{\bD(S_\infty/\varpi)}(C'_\infty \otimes^\bL_{S_\infty} S_{\infty} / \varpi ), \]
where these endomorphism algebras are identified using the fixed isomorphism. Call this common image 
$\overline{T}_\infty$. Note that $T_\infty$ and $T'_\infty$ are finite 
$S_\infty$-algebras.
\item Two Noetherian complete local $S_\infty$-algebras $R_\infty$ and $R'_\infty$ and 
surjections $R_\infty\onto T_\infty/I_\infty$, $R'_\infty\onto 
T'_\infty/I'_\infty$, where $I_\infty$ and $I'_\infty$ are nilpotent ideals. 
We write $\overline{I}_\infty$ and $\overline{I}'_\infty$ for the image of these 
ideals in $\overline{T}_\infty$. Note that it then makes sense to 
talk about the support of $H^*(C_\infty)$ and $H^*(C'_\infty)$ over $R_\infty$, 
$R'_\infty$, even though they are not genuine modules over these rings. These 
supports actually belong to the closed subsets of $\Spec R_\infty$, $\Spec 
R'_\infty$ given by $\Spec T_\infty$, $\Spec T'_\infty$, and hence are finite 
over $\Spec S_\infty$.
\item An isomorphism $R_\infty/\varpi\cong R'_\infty/\varpi$ 
compatible with the $S_\infty$-algebra structure and the actions 
(induced from $T_\infty$ and $T'_\infty$) on 
\[ H^*( C_\infty \otimes^\bL_{S_\infty} S_{\infty} / \varpi)/(\overline{I}_\infty+\overline{I}'_\infty) = H^*( C'_\infty \otimes^\bL_{S_\infty} S_{\infty} / \varpi)/(\overline{I}_\infty+\overline{I}'_\infty), \]
where these cohomology groups are identified using the fixed isomorphism. 
\item Integers $q_0 \in \bZ$ and $l_0 \in \bZ_{\geq 0}$.
\end{enumerate}

\begin{assumption}    \label{assumptionsetup} Our set-up is assumed to satisfy the following: 
\begin{enumerate}
\item $\dim R_\infty=\dim R'_\infty=\dim S_\infty -l_0$, and $\dim R_\infty/\varpi=\dim R'_\infty/\varpi=\dim S_\infty-l_0-1$. 
\item (Behavior of components) Assume that each generic point of $\Spec 
R_\infty/\varpi$ of maximal dimension (i.e. of dimension $\dim R_\infty -1$) is the specialization of a unique generic point of $\Spec 
R_\infty$ of dimension $\dim R_\infty$, and $\Spec R'_\infty$ has a unique generic point $x'$ of dimension $\dim R_\infty$. Assume also that any generic points of $\Spec R_\infty$, $\Spec R'_\infty$, $\Spec R_\infty/\varpi$ which are not of maximal dimension have dimension 
$<\dim S_\infty-l_0-1$.

These hypotheses imply every generic point 
of $\Spec R_\infty$ and $\Spec R'_\infty$ of dimension $\dim R_\infty$ has characteristic $0$. 
\item (Generic concentration) \label{part:genericconcentration} 
There exists a dimension 1 characteristic 0 prime $\mathfrak{p}$ of $S_\infty$ containing $\mathfrak{a}_\infty$ such that  
\[ H^\ast(C_\infty\otimes_{S_\infty}^{\bL} S_\infty/\mathfrak{p} )[\frac{1}{p}]\neq 0, \]
and these groups are non-zero only for degrees in the interval $[q_0,q_0+l_0]$.
\end{enumerate}
\end{assumption}

Note that $\Supp_{R_\infty}(H^*(C_\infty))=\Spec T_\infty$ and $\Supp_{R_\infty'}(H^*(C_\infty'))=\Spec T_\infty'$. (The only point being that the kernel of $T_\infty \ra \End_{S_\infty}(H^*(C_\infty))$ is nilpotent.)

The following result is an immediate corollary of Lemma \ref{comalg2}.

\begin{lem}\label{comalg3} Suppose that $\barx$ is a minimal prime of $R_\infty/(\varpi)$ 
of dimension $\dim S_\infty -l_0-1$ containing a minimal prime $x$ of $R_\infty$ 
of dimension $\dim S_\infty-l_0$. Note that $R_{\infty,\barx}$
has a unique minimal prime ideal $x$, and that $R_{\infty,\barx}/(\varpi)$ 
has Krull dimension $0$. Moreover:
\begin{enumerate} 
\item If $\barx \in \Spec T_\infty$ 
and if $\lg_{T_{\infty,\barx}}((C_{\infty} \otimes^{\bL}_{S_\infty} S_\infty/(\varpi))_{\barx}) \neq 0$ 
then $x \in \Spec T_\infty$ 
and $\lg_{T_{\infty,x}}(C_{\infty,x})  \neq 0$. 

\item If $x \in \Spec T_\infty$ 
then $\barx  \in \Spec T_\infty$. 
If moreover $\lg_{T_{\infty,x}}(C_{\infty,x}) \neq 0$ 
then $\lg_{T_{\infty,\barx}}((C_{\infty} \otimes^{\bL}_{S_\infty} S_\infty/(\varpi))_{\barx}) \neq 0$. 
\end{enumerate}
The same is true with $R_\infty'$, $T_\infty'$, $C_\infty'$ replacing $R_\infty$, $T_\infty$ and $C_\infty$.
\end{lem}

We now come to the principal result of this section.

\begin{prop}\label{prop:full_support_for_unipotent_deformation_rinG} With the notation and assumptions just established, 
\[ \Supp_{R_\infty}(H^*(C_\infty)) = \Spec T_\infty \subset \Spec R_\infty\]
 contains every irreducible component of $\Spec R_\infty$ of maximal dimension. \end{prop}

\begin{proof} As $T_\infty$ and $H^*(C_\infty)$ are finite over $S_\infty$ we see that $\Supp_{S_\infty}(H^*(C_\infty))$ is the image of $\Spec T_\infty \ra \Spec S_\infty$. Thus $\Supp_{S_\infty}(H^*(C_\infty))$ has dimension at most $\dim R_\infty=\dim S_\infty -l_0$, and any prime in 
$\Supp_{S_{\infty,\gp}}(H^*(C_{\infty,\gp}))$ has codimension   at least $l_0$ in $\Spec S_{\infty,\gp}$. 

Since $S_{\infty,\mathfrak{p}}/\mathfrak{p}\cong (S_\infty/\mathfrak{p})[\frac{1}{p}]$, our assumptions imply that $C_{\infty,    \mathfrak{p}} \otimes^{\bL} S_{\infty,\gp}/\gp$ is non-zero and has cohomology concentrated in degrees $[q_0,q_0+l_0]$. Thus $C_{\infty,\gp}$ is quasi-isomorphic to a perfect complex of $S_{\infty,\gp}$-modules concentrated in degrees  $[q_0,q_0+l_0]$. (See for instance \cite[Lemma 2.3]{KT}.) From the key \cite[Lemma 6.2]{CG} we deduce that $H^\ast(C_{\infty, 
  \mathfrak{p}})$ is non-zero exactly in degree $q_0 + l_0$ and that $\Supp_{S_{\infty,\gp}}(H^{q_0+l_0}(C_{\infty, 
  \mathfrak{p}}))$ contains a prime of codimension at most $l_0$ in $\Spec S_{\infty,\gp}$.
Thus $\Supp_{S_\infty}(H^*(C_\infty))$ contains a prime of  dimension $\dim S_\infty -l_0$. Let $x_1$ denote a pre-image of this prime in $\Spec T_\infty$, so that $x_1$ must be a generic point of $\Spec R_\infty - \Spec R_\infty/(\varpi)$ of dimension $\dim S_\infty -l_0$. Moreover $\lg_{T_{\infty,x_1}}(C_{\infty,x_1}) \neq (0)$. Choose a generic point $\barx_1$ of $\Spec R_\infty/(x_1,\varpi)$, which must have dimension $\dim S_\infty-l_0-1$ and be a generic point of $\Spec R_\infty/(\varpi)$ in the image of $\Spec T_\infty$. Let $\barx_1'$ denote the corresponding point of $\Spec R_\infty'/(\varpi)$. It can not be generic in $\Spec R_\infty'$ and so must generalize to $x'$.

Now let $x_2$ be any other generic point of $\Spec R_\infty$ of dimension $\dim S_\infty -l_0$. We wish to show that it lies in $\Spec T_\infty$. Choose a generic point $\barx_2$ of $\Spec R_\infty/(x_2,\varpi)$, which must have dimension $\dim S_\infty-l_0-1$ and be a generic point of $\Spec R_\infty/(\varpi)$. Let $\barx_2'$ denote the corresponding point of $\Spec R_\infty'/(\varpi)$. It can not be generic in $\Spec R_\infty'$ and so must generalize to $x'$. 

We now repeatedly use Lemma \ref{comalg3}.
As $\lg_{T_{\infty,x_1}}(C_{\infty,x_1}) \neq (0)$, we deduce that $\barx_1 \in \Spec T_\infty$ and $\lg_{T_{\infty,\barx_1}}((C_{\infty} \otimes^{\bL}_{S_\infty} S_\infty/(\varpi))_{\barx_1}) \neq 0$. Hence $\barx_1' \in \Spec T_\infty'$ and $\lg_{T_{\infty,\barx_1'}'}((C_{\infty}' \otimes^{\bL}_{S_\infty} S_\infty/(\varpi))_{\barx_1'}) \neq 0$, from which we deduce that $x'\in \Spec T_\infty'$ and $\lg_{T_{\infty,x'}'}(C'_{\infty,x'}) \neq 0$. We further deduce that $\barx_2' \in \Spec T_\infty'$ and $\lg_{T_{\infty,\barx_2'}'}((C_{\infty}' \otimes^{\bL}_{S_\infty} S_\infty/(\varpi))_{\barx_2'}) \neq 0$. Hence $\barx_2 \in \Spec T_\infty$ and $\lg_{T_{\infty,\barx_2}}((C_{\infty} \otimes^{\bL}_{S_\infty} S_\infty/(\varpi))_{\barx_2}) \neq 0$, from which we finally deduce that $x_2 \in \Spec T_\infty$ (and $\lg_{T_{\infty,x_2}}(C_{\infty,x_2}) \neq (0)$). 
\end{proof}

\begin{cor}\label{cor:char0automorphy} 
Let $x$ be a prime of $R_\infty$ lying in an irreducible component of 
 $\Spec R_\infty$ of maximal dimension. Let $y$ be the contraction of $x$ in 
$S_\infty$. Then the support of $H^\ast(C_\infty\otimes^{\bL}_{S_\infty} 
S_\infty/y)_y$ over $\Spec R_\infty$ contains 
$x$. If $y$ is one dimensional of characteristic $0$ this says that $x$ is in the support of $H^\ast(C_\infty\otimes^{\bL}_{S_\infty} 
S_\infty/y)[1/p]$.
\end{cor}
\begin{proof}  It follows from Proposition~\ref{prop:full_support_for_unipotent_deformation_rinG} that $x$ is contained in~$\Spec T_\infty$ and occurs in the support of $H^\ast(C_\infty)$. It also occurs in the support of $H^\ast(C_{\infty,y})=H^\ast(C_\infty)_y$. Let $r$ be maximal such that $H^r(C_{\infty,y})_x$ is non-zero. From the Tor spectral sequence
  \[ \Tor^{S_{\infty,y}}_{-i}(H^j(C_{\infty,y}), S_{\infty,y}/y) \Rightarrow H^{i+j}(C_{\infty, y} \otimes^{\bL}_{S_{\infty,y}} S_{\infty,y}/y) \]
  we see that $H^r(C_{\infty, y} \otimes^{\bL}_{S_{\infty,y}} S_{\infty,y}/y)_x$ surjects onto $H^r(C_{\infty,y})_x/y \neq (0)$, so that $x$ lies in the support of $H^r(C_{\infty, y} \otimes^{\bL}_{S_{\infty,y}} S_{\infty,y}/y)=H^r(C_{\infty} \otimes^{\bL}_{S_{\infty}} S_{\infty}/y)_y$, as desired. 
\end{proof}

\subsection{Ultrapatching}\label{sec:patching}
\subsubsection{Set-up for patching}\label{subsec:patchingsetup}
We begin by fixing a non-principal ultrafilter $\gF$ on the set $\NN = \{N \ge 
1\}$. We fix %
a ring $\Lambda$ which is isomorphic to a power series ring over $\cO$.

Let $\delta, g, q$ be positive integers and set $\Delta_\infty = \Zp^{nq}$. We 
let $\cT$ be a formal power series ring over $\Lambda$ (it will come from 
framing variables in our application) and let $S_\infty = 
\cT[[\Delta_\infty]]$. We view $S_\infty$ as an augmented $\Lambda$-algebra, 
and denote the augmentation ideal by $\ga_\infty$.
We also suppose we have two rings 
$R^{\loc}, R'^{\loc}$ 
in $\CNL_{\Lambda}$ with a fixed isomorphism $R^{\loc}/\varpi \cong 
R'^{\loc}/\varpi$
and denote by $R_\infty$ and $R'_\infty$ the formal power series rings in $g$ 
variables over $R^{\loc}$ and $R'^{\loc}$.

Our input for patching is the following data for each $N \in \NN\cup\{0\}$:
\begin{enumerate} 
	\item A quotient $\Delta_N$ of $\Delta_\infty$ such that the kernel of 
	$\Delta_\infty \rightarrow \Delta_N$ is contained in $(p^N\Zp)^{nq} \subset 
	\Delta_\infty$. If $N = 0$, we let $\Delta_0$ be the trivial group, thought 
	of as a quotient of $\Delta_\infty$. We set $S_N = \cT[\Delta_N]$.
	\item\label{setupcomplexes} A pair of perfect complexes $\cC_N$, $\cC'_N$ 
	in 
	$\bD(\Lambda[\Delta_N])$,
      together with an isomorphism 
	$\cC_N\otimes^\LL_{\Lambda[\Delta_N]}\Lambda/\varpi[\Delta_N]
	\cong  \cC'_N\otimes^\LL_{\Lambda[\Delta_N]}\Lambda/\varpi[\Delta_N]$ in 
	$\bD(\Lambda/\varpi[\Delta_N])$. We denote these complexes by 
	$\cC_N/\varpi$ and $\cC'_N/\varpi$ for short. We moreover 
	assume that we have commutative 
	$\Lambda[\Delta_N]$-subalgebras $T_N\subset 
	\End_{\bD(\Lambda[\Delta_N])}(\cC_N)$, $T'_N \subset 
	\End_{\bD(\Lambda[\Delta_N])}(\cC'_N)$ which map to the same subalgebra 
	\[\overline{T}_N \subset 
	\End_{\bD(\Lambda/\varpi[\Delta_N])}(\cC_N/\varpi) = 
	\End_{\bD(\Lambda/\varpi[\Delta_N])}(\cC'_N/\varpi),\] where these 
	endomorphism algebras are identified using our fixed quasi-isomorphism 
	$\cC_N/\varpi \cong \cC'_N/\varpi$.
	\item A pair of rings $R_N$, $R'_N$ in $\CNL_{\Lambda[\Delta_N]}$ with an 
	isomorphism $R_N/\varpi\cong R'_N/\varpi$ together 
	with $R^{\loc}$- and $R'^{\loc}$-algebra structures on $\cT\wotimes_\Lambda 
	R_N$ and $\cT\wotimes_\Lambda R'_N$ respectively which are compatible 
	modulo $\varpi$ with the isomorphisms $R_N/\varpi\cong R'_N/\varpi$ and 
	$R^{\loc}/\varpi \cong R'^{\loc}/\varpi$.
	\item Surjective 
	$R^{\loc}$- and $R'^{\loc}$-algebra maps $R_\infty \rightarrow 
	\cT\wotimes_\Lambda R_N$ and $R'_\infty \rightarrow \cT\wotimes_\Lambda 
	R_N$, which are compatible modulo $\varpi$.
	\item\label{RtoTcompmodl} Nilpotent ideals $I_N$ of $T_N$ and $I'_N$ of 
	$T'_N$ with nilpotence 
	degree $\le \delta$, and continuous surjections $R_N \rightarrow T_N/I_N$, 
	$R'_N 
	\rightarrow T_N/I'_N$. We demand that these maps are also compatible modulo 
	$\varpi$, in the following sense: denote by $\overline{I}_N$ and 
	$\overline{I}'_N$ the images of $I_N$ and $I'_N$ in $\overline{T}_N$. Then 
	the induced maps $R_N/\varpi \rightarrow 
	\overline{T}_N/(\overline{I}_N+\overline{I}'_N)$ and $R'_N/\varpi 
	\rightarrow \overline{T}_N/(\overline{I}_N+\overline{I}'_N)$ are equal when 
	we identify $R_N/\varpi$ and $R'_N/\varpi$ via the fixed isomorphism 
	between them.
\end{enumerate} 

We moreover assume that for each $N \ge 1$ we have
isomorphisms  
$\pi_N: \cC_N\otimes^\LL_{\Lambda[\Delta_N]}\Lambda \cong \cC_0$ and 
$\pi'_N: \cC'_N\otimes^\LL_{\Lambda[\Delta_N]}\Lambda \cong \cC'_0$ in 
$\bD(\Lambda)$ which are 
compatible mod $\varpi$. We obtain induced maps 
$T_N\otimes_{\Lambda[\Delta_N]}\Lambda \rightarrow \End_{\bD(\Lambda)}(\cC_0)$ 
and $T'_N\otimes_{\Lambda[\Delta_N]}\Lambda\rightarrow 
\End_{\bD(\Lambda)}(\cC'_0)$ which we assume factor through maps 
$T_N\otimes_{\Lambda[\Delta_N]}\Lambda \rightarrow T_0$ and 
$T'_N\otimes_{\Lambda[\Delta_N]}\Lambda \rightarrow T'_0$ which are surjective 
when composed with the projections to $T_0/I_0$ and $T'_0/I'_0$.

Finally, we assume that we have isomorphisms $R_N 
\otimes_{\Lambda[\Delta_N]}\Lambda \cong R_0$ and $R'_N 
\otimes_{\Lambda[\Delta_N]}\Lambda \cong R'_0$ which are compatible mod 
$\varpi$
 and with the maps from $R_\infty$ in part (4). We also also 
demand 
compatibility with the maps 
$T_N\otimes_{\Lambda[\Delta_N]}\Lambda \rightarrow T_0$ and 
$T'_N\otimes_{\Lambda[\Delta_N]}\Lambda \rightarrow T'_0$ above. More 
precisely, 
we denote by $I_{N,0}$ and $I'_{N,0}$ the images of $I_N$ and $I'_N$ in 
$T_0/I_0$ 
and $T'_0/I'_0$, 
and demand that the surjective maps $R_N \otimes_{\Lambda[\Delta_N]}\Lambda 
\rightarrow (T_0/I_0)/I_{N,0}$ and  $R'_N \otimes_{\Lambda[\Delta_N]}\Lambda 
\rightarrow (T'_0/I'_0)/I'_{N,0}$ are identified with the maps $R_0 
\rightarrow (T_0/I_0)/I_{N,0}$ and $R'_0 \rightarrow (T'_0/I'_0)/I'_{N,0}$ 
via the isomorphisms $R_N 
\otimes_{\Lambda[\Delta_N]}\Lambda \cong R_0$ and $R'_N 
\otimes_{\Lambda[\Delta_N]}\Lambda \cong R'_0$.

\subsubsection{Patched complexes}\label{subsec:patchedcomplexes}
Apart from Remark \ref{rem:liftStoR} and Proposition 
\ref{prop:primeandunprime}, results and
definitions in this subsection will be stated just for the 
complexes $\cC_N$ and the associated objects and structures, but they also 
apply to the complexes $\cC'_N$.
\begin{defn}\label{def:Jcomplexes}
	Let $J$ be an open 
	ideal in 
	$S_\infty$. Let~$I_J$ be the (cofinite) subset of~$N\in\NN$ such that $J$ 
	contains the kernel of~$S_\infty 
	\rightarrow S_N.$ For $N \in I_J$, we define \[\cC(J,N) 
	= 
	S_\infty/J\otimes^\LL_{\Lambda[\Delta_N]}\cC_N \in \bD(S_\infty/J),\] let 
	$T(J,N)$ denote the 
	image 
	of $S_\infty/J\otimes_{\Lambda[\Delta_N]}T_N$ in 
	$\End_{\bD(S_\infty/J)}(\cC(J,N))$, and denote by $I(J,N)$ the  
	ideal generated by the image of $I_N$ in $T(J,N)$. We have $I(J,N)^\delta = 
	0$.
	
	Additionally, for $d \ge 1$ we define \[R(d,J,N) = R_N/\m_{R_N}^d 
	\otimes_{\Lambda[\Delta_N]}S_\infty/J.\] For every $d,J$ and $N$ we have a 
	surjective $R^{\loc}$-algebra map $R_\infty \rightarrow R(d,J,N)$, which 
	factors through a finite quotient $R_\infty/\m_{R_\infty}^{e(d,J)}$ for 
	some $e(d,J)$ which is independent of $N$.
\end{defn}
For each pair $(J, N)$ such that $\cC(J, N)$ is defined, we fix a choice $\cF(J,N)$
of minimal complex of finite free $S_\infty/J$-modules which is 
quasi-isomorphic to $\cC(J,N)$ (cf. \cite[Lemma 2.3]{KT}). Then for any $i \in \Z$ we have
\[ \mathrm{rk}_{S_\infty/J}(\cF(J,N)^i) = 
\dim_k H^i(\cC_0\otimes^\LL_\Lambda k). \]
\begin{rem}\label{rem:galoisact}
	Recall that we have a surjective map $R_N \rightarrow T_N/I_N$. We 
	therefore obtain a surjective map $R_N\otimes_{\Lambda[\Delta_N]}S_\infty/J 
	\rightarrow T(J,N)/I(J,N)$. For $d$ sufficiently large depending on $J$ 
	(but \emph{not} depending on $N$) this map factors through a surjective map 
	$R(d,J,N) \rightarrow T(J,N)/I(J,N)$. Indeed, it suffices to show that there is an integer $d_0(J)$ such that for any $d \geq d_0(J)$, and for any $x \in \ffrm_{T_N}$, the image of $x^d$ in $\End_{\mathbf{D}(S_\infty / J)}(\cC(J, N))$ (and therefore the image of $x^d$ in $T(J, N)$) is zero. Since 
	\[ H^\ast(\cC(J, N) \otimes^{\bL}_{S_\infty / J} k) \cong H^\ast(\cC_0 \otimes^{\bL}_\Lambda k) \]
	 is a vector space of finite dimension independent of $N$ and
         $J$, we can find an integer $d_1$ such that $x^{d_1}
         H^\ast(\cC(J, N) \otimes^{\bL}_{S_\infty / J} k) = 0$
         (because $x$ acts through a nilpotent endomorphism). The existence of the spectral sequence of a filtered complex implies that there is an integer $d_2$ such that $x^{d_2} H^\ast(\cC(J, N)) = 0$ (here we are using the fact that $S_\infty / J$ has finite length as a module over itself). Finally, the fact that $\cC(J, N)$ is a perfect complex, with cohomology bounded in a range which depends only on $\cC_0$, implies the existence of the integer $d_0(J)$ (use \cite[Lemma 2.5]{KT}). (A similar argument appears at the start of the proof of \cite[Proposition 3.1]{KT}.) 

\end{rem}

\begin{rem}\label{rem:backtolevel0}
	If $J$ contains $\ga_\infty$, then we can identify $S_\infty/J$ with 
	$\Lambda/s(J)$, where $s(J)$ is an open ideal of $\Lambda$. For each $N \in 
	I_J$, the isomorphism $\pi_N: 
	\cC_N\otimes^\LL_{\Lambda[\Delta_N]}\Lambda 
	\cong \cC_0$ induces an isomorphism $\pi_{J,N}:\cC(J,N) 
	\cong
	\cC_0\otimes_\Lambda^\LL\Lambda/s(J).$
\end{rem}

\begin{rem}\label{rem:varyJ}
	Suppose we have open ideals $J_1 \subset J_2$ of $S_\infty$ and $N \in 
	I_{J_1}$. Then we have a natural map $\cC(J_1,N) \rightarrow 
	\cC(J_2,N)$ which induces a quasi-isomorphism 
	\[S_\infty/J_2\otimes^\LL_{S_\infty/J_1}\cC(J_1,N) \cong \cC(J_2,N).\] We 
	obtain a 
	surjective map $T(J_1,N)\rightarrow T(J_2,N)$ and the image of $I(J_1,N)$ 
	under 
	this map is equal to $I(J_2,N)$. So we also obtain a surjective map 
	$T(J_1,N)/I(J_1,N)\rightarrow T(J_2,N)/I(J_2,N)$.
\end{rem}

For $J$ an open ideal in $S_\infty$, $\gF$ restricts to give a non-principal 
ultrafilter on $I_J$, which we again denote by $\gF$. This corresponds to a 
point $x_{\gF} \in \Spec(\prod_{N\in I_J}S_\infty/J)$ by \cite[Lemma 
2.2.2]{geenew}, with localization $(\prod_{N\in I_J}S_\infty/J)_{x_\gF}$ 
canonically isomorphic to 
$S_\infty/J$.

\begin{defn}\label{def:finitepatchedcomplexes}We make the following definitions:
	\[\cC(J,\infty) 
	= 
	(\prod_{N\in I_J}S_\infty/J)_{x_\gF} \otimes_{\prod_{N\in 
			I_J}S_\infty/J}\prod_{N\in I_J}\cC(J,N) \in \bD(S_\infty/J),\] 
	\[R(d,J,\infty) = (\prod_{N\in I_J}S_\infty/J)_{x_\gF} \otimes_{\prod_{N\in 
			I_J}S_\infty/J}\prod_{N\in I_J}R(d,J,N),\]
	$T(J,\infty)$ is defined to be the image of
	$(\prod_{N\in I_J}S_\infty/J)_{x_\gF} \otimes_{\prod_{N\in 
			I_J}S_\infty/J}\prod_{N\in I_J}T(J,N)$ in
	$\End_{\bD(S_\infty/J)}(\cC(J,\infty))$, and the ideal
	$I(J,\infty)\subset T(J,\infty)$ is defined to be the image of 
	$(\prod_{N\in 
		I_J}S_\infty/J)_{x_\gF} \otimes_{\prod_{N\in 
			I_J}S_\infty/J}\prod_{N\in I_J}I(J,N)$ in $T(J,\infty)$.
\end{defn}
\begin{remark}
Since the rings $R(d,J,N)$ are all quotients of 
$R_\infty/\m_{R_\infty}^{e(d,J)}$ (and are in particular finite of bounded 
cardinality), the ultraproduct $R(d,J,\infty)$ is itself a quotient of 
$R_\infty/\m_{R_\infty}^{e(d,J)}$.
\end{remark}
\begin{lem}\label{lem:patchlem} \leavevmode
	\begin{enumerate}
		\item\label{lpla} $I(J,\infty)$ is a nilpotent ideal of $T(J,\infty)$, 
		with 
		$I(J,\infty)^\delta = 0$.
		\item\label{lplb} For $d$ sufficiently large depending on $J$, the maps 
		$R(d,J,N) 
		\rightarrow T(J,N)/I(J,N)$ (see Remark 
		\ref{rem:galoisact}) induce a surjective $S_\infty/J$-algebra map 
		$R(d,J,\infty) 
		\rightarrow T(J,\infty)/I(J,\infty)$.
	\end{enumerate}
\end{lem}
\begin{proof}
	The first part follows from the fact that $\prod_{N\in I_J}I(J,N)$ is a 
	nilpotent ideal of $\prod_{N\in I_J}T(J,N)$ with nilpotence degree $\le 
	\delta$. The second part follows by first considering the map 
	$\prod_{N\in I_J}R(d,J,N)\rightarrow \prod_{N\in I_J}(T(J,N)/I(J,N)) =  
	(\prod_{N\in I_J}T(J,N))/(\prod_{N\in I_J}I(J,N))$, localising at 
	$x_\gF$ and finally passing to the image in $T(J,\infty)/I(J,\infty)$.
\end{proof}

\begin{prop}\label{prop:finitepatchedproperties} \leavevmode
	\begin{enumerate}
		\item\label{fppa} $\cC(J,\infty)$ is a perfect complex of 
		$S_\infty/J$-modules.
		\item\label{fppb} The maps $R_\infty \rightarrow T(J,N)/I(J,N)$ induce 
		a surjection 
		$R_\infty \rightarrow T(J,\infty)/I(J,\infty)$.
		\item\label{fppc} If $J$ contains $\ga_\infty$, then the isomorphisms 
		$\pi_{J,N}$ 
		induce an isomorphism
		$\pi_{J,\infty}:\cC(J,\infty) \cong 
		\cC_0\otimes_\Lambda^\LL\Lambda/s(J)$.
		\item\label{fppd} Suppose we have open ideals $J_1 \subset J_2$ of 
		$S_\infty$. Then 
		the maps $\cC(J_1,N) \rightarrow \cC(J_2,N)$ in $\bD(S_\infty/J_1)$ for 
		$N \in 
		I_{J_1}$ induce an isomorphism 
		\[S_\infty/J_2\otimes^\LL_{S_\infty/J_1}\cC(J_1,\infty) \cong 
		\cC(J_2,\infty).\]
		\item\label{fppe} Let $J_1, J_2$ be as in the previous part. The 
		map 
		$\cC(J_1,\infty) \rightarrow 
		\cC(J_2,\infty)$ induces a surjective map $T(J_1,\infty)\rightarrow 
		T(J_2,\infty)$ and the image of $I(J_1,\infty)$ under this map is equal 
		to $I(J_2,\infty)$.
	\end{enumerate}
\end{prop}
\begin{proof}
\begin{enumerate}
	\item Perfectness of $\cC(J,\infty)$ follows 
from \cite[Corollary 2.2.7]{geenew} --- to apply this Corollary we need to show 
that 
there are constants $D, a, b$ (independent of $N$) such that the complexes 
$\cC(J,N)$ are each quasi-isomorphic to complexes of finite free 
$S_\infty/J$-modules of rank $\le D$ concentrated in degrees $[a,b]$. This 
follows from the theory of minimal resolutions, which we have already applied in order to assert the existence of the complexes $\cF(J, N)$ above. 

\item By the previous part of the proof, the complexes $\cF(J, N)$ $(N \in I_J)$ fall into finitely many isomorphism classes. Therefore there is an element $\Sigma_0$ of the ultrafilter $\gF$ on $I_J$ such that the  
$\cF(J,N)$ are isomorphic for all $N \in \Sigma_0$. We fix isomorphisms (of complexes)
between the 
$\cF(J,N)$ for $N \in \Sigma_0$ and a single complex $\cF(J,\infty)$. Then for all $N \in \Sigma_0$ 
we can identify all the finite endomorphism algebras 
$\End_{\bD(S_\infty/J)}(\cC(J,N))$ with each other. We deduce that there is a 
subset 
$\Sigma_1 \subset \Sigma_0$ with $\Sigma_1 \in \gF$ such that, under this 
identification, the finite Hecke rings $T(J,N)$ and their ideals $I(J,N)$ are 
also identified. So $T(J,\infty) \cong T(J,N)$ and $I(J,\infty)$ corresponds to 
$I(J,N)$ 
for $N \in \Sigma_1$. Since each map $R_\infty \rightarrow T(J,N)/I(J,N)$ is 
surjective, the map $R_\infty \rightarrow T(J,\infty)/I(J,\infty)$ is also 
surjective.
\item The third part follows immediately from the exactness of products and 
localization.
\item First we consider the map of complexes \[\prod_{N\in I_{J_1}}\cC(J_1,N) 
\rightarrow \prod_{N\in I_{J_1}}\cC(J_2,N).\] Since $\prod_{N\in 
I_{J_1}}S_\infty/J_2$ is a finitely presented $\prod_{N\in 
I_{J_1}}S_\infty/J_1$-module (as direct products are exact) the tensor product 
$(\prod_{N \in I_{J_1}} S_\infty/J_2)\otimes_{\prod 
{S_\infty/J_1}}$ commutes with direct products (\cite[Tag 
059K]{stacks-project}). We 
deduce (using Remark \ref{rem:varyJ}) that 
\begin{multline*}
(\prod_{N \in I_{J_1}} S_\infty/J_2)\otimes^\LL_{\prod 
{S_\infty/J_1}}\prod_{N\in 
I_{J_1}}\cC(J_1,N) \\= \prod_{N\in 
I_{J_1}}S_\infty/J_2\otimes^\LL_{S_\infty/J_1}\cC(J_1,N) \cong \prod_{N\in 
I_{J_1}}\cC(J_2,N).\end{multline*} Localizing at $x_\gF$ gives the desired 
statement --- 
since $I_{J_1}$ is cofinite in $I_{J_2}$ we can naturally identify the 
localization of  $\prod_{N\in 
I_{J_1}}\cC(J_2,N)$ with the localization of $\prod_{N\in 
I_{J_2}}\cC(J_2,N)$.
\item The final statement follows from the proof of part (2): there is a 
$\Sigma \subset I_{J_1}$ with $\Sigma \in \gF$ such that $T(J_i,\infty) \cong 
T(J_i, N)$ and $I(J_i,\infty)$ corresponds to $I(J_i,N)$ under these 
isomorphisms for all $N \in 
\Sigma$. Now the desired statement is a consequence of Remark \ref{rem:varyJ}.\qedhere
\end{enumerate}
\end{proof}
We write $\cF(J, \infty)$ for the minimal complex isomorphic to $\cC(J, \infty)$ in $\mathbf{D}(S_\infty / J)$ constructed in the proof of the previous proposition.
\begin{defn}
	We define a complex of $S_\infty$-modules
	\[ \cC_\infty = 
	\invlim_r\cF(\m_{S_\infty}^r,\infty), \]
	where	
	the transition maps in the inverse limit 
	are given by 
	making a choice for each $r \ge 1$ of a map of complexes lifting the 
	natural maps 
	$\cC(\m_{S_\infty}^{r+1},\infty) \rightarrow \cC(\m_{S_\infty}^r,\infty)$ 
	in $\bD(S_\infty/\m_{S_\infty}^{r+1})$. To 
	see that such a map of complexes exists, 
	note that since 
	$\cF(\m_{S_\infty}^{r+1},\infty)$ is a bounded complex of free
	$S_\infty/\m_{S_\infty}^{r+1}$-modules,  viewed as an element of the homotopy category~$\bK(S_\infty/\m_{S_\infty}^{r+1})$ of chain complexes of~$S_\infty/\m_{S_\infty}^{r+1}$-modules, 
	we 
	have
	\[\Hom_{\bK(S_\infty/\m_{S_\infty}^{r+1})}(\cF(\m_{S_\infty}^{r+1},\infty),
	\cF(\m_{S_\infty}^{r},\infty)) 
	=\Hom_{\bD(S_\infty/\m_{S_\infty}^{r+1})}(\cC(\m_{S_\infty}^{r+1},\infty),
	\cC(\m_{S_\infty}^{r},\infty)).\]
	
	Similarly, we let $T_\infty = \invlim_{J}T(J,\infty)$, where the transition 
	maps in the inverse limit are described in Proposition 
	\ref{prop:finitepatchedproperties}(\ref{fppe}). The inverse system of 
	ideals 
	$I(J,\infty)$ defines an ideal $I_\infty$ of $T_\infty$ which satisfies 
	$I_\infty^\delta = 0$.
\end{defn}
\begin{prop}\label{prop:patchedproperties} \leavevmode
	\begin{enumerate}
		\item\label{ppa} $\cC_\infty$ is a bounded complex of finite free 
		$S_\infty$-modules and 
		for each open ideal $J$ of $S_\infty$ there is an isomorphism 
		$\cC_\infty 
		\otimes_{S_\infty}S_\infty/J 
		\cong \cC(J,\infty)$ in $\mathbf{D}(S_\infty / J)$.
		\item\label{ppb} The natural map $\End_{\bD(S_\infty)}(\cC_\infty) 
		\rightarrow  \invlim_J\End_{\bD(S_\infty/J)}(\cC(J,\infty))$ is an 
		isomorphism, and we therefore obtain an injective map 
		$T_\infty\rightarrow 
		\End_{\bD(S_\infty)}(\cC_\infty)$.
		\item\label{ppc} The surjective $\Lambda$-algebra maps $R_\infty 
		\rightarrow 
		T(J,\infty)/I(J,\infty)$ induce a surjection 
		$R_\infty \rightarrow T_\infty/I_\infty$, which factors as a 
		composition of the map $R_\infty 
		\rightarrow \invlim_{d,J}R(d,J,\infty)$ and the $S_\infty$-algebra map 
		$\invlim_{d,J}R(d,J,\infty) 
		\rightarrow T_\infty/I_\infty$ defined by taking the inverse limit of 
		the maps 
		in Lemma \ref{lem:patchlem}(\ref{lplb}).
	\end{enumerate}
\end{prop}
\begin{proof}
	\begin{enumerate}
		\item It follows from the proof of Proposition 
		\ref{prop:finitepatchedproperties}(\ref{fppa}) that 
		
		\[ \mathrm{rk}_{S_\infty/\m_{S_\infty}^{r}}(\cF(\m_{S_\infty}^{r},\infty)^i)
		 = 
		\dim_k H^i(\cC_0\otimes^\LL_\Lambda k) \]
		for all $r$. Moreover, it 
		follows from Proposition \ref{prop:finitepatchedproperties}(\ref{fppd}) 
		and the 
		fact that any quasi-isomorphism of minimal complexes is an isomorphism
		that the transition 
		map 
		$\cF(\m_{S_\infty}^{r+1},\infty)\rightarrow\cF(\m_{S_\infty}^{r},\infty)$
		 induces an isomorphism 
		\[S_\infty/\m_{S_\infty}^{r}\otimes_{S_\infty/\m_{S_\infty}^{r+1}}
		\cF(\m_{S_\infty}^{r+1},\infty)\cong 
		\cF(\m_{S_\infty}^{r},\infty).\] It is now clear that $\cC_\infty$ is a 
		bounded complex of finite free $S_\infty$-modules. If $J$ is an open 
		ideal of $S_\infty$, then for $r$ sufficiently large so that 
		$\m_{S_\infty}^r \subset J$, $\cC_\infty\otimes_{S_\infty}S_\infty/J$ 
		is isomorphic to 
		$S_\infty/J\otimes_{S_\infty/\m_{S_\infty}^{r}}\cF(\m_{S_\infty}^{r},\infty)$,
		which is quasi-isomorphic to $\cC(J,\infty)$ by  Proposition 
		\ref{prop:finitepatchedproperties}(\ref{fppd}).
	\item For the second part, we first note that $T_\infty$ injects into 
	the inverse limit
	$\invlim_J \End_{\bD(S_\infty/J)}(\cC(J,\infty))$, since inverse limits are 
	left exact. The natural map $\End_{\bD(S_\infty)}(\cC_\infty) 
	\rightarrow  \invlim_J\End_{\bD(S_\infty/J)}(\cC(J,\infty))$ is an 
	isomorphism, by the first part of this proposition and (the proof of) 
	\cite[Lemma 2.13(3)]{KT}.
	\item Since the $T(J,\infty)$ are finite rings, the inverse system
	$I(J,\infty)_J$ satisfies the Mittag-Leffler condition and the natural map
	$T_\infty/I_\infty \rightarrow \invlim_J T(J,\infty)/I(J,\infty)$ is an
	isomorphism. For each $J$ the surjective map $R_\infty \rightarrow
	T(J,\infty)/I(J,\infty)$ factors through a finite quotient
	$R_\infty/\m_{R_\infty}^{d(J)}$ of $R_\infty$. Again,
        finiteness implies that
	the
	Mittag-Leffler condition holds, so taking the inverse limit over $J$ gives
	a surjective map $R_\infty = \invlim_J R_\infty/\m_{R_\infty}^{d(J)}
	\rightarrow T_\infty/I_\infty = \invlim_J T(J,\infty)/I(J,\infty)$. The
	desired factorization of the map $R_\infty \rightarrow T_\infty/I_\infty$
	follows from the fact that the maps $R_\infty \rightarrow
	T(J,\infty)/I(J,\infty)$ factor through $R(d,J,\infty)$ for $d$
    sufficiently large.\qedhere
	\end{enumerate}
\end{proof}
\begin{remark}
	There is a natural isomorphism $H^*(\cC_\infty) \cong \invlim_J 
	H^*(\cC_\infty/J) = \invlim_J H^*(\cC(J,\infty))$, so the cohomology of 
	$\cC_\infty$ is independent of the choices of transition maps made to 
	construct $\cC_\infty$. Moreover, if we denote by $\cD_\infty$ the complex 
	constructed with a different choice of transition maps, we have 
	$\Hom_{\bD(S_\infty)}(\cC_\infty,\cD_\infty) = \invlim_J 
	\Hom_{\bD(S_\infty/J)}(\cC(J,\infty),\cC(J,\infty))$ by the argument of 
	Proposition \ref{prop:patchedproperties}(\ref{ppb}), so there is a 
	canonical 
	isomorphism between $\cC_\infty$ and $\cD_\infty$ in $\bD(S_\infty)$.
\end{remark}
\begin{rem}\label{rem:liftStoR}
	Note that the map $\alpha: R_\infty \rightarrow \invlim_{d,J} 
	R(d,J,\infty)$ is surjective, and $\invlim_{d,J}R(d,J,\infty)$ is an 
	$S_\infty$-algebra. 
	As $S_\infty$ is formally smooth over $\Lambda$, we can choose a 
	lift of the map 
	$S_\infty \rightarrow \alpha(R_\infty)$ to a map $S_\infty \rightarrow 
	R_\infty$. In fact, we can and do make such a choice for $R_\infty$ and 
	$R'_\infty$ compatibly mod
	$\varpi$
	since \[(\invlim_{d,J}R(d,J,\infty))/\varpi = 
	\invlim_{d,J}(R(d,J,\infty)/\varpi) \cong 
	\invlim_{d,J}(R'(d,J,\infty)/\varpi),\]
	and since the sequence
		\[
\begin{tikzcd}[column sep=85]
R_\infty  
\arrow{r}{x \mapsto (x \text{ mod } \varpi, \alpha(x))}
&  R_\infty / \varpi \times\invlim_{d,J} 
		R(d,J,\infty)
		\arrow{r}{(y, z) \mapsto \alpha(y) - z \text{ mod } \varpi} 
		& R(d,J,\infty) / \varpi
\end{tikzcd}
\]	
	(and the analogous one for $R'_\infty$) is exact.
	We regard $R_\infty$ as an 
	$S_\infty$-algebra from now on. The map $R_\infty \rightarrow 
	T_\infty/I_\infty$ is an $S_\infty$-algebra map.
\end{rem}
\begin{lem}\label{rem:Gal0}
	The isomorphisms $R(d,J,N)\otimes_{S_\infty/J}S_\infty/(J + \ga_\infty) \cong 
	R_0/(\m_{R_0}^d,s(J+\ga_\infty))$  induce a surjective map $R_\infty/\ga_\infty 
	\rightarrow 
	R_0$.
\end{lem}
\begin{proof}
	First we note that, following the proof of 
	\ref{prop:finitepatchedproperties}(\ref{fppd}), the isomorphisms
	\[ R(d,J,N)\otimes_{S_\infty/J}S_\infty/(J + \ga_\infty) \cong
	R_0/(\m_{R_0}^d,s(J+\ga_\infty)) \]
	 induce an isomorphism
	\[ R(d,J,\infty)\otimes_{S_\infty/J}S_\infty/(J + \ga_\infty) \cong
	R_0/(\m_{R_0}^d,s(J+\ga_\infty)). \]
	 In particular, the map $R_\infty/\ga_\infty
	\rightarrow
	R(d,J,\infty)\otimes_{S_\infty/J}S_\infty/(J + \ga_\infty) =
	R_0/(\m_{R_0}^d,s(J+\ga_\infty))$ is surjective, and factors through
	$R_\infty/(\m_{R_\infty}^{e(d,J)}+\ga_\infty)$ for some $e(d,J)$. Taking the
	inverse limit, we obtain a surjective map $R_\infty/\ga_\infty \rightarrow R_0$.
\end{proof}
\begin{prop}\label{prop:backtolevel0}
	There is an isomorphism $\cC_\infty/\ga_\infty \rightarrow \cC_0$ in $\mathbf{D}(\Lambda)$ which 
	induces a  map $T_\infty \rightarrow T_0$ which becomes surjective when 
	composed with the projection $T_0\rightarrow T_0/I_0$. Denoting the image 
	of 
	$I_\infty$ under this surjective map by $I_{\infty,0}$, we obtain a 
	surjective map 
	$R_\infty/\ga_\infty  
	\rightarrow (T_0/I_0)/I_{\infty,0}$. This map is the composition of the 
	map $R_\infty/\ga_\infty \rightarrow 
	R_0$ in Lemma \ref{rem:Gal0} with the map $R_0 \rightarrow  (T_0/I_0)/ 
	I_{\infty,0}$ coming from our original set-up.
\end{prop}
\begin{proof}
	We have $\cC_\infty/\ga_\infty =
	\invlim_r\cF(\m_{S_\infty}^r,\infty)/\ga_\infty$, and
	$\cF(\m_{S_\infty}^r,\infty)/\ga_\infty$ is a minimal resolution of
	$\cC_\infty/(\m_{S_\infty}^r+\ga_\infty) \cong \cC(\m_{S_\infty}^r+\ga_\infty,\infty)$. By
	Proposition \ref{prop:finitepatchedproperties}(\ref{fppc}), this is 
	quasi-isomorphic 
	to
	$\cC_0\otimes^\LL_{\Lambda}\Lambda/s(\m_{S_\infty}^r+\ga_\infty)$. Replacing 
	$\cC_0$
	by a quasi-isomorphic bounded complex of finite projective 
	$\Lambda$-modules and applying \cite[Lemma 2.13]{KT},
	we see that the quasi-isomorphisms $\cF(\m_{S_\infty}^r,\infty)/\ga_\infty \cong
	\cC_0\otimes^\LL_{\Lambda}\Lambda/s(\m_{S_\infty}^r+\ga_\infty)$ induce a
	quasi-isomorphism $\invlim_r\cF(\m_{S_\infty}^r,\infty)/\ga_\infty \cong \cC_0$. 

	The induced map $T_\infty \rightarrow \End_{\bD(\Lambda)}(\cC_0)$ is the
	composite of the surjective map $T_\infty \rightarrow \invlim_{\ga_\infty\subset 
	J}T(J,\infty)$ and an inverse limit of maps $T(J,\infty) \rightarrow
	\End_{\bD(\Lambda/s(J))}(\cC_0\otimes^\LL_{\Lambda}\Lambda/s(J))$. Each of
	these maps factors through $T_0$, and if we denote the image of $T_0$ in 
	$\End_{\bD(\Lambda/s(J))}(\cC_0\otimes^\LL_{\Lambda}\Lambda/s(J))$ by 
	$T_0^J$ then $T(J,\infty)$ surjects onto $T_0^J/I_0$. Passing to the inverse
	limit gives the desired map $T_\infty \rightarrow T_0$.
		
	The compatibility with the map $R_\infty \rightarrow R_0$ follows from the
	compatibility between the maps $T_N \rightarrow T_0$ and $R_N \rightarrow 
	R_0$
	in our original set-up.
	\end{proof}
We now separate out the primed and unprimed situations;  so we have two 
perfect complexes of $S_\infty$-modules, $\cC_\infty$ and $\cC'_\infty$. 
\begin{prop}\label{prop:primeandunprime} \leavevmode
	\begin{enumerate}
\item 	The quasi-isomorphisms $\cC_N/\varpi \cong \cC'_N/\varpi$ induce a 
	quasi-isomorphism $\cC_\infty/\varpi \cong \cC'_\infty/\varpi$. 
\item $T_\infty$ and $T'_\infty$ have the same image in 
$\End_{\bD(S_\infty)}(\cC_\infty/\varpi)$ and 
$\End_{\bD(S_\infty)}(\cC'_\infty/\varpi)$, via the identification 
$\cC_\infty/\varpi\cong \cC'_\infty/\varpi$ of the previous part. Call this 
common image 
$\overline{T}_\infty$.
\item Write $\overline{I}_\infty$ and $\overline{I}'_\infty$ for the images of 
$I_\infty$ and $I'_\infty$ in $\overline{T}_\infty$ The actions of 
$R_\infty/\varpi \cong R'_\infty/\varpi$  
(induced from $T_\infty$ and $T'_\infty$ respectively) on 
$H^*(\cC_\infty/\varpi)/(\overline{I}_\infty+\overline{I}'_\infty)$ 
and $H^*(\cC'_\infty/\varpi)/(\overline{I}_\infty+\overline{I}'_\infty)$ are 
identified via  
$\cC_\infty/\varpi \cong \cC'_\infty/\varpi$.
	\end{enumerate}
\end{prop}
\begin{proof}
\begin{enumerate}
	\item The 
	isomorphisms $\cC_N/\varpi \cong \cC'_N/\varpi$ in $\mathbf{D}(\Lambda[\Delta_N])$ induce compatible 
	isomorphisms $\cC(J+\varpi,\infty) \cong \cC'(J+\varpi,\infty)$ for 
	all $J$. Since $\cC_\infty/\varpi = \invlim_r 
	\cF(\m_{S_\infty}^r,\infty)/\varpi$ and 
	$\cF(\m_{S_\infty}^r,\infty)/\varpi$ is a minimal resolution of 
	$\cC(\m_{S_\infty}^r+\varpi,\infty)$ we have 
	\[\Hom_{\bD(S_\infty/\varpi)}(\cC_\infty/\varpi,\cC'_\infty/\varpi) = 
	\invlim_J\Hom_{\bD(S_\infty/(J+\varpi))}(\cC(J+\varpi,\infty),
	\cC'(J+\varpi,\infty)).\] We therefore deduce the first part of the 
	Proposition.
	\item By the proof of the previous part, it suffices to show that 
	the images of $T_\infty$ and $T'_\infty$ in 
	$\End_{\bD(S_\infty/(J+\varpi))}(\cC(J+\varpi,\infty))$ and 
	$\End_{\bD(S_\infty/(J+\varpi))}(\cC'(J+\varpi,\infty))$ respectively 
	(which are $T(J+\varpi,\infty)$ and $T'(J+\varpi,\infty)$), are 
	identified 
	via 
	the quasi-isomorphisms $\cC(J+\varpi,\infty) \cong 
	\cC'(J+\varpi,\infty)$. This follows from the fact that
	for every $N \in I_{J+\varpi}$, $T(J+\varpi,N)$ and $T'(J+\varpi,N)$ are 
	identified via the quasi-isomorphism $\cC(J+\varpi,N) \cong 
	\cC'(J+\varpi,N)$, which is a 
	consequence of our original assumptions (see point (\ref{setupcomplexes}) 
	in Section \ref{subsec:patchingsetup}).
	\item It suffices to show that the maps $R_\infty/\varpi \rightarrow 
	T(J+\varpi,\infty)/I(J+\varpi,\infty)$ and $R'_\infty/\varpi \rightarrow 
	T'(J+\varpi,\infty)/I'(J+\varpi,\infty)$ are equal when we identify 
	$R_\infty/\varpi$ with $R'_\infty/\varpi$, $T(J+\varpi,\infty)$ with 
	$T'(J+\varpi,\infty)$, and pass to the quotient by 
	$I(J+\varpi,\infty)+I'(J+\varpi,\infty)$.  This follows from the 
	compatibility in point (\ref{RtoTcompmodl}) of Section 
	\ref{subsec:patchingsetup}.\qedhere
\end{enumerate}
\end{proof}

\subsection{The proof of Theorem \ref{thm:main_automorphy_lifting_theorem}}\label{sec:AppToALT}

We are now in a position to prove the first main theorem of this chapter (Theorem \ref{thm:main_automorphy_lifting_theorem}). We first establish the result under additional conditions in \S \ref{subsec:an_r_equals_t_theorem}, then reduce to this case using soluble base change in \S \ref{sec:proof_of_main_automorphy_lifting_theorem}.

\subsubsection{Application of the patching argument (Fontaine--Laffaille case)}\label{subsec:an_r_equals_t_theorem}
We take $F$ to be an imaginary CM number field, and fix the following data:
\begin{enumerate}
\item\label{item:first_fl_hyp} An integer $n \geq 2$ and a prime $p > n^2$.
\item A finite set $S$ of finite places of $F$, including the places above $p$. 
\item A (possibly empty) subset $R \subset S$ of places prime to $p$.
\item A cuspidal automorphic representation $\pi$ of $\GL_n(\bA_F)$, regular algebraic of some weight $\lambda$. 
\item A choice of isomorphism $\iota : \overline{\bQ}_p \cong \bC$.
\end{enumerate}
We assume that the following conditions are satisfied:
\begin{enumerate}
\addtocounter{enumi}{5}
\item If $l$ is a prime lying below an element of $S$, or which is ramified in $F$, then $F$ contains an imaginary quadratic field in which $l$ splits. In particular, each place of $S$ is split over $F^+$ and the extension $F / F^+$ is everywhere unramified. 
\item The prime $p$ is unramified in $F$. 
\item For each embedding $\tau : F \hookrightarrow \bC$, we have
\[ \lambda_{\tau, 1} + \lambda_{\tau c, 1} - \lambda_{\tau, n} - \lambda_{\tau c, n} < p - 2n. \]
\item For each $v \in S_p$, let $\overline{v}$ denote the place of $F^+$ lying below $v$. Then there exists a place $\overline{v}' \neq \overline{v}$ of $F^+$ such that $\overline{v}' | p$ and 
\[ \sum_{\overline{v}'' \neq \overline{v}, \overline{v}'} [ F^+_{\overline{v}''} : \bQ_p ] > \frac{1}{2} [ F^+ : \bQ ]. \]
\item The residual representation $\overline{r_\iota(\pi)}$ is absolutely irreducible.
\item If $v$ is a place of $F$ lying above $p$, then $\pi_v$ is unramified.
\item If $v \in R$, then $\pi_v^{\Iw_v} \neq 0$.
\item If $v \in S - (R \cup S_p)$, then $\pi_v$ is unramified, $v\notin R^c$, and $H^2(F_v, \ad \overline{r_\iota(\pi)}) = 0$. %
\item $S-(R \cup S_p)$ contains at least two places with distinct residue characteristics.
\item If $v \not\in S$ is a finite place of $F$, then $\pi_v$ is unramified. 
\item If $v \in R$, then $q_v \equiv 1 \text{ mod }p$ and $\overline{r_\iota(\pi)}|_{G_{F_v}}$ is trivial. 
\item\label{item:last_fl_hyp} The representation $\overline{r_\iota(\pi)}$ is decomposed generic in the sense of Definition~\ref{defn:decomposed_generic} and the image of~$\overline{r_\iota(\pi)}|_{G_{F(\zeta_p)}}$ is enormous in the sense of Definition~\ref{defn:enormous image}.
\end{enumerate}
We define an open compact subgroup $K = \prod_v K_v$ of $\GL_n(\widehat{\cO}_F)$ as follows:
\begin{itemize}
\item If $v \not\in S$, or $v \in S_p$, then $K_v = \GL_n(\cO_{F_v})$.
\item If $v \in R$, then $K_v = \Iw_v$.
\item If $v \in S - (R \cup S_p)$, then $K_v = \Iw_{v,1}$ is the pro-$v$ Iwahori subroup of $\GL_n(\cO_{F_v})$.
\end{itemize}
The following lemma shows that $K$ is neat, hence is a good subgroup of $\GL_n(\A_F^\infty)$.
\begin{lemma}\label{lem:neat_subgroups}
Suppose that $K = \prod_v K_v \subset \GL_n(\widehat{\cO}_F)$ is an open compact subgroup and that there exists two places $v,v'$ of $F$ such that $v, v'$ have distinct residue characteristics $q,q'$ and $K_v = \Iw_{v,1}$, $K_{v'} = \Iw_{v',1}$. Then $K$ is neat.
\end{lemma}
\begin{proof}
We show that if $(g_v,g_{v'}) \in \Iw_{v,1}\times \Iw_{v',1}$, then the group $\Gamma_v \cap \Gamma_{v'}$ (see the definition 
of neat in \S\ref{sssec:symmetric}) is trivial. Suppose this is not the case, then it contains a root of unity $\zeta$ of some prime order $q''$.

If $\alpha$ is an eigenvalue of 
$g_v$ in $\overline{F}_v$, then $\alpha-1$ is in the maximal ideal of $\cO_{\overline{F}_v}$. The same is then true for $\zeta$, thus $q''=q$. However, running the above for $v'$ instead of $v$ also shows $q''=q'$, so $q'=q$, a contradiction.
\end{proof}
By Theorem \ref{thm:application_of_matsushima}, we can find a coefficient field $E \subset \overline{\bQ}_p$ and a maximal ideal $\ffrm \subset \bT^S(K, \cV_\lambda)$ such that $\overline{\rho}_\ffrm \cong \overline{r_\iota(\pi)}$. After possibly enlarging $E$, we can and do assume that the residue field of $\m$ is equal to $k$.
For each tuple $(\chi_{v, i})_{v \in R, i = 1, \dots, n}$ of characters $\chi_{v, i} : k(v)^\times \to \cO^\times$ which are trivial modulo $\varpi$, we define a global deformation problem  by the formula
\[ \cS_\chi = (\overline{\rho}_\ffrm, S, \{ \cO \}_{v \in S}, \{ \cD_v^\text{FL} \}_{v \in S_p} \cup \{ \cD_v^\chi \}_{v \in R} \cup \{ \cD_v^\square \}_{v \in S - (R \cup S_p)}). \]
We fix representatives $\rho_{\cS_\chi}$ of the universal deformations which are identified modulo $\varpi$ (via the identifications $R_{\cS_\chi} / \varpi \cong R_{\cS_1} / \varpi$). We observe that the local deformation problems defining $\cS_\chi$ are formally smooth away from the places in $R$. We define an $\cO[K_S]$-module $\cV_\lambda(\chi^{-1}) = \cV_\lambda \otimes_\cO \cO(\chi^{-1})$, where $K_S$ acts on $\cV_\lambda$ by projection to $K_p$ and on $\cO(\chi^{-1})$ by the projection $K_S \to K_R = \prod_{v \in R} \Iw_v \to \prod_{v \in R} (k(v)^\times)^n$.
\begin{prop}\label{prop:existence_of_Hecke_Galois_with_LGC}
There exists an integer $\delta \geq 1$, depending only on $n$ and $[F : \Q]$, an ideal $J \subset \T^S( R \Gamma(X_K, \cV_\lambda(\chi^{-1})))_\ffrm$ such that $J^\delta = 0$, and a continuous surjective homomorphism
\[ f_{\cS_\chi} : R_{\cS_\chi}  \to \T^S( R \Gamma(X_K, \cV_\lambda(\chi^{-1})))_\ffrm / J \]
such that for each finite place $v \not \in S$ of $F$, the characteristic polynomial of $f_{\cS_\chi} \circ \rho_{\cS_\chi}(\Frob_v)$ equals the image of $P_v(X)$ in $\T^S( R \Gamma(X_K, \cV_\lambda(\chi^{-1})))_\ffrm / J$.
\end{prop}
\begin{proof}
This is a matter of combining the various local-global compatibility results we have proved so far. The existence of a Galois representation $\rho_\m : G_{F, S\cup S^c} \to \GL_n(\T^S( R \Gamma(X_K, \cV_\lambda(\chi^{-1})))_\ffrm/J)$ satisfying the required condition at finite places $v \not\in S\cup S^c$ is contained in Theorem \ref{thm:existence_of_Hecke_repn_for_GL_n}. After conjugation, we can assume that $\rho_\m \text{ mod }\m$ equals $\overline{\rho}_\m$. To prove the proposition, we need to show that for each $v \in S$, $\rho_\m|_{G_{F_v}}$ is a lifting of $\overline{\rho}|_{G_{F_v}}$ of the appropriate type, and that for each $v\in S^c - S$, $\rho_{\m}|_{G_{F_v}}$ is unramifed and the characteristic polynomial of $\rho_{\m}(\Frob_v)$ has the correct form. Theorem \ref{thm:lgcfl} shows that the Fontaine--Laffaille condition is satisfied for each $v | p$. We apply Theorem \ref{thm:lgc_at_l_neq_p} with the set $S$ of places there equal to $S \cup S^c$ and the set $R$ equal to $S-S_p$. This shows that the appropriate condition on the characteristic polynomials of elements $\rho_\m(\sigma)$ ($\sigma \in I_{F_v}$) is satisfied for each $v \in R$, and that $\rho_{\m}|_{G_{F_v}}$ is unramifed with the characteristic polynomial $\rho_\m(\Frob_v)$ of the correct form for $v\in S^c-S$.
\end{proof}

 Recall (as in (\ref{sec:AvoidIharasetup})) that it makes sense to talk about the support of  $H^\ast(X_K, \cV_{\lambda}(1))_\ffrm$ over $ R_{\cS_{1}}$, even though $H^\ast(X_K, \cV_{\lambda}(1))_\ffrm$ is not literally an $R_{\cS_{1}}$-module. We can now state our first key technical result, which we will prove below.
\begin{theorem}\label{thm:R_equals_T}
Under assumptions (\ref{item:first_fl_hyp})--(\ref{item:last_fl_hyp}) above, $H^\ast(X_K, \cV_{\lambda}(1))_\ffrm$ has full support over $R_{\cS_{1}}$. %
\end{theorem}

\begin{cor}\label{cor:R_equals_T_implies_automorphy}
Under assumptions (\ref{item:first_fl_hyp})--(\ref{item:last_fl_hyp}) above, 
suppose given a continuous representation $\rho : G_F \to 
\GL_n(\overline{\bQ}_p)$ satisfying the following conditions:
\begin{enumerate}
\item We have $\overline{\rho} \cong \overline{r_\iota(\pi)}$.
\item For each place $v | p$ of $F$, $\rho|_{G_{F_v}}$ is crystalline. For each embedding $\tau : F \hookrightarrow \overline{\bQ}_p$, we have
\[ \mathrm{HT}_\tau(\rho) = \{ \lambda_{\iota\tau, 1} + n-1, \dots, \lambda_{\iota\tau, n} \}. \]
\item For each finite place $v\not\in S$ of $F$, $\rho|_{G_{F_v}}$ is unramified.
\item For each place $v \in R$, $\rho|_{G_{F_v}}$ is unipotently ramified.
\end{enumerate}
Then $\rho$ is automorphic: there exists a cuspidal, regular algebraic automorphic representation $\Pi$ of weight $\lambda$ such that $\rho \cong r_\iota(\Pi)$. Moreover, if $v$ is a finite place of $F$ such that $v | p $ or $v \not\in S$, then $\Pi_v$ is unramified.
\end{cor}
\begin{proof}
After possibly enlarging the coefficient field $E$, and replacing $\rho$ by a $\GL_n(\overline{\bQ}_p)$-conjugate, we can assume that it takes values in $\GL_n(\cO)$, and that $\rho \text{ mod }\varpi = \overline{\rho}_\ffrm$. Then $\rho$ is a lifting of type $\cS_1$, so determines a homomorphism $f : R_{\cS_1} \to E$. Theorem \ref{thm:R_equals_T} implies that $\ker f$ is in the support of $H^\ast(X_K, \cV_{\lambda}(1))_\ffrm[1/p]$; Theorem \ref{thm:application_of_matsushima} then implies that there exists a cuspidal, regular algebraic automorphic representation $\Pi$ of weight $\lambda$ such that $\rho \cong r_\iota(\Pi)$ and $(\Pi^\infty)^K \neq 0$. This is the desired result (recall that $K_v = \GL_n(\cO_{F_v})$ if $v | p $ or $v \not\in S$).
\end{proof}
Before proceeding to the proof of Theorem \ref{thm:R_equals_T}, we need to 
introduce auxiliary level subgroups. These will be associated to a choice of 
Taylor--Wiles datum $(Q,(\alpha_{v,1},\ldots,\alpha_{v,n})_{v\in Q})$ for 
$\cS_1$ (see~\S\ref{sec:TWprimes}). We assume that for each $v \in Q$, there 
exists an imaginary quadratic subfield of $F$ in which the residue characteristic $l_v$ 
of $v$ splits. This Taylor--Wiles datum is  automatically a Taylor--Wiles datum 
for all the global deformation problems $\cS_\chi$, and so the auxiliary 
deformation problems $\cS_{\chi, Q}$ are defined, and the deformation ring 
$R_{\cS_{\chi, Q}}$ has a natural structure of $\cO[\Delta_Q]$-algebra, where 
$\Delta_Q = \prod_{v \in Q} \Delta_v = \prod_{v \in Q} k(v)^\times(p)^n$. The 
constructions we are about to give necessarily involve a lot of notation. 
Accordingly, we invite the reader to review the notation related to Hecke 
algebras in \S \ref{sec:unitary_group_setup} before continuing.

We define two auxiliary level subgroups $K_1(Q) \subset K_0(Q) \subset K$. They are good subgroups of $\GL_n(\A_F^\infty)$, determined by the following conditions:
\begin{itemize}
	\item If $v \not\in S \cup Q$, then $K_1(Q)_v = K_0(Q)_v = K_v$.
	\item If $v \in Q$, then $K_0(Q)_v = \Iw_v$ and $K_1(Q)_v$ is the maximal pro-prime-to-$p$ subgroup of $\Iw_v$. 
\end{itemize}
Then there is a natural isomorphism $K_0(Q) / K_1(Q) \cong \Delta_Q$, and surjective morphisms of $\T^{S \cup Q}$-algebras
\numequation \label{eqn:maps_of_Q_Hecke_algebras} \begin{aligned}
 _{K_0(Q)/K_1(Q)} & \T^{S \cup Q}(K_0(Q) / K_1(Q), \cV_\lambda(\chi^{-1})) \to  
  \T^{S \cup Q}(K_0(Q), \cV_\lambda(\chi^{-1}))  \\ &  \to \T^{S \cup Q}(K, \cV_\lambda(\chi^{-1})). 
  \end{aligned}
\end{equation}
The first of these arises by taking $K_0(Q)$-invariants (cf. \S
\ref{sec:unitary_group_setup} and note $\cO[\Delta_Q]$ acts trivially on invariants) and the second is given by the formula
$t \mapsto [K : K_0(Q)]^{-1} \pi_{Q, \ast} \circ t \circ \pi_Q^\ast$,
where $\pi_Q : X_{K_0(Q)} \to X_K$ is the canonical projection; note
that $[K : K_0(Q)] \equiv (n!)^{|Q|} \text{ mod } p$ is a unit in
$\cO$ because of our assumption that $p > n$. We define
\[ \T^{S \cup Q}_Q(K_0(Q),\cV_\lambda(\chi^{-1}))\subset \End_{\mathbf{D}(\cO)}(R \Gamma(X_{K_0(Q)}, \cV_\lambda(\chi^{-1})))  \]
as in \S \ref{sec:lneqp_statements}; it is the commutative $\T^{S \cup Q}(K_0(Q),\cV_\lambda(\chi^{-1}))$-subalgebra generated by the operators $U_{v, i}$ ($v \in Q$, $i = 1, \dots, n$), or equivalently the image of the algebra $\T^{S \cup Q}_Q$ defined in \S \ref{sec:lneqp_statements}. Similarly we define
\[ \T^{S \cup Q}_Q(K_0(Q) / K_1(Q),\cV_\lambda(\chi^{-1}))\subset \End_{\mathbf{D}(\cO[\Delta_Q])}(R \Gamma_{K_0(Q)/K_1(Q)}(X_{K_1(Q)}, \cV_\lambda(\chi^{-1})));  \]
it is an $\cO[\Delta_Q]$-algebra, which coincides with the image of the algebra $\T^{S \cup Q}_Q$. The first map in (\ref{eqn:maps_of_Q_Hecke_algebras}) extends to a surjective homomorphism
\numequation\label{eqn:maps_of_augmented_Q_Hecke_algebras}
 \T^{S \cup Q}_Q(K_0(Q) / K_1(Q), \cV_\lambda(\chi^{-1})) \to \T^{S \cup Q}_Q(K_0(Q), \cV_\lambda(\chi^{-1}))
\end{equation}
which takes $U_{v, i}$ to $U_{v, i}$ for each $v \in Q$ and for each $i = 1, \dots, n$.

We define $\m^Q \subset \T^{S \cup Q}(K, \cV_\lambda(\chi^{-1}))$ to be the pullback of $\m$ under the inclusion
\[ \T^{S \cup Q}(K, \cV_\lambda(\chi^{-1})) \subset \T^{S}(K, \cV_\lambda(\chi^{-1})).  \]
We define 
\[ \m_0^Q \subset \T^{S \cup Q}(K_0(Q), \cV_\lambda(\chi^{-1})) \]
to be the pullback of $\m^Q$ and 
\[ \m_1^Q \subset _{K_0(Q)/K_1(Q)}\T^{S \cup Q}(K_0(Q) / K_1(Q), \cV_\lambda(\chi^{-1})) \]
 to be the pullback of $\m_0^Q$, these pullbacks being taken under the maps in 
 (\ref{eqn:maps_of_Q_Hecke_algebras}). We define $\n_0^Q \subset \T^{S \cup 
 Q}_Q(K_0(Q), \cV_\lambda(\chi^{-1}))$ to be the ideal generated by $\m_0^Q$ 
 and the elements $U_{v, i} - q_v^{i(1-i)/2} \alpha_{v, 1}\cdots\alpha_{v, i}$ 
 for each $v \in Q$ and $i = 1, \dots, n$. We define $\n_1^Q \subset \T^{S \cup 
 Q}_Q(K_0(Q) / K_1(Q), \cV_\lambda(\chi^{-1}))$ to be the pre-image of $\n_0^Q$ 
 under the map (\ref{eqn:maps_of_augmented_Q_Hecke_algebras}).
\begin{lemma}\label{lem:maximal_ideals_of_Q_hecke_algebra_are_proper}
	Each ideal $\m^Q$, $\m_0^Q$, $\m_1^Q$, $\n_0^Q$, and $\n_1^Q$ is a (proper) maximal ideal.
\end{lemma}
\begin{proof}
	This is clear for the ideals $\m^Q$, $\m_0^Q$, and $\m_1^Q$. Since $\n_1^Q$ is the pre-image of $\n_0^Q$ under a surjective algebra homomorphism, we just need to check that $\n_0^Q$ is a proper ideal. Equivalently, we must check that 
	\[ H^\ast(X_{K_0(Q)}, \cV_\lambda(\chi^{-1}) / \varpi)[\m_0^Q] \]
	 contains a non-zero vector on which each operator $U_{v, i}$ ($v \in Q$, 
	 $i = 1, \dots, n$) acts by the scalar $\alpha_{v, 1} \cdots \alpha_{v, 
	 i}$. This will follow from \cite[Lemma 5.3]{KT} (or rather its proof) if 
	 we can show that $H^\ast(X_{K}, \cV_\lambda(\chi^{-1}))[\m^Q]$ is 
	 annihilated by a power of $\m$. This follows from the existence of 
	 $\overline{\rho}_\m$ and its local-global compatibility at the places $v 
	 \in Q$. 
\end{proof}
We can therefore form the localized complexes 
\[ R \Gamma(X_K, \cV_\lambda(\chi^{-1}))_{\m}, R \Gamma(X_K, \cV_\lambda(\chi^{-1}))_{\m^Q}, \]
\[ R \Gamma(X_{K_0(Q)}, \cV_\lambda(\chi^{-1}))_{\m_0^Q}, R \Gamma(X_{K_0(Q)}, \cV_\lambda(\chi^{-1}))_{\n_0^Q}, \]
\[ R \Gamma_{K_0(Q)/K_1(Q)}(X_{K_1(Q)}, \cV_\lambda(\chi^{-1}))_{\m_1^Q}, R \Gamma_{K_0(Q)/K_1(Q)}(X_{K_1(Q)}, \cV_\lambda(\chi^{-1}))_{\n_1^Q}. \]
The first four lie in $\mathbf{D}(\cO)$, the last two in $\mathbf{D}(\cO[\Delta_Q])$.
\begin{lemma}\label{lem:identification_of_complexes}
	The natural morphisms
	\[ R \Gamma(X_K, \cV_\lambda(\chi^{-1}))_{\m^Q} \to R \Gamma(X_K, \cV_\lambda(\chi^{-1}))_{\m} \]
	and
	\[ R \Gamma(X_{K_0(Q)}, \cV_\lambda(\chi^{-1}))_{\n_0^Q} \to R \Gamma(X_K, \cV_\lambda(\chi^{-1}))_{\m^Q} \]
	and
	\[  R \Gamma(\Delta_Q, R \Gamma_{K_0(Q)/K_1(Q)}(X_{K_1(Q)}, \cV_\lambda(\chi^{-1}))_{\n_1^Q}) \to R \Gamma(X_{K_0(Q)}, \cV_\lambda(\chi^{-1}))_{\n_0^Q} \]
	in $\mathbf{D}(\cO)$ are isomorphisms.
\end{lemma}
\begin{proof}
We must show that these morphisms in the derived category give isomorphisms at the level of cohomology. For the first morphism, it is enough to show that $\m$ is the unique maximal ideal of $\T^{S \cup Q}(K_0(Q), \cV_\lambda(\chi^{-1}))$ lying above $\m^Q$, and we have seen this already in the proof of Lemma \ref{lem:maximal_ideals_of_Q_hecke_algebra_are_proper}. It is clear from the definitions for the third morphism. For the second, it is enough to check that we have an isomorphism after applying the functor $ - \otimes^\bL_\cO k : \bD(\cO) \to \bD(k)$. We are therefore reduced to showing that the map of $k$-vector spaces
\[ \tr_{K / K_0(Q)} : H^\ast(X_{K_0(Q)}, \cV_{\lambda}(\chi^{-1}) / \varpi)_{\n_0^Q} \to H^\ast(X_K, \cV_{\lambda}(\chi^{-1}) / \varpi)_{\m^Q}  \]
is an isomorphism. This is the content of \cite[Lemma 5.4]{KT}.
\end{proof}
We see that there is a surjective homomorphism
\numequation\label{eqn:final_morphism_of_Q_augmented_Hecke_algebras}
\begin{split}
  _{K_0(Q)/K_1(Q)}\T^{S \cup Q}( R \Gamma_{K_0(Q)/K_1(Q)}(X_{K_1(Q)},
  \cV_\lambda(\chi^{-1}))_{\n_1^Q} ) \to\\ \T^{S \cup Q}( R \Gamma(X_K,
  \cV_\lambda(\chi^{-1}))_{\m^Q} ) = \T^{S \cup Q}(K,
  \cV_\lambda(\chi^{-1}))_{\m^Q}.
\end{split}
\end{equation}  
The first ring $_{K_0(Q)/K_1(Q)}\T^{S \cup Q}( R \Gamma_{K_0(Q)/K_1(Q)}(X_{K_1(Q)}, \cV_\lambda(\chi^{-1}))_{\n_1^Q} )$ is a local $\cO[\Delta_Q]$-algebra, its unique maximal ideal being identified with the pre-image of $\m^Q$ under the surjective homomorphism (\ref{eqn:final_morphism_of_Q_augmented_Hecke_algebras}); indeed, this follows from the fact that it acts nearly faithfully on $H^\ast( X_{K_1(Q)}, \cV_\lambda(\chi^{-1}))_{\n_1^Q}$ (We recall (\cite[Def. 2.1]{tay}) that a finitely generated module over a Noetherian local ring is said to be nearly faithful if its annihilator is a nilpotent ideal). We can now state a result asserting the existence of Galois representations valued with coefficients in this Hecke algebra.
\begin{prop}\label{prop:existence_of_Q_augmented_Hecke_Galois_with_LGC}
	There exists an integer $\delta \geq 1$, depending only on $n$ and $[F : \Q]$, an ideal $J \subset _{K_0(Q)/K_1(Q)}\T^{S \cup Q}( R \Gamma_{K_0(Q)/K_1(Q)}(X_{K_1(Q)}, \cV_\lambda(\chi^{-1}))_{\n_1^Q} )$ such that $J^\delta = 0$, and a continuous surjective $\cO[\Delta_Q]$-algebra homomorphism
	\[ f_{\cS_{\chi, Q}} : R_{\cS_{\chi, Q}}  \to _{K_0(Q)/K_1(Q)}\T^{S \cup Q}( R 
	\Gamma_{K_0(Q)/K_1(Q)}(X_{K_1(Q)}, \cV_\lambda(\chi^{-1}))_{\n_1^Q} ) / J \]
	such that for each finite place $v \not \in S \cup Q$ of $F$, the characteristic polynomial of $f_{\cS_{\chi, Q}} \circ \rho_{\cS_{\chi, Q}}(\Frob_v)$ equals the image of $P_v(X)$ 
	in 
	$$_{K_0(Q)/K_1(Q)}\T^{S \cup Q}( R \Gamma_{K_0(Q)/K_1(Q)}(X_{K_1(Q)}, \cV_\lambda(\chi^{-1}))_{\n_1^Q} ) / J.$$
\end{prop}
\begin{proof}
	To save notation, let 
	$$\T =_{K_0(Q)/K_1(Q)} \T^{S \cup Q}( R \Gamma_{K_0(Q)/K_1(Q)}(X_{K_1(Q)}, \cV_\lambda(\chi^{-1}))_{\n_1^Q} ),$$
	 and $\T' =\T^{S \cup Q}_Q(K_0(Q) / K_1(Q), \cV_\lambda(\chi^{-1}))_{\n_1^Q}$. Then $\T \subset \T'$, and the inclusion $\T \to \T'$ is a local homomorphism of finite $\cO[\Delta_Q]$-algebras. By Theorem \ref{thm:existence_of_Hecke_repn_for_GL_n}, there is a nilpotent ideal $J' \subset \T'$ and a Galois representation $\rho_{\n_1^Q} : G_{F, S \cup Q} \to \GL_n(\T' / J')$ satisfying local-global compatibility at unramified places. After conjugation, we can assume that $\rho_{\n_1^Q} \text{ mod }\n_1^Q$ equals $\overline{\rho}_\m$. We first need to show that $\rho_{\n_1^Q}$ is a lifting of $\overline{\rho}_\m$ of type $\cS_{\chi, Q}$. The necessary conditions at places of $S$ can be checked just as in the proof of Proposition \ref{prop:existence_of_Hecke_Galois_with_LGC}. There is no condition at places of $Q$, so we obtain a morphism $f_{\cS_{\chi, Q}} : R_{\cS_{\chi, Q}}  \to \T' / J'$ (which in fact factors through the image of $\T$ in $\T' / J'$).
	
	It remains to check that $f_{\cS_{\chi, Q}}$ is a homomorphism of $\cO[\Delta_Q]$-algebras. Equivalently, we must check that it is a homomorphism of $\cO[\Delta_v]$-algebras for each place $v \in Q$. To this end, let us fix a place $v \in Q$. For each $i = 1, \dots, n$ we define a character $\psi_{v, i} : W_{F_v} \to (\T')^\times$ by the formula $\psi_{v, i}(\Art_{F_v}(\alpha)) = t_{v, i}(\alpha)$ (notation as in \S \ref{sec:some_useful_hecke_operators}). Theorem \ref{thm:lgc_at_l_neq_p} shows that (after possibly enlarging $J'$) for each $\sigma \in W_{F_v}$, we have the identity
	\[ \det(X - \rho_{\n_1^Q}(\sigma)) = \prod_{i=1}^n (X - \psi_{v, i}(\sigma)). \]
	Observe that the characters $\psi_{v, i} \text{ mod }\n_1^Q$ are pairwise distinct (because they take Frobenius to $\alpha_{v, i}$, and these elements of $k$ are pairwise distinct, by definition of a Taylor--Wiles datum). We can therefore apply \cite[Prop. 1.5.1]{bellaiche_chenevier_pseudobook} to conclude that $\rho_{\n_1^Q}|_{W_{F_v}}$ is isomorphic to $\oplus_{i=1}^n \psi_{v, i}$, which shows that $f_{\cS_{\chi, Q}}$ is indeed a homomorphism of $\cO[\Delta_v]$-algebras (cf. \S \ref{sec:TWdef} for the definition of the $\cO[\Delta_v]$-algebra structure on $R_{\cS_{\chi, Q}}$). The proof is complete on taking $J$ to be the kernel of the map $\T \to \T' / J'$.
\end{proof}
We are now ready to begin the proof of Theorem  \ref{thm:R_equals_T}.
\begin{proof}[Proof of Theorem \ref{thm:R_equals_T}]
Let
  \[
   q = h^1(F_S/F,\ad\rhobar_{\frakm}(1)) \quad \text{and} \quad g = qn - n^2[F^+ : \Q],
  \]
and set $\Delta_\infty = \Z_p^{nq}$. 
Let $\cT$ be a power series ring over $\cO$ in $n^2\lvert S\rvert - 1$ many variables, and let $S_\infty = \cT\llbracket \Delta_\infty \rrbracket$. 
Viewing $S_\infty$ as an augmented $\cO$-algebra, we let $\ga_\infty$ denote the augmentation ideal. 

Enlarging $E$ if necessary, we can assume that $E$ contains a primitive $p$th root of unity. 
Then since $p>n$, for each $v\in R$ we can choose a tuple of pairwise
distinct characters $\chi_v = (\chi_{v,1},\ldots,\chi_{v,n})$, with $\chi_{v,i} \colon \cO_{F_v}^\times \rightarrow \cO^\times$ 
trivial modulo $\varpi$.
We write $\chi$ for the tuple $(\chi_v)_{v\in R}$ as well as for the induced character $\chi = \prod_{v\in R} \chi_v \colon \prod_{v\in R} I_v \rightarrow \cO^\times$.
For each $N\ge 1$, we fix a choice of Taylor--Wiles datum $(Q_N,(\alpha_{v,1},\ldots,\alpha_{v,n})_{v\in Q_N})$ as in Proposition \ref{thm:TWgen}
(this is possible by our assumption that $\overline{r_\iota(\pi)}(G_{F(\zeta_p)})$ is enormous; we choose any imaginary quadratic subfield of $F$ in the application of Proposition \ref{thm:TWgen}).
For $N = 0$, we set $Q_0 = \emptyset$. 
For each $N\ge 1$, we let $\Delta_N = \Delta_{Q_N}$ and fix a surjection $\Delta_\infty \rightarrow \Delta_N$. 
The kernel of this surjection is contained in $(p^N\Z_p)^{nq}$, since each $v\in Q$ satisfies $q_v \equiv 1 \bmod p^N$. 
We let $\Delta_0$ be the trivial group, viewed as a quotient of $\Delta_\infty$. 

For each $N\ge 0$, the auxiliary deformation problems $\cS_{1, Q_N}$ and $\cS_{\chi, Q_N}$ are defined, and we set $R_N = R_{\cS_{1,Q_N}}$ and $R_N' = R_{\cS_{\chi,Q_N}}$. 
Note that $R_0 = R_{\cS_1}$ and $R_0' = R_{\cS_\chi}$. 
Let $R^{\loc} = R_{\cS_1}^{S,\loc}$ and $R'^{\loc} = R_{\cS_\chi}^{S,\loc}$ denote the corresponding local deformation rings as in~\S\ref{sec:present}.
For any $N\ge 1$, we have $R_{\cS_{1,Q_N}}^{S,\loc} = R^{\loc}$ and $R_{\cS_{\chi,Q_N}}^{S,\loc} = R'^{\loc}$. 
There are canonical isomorphisms $R^{\loc}/\varpi \cong R'^{\loc}/\varpi$ and $R_N/\varpi \cong R_N'/\varpi$ for all $N\ge 0$. 
For each $N\ge 1$, $R_N$ and $R_N'$ are canonically $\cO[\Delta_N]$-algebras and there are canonical isomorphisms $R_N \otimes_{\cO[\Delta_N]} \cO \cong R_0$ 
and $R_N' \otimes_{\cO[\Delta_N]} \cO \cong R_0'$, which are compatible with the isomorphisms modulo $\varpi$.
By Lemma~\ref{thm:framepresentation}, we have an $R^{\loc}$-algebra structure on $R_N \widehat{\otimes}_{\cO} \cT$ and an $R'^{\loc}$-algebra structure on 
$R_N' \widehat{\otimes}_{\cO} \cT$. 
The canonical isomorphism $R^{\loc}/\varpi \cong R'^{\loc}/\varpi$ is compatible with these algebra structures and with the canonical isomorphisms 
$R_N/\varpi \cong R_N'/\varpi$. 
We let $R_\infty$ and $R_\infty'$ be formal power series rings in $g$ variables over $R^{\loc}$ and $R'^{\loc}$, respectively. 
Using Proposition~\ref{thm:genoverloc} when $N = 0$ (noting that $H^0(F_S/F,\ad\rhobar_{\frakm}(1)) = 0$, because $\overline{r_\iota(\pi)}|_{G_{F(\zeta_p)}}$ is irreducible and $\zeta_p \not\in F$), and 
Proposition~\ref{thm:TWprimes} when $N\ge 1$, there are local $\cO$-algebra surjections $R_\infty \rightarrow R_N$ and $R_\infty' \rightarrow R_N'$ for any $N\ge 0$. 
We can (and do) assume that these are compatible with our fixed identifications modulo $\varpi$, and with the isomorphisms 
$R_N \otimes_{\cO[\Delta_N]} \cO \cong R_0$ and $R_N' \otimes_{\cO[\Delta_N]} \cO \cong R_0'$.

Let $\cC_0 = R \Hom_\cO( R \Gamma(X_K, \cV_\lambda(1))_\m, \cO)[-d]$,
and let $T_0 = \bT^S(K,\cV_\lambda(1))_{\frakm}$. Then $H^i(\cC_0)[1/p] \cong \Hom_E(H^{d-i}(X_K, \cV_\lambda(1))_{\frakm}[1/p], E)$ as $T_0$-modules.
Similarly, we let $\cC_0' =  R \Hom_\cO( R \Gamma(X_K, \cV_\lambda(\chi^{-1}))_\m, \cO)[-d]$, and $T_0' = \bT^S(K,\cV_\lambda(\chi^{-1}))_{\frakm}$. 
For any $N\ge 1$, we let 
  \[
   \cC_N = R \Hom_{\cO[\Delta_N]}( R \Gamma_{K_0(Q)/K_1(Q)}(X_{K_1(Q)}, \cV_\lambda(1))_{\n_1^Q}, \cO[\Delta_N])[-d]\quad 
  \]
  and 
  \[
   T_N = _{K_0(Q)/K_1(Q)}\bT^{S \cup Q_N}(R \Gamma_{K_0(Q)/K_1(Q)}(X_{K_1(Q)}, \cV_\lambda(1))_{\n_1^Q}).
  \] 
  Similarly, we let 
  \[
  \cC_N' = R \Hom_{\cO[\Delta_N]}( R \Gamma_{K_0(Q)/K_1(Q)}(X_{K_1(Q)}, \cV_\lambda(\chi^{-1}))_{\n_1^Q}, \cO[\Delta_N])[-d]\quad 
  \]
  and 
  \[
  T_N' = _{K_0(Q)/K_1(Q)}\bT^{S \cup Q_N}(R \Gamma_{K_0(Q)/K_1(Q)}(X_{K_1(Q)}, \cV_\lambda(\chi^{-1}))_{\n_1^Q}).
  \] 
For any $N\ge 0$, there are canonical isomorphisms $\cC_N \otimes_{\cO[\Delta_N]}^{\bL} k[\Delta_N] \cong \cC_N' \otimes_{\cO[\Delta_N]}^{\bL} k[\Delta_N]$ 
in $\bD(k[\Delta_N])$. 
Using this isomorphism to identify $\End_{\bD(\cO)}(\cC_N \otimes_{\cO}^{\bL} k) =\End_{\bD(\cO)}(\cC_N' \otimes_{\cO}^{\bL} k)$, the images of $T_N$ and 
$T_N'$ in this endomorphism algebra are the same, and we denote it by $\overline{T}_N$.
By Lemma~\ref{lem:identification_of_complexes}, there are canonical isomorphisms $\cC_N \otimes_{\cO[\Delta_N]}^{\bL} \cO \cong \cC_0$ 
and $\cC_N' \otimes_{\cO[\Delta_N]}^{\bL} \cO \cong \cC_0'$ in $\bD(\cO)$, and 
these isomorphisms are compatible with our fixed isomorphisms modulo 
$\varpi$.
By Proposition \ref{prop:existence_of_Q_augmented_Hecke_Galois_with_LGC} we have nilpotent ideals $I_N$ of $T_N$ and $I_N'$ of $T_N'$ for each $N\ge 0$,
both of nilpotence degree $\le \delta$, and local $\cO[\Delta_N]$-algebra surjections $R_N \rightarrow T_N/I_N$ and $R_N' \rightarrow T_N'/I_N'$. 
The surjections are compatible with the canonical isomorphisms modulo $\varpi$.
Moreover, using the isomorphism $R_N/\varpi \cong R_N'/\varpi$ and letting $\overline{I}_N$ and $\overline{I}_N'$ denote the images of $I_N$ and $I_N'$, respectively, in $\overline{T}_N$, 
the induced surjections $R_N/\varpi \rightarrow \overline{T}_N/(\overline{I}_N + \overline{I}_N')$ and 
$R_N'/\varpi \rightarrow \overline{T}_N/(\overline{I}_N + \overline{I}_N')$ agree.
The maps $T_N \otimes_{\cO[\Delta]} \cO \rightarrow T_0$ and $T_N' \otimes_{\cO[\Delta]} \cO \rightarrow T_0'$ 
induce surjections onto $T_0/I_0$ and $T_0'/I_0'$ respectively (surjectivity 
follows from Chebotarev density and the existence of the Galois representations 
with coefficients in $T_0/I_0$ and $T_0/I_0'$).

The objects introduced above satisfy the setup described in~\S\ref{subsec:patchingsetup}. 
We can then apply the results of~\S\ref{subsec:patchedcomplexes} and obtain the following.
\begin{itemize}
 \item Bounded complexes $\cC_\infty$ and $\cC_\infty'$ of free $S_\infty$-modules, subrings $T_\infty \subset \End_{\bD(S_\infty)}(\cC_\infty)$ and $T_\infty' \subset \End_{\bD(S_\infty)}(\cC_\infty')$, and ideals 
 $I_\infty$ and $I_\infty'$ satisfying $I_\infty^\delta = 0$ and $I_\infty'^\delta = 0$. 
 We also have $S_\infty$-algebra structures on $R_\infty$ and $R_\infty'$ and $S_\infty$-algebra surjections $R_\infty \rightarrow T_\infty/I_\infty$ 
 and $R_\infty' \rightarrow T_\infty'/I_\infty'$. (See Proposition~\ref{prop:patchedproperties} and Remark~\ref{rem:liftStoR}.)
 \item Surjections of local $\cO$-algebras $R_\infty/\ga_\infty \rightarrow R_0$ and $R_\infty'/\ga_\infty \rightarrow R_0'$. 
 We have isomorphisms $\cC_\infty \otimes_{S_\infty}^{\bL} S_\infty/\ga_\infty \cong \cC_0$ and $\cC'_\infty \otimes_{S_\infty}^{\bL} S_\infty/\ga_\infty \cong \cC_0'$ in $\mathbf{D}(\cO)$, inducing maps $T_\infty \rightarrow T_0$ and 
 $T_\infty' \rightarrow T_0'$ that become surjective when composed with the projections $T_0 \rightarrow T_0/I_0$ and $T_0' \rightarrow T_0'/I_0'$, 
 respectively.
 We let $I_{\infty,0}$ and $I_{\infty,0}'$ denote the images of $I_\infty$ and $I_\infty'$, respectively, under these surjective maps. 
 Then the induced maps $R_\infty/\ga_\infty \rightarrow (T_0/I_0)/I_{\infty,0}$ and $R_\infty'/\ga_\infty \rightarrow (T_0'/I_0')/I_{\infty,0}'$ factor through 
 $R_\infty/\ga_\infty \rightarrow R_0$ and $R_\infty'/\ga_\infty \rightarrow R_0'$, respectively. (See Lemma~\ref{rem:Gal0} and Proposition~\ref{prop:backtolevel0}.)
 \item An isomorphism 
 \[ \cC_\infty \otimes_{S_\infty}^{\bL} S_\infty / \varpi \cong \cC_\infty' 
 \otimes_{S_\infty}^{\bL} S_\infty / \varpi \]
  in $\mathbf{D}(S_\infty / \varpi)$.
 Under this identification, $T_\infty$ and $T_\infty$ have the same image $\overline{T}_\infty$ in 
  \[
   \End_{\bD(S_\infty / \varpi)}(\cC_\infty \otimes_{S_\infty}^{\bL} S_\infty / \varpi) = \End_{\bD(S_\infty)}(\cC_\infty' \otimes_{S_\infty}^{\bL} S_\infty / \varpi).
  \]
 Let $\overline{I}_\infty$ and $\overline{I}_\infty'$ denote the images of $I_\infty$ and $I_\infty'$, respectively, in $\overline{T}_\infty$. 
 Then the actions of $R_\infty/\varpi \cong R_\infty'/\varpi$ on 
  \[
   H^\ast(\cC_\infty \otimes_{S_\infty}^{\bL} \varpi)/(\overline{I}_\infty + \overline{I}_\infty') \cong H^\ast(\cC_\infty' \otimes_{S_\infty}^{\bL} S_\infty / \varpi)/(\overline{I}_\infty + \overline{I}_\infty')
  \]
 are identified via $\cC_\infty \otimes_{S_\infty}^{\bL} S_\infty / \varpi \cong \cC_\infty' \otimes_{S_\infty}^{\bL} S_\infty / \varpi$. (See Proposition~\ref{prop:primeandunprime}.)
\end{itemize}

Recall that $R_\infty$ and $R_\infty'$ are power series rings over $R^{\loc}$ and $R'^{\loc}$, respectively, 
in $g = qn - n[F^+:\Q]$ many variables. 
By Lemma~\ref{lem:localFLcase}, we have:
\begin{itemize}
 \item Each generic point of $\Spec R_\infty/\varpi$ is the specialization of a unique generic point of $\Spec R_\infty$, and every generic 
 point of $\Spec R_\infty$ has characteristic zero. 
 Also, $\Spec R_\infty'$ is irreducible and has characteristic zero generic point.
 \item $R_\infty$ is equidimensional, and $R_\infty$ and $R_\infty'$ have the common dimension
  \[
     1 + g + n^2 \lvert S \rvert + \textstyle{\frac{n(n-1)}{2}}[F:\Q]. 
  \]
  Since $F$ is CM, the quantity $\ell_0$ for the locally symmetric space $X_K$ is $\ell_0 = n[F^+:\Q] - 1$. 
  Then since $\dim S_\infty = n^2\lvert S \rvert + qn$ and $g = qn - n[F^+:\Q]$, we have
    \[
     \dim R_\infty = \dim R_\infty' = \dim S_\infty - \ell_0.
    \]
\end{itemize}
Finally, the isomorphism $\cC_\infty\otimes_{S_\infty}^{\bL} S_\infty/\ga_\infty \cong \cC_0$ implies that $(\cC_\infty \otimes_{S_\infty}^{\bL} S_\infty/\ga_\infty)[1/p]$ has cohomology isomorphic to $\Hom_E(H^{d - \ast}(X_K,\cV_\lambda(1))_{\frakm}[1/p], E)$. 
So Theorem~\ref{thm:application_of_matsushima} implies that $H^\ast(\cC_\infty \otimes_{S_\infty}^{\bL} S_\infty/\ga_\infty)[1/p] \ne 0$ and that the cohomology is concentrated in degrees $[q_0,q_0+\ell_0]$. 
We have now satisfied all the assumptions of~\S\ref{sec:AvoidIharasetup}, so we can apply Proposition \ref{prop:full_support_for_unipotent_deformation_rinG}
to conclude that $H^\ast(\cC_\infty)$ is has full support over $R_\infty$, hence that $H^\ast(\cC_\infty \otimes^{\bL}_{S_\infty} S_\infty / \fra_\infty) = H^\ast(\cC_0)$ has full support over $R_\infty / (\ga_\infty)$, hence that $H^\ast(\cC_0)$ has full support over $R_{\cS_1}$. This concludes the proof.
\end{proof}

\subsubsection{End of the proof (Fontaine--Laffaille case)}\label{sec:proof_of_main_automorphy_lifting_theorem}

We now deduce Theorem \ref{thm:main_automorphy_lifting_theorem} from
Corollary \ref{cor:R_equals_T_implies_automorphy}. The proof will be
an exercise in applying soluble base change. We first state the
results that we need. Note that while up to now~$E$ has denoted the
coefficient field of our Galois representations, having carried out
our patching argument we no longer need this notation, and we find it
convenient to use~$E$ to denote a number field in the rest of the
proof. %
\begin{prop}\label{prop:soluble_base_change}
Fix an integer $n \geq 2$, a prime $p$, and an isomorphism $\iota : \overline{\bQ}_p \to \bC$. Let $F$ be an imaginary CM or totally real number field, and let $E / F$ be finite Galois extension such that $\Gal(E/F)$ is soluble and $E$ is also imaginary CM or totally real. Then:
\begin{enumerate}
\item Let $\pi$ be a cuspidal, regular algebraic automorphic representation of 
$\GL_n(\bA_F)$ of weight $\lambda = (\lambda_\tau)_{\tau \in \Hom(F, \bC)}$. 
Suppose that $r_\iota(\pi)|_{G_E}$ is irreducible. Then there exists a 
cuspidal, regular algebraic automorphic representation $\pi_E$ of 
$\GL_n(\bA_E)$ of weight $\lambda_{E, \tau} = \lambda_{\tau|_E}$ such that 
$r_\iota(\pi_E) \cong r_\iota(\pi)|_{G_E}$. If $w$ is a finite place of $E$ 
lying above the place $v$ of $F$, then we have $\rec_{E_w}(\pi_E) = 
\rec_{F_v}(\pi)|_{W_{E_w}}$.
\item Let $\rho : G_F \to \GL_n(\overline{\bQ}_p)$ be a continuous representation such that $\rho|_{G_E}$ is irreducible. Suppose that there exists a cuspidal, regular algebraic automorphic representation $\pi$ of $\GL_n(\bA_E)$ of weight $\lambda$ such that $\rho|_{G_E} \cong r_\iota(\pi)$. Define $\lambda_F = (\lambda_{F, \tau})_{\tau \in \Hom(F, \bC)}$ by the formula $\lambda_{F, \tau} = \lambda_{\tau'}$, where $\tau' : E \hookrightarrow \bC$ is any extension of $\tau$ from $F$ to $E$. Then $\lambda_F$ is independent of any choices, and there exists a cuspidal, regular algebraic automorphic representation $\pi_F$ of $\GL_n(\bA_F)$ of weight $\lambda_F$ such that $\rho \cong r_\iota(\pi_F)$. If $w$ is a finite place of $E$ lying above the place $v$ of $F$, then we have $\rec_{E_w}(\pi) = \rec_{F_v}(\pi_F)$.
\end{enumerate}
\end{prop}
\begin{proof}
In either case we can reduce, by induction, to the case that $E/F$ is cyclic of prime order. Let $\sigma \in \Gal(E/F)$ be a generator of the Galois group, and let $\eta : \Gal(E / F) \to \bC^\times$ be a non-trivial character. We first treat the first part of the proposition.  We claim that $\pi \otimes (\eta \circ \Art_F^{-1}) \not\cong \pi$. Otherwise, there would be an isomorphism
\[ r_\iota(\pi) \otimes \iota^{-1} \eta \cong r_\iota(\pi), \]
implying that $r_\iota(\pi)|_{G_E}$ is reducible. We can therefore apply \cite[Ch. 3, Theorem 4.2]{MR1007299} and \cite[Ch. 3, Theorem 5.1]{MR1007299} to conclude the existence of a cuspidal, regular algebraic automorphic representation $\Pi$ of $\GL_n(\bA_E)$ of weight $\lambda_E$ such that for almost all finite places $w$ of $E$ such that $\pi_{w|_F}$ is unramified, $\Pi_w$ is a lift of $\pi_{w|_F}$. The Chebotarev density theorem then implies that we must have $r_\iota(\Pi) \cong r_\iota(\pi)|_{G_E}$, so we can take $\pi_E = \Pi$.

We now treat the second part of the proposition. The isomorphism
$\rho|_{G_E} \cong r_\iota(\pi)$, together with strong multiplicity
one for $\GL_n$, implies that we have $\pi^\sigma \cong \pi$. By
\cite[Ch. 3, Theorem 4.2]{MR1007299} and \cite[Ch. 3, Theorem
5.1]{MR1007299}, there exists a cuspidal automorphic representation
$\Pi$ of $\GL_n(\bA_F)$, which is regular algebraic of weight $\lambda_F$, such that for almost all finite places $w$ of $E$ such that $\Pi_{w|_F}$ is unramified, $\pi_w$ is a lift of $\Pi_{w|_F}$. The Chebotarev density theorem then implies that we must have $r_\iota(\Pi)|_{G_E} \cong r_\iota(\pi) \cong \rho|_{G_E}$. Using the irreducibility of $\rho|_{G_E}$, we conclude that there is a twist $\Pi \otimes (\eta\circ\Art_F^{-1})^i$ such that $r_\iota(\Pi \otimes  (\eta\circ\Art_F^{-1})^i) \cong \rho$. We are done on taking $\pi_F = \Pi \otimes  (\eta\circ\Art_F^{-1})^i$. 
\end{proof}
\begin{proof}[Proof of Theorem \ref{thm:main_automorphy_lifting_theorem}]
For the convenience of the reader, we recall the hypotheses of Theorem \ref{thm:main_automorphy_lifting_theorem}. Let $F$ be an imaginary CM or totally real field, and let $c \in \Aut(F)$ be complex conjugation. We are given a continuous representation $\rho : G_F \to \GL_n(\overline{\bQ}_p)$ satisfying the following conditions:
\begin{enumerate}
\item $\rho$ is unramified almost everywhere.
\item For each place $v | p$ of $F$, the representation $\rho|_{G_{F_v}}$ is crystalline. The prime $p$ is unramified in $F$.
\item $\overline{\rho}$ is absolutely irreducible and decomposed generic (Definition \ref{defn:decomposed_generic}). The image of $\overline{\rho}|_{G_{F(\zeta_p)}}$ is enormous (Definition \ref{defn:enormous image}). 
\item There exists $\sigma \in G_F - G_{F(\zeta_p)}$ such that $\overline{\rho}(\sigma)$ is a scalar. We have $p > n^2$.
\item There exists a cuspidal automorphic representation $\pi$ of $\GL_n(\bA_F)$ satisfying the following conditions:
\begin{enumerate}
	\item $\pi$ is regular algebraic of weight $\lambda$, this weight satisfying 
	\[ \lambda_{\tau, 1} + \lambda_{\tau c, 1} - \lambda_{\tau, n} - \lambda_{\tau c, n} < p - 2n. \]
	\item There exists an isomorphism $\iota : \overline{\bQ}_p \to \bC$ such that $\overline{\rho} \cong \overline{r_\iota(\pi)}$ and the Hodge--Tate weights of $\rho$ satisfy the formula for each $\tau : F \hookrightarrow \overline{\bQ}_p$:
	\[ \text{HT}_\tau(\rho) = \{ \lambda_{\iota \tau, 1} + n-1, \lambda_{\iota \tau, 2} + n - 2, \dots, \lambda_{\iota \tau, n} \}. \]
	\item If $v | p$ is a place of $F$, then $\pi_v$ is unramified.
\end{enumerate}
\end{enumerate}
The case where $F$ is a totally real field can be reduced to the case where $F$ is totally imaginary by base change. We therefore assume now that $F$ is imaginary, and write $F^+$ for its maximal totally real subfield.  Let $K / F(\zeta_p)$ be the extension cut out by $\overline{\rho}|_{G_{F(\zeta_p)}}$. Choose finite sets $V_0, V_1, V_2$ of finite places of $F$ having the following properties:
\begin{itemize}
\item For each $v \in V_0$, $v$ splits in $F(\zeta_p)$. For each proper subfield $K / K' / F(\zeta_p)$, there exists $v \in V_0$ such that $v$ splits in $F(\zeta_p)$ but does not split in $K'$.
\item For each proper subfield $K / K' / F$, there exists $v \in V_1$ which does not split in $K'$.
\item There exists a rational prime $p_0 \neq p$ which is decomposed generic for $\overline{\rho}$, and $V_2$ is equal to the set of $p_0$-adic places of $F$.
\item For each $v \in V_0 \cup V_1 \cup V_2$, $v \nmid 2$, $v \nmid p$, and $\rho$ and $\pi$ are both unramified at $v$. 
\end{itemize}
If $E / F$ is any finite Galois extension which is $V_0 \cup V_1 \cup V_2$-split, then $\overline{\rho}|_{G_E}$ has the following properties:
\begin{itemize}
\item $\overline{\rho}(G_E) = \overline{\rho}(G_F)$ and $\overline{\rho}(G_{E(\zeta_p)}) = \overline{\rho}(G_{F(\zeta_p)})$. In particular, $\overline{\rho}|_{G_{E(\zeta_p)}}$ has enormous image and there exists $\sigma \in G_E - G_{E(\zeta_p)}$ such that $\overline{\rho}(\sigma)$ is a scalar.
\item $\overline{\rho}|_{G_E}$ is decomposed generic. Indeed, the rational prime $p_0$ splits in $E$. 
\end{itemize}
Let $E_0 / F$ be a soluble CM extension satisfying the following conditions:
\begin{itemize}
	\item Each place of $V_0 \cup V_1 \cup V_2$ splits in $E_0$ and the rational prime $p$ is unramified in $E_0$.
	\item For each finite place $w$ of $E_0$, $\pi_{E_0, w}^{\Iw_w} \neq 0$.
	\item For each finite prime-to-$p$ place $w$ of $E_0$, either both $\pi_{E_0, w}$ and $\rho|_{G_{E_{0,w}}}$ are unramified or $\rho|_{G_{E_{0, w}}}$ is unipotently ramified, $q_w \equiv 1 \text{ mod }p$, and $\overline{\rho}|_{G_{E_{0, w}}}$ is trivial.
	\item For each place $\overline{w} | p$ of $E_0^+$, $\overline{w}$ splits in $E_0$ and there exists a place $\overline{w}' \neq \overline{w}$ of $E_0^+$ such that $\overline{w}' | p$ and
	\[ \sum_{\overline{w}'' \neq \overline{w}, \overline{w}'} [ E^+_{0,\overline{w}''} : \bQ_p] > \frac{1}{2} [ E_0^+ : \bQ ]. \]
\end{itemize}
We can find imaginary quadratic fields $E_a, E_b, E_c$ satisfying the following conditions:
\begin{itemize}
	\item Each rational prime lying below a place of $V_0 \cup V_1 \cup V_2$ splits in $E_a \cdot E_b \cdot E_c$. The prime $p$ is unramified in $E_a \cdot E_b \cdot E_c$. 
	\item The primes $2, p$ split in $E_a$.
	\item If $l \not\in \{2, p \}$ is a rational prime lying below a place of $E_0$ at which $\pi_{E_0, w}$ or $\rho|_{E_{0,w}}$ is ramified, or which is ramified in $E_0 \cdot E_a \cdot E_c$, then $l$ splits in $E_b$.
	\item If $l \not\in \{2, p \}$ is a rational prime which is ramified in $E_b$, then $l$ splits in  $E_c$.
\end{itemize}
For example, we can choose any $E_a$ satisfying the given condition. Then we can choose $E_b = \Q(\sqrt{-p_b})$, where $p_b$ is a prime satisfying $p_b \equiv 1 \text{ mod }4$ and $p_b \equiv -1 \text{ mod }l$ for any prime $l \not\in \{2, p \}$ either lying below a place $w$ of $E_0$ at which $\pi_{E_0, w}$ or $\rho|_{E_{0,w}}$ is ramified, or ramified in $E_0 \cdot E_a$, and $E_c = \Q(\sqrt{-p_c})$, where $p_c \equiv 1 \text{ mod }4 p_b$ is any prime not equal to $p$. (Use quadratic reciprocity to show that $p_c$ splits in $E_b$.)

We let $E = E_0 \cdot E_a \cdot E_b \cdot E_c$. Then $E / F$ is a soluble CM extension in which each place of $V_0 \cup V_1 \cup V_2$ splits, and the following conditions hold by construction: 
\begin{itemize}
\item The prime $p$ is unramified in $E$. 
\item Let $R$ denote the set of prime-to-$p$ places $w$ of $E$ such that $\pi_{E, w}$ or $\rho|_{G_{E_w}}$ is ramified. Let $S_p$ denote the set of $p$-adic places of $E$. Let $S' = S_p \cup R$. Then if $l$ is a prime lying below an element of $S'$, or which is ramified in $E$, then $E$ contains an imaginary quadratic field in which $l$ splits.
\item If $w \in R$ then $\overline{\rho}|_{G_{E_w}}$ is trivial and $q_w \equiv 1 \text{ mod }p$.
\item The image of $\overline{\rho}|_{G_{E(\zeta_p)}}$ is enormous. The representation $\overline{\rho}|_{G_E}$ is decomposed generic.
\item There exists $\sigma \in G_E - G_{E(\zeta_p)}$ such that $\overline{\rho}(\sigma)$ is a scalar.
\item For each place $\overline{w} | p$ of $E^+$, there exists a place $\overline{w}' \neq \overline{w}$ of $E^+$ such that $\overline{w}' | p$ and
\[ \sum_{\overline{w}'' \neq \overline{w}, \overline{w}'} [ E^+_{\overline{w}''} : \bQ_p] > \frac{1}{2} [ E^+ : \bQ ]. \]
\end{itemize}
By the Chebotarev density theorem, we can find infinitely many places $v_0$ of $E$ of degree 1 over $\bQ$ such that $\overline{\rho}(\Frob_{v_0})$ is scalar and $q_{v_0} \not\equiv 1 \text{ mod }p$, $v_0\notin S'\cup R^c$ and the residue characteristic of $v_0$ is odd. Then $H^2(E_{v_0}, \ad \overline{\rho}) = H^0(E_{v_0}, \ad \overline{\rho}(1))^\vee = 0$. We choose $v_0,v'_0$ with distinct residue characteristics satisfying these conditions, and set $S=S'\cup \{v_0,v'_0\}$. Note that if $l_0,l'_0$ denotes the residue characteristic of $v_0$, $v'_0$, then $l_0$, $l'_0$ splits in any imaginary quadratic subfield of $E$.

We see that the hypotheses (\ref{item:first_fl_hyp})--(\ref{item:last_fl_hyp}) of~\S\ref{subsec:an_r_equals_t_theorem} are now satisfied for $E$, $\pi_E$, and the set $S$. We can therefore apply Corollary \ref{cor:R_equals_T_implies_automorphy} to $\rho|_{G_E}$ and Proposition \ref{prop:soluble_base_change} to conclude that $\rho$ is associated to a cuspidal, regular algebraic automorphic representation $\Pi$ of $\GL_n(\bA_F)$ of weight $\lambda$. Taking into account the final sentence of Corollary \ref{cor:R_equals_T_implies_automorphy}, we see that $\Pi_{E, w}$ is unramified if $w \not\in S$.

To finish the proof, we must show that $\Pi_v$ is unramified if $v$ is a finite place of $F$ such that $v \nmid p$ and both $\rho$ and $\pi$ are unramified at $v$. Using our freedom to vary the choice of places $v_0$, $v'_0$, we see that if $v \nmid p$ is a place of $F$ such that both $\rho$ and $\pi$ are unramified at $v$, then $\Pi_{E, w}$ is unramified for any place $w | v$ of $E$. This implies that $\rec_{F_v}(\Pi_v)$ is a finitely ramified representation of the Weil group $W_{F_v}$. Using the main theorem of \cite{ilavarma} and the fact that $\rho$ is unramified at $v$, we see that $\rec_{F_v}(\Pi_v)$ is unramified, hence that $\Pi_v$ itself is unramified. This concludes the proof.
\end{proof}

\subsection{The proof of Theorem \ref{thm:main_ordinary_automorphy_lifting_theorem}}

We proceed to the proof of the second main theorem of this chapter (Theorem 
\ref{thm:main_ordinary_automorphy_lifting_theorem}). As in the case of the 
first theorem, we begin by establishing the result under additional conditions 
(\S \ref{subsec:ord_r_equals_t}), then reduce the general case to this one by 
using soluble base change (\S 
\ref{sec:proof_of_main_ordinary_automorphy_lifting_theorem}).

\subsubsection{Application of the patching argument (ordinary case)}\label{subsec:ord_r_equals_t}

We take $F$ to be an imaginary CM number field, and fix the following 
data:
\begin{enumerate}
	\item\label{item:first_ord_hyp} An integer $n \geq 2$ and a prime $p > n$.
	\item A finite set $S$ of finite places of $F$, including the places above 
	$p$. 
	\item A (possibly empty) subset $R \subset S$ of places prime to $p$.
	\item A cuspidal automorphic representation $\pi$ of $\GL_n(\bA_F)$, 
	regular algebraic of some weight $\mu$. 
	\item A choice of isomorphism $\iota : \overline{\bQ}_p \cong \bC$.
\end{enumerate}
We assume that the following conditions are satisfied:
\begin{enumerate}
	\addtocounter{enumi}{5}
	\item If $l$ is a prime lying below an element of $S$, or which is ramified in $F$, then $F$ contains an imaginary quadratic field in which $l$ splits. In particular, each place of $S$ is split over $F^+$ and the extension $F / F^+$ is everywhere unramified. 
	\item The residual representation $\overline{r_\iota(\pi)}$ is absolutely 
	irreducible.
	\item \label{assumeord} If $v \in S_p$ then $\pi_v^{\Iw_v(1,1)} \neq 0$ and 
	$\pi$ is $\iota$-ordinary at 
	$v$ (in the sense of 
	\cite[Def.~5.3]{ger}).
	\item If $v \in R$, then $\pi_v^{\Iw_v} \neq 0$.
	\item If $v \in S - (R \cup S_p)$, then $\pi_v$ is unramified, $v\notin R^c$, and $H^2(F_v, 
	\ad \overline{r_\iota(\pi)}) = 0$. %
	\item $S-(R \cup S_p)$ contains at least two places with distinct residue characteristics.
	\item If $v \not\in S$ is a finite place of $F$, then $\pi_v$ is 
	unramified. 
	\item If $v \in R$, then $q_v \equiv 1 \text{ mod }p$ and 
	$\overline{r_\iota(\pi)}|_{G_{F_v}}$ is trivial.
	\item The representation $\overline{r_\iota(\pi)}$ is decomposed generic and the image
	of~$\overline{r_\iota(\pi)}|_{G_{F(\zeta_p)}}$ is enormous.
		\item\label{item:last_ord_hyp} If $v \in S_p$ then $[F_v:\bQ_p] > \frac{n(n+1)}{2}+1$ and 
		$\overline{r_\iota(\pi)}|_{G_{F_v}}$ is trivial.
\end{enumerate}

\begin{theorem}\label{thm:ord_automorphy}
	With assumptions (\ref{item:first_ord_hyp})--(\ref{item:last_ord_hyp}) as above, suppose given a continuous representation $\rho 
	: G_F \to \GL_n(\overline{\bQ}_p)$ and a weight $\lambda \in (\bZ_+^n)^{\Hom(F, \overline{\bQ}_p)}$ satisfying the following conditions: 
	\begin{enumerate}
		\item We have $\overline{\rho} \cong \overline{r_\iota(\pi)}$.
\item For each place $v | p$, there is an isomorphism
\[ \rho|_{G_{F_v}} \sim \left( \begin{array}{cccc} \psi_{v, 1} & \ast & \ast & \ast \\
0 & \psi_{v, 2} & \ast & \ast \\
\vdots &\ddots & \ddots & \ast \\
0 & \cdots & 0 & \psi_{v, n}
\end{array}\right), \]
where for each $i = 1, \dots, n$ the character $\psi_{v, i} : G_{F_v} \to \overline{\bQ}_p^\times$ agrees with the character
\[ \sigma \in I_{F_v} \mapsto \prod_{\tau \in \Hom(F_v, \overline{\bQ}_p)} \tau( \Art_{F_v}^{-1}(\sigma))^{-(\lambda_{\tau, n - i + 1} + i - 1)} \]
on the inertia group $I_{F_v}$.
\item\label{item:same_ord_component} For each place $v | p$ of $F$, for each $i = 1, \dots, n$, and for each $p$-power root of unity $x \in \cO_{F_v}$, we have
\[ \prod_{\tau \in \Hom(F_v, \overline{\bQ}_p)} \tau(x)^{\lambda_{\tau, n + 1 - i} - \mu_{\iota \tau, n +1 - i}} = 1. \]
		\item For each finite place $v\not\in S$ of $F$, $\rho|_{G_{F_v}}$ is 
		unramified.
		\item For each place $v \in R$, $\rho|_{G_{F_v}}$ is unipotently 
		ramified.
	\end{enumerate}
	Then $\rho$ is ordinarily automorphic of weight $\iota \lambda$: there exists an $\iota$-ordinary 
	cuspidal automorphic representation $\Pi$ of $\GL_n(\bA_F)$ of weight $\iota \lambda$ such that $\rho \cong 
	r_\iota(\Pi)$. Moreover, if $v$ is a finite place of $F$ and $v \not\in S$, then $\Pi_v$ is unramified.
\end{theorem}
Note that we do not prove an analogue of Theorem \ref{thm:R_equals_T} here, but rather only an analogue of Corollary \ref{cor:R_equals_T_implies_automorphy}. This is due to our poor understanding of the irreducible components of the local lifting rings of type $\cD_v^{\detord}$. Before giving the proof of Theorem \ref{thm:ord_automorphy}, we need to introduce some deformation rings, Hecke algebras, and complexes on which they act. These complexes will represent the ordinary part of completed \emph{homology} with $\cO$-coefficients. We will use the notation for ordinary parts established in \S \ref{sec:ord_statements}.

We define an open compact subgroup $K = \prod_v K_v$ of 
$\GL_n(\widehat{\cO}_F)$ as follows:
\begin{itemize}
	\item If $v \not\in S$ then $K_v = \GL_n(\cO_{F_v})$.
	\item If $v \in S_p$, then $K_v = \Iw_v(1,1)$.
	\item If $v \in R$, then $K_v = \Iw_v$.
	\item If $v \in S - (R \cup S_p)$, then $K_v = \Iw_{v,1}$ is the pro-$v$ Iwahori subroup of $\GL_n(\cO_{F_v})$.
\end{itemize}
Then (by Lemma \ref{lem:neat_subgroups}) $K$ is neat, so is a good subgroup of $\GL_n(\A_F^\infty)$. By Theorem \ref{thm:application_of_matsushima}, we can find a 
coefficient field $E \subset \overline{\bQ}_p$ and a maximal ideal $\ffrm \subset \T^{S}(K, \mu)^{\ord}$ of residue field $k$ such that $\overline{\rho}_\m \cong \overline{r_\iota(\pi)}$. If $v \in S_p$, we let $\Lambda_{1, v} =  \cO \llbracket \cO_{F_v}^\times(p)^n \rrbracket $. We define $\Lambda_1 = \widehat{\otimes}_{v \in S_p} \Lambda_{1, v}$, the completed tensor product being over $\cO$. The $n$-tuple of characters 
\[ \chi_{\mu, v, i} : \cO_{F_v}^\times(p) \to \cO^\times,\,\, x \mapsto \prod_{\tau \in \Hom_{\Q_p}(F_v, E)} \tau(x)^{-(\mu_{\iota \tau, n - i + 1} + i - 1)} \,\, (i = 1, \dots, n) \]
determines a homomorphism $p_{\mu, v} : \Lambda_{1, v} \to \cO$. We define 
$\wp_{\mu, v} = \ker p_{\mu, v}$, and write $\wp_{0, v}$ for the unique minimal 
prime of $\Lambda_{1, v}$ which is contained in $\wp_{\mu, v}$. We set 
$\Lambda_v = \Lambda_{1, v} / \wp_{0, v}$ and $\Lambda = \widehat{\otimes}_{v 
\in S_p} \Lambda_v$. We write $p_\mu : \Lambda \to \cO$ for the homomorphism 
induced by the $p_{\mu, v}$ and the universal property of the completed tensor 
product, and set $\wp_\mu = \ker p_\mu$. We use similar notation for 
$p_\lambda$; note that condition (\ref{item:same_ord_component}) in the 
statement of the theorem implies that $\wp_{0, v}$ is also the unique minimal 
prime contained in $\wp_{\lambda, v}$ for each $v \in S_p$.

 We define a global deformation problem for each character $\chi : K_R \to \cO^\times$ which is trivial modulo $\varpi$  by the formula
\[ \cS_\chi = (\overline{\rho}_\ffrm, S, \{ \cO \}_{v \in S - S_p} \cup \{ \Lambda_v \}_{v \in S_p}, \{ \cD_v^{\detord} \}_{v \in S_p} \cup \{ \cD_v^\chi \}_{v \in R} \cup \{ \cD_v^\square \}_{v \in S - (R \cup S_p)}). \]
We fix representatives $\rho_{\cS_\chi}$ of the universal deformations which are identified modulo $\varpi$ (via the identifications $R_{\cS_\chi} / \varpi \cong R_{\cS_1} / \varpi$). We define an $\cO[K_S]$-module $\cV_\mu(\chi^{-1}) = \cV_\mu \otimes_\cO \cO(\chi^{-1})$, where $K_S$ acts on $\cV_\mu$ by projection to $K_p$ and on $\cO(\chi^{-1})$ by projection to $K_R$. After possibly enlarging $E$, we can assume that $\rho$ takes values in $\cO$ and that $\rho \text{ mod }\varpi = \overline{\rho}_\m$; then $\rho$ is a lifting of $\overline{\rho}_\m$ of type $\cS_1$. 

If $c \geq 1$ is an integer, then we define 
\[ \Lambda_{1, c} = \cO[ \prod_{v \in S_p} \ker(T_n(\cO_{F_v} / \varpi_v^c) \to T_n(\cO_{F_v} / \varpi_v)) ]; \]
 it is naturally a quotient of $\Lambda_{1}$. For any $c \geq 1$, the complex $R \Gamma(X_{K(c, c)}, \cV_\mu(\chi^{-1}))^{\ord}$ is defined, as an object of $\mathbf{D}(\Lambda_{1, c})$. We define 
\[ A_1(\mu, \chi, c) = R \Hom_{\Lambda_{1, c}}( R \Gamma(X_{K(c, c)}, \cV_\mu(\chi^{-1}))^{\ord}, \Lambda_{1, c})[-d]. \]
It is a perfect complex in $\mathbf{D}(\Lambda_{1, c})$ (because $R \Gamma(X_{K(c, c)}, \cV_\mu(\chi^{-1}))^{\ord}$ is). The Hecke algebra $\T^{S, \ord}$ acts on this complex by transpose. Moreover, Corollary \ref{cor:independence_of_level} shows that for any $c' \geq c \geq 1$ there is a $\T^{S, \ord}$-equivariant isomorphism
\numequation\label{eqn:independence_of_level} A_1(\mu, \chi, c') \otimes^{\bL}_{\Lambda_{1, c'}} \Lambda_{1, c} \cong A_1(\mu, \chi, c)
\end{equation}
in $\mathbf{D}(\Lambda_{1, c})$. By construction, there are canonical $\T^{S, \ord}$-equivariant isomorphisms
\numequation\label{eqn:ord_identification_mod_l} A_1(\mu, \chi, c)\otimes^{\bL}_{\Lambda_{1, c}} \Lambda_{1, c} / \varpi \cong A_1(\mu, 1, c)\otimes^{\bL}_{\Lambda_{1, c}} \Lambda_{1, c} / \varpi \end{equation}
in $\mathbf{D}(\Lambda_{1, c} / \varpi)$. By \cite[Lemma 2.13]{KT}, we can find a perfect complex $A_1(\mu, \chi) \in \mathbf{D}(\Lambda_1)$ which comes equipped an action by $\T^{S, \ord}$ and with $\T^{S, \ord}$-equivariant isomorphisms
\[ A_1(\mu, \chi) \otimes^{\bL}_{\Lambda_1} \Lambda_{1, c} \cong A_1(\mu, \chi, c) \]
in $\mathbf{D}(\Lambda_{1, c})$ (for each $c \geq 1$) and
\[ A_1(\mu, \chi) \otimes^{\bL}_{\Lambda_1} \Lambda_1 /\varpi \cong  A_1(\mu, 1) \otimes^{\bL}_{\Lambda_1} \Lambda_1 / \varpi \]
in $\mathbf{D}(\Lambda_1 / \varpi)$. These isomorphisms are compatible with the isomorphisms (\ref{eqn:independence_of_level}) for $c' \geq c$ and with the isomorphisms (\ref{eqn:ord_identification_mod_l}) for varying characters $\chi$, trivial modulo $\varpi$. Finally, we define $A(\mu, \chi) = A_1(\mu, \chi) \otimes^{\bL}_{\Lambda_1} \Lambda \in \mathbf{D}(\Lambda)$.

Let $\nu \in X^\ast((\Res_{F / \Q} T)_E) = (\Z^n)^{\Hom(F, E)}$ be defined by 
\[ \nu_\tau = (0, 1, \dots, n-1) \]
for all $\tau \in \Hom(F, E)$. We define $B_1(\mu, \chi) = A_1(\mu, \chi) \otimes_\cO \cO(\nu + w_0^G \mu)^{-1}$, where $\cO(\nu + w_0^G \mu)^{-1}$ is the $\cO[T_n(F_p)]$-module described in \S \ref{sec_ordinary_part_of_smooth_rep}. (In particular, the action of $T_n(\cO_{F, p})$ extends uniquely to an action of the completed group algebra $\cO \llbracket T_n(\cO_{F, p}) \rrbracket$.) Thus $B_1(\mu, \chi)$ is a perfect complex in $\mathbf{D}(\Lambda_1)$, on which the algebra $\T^{S, \ord}$ acts. We define $B(\mu, \chi) = B_1(\mu, \chi) \otimes^{\bL}_{\Lambda_1} \Lambda$. 
\begin{lemma}
The complex $B_1(\mu, \chi)$ is independent of $\mu$. More precisely, for any $\mu' \in (\Z^n_+)^{\Hom(F, E)}$, there is a $\T^{S, \ord}$-equivariant isomorphism $B_1(\mu, \chi) \cong B_1(\mu', \chi)$ in $\mathbf{D}(\Lambda_1)$. 
\end{lemma}
\begin{proof}
This follows from Proposition \ref{prop:independence_of_weight} and \cite[Lemma 2.13]{KT}.
\end{proof}
\begin{cor}\label{cor:weight_specialization}
Let $\mu' \in (\Z^n_+)^{\Hom(F, E)}$. Then there is a $\T^{S, \ord}$-equivariant isomorphism in $\mathbf{D}(\cO)$:
\[ B_1(\mu, \chi) \otimes^{\bL}_{\Lambda_{1}} \cO(\nu + w_0^G \mu')^{-1} \cong A_1(\mu', \chi, 1) \otimes_\cO \cO(\nu + w_0^G \mu')^{-1}. \]
\end{cor}
\begin{proof}
By the lemma, it suffices to treat the case $\mu' = \mu$. In this case the left-hand side may be identified with 
\[ A_1(\mu, \chi) \otimes_\cO \cO(\nu + w_0^G \mu)^{-1} \otimes^{\bL}_{\Lambda_1} \cO(\nu + w_0^G \mu)^{-1} \cong A_1(\mu, \chi) \otimes^{\bL}_{\Lambda_1} \cO \otimes_\cO \cO(\nu + w_0^G \mu)^{-1} \]
Essentially by definition, this complex admits a $\T^{S, \ord}$-equivariant isomorphism to $A_1(\mu, \chi, 1) \otimes_\cO \cO(\nu + w_0^G \mu)^{-1}$ in $\mathbf{D}(\cO)$. This completes the proof. 
\end{proof}
Let $\T^{S, \Lambda_1} = \T^S \otimes_\cO \Lambda_1 \subset \T^{S, \ord}$.
\begin{prop}\label{prop:ord_hecke_LGC}
	There exists an integer $\delta \geq 1$, depending only on $n$ and $[F : \Q]$, an ideal $J \subset \T^{S, \Lambda_1}(A(\mu, \chi)_\m \otimes_\cO \cO(\nu + w_0^G \mu)^{-1})$ such that $J^\delta = 0$, and a continuous surjective homomorphism of $\Lambda$-algebras
	\[ f_{\cS_{\chi}} : R_{\cS_\chi} \to \T^{S, \Lambda_1}(A(\mu, \chi)_\m\otimes_\cO \cO(\nu + w_0^G \mu)^{-1}) / J \]
	such that for each finite place $v \not\in S$ of $F$, the characteristic polynomial of $f_{\cS_\chi} \circ \rho_{\cS_\chi}(\Frob_v)$ equals the image of $P_v(X)$ in $\T^{S, \Lambda_1}(A(\mu, \chi)_\m \otimes_\cO \cO(\nu + w_0^G \mu)^{-1}) / J$.
\end{prop}
\begin{proof}
    We will construct a compatible family of homomorphisms 
    \[ f_{\cS_{\chi}, c} : R_{\cS_\chi} \to \T^{S, \Lambda_1}(A(\mu, \chi, c)_\m\otimes_\cO \cO(\nu + w_0^G \mu)^{-1}) / J_c, \]
    one for each $c \geq 1$. The desired homomorphism $f_{\cS_\chi}$ is then obtained by passage to the limit, in a similar way to the proof of Theorem \ref{thm:lgcfl}. It even suffices to construct a family of homomorphisms
    \[  R_{\cS_\chi} \to \T^{S, \Lambda_1}( R \Gamma(X_{K(c, c)}, \cV_\mu(\chi^{-1}))^{\ord}_\m) / J_c; \]
    in fact, the Hecke algebras are the same (the isomorphism being given by transpose and twist by $\cO(\nu + w_0^G \mu)$). Finally, it even suffices to construct a family of homomorphisms
    \[ R_{\cS_\chi} \to \T^{S, \ord}( R \Gamma(X_{K(c, c)}, \cV_\mu(\chi^{-1}))^{\ord})_\m / J_c; \]
    an application of Carayol's lemma (cf. \cite[Lemma 2.1.10]{cht}) then implies that the image of $R_{\cS_\chi}$ is in fact contained in a nilpotent quotient of the subalgebra
    \[  \T^{S, \Lambda_1}( R \Gamma(X_{K(c, c)}, \cV_\mu(\chi^{-1}))^{\ord}_\m) \subset \T^{S, \ord}( R \Gamma(X_{K(c, c)}, \cV_\mu(\chi^{-1}))^{\ord})_\m. \]
    This family of homomorphisms can be constructed exactly as in the proof of Proposition \ref{prop:existence_of_Hecke_Galois_with_LGC}, with the appeal to Theorem \ref{thm:lgcfl} being replaced instead with an appeal to Theorem \ref{thm:lgcord_intro}; here we are using the characterization of the deformation functor $\cD_v^{\detord}$ given in \S \ref{sec:orddef}.
\end{proof}
We now need to describe the auxiliary objects associated to a choice of Taylor--Wiles datum $(Q,(\alpha_{v,1},\ldots,\alpha_{v,n})_{v\in Q})$ for $\cS_1$ (see~\S\ref{sec:TWprimes}), where each place of $Q$ is assumed to have residue characteristic split in some imaginary quadratic subfield of $F$. Once again, this datum is automatically a Taylor--Wiles datum for all the global deformation problems $\cS_\chi$, and so the auxiliary deformation problems $\cS_{\chi, Q}$ are defined, and the deformation ring $R_{\cS_{\chi, Q}}$ has a natural structure of $\cO[\Delta_Q]$-algebra, where $\Delta_Q = \prod_{v \in Q} \Delta_v = \prod_{v \in Q} k(v)^\times(p)^n$.

If $c \geq 1$ is an integer then we define two auxiliary level subgroups 
\[ K(c, c)_1(Q) \subset K(c, c)_0(Q) \subset K(c, c). \]
 They are good subgroups of $\GL_n(\A_F^\infty)$, determined by the following conditions:
\begin{itemize}
	\item If $v \not\in S \cup Q$, then $K(c, c)_1(Q)_v = K(c, c)_0(Q)_v = K(c, c)_v$.
	\item If $v \in Q$, then $K(c, c)_0(Q)_v = \Iw_v$ and $K(c, c)_1(Q)_v$ is the maximal pro-prime-to-$p$ subgroup of $\Iw_v$. 
\end{itemize}
Then there is a natural isomorphism $K(c, c)_0(Q) / K(c, c)_1(Q) \cong \Delta_Q$. We define~$A_1(\mu, \chi, Q, c)$ to be
\[ R \Hom_{\Lambda_{1, c}[\Delta_Q]}( R \Gamma_{K(c,c)_0(Q)/K(c,c)_1(Q)}(X_{K(c, c)_1(Q)}, \cV_\mu(\chi^{-1}))^{\ord}, \Lambda_{1, c}[\Delta_Q])[-d], \]
an object of $\mathbf{D}(\Lambda_{1, c}[\Delta_Q])$. The algebra $\T^{S \cup Q, \ord}_Q = \T^{S \cup Q, \ord} \otimes_{\T^{S \cup Q}} \T^{S \cup Q}_Q$ acts on $A_1(\mu, \chi, Q, c)$ by transpose. As in the case where $Q$ is empty, we can pass to the limit with respect to $c$ to obtain a complex $A_1(\mu, \chi, Q) \in \mathbf{D}(\Lambda_1[\Delta_Q])$ which comes equipped with an action of $\T^{S \cup Q, \ord}_Q$ and with $\T^{S \cup Q, \ord}_Q$-equivariant isomorphisms
\[ A_1(\mu, \chi, Q) \otimes^{\bL}_{\Lambda_1} \Lambda_{1, c} \cong A_1(\mu, \chi, Q, c) \]
in $\mathbf{D}(\Lambda_{1, c})$ (for each $c \geq 1$) and
\[ A_1(\mu, \chi, Q) \otimes^{\bL}_{\Lambda_1} \Lambda_1 / \varpi \cong  A_1(\mu, 1, Q) \otimes^{\bL}_{\Lambda_1} \Lambda_1 / \varpi \]
in $\mathbf{D}(\Lambda_1 / \varpi)$, all compatible with the similar data at level $c$. 
We define $\m^Q$ to be the contraction of $\m$ to $\T^{S \cup Q, \ord}$, and $\n^Q$ to be the ideal of $\T^{S \cup Q, \ord}_Q$ generated by $\m^Q$ and the elements $U_{v, i} - \alpha_{v, 1} \cdots \alpha_{v, i}$ ($v \in Q, i = 1, \dots, n $). 
\begin{lemma}
	The ideal $\n^Q$ occurs in the support of $H^\ast(A_1(\mu, \chi, Q))$. There are $\T^{S \cup Q, \ord}$-equivariant isomorphisms
	\[ A_1(\mu, \chi, Q)_{\n^Q} \otimes^{\bL}_{\Lambda_1[\Delta_Q]} \Lambda_1 \cong A_1(\mu, \chi)_{\m^Q} \cong A_1(\mu, \chi)_\m. \]
\end{lemma}
\begin{proof}
	These properties can be established in the same way as in the finite level (Fontaine--Laffaille) case. See \S \ref{subsec:an_r_equals_t_theorem}. We omit the details. 
\end{proof}
We define $A(\mu, \chi, Q) = A_1(\mu, \chi, Q) \otimes^{\bL}_{\Lambda_1} \Lambda$, and $_{\Delta_Q}\T^{S \cup Q, \Lambda_1}=\T^{S \cup Q, \Lambda_1}\otimes_{\cO}\cO[\Delta_Q]$. Note this acts on $A(\Lambda, \chi, Q)_{\n^Q}$ via our identifications
$$K(c,c)_0(Q)/K(c,c)_1(Q)\cong \Delta_Q$$
 for each $c$ and passing to the limit. Thus $_{\Delta_Q}\T^{S \cup Q, \Lambda_1}(A(\Lambda, \chi, Q)_{\n^Q})$ is a local $\Lambda[\Delta_Q]$-algebra.
\begin{prop}\label{prop:ord_Q_augmented_hecke_LGC}
		There exists an integer $\delta \geq 1$, depending only on $n$ and $[F : \Q]$, an ideal $J \subset _{\Delta_Q}\T^{S \cup Q, \Lambda_1}(A(\Lambda, \chi, Q)_{\n^Q} \otimes_\cO \cO(\nu + w_0^G \mu)^{-1})$ such that $J^\delta = 0$, and a continuous surjective homomorphism of $\Lambda[\Delta_Q]$-algebras
	\[ f_{\cS_{\chi, Q}} : R_{\cS_{\chi, Q}} \to  _{\Delta_Q}\T^{S \cup Q, \Lambda_1}(A(\mu, \chi, Q)_{\n^Q} \otimes_\cO \cO(\nu + w_0^G \mu)^{-1}) / J \]
	such that for each finite place $v \not\in S$ of $F$, the characteristic polynomial of $f_{\cS_\chi} \circ \rho_{\cS_\chi}(\Frob_v)$ equals the image of $P_v(X)$ in $ _{\Delta_Q}\T^{S \cup Q, \Lambda_1}(A(\mu, \chi, Q)_{\n^Q} \otimes_\cO \cO(\nu + w_0^G \mu)^{-1}) / J$.
\end{prop}
\begin{proof}
	The existence of a $\Lambda$-algebra homomorphism 
	\[ R_{\cS_{\chi, Q}} \to  _{\Delta_Q}\T^{S \cup Q, \Lambda_1}(A(\mu, \chi, Q)_{\n^Q} \otimes_\cO \cO(\nu + w_0^G \mu)^{-1}) / J \]
	 satisfying the given condition at finite places $v\not\in S \cup Q$ of $F$ is proved just as in the proof of Proposition \ref{prop:ord_hecke_LGC} above. The key point is to show that this is a homomorphism of $\Lambda[\Delta_Q]$-algebras. This can be proved in the same way as in the proof of Proposition \ref{prop:existence_of_Q_augmented_Hecke_Galois_with_LGC}, by considering the enlarged algebra $\T^{S \cup Q, \ord}_Q(A(\mu, \chi, Q) \otimes_\cO \cO(\nu + w_0^G \mu)^{-1})_{\n^Q}$.
\end{proof}
We are now ready to begin the proof of Theorem \ref{thm:ord_automorphy}.
\begin{proof}[Proof of Theorem \ref{thm:ord_automorphy}]
	We recall that we have constructed a homomorphism $f : R_{\cS_1} \to \cO$, classifying the representation $\rho$ that we wish to show is automorphic. We will show that 
	$\ker f$ is in the support of 
	\[ H^\ast(B(\mu, 1)_\m \otimes^{\bL}_{\Lambda}  \cO(\nu + w_0^G \lambda)^{-1}). \]
	By Corollary \ref{cor:weight_specialization}, this will show that $\ker f$ is in the support of 
	\[ H^\ast( A(\lambda, 1, 1)_\m \otimes_\cO \cO(\nu + w_0^G \lambda)^{-1})[1/p], \]
	in turn a quotient of
	\[ \Hom_E(H^{d - \ast}(X_{K(1, 1)}, \cV_\lambda)_\m, \cO(\nu + w_0^G \lambda)^{-1}[1/p]). \]
	The $\iota$-ordinary automorphy of $\rho$ will then follow from Theorem \ref{thm:application_of_matsushima}. 
	
	Our proof now closely follows the proof of Thm.~\ref{thm:R_equals_T}. Let
	\[
	q = h^1(F_S/F,\ad\rhobar_{\frakm}(1)) \quad \text{and} \quad g = qn - 
	n^2[F^+ : \Q],
	\]
and set $\Delta_\infty = \Z_p^{nq}$. 
	Let $\cT$ be a power series ring over $\Lambda$ in $n^2\lvert S\rvert - 1$ 
	many variables, and let $S_\infty = \cT\llbracket \Delta_\infty 
	\rrbracket$. 
	Viewing $S_\infty$ as an augmented $\Lambda$-algebra, we let $\ga_\infty$ 
	denote the augmentation ideal. 
	
	As in the proof of Thm.~\ref{thm:R_equals_T}, we choose a character $\chi = 
	\prod_{v\in R} \chi_v \colon \prod_{v\in R} \Iw_v \rightarrow \cO^\times$ 
	such that for each $v \in R$ the $n$ characters $\chi_{v,i} \colon 
	k(v)^\times \rightarrow \cO^\times$ are
	trivial modulo $\varpi$ and pairwise distinct.
	
	Let $R^{\loc} = R_{\cS^{\ord}_1}^{S,\loc}$ and $R'^{\loc} = 
	R_{\cS^{\ord}_\chi}^{S,\loc}$ denote the corresponding local deformation 
	rings as in~\S\ref{sec:present}. We let $R_\infty$ and $R_\infty'$ be 
	formal power series rings in $g$ variables over $R^{\loc}$ and $R'^{\loc}$, 
	respectively. 
	
	We can then apply the results of~\S\ref{subsec:patchedcomplexes} to 
	complexes $A(\mu, \chi, Q)_{\n^Q} \otimes_\cO \cO(\nu + w_0^G \mu)^{-1}$ (for choices of Taylor--Wiles data 
	$(Q,(\alpha_{v,1},\ldots,\alpha_{v,n})_{v\in Q})$, proved to exist using Proposition \ref{thm:TWgen}) and obtain the following.
		\begin{itemize}
			\item Bounded complexes $\cC_\infty$ and $\cC_\infty'$ of free 
			$S_\infty$-modules, subrings $T_\infty \subset 
			\End_{\bD(S_\infty)}(\cC_\infty)$ and $T_\infty' \subset 
			\End_{\bD(S_\infty)}(\cC_\infty')$, and ideals 
			$I_\infty$ and $I_\infty'$ satisfying $I_\infty^\delta = 0$ and 
			$I_\infty'^\delta = 0$. 
			We also have $S_\infty$-algebra structures on $R_\infty$ and 
			$R_\infty'$ and $S_\infty$-algebra surjections $R_\infty 
			\rightarrow T_\infty/I_\infty$ 
			and $R_\infty' \rightarrow T_\infty'/I_\infty'$.
			\item Surjections of local $\Lambda$-algebras $R_\infty/\ga_\infty 
			\rightarrow 
			R_{\cS^{\ord}_1}$ and $R_\infty'/\ga_\infty \rightarrow 
			R_{\cS^{\ord}_\chi}$. 
			\item Isomorphisms $\cC_\infty \otimes_{S_\infty}^{\bL} 
			S_\infty/\ga_\infty \cong A(\mu, 1)_\m \otimes_\cO \cO(\nu + w_0^G \mu)^{-1} = B(\mu, 1)_\m$ and $\cC'_\infty 
			\otimes_{S_\infty}^{\bL} 
			S_\infty/\ga_\infty \cong A(\mu, \chi)_\m \otimes_\cO \cO(\nu + w_0^G \mu)^{-1} = B(\mu, \chi)_\m$ in $\mathbf{D}(\Lambda)$.
		\end{itemize}
	This gives the necessary input for~\S\ref{sec:AvoidIharasetup}.
	Recall that $R_\infty$ and $R_\infty'$ are power series rings over 
	$R^{\loc}$ and $R'^{\loc}$, respectively, 
	in $g = qn - n[F^+:\Q]$ many variables. It follows from parts (1) and (2) of Lemma~\ref{lem:localordcase} that we have satisfied 
	assumptions (1) and (2) of~\S\ref{sec:AvoidIharasetup}. To verify 
	assumption (3), if we let $\mathfrak{p}$ denote the inverse image in 
	$S_\infty$ of $\operatorname{Ann}_\Lambda( \cO(\nu + w_0^G \mu)^{-1})$, then (Corollary \ref{cor:weight_specialization}) the complex 
	\[(\cC_\infty 
	\otimes_{S_\infty}^{\bL} S_\infty/\mathfrak{p})[1/p] \cong 
	(B(\mu, 1)_\m \otimes^{\bL}_{\Lambda} \cO(\nu + w_0^G \mu)^{-1})[1/p] \]
	 has 
	cohomology isomorphic to a quotient of
	$\Hom_E(H^{d-*}(X_{K(1, 1)},\cV_\mu)_\m[1/p], E)$. Since $\pi$ contributes to this quotient, Theorem \ref{thm:application_of_matsushima} implies that 
	$H^\ast(\cC_\infty 
	\otimes_{S_\infty}^{\bL} S_\infty/\mathfrak{p})[1/p] \ne 0$ and 
	that the cohomology is concentrated in degrees $[q_0,q_0+\ell_0]$.

	We have now satisfied all the assumptions of~\S\ref{sec:AvoidIharasetup}, 
	and we apply Corollary~\ref{cor:char0automorphy} with $x \in \Spec(R_\infty)$ 
	the inverse image of $\ker f$, so $y \in \Spec(S_\infty)$ is the inverse 
	image of $\operatorname{Ann}_\Lambda( \cO(\nu + w_0^G \lambda)^{-1})$. For each $v \in S_p$, the inertial characters on the diagonal of 
	$\rho|_{G_{F_v}}$ are distinct, so $x$ lies on a 
	maximal dimension irreducible component of $\Spec(R_\infty)$ by part (3) of Lemma~\ref{lem:localordcase}, 
	and this Corollary does apply. We deduce that the support of 
	\[H^*(B(\mu, 1)_\m \otimes^\LL_{\Lambda} \cO(\nu + w_0^G \lambda)^{-1})\
	[1/p]\] contains $\ker f$. This completes the proof. 
\end{proof}

\subsubsection{End of the proof (ordinary case)}\label{sec:proof_of_main_ordinary_automorphy_lifting_theorem}

We can now deduce Theorem \ref{thm:main_ordinary_automorphy_lifting_theorem}, our main automorphy lifting result in the ordinary case, from Theorem \ref{thm:ord_automorphy}. The proof is a minor variation of the proof of our main automorphy lifting result in the Fontaine--Laffaille case (see~\S\ref{sec:proof_of_main_automorphy_lifting_theorem}).

\begin{proof}[Proof of Theorem 
\ref{thm:main_ordinary_automorphy_lifting_theorem}]
For the convenience of the reader, we recall the hypotheses of Theorem \ref{thm:main_ordinary_automorphy_lifting_theorem}. Let $F$ be an imaginary CM or totally real field, and let $c \in \Aut(F)$ be complex conjugation. We are given a continuous representation $\rho : G_F \to \GL_n(\overline{\bQ}_p)$ satisfying the following conditions:
\begin{enumerate}
\item $\rho$ is unramified almost everywhere.
\item For each place $v | p$ of $F$, the representation $\rho|_{G_{F_v}}$ is potentially semi-stable, ordinary with regular Hodge--Tate weights. In other words, there exists a weight $\lambda \in (\bZ_+^n)^{\Hom(F, \overline{\bQ}_p)}$ such that for each place $v | p$, there is an isomorphism
\[ \rho|_{G_{F_v}} \sim \left( \begin{array}{cccc} \psi_{v, 1} & \ast & \ast & \ast \\
0 & \psi_{v, 2} & \ast & \ast \\
\vdots &\ddots & \ddots & \ast \\
0 & \cdots & 0 & \psi_{v, n}
\end{array}\right), \]
where for each $i = 1, \dots, n$ the character $\psi_{v, i} : G_{F_v} \to \overline{\bQ}_l^\times$ agrees with the character
\[ \sigma \in I_{F_v} \mapsto \prod_{\tau \in \Hom(F_v, \overline{\bQ}_p)} \tau (\Art_{F_v}^{-1}(\sigma))^{-(\lambda_{\tau, n - i + 1} + i - 1)} \]
on an open subgroup of the inertia group $I_{F_v}$.
\item $\overline{\rho}$ is absolutely irreducible and generic. The image of $\overline{\rho}|_{G_{F(\zeta_p)}}$ is enormous. There exists $\sigma \in G_F - G_{F(\zeta_p)}$ such that $\overline{\rho}(\sigma)$ is a scalar. We have $p > n$.
\item There exists a regular algebraic cuspidal automorphic representation $\pi$ of $\GL_n(\bA_F)$ and an isomorphism $\iota : \overline{\bQ}_p \to \bC$ such that $\pi$ is $\iota$-ordinary and $\overline{r_\iota(\pi)} \cong \rho$.
\end{enumerate}
The case where $F$ is a totally real field can be reduced to the case where $F$ is totally imaginary by base change. We therefore assume now that $F$ is imaginary, and write $F^+$ for its maximal totally real subfield.  Let $K / F(\zeta_p)$ be the extension cut out by $\overline{\rho}|_{G_{F(\zeta_p)}}$. Choose finite sets $V_0, V_1, V_2$ of finite places of $F$ having the following properties:
\begin{itemize}
	\item For each $v \in V_0$, $v$ splits in $F(\zeta_p)$. For each proper subfield $K / K' / F(\zeta_p)$, there exists $v \in V_0$ such that $v$ splits in $F(\zeta_p)$ but does not split in $K'$.
	\item For each proper subfield $K / K' / F$, there exists $v \in V_1$ which does not split in $K'$.
	\item There exists a rational prime $p_0 \neq p$ which is decomposed generic for $\overline{\rho}$, and $V_2$ is equal to the set of $p_0$-adic places of $F$.
	\item For each $v \in V_0 \cup V_1 \cup V_2$, $v \nmid 2$, $v \nmid p$, and $\rho$ and $\pi$ are both unramified at $v$. 
\end{itemize}
If $E / F$ is any finite Galois extension which is $V_0 \cup V_1 \cup V_2$-split, then $\overline{\rho}|_{G_E}$ has the following properties:
\begin{itemize}
	\item $\overline{\rho}(G_E) = \overline{\rho}(G_F)$ and $\overline{\rho}(G_{E(\zeta_p)}) = \overline{\rho}(G_{F(\zeta_p)})$. In particular, $\overline{\rho}|_{G_{E(\zeta_p)}}$ has enormous image and there exists $\sigma \in G_E - G_{E(\zeta_p)}$ such that $\overline{\rho}(\sigma)$ is a scalar.
	\item $\overline{\rho}|_{G_E}$ is decomposed generic. Indeed, the rational prime $p_0$ splits in $E$. 
\end{itemize}
Let $E_0 / F$ be a soluble CM extension satisfying the following conditions:
\begin{itemize}
	\item Each place of $V_0 \cup V_1 \cup V_2$ splits in $E_0$.
	\item For each finite place $w$ of $E_0$, $\pi_{E_0, w}^{\Iw_w} \neq 0$.
	\item For each finite prime-to-$p$ place $w$ of $E_0$, either both $\pi_{E_0, w}$ and $\rho|_{G_{E_{0,w}}}$ are unramified or $\rho|_{G_{E_{0, w}}}$ is unipotently ramified, $q_w \equiv 1 \text{ mod }p$, and $\overline{\rho}|_{G_{E_{0, w}}}$ is trivial.
	\item For each place $w | p$ of $E_0$, $\overline{\rho}|_{G_{E_{0, w}}}$ is trivial and $[E_{0, w} : \Q_p] > n(n+1)/2 + 1$.
	\item For each place $v | p$ of $F$, for each $w | v$ of $E_0$, and for each $i = 1, \dots, n$ the character $\psi_{v, i} : G_{F_v} \to \overline{\bQ}_p^\times$ agrees with the character
	\[ \sigma \in I_{F_v} \mapsto \prod_{\tau \in \Hom(F_v, \overline{\bQ}_p)} \tau (\Art_{F_v}^{-1}(\sigma))^{-(\lambda_{\tau, n - i + 1} + i - 1)} \]
	on the whole of the inertia subgroup $I_{E_{0, w}} \subset I_{F_v}$.
	\item Let $\mu$ denote the weight of $\pi_{E_0}$. Then for each place $w | p$ of $E_0$, and for each $p$-power root of unity $x \in E_{0, w}$, we have
	\[ \psi_{v, i}(\Art_{E_{0, w}}(x)) \prod_{\tau \in \Hom(E_{0, w}, \overline{\bQ}_p)} \tau(x)^{\mu_{\iota\tau, n - i + 1} + i - 1} = 1.  \]
\end{itemize}
We can find imaginary quadratic fields $E_a, E_b, E_c$ satisfying the following conditions:
\begin{itemize}
	\item Each rational prime lying below a place of $V_0 \cup V_1 \cup V_2$ splits in $E_a \cdot E_b \cdot E_c$. 
	\item The primes $2, p$ split in $E_a$.
	\item If $l \not\in \{2, p \}$ is a rational prime lying below a place of $E_0$ at which $\pi_{E_0, w}$ or $\rho|_{E_{0,w}}$ is ramified, or which is ramified in $E_0 \cdot E_a \cdot E_c$, then $l$ splits in $E_b$.
	\item If $l \not\in \{2, p \}$ is a rational prime which is ramified in $E_b$, then $l$ splits in $E_c$.
\end{itemize}
For example, we can choose any $E_a$ satisfying the given conditions. Then we can choose $E_b = \Q(\sqrt{-p_b})$, where $p_b$ is a prime satisfying $p_b \equiv 1 \text{ mod }4$ and $p_b \equiv -1 \text{ mod }l$ for any prime $l \not\in \{2, p \}$ either lying below a place $w$ of $E_0$ at which $\pi_{E_0, w}$ or $\rho|_{E_{0,w}}$ is ramified, or ramified in $E_0 \cdot E_a$, and $E_c = \Q(\sqrt{-p_c})$, where $p_c \equiv 1 \text{ mod }4 p_b$ is a prime. (Use quadratic reciprocity to show that $p_c$ splits in $E_b$.)

We let $E = E_0 \cdot E_a \cdot E_b \cdot E_c$. Then $E / F$ is a soluble CM extension in which each place of $V_0 \cup V_1 \cup V_2$ splits, and the following conditions hold by construction: 
\begin{itemize}
	\item Let $R$ denote the set of prime-to-$p$ places $w$ of $E$ such that $\pi_{E, w}$ or $\rho|_{G_{E_w}}$ is ramified. Let $S_p$ denote the set of $p$-adic places of $E$. Let $S' = S_p \cup R$. Then if $l$ is a prime lying below an element of $S'$, or which is ramified in $E$, then $E$ contains an imaginary quadratic field in which $l$ splits.
	\item If $w \in R$ then $\overline{\rho}|_{G_{E_w}}$ is trivial and $q_w \equiv 1 \text{ mod }p$.
	\item The image of $\overline{\rho}|_{G_{E(\zeta_p)}}$ is enormous. The representation $\overline{\rho}|_{G_E}$ is decomposed generic.
	\item There exists $\sigma \in G_E - G_{E(\zeta_p)}$ such that $\overline{\rho}(\sigma)$ is a scalar.
	\item For each place $w | p$ of $E$, $\overline{\rho}|_{G_{E_{w}}}$ is trivial and $[E_{w} : \Q_p] > n(n+1)/2 + 1$.
	\item Let $\pi_E$ denote the base change of $\pi$ to $E$, which exists, by Proposition \ref{prop:soluble_base_change}. Then $\pi_E$ is $\iota$-ordinary, by \cite[Lemma 5.7]{ger}.
\end{itemize}

By the Chebotarev density theorem, we can find infinitely many places $v_0$ of $E$ of degree 1 over $\bQ$ such that $\overline{\rho}(\Frob_{v_0})$ is scalar and $q_{v_0} \not\equiv 1 \text{ mod }p$, $v_0\notin S'\cup R^c$ and the residue characteristic of $v_0$ is odd. Then $H^2(E_{v_0}, \ad \overline{\rho}) = H^0(E_{v_0}, \ad \overline{\rho}(1))^\vee = 0$. We choose $v_0,v'_0$ with distinct residue characteristics satisfying these conditions, and set $S=S'\cup \{v_0,v'_0\}$. Note that if $l_0,l'_0$ denotes the residue characteristic of $v_0$, $v'_0$, then $l_0$, $l'_0$ splits in any imaginary quadratic subfield of $E$.

We see that the hypotheses (\ref{item:first_ord_hyp})--(\ref{item:last_ord_hyp}) of~\S\ref{subsec:ord_r_equals_t} are now satisfied for $E$, $\pi_E$, and the set $S$. We can therefore apply Theorem \ref{thm:ord_automorphy} to $\rho|_{G_E}$ to conclude the existence of a cuspidal, regular algebraic automorphic representation $\Pi_E$ of $\GL_n(\bA_E)$ such that $\Pi_E$ is $\iota$-ordinary of weight $\lambda_E$ and $r_\iota(\Pi_E) \cong \rho|_{G_E}$. By Proposition \ref{prop:soluble_base_change} and \cite[Lemma 5.7]{ger}, we can descend $\Pi_E$ to obtain a cuspidal, regular algebraic automorphic representation $\Pi$ of $\GL_n(\bA_F)$ such that $\Pi$ is $\iota$-ordinary of weight $\lambda$ and $r_\iota(\Pi) \cong \rho$. Taking into account the final sentence of the statement of Theorem \ref{thm:ord_automorphy}, we see that $\Pi_{E, w}$ is unramified if $w \not\in S$.

To finish the proof, we must show that $\Pi_v$ is unramified if $v$ is a finite place of $F$ such that $v \nmid p$ and both $\rho$ and $\pi$ are unramified at $v$. Using our freedom to vary the choice of places $v_0$ $v'_0$, we see that if $v \nmid p$ is a place of $F$ such that both $\rho$ and $\pi$ are unramified at $v$, then $\Pi_{E, w}$ is unramified for any place $w | v$ of $E$. This implies that $\rec_{F_v}(\Pi_v)$ is a finitely ramified representation of the Weil group $W_{F_v}$. Using the main theorem of \cite{ilavarma} and the fact that $\rho$ is unramified at $v$, we see that $\rec_{F_v}(\Pi_v)$ is unramified, hence that $\Pi_v$ itself is unramified. This concludes the proof.
\end{proof}

\section{Applications}
\label{section:tricks}

\subsection{Compatible systems}

Suppose that $F$ is a number field. We will use a slight weakening of
the definition of a weakly compatible system from \cite{BLGGT}: By a
{\em rank $n$ very weakly compatible system~ $\CR$ of $l$-adic representations} 
{\em of} $G_F$ {\em defined over} $M$ we shall mean a $5$-tuple 
\[ (M,S,\{ Q_v(X) \}, \{r_\lambda \}, \{H_\tau\} ) \]
where
\begin{enumerate}
\item $M$ is a number field;
\item $S$ is a finite set of primes of~$F$;
\item for each  prime $v\not\in S$ of $F$, $Q_v(X)$ is a monic degree $n$
polynomial in $M[X]$;
\item for $\tau:F \into \barM$, $H_\tau$ is a multiset of $n$ integers;
\item for each prime $\lambda$ of $M$ (with residue characteristic $l$ say), 
\[r_\lambda:G_F \lra \GL_n(\barM_\lambda) \]
is a continuous, semi-simple representation such that 
\begin{enumerate}
\item if $v \notin S$ and $v \ndiv l$ is a prime of $F$, then $r_\lambda$
is unramified at $v$ and $r_\lambda(\Frob_v)$ has characteristic
polynomial $Q_v(X)$.
\item \label{extreme} For $l$ outside a set of primes of Dirichlet density $0$,
   the representation  $r_\lambda|_{G_{F_v}}$ is %
  crystalline for all~$v|l$,  and
for any $\barM \into \barM_\lambda$ over $M$, we have 
$\HT_\tau(r_\lambda)=H_\tau$.
\item \label{part:determinant} For all~$\lambda$, we have~$\HT_\tau(\det r_{\lambda}) = \sum_{h \in H_{\tau}} h$.
\end{enumerate}
\end{enumerate}
If we further  drop
hypothesis~(\ref{extreme}), then we say that~$\CR$ is an \emph{extremely weakly compatible system}.
The only dependence of an extremely weakly compatible system on~$H_{\tau}$ is via the condition
on the determinant via hypothesis~(\ref{part:determinant}).
The difference between very weakly compatible systems and the (merely) weakly compatible systems in~\cite{BLGGT} is
that, if $v|l$, then we only insist that $r_\lambda|_{G_{F_v}}$ is de
Rham  for $l$ in a set of Dirichlet density $1$. The notion of an extremely weakly compatible system is what used to be known as a compatible system,
but we use this language so as to emphasize that the condition of being a very weakly compatible system is more stringent than being an extremely
weakly compatible system. (Here we implicitly use the following fact: \emph{any}
compatible family of one dimensional representations is always de Rham \cite{henniart}.)
Of course, we expect that any extremely weakly compatible system should give rise to a weakly compatible system for an appropriate
choice of~$H_{\tau}$.
We have adopted the present definition so that, as a
consequence of Theorem~\ref{thm:lgcfl}, we can deduce that the Galois
representations constructed in~\cite{hltt} for~$n = 2$ form a  very weakly compatible
system. (See Lemma~\ref{lemma:sometimes2}.)

We will often write $l$ for the residue characteristic of a prime
$\lambda$ of $M$ without comment. We shall write $\barr_\lambda$ for the semi-simplified reduction of $r_\lambda$. The representation~$\barr_{\lambda}$  is  \emph{a priori}  defined over the algebraic closure of $\CO_M/\lambda$. However, because its trace lies in $\CO_M/\lambda$ and because the Brauer groups of all finite fields are trivial, it is actually a representation
\[ \barr_\lambda: G_F \lra \GL_n(\CO_M/\lambda). \]

 We recall some further definitions from section 5.1 of \cite{BLGGT}  which apply \emph{mutatis mutandis} to both very weakly  and extremely weakly compatible families:
 
A very (or extremely) weakly compatible system $\CR$  is {\em regular} if, for each~$\tau$, the set $H_\tau$ has $n$ distinct elements.

A very  (or extremely)  weakly compatible system $\CR$  is {\em irreducible}
if there is a set $\CL$ of rational primes of Dirichlet density $1$
such that, for $\lambda|l \in \CL$, the representation $r_\lambda$ is
irreducible. We say that it is \emph{strongly irreducible} if for all
finite extensions $F'/F$ the compatible system $\CR|_{G_{F'}}$ is
irreducible.

\begin{lem}\label{lem: rank 2 irreducible equivalences}
  If $\CR$ is an extremely weakly compatible system of rank $2$, then either $r_\lambda$ is irreducible for all $\lambda$ or there exist weakly
  compatible systems $\CX_1$ and $\CX_2$ of rank $1$ with
  $r_{\lambda}\cong\chi_{1,\lambda} \oplus \chi_{2,\lambda}$ for all $\lambda$.
\end{lem}
\begin{proof}
  Suppose that for one prime $\lambda_0$ the representation
  $r_{\lambda_0}$ is a sum of characters
  $r_{\lambda_0}=\chi_1 \oplus \chi_2$. By the main result of~\cite{henniart}, we see that 
 $r_{\lambda_0}$ is de Rham. Hence each $\chi_i$ is also de Rham and so there are weakly
  compatible systems $\CX_1$ and $\CX_2$ of rank $1$ with
  $\chi_{i,\lambda_0}=\chi_i$ for $i=1,2$.  Then for all $\lambda$ we have
  $r_{\lambda}\cong\chi_{1,\lambda} \oplus \chi_{2,\lambda}$.
\end{proof}

In view of Lemma~\ref{lem: rank 2 irreducible equivalences},  we say
that an extremely weakly compatible system of rank ~$2$ is
\emph{reducible} if it is not irreducible, in which case every
representation~$r_\lambda$ is reducible.  Say that a very (or
extremely) weakly compatible system of rank~$2$ is \emph{Artin up to twist}  if there exists an irreducible Artin representation~$\rho: G_F \rightarrow \GL_2(\overline{M})$ with traces in~$M$
(possibly after increasing~$M$)  and a weakly compatible system of one dimensional representations~$\chi_{\lambda}$
such that~$r_{\lambda} \simeq \rho \otimes \chi_{\lambda}$.

\begin{lem}
  \label{lem: rank 2 strongly irreducible}If $\CR$ is an extremely weakly compatible system of rank $2$ and~$\CR$ is irreducible, then either
  \begin{enumerate}
  \item $\CR$ is strongly irreducible, or
  \item $\CR$ is  Artin up to twist, or 
  \item there is
    a quadratic extension $F'/F$ and a weakly compatible system $\CX$
    of characters of $G_{F'}$ such that
    \[ \CR \cong \Ind_{G_{F'}}^{G_F} \CX, \]in which case we say
    that~$\cR$ is \emph{induced}.
  \end{enumerate}

\end{lem}
\begin{proof}
  Suppose that $\CR$ is not strongly irreducible, so that there exists a finite extension $E/F$ such that
  $\cR|_{G_E}$ is reducible. We may suppose that $E/F$ is
  Galois. Choose a prime $\lambda$ of $M$ of residue characteristic greater than~$2$. Write $r_\lambda|_{G_E}=\chi_{1} \oplus \chi_{2}$. 
  
  Suppose that~$\chi_1 = \chi_2 = \chi$. As in the proof of Lemma~\ref{lem: rank 2 irreducible equivalences},
  we deduce that~$\chi$ is de Rham. On the other hand, let~$\phi$ denote the determinant of~$r_{\lambda}$,
  and let~$\langle \phi \rangle$ be the character  such that~$\phi/\langle \phi \rangle$
  is the Teichm\"{u}ller lift of the reduction~$\phibar$ of~$\phi$.  
  Since~$\phibar$
  is a finite order character, we may assume (increasing~$E$ if necessary) that this character is trivial after restriction to~$G_{E}$.
  By construction, $\overline{\langle \phi \rangle}
= 1$ and thus (because~$\lambda$ is assumed to have odd residue characteristic)  $\langle \phi \rangle$
admits a square root character~$\psi$ as a representation of~$G_F$.
But then~$\psi^2 |_{G_E}$ and~$\chi^2$ coincide as representations of~$G_E$, since they are both equal to the determinant
of~$r_\lambda|_{G_E}$.  In particular, their ratio is a character of order dividing~$2$. Increasing~$E$ by a finite extension
if necessary, we may assume that~$\psi |_{G_E} = \chi$.
Hence~$\psi |_{G_{E}}$ is de Rham, and thus~$\psi$ is de Rham and extends to a compatible system of characters of~$G_F$.
After twisting~$\CR$ by this compatible system,
we may assume that~$r_{\lambda}|_{G_{E}}$ is trivial.
 In particular, $r_{\lambda}$ factors through~$\Gal(E/F)$, and is thus coming from  an Artin representation~$\rho_{\lambda}: G_F \rightarrow \GL_2(\barM_{\lambda})$, which automatically extends to a (strongly) compatible system coming
 from an Artin representation~$\rho: \Gal(E/F) \rightarrow \GL_2(\barM)$ with traces in some finite extension of~$M$ (specifically, the extension of~$M$ coming from the coefficient
 field of the compatible family~$\psi$). 
 Hence~$\CR$ is Artin up to twist in this case.

 Now assume that~$\chi_{1} \neq \chi_{2}$. The group $\Gal(E/F)$ permutes the two
  characters $\chi_{i}$ and, because $r_{\lambda}$ is irreducible,
  this action is transitive. Let $F'$ denote the the stabilizer of
  $\chi_{1}$. Then $\chi_{1}$ extends to a character of $G_{F'}$ and
  $r_\lambda = \Ind_{G_{F'}}^{G_F} \chi_1$. As in the proof of Lemma~\ref{lem: rank 2 irreducible equivalences}, there is a weakly compatible system of characters $\CX$ of
  $G_{F'}$ with $\chi_\lambda=\chi_1$. Then
  $\CR \cong \Ind_{G_{F'}}^{G_F} \CX$, as desired.
\end{proof}

\begin{lem}\label{resirred} If $\CR$  is an extremely weakly compatible system of rank $2$ and~$\CR$ is  irreducible,  then
for all $l$ in a set of Dirichlet density $1$ and all~$\lambda | l$, the residual representation $\barr_\lambda$ is absolutely irreducible.

If moreover~$\CR$ is neither induced nor Artin up to twist and $\tF$ denotes the normal closure of $F/\Q$, then one may additionally assume that the image $\barr_\lambda(G_{\tF})$  contains $\SL_2(\F_l)$.
\end{lem}

\begin{proof} This is immediate if~$\CR$ is Artin up to twist. If $\CR\cong \Ind_{G_{F'}}^{G_F} \CX$ then choose a prime $v\not\in S$ of $F$ which splits in $F'$ and such that $Q_v(X)$ has distinct roots. (If no such prime $v$ existed then we would have $\cX={}^\sigma\CX$, where $1 \neq \sigma \in \Gal(F'/F)$, contradicting the irreducibility of $\CR$.) Then for any $\lambda$ not dividing the residue characteristic of $v$ and modulo which $Q_v(X)$ still has distinct roots, we see that $\barr_\lambda$ is irreducible.

Hence we may assume that~$\CR$ is strongly
irreducible. In particular, since the only connected Zariski closed subgroups of~$\GL_2$ which act irreducibly contain~$\SL_2$, it
follows that
the Zariski closure of the image of~$r_{\lambda}$ contains~$\SL_2(\barM_{\lambda})$ 
for all~$\lambda$. We first prove, replacing~$M$ by a finite extension if necessary, that the Galois representations~$r_{\lambda}$
can all be made to land inside~$\GL_2(M_{\lambda})$.

The image of~$r_{\lambda}$  contains  an element with distinct eigenvalues.
 Hence, by the Cebotarev density theorem, there exists an auxiliary prime~$v \not\in S$ such that~$r_{\lambda}(\Frob_v)$ has distinct eigenvalues. These eigenvalues are defined
over a (at most) quadratic extension of~$M$. By enlarging~$M$ if necessary, we deduce that the images of~$r_{\lambda}$ for all~$\lambda \nmid N(v)$ contain
an element with distinct eigenvalues in~$M_{\lambda}$, which allows one to conjugate the representation~$r_{\lambda}$  to land in~$M_{\lambda}$. By choosing a second auxiliary
prime of different residue characteristic and enlarging~$M$ once again, we may ensure the image of~$r_{\lambda}$ lies in~$\GL_2(M_{\lambda})$ for all~$\lambda$. 

Let 
\[ s_l = \bigoplus_{\lambda|l} r_\lambda: G_F \lra \GL_{2[M:\Q]}(\Q_l), \]
so that $\CS = \{s_l\}$ form an extremely weakly compatible system with coefficients $\Q$. 
Let~$\mathbf{G}_l$  denote the Zariski closure of the image of $s_l$.
 It is contained in $(\RS^M_\Q \GL_2) \times_\Q \Q_l$. %
 The pushforward of $\mathbf{G}_l$ to $\GL_2/\barQQ_l$ via any embedding of $M \into \barQQ_l$ will contain $\SL_2$. 
  We will write
 $\mathbf{G}^{\circ}_l$ for the connected component of the identity of $\mathbf{G}_l$,  $\mathbf{G}_l^{\ad}$ for the quotient of $\mathbf{G}_l$ by its center and 
and~$\mathbf{G}^{\ssc}_l$ for the (simply connected) universal cover
of~$\mathbf{G}^{\ad}_l$. Then $\mathbf{G}_l^0$ is unramified for all
$l\in\cL$ a set of rational primes of Dirichlet density $1$ (see
\cite[Prop.8.9]{lp}). 
 Also over $\barQQ_l$, we see that $\mathbf{G}^{\ad}_l$ is contained in $\PGL_2^{[M:\Q]}$ and surjects onto each factor. 

The following facts are either well known or easy to check in the order indicated:
  \begin{enumerate}
  \item The only morphisms $\PGL_2 \ra \PGL_2$ over $\barQQ_l$ are the trivial map and conjugation by an element of $\PGL_2(\barQQ_l)$.
  \item The only morphisms $\PGL_2^r \ra \PGL_2$ over $\barQQ_l$ are the trivial map and projection onto one factor composed with conjugation by an element of $\PGL_2(\barQQ_l)$.
  \item If $I$ and $J$ are finite sets then, up to conjugation by an element of $\PGL_2(\barQQ_l)^J$, the only morphisms $\PGL_2^I \ra \PGL_2^J$ over $\barQQ_l$ are induced by a pair $(J_0,\phi)$ where $J_0 \subset J$ and $\phi:J_0 \ra I$. 
  \item If $I$ is a finite set then the  automorphism group of $\PGL_2^I$ is $\PGL_2^I \rtimes S_I$, where $S_I$ is the group of permutations of $I$.
  \item If $J$ is a finite set and $G$ is a connected algebraic subgroup of $\PGL_2^J$ over $\barQQ_l$ which surjects onto $\PGL_2$ via each projection, then $G \cong \PGL_2^I$ and the inclusion $\PGL_2^I \into \PGL_2^J$ corresponds, up to conjugation by an element of $\PGL_2(\barQQ_l)^I$ to a map $\phi:J \onto I$. 
 (Use induction on $\# J$ and Goursat's lemma.)
  \item If $M/\Q_l$ is a finite extension, then $(\RS^M_{\Q_l} \PGL_2) \times_{\Q_l} \barQQ_l \cong \PGL_2^{\Hom_{\Q_l}(M,\barQQ_l)}$ and the action of $G_{\Q_l}$ is via the map $G_{\Q_l} \ra S_{\Hom_{\Q_l}(M,\barQQ_l)}$ where $G_{\Q_l}$ acts by left translation.
   \item Forms of~$\PGL_2^r$ are classified by the middle term of the (split) exact sequence of pointed sets
 $$H^1(\Q_l,\PGL_2^r/\barQQ_l) \rightarrow H^1(\Q_l,\Aut(\PGL_2^r/\barQQ_l)) \rightarrow H^1(\Q_l,S_r)$$
 In order to split over an unramified extension, the image in~$H^1(\Q_l,S_r) = \Hom(G_{\Q_l},S_r)$ must be unramified and hence land in~$H^1(\FF_l,S_r)$. Every class in~$H^1(\FF_l,S_r)$ comes from the image of a group of the form~$G = \prod_i \RS^{N_i}_{\Q_l} \PGL_2$, where $N_i/\Q_l$ are unramified extensions. 
 On the other hand, the fibres of~$[G] \in H^1(\Q_l,S_r)$  are inner forms  of~$G$, and there is a unique quasi-split form amongst all inner forms. Since~$G$
 is quasi-split,  the only forms of~$\PGL_2^r$ which are unramified ($=$ quasi-split and split over an unramified extension) are thus given by~$\prod_i \RS^{N_i}_{\Q_l} \PGL_2$ for unramified~$N_i$.
 \item Suppose that, for $j \in J$ a finite set, $M_j/\Q_l$ is a finite extension, and that $G \subset \prod_{j\in J} \RS^{M_j}_{\Q_l} \PGL_2$ is an unramified connected algebraic subgroup over~$\Q_l$ such that, after base change to $\barQQ_l$, the projection of $G$ onto each factor of $\prod_{j \in J} (\RS^{M_j}_{\Q_l} \PGL_2) \times_{\Q_l} \barQQ_l \cong \PGL_2^{\coprod_j \Hom_{\Q_l}(M_j,\barQQ_l)}$ is surjective.  Then there are unramified extensions $N_i/\Q_l$ for $i$ in some finite set $I$ such that $G \cong \prod_{i \in I} \RS_{\Q_l}^{N_i} \PGL_2$. Moreover for each $j \in J$ and each $\tau:M_j \into \barQQ_l$ 
 the projection of the base change of $G$ to $\barQQ_l$ to the $(j,\tau)$ factor of $\prod_{j \in J} (\RS^{M_j}_{\Q_l} \PGL_2) \times_{\Q_l} \barQQ_l \cong \PGL_2^{\coprod_j \Hom_{\Q_l}(M_j,\barQQ_l)}$ is conjugate by an element of $\PGL_2(\barQQ_l)$ to projection onto one of the factors of $\prod_{i \in I} (\RS^{N_i}_{\Q_l} \PGL_2) \times_{\Q_l} \barQQ_l \cong \PGL_2^{\coprod_i\Hom_{\Q_l}(N_i,\barQQ_l)}$. 
    \end{enumerate}
 Thus for $l \in \cL$ there are finite unramified extensions $N_{l,i} / \Q_l$ for $i$ in some finite index set $I_l$ such that
 $\mathbf{G}_l^{\ad} \cong \prod_{i \in I_l} \RS_{\Q_l}^{N_{l,i}} \PGL_2$. Moreover for any prime $\lambda$ of $M$ there is an index $i\in I_l$ and an embedding $\tau:N_{l,i} \into \barM_\lambda$ such that the projection of $\mathbf{G}_l^{\ad} \times_{\Q_l} \barM_\lambda$ to $\PGL_2/\barM_\lambda$ is conjugate by an element of $\PGL_2(\barM_\lambda)$ to to the projection onto the $(i,\tau)$ factor of $\mathbf{G}_l^{\ad} \times_{\Q_l} \barM_\lambda \cong \PGL_2^{\coprod_{i\in I_l}\Hom_{\Q_l}(N_{l,i},\barM_\lambda)}$.

Let~$\Gamma_l$ denote the image of~$s_l$,  let~$\Gamma^{\circ}_l = \Gamma_l \cap \mathbf{G}^{\circ}$, let $\Gamma^{\ad}$ denote the image of $\Gamma_l^0$ in $\mathbf{G}_l^{\ad}$. By~\cite[Theorem 3.17]{larsen},
after replacing $\cL$ by a smaller set of Dirichlet density $1$, we may suppose that for $l \in \cL$ the group $\Gamma_l^{\ad}$ contains a conjugate of $\prod_{i \in I_l} \SL_2(\cO_{N_i})/\{ \pm 1_2\}$.
Thus, for $l \in \CL$ and $\lambda|l$, we may suppose that the image of $r_\lambda(G_F)$ in $\PGL_2(\barM_\lambda)$ contains $\SL_2(\Z_l)/\{ \pm 1_2\}$ and the image $\barr_\lambda(G_F)$ in $\PGL_2(\cO_{M}/\lambda)$ contains  $\SL_2(\F_l)/\{ \pm 1_2\}$. 

Now we may suppose that $l\in \cL$ implies that $l>3$ so that $\SL_2(\F_l)$ is perfect. Suppose $\lambda|l\in \cL$. 
 For every $g \in \SL_2(\F_l)$ the image of $\barr_\lambda$ contains an element $z(g)g$ where $z(g) \in (\cO_M/\lambda)^\times$ and is well defined modulo $Z=(\cO_M/\lambda)^\times \cap \im \barr$. Then $z$ defines a homomorphism $\SL_2(\F_l) \ra (\cO_M/\lambda)^\times/Z$ which must be identically $1$. Thus $\SL_2(\F_l)$ is contained in the image of $\barr_\lambda$.
 
  Finally if we remove finitely many primes from $\cL$ we may suppose that $\PSL_2(\F_l)$ is not a subquotient of $\Gal(\tF/F)$ from which the last assertion follows.
\end{proof}

We now prove some further preliminary lemmas concerning enormous and decomposed generic representations.

\begin{lem}\label{lemen}
If $n \geq 2$ and $l>2n+1$ and $H$ is a finite subgroup of $\GL_2(\barFF_l)$ containing $\SL_2(\F_l)$, then $\Symm^{n-1} H \subset \GL_2(\barFF_l)$ is enormous. 
\end{lem}
\begin{proof}
The image of $H$ in $\PGL_2(\barFF_l)$ must be conjugate to $\PSL_2(k)$ or $\PGL_2(k)$ for some finite extension $k/\F_l$. (See for instance~\cite[Thm.\
  2.47(b)]{MR1605752}.) Thus
\[\barFF_l^\times\Symm^{n-1}\GL_2(k)\supset H\supset\Symm^{n-1}\SL_2(k),\]
and the lemma follows from~\cite[Lem.\ 3.2.5]{geenew}.
\end{proof}

\begin{lem}\label{decoge} Suppose that $L$ is a number field, that $k$ is a finite field of characteristic $l$ and that $\barr:G_L \ra GL_n(k)$ is a continuous representation. Let $M$ denote the normal closure over $\Q$ of $\barF^{\ker \ad \barr}$. If $M$ does not contain a primitive $l^{th}$ root of unity, then $\barr$ is decomposed generic.
\end{lem}
\begin{proof} If a rational prime $p$ splits completely in $M$, but not in $M(\zeta_l)$, then $p$ is decomposed generic for $\barr$.
\end{proof}

\begin{lemma} \label{lemma:enormousgeneric} Suppose that $F/\Q$ is a finite extension with normal closure $\tF/\Q$ and that $m\in \Z_{>0}$. Suppose also that $l>2m+3$ is a rational prime and that $\barr:G_F \lra \GL_2(\barFF_l)$ is a continuous representation such that $\barr(G_\tF) \supset \SL_2(\F_l)$. Finally suppose that $F'/F$ is a finite extension which is linearly disjoint from $\barF^{\ker \barr}$ over $F$.
\begin{enumerate}
\item  \label{part:scalarslimeball} If $l$ is unramified in $F'/\Q$ then $\zeta_l \not\in \overline{F}^{\ker \ad \Symm^m \barr}F'$.
\item $(\Symm^m \barr)(G_{F'(\zeta_l)})$ is enormous. \label{part:enormousslimeball}
\item   \label{part:three} Let $\tF'$ denote the normal closure of $F'$ over $\Q$. Suppose that $\ad \barr(G_{\tF'}) \supset \PSL_2(\F_l)$,  then $\Symm^m\barr|_{G_{F'}}$ is decomposed generic.
 \item \label{part:genericslimeball1} Suppose that $F'/\Q$ is unramified at $l$ and that no quotient of $\im \ad \barr$ is unramified at all primes above $l$, then $\Symm^m\barr|_{G_{F'}}$ is decomposed generic.
\item \label{part:genericslimeball2} If $l>[\tF:F]$, then $\Symm^m\barr$ is decomposed generic.
\end{enumerate} \end{lemma}

\begin{proof}
The image $\barr_\lambda(G_F)$ in $\PGL_2(\barFF_l)$ must be conjugate to $\PSL_2(k)$ or $\PGL_2(k)$ for some finite extension $k/\F_l$. (See for instance~\cite[Thm.\
  2.47(b)]{MR1605752}.) 

For assertion~(\ref{part:scalarslimeball}), it suffices to treat the case $m=1$, in which case the assertion follows because $\Gal(F'(\zeta_l)/F')\cong (\Z/l\Z)^\times$, while $(\ad \barr)(G_{F'})=(\ad \barr)(G_F)$ does not surject onto $(\Z/l\Z)^\times$.

For assertion~(\ref{part:enormousslimeball}),  note that $\barr (G_{F'})=\barr(G_F) \supset \SL_2(\F_l)$ and so, because $\SL_2(\F_l)$ is perfect, we have that $\barr(G_{F'(\zeta_l)}) \supset \SL_2(\F_l)$. The assertion now follows from Lemma \ref{lemen}.

For assertion~(\ref{part:three}), it suffices to prove
that~$\Symm^m\barr|_{G_{F'}}$ is decomposed generic after replacing~$F'$ with
some finite extension. We first replace~$F'$ by $\tF'(\zeta_l)$, which we can do as $\PSL_2(\F_l)$ is perfect. Then, as above, (up to conjugacy) the image of $(\ad \barr)(G_{F'})$ is $\PSL_2(k)$ or $\PGL_2(k)$ for some finite extension $k/\F_l$.
Perhaps making a further extension, we may  assume that~$\ad \barr(G_{F'})=\PSL_2(k)$ for some finite extension $k/\F_l$, while maintaining the fact that~$F'/\Q$ is Galois.
 Let $H/F'$ denote the finite Galois extension with Galois group~$\PSL_2(k)$ cut out by this projective representation; and let $H'$ denote its normal closure over~$\Q$.
Using 
 the simplicity of~$\PSL_2(k)$, we deduce, from Goursat's Lemma, that~$\Gal(H'/F')
 = \PSL_2(k)^n$ for some~$n$. Moreover the conjugation action of any $\sigma \in \Gal(H'/\Q)$ on $\Gal(H'/F') \cong \PSL_2(k)^n$ is via an element of $\Aut(\PSL_2(k)^n) \cong (\PGL_2(k) \rtimes \Gal(k/\F_l))^n \rtimes S_n$. (To see this note two things. Firstly $\PSL_2(k)$ has automorphism group $\PGL_2(k) \rtimes \Gal(k/\F_l)$ - see for instance \cite{dieudonneauts}.
 Secondly the only normal subgroups of $\PSL_2(k)^n$ are $\PSL_2(k)^I$ for $I \subset \{1,...,n\}$, as can be seen by induction on $n$, and so any automorphism of $\PSL_2(k)^n$ permutes the $n$ factors of this product.)

There exists an element~$A \in \PSL_2(\F_l) \subset \PSL_2(k)$ such that a preimage in $\SL_2(\F_l)$ has two distinct $\F_l$-rational eigenvalues
  with ratio~$\alpha$ satisfying~$\alpha^{\pm i} \neq 1$ for~$1 \le i \le m$. %
By the Cebotarev density theorem, there exists a rational prime~$p$
 such that~$\langle \Frob_p \rangle$ in~$\Gal(H'/\Q)$ is (the conjugacy class of) the element~$(A,\ldots,A)$ in~$\PSL_2(k)^n=\Gal(H'/F')$. The image of this element is trivial in the quotient~$\Gal(F'/\Q)$, and thus, in addition, we 
 see that~$p$ splits completely in~$F'$ and (hence) that~$p \equiv 1 \mod l$. By construction, the ratio of any two roots of the characteristic polynomial of Frobenius of
\emph{any} prime above~$p$ in~$\Symm^{m} \barr$ is given
by~$\alpha^{\pm i}$ for~$i = 1, \ldots, m$.  %
 In particular,  these ratios are not equal to~$p \equiv 1 \mod l$. Hence~$\Symm^{m} \barr|_{G_{F'}}$
 is decomposed generic.
 
 Assertion~(\ref{part:genericslimeball1}) follows from assertion~(\ref{part:three}), because~$\tF'$ is unramified above~$l$ so that~$(\tF' \cap \barF^{\ker \ad \barr})/F$ is unramified above~$l$ and hence~$\tF' \cap \barF^{\ker \ad \barr}=F$ and~$\tF'$ is linearly disjoint from~$\barF^{\ker \ad \barr}$ over~$F$.

For assertion~(\ref{part:genericslimeball2}), note that $[\tF \cap \barF^{\ker \ad \barr}:\barF^{\ker \ad \barr}]<l$, so that we have $(\ad \barr)(G_{\tF \cap \barF^{\ker \ad \barr}}) \supset \PSL_2(\F_l)$. However, being Galois extensions of $F$, the fields $\tF$ and $\barF^{\ker \ad \barr}$ are linearly disjoint over $\tF \cap \barF^{\ker \ad \barr}$, so that again the result follows from assertion~(\ref{part:three}).
\end{proof}

\begin{lemma} \label{neverendinglemmas}
Suppose that~$\rbar: G_F \rightarrow \GL_n(\Fbar_l)$
is decomposed generic and absolutely irreducible. Let~$E/\Q$ be a Galois extension which is linearly disjoint from the Galois
closure of~$\barF^{\ker \rbar}(\zeta_l)$ over~$\Q$. Then~$\rbar
|_{G_{FE}}$ is decomposed generic and absolutely irreducible.
\end{lemma}

\begin{proof} The irreducibility claim is clear. Write~$H$ for the
  Galois closure of~$\barF^{\ker \rbar}(\zeta_l)$ over~$\Q$. As in the proof of Lemma~\ref{lem:weak_generic}, there exists a conjugacy class of elements~$\sigma \in \Gal(H/\Q)$
such that any rational prime unramified in~$H$ whose Frobenius element corresponds to~$\sigma$ is decomposed generic for~$\rbar$.
By assumption, $\Gal(HE/\Q) = \Gal(H/\Q) \times \Gal(E/\Q)$, and
now any rational prime whose conjugacy class in~$\Gal(HE/\Q)$
is of the form~$(\sigma,1) \in  \Gal(H/\Q) \times \Gal(E/\Q)$ will be
decomposed generic for~$\rbar |_{G_{FE}}$.
\end{proof}

\begin{lemma} \label{nuq} \leavevmode 
\begin{enumerate}
\item \label{part:nuq1}Suppose that $K/\Q_l$ is an unramified extension and that $r:G_K \ra \GL_2(\Z_l)$ is a crystalline representation with Hodge--Tate numbers $\{ 0,1\}$ for each embedding $K \into \barQQ_l$. Either $\barr|_{I_l}^\semis\cong 1 \oplus \barepsilon_l^{-1}$ or $\barr|_{I_l}\cong \omega_{l,2}^{-1} \oplus \omega_{l,2}^{-l}$.

\item \label{part:nuq2}Suppose that $K$ is a number field in which $l$ is unramified and that $r:G_K \ra \GL_2(\Z_l)$ is a crystalline representation with Hodge--Tate numbers $\{ 0,1\}$ for each embedding $K \into \barQQ_l$. If the image of $\barr$ contains $\SL_2(\F_l)$, then the only subextension of $\barK^{\ker \barr}/K$ unramified at all primes above $l$ is $K$ itself. 
\end{enumerate} \end{lemma}

\begin{proof}
The first part is presumably well known, but for lack of a reference
we give a proof. (Note the slight subtlety that the result would be false if we replaced
the coefficients~$\Z_l$ with the ring of integers in an arbitrary
extension of~$\Ql$, which is one obstacle to finding a suitable reference.) Recall that $r^\vee$ arises from the Tate module of a height 2 $l$-divisible group $G$ over $\cO_K$ (see \cite{BreuilAnnals} and \cite{MR2263197}.)
Moreover $G \neq G^0 \neq (0)$, as otherwise we would have Hodge--Tate numbers $\{1,1\}$ or $\{0,0\}$. Thus there is a finite flat group scheme $H$ over the ring of integers of the completion of the maximal unramified extension of $K$ of order $l^2$ killed by $l$ giving rise to $\barr^\vee$. Moreover $H \neq H^0 \neq (0)$.  By~\cite[Prop.\ 3.2.1, Thm.\ 3.4.3]{raynaud},
 either  $\barr|_{I_l}^\semis\cong 1 \oplus
 \barepsilon_l^{-1}$ or $\barr|_{I_l}\cong \omega_{l,2}^{-1} \oplus
 \omega_{l,2}^{-l}$ or $\barr|_{I_l}^\semis\cong 1 \oplus 1$ or
 $\barr|_{I_l}^\semis\cong \barepsilon_l^{-1} \oplus
 \barepsilon_l^{-1}$. If $l=2$ then $1=\barepsilon_l^{-1}$ and we are done, so suppose that
 $l>2$. Then since~$\det r$ is a crystalline character with all
 Hodge--Tate weights equal to~$1$, we have $\det
 r|_{I_l}=\epsilon_l^{-1}$, so the last two possibilities cannot occur,
 and the first part follows.

Consider the now the second part. It follows from the first part that the image under $\det \barr$ of any inertia group above $l$ is $\F_l^\times$, and so $\im \barr = \GL_2(\F_l)$. Let $\Delta$ denote the  subgroup of $\GL_2(\F_l)$ generated by the images of all inertia groups above $l$. It is a normal subgroup of $\GL_2(\F_l)$ which surjects under the determinant map onto $\F_l^\times$. But any normal subgroup of $\GL_2(\F_l)$ either contains $\SL_2(\F_l)$ or is central, and so $\Delta=\GL_2(\F_l)$ and the second part follows. 
\end{proof}

A very (or extremely) weakly compatible system $\CR$ is defined to be {\em pure} of weight $w$ if
\begin{itemize}
\item for each $v \not\in S$, each root $\alpha$ of $Q_v(X)$ in $\barM$ and each $\imath:\barM \into \C$
we have 
\[ | \imath \alpha |^2 = q_v^w; \]
\item and for each $\tau:F \into \barM$ and each complex conjugation $c$ in $\Gal(\barM/\Q)$ we have
\[ H_{c \tau} = \{ w-h: \,\,\, h \in H_\tau\}. \]
\end{itemize}
If $\CR$ is rank one then it is automatically pure. (See \cite{serreabladic}.) The same is true if $\CR$ is induced from an extremely weakly compatible system of characters over a finite extension of $F$, or if $\cR$ is Artin up to twist.

If $\CR$ is pure of weight $w$ and if $\imath: M \into \C$, then the partial L-function $L^S(\imath \CR, s)$ is defined as an analytic function in $\Re s > 1+w/2$. 
If $\CR$ is pure and regular and if $v$ is an infinite place of $F$,
then the Euler factor $L_v(\imath \CR, s)$ can be defined
(see~\cite[\S5.1]{BLGGT}). 

The very (or extremely) weakly compatible system $\CR$ is defined to be {\em automorphic} if there is a regular algebraic, cuspidal  automorphic
representation $\pi$ of $\GL_n(\A_F)$ and an embedding $\imath:M \into \C$, such that if $v \not\in S$, then $\pi_v$ is unramified and 
$\rec(\pi_v |\det|_v^{(1-n)/2})(\Frob_v)$ has characteristic
polynomial $\imath (Q_v(X))$. Note that if $\CR$ is automorphic, then
$L^S(\imath \CR,s)$ defines an analytic function in $\Re s \gg 0$ which,
for $n>1$, has analytic continuation to the whole complex plane. It
follows from \cite[Thm.\ 3.13]{MR1044819} that if $\CR$ is automorphic, then for any embedding $\imath':M \into \C$ there is a regular algebraic, cuspidal  automorphic
representation $\pi_{\imath'}$ of $\GL_n(\A_F)$ such that if $v \not\in S$, then $\pi_{\imath',v}$ is unramified and 
$\rec(\pi_{\imath',v} |\det|_v^{(1-n)/2})(\Frob_v)$ has characteristic
polynomial $\imath' (Q_v(X))$.

Suppose that~$F$ is a CM field and $\pi$ is a regular algebraic cuspidal
automorphic representation on $\GL_n(\A_F)$ of weight $(a_{\tau,i})$.
From the main theorems of~\cite{hltt} and \cite{ilavarma} we may
associate to~$\pi$ an extremely weakly compatible system
\[ \CR_\pi=(M_\pi,S_\pi,\{ Q_{\pi,v}(X) \}, \{r_{\pi,\lambda} \}, \{H_{\pi,\tau}\} ), \]
where
\begin{itemize}
\item $M_\pi \subset \C$ is the fixed field of $\{ \sigma \in \Aut(\C):\,\, {}^\sigma \pi^\infty \cong \pi^\infty\}$;
\item $S_\pi$ is the set of primes of $F$ with $\pi_v$ ramified;
\item $Q_{\pi,v}(X)$ is the characteristic polynomial of $\rec(\pi_v |\det|_v^{(1-n)/2})(\Frob_v)$;
\item $H_{\pi,\tau} =\{ a_{\tau,1}+n-1,\dots,a_{\tau,n} \}$.
\end{itemize}
We now note that this can be upgraded to a very weakly compatible system under some hypotheses.

\begin{lemma} \label{lemma:sometimes}
Let $F$ be a CM field and let $\pi$ be a regular algebraic cuspidal
automorphic representation on $\GL_n(\A_F)$ of weight $\xi = (a_{\tau,i})$.
Suppose that the following hypothesis holds:
\begin{itemize}
\item[(\HHH)] (decomposed generic and absolutely irreducible)
For a set of primes~$l$ of Dirichlet density one, the representations
$$\rbar_{\pi,\lambda}: G_F \rightarrow \GL_n(\CO_{M}/\lambda)$$
are decomposed generic and absolutely irreducible for all~$\lambda\mid
l$.
\end{itemize}
Then~$\CR_\pi$ is a very weakly compatible system. 
\end{lemma}

\begin{proof}
The lemma follows from
Theorem~\ref{thm:lgcfl} (taking~$p$ there to be our~$l$). Indeed, the assumption that~$\m$ is not Eisenstein is implied (for a set~$l$ of density one) by hypothesis~(\HHH).
Conditions~(\ref{part:part3}),~(\ref{part:part4}), and~(\ref{part:pbig}) hold automatically for large enough~$l$.
Similarly, $l$ will be unramified in~$F$ for large enough~$l$.
Conditions~(\ref{part:unramifiedandsplit}), (\ref{part:runningassumption}), and~(\ref{part:splittingcondition}),
can always be satisfied after after making a solvable
Galois base change~$F'/F$  (using~\cite{MR1007299}) which is disjoint over~$F$
from the Galois closure of~$\barF^{\ker \barr}$ over~$\Q$ and in which
all primes dividing either~$S$ or~$l$ are unramified. (We are free to make a different such base change for each prime~$l$.)
In particular, one can take the compositum of~$F$ with a Galois
extension~$E/\Q$ that is the compositum of various imaginary quadratic fields
in which all primes  dividing~$S$ or~$l$ split completely
for~(\ref{part:unramifiedandsplit}), (\ref{part:runningassumption}),
and the compositum with a large totally real cyclic extension~$E/\Q$ in which~$l$ splits completely
for condition~(\ref{part:splittingcondition}), where~$E$ may be easily be chosen
to be linearly disjoint over~$\Q$ from~$\barF^{\ker \rbar_{\pi,\lambda}}(\zeta_l)$. 
By Lemma~\ref{neverendinglemmas},
hypothesis~(\HHH) is preserved under such base extensions.
Condition~(\ref{part:therearecharzero}) holds by the existence of~$\pi$, and finally, 
condition~(\ref{part:decomposedgeneric}) holds for~$l$ for a set of~$l$ of density one, by hypothesis~(\HHH).
\end{proof}

\begin{lemma} \label{lemma:sometimes2}
Let $F$ be a CM field and let $\pi$ be a regular algebraic cuspidal
automorphic representation on $\GL_2(\A_F)$ of weight $\xi = (a_{\tau,i})$. Then the extremely weakly compatible system $\CR_\pi$ is irreducible. Moreover hypothesis (\HHH) of Lemma \ref{lemma:sometimes} holds and 
$\CR_\pi$ is a very weakly compatible system. 
\end{lemma}

\begin{proof}
If $\cR_\pi$ were reducible then, by Lemma~\ref{lem: rank 2 irreducible equivalences} and the automorphy of all weakly compatible systems of rank $1$, we see that there would be grossencharacters $\chi_1$ and $\chi_2$ of $\A_F^\times/F^\times$ such that $\pi_v \cong \chi_{1,v} \boxplus \chi_{2,v}$ for all but finitely many $v$.   By~\cite{MR623137}, this would contradict the cuspidality of $\pi$. Thus $\cR_\pi$ is irreducible. 

By Lemma \ref{lemma:sometimes} it only remains to verify hypothesis (\HHH). The absolute irreducibility condition follows from Lemma~\ref{resirred}. For the decomposed generic condition,
we treat the three possibilities of Lemma~\ref{lem: rank 2 strongly irreducible} separately.

 Suppose first that~$\CR_\pi$ is strongly irreducible.
By Lemma~\ref{resirred} and part~(\ref{part:genericslimeball2}) of Lemma~\ref{lemma:enormousgeneric}, we deduce that hypothesis~(\HHH) holds and so $\cR_\pi$ is very weakly compatible.

Suppose second that~$\CR_\pi \cong \Ind_{G_E}^{G_F} \cX$  for some quadratic extension $E/F$ and some very weakly compatible system of characters $\cX$ of $G_E$. 
Let $\tF$ (resp. $\tE$) denote the normal closure of $F$ (resp. $E$) over $\Q$, so that $\Gal(\tE/\tF)$ is an elementary abelian $2$-group. Let $1 \neq \tau \in \Gal(E/F)$. Then $\Gal(E/F)$ acts on $\Gal(\barE^{\ker \chi_\lambda/\chi_\lambda^\tau}/E)$ via the non-trivial character $\Gal(E/F) \ra \{\pm 1\}$. If $L^{(\lambda)}$ denotes the normal closure of $\barE^{\ker \barchi_\lambda/\barchi_\lambda^\tau}$ over $\Q$, then $L^{(\lambda)}/\tE$ is the compositum of abelian Galois extensions on which some subgroup of $\Gal(\tE/\Q)$ acts by a non-trivial character of order $2$.

Suppose that a rational prime $l$ is unramified in $\tE$. Then $\barepsilon_l(\Gal(L^{(\lambda)}/\tE))$ can have order at most $2$ (as any subgroup of $\Gal(\tE/\Q)$ will act trivially on it). Thus, if $l>3$ then $\zeta_l \not\in L^{(\lambda)}$ for any $\lambda$. It follows from Lemma \ref{decoge} that if $\lambda$ lies above a rational prime $l>3$ which is unramified in $\tE$, then $\barr_{\pi,\lambda}$ is decomposed generic.

Finally suppose that~$\CR_\pi$ is Artin up to twist, i.e. there exists an irreducible Artin representation $\rho:G_F \ra GL_2(\overline{M}_\pi)$ such that for all $\lambda$ the representation $r_{\pi,\lambda}$ is the twist of $\rho$ by some character. In particular $\barF^{\ker \ad \barr_{\pi,\lambda}} \subset \barF^{\ker \ad \rho}$. Let $L$ denote the normal closure of $\barF^{\ker \ad \rho}$ over $\Q$. If $l>2$ is unramified in $L$, then $\zeta_l \not\in L$ and by Lemma \ref{decoge} we see that $\barr_{\pi,\lambda}$ is decomposed generic for all $\lambda|l$.
\end{proof}

If $\imath_0$ is the canonical embedding $M_\pi \into \C$, then
$L^S(\imath_0 \CR_\pi,s)=L^S(\pi,s)$. If moreover $\CR_\pi$ is pure,
and  hypothesis~(\HHH) of Lemma~\ref{lemma:sometimes} holds,
then for each infinite place $v$ of $F$ we have $L_v(\imath_0
\CR_\pi,s)=L_v(\pi,s)$. (This follows from the definition of~$L_v(\imath_0
\CR_\pi,s)$ in~\cite[\S5.1]{BLGGT} together with the determination of
the Hodge--Tate weights of~$\CR_\pi$ in  Lemma~\ref{lemma:sometimes},
  and, in the case that~$F$ is totally real, the main
result of~\cite{caraianilehung}.)

The following is our main theorem.

\begin{thm}\label{mainthm} Suppose that $F/F_0$ is a finite Galois extension of CM fields. Suppose also that $F_0^\avoid$ is a finite Galois extension of $F$ and that $\cL_0$ is a finite set of rational primes. Suppose moreover that $\cI$ is a finite set and that for $i\in \cI$ we are given $m_i \in \Z_{>0}$ and a strongly irreducible rank $2$ very weakly compatible system of $l$-adic representations of $G_F$
\[ \cR_i=(M_i, S_i, \{Q_{i,v}(X)\}, \{ r_{i,\lambda}\}, \{ \{ 0,1\}\}) \]
with $S_i$ disjoint from $\cL_0$.

Then there is a finite set $\cL \supset \cL_0$ of rational primes; a
finite CM Galois extension $F^\suffi/F$ unramified above $\cL$, such
that $F^\suffi$ is Galois over $F_0$; and a finite Galois extension
$F^\avoid/F$ containing $F_0^\avoid$, which is linearly disjoint from
$F^\suffi$ over $F$; with the following property: For any finite CM 
extension $F'/F$ containing~$F^\suffi$ which is unramified above $\cL$ and 
linearly
disjoint from $F^\avoid$ over~$F$, the representations $\Symm^{m_i}
\cR_i|_{G_{F'}}$ are all automorphic, and each arises from an automorphic representation unramified above $\cL_0$.
\end{thm}

We have phrased this in a rather technical way in the hope that it will be helpful for applications. However let us record a simpler immediate consequence.

\begin{cor}
Suppose that $F$ is a CM field and that  the~$5$-tuple $\CR=(M,S,\{ Q_v(X) \},
\{r_\lambda \}, \{\{0,1\}\} )$ is a strongly irreducible rank $2$ very weakly compatible system
of $l$-adic representations of $G_F$. If $m$ is a non-negative integer, then there exists a finite Galois CM extension $F'/F$ such that the weakly compatible system $\Symm^m \CR|_{G_{F'}}$ is automorphic.
\end{cor}

Before proving Theorem~\ref{mainthm} in the next section, we record some consequences.

\begin{cor}[Potential modularity and purity for   rank two compatible systems over CM fields of weight zero and their symmetric powers] \label{maincor} Suppose that $F$ is a CM field and that the~$5$-tuple $\CR=(M,S,\{ Q_v(X) \}, \{r_\lambda \}, \{H_\tau\} )$ is an irreducible rank $2$ very weakly compatible system of $l$-adic representations of $G_F$ such that $H_\tau=\{0,1\}$ for all $\tau$. Suppose further that $m$ is a non-negative integer. Then:
\begin{enumerate}
\item $\CR$ is pure of weight $1$.
\item The partial L-functions $L^S(\imath \Symm^m \CR, s)$ have meromorphic continuation to the entire complex plane. 
\item For $v\in S$ there are Euler factors $L_v(\imath \Symm^m \CR,s)=P_{m,\imath,v}(q_v^{-s})^{-1}$, where $P_{m,\imath,v}$ is a polynomial of degree at most $m+1$ and $q_v$ is the order of the residue field of $v$, such that
\[ \Lambda(\imath \Symm^m\CR,s)=L^S(\imath \Symm^m\CR,s) \prod_{v|\infty} L_v(\imath\Symm^m \CR,s) \prod_{v\in S} L_v(\imath\Symm^m \CR,s) \]
satisfies a functional equation of the form
\[ \Lambda(\imath \Symm^m\CR,s) = A B^s \Lambda(\imath \Symm^m\CR^\vee,1-s). \]
\end{enumerate} 

Suppose further that $\CR$ is strongly irreducible and that
$m>0$. Then $L^s(\imath \Symm^m \CR,s)$ is holomorphic and non-vanishing
for~$\mathrm{Re}(s) \ge m/2 + 1$, and in particular has neither a pole nor a zero at~$s = m/2+1$.
\end{cor}

\begin{proof}
  If~$\CR$ is not strongly irreducible, then by Lemma~\ref{lem: rank 2 strongly irreducible}
 there is a quadratic extension $F'/F$ and a weakly
  compatible system $\CX$ of characters of $G_{F'}$ such that
  $\CR=\Ind_{G_{F'}}^{G_F} \CX$. In this case $\CX$ is pure,
  necessarily of weight $1$, and automorphic. The corollary follows
  easily.

  So suppose that $\CR$ is strongly irreducible.  Then for any
  positive integer~$m$  we see from Theorem~\ref{mainthm} that there is a finite Galois CM
  extension $F_m/F$ and, for any embedding $\imath:\barM \into \C$, a
  cuspidal automorphic representation $\pi_{\imath,m}$ of
  $\GL_{m+1}(\A_{F_m})$, such that for each $w|v \not\in S$ the roots
  of the characteristic polynomial of
  $\rec(\pi_{\imath,m,w} |\det|_w^{-m/2})(\Frob_w)$ are the images
  under $\imath$ of the roots of $Q_{\Symm^m \CR|_{G_{F_m}},w}(X)$.

  For the first part of the corollary we combine the
  `Deligne--Langlands method' with our theorem: because $\det \CR$ is
  pure of weight $2$, it suffices to show that for every
  $v \not\in S$, for every root $\alpha$ of $Q_{\CR,v}(X)$, and every
  $\imath:\barM \into \C$, we have
  \[ |\imath \alpha | \leq q_v^{1/2}. \] It even suffices to show that
  for every $m>0$, for every $v \not\in S$, for every root $\beta$ of
  $Q_{\Symm^m \CR,v}(X)$ and every $\imath:\barM \into \C$ we have
  \[ |\imath \beta | \leq q_v^{(m+1)/2}. \] (For then
  $|\imath \alpha| \leq q_v^{1/2+1/(2m)}$.) Equivalently, it suffices
  to show that for every $m>0$, for every $w | v \not\in S$, for every
  root $\gamma$ of $Q_{\Symm^m \CR|_{G_{F_m}},w}(X)$, and every
  $\imath:\barM \into \C$, we have
  \[ |\imath \gamma | \leq q_w^{(m+1)/2}. \]
 
  If $\chi_{\pi_{\imath,m}}$ denotes the central character of
  $\pi_{\imath,m}$, then we see that $\det \Symm^m \CR|_{G_{F_m}}$ is
  equivalent to $\CR_{\chi_{\pi_{\imath,m} ||\det||^{-m/2}}}$ and so
  \[ |\chi_{\pi_{\imath,m}}(x)| = 1 \] for all
  $x \in \A_{F_m}^\times$. Thus $\pi_{\imath,m}$ is unitary and,
  applying the bound of \cite[Cor.\ 2.5]{jsajm103} (which applies
  since each local factor of~$\pi_{\imath,m}$ is generic, by the final
  Corollary of~\cite{MR0348047}), we see that the image under
  $\imath$ of all the roots of the characteristic polynomial of
  $\rec(\pi_{\imath,m,w})(\Frob_w)$ have absolute value
  $\leq q_w^{1/2}$. Thus the absolute value of the image under
  $\imath$ of any root of $Q_{\Symm^m \CR|_{G_{F_m}},w}(X)$ is
  $\leq q_w^{(m+1)/2}$. The first part of the corollary follows.

  The rest of the corollary follows on using the usual Brauer's theorem
  argument (together with known non-vanishing properties of automorphic~$L$-functions
  as in~\cite{MR0432596})
  as in~\cite[Thm.\ 4.2]{HSBT}.
\end{proof}

\begin{cor}[Sato--Tate for Elliptic curves over CM fields]  \label{cor:satotate} Suppose that $F$ is a CM field and that $E/F$ is a non-CM elliptic curve. Then the numbers 
\[ (1+\# k(v) -\# E(k(v)))/2 \sqrt{\# k(v)} \]
are equidistributed in $[-1,1]$ with respect to the measure $(2/\pi)\sqrt{1-t^2} \, dt$.
\end{cor}
\begin{proof}
  This follows from Corollary~\ref{maincor} and the corollary to
  ~\cite[Thm.\ 2]{serreabladic}, as explained on page I-26 of
  \cite{serreabladic}.
\end{proof}

\begin{cor}[Ramanujan conjecture for weight~$0$ automorphic representations for~$\GL(2)$ over CM fields] \label{cor:ramanujan}
 Suppose that $F$ is a CM field and that $\pi$ is a regular
  algebraic cuspidal automorphic representation of $\GL_2(\A_F)$ of
  weight $(0)_{\tau,i}$. Then, for all  primes $v$ of
  $F$, the representation $\pi_v$ is tempered. \end{cor}
  
\begin{proof}The result is immediate for all primes~$v$ such
  that~$\pi_v$ is a twist of the Steinberg
  representation.  At the
  remaining places, since~$\pi_v$ is not a twist of the Steinberg
  representation, it follows from  the main theorems of~\cite{hltt}
  and~\cite{ilavarma}, together with ~\cite[Lem.\ 1.4 (3)]{ty}, that
  it suffices to prove that if~$v\nmid l$, then the
  restriction to~$G_{F_v}$ of any of the $l$-adic Galois
  representations associated to~$\pi$ is pure in the sense of~\cite[\S
  1]{ty}. By ~\cite[Lem.\ 1.4 (2)]{ty} and solvable base change, we can reduce to the case
  that~$\pi_v$ is unramified, in which case the result follows from
  Corollary~\ref{maincor}~(1), after noting
  by Lemma~\ref{lemma:sometimes2}
    that the compatible system~$\CR$ associated
  to~$\pi$ is very weakly compatible of the expected Hodge--Tate weights. %
\end{proof}

\begin{cor}\label{cor:harris pairs} Suppose that $F$ is a CM field and that the~$5$-tuples $\CR=(M,S,\{ Q_v(X) \}, \{r_\lambda \}, \{\{0,1\}\} )$
and~$\CR' = (M',S',\{ Q'_v(X) \}, \{r'_\lambda \}, \{\{0,1\}\} )$ are a pair of strongly irreducible rank $2$ very  weakly compatible systems of $l$-adic representations of $G_F$.
Suppose further that $m$ and~$m'$ are non-negative integers, and that~$\CR$ and~$\CR'$ are not twists of each other.
 Then~$L^S(\imath \Symm^m\CR \otimes \Symm^{m'} \CR',s)$ is meromorphic for~$s \in \C$, has no zeroes 
 or poles for~$\mathrm{Re}(s) \ge 1 + m/2 + m'/2$, and satisfies a functional equation relating $L^S(\imath \Symm^m\CR \otimes \Symm^{m'} \CR',s)$ and $L^S(\imath \Symm^m\CR^\vee \otimes \Symm^{m'} (\CR')^\vee,1+m+m'-s)$.
 \end{cor} 
 
\begin{proof} This follows from Theorem \ref{mainthm} by the same argument as~\cite{HarrisDouble} (for example Theorem~5.3 of \emph{ibid}).
(As usual, this argument involves the known non-vanishing results of Rankin--Selberg convolutions as established
in Theorem~5.2 of~\cite{MR610479}). \end{proof}

\subsection{Proof of the main potential automorphy theorem}\subsubsection{Preliminaries}

Before turning to the proof of Theorem~\ref{mainthm}, we record some preliminaries.

If $L/\Q_l$ is a finite extension and $\chibar$ (resp.\ $\chi$) is an unramified character of $G_L$ valued in $\F_{l}^\times$ (resp.\ $\Z_{l}^\times$), we will write $H^1_f(G_L,\F_{l}(\epsilon_l \chibar))$ (resp.\ $H^1_f(G_L,\Z_{l}(\epsilon_l \chi))$) for the kernel of the composite
\[ H^1(G_{L}, \F_l(\barepsilon_{l}\chibar)) \lra H^1(G_{L^\nr}, \F_l(\barepsilon_l)) \cong L^{\nr,\times}/(L^{\nr,\times})^l \lra \Z/l\Z \]
(resp.
\[ H^1(G_{L}, \Z_l(\epsilon_l\chi)) \lra H^1(G_{L^\nr}, \Z_l(\epsilon_l)) \cong \lim_{\leftarrow r}L^{\nr,\times}/(L^{\nr,\times})^{l^r} \lra \Z_l), \]
where the latter maps are induced by the valuation map. Note that if
$\chibar$ (resp.\ $\chi$) is non-trivial, then
\[ H^1_f(G_{L}, \F_l(\barepsilon_l\chibar))=H^1(G_{L}, \F_l(\barepsilon_l\chibar)) \]
(resp.\
\[ H^1_f(G_{L}, \Z_l(\epsilon_l\chi))=H^1(G_{L}, \Z_l(\epsilon_l\chi))). \] 
Also note that 
\[ H^1(G_{L}, \F_l(\barepsilon_l))/H^1_f(G_{L}, \F_l(\barepsilon_l)) \cong \F_l. \]

\begin{lem}\label{lem: H1f surjectivity}
The map
\[ H^1_f(G_{L}, \Z_l(\epsilon_l\chi)) \lra H^1_f(G_{L}, \F_l(\barepsilon_l\barchi)) \]
is always surjective.\end{lem}

\begin{proof}
  We will consider three cases. If the reduction of $\chi$ is
  non-trivial, we may suppress the $f$ and the cokernel is simply
  $H^2(G_{L}, \Z_l(\epsilon_l\chi))[l]$. Because
  $H^0(G_{L}, \Q_l/\Z_l(\chi^{-1}))=(0)$, Tate duality shows that this
  cokernel is zero.

  Suppose now that $\chi$ is non-trivial but that $\barchi$ is
  trivial. Using duality as above, we have an exact sequence
  \[ H^1(G_{L}, \Z_l(\epsilon_l\chi)) \lra H^1(G_{L},
    \F_l(\barepsilon_l)) \lra \F_l \lra (0). \] The image of
  $H^1(G_{L}, \Z_l(\epsilon_l\chi))=H^1_f(G_{L},
  \Z_l(\epsilon_l\chi))$ in $H^1(G_{L}, \F_l(\barepsilon_l))$ is
  contained in $H^1_f(G_{L}, \F_l(\barepsilon_l))$. As
  $H^1(G_{L}, \F_l(\barepsilon_l))/H^1_f(G_{L}, \F_l(\barepsilon_l))
  \cong \F_l$, we conclude that this image equals
  $H^1_f(G_{L}, \F_l(\barepsilon_l))$, as desired.

  Suppose finally that $\chi=1$. In this case the assertion of the
  lemma is just the surjectivity of
  \[ \lim_{\leftarrow r} \CO_L^\times/(\CO_L^\times)^{l^r} \onto
    \CO_L^\times/(\CO_L^\times)^l. \qedhere\]
\end{proof}

We will need a slight strengthening of~ \cite[Thm.\ 3.1.2]{BLGGT}, which we now state. We will use the notation and definitions from \cite{BLGGT}. The proof of this theorem given in \cite{BLGGT} immediately proves this variant also.

\begin{prop} \label{prop: potential modularity of symplectic over Q} Suppose that:
\begin{itemize}
\item $F/F_0$ is a finite, Galois extension of totally real fields, 
\item $\cI$ is a finite set,
\item for each $i \in \cI$, $n_i$ is a positive even integer,
  $l_i$ is an odd rational prime, and $\imath_i:\barQQ_{l_i} \iso
  \C$,%
\item $F^\avoid/F$ is a finite Galois extension, 
\item $\cL$ is a finite set of rational primes which are unramified in $F$ and not equal to
   $l_i$ for any $i \in \cI$, and %
\item $\bar{r}_i: G_{F} \to \GSp_{n_i}(\Fbar_{l_i})$ is a mod $l_i$
  Galois representation with open kernel and
multiplier $\barepsilon_{l_i}^{1-n_i}$, which is unramified above $\cL$. 
\end{itemize}
Then we can find finite Galois extensions $F^\suffi/F_0 $ and $F^\avoid_1/\Q$, such that 
\begin{itemize}
\item $F^\suffi$ contains $F$ and is linearly disjoint from $F^\avoid F_1^\avoid$ over $F$, 
\item $F^\avoid_1$ and $F^\avoid$ are linearly disjoint over $\Q$, and
\item $F^\suffi$ is totally real and unramified above $\cL$;
\end{itemize}
and which has the following property: For each finite totally real extension $F_1/F^\suffi$ which is linearly disjoint from $F^\avoid F^\avoid_1$ over $F$ and for each $i \in \cI$, there is a regular algebraic, cuspidal, polarized automorphic representation $(\pi_i,\chi_i)$ of $\GL_{n_i}(\A_{F_1})$ such that
\begin{enumerate}
\item $(\barr_{l_i,\imath_i}(\pi_i),\barr_{l_i,\imath_i}(\chi_i)\barepsilon_{l_i}^{1-n_i})\cong ({\barr}_i|_{G_{F_1}},
  \barepsilon_{l_i}^{1-n_i})$;
\item $\pi_i$ is $\imath_i$-ordinary of weight 0.
\end{enumerate}
\end{prop}

(In the notation of the proof of~\cite[Thm.\ 3.1.2]{BLGGT} one must
choose $N$ not divisible by any prime in $\CL$; $M_i/\Q$ unramified at primes in
$\CL$ and primes dividing $N$; $q$ unramified in $F^\avoid(\zeta_{4N})$ and not in $\cL$; $\phi_i$ unramified above $\CL$ and all rational primes that ramify in $F^\avoid$; and $l' \not\in \CL$ and not ramified in $F^\avoid$. We set $F_1^\avoid=\barQQ^{\ker \prod_i \barr_i'}(\zeta_{l'})$. It is linearly disjoint from $F^\avoid$ over $\Q$ because no rational prime ramifies in both these fields. We choose $F'/F(\zeta_N)^+$ to be linearly disjoint from $F^\avoid F_1^\avoid F(\zeta_N)^+$ over $F(\zeta_N)^+$ with
$F'/F(\zeta_N)^+$ unramified above $\CL$. The last choice is possible
because $\barr_i$ and $\barr_i'$ are unramified above $\CL$, so that
$\barr_i$ becomes isomorphic to $V_{n_i}[\lambda_i]((N-1-n_i)/2)_0$
and $\barr_i'$ becomes isomorphic to
$V_{n_i}[\lambda']((N-1-n_i)/2)_0$ over some unramified extension of
$F(\zeta_N)^+_v$ for any prime $v$ above $\CL$.  We
take~$F^\suffi$ to be the field~$F'$. The fields $F^\suffi$ and $F^\avoid F_1^\avoid$ are linearly disjoint over $F$ because $F^\avoid F_1^\avoid$ and $F(\zeta_N)^+$ are linearly disjoint over $F$, because, in turn, all primes dividing $N$ are unramified in $F^\avoid F_1^\avoid$. The point
$P\in\widetilde{T}(F')$ also provides a point of~$\widetilde{T}(F_1)$. Moreover $\barr_i'(G_{F_1(\zeta_{l'})})$ is adequate because $F_1$ is linearly disjoint from $\barF^{\ker \barr_i}(\zeta_{l'})$ over $F$.)

\begin{cor}\label{ecpm} Suppose that $\cM$ is a finite set of positive integers, that
  $E/\Q$ is a non-CM elliptic curve, and that $\CL$ is a finite set of
  rational primes at which $E$ has good reduction.  Suppose also that
  $F^\avoid/\Q$  is a finite extension. 

  Then we can find
  \begin{itemize}
  \item a finite Galois extension $F^{\avoid}_2/\Q$ linearly disjoint from $F^\avoid$ over $\Q$, and
  \item a finite totally real Galois extension $F^\suffi/\Q$ unramified above $\cL$ such that $F^\suffi$ is linearly disjoint from $F^\avoid F_2^\avoid$ over $\Q$;
  \end{itemize} 
which have the following property:

For any finite totally real extension $F'/F^\suffi$, which is linearly disjoint from $F_2^\avoid$ over $\Q$, and for any $m \in \cM$, there is a regular algebraic, cuspidal, polarizable automorphic representation $\pi$ of $\GL_{m+1}(\A_{F'})$ of weight $(0)_{\tau,i}$ such that for some, and hence every,  
rational prime $l$ and any $\imath:\Qbar_l \cong \C$ we have
\[ \Symm^m r_{E,l}|_{G_{F'}}^\vee \cong r_{l,\imath}(\pi). \]
 Moreover, $\pi$ is unramified above any prime where $E$ has good reduction.
\end{cor}

\begin{proof} We may, and will, suppose that $F^\avoid/\Q$ is Galois. Choose a rational prime
  $l\geq \max_{m \in \cM} 2(m+2)$ such that $E$ has good ordinary reduction at $l$, 
  $\barr_{E,l}$ has image $\GL_2(\F_l)$, $l \not\in
  \CL$, and  $l$ is unramified in $F^\avoid$. (By~\cite[Thm.\ 20]{MR644559}, the
  condition that $E$ is ordinary at $l$ excludes a set of primes of
  Dirichlet density $0$. By the main result of~\cite{MR0387283}, each of the other conditions excludes a
  finite number of primes.) %
  Note that $\Qbar^{\ker \barr_{E,l}}$ contains $\zeta_l$ and, by part \ref{part:nuq2} of Lemma \ref{nuq}, is
  linearly disjoint from $F^\avoid$ over $\Q$.

Choose an imaginary quadratic field $L$ which is unramified at all
primes in $\cL$, at all primes where~$E$ has bad reduction, and all primes which ramify in $F^\avoid$, and in which $l$ splits. Also choose a rational prime $q \not\in \CL \cup \{ l\}$ which splits as $v_qv_q'$ in $L$, which is unramified in $F^\avoid$ and at which $E$ has good reduction. %

 If $m \in \cM$ is even also choose a character
  \[ \psi_m:G_L \lra \Qbar_l^\times \]
  such that 
  \begin{itemize}
  \item $\psi_m$ is crystalline above $l$ with Hodge--Tate numbers $0$ at one place above $l$ and $m+1$ at the other.
  \item $q|\# (\psi_m/\psi_m^c)(G_{F_{v_q}^\nr})$.
  \item $\psi_m$ is unramified above $\CL$ and all primes which ramify in $F^\avoid$ and all primes at which $E$ has bad reduction.
  \item $\psi_m \psi_m^c=\epsilon_l^{-(m+1)}$.
  \end{itemize}
(\cite[Lem.\ A.2.5]{BLGGT} tells us that this is possible.) The representation $\Ind_{G_L}^{G_\Q} \psi_m$ has determinant $\epsilon_l^{-(m+1)}$. (This is true on $G_L$ by the construction of $\psi_m$ and true on complex conjugation because $m$ is even.) 

Let $L_2$ denote the compositum of the $\barQQ^{\ker \Ind_{G_\Q}^{G_L}
  \barpsi_m}$ for $m\in\cM$ even. Let $L_1$ denote the maximal
sub-extension of $L_2$ ramified only at $l$ and let $T$ denote the set
of primes other than $l$ that ramify in $L_2$. Then $L_2 \cap
\barQQ^{\ker \barr_{E,l}} = L_1 \cap \barQQ^{\ker \barr_{E,l}}$. Let
$L_3=L_2 \barQQ^{\ker \barr_{E,l}}$. We will now show that~$F^\avoid$
is linearly disjoint from~$L_3$ over~$\Q$; the argument is somewhat
involved, and the reader may find it helpful to refer to the diagram
of field extensions below.

Let $M_1$ denote the maximal
subfield of $L_3$ in which the primes of $T$ are all unramified. Then
$M_1 \supset \barQQ^{\ker \barr_{E,l}}$ (because~$\barr_{E,l}$ can
only be ramified at~$l$ and places where~$E$ has bad reduction), and $M_1 \cap L_2=L_1$ and $M_1=L_1 \barQQ^{\ker \barr_{E,l}}$. Thus
\[ \begin{array}{rcl} \Gal(M_1/(L_1 \cap \barQQ^{\ker \barr_{E,l}})) &\cong &\Gal(\barQQ^{\ker \barr_{E,l}}/(L_1 \cap \barQQ^{\ker \barr_{E,l}})) \times \Gal(L_1/(L_1 \cap \barQQ^{\ker \barr_{E,l}})) \\ &\cong& \Gal(M_1/L_1) \times \Gal(M_1/\barQQ^{\ker \barr_{E,l}}).\end{array} \]
As $\Gal(L_1/\Q)$ is soluble, we see that $L_1 \cap \barQQ^{\ker
  \barr_{E,l}} \subset \Q(\zeta_l)$ and that $\Gal(M_1/\barQQ^{\ker
  \barr_{E,l}})$ is soluble and hence that $\Gal(M_1/(L_1 \cap
\barQQ^{\ker \barr_{E,l}}))$ contains a unique copy of $\SL_2(\F_l)$
(because this latter group is perfect, and in particular admits no
solvable quotient) and this copy is therefore normal in
$\Gal(M_1/\Q)$. Its fixed field is $L_1(\zeta_l)$.

Let $H$ be the
subgroup of $\Gal(M_1/\Q)$ generated by the inertia groups above
$l$. The group $H$ maps surjectively to $\Gal(\barQQ^{\ker
  \barr_{E,l}}/\Q)$ (because~$H$ is normal, and the only subfield of~$\barQQ^{\ker
  \barr_{E,l}}$ unramified at~$l$ is~$\Q$ itself, by Lemma \ref{nuq}) and
so must contain the unique copy of $\SL_2(\F_l)$. Thus the maximal
sub-extension $M_0$ of $L_3$ in which $l$ and all elements of $T$ are
unramified is contained in $L_1(\zeta_l)$. This latter field is only
ramified above $l$ and so $M_0=\Q$. Finally we deduce that $F^\avoid$
is linearly disjoint from $L_3$ over $\Q$ (using that all of the
primes in~$T\cup\{l\}$ are unramified in~$F^\avoid$). 

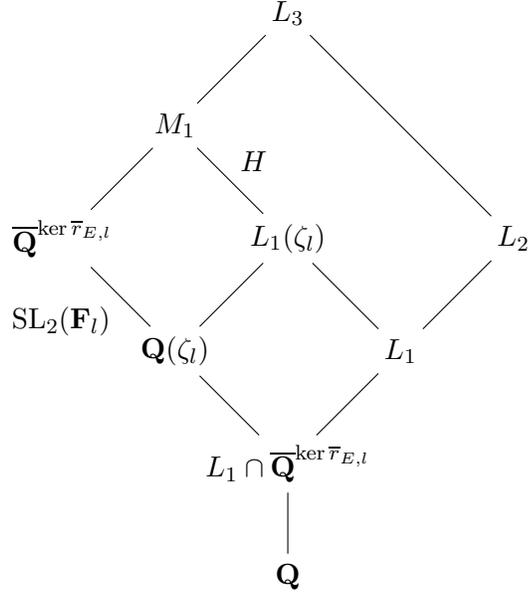
\begin{tikzpicture}[node distance = 1.5cm, auto]
\node (Q) {$\Q$};
\node (A)[above of=Q]{$L_1 \cap \barQQ^{\ker \barr_{E,l}}$};
\node(B)[above of =A, left of =A]{$\Q(\zeta_l)$};
\node(C)[above of = A, right of =A]{$L_1$};
\node(D)[above of =B, left of =B]{$\barQQ^{\ker \barr_{E,l}}$};
\node(E)[above of=A, node distance =3cm]{$L_1(\zeta_l)$};
\node(F)[above of =B, node distance=3cm]{$M_1$};
\node(G)[above of =C, right of=C]{$L_2$};
\node(H)[above of = E, node distance=3cm]{$L_3$};
\draw[-] (Q) -- (A);
\draw[-] (A) -- (B);
\draw[-]  (B) -- node {$\SL_2(\F_l)$} (D);
\draw[-] (B) -- (E);
\draw[-] (A) -- (C);
\draw[-] (C) -- (E);
\draw[-] (C) -- (G);
\draw[-] (D) -- (F);
\draw[-] (E) -- node [swap] {$H$} (F);
\draw[-] (F) -- (H);
\draw[-] (G) -- (H);

\end{tikzpicture}
  
 If $m \in \cM$ is odd, set 
\[ r_m = \Symm^m r_{E,l}^\vee: G_\Q \lra \GSp_{m+1}(\Qbar_l). \]
It has multiplier $\epsilon_l^{-m}$, is unramified above $\CL$, and is crystalline and ordinary at $l$ with Hodge--Tate numbers $\{0,1,...,m\}$. If $m \in \cM$ is even, set
 \[ r_m= (\Symm^m r_{E,l}^\vee) \otimes \Ind_{G_F}^{G_\Q} \psi_m .\]
As the representation $(\Symm^m r_{E,l}^\vee)$ is orthogonal with multiplier $\epsilon_l^{-m}$, we see that
\[ r_m: G_\Q \lra \GSp_{2(m+1)}(\Qbar_l) \]
with multiplier $\epsilon_l^{-(2m+1)}$. It is unramified above $\CL$ and it is crystalline and ordinary at $l$ with Hodge--Tate numbers $\{0,1,...,2m+1\}$.  %

We apply Proposition~\ref{prop: potential modularity of symplectic
  over Q} to $F=F_0=\Q$, $\{r_m: m \in \cM\}$, $\cL$
and~
  $F^\avoid L_3$, producing fields $F^{\avoid}_1$ and $F^\suffi$. Set $F^\avoid_2 = F_1^\avoid L_3$. Then $F_2^\avoid$ is linearly disjoint from $F^\avoid$ over $\Q$, and $F^\suffi$ is linearly disjoint from $F^\avoid F_2^\avoid$ over $\Q$. 

Suppose that $F'/F^\suffi$ is a finite totally real extension linearly disjoint from $F_2^\avoid$ over $\Q$. Then $\Symm^m \SL_2(\F_l) \subset \Symm^m \barr_{E,l}^\vee (G_{LF'(\zeta_l)})$, and so for $m \in \cM$ the tautological representation of the subgroup of $\Symm^m \barr_{E,l}^\vee (G_{LF'(\zeta_l)})$ generated by its elements of $l$-power order is absolutely irreducible. 
 If $m \in \cM$ is even, then $\barr_m|_{G_{LF'(\zeta_l)}}$ is the direct sum of two absolutely
irreducible constituents. The group $\Gal(L_3F'/\barQQ^{\ker \barr_{E,l}}LF')$
acts by different characters on these two
constituents, and $\Gal(\barQQ^{\ker \barr_{E,l}}LF'/\barQQ^{\ker
  \barr_{E,l}}F')$ interchanges these two characters (consider the
action of inertia above~$q$). Thus
$\barr_m|_{G_{F'(\zeta_l)}}$ is absolutely irreducible. It follows from
~ \cite[Prop.\ 2.1.2]{BLGGT} that, for $m \in \cM$ odd or even, $\barr_m(G_{F'(\zeta_l)})$ is adequate. 
  
 Combining
  Proposition~\ref{prop: potential modularity of symplectic over Q} with~
  \cite[Thm.\ 2.4.1]{BLGGT}, we deduce that $r_m$ is
  automorphic for $m \in \cM$. It
  follows (using, in the case that~$m$ is even, \cite[Lem.\ 2.2.4]{BLGGT} and the argument of~\cite[Lem.\ 4.2.2]{cht}) that $\Symm^m r_{E,l}^\vee|_{G_{F'}}$ is
  automorphic. This finishes the proof of the corollary.
 \end{proof}

\subsubsection{The main proof}Finally we turn to the proof of
Theorem~\ref{mainthm}.

\begin{proof}[Proof of Theorem~\ref{mainthm}] \label{proofbrauer}

 Choose a non-CM elliptic curve $E/\Q$ with good reduction above $\cL_0$. Choose distinct rational primes $l_1$ and $l_2$ and a prime $\lambda_i|l_2$ of $M_i$ for each $i \in \cI$ such that:
 \begin{assumption} 
\label{weactuallyusethissomewhere} \leavevmode
\begin{enumerate}
\item  $l_2$ splits completely in each $M_i$.%
\item \label{sortedout} %
The image of $G_F$ on $E[l_1]$ contains $\SL_2(\F_{l_1})$, and $\barr_{i,\lambda_i}(G_F)$ contains $\SL_2(\F_{l_2})$ for each $i\in \cI$.

\item $l_1$ and $l_2$ are unramified in $F$.
\item $E$ has good reduction above $l_1$ and $l_2$.
\item $l_1$ and $l_2$ lie under no prime of any $S_i$.
\item $l_1,l_2 > 2m_i+3$ for all $i$. 
\end{enumerate}
\end{assumption}
This is possible because all the conditions are satisfied for a set of primes of Dirichlet density $1$ (using Lemma \ref{resirred}), except for the first condition for $l_2$, which is satisfied for a set of primes of positive Dirichlet density. 

 Set $\cL=\cL_0 \cup \{ l_1,l_2\}$.  The weakly compatible system of characters
  \[(M_i,S_i,\ell_i,\{ q_v^{-1} Q_{i,v}(0)\}, \{ \epsilon_l \det r_{i,\lambda} \}, \{
  \{0\} \})\] 
  has all Hodge--Tate numbers $0$ and so there is a
  character $\psi_i: G_F \to M_i^\times$ with open kernel unramified
  outside $S_i$ such that $\det r_{i,\lambda} = \psi_i \epsilon_l^{-1}$ for
  all $\lambda$ (a prime of $M_i$ with residue characteristic $l$). There is a sequence
  \[ \Hom(G_F,\barM_i^\times) \stackrel{2}{\lra} \Hom(G_F,\barM_i^\times)
    \stackrel{\partial}{\lra} H^2(G_F,\{\pm 1\})= \Br_F[2] \into
    \oplus_v \Br_{F_v}[2], \] 
    which is exact at the second term. The
  element $\partial \psi_i$ is non-trivial only at places $v \in S_i$. We can find a soluble Galois totally real extension $F_1^+/\Q$, unramified above $\cL$
  and linearly disjoint over $\Q$ from the normal closure $F_1^\avoid$ of $F_0^\avoid \barF^{\ker (\barr_{E,l_1} \times \prod_i \barr_{i,\lambda_i} )}$ over $\Q$, such that for each $v \in \bigcup_{i \in \cI} S_i$, the rational prime $p_v$ below $v$ has inertia degree in $F_1^+$ divisible by $2[F_v:\Q_{p_v}]$. (See for instance~\cite[Lem.\
  4.1.2]{cht}.) Then we see that, for each $i \in \cI$,
  $\partial \psi_i|_{G_{FF_1^+}}$ is trivial so that there is a continuous
  homomorphism
  \[ \phi_i: G_{FF_1^+} \lra \barM_i^\times \] such that
  $\phi_i^2=\psi_i|_{G_{FF_1^+}}$. 
  For $v|l \in \cL$ we see that $\phi_i|_{I_{FF_1^+,v}}^2=1$. By the Grunwald--Wang theorem (see Theorem 5 of Chapter X of \cite{artintate}),
  we can find a continuous character $\delta_i:G_{FF_1^+} \ra \{\pm 1\}$ such that $\phi_i\delta_i$ is unramified at all places above $\cL$. Replacing $\phi_i$ by $\phi_i\delta_i$ we may suppose that $\phi_i$ is unramified at all places above $\cL$.
  Set
  \[ \cR'_i=\CR|_{G_{FF_1^+}} \otimes \phi_i^{-1} \]
  with $S_i'=S_i$.

 We apply Corollary \ref{ecpm} to $\{ m_i:\,\, i \in \cI\}$, $E$,
 $\cL$ and $F_1^\avoid F_1^+$. We obtain a finite Galois extension
 $F^{\avoid}_2/\Q$ linearly disjoint from $F_1^\avoid F_1^+$ over $\Q$
 and a finite totally real Galois extension $F^{+,\suffi}/\Q$, which
 is unramified above $\cL$ and linearly disjoint from $F_1^\avoid
 F_2^\avoid F^+_1$ over $\Q$. Set $F^\avoid=F_1^\avoid F_2^\avoid$. It
 is Galois over $\Q$, and certainly contains~$\barF^{\ker (\barr_{E,l_1} \times \prod_i \barr_{i,\lambda_i} )} \subset F_1^\avoid$
  by definition. 
 Moreover it is linearly
 disjoint from $F_1^+F^{+,\suffi}$ over $\Q$ (see the diagram of field
 extensions later in this proof).

  Let $V_{\barr_{i,\lambda_i}'}$ denote the vector space underlying $\barr_{i,\lambda_i}'$ and give it a non-degenerate symplectic pairing, which $\barr_{i,\lambda_i}'$ will then preserve up to multiplier $\barepsilon_{l_2}^{-1}$. Let $Y_i/FF_1^+$ denote the moduli space of elliptic curves $D$ along with isomorphisms $\alpha_1: E[l_1] \iso D[l_1]$ and $\alpha_2:V_{\barr_{i,\lambda_i}'}^\vee \iso D[l_2]$, which preserve symplectic pairings. Let $X_i/F^+F_1^+$ denote the restriction of scalars of $Y_i$. 

  If  $v$ is an infinite place of $F^+F_1^+$, then a point of $X_i((F^+F_1^+)_v)$ is
  the same as an $\overline{(FF_1^+)_v}$-point of $Y_i$ and hence
  $X_i((F^+F_1^+)_v) \neq \emptyset$. 

  Suppose that $w$ is a place of $F_1^+F$ above $\cL_0 \cup \{l_1\}$. Then we can find a positive integer $f$ such that
  $\barr_{i,\lambda_i}'(\Frob_w)^{-f} \sim \barr_{E,l_2}(\Frob_w)^f$. Thus $E$ gives rise to a
  point of $Y_i$ over the unramified extension of degree $f$ of
  $(F_1^+F)_w$. Hence $X_i$ has a  point over an unramified extension
  of $(F_1^+F^+)_v$ for every place $v$ above $\cL_0 \cup
  \{l_1\}$. Moreover this point corresponds to an elliptic curve with good reduction.

  Now suppose that $v$ is a place of $F_1^+F^+$ above $l_2$. We will show that $X_i$ has a rational point over an
  unramified extension of $(F_1^+F^+)_v$ corresponding to an elliptic curve with good reduction. It suffices to show that $Y_i$
  has a point over an unramified extension of $(F_1^+F)_w$ for every
  prime $w$ of $F_1^+F$ over $v$ and that this point corresponds to an
  elliptic curve with good reduction.   Because $\Q_{l_2} \cong
  M_{i,\lambda_i}$, part \ref{part:nuq1} of Lemma \ref{nuq} implies
  that the restriction %
  $\barr_{i,\lambda_i}'|_{G_{(F_1^+F)_w}}$ must have one of the following two forms:
  \begin{enumerate}
  \item The induction from the unramified quadratic extension of $(F_1^+F)_w$
    of $\omega_{l_2,2}^{-1}\delta$, where $\delta$ is the unramified
    quadratic character.
  \item $\mat{\barchi}{*}{0}{\barchi^{-1}\barepsilon_{l_2}^{-1}}$
    where $\barchi$ is unramified and where the extension class is peu
    ramifi\'{e} in the sense that it lies in
    \[ H^1_f(G_{(F_1^+F)_w}, \F_{l_2}(\barepsilon_{l_2}\barchi^2)) \subset
      H^1(G_{(F_1^+F)_w}, \F_{l_2}(\barepsilon_{l_2}\barchi^2)).\]
  \end{enumerate}
  (While the statement of Lemma~\ref{nuq} does not prescribe the direction of the
  extension in the second possibility, nor specify that it is peu
    ramifi\'{e}, these follow easily from the connected--\'etale
    sequence for the finite flat group scheme~$H$ considered in the
    proof of Lemma~\ref{nuq}.)
  In the first case, let $D/(F_1^+F)_w$ be an elliptic
  curve with good supersingular reduction. Choose a positive integer $f$ such that
  $\barr_{E,l_1}(\Frob_v)^{2f}=(-l_2)^f \bmod l_1$. Then $D$ provides a point of
  $Y_i$ over the unramified extension of $(F_1^+F)_w$ of degree $2f$.

 In the second case, let $\barD/k(w)$ be an ordinary elliptic curve
  and let $\psi: G_{k(w)} \rightarrow \Z_{l_2}^\times$ denote the
  character by which $G_{k(w)}$ acts on the Tate module
  $T_{l_2}\barD$. If $L/F_w$ is a finite extension then, by
  Serre--Tate theory, liftings of $\barD$ to $\CO_L$ are parametrized
  by
  $H^1(G_L, \Z_{l_2}(\epsilon_{l_2} \psi^{-2}))=H^1_f(G_L,
  \Z_{l_2}(\epsilon_{l_2} \psi^{-2}))$, and we shall write $D_e$ for
  the lifting corresponding to a class $e$. (Note that $\psi^{-2}$
  always has infinite order.) Then
  \[ \barr_{D_e,l_2} \cong
    \mat{\barepsilon_{l_2}\barpsi^{-1}}{*}{0}{\barpsi} \] and the
  extension class is the image of $e$ in
  $H^1_f(G_L, \F_{l_2}(\barepsilon_{l_2} \barpsi^{-2}))$.  Choose a
  positive integer $f$ such that $\barchi^f=1$ and
  $\psi^f \equiv 1 \bmod l_2$ and $\barr_{E,l_1}(\Frob_{l_2})^{f}=1$
  and $\Frob_w^f=1$ on $\barD[l_1](\overline{k(w)})$. Let $L/F_w$
  denote the unramified extension of degree $f$, and let
  $e \in H^1_f(G_L, \Z_{l_2}(\epsilon_{l_2} \psi^{-2}))$ lift the
  negative of the class of $\barr_{i,\lambda_i}'|_{G_L}$ in
  $H^1_f(G_{L}, k(w)(\barepsilon_{l_2}\barchi^2))$ (the existence of
  such an~$e$ follows from Lemma~\ref{lem: H1f surjectivity}). Then $D_e/L$ has
  $D_e[l_1] \cong E[l_1]$ and
  $\barr_{D_e,l_2} \cong (\barr_{i,\lambda_i}')^{\vee}$.
  
   It follows  (for instance, by~\cite[Prop.\ 3.1.1]{BLGGT}) that there is a finite extension $F_2^+/F_1^+F^+$ such that:
\begin{itemize}
\item $F_2^+$ is Galois over $F_0^+$.
\item $F_2^+$ is totally real.
\item All primes above $\cL$ are unramified in $F_2^+/F^+$, and $D_i$ has good reduction at all primes in $\cL$.
\item $F_2^+$ is linearly disjoint over $F_1^+F^+$ from 
$F^\avoid F^{+,\suffi}F_1^+$. 
\item $\prod_i X_i$ has an $F_2^+$-rational point, i.e.\ there exist elliptic curves $D_i$ over~ $F_2$ such that $D_i[l_1] \cong E[l_1]|_{G_{F_2}}$ and $D_i[l_2] \cong \barr_{i,\lambda_i}'|^\vee_{G_{F_2}}$.
\end{itemize}

Set $F^\suffi=F^{+,\suffi}F_2^+F$, a CM extension of $F$ which is
unramified  above $\cL$ and Galois over $F_0$. We now show that this
is linearly disjoint from~$F^\avoid$ over~$F$; the reader may find it
helpful to consult the
diagram of field extensions below. 
As $F^{+,\suffi}$ is linearly disjoint from $F_1^+F^\avoid$ over $\Q$, we see that
$F^+F_1^+F^{+,\suffi}$ is linearly disjoint from $F_1^+F^\avoid$ over $F^+F_1^+$, and so $F_2^+F^{+,\suffi}$ is linearly disjoint from $F_1^+F^\avoid$ over $F^+F_1^+$. Thus $F^{\suffi}$ is linearly disjoint from $F_1^+F^\avoid$ over $FF_1^+$.
On the other hand $F_1^+$ is linearly disjoint from $F^\avoid$ over
$\Q$ and so $F^\avoid$ is linearly disjoint from $FF_1^+$ over $F$. We
conclude that $F^\suffi$ is linearly disjoint from $F^\avoid$ over
$F$.
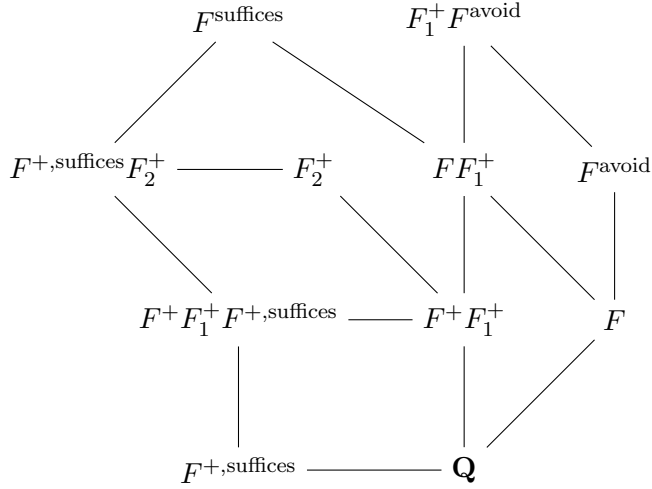
\begin{tikzpicture}[node distance = 2cm, auto]
\node (Q) {$\Q$};
\node (A)[left of=Q, node distance =3cm]{$F^{+,\suffi}$};
\node(B)[above of =A]{$F^+F_1^+F^{+,\suffi}$};
\node(C)[above of = Q]{$F^+F_1^+$};
\node(D)[above of =Q, right of =Q]{$F$};
\node(E)[above of=B, left of =B]{$F^{+,\suffi}F_2^+$};
\node(F)[right of =E, node distance =3cm]{$F_2^+$};
\node(G)[above of =C]{$FF_1^+$};
\node(H)[above of = D]{$F^\avoid$};
\node(I)[above of = E, right of =E]{$F^\suffi$};
\node(J)[above of =G]{$F_1^+F^\avoid$};
\draw[-] (Q) -- (A);
\draw[-] (A) -- (B);
\draw[-]  (B) -- (C);
\draw[-] (Q) -- (D);
\draw[-] (F) -- (C);
\draw[-] (B) -- (E);
\draw[-] (C) -- (G);
\draw[-] (D) -- (G);
\draw[-] (E) --  (F);
\draw[-] (E) -- (I);
\draw[-] (G) -- (I);
\draw[-] (G)--(J);
\draw[-] (H)--(J);
\draw[-] (D)--(H);
\draw[-] (Q)--(C);

\end{tikzpicture}

Suppose that $F'/F$ is a finite CM extension  containing~$F^\suffi$ and
which is unramified above $\cL$ and linearly
disjoint from $F^\avoid$ over~$F$.
   By Corollary~ \ref{ecpm}, there are regular algebraic, cuspidal, polarizable automorphic representations $\pi_i$ of $\GL_{1+m_i}(\A_{F'})$ unramified above $\cL$ and of weight $(0)_{\tau,i}$ such that for any $\imath:\Qbar_{l_1} \cong \C$ we have
\[ \Symm^{m_i} r_{E,l_1}|_{G_{F'}}^\vee \cong r_{l_1,\imath}(\pi_i). \]
Applying Theorem~\ref{thm:main_automorphy_lifting_theorem} (the conditions on the residual representations
 are satisfied by parts~(\ref{part:scalarslimeball}), (\ref{part:enormousslimeball}) and~(\ref{part:genericslimeball1}) of Lemma~\ref{lemma:enormousgeneric} and Lemma~\ref{nuq}),
 we see that there are regular algebraic, cuspidal automorphic representations $\pi_i'$ of $\GL_{1+m_i}(\A_{F'})$ unramified above $\cL$ and of weight $(0)_{\tau}$ and $\imath:\Qbar_{l_1} \cong \C$ such that
\[ \Symm^{m_i} r_{D_i,l_1}|_{G_{F'}}^\vee \cong r_{l_1,\imath}(\pi_i'), \]
and so, for some  $\imath: \Qbar_{l_2} \cong \C$, we have
\[ \Symm^{m_i} r_{D_i,l_2}|_{G_{F'}}^\vee \cong r_{l_2,\imath}(\pi_i'). \]
Applying Theorem~ \ref{thm:main_automorphy_lifting_theorem} again (the conditions on the residual representations again being
 satisfied by parts~(\ref{part:scalarslimeball}), (\ref{part:enormousslimeball}) and~(\ref{part:genericslimeball1}) of Lemma~\ref{lemma:enormousgeneric} and Lemma~\ref{nuq}) we
see that there is are regular algebraic, cuspidal automorphic representations $\pi_i''$ of $\GL_{1+m_i}(\A_{F'})$ unramified above $\cL$ and of weight $(0)_{\tau,i}$, and $\imath:\Qbar_{l_2} \cong \C$ such that
\[ \Symm^{m_i} r_{i,\lambda_i}'|_{G_{F'}} \cong r_{l_2,\imath}(\pi_i''). \]
Untwisting completes the proof of Theorem~\ref{mainthm}. 
\end{proof}

\emergencystretch=3em

\bibliographystyle{amsalpha}
\bibliography{CMpatching}

\end{document}